\documentclass[12pt]{amsart}

\usepackage{amsmath} 
\usepackage{amssymb} 
\usepackage{amsthm} 

\usepackage{microtype} 
\usepackage{pinlabel} 

\usepackage[
hidelinks,
pagebackref,
pdftex]{hyperref}

\renewcommand*{\backref}[1]{}
\renewcommand*{\backrefalt}[4]{
  \ifcase #1 
  [No citations.]
  \or [#2]
  \else [#2]
  \fi }

\let\originalleft\left
\let\originalright\right
\renewcommand{\left}{\mathopen{}\mathclose\bgroup\originalleft}
\renewcommand{\right}{\aftergroup\egroup\originalright}

\newcommand{\calE}{\mathcal{E}}

\newcommand{\calK}{\mathcal{K}}
\newcommand{\calL}{\mathcal{L}}

\newcommand{\calO}{\mathcal{O}}

\newcommand{\calQ}{\mathcal{Q}}

\newcommand{\EE}{\mathbb{E}}

\newcommand{\HH}{\mathbb{H}}

\newcommand{\NN}{\mathbb{N}}

\newcommand{\RR}{\mathbb{R}}

\newcommand{\ZZ}{\mathbb{Z}}

\newcommand{\grad}{\operatorname{grad}}

\newcommand{\from}{\colon} 

\newcommand{\homeo}{\mathrel{\cong}} 

\newcommand{\cover}[1]{{\widetilde{#1}}}

\newcommand{\bdy}{\partial} 

\newcommand{\Isom}{\operatorname{Isom}} 

\newcommand{\SL}{\operatorname{SL}} 

\theoremstyle{plain}

\newtheorem{lemma}[equation]{Lemma}

\theoremstyle{definition}
\newtheorem{definition/}[equation]{Definition}

\newtheorem{exercise}[equation]{Exercise}
\newtheorem{example/}[equation]{Example}

\newtheorem{question/}[equation]{Question}
\newtheorem*{question*}{Question}
\newtheorem*{answer*}{Answer}
\newtheorem*{application*}{Application}

\newenvironment{definition}
  {%
   \pushQED{\qed}\begin{definition/}}
  {\popQED\end{definition/}}

\theoremstyle{remark}
\newtheorem{remark/}[equation]{Remark}
\newtheorem*{remark*}{Remark}

\newtheorem*{case*}{Case}

\newtheorem*{claim*}{Claim}

\newenvironment{remark}
  {%
   \pushQED{\qed}\begin{remark/}}
  {\popQED\end{remark/}}

\newcommand{\refsec}[1]{Section~\ref{Sec:#1}}

\newcommand{\reflem}[1]{Lemma~\ref{Lem:#1}}

\newcommand{\refrem}[1]{Remark~\ref{Rem:#1}}

\newcommand{\reffig}[1]{Figure~\ref{Fig:#1}}
\newcommand{\reftab}[1]{Table~\ref{Tab:#1}}

\newcommand{\refeqn}[1]{Equation~(\ref{Eqn:#1})}

\newcommand{\refitm}[1]{(\ref{Itm:#1})}

\newcommand{\fakeenv}{} 

{ 
 \renewcommand{\fakeenv}{#2} 
 \theoremstyle{plain} 
 \newtheorem*{\fakeenv}{#1~\ref{#2}} 
 \begin{\fakeenv}
}
{
 \end{\fakeenv}
}

\newenvironment{restated}[2]  
{ 
 \renewcommand{\fakeenv}{#2} 
 \theoremstyle{definition} 
 \newtheorem*{\fakeenv}{#1~\ref{#2}} 
 \begin{\fakeenv}
}
{
 \end{\fakeenv}
}

\numberwithin{figure}{section}
\numberwithin{equation}{section}

\usepackage{enumitem}
\usepackage{color}
\usepackage{subfig, caption}
\usepackage{wrapfig}
\usepackage{pdflscape}
\usepackage{array}
\usepackage{adjustbox}
\usepackage{multirow}
\usepackage{lipsum}
\usepackage{comment}
\usepackage{pdflscape}
\usepackage{afterpage}
\usepackage{mathrsfs}
\usepackage{stackengine}
\usepackage{gensymb}

\newcommand{\sinj}{\operatorname{sn}}
\newcommand{\cosj}{\operatorname{cn}}
\newcommand{\ampj}{\operatorname{dn}}

\newcommand{\Riem}{\mathfrak{R}}

\newcommand{\arccosh}{\operatorname{arccosh}}
\newcommand{\arcsinh}{\operatorname{arcsinh}}
\newcommand{\arctanh}{\operatorname{arctanh}}
\newcommand{\sign}{\operatorname{sign}}

\newcommand{\SLR}{{\rm \widetilde{SL}(2,\RR)}}
\newcommand{\bv}{\mathbf{v}}
\newcommand{\be}{\mathbf{e}}

\newcommand{\dist}{{\rm dist}}
\newcommand{\sdf}{{\rm sdf}}

\renewcommand{\leq}{\leqslant}
\renewcommand{\geq}{\geqslant}

\title{Ray-marching Thurston geometries}

\author[Coulon, Matsumoto, Segerman and Trettel]{R\'emi Coulon, Elisabetta A. Matsumoto,\\ Henry Segerman, and Steve J. Trettel} 
\date{\today}

\begin{document}

\begin{abstract}
We describe algorithms that produce accurate real-time interactive in-space views of the eight Thurston geometries using ray-marching. We give a theoretical framework for our algorithms, independent of the geometry involved. In addition to scenes within a geometry $X$, we also consider scenes within quotient manifolds and orbifolds $X / \Gamma$. We adapt the Phong lighting model to non-euclidean geometries. The most difficult part of this is the calculation of light intensity, which relates to the area density of geodesic spheres. We also give extensive practical details for each geometry.
\end{abstract}

\maketitle

\tableofcontents

\section{Introduction}

In this paper we describe a project we initiated at the \emph{Illustrating Mathematics} semester program at ICERM in Fall 2019. The goal of this project is to implement real-time simulations of the eight Thurston geometries in the \emph{in-space view} -- that is, viewed from the perspective of an observer inside of each space, where light rays travel along geodesics. See \reffig{CompareGeometries}. We have collected many of our simulations and videos of them at the website \url{http://www.3-dimensional.space}.
 
These simulations may be experienced with an ordinary keyboard and screen interface, and in some cases in virtual reality. We expect that these simulations will be useful in outreach, teaching, and research. Seeing and moving within a space gives a visceral experience of the geometry, often engendering understanding that is hard or impossible to obtain from ``book learning'' alone. Recent research on embodied understanding~\cite{Lindgren:2016p174,JohnsonGlenberg:2017p24,Gregorcic:2017p020104} addresses these advantages.

The code for our simulations is available online~\cite{github}. We hope that other researchers will be able to use and extend our work to visualize objects of interest in the Thurston geometries and beyond.
In two previous expository papers, we described some surprising features of the Nil~\cite{NEVR3} and Sol~\cite{NEVR4} geometries using images from our simulations. 

\begin{figure}[htbp]
\centering
\subfloat[$\EE^3$]{%
\includegraphics[width=0.49\textwidth]{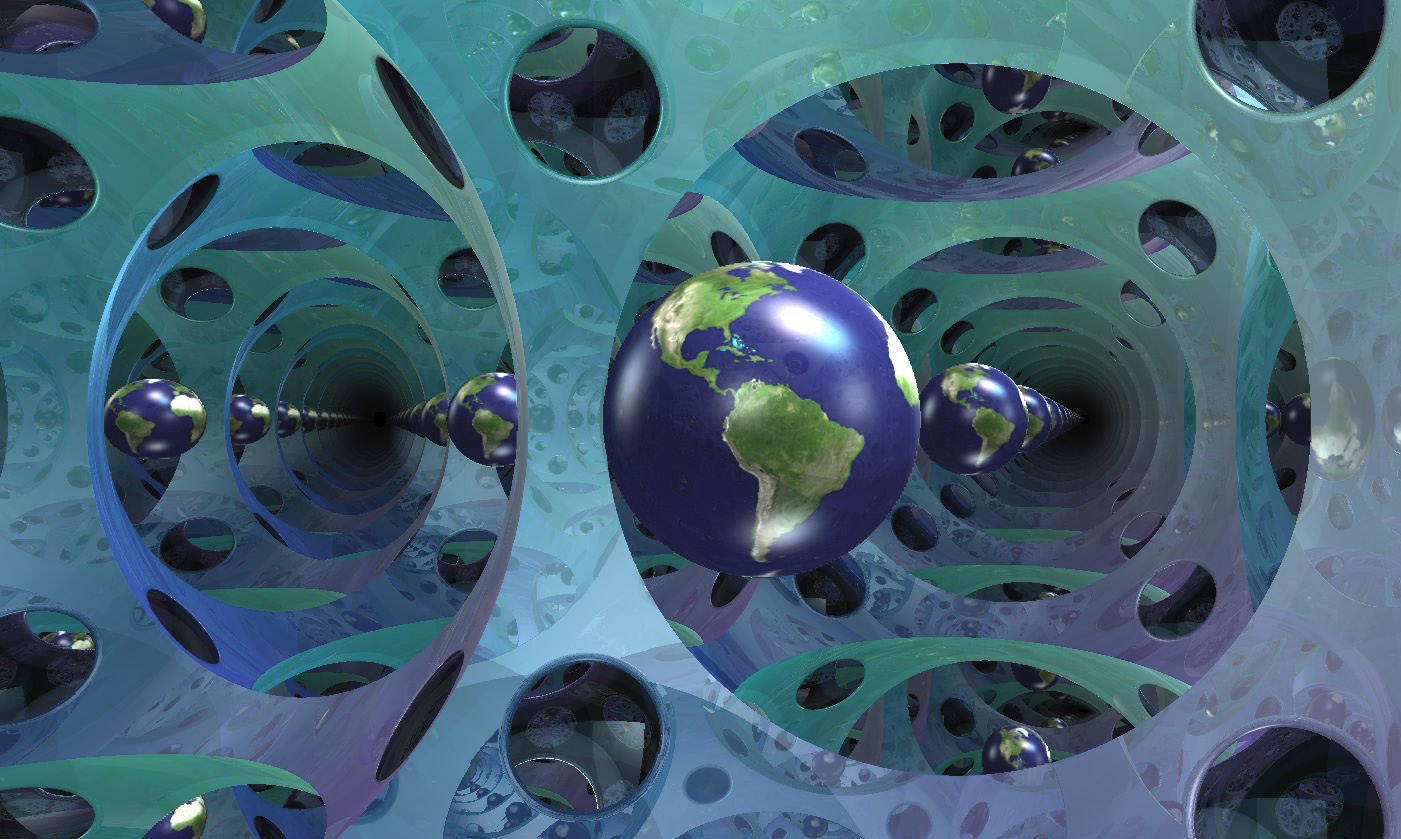}%
\label{Fig:CompareE3}%
}%
\thinspace
\subfloat[$S^3$]{%
\includegraphics[width=0.49\textwidth]{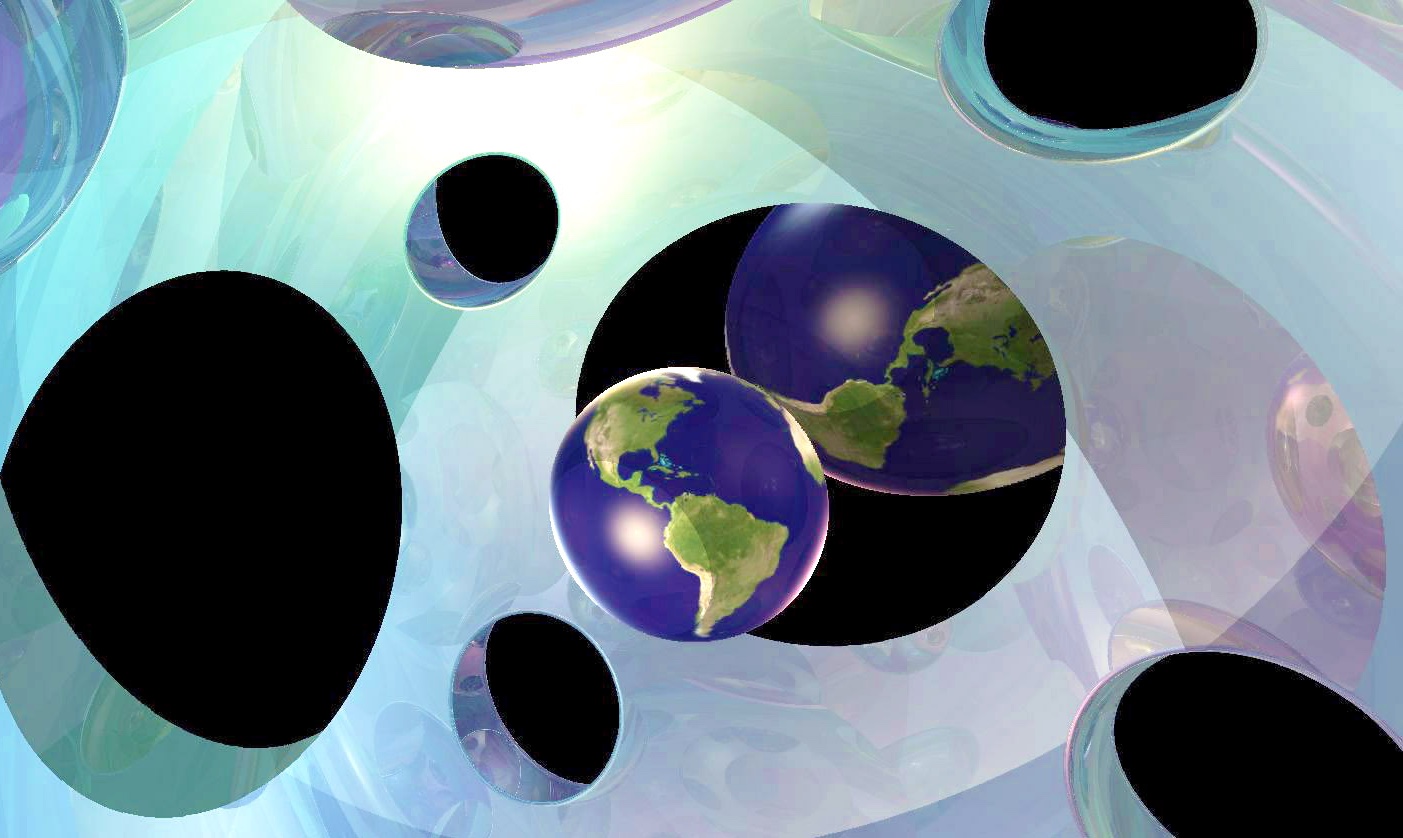}%
\label{Fig:CompareS3}%
}

\subfloat[$\HH^3$]{%
\includegraphics[width=0.49\textwidth]{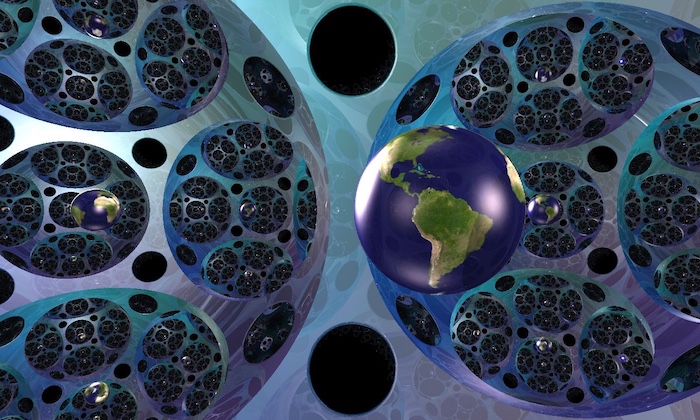}%
\label{Fig:CompareH3}%
}%
\thinspace
\subfloat[$S^2 \times \EE$]{%
\includegraphics[width=0.49\textwidth]{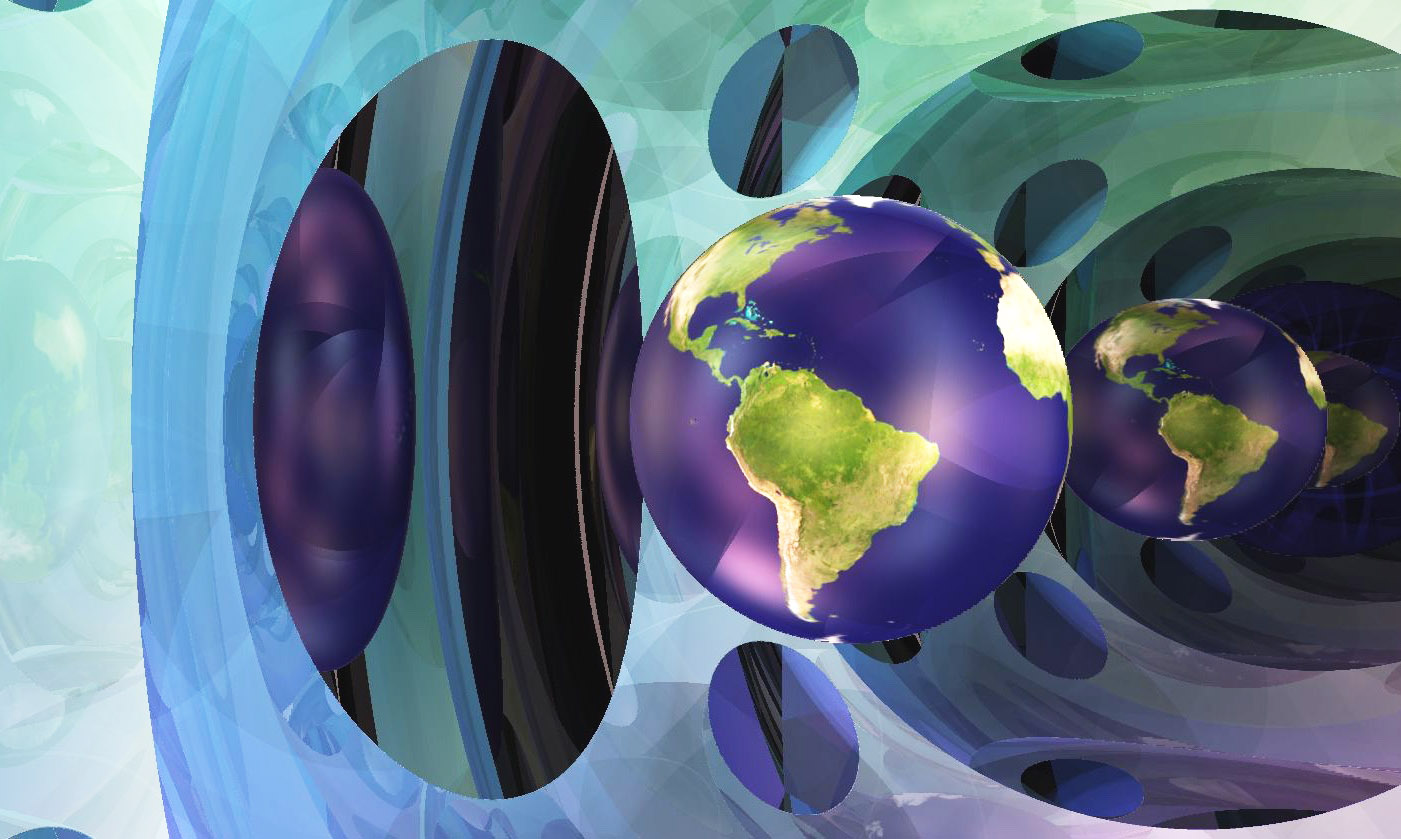}%
\label{Fig:CompareS2xE}%
}

\subfloat[$\HH^2 \times \EE$]{%
\includegraphics[width=0.49\textwidth]{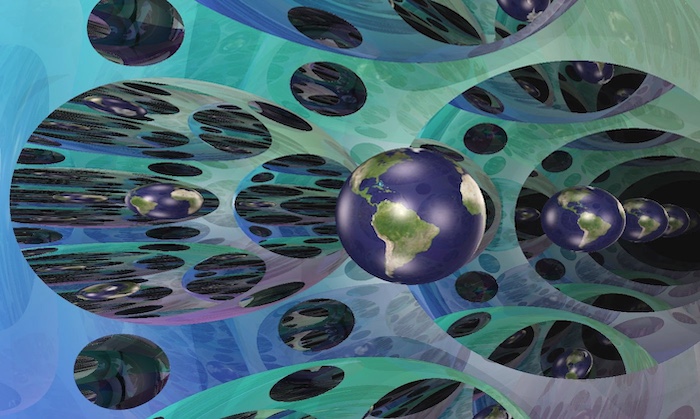}%
\label{Fig:CompareH2xE}%
}%
\thinspace
\subfloat[Nil]{%
\includegraphics[width=0.49\textwidth]{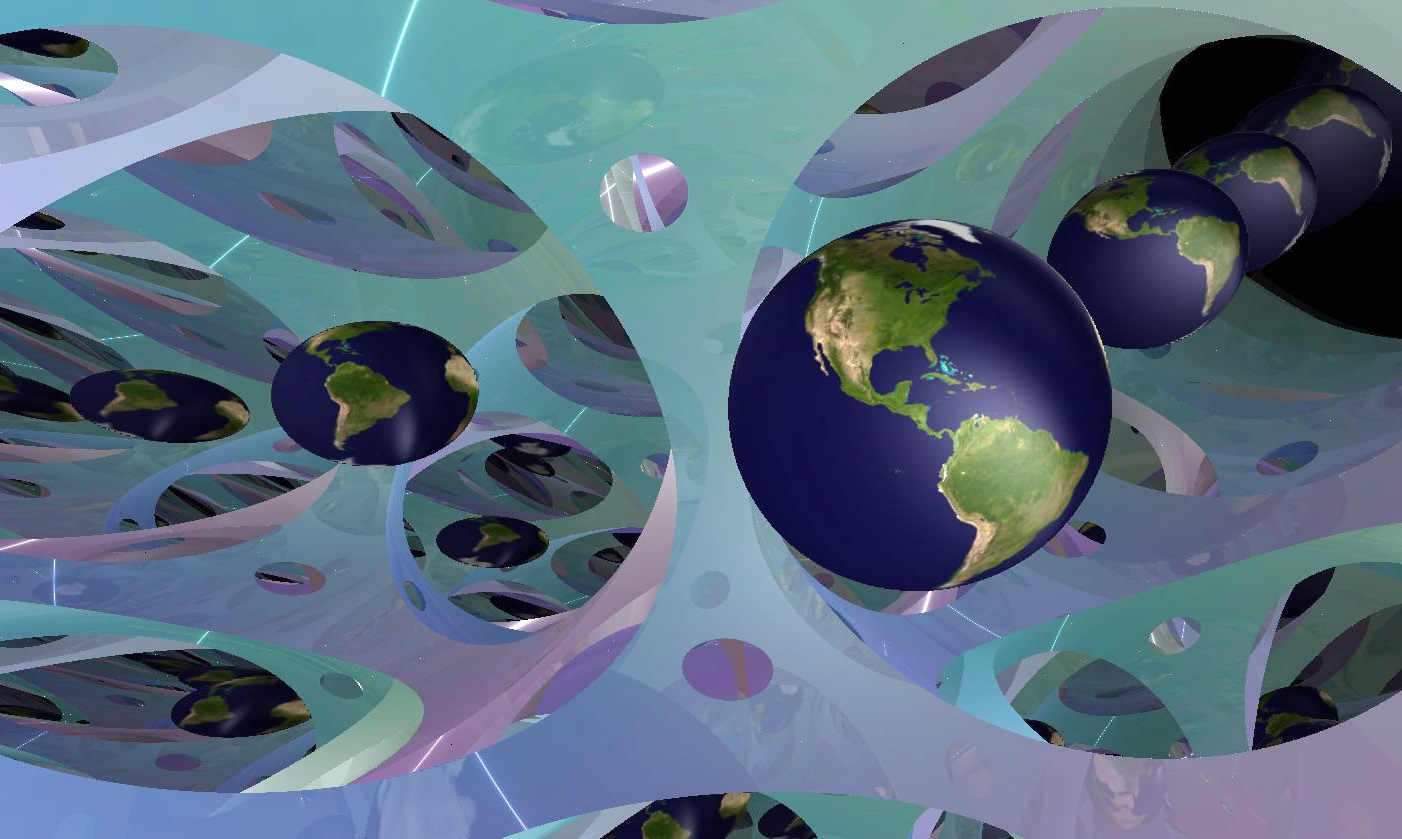}%
\label{Fig:CompareNil}%
}

\subfloat[$\SLR$]{%
\includegraphics[width=0.49\textwidth]{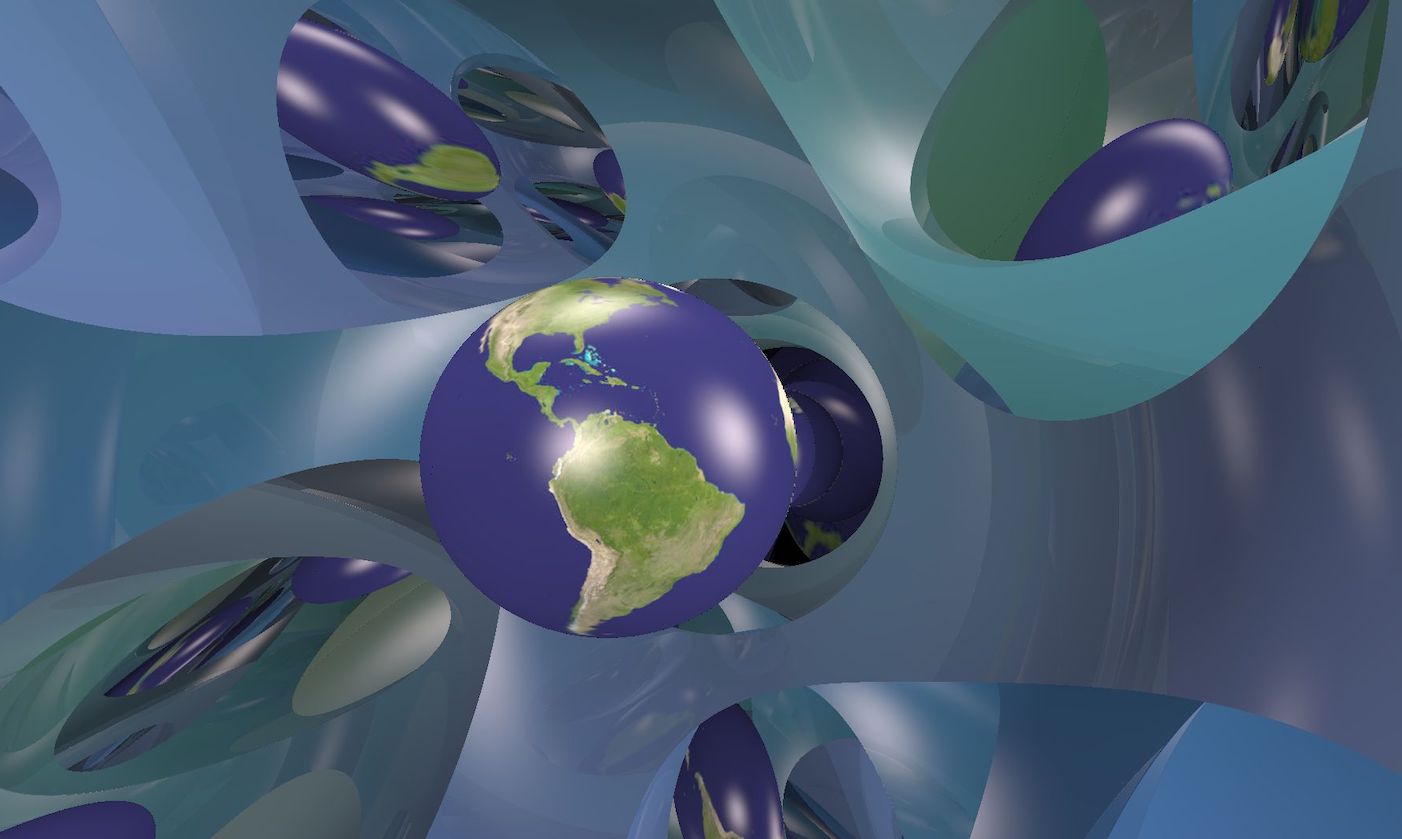}%
\label{Fig:CompareSLR}%
}%
\thinspace
\subfloat[Sol]{%
\includegraphics[width=0.49\textwidth]{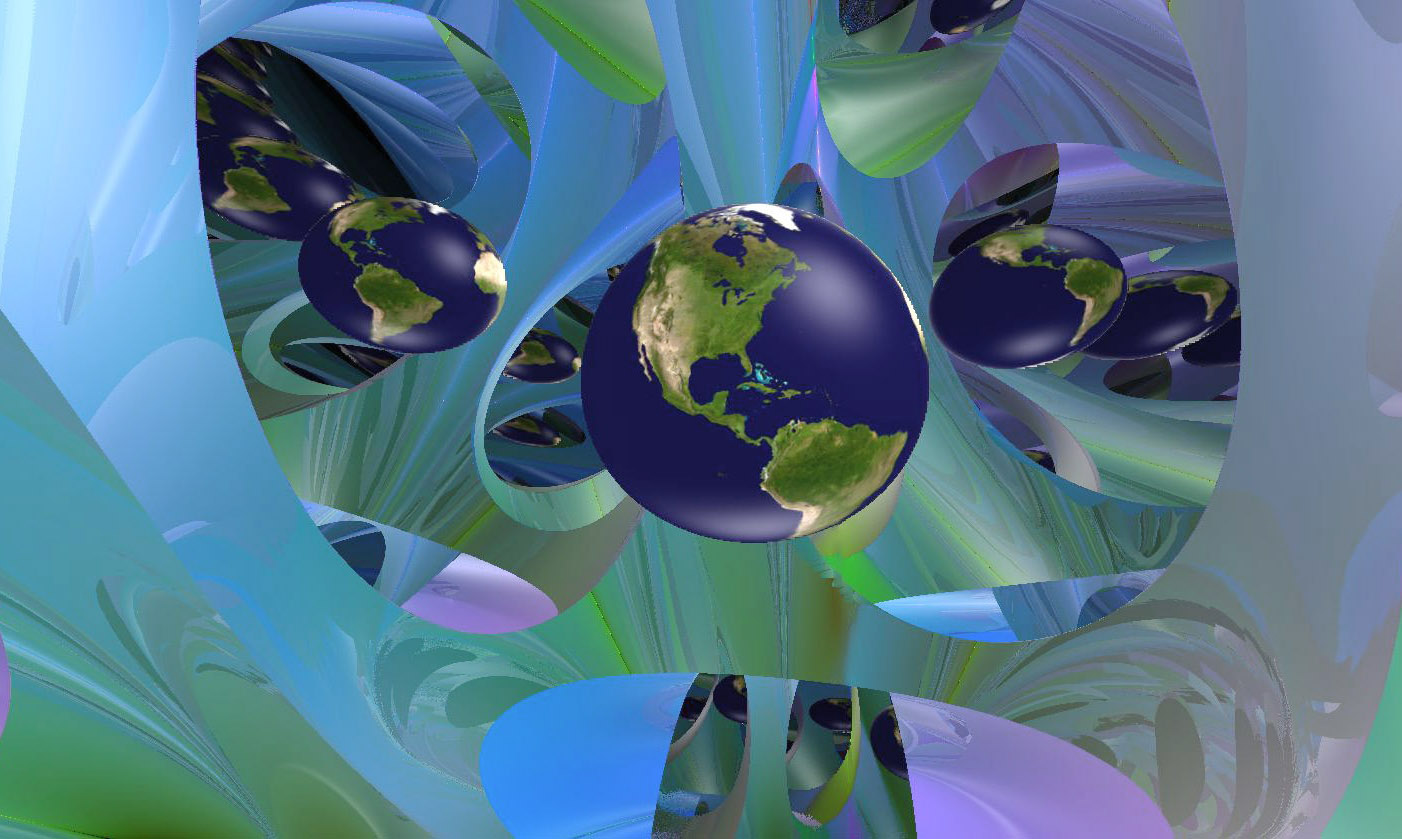}%
\label{Fig:CompareSol}%
}

\caption{Inside views of tilings within each of the eight Thurston geometries. Here we have chosen similar scenes to highlight the differences stemming from the geometries. Each scene is made from tiles as illustrated in \reffig{primitive cell E3}.}
\label{Fig:CompareGeometries}
\end{figure}

\subsection{Thurston's eight geometries}
\label{Sec:ThurstonGeometries}

The expansion of geometry beyond euclidean $n$-space traces its origins to the $19^\textrm{th}$ century discovery of hyperbolic geometry.
From here, Klein made the following wide-reaching generalization.
A \emph{homogeneous geometry} is a pair $(G,X)$ consisting of a smooth manifold $X$, equipped with the transitive action of a Lie group $G$.
The manifold $X$ defines the underlying space of the geometry, and the group $G$ defines the collection of allowable motions.
This convenient mathematical formalism turns some of our traditional geometric thinking upside down.
Instead of defining euclidean geometry as $\RR^n$ with a particular metric, we define it as $\RR^n$ with a particular group of allowable diffeomorphisms (rotations, reflections, and translations), and derive as a consequence the existence of an invariant metric.

In dimension two, homogeneous geometries play an outsized role in mathematics, in large part due to the uniformization theorem.
This implies that every two-dimensional manifold can actually be equipped with a geometric structure modeled on one of the homogeneous spaces $\HH^2,\EE^2,$ or $S^2$.
Because of this, one may often use geometric tools in settings without an obviously geometric nature.
In the 1970s and 1980s, Thurston came to realize that a similar (but more complicated) result might hold in three dimensions.
Thurston's geometrization conjecture stated that every closed three-manifold may be cut into finitely many pieces, each can be built from some homogeneous geometry.
The proof of geometrization was completed by Perelman in 2003~\cite{perelman2002entropy, perelman2003finite, perelman2003ricci} and provides a powerful tool in three-dimensional topology. This also resolved the Poincar\'e conjecture, which had been open for more than a century.
The eight geometries required for geometrization can be defined abstractly as follows.
A homogeneous space $(G,X)$ is a \emph{Thurston geometry} if it has the following four properties:

\begin{enumerate}
\item $X$ is connected and simply connected.
\item $G$ acts transitively on $X$ with compact point stabilizers.
\item $G$ is not contained in any larger group of diffeomorphisms acting with compact stabilizers.
\item There is at least one compact $(G,X)$ manifold.	
\end{enumerate}

The first of these conditions rules out unnecessary duplicity in our classification.
Every connected $(G,X)$ geometry is covered by a simply connected universal covering geometry, so it suffices to consider these.
The second condition is the group-theoretic way of requiring that $X$ has a $G$-invariant riemannian metric, and the third condition is just the statement that $G$ is actually the full isometry group.
A geometry satisfying (1)--(3) is called \emph{maximal}.  
The fourth condition recalls our original motivation: to study geometric structures on compact manifolds in dimension three; we need only concern ourselves with geometries which can be used to build geometric structures!

Three dimensions is small enough that all of the Thurston geometries arise from relatively simple constructions\footnote{There are 19 maximal geometries in dimension four~\cite{Hillman2002}, and 58 in dimension five~\cite{Geng20165dimensionalGI}.  While many of these can be constructed by analogous procedures, some new phenomena also arise.}, growing out of either two-dimensional geometry or three dimensional Lie theory. 
This divides the set of Thurston geometries into a collection of overlapping families of geometries constructed by similar means.
Some of these families are listed below and illustrated in \reffig{ThurstonGeos1}. 

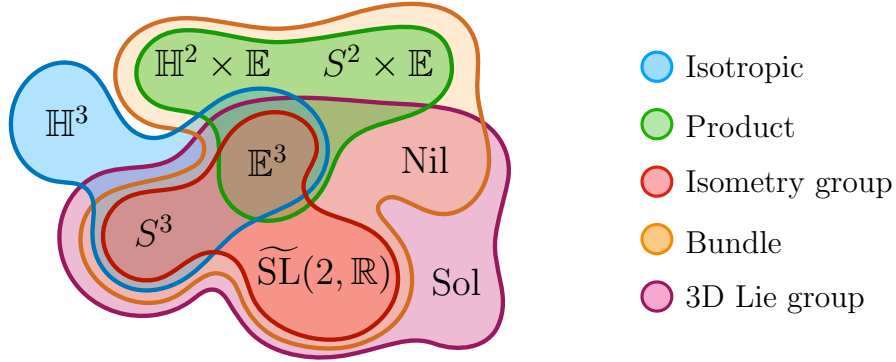
\begin{figure}
\centering
\labellist
\large\hair 2pt
\pinlabel $\HH^2\times\EE$ at 520 760
\pinlabel $S^2\times\EE$ at 940 760
\pinlabel $\HH^3$ at 150 610
\pinlabel $\EE^3$ at 660 510
\pinlabel Nil at 1060 510
\pinlabel $S^3$ at 370 330
\pinlabel $\SLR$ at 805 230
\pinlabel Sol at 1140 200

\normalsize
\pinlabel Isotropic [l] at 1710 742
\pinlabel Product [l] at 1710 596
\pinlabel {Isometry group} [l] at 1710 449
\pinlabel Bundle [l] at 1710 302
\pinlabel {3D Lie group} [l] at 1710 156
\endlabellist
\includegraphics[width=0.9\textwidth]{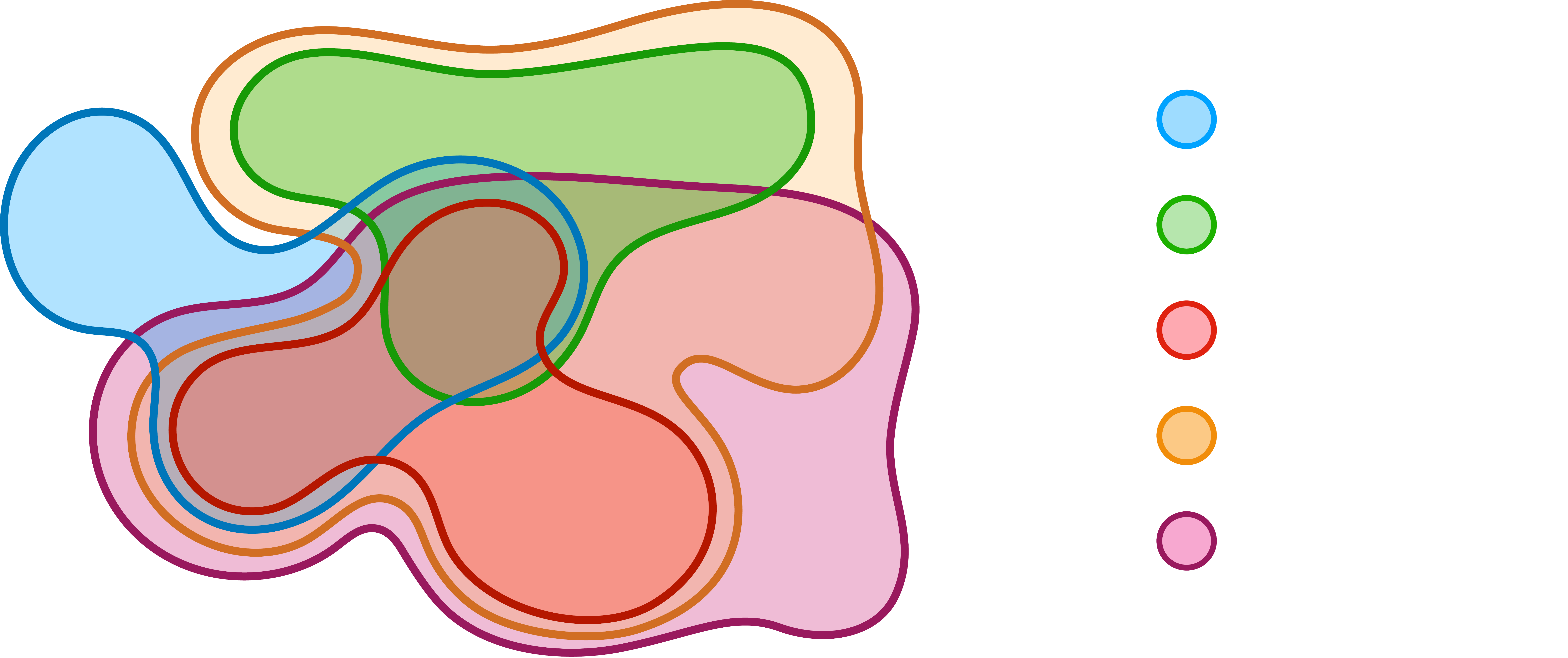}
\caption{The Thurston Geometries, and natural families grouping geometries with similar constructions.}
\label{Fig:ThurstonGeos1}
\end{figure}

\begin{enumerate}
	\item {\bf Isotropic Geometries.} A geometry $(G,X)$ is \emph{isotropic} if the point stabilizer contains $O(3)$.  This acts transitively on the unit tangent sphere at a point. Since directions and planes are dual to each other, 
	any $G$-invariant metric on $X$ must have constant sectional curvature.  Thus, this family consists of $S^3=(O(4),S^3),\EE^3=(O(3)\rtimes\RR^3,\RR^3)$ and $\HH^3=(O(3,1),\HH^3)$.
	\item {\bf Products of Lower Dimensional Geometries.} The product of the unique one-dimensional geometry (denoted $\EE$ in this paper) and any two-dimensional geometry gives a geometry of dimension three. This family consists of the three geometries $S^2\times \EE, \HH^2\times\EE$ and $\EE^2\times\EE$. The latter is not maximal: its isometry group is contained in that of $\EE^3$.
	\item {\bf Isometry groups of two-dimensional geometries.} Each of the two-dimensional geometries $(G,X)$ is isotropic, so $G$ acts transitively on the unit tangent bundle $UTX$.  
	Thus we may consider the three-dimensional geometry $(G,UTX)$, and get a maximal geometry by taking covers and extending the isometry group if necessary.	This gives the geometries $S^3$ and $\EE^3$, as well as the new geometry $\SLR$ (built from $UT S^2$, $UT\EE^2$ and $UT\HH^2$ respectively).
 \item {\bf Bundles over two-dimensional geometries.} Generalizing both of the previous cases, we may construct all geometries $(G,X)$ where $X$ has a $G$-invariant bundle structure over a two-dimensional geometry.  This produces one new example: Nil, a line bundle over $\EE^2$.  This bundle structure has an important geometric consequence: all manifolds with these geometries are \emph{Seifert fibered}.
 	\item {\bf Three-dimensional Lie groups.} Every three-dimensional Lie group $H$ acts on itself freely by left translation.  Starting from the homogeneous geometry $(H,H)$, we may build a maximal geometry by taking covers and extending the group of isometries, if necessary.
Assuming that $H$ is unimodular, this construction recovers the unit tangent bundle geometries and Nil, and produces our final geometry, $\mathrm{Sol}$~\cite[Section 4]{Milnor1976}. Allowing Lie groups that are not unimodular, we also recover $\HH^3$ and $\HH^2 \times \EE$.
	\end{enumerate}
	
\noindent
For a proof that there are only eight Thurston geometries, see for example~\cite{patrangenaru1996}.
As a general reference for Thurston's geometries, see \cite{Scott}.

\subsection{Goals}
\label{Sec:Goals}

We have the following goals for the algorithms we use to render our in-space views. 
\begin{enumerate}
\item Our images must be accurate -- assuming that light rays travel along geodesics, there is a correct picture of what an observer inside of a given geometry would see. Our images should accurately portray this picture.
\item Real-time graphics algorithms must be very efficient in order to run at an acceptable frame rate. This is particularly important  in virtual reality -- around 90 frames per second is recommended to reduce nausea. Modern graphics cards allow for the required speed, given efficient algorithms.
\item Our algorithms must allow for a full six degrees of freedom in the position and orientation of the camera, even when the simulated geometry may not have a natural corresponding isometry. A user in a virtual reality headset can make such motions, and the view they see must react in a sensible way.
\item As much as is possible, our algorithm should be independent of the geometry being simulated. The idea here is that it should be possible to change the code in a small number of places to convert between simulations of different geometries. Compartmentalizing the code in this way will make it easier to extend it to further geometries, beyond Thurston's eight.
\item When possible, we should make our images beautiful, allowing for graphical effects including lighting, (hard and soft) shadows, reflections, fog, etc.
\end{enumerate}

Some of these goals are of course in conflict. Adding features such as shadows and reflections increases the amount of work needed to be done per frame, which can reduce the frame rate. The frame rate is also dependent on the desired screen resolution. There are many trade-offs to be made between fidelity and speed.

We use the relatively new technique of \emph{ray-marching} in our implementation. We discuss this technique and compare it with other graphics techniques in \refsec{Ray-marching}. One key feature is that the data and calculations needed to generate images for each geometry are relatively simple in comparison to other techniques, which makes it easier to write geometry independent code.

\subsection{Related work}
\label{Sec:RelatedWork}
This project owes its existence to a long history of previous work.  
It is a direct descendant of the hyperbolic ray-marching program created by Nelson, Segerman, and Woodard~\cite{woodard_github}, which itself was inspired by previous
work in $\mathbb H^3$ and $\mathbb H^2\times \mathbb E$ by  Hart, Hawksley,  Matsumoto, and Segerman \cite{NEVR1,NEVR2}, all of which aim to expand upon Weeks' \emph{Curved Spaces}~\cite{curved_spaces} which in turn is a descendant of work by Gunn, Levy and Phillips~\cite{PhillipsGunn, Geomview} and others at the Geometry Center in the 1990's.  Thurston was a driving force for much of this visualization work. He often spoke about what it would be like to be inside of a three-manifold~\cite{Thurston98}. The software SnapPy~\cite{SnapPy} was originally developed by Weeks to calculate the geometry on hyperbolic three-manifolds using Thurston's hyperbolic ideal triangulations. Concurrent with this project's development at ICERM in Fall 2019, Matthias Goerner implemented an inside view for hyperbolic manifolds within SnapPy, using a ray-tracing strategy.

Perhaps the earliest work concerned with rendering the inside-view of non-euclidean geometries is due to theoretical physicists predicting the appearance of black holes; this field goes back to the 1970's~\cite{Luminet}. 

The past few years have seen a number of independent projects building real-time simulations of inside views for the Thurston geometries, including the last three ``harder'' geometries. To our knowledge, Berger~\cite{berger, BergerLaierVelho} produced the first in-space images of all eight Thurston geometries. He uses ray-tracing, with a fourth-order Runge--Kutta method for numerical integration to approximate geodesic rays. 

The \emph{HyperRogue} project~\cite{HyperRogue}, by Kopczy\'nski and Celi\'nska-Kopczy\'nska implements all eight geometries with a triangle rasterization based strategy. They restrict the 
parts of the world that the viewer can see in order to avoid some issues with this approach that we identify in \refsec{Polygons}. For example, in certain geometries one can only see a limited distance in particular directions. They also use a fourth-order Runge--Kutta method to approximate geodesic rays, and rely in part on lookup tables for speed. Their motivation is more towards implementation for use in computer games. Here it is very useful to be able to use polygon meshes to represent the player character, enemies, and other objects in the game world. Kopczy\'nski and  Celi\'nska-Kopczy\'nska~\cite{Kopczysk} also provide a real-time ray-tracing implementation of Nil,  $\SLR$ and Sol.

Novello, Da Silva, and Velho~\cite{Velho:2020:10.20380/GI2020.42, NOVELLO2020219} share our interest in implementing virtual reality experiences. They also implement in-space views with a ray-tracing approach, tackling all of the Thurston geometries other than the product geometries. They use Euler's method for numerical integration to approximate geodesic rays for $\SLR$ and Sol. 

Other than ours, the only ray-marching approach we are aware of is due to MagmaMcFry~\cite{MagmaMcFry}, who implements $\EE^3, \HH^3$, Nil, $\SLR$, and Sol. They use a second-order Runge--Kutta method to approximate geodesic rays. 

A numerical integration approach is unavoidable in some cases, for example in generic inhomogeneous geometries~\cite{novello2020design}. These approaches can also minimize the differences in the code for different geometries. However, such algorithms must take many steps along each ray to maintain accuracy, and so may be slow. This may be acceptable when the scene is ``dense'' -- implying that few rays travel very far before hitting an object. This often happens for example, with a co-compact lattice. For scenes in which rays travel large distances we lose accuracy unless the number of steps is large, meaning that we lose rendering speed. 

We instead use explicit solutions for our geodesic rays in almost all cases. This moves the problem of accuracy versus speed to the implementation of the functions involved in the solutions. In this setting however, we have reduced the problem of understanding the long-term behavior of the geodesic flow to studying the long-term behavior of these component functions. It turns out that these functions are well-understood for the eight Thurston geometries (they are trigonometric, hyperbolic trigonometric, and Jacobi elliptic functions). 
Thus we can often take large steps along geodesics and achieve both accuracy and speed, even for objects that are distant from the viewer. 
We exploit this ability to illustrate counterintuitive, long-range behavior of geodesics in Nil and Sol~\cite{NEVR3, NEVR4}. In Appendix \ref{Sec:CompareFlowAlgorithms} we give the results of some numerical experiments comparing the performance and accuracy of Euler and Runge--Kutta numerical integration with explicit solutions in Nil and $\SLR$.

\subsection*{Acknowledgements}
This material is based in part upon work supported by the National Science Foundation under Grant No. DMS-1439786 and the Alfred P. Sloan Foundation award G-2019-11406 while the authors were in residence at the Institute for Computational and Experimental Research in Mathematics in Providence, RI, during the Illustrating Mathematics program.
The first author acknowledges support from the \emph{Centre Henri Lebesgue} ANR-11-LABX-0020-01 and the \emph{Agence Nationale de la Recherche} under Grant \emph{Dagger} ANR-16-CE40-0006-01.
The second author was supported in part by National Science Foundation grant DMR-1847172 and a Cottrell Scholars Award from the Research Corporation for Science Advancement. The third author was supported in part by National Science Foundation grant DMS-1708239.

We thank Joey Chahine for telling us about a computable means of finding area density. We thank Arnaud Ch\'eritat, Matei Coiculescu, Jason Manning, Saul Schleimer, and Rich Schwartz for enlightening discussions about the Thurston geometries at ICERM.

\section{Ray-marching}
\label{Sec:Ray-marching}

\emph{Ray-marching} is a relatively new technique to produce real-time graphics using modern GPUs~\cite{Wong}, although its roots go back to the 1980's at least~\cite{HartSandinKauffman89}. Ray-marching is similar to ray-tracing in that for each pixel of the screen, we shoot a ray from a virtual camera to determine what color the pixel should be. Unlike most ray-tracing implementations however, the objects in the world that our ray can hit are not described using polygons. Instead, we use \emph{signed distance functions}, which we describe in the following. 

\begin{definition}
\label{Def:SDF}
Let $X$ be the ambient space, and suppose that $S$ is a closed subset of $X$. We refer to $S$ as a \emph{scene}. We define the \emph{signed distance function} $\sigma\from X \to \RR$ for $S$ as follows. For a point $p \in X - S$, the function $\sigma$ returns the radius of the largest ball centered at $p$ whose interior is disjoint from $S$. For $p\in S$ the function is non-positive, and $|\sigma(p)|$ is the radius of the largest ball centered at $p$ contained in $S$.
\end{definition}
We will sometimes write $\sdf(p,S)$ for $\sigma(p)$. We often refer to a part of a scene as an \emph{object}.
As an example, suppose that $X$ is euclidean three-space, $\EE^3$, and our scene $S$ is a ball of radius $R$, centered at the origin. Then the signed distance function is
\begin{equation}
\label{Eqn:Ball}
\sigma(p) = |p| - R.
\end{equation}
Suppose that we have multiple scenes, described by signed distance functions $\sigma_i$. Then the signed distance function for the union of the scenes is $\min_i \{\sigma_i\}$. The complement of a scene is given by the negative of its signed distance function. We often draw a tiling in an inexpensive manner by deleting a ball from the center of each tile. See~\reffig{primitive cell E3} and \refrem{SphereFundDom}.  For more examples of signed distance functions in $\EE^3$, and more ways to combine signed distance functions, see~\cite{Quilez}. 

\begin{figure}[htbp]
		\centering%
		\subfloat[A tile.]{%
		\includegraphics[width=0.32\textwidth]{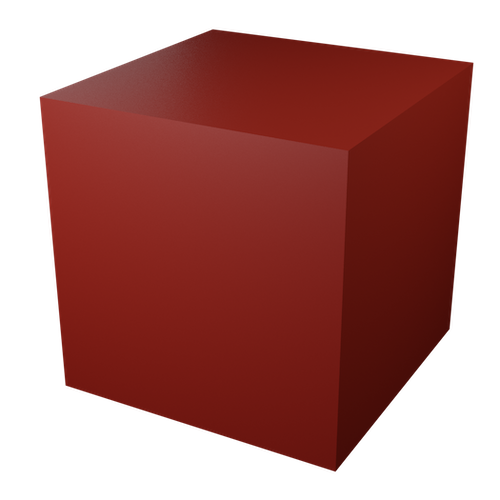}%
		\label{Fig:primitive cell E3 - fundamental domain}
		}%
		\subfloat[A ball is deleted from the center of the tile.]{%
		\includegraphics[width=0.32\textwidth]{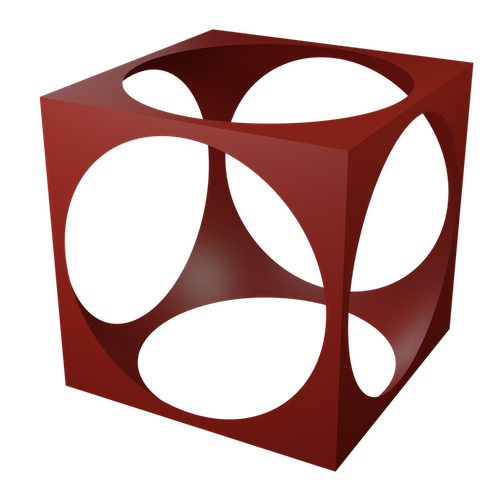}%
		\label{Fig:primitive cell E3 - primitive cell}
		}%
		\subfloat[A ball is deleted from the center and each vertex of the tile.]{%
		\includegraphics[width=0.32\textwidth]{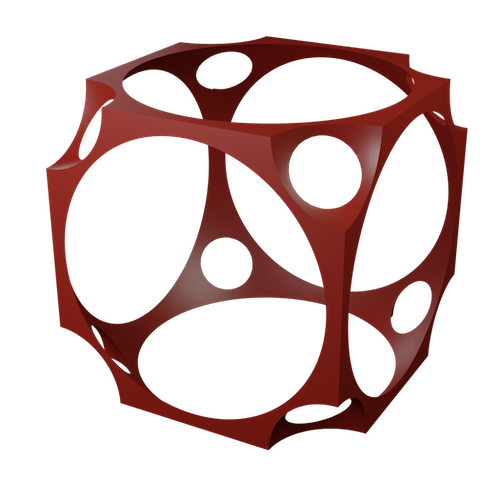}%
		\label{Fig:primitive cell E3 - advanced cell}%
		}%
		\caption{Extrinsic view of some scenes with inexpensive signed distance functions for a $\ZZ^3$-invariant tiling in $\EE^3$.}%
		\label{Fig:primitive cell E3}%
	\end{figure}

To render an image of our scene, we place a virtual camera in the space $X$ at a point $p_0$. We identify each pixel of the computer screen with a tangent vector at $p_0$, and so determine a geodesic ray for this pixel, starting at $p_0$. To color the pixel, we must work out what part of the scene the ray hits. The algorithm is illustrated in \reffig{Ray-march}. We start at $p_0$, the position of the camera, as shown in \reffig{Ray-march1}. We assume that $p_0$ is not inside the scene. We evaluate the signed distance function $\sigma$ at $p_0$. Since no part of the scene is within $\sigma(p_0)$ of $p_0$, we can safely march along our ray by a distance of $\sigma(p_0)$ without hitting the scene. We call the resulting point $p_1$. We can then safely march forward again by $\sigma(p_1)$ to reach $p_2$. We repeat this procedure until either we reach a maximum number of iterations, or we reach a maximum distance, or the signed distance function evaluates to a sufficiently small threshold value, $\varepsilon$ say. In the first two cases we color the pixel by some background color. In the third case (as shown in \reffig{Ray-marchFinal}) we declare that we have hit the scene. 

In the case that we hit the scene, we may then choose a color for the pixel based on which part of the scene we hit, apply a texture, and/or apply various lighting techniques, for example the Phong reflection model~\cite{Phong}. 
Note that this model requires the normal vector to the surface at the point our ray hits; this is easily approximated using the gradient of the signed distance function.

\begin{figure}[htbp]
\centering
\subfloat[]{
\labellist
\scriptsize\hair 2pt
\pinlabel $p_0$ [bl] at 160 210
\endlabellist
\includegraphics[width=0.45\textwidth]{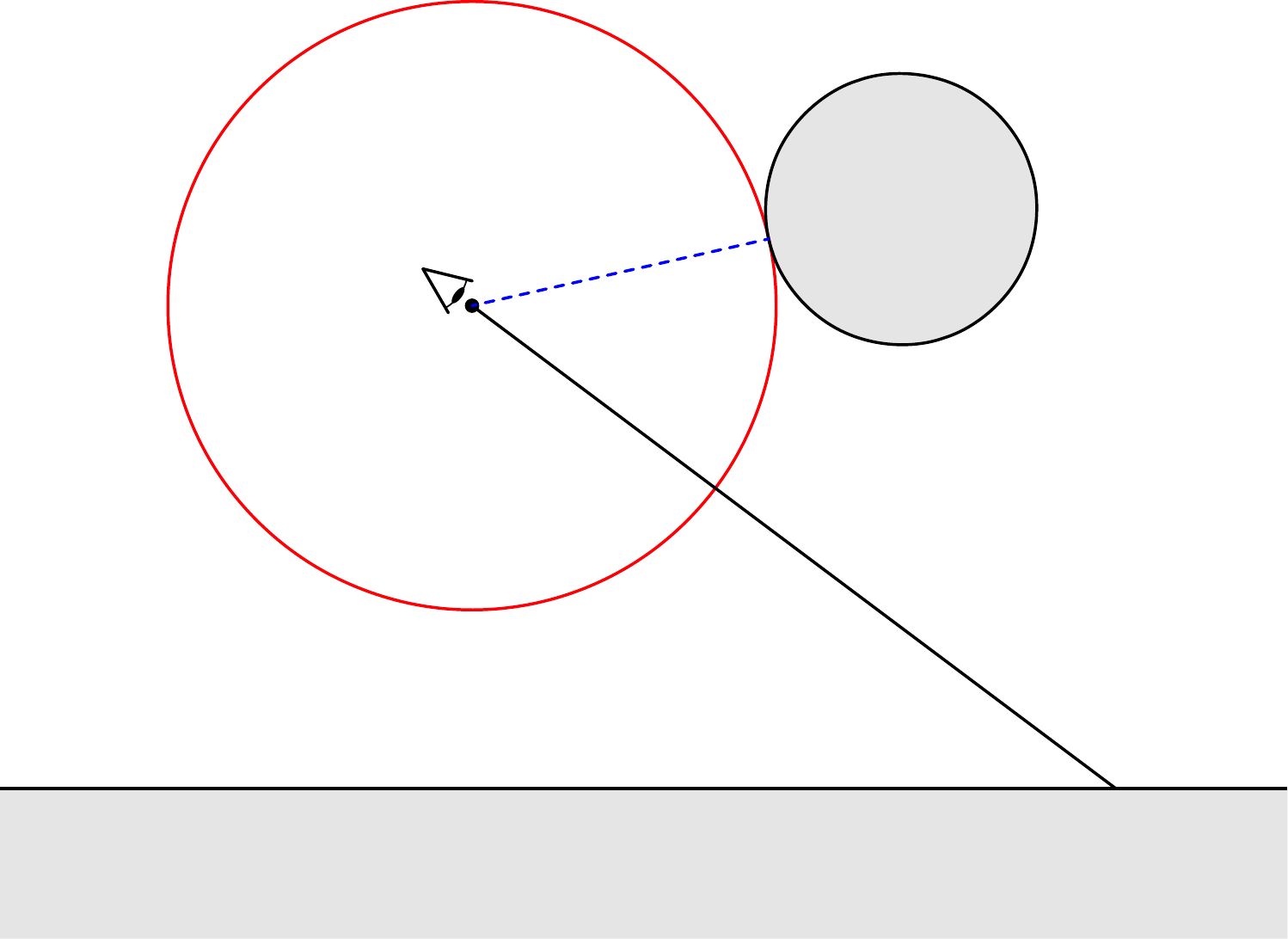}
\label{Fig:Ray-march1}
}
\quad
\subfloat[]{
\labellist
\scriptsize\hair 2pt
\pinlabel $p_0$ [bl] at 160 210
\pinlabel $p_1$ [b] at 237 160
\endlabellist
\includegraphics[width=0.45\textwidth]{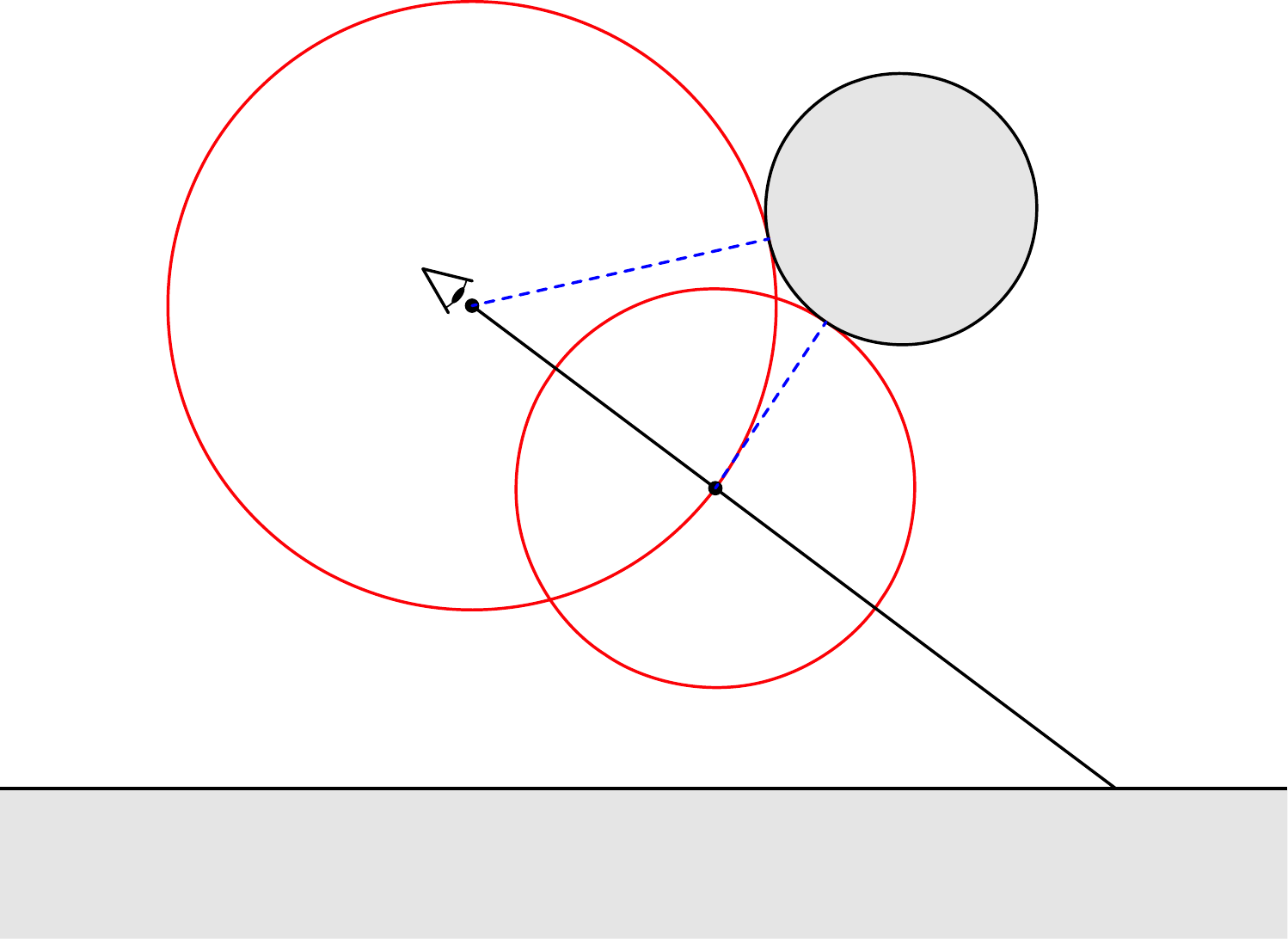}
\label{Fig:Ray-march2}
}

\subfloat[]{
\labellist
\scriptsize\hair 2pt
\pinlabel $p_0$ [bl] at 160 210
\pinlabel $p_1$ [b] at 237 160
\pinlabel $p_2$ [b] at 290 120
\endlabellist
\includegraphics[width=0.45\textwidth]{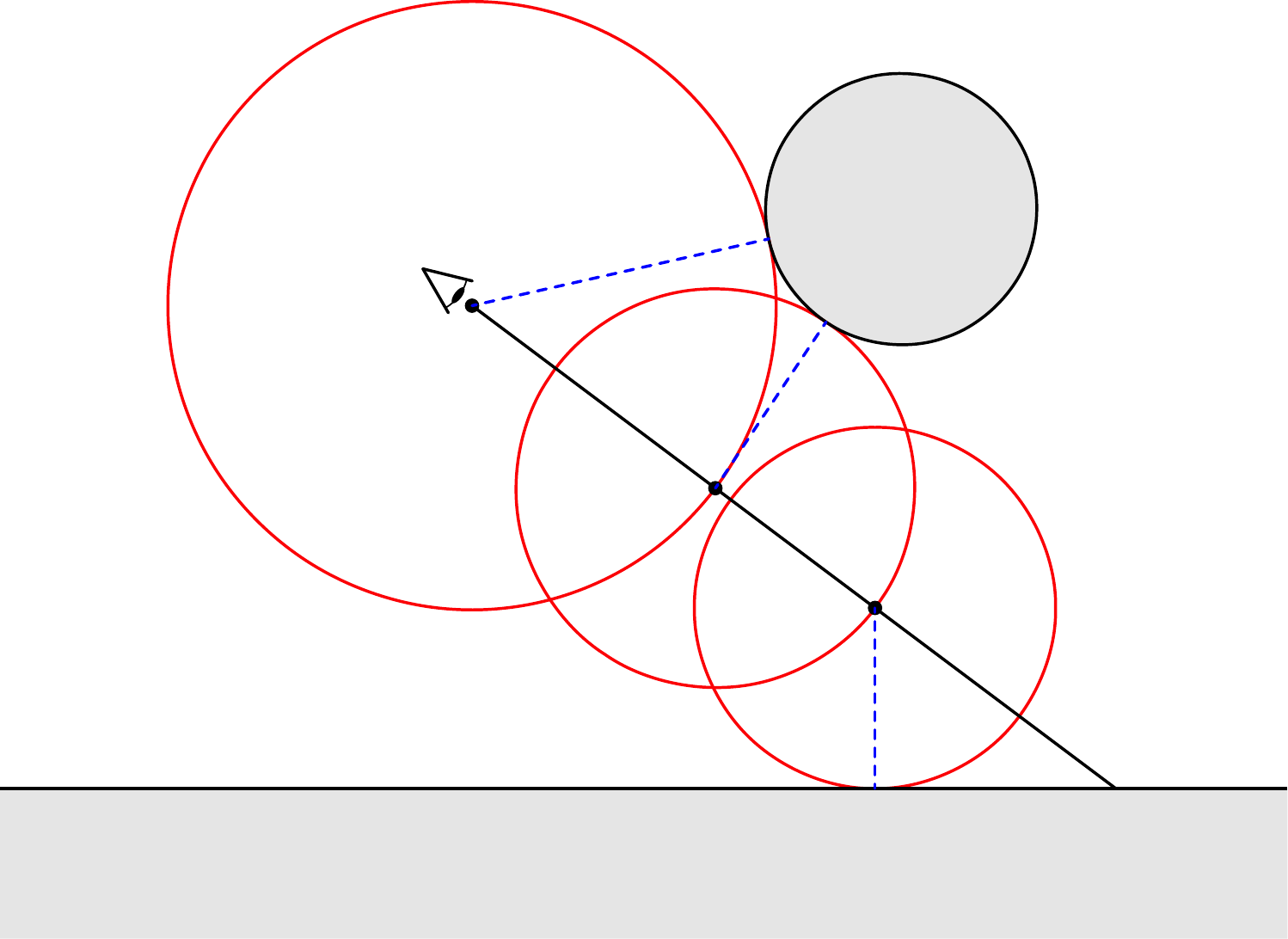}
\label{Fig:Ray-march3}
}
\quad
\subfloat[]{
\labellist
\scriptsize\hair 2pt
\pinlabel $p_0$ [bl] at 160 210
\pinlabel $p_1$ [b] at 237 160
\pinlabel $p_2$ [b] at 290 120
\endlabellist
\includegraphics[width=0.45\textwidth]{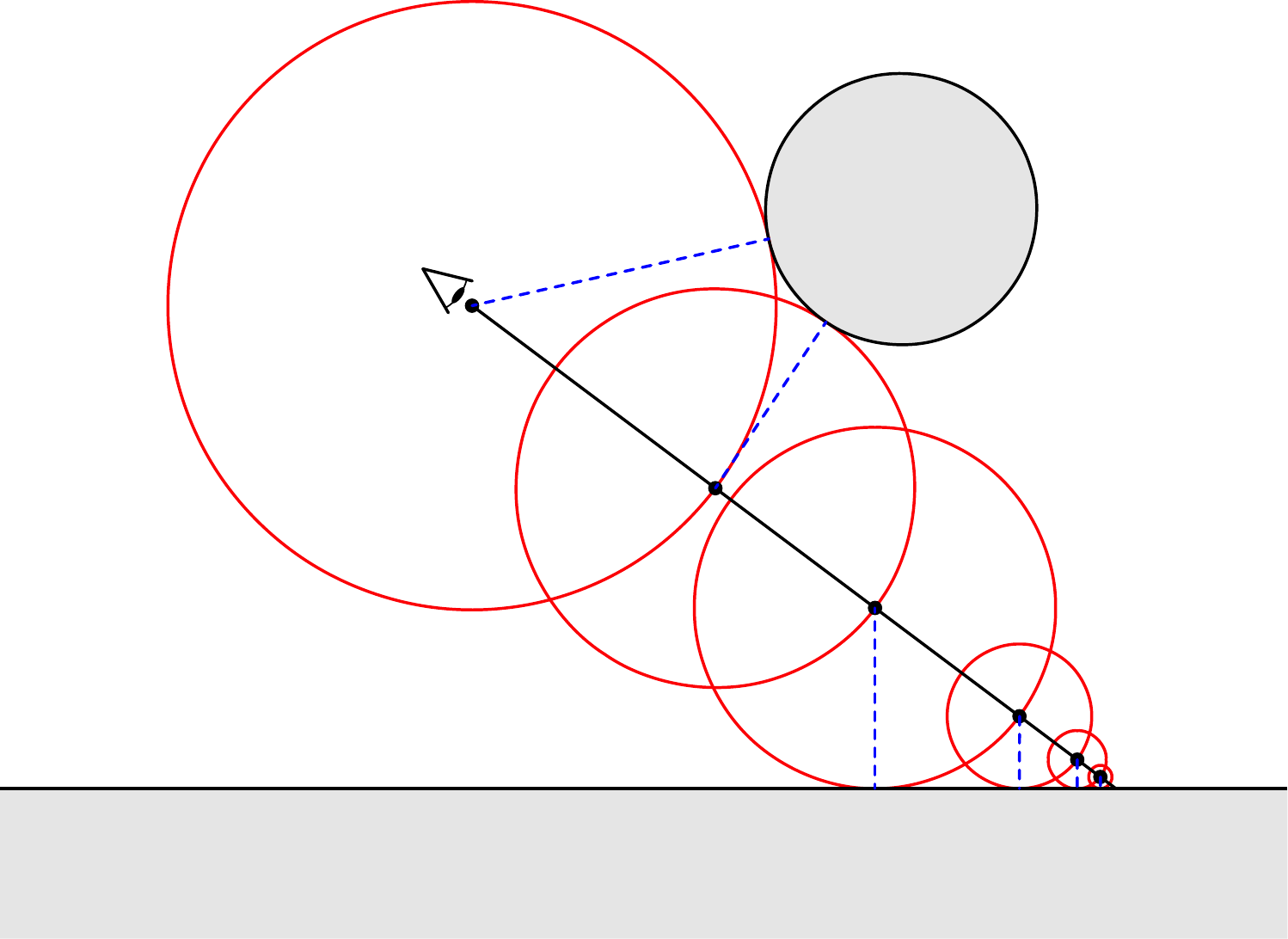}
\label{Fig:Ray-marchFinal}
}
\caption{Ray-marching to find the point at which a ray hits an object, for a scene in $\EE^2$ consisting of a disk and a half-plane.}
\label{Fig:Ray-march}
\end{figure}
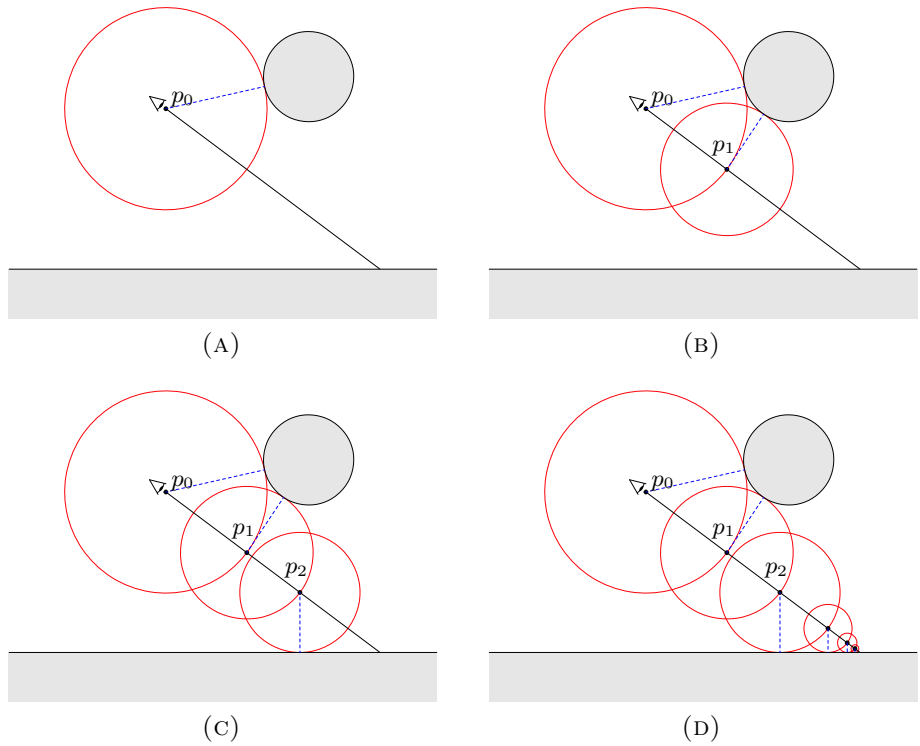

\subsection{Geometric convergence}
\label{Sec:GeometricConvergence}
A concern one might have over the ray-marching algorithm is the potentially large number of steps taken before we are close enough to the scene to declare that we have hit it. Indeed, functions called in the innermost loop of the algorithm must be made as efficient as possible. However, the number of steps used is generally not prohibitive. Suppose that our scene $S$ has a smooth boundary. In this case, when we are close enough to $S$ its boundary may be approximated by a plane $P$. If our ray continues to approach $P$, then we converge to it as a geometric series, see \reffig{Ray-marchGeometricSeries}. The base of the exponent $\lambda$ depends on the angle of incidence of the ray, approaching the worst case of $\lambda = 1$ as the ray becomes tangent to $S$. 

\begin{figure}[htbp]
\centering
\labellist
\small\hair 2pt
\pinlabel $d$ [b] at 140 15
\pinlabel $d\lambda$ [b] at 251 15
\pinlabel $d\lambda^2$ [b] at 295 15
\pinlabel $d\lambda^3$ [b] at 330 15
\endlabellist
\includegraphics[width=0.65\textwidth]{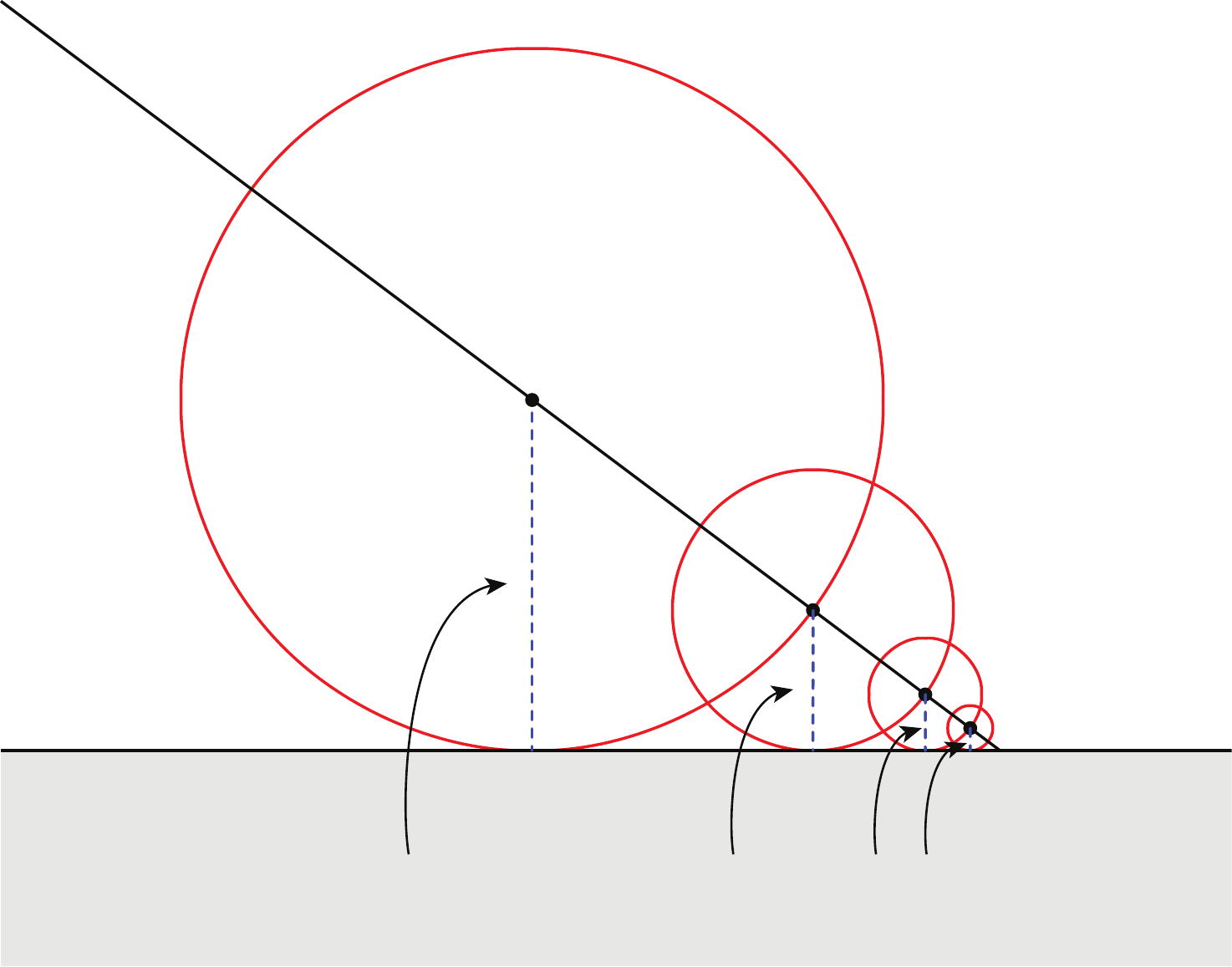}
\caption{Typically, convergence of a ray marching into an object is geometric. If our first distance from the object is $d$, then subsequent distances follow a geometric sequence with the base of the exponent some number $\lambda < 1$.}
\label{Fig:Ray-marchGeometricSeries}
\end{figure}
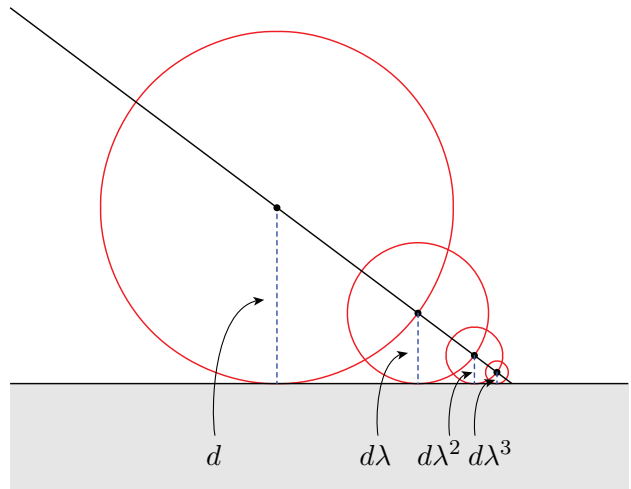

\begin{remark}
If the maximum number of steps we allow before giving up is too small, then we may erroneously color pixels with the background color whose rays would eventually hit an object. This will often be most visible around the outer edges of an object in the scene, as these rays are the closest to tangent. These rays spend many steps moving a small distance close the object. Thus, they may run out of iterations before converging.
\end{remark}

\subsection{Distance underestimators}
\label{Sec:DistanceUnderestimators}
The signed distance function for a scene may be difficult or expensive to calculate. In these cases we may wish to replace it with an easier to calculate approximation. 
\begin{definition}
Suppose that $\sigma\from X \to \RR$ is the signed distance function for a scene $S$. We say that a function $\sigma'\from X \to \RR$
is a \emph{distance underestimator} if 
\begin{enumerate}
\item The signs of $\sigma'(p)$ and $\sigma(p)$ are the same for all points $p \in X$, 
\item $ |\sigma'(p)| \leq |\sigma(p)|$ for all $p\in X$, and 
\item If $\{p_1, p_2, \ldots\}$ is a sequence of points in $X$ such that $\lim \sigma'(p_n) = 0$, then $\lim \sigma(p_n) = 0$. \qedhere
\end{enumerate}
\end{definition}

We do not require that $\sigma'$ is continuous, but the second and third conditions here imply that a distance underestimator vanishes only on the boundary of $S$.

\begin{lemma}
\label{Lem:Underestimator}
When ray-marching with a distance underestimator $\sigma'$ in place of a signed distance function $\sigma$, we limit to the same point as when using $\sigma$. 
\end{lemma}

This result implies that a distance underestimator will give us essentially the same images as the signed distance function, given enough iterations and a small enough threshold $\varepsilon$. If a distance underestimator is significantly easier to compute than the signed distance function then trading an increased number of iterations for improved speed of computation can be advantageous. See Sections~\ref{Sec:distance underestimator Nil} and \ref{Sec:distance underestimator SL2} for examples of distance underestimators.

\begin{proof}[Proof of \reflem{Underestimator}]
Consider a ray $\gamma$ starting at a point $p \notin S$. Suppose that $\gamma$ first meets the scene $S$ at the point $q$. Using the distance underestimator $\sigma'$, we march through a sequence of points $p=p_1, p_2, \ldots$ Consider the distances $d_n = \dist_\gamma(p_n, q)$, measured along the ray $\gamma$ from $p_n$ to $q$. 
By conditions (1) and (2), we know that the sequence $\{d_n\}$ is a non-negative non-increasing sequence. Thus $\{d_n\}$ converges, and so the sequence of points $\{p_n\}$, converges. Thus the distances $\dist_\gamma(p_n, p_{n+1})$ must go to zero. These are the distances we march along the ray, using the distance underestimator $\sigma'$, so $\dist_\gamma(p_n, p_{n+1}) = \sigma'(p_n)$. Therefore $\lim \sigma'(p_n) = 0$. By condition (3), $\lim \sigma(p_n) = 0$, and so $\lim p_n = q$.
\end{proof}

In practice we want $\sigma'$ and $\sigma$ to be ``coarsely the same''. In particular, to get condition (3), we want $|\sigma'(p)|$ to be bounded below by some function of $|\sigma(p)|$. This also allows us to control how many extra iterations are needed in ray-marching with a distance underestimator.

Any real-world implementation cannot go all the way to the limit point $q$ and instead stops at some sufficiently small value, $\varepsilon$. Thus, a distance underestimator should not return a value smaller than $\varepsilon$ unless the signed distance function is also small.

\subsection{Advantages of ray-marching in non-euclidean geometries}

Ray-marching is an attractive technique in euclidean geometry, in part because of the simplicity of its implementation. This is also true for non-euclidean geometries. Here we discuss some alternative techniques.

\subsubsection{Z-buffer triangle rasterization}
\label{Sec:Polygons}
Real-time graphics in euclidean geometry are usually rendered using \emph{z-buffer triangle rasterization}. In this technique, objects in the scene are represented by polygon meshes. A projection matrix maps each triangle of a mesh onto the plane of the virtual camera's screen. For each pixel $P$, we look at the triangles whose projections contain the center of $P$. Of these triangles, the one closest to the camera determines the color of $P$.

This works well for the isotropic Thurston geometries, $\EE^3, S^3$ and $\HH^3$, in particular because geodesics in these geometries are straight lines in their projective models, see~\cite{Weeks:2002}.
Jeff Weeks uses these in his \emph{Curved Spaces} software~\cite{curved_spaces}. There is one complication with $S^3$ here, in that a single object is visible in two different directions: the two directions along the great circle containing the camera and the object. This means that each object must be projected twice. This is acceptable for $S^3$. In Nil, Sol, and $\SLR$, a single object can be visible from the camera in many directions, with no  uniform bound on the number of such directions. Even worse, in $S^2 \times \EE$ a single object can be visible in infinitely many directions from a single camera position. 

The projection matrix used in triangle rasterization implements the inverse of the exponential map. In the cases listed above, the exponential map is not one-to-one. This is not a problem for ray-tracing and ray-marching, which both use the forward direction of the exponential map instead.

\subsubsection{Ray-tracing}
Ray-tracing is very similar to ray-marching, with the difference being in how we determine where in the scene a ray hits. In many applications the objects in the scene are described by polygon meshes, as in triangle rasterization. The algorithm checks for intersection between the ray and the polygons of the mesh. To make this efficient for (euclidean) scenes with a large number of polygons, much effort is put into checking as few triangles for collision as possible, even though each individual check is inexpensive. However, objects described by simple equations such as spheres and other conics can also be used: all that is needed is a way to check whether or not a ray intersects the object, and at what distance along the ray. The distance is used to decide which object is closest to the camera and so should be drawn. For a conic in euclidean space for example, this check and distance may be calculated by solving a quadratic equation. 

One advantage of ray-tracing over ray-marching is that ray-tracing is well suited to rendering objects given by polygon meshes. It therefore has access to decades of development in polygon modeling techniques and rendering efficiency for polygonal models. On the other hand, depending on the geometry, checking for intersection between a ray and an object may be difficult. In place of this check in ray-tracing, for ray-marching we only need a signed distance function (or distance underestimator). If for example we make our scene from balls, then we only need to calculate distances between points.

\subsection{Accuracy}
\label{Sec:Accuracy}

One of our goals in this project is to be able to render features accurately, even at long distances. 
We identify two potential sources of error here. 

\subsubsection{Floating point representation of number}
\label{Sec:FloatingPoint}
First, the representation of real numbers by floating point numbers is necessarily inaccurate. This can be a problem in a number of ways, whether one is ray-marching, ray-tracing, or using polygon rasterizing methods.

\begin{enumerate}

\item \label{Itm:ExpCoords}
In certain models, the coordinates of points grow exponentially with distance in the geometry, and floating point numbers quickly lose precision. In particular, this causes problems when rendering objects that are far from the camera. Of the eight Thurston geometries, this is an issue in $\HH^3, \HH^2 \times \EE$, Sol, and $\SLR$. This can be mitigated by the choice of model~\cite{FloydWeberWeeks}. Even without exponential growth in coordinates, floating point numbers cannot exactly represent geometric data.
\item \label{Itm:RemovableSingularities}
In certain regimes, a formula may be unstable. For example, the formula $(1 - \cos(t))/t^2$ approaches $1/2$ as $t$ approaches zero. However, the available precision in the floating point representation of $(1 - \cos(t))$ near $t=0$ is not very good in comparison to the precision of $t^2$. In such a regime, it is better to use a different representation of the formula. Here for example, we will get much better results by using an asymptotic expansion, say $1/2 - t^2/24 + \cdots$.


\end{enumerate}

\subsubsection{Accumulation of errors}
\label{Sec:Accumulation}

Any iterative algorithm that takes the result from the previous step as the input for the next step may accumulate errors. These errors may come from lack of precision due to floating point representations as described above. They may also come from limitations in the methods used to calculate geodesic flow. As mentioned at the end of \refsec{RelatedWork}, to remove this second source of error we avoid the numerical integration approach whenever possible, preferring explicit solutions. 

\section{General implementation details}
\label{Sec:General}

As mentioned in \refsec{Goals}, one of our goals in this project is to make as much of our code as possible independent of the geometry being simulated. Following this goal, in the next few sections we describe components needed for our simulations that are shared across geometries. Many of these apply to all eight Thurston geometries.
However, it is also useful to discuss strategies for tackling smaller collections of geometries with certain geometric or group theoretic features. Thus to begin, we provide a second grouping of the Thurston geometries into overlapping families, distinct from our first grouping by method of construction in \refsec{ThurstonGeometries}.

Consider the following properties:

\begin{enumerate}
\item The geodesic flow is achieved by isometries. That is, every geodesic is the orbit of a point under a one-parameter subgroup.
\item The projective model has straight-line geodesics. Each of the Thurston geometries (up to covers) has a model with  $X\subset\mathbb{RP}^3$ and $G<\mathsf{GL}(4;\RR)$.  With this property, the geodesics of $(G,X)$ are projective lines in this model.
\item The group $G$ has a normal subgroup whose action is free and transitive on $X$.
\end{enumerate}

Property (1) implies that parallel transport is achievable directly via elements of $G$. Property (2) implies that totally geodesic surfaces are planes in the projective model, which makes testing membership in polyhedral domains (for example, Dirichlet domains) efficient.
Property (3) allows us to canonically identify tangent spaces at distinct points of $X$. This allows us to reduce certain difficult calculations (for example, the geodesic flow) to differential equations in a fixed tangent space.

The constant curvature and product geometries all have properties (1) and (2), while Nil, Sol, and $\SLR$ have neither.
These properties are very useful in practice, so we call the five geometries possessing them  the \emph{easier geometries}, while Nil, Sol, and $\SLR$ are the \emph{harder geometries}.
However, Nil, Sol, and $\SLR$ do have property (3) (along with $\EE^3$ and $S^3$). See \reffig{ThurstonGeos2}.

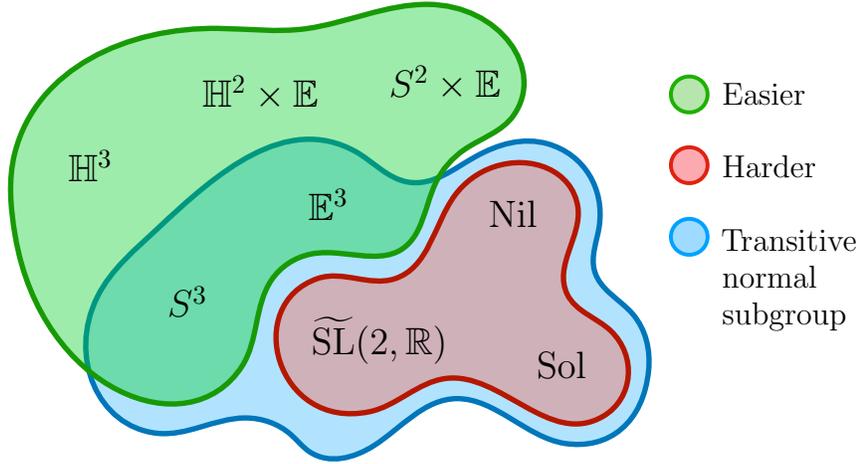
\begin{figure}
\centering
\labellist
\large\hair 2pt
\pinlabel $\HH^2\times\EE$ at 520 760
\pinlabel $S^2\times\EE$ at 900 780
\pinlabel $\HH^3$ at 170 610
\pinlabel $\EE^3$ at 660 530
\pinlabel Nil at 1040 510
\pinlabel $S^3$ at 370 330
\pinlabel $\SLR$ at 765 250
\pinlabel Sol at 1140 200

\normalsize
\pinlabel Easier [l] at 1460 758
\pinlabel Harder [l] at 1460 609
\pinlabel {\parbox{1.0in}{Transitive\\ normal\\ subgroup}} [l] at 1460 375
\endlabellist
\includegraphics[width=0.95\textwidth]{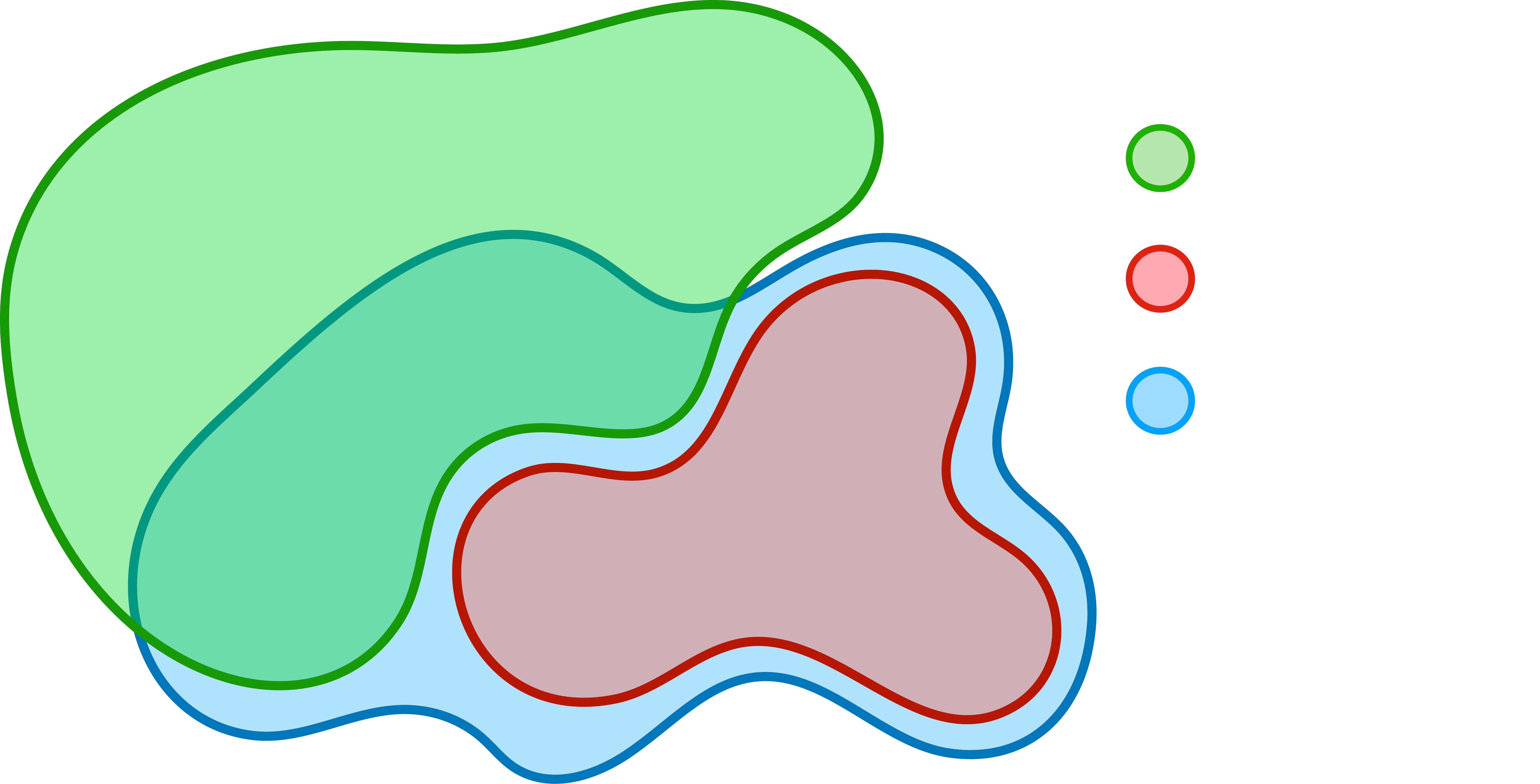}
\caption{The Thurston geometries, grouped into useful categories for our implementation.}
\label{Fig:ThurstonGeos2}
\end{figure}

\subsection{Notation}
Recall that the underlying space $X$ of a Thurston geometry $(G,X)$ is both connected and simply connected, and can be equipped with a $G$-invariant riemannian metric $ds^2$.
We fix a base point $o \in X$, which we call the \emph{origin} of the space $X$. We denote by $K$ the stabilizer of $o$ in $G$. Thus $X$ is isomorphic to $G/K$.

\medskip
\subsection{Geodesic flow}
\label{Sec:GeodesicFlow}
In order to follow light rays, we need to understand geodesics in $X$. Moreover, since we want to march along our geodesics by specified distances, they must be given by arc length parametrizations.
These are paths $\gamma \colon \RR \to X$ such that 
\begin{equation*}
	\nabla_{\dot \gamma(t)}\dot \gamma(t) = 0, \quad \forall t \in \RR.
\end{equation*}
where $\nabla$ is the Levi-Civita connection on $(X, ds^2)$.
This condition corresponds to a five-dimensional second-order (non-linear) differential system.
In some cases (for example, $\EE^3$, $S^3$ or $\HH^3$) these systems are comparatively easy to solve. See \reftab{GeometryDetails}.
Other geometries such as Nil, Sol, and $\SLR$ are more subtle.
Next, we describe a method to split this problem into two first-order differential systems.
This strategy has both practical and theoretical advantages that we will discuss later.

\subsubsection{Grayson}
\label{Sec:GeodesicFlow - Grayson}
We follow here an idea of Grayson~\cite{Grayson:1983aa}. 
Assume that $G$ contains a normal subgroup $G_0$ which acts freely and transitively on $X$.
The group $G_0$ provides a preferred way to compare the tangent space at different points of $X$.
For every $x \in X$ we denote by $L_x$ the (unique) isometry in $G_0$ sending the origin $o$ to $x$.
Let $\gamma \colon \RR \to X$ be a geodesic of $X$.
For every $t \in \RR$, we denote by $u(t) \in T_oX$ the vector such that
\begin{equation}
\label{Eqn:GraysonMethodPullBack}
	\dot \gamma(t) = d_o L_{\gamma(t)}u(t)
\end{equation}
It follows from the construction that $u$ is a path on the unit sphere of the tangent space $T_oX$.
Observe that once $u$ is known, the trajectory $\gamma$ is the solution of the first-order differential equation given by \refeqn{GraysonMethodPullBack}.

Since geodesics are invariant under isometries, the path $u$ satisfies a two-dimensional first-order autonomous differential system 
\begin{equation}
\label{Eqn:GraysonMethodFlowSphere}
	\dot u = F(u)
\end{equation}
where $F$ does not depend on $\gamma$.
In practice, \refeqn{GraysonMethodFlowSphere} is often straightforward to solve, see for example \refsec{Nil} and \refsec{SLR}.
The corresponding flow on the unit sphere also provides qualitative information on the geodesic flow \cite{Coiculescu:2019aa}.

Let $h \in K$ be an isometry on $X$ fixing $o$. Observe that the path
\begin{equation*}
u' = d_oh \circ u
\end{equation*}
is also a solution of \refeqn{GraysonMethodFlowSphere}.
Indeed, consider the geodesic $\gamma' \colon \RR \to X$ defined by $\gamma' =  h \circ \gamma$.
Since $G_o$ is a normal subgroup of $G$, for every $x \in X$ we have
\begin{equation*}
	h \circ L_x \circ h^{-1} = L_{hx}.
\end{equation*}
It follows that 
\begin{equation*}
	\dot \gamma'(t) = d_o L_{\gamma'(t)}u'(t), \quad \forall t \in \RR.
\end{equation*}
This proves our claim.
Thanks to this observation we can take advantage of the symmetries of $X$ to reduce the amount of computation needed to solve \refeqn{GraysonMethodFlowSphere}. 
See for example Sections \ref{Sec:nil geo flow} and \ref{Sec:flow sl2}.

\medskip
\subsection{Position and facing}
\label{Sec:PositionFacing}

For the moment, we will think of the observer as a single camera, based at a point of $X$. In \refsec{Stereoscopic}, we will consider an observer with stereoscopic vision.

In order to render the scene viewed by such an observer, we need to know its \emph{position}, given by a point $p \in X$, and its orientation in the space (which we call its \emph{facing}).
The latter is represented by an orthonormal frame $f = (f_1, f_2, f_3)$ of the tangent space $T_pX$.
We adopt the following convention: from the viewpoint of the observer,
\begin{itemize}  
	\item $f_1$ points to the right
	\item $f_2$ points upward
	\item $f_3$ points backward.
\end{itemize}
See \reffig{FrameScreen}.

\begin{figure}[htbp]
\centering
\vspace{10pt}
\labellist
\small\hair 2pt
\pinlabel $f_1$ [l] at 76 41
\pinlabel $f_2$ [b] at 52 102
\pinlabel $f_3$ [r] at 0 45
\endlabellist
\includegraphics[width=0.65\textwidth]{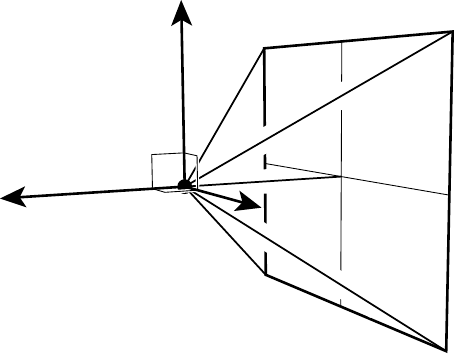}
\caption{The initial tangent vector is of the form $s f_1 + t f_2 - f_3$, where $s$ and $t$ are coordinates on the screen. }
\label{Fig:FrameScreen}
\end{figure}

Let $\calO X$ be the bundle of all orthonormal frames on $X$.
We fix once and for all a reference frame $e = (e_1,e_2, e_3)$ at the origin $o$.
This provides an identification of $\calO_oX$, the space of orthonormal frames at $o$, with ${\rm O}(3)$.
In particular, this induces an embedding of the stabilizer of the origin, $K$, into ${\rm O}(3)$, given by $k \mapsto d_ok$.

\subsubsection{Parametrizing the frame bundle}
Our goal is to make simulations of Thurston geometries to better understand their properties. Our audience in this endeavor consists of entities with primary experience in $\EE^3$, as far as we are aware. Thus, our audience will naturally expect to be able to move in any direction, and orient their view in any way they wish. Thus, the user should be able to move and rotate to achieve any element of the frame bundle $\calO X$ (while preserving their orientation class). Therefore the data we use to record the position and facing of the user must map \emph{onto} $\calO X$.

When $X$ is isotropic, $G$ acts transitively on the frame bundle $\calO X$. In this case one could use an element of $G$ to record this data. However, when $X$ is anisotropic, this action is not transitive. For example, if $X$ is one of the product geometries $S^2 \times \EE$ or $\HH^2 \times \EE$, there is no isometry that rotates in way that breaks the product structure. 

Thus, we parametrize $\calO X$ by the following map.

\begin{equation*}
	\begin{array}{ccc}
		G \times {\rm O}(3) &  \to & \mathcal O(X) \\
		(g,m) & \mapsto & d_og \circ m  (e)
	\end{array}
\end{equation*}

\noindent Since the action of $G$ on $X$ is transitive, there is an element $g$ taking $o$ to any given point $p = g o$. 
The map $d_o g$ sends $T_oX$ to $T_pX$. By varying $m$, we can send the reference frame $e$ to any frame in $T_pX$. Thus, the map is onto.

The group $G$ acts on the left on $G\times {\rm O}(3)$ by multiplication of the first factor so that the map $G \times {\rm O}(3)  \to  \mathcal O(X)$ is $G$-equivariant.
Note that the stabilizer $K$ of the origin $o$, also acts on the right on $G\times {\rm O}(3)$ as follows: for every $(g,m) \in G \times {\rm O}(3)$ and for every $k \in K$ we have 
\begin{equation*}
	(g,m) \cdot k = \left(gk, d_ok^{-1} \circ m\right).
\end{equation*}
This action commutes with the left action of $G$.
Moreover the application $G \times {\rm O}(3)  \to  \calO X$ above induces a $G$-equivariant bijection from the quotient $(G\times {\rm O}(3))/K$ to $\calO X$.

\subsubsection{Using a transitive normal subgroup}
For geometries with a transitive normal subgroup $G_0<G$ of isometries, there is a natural section of the frame bundle $X\to\mathcal{O}X$ given by the $G_0$-orbit of the reference frame $e$ at the origin.
Using this frame, we can encode unit tangent vectors in $T_pX$ by points of the unit sphere of $\RR^3$. The coordinates needed to describe these unit tangent vectors are thus uniformly bounded at all points $p\in X$.
This choice of representation helps reduce numerical errors, for example its implementation in Sol.

\medskip
\subsection{Moving in the space} 
\label{Sec:moving in the space}
Using the parameterization above, a pair $(g,m)\in G\times O(3)$ specifies a location $p\in X$ of the user, and a frame $f$ in $T_pX$. This provides the necessary data to orient the user's virtual camera within the space.  
To produce a real-time simulation, we need a means of converting user input into this form.

Assume that at the current frame, the virtual camera is at a point $p \in X$.
At each frame of the simulation, the virtual reality system records the position and facing of the headset in the play area, which is (very well) approximated as a subset of $\EE^3$.  We interpret the change in position between this frame and the next as a tangent vector $v \in T_p X \homeo \EE^3$, given by coordinates in the local frame $f = (f_1,f_2,f_3)$ representing the facing of the observer. Alternatively, keyboard input can provide the same information.

\begin{remark}
There is a choice to be made here in the relationship between the distance moved in the real world and the magnitude of the vector $v$. In our implementation, by default one meter in the real world corresponds to one unit in the virtual world. One may wish to change this relationship by a scaling factor to, for example, vary the perceived effects of curvature in $\HH^3$~\cite{SteveBlog}.
\end{remark}

We move the observer along the geodesic $\gamma \colon \RR \to X$ such that $\gamma(0) = p$ and $\dot \gamma(0) = v$.
In addition, we update the facing of the observer using parallel transport. 
Parallel transport along $\gamma$ can be seen as a collection of orientation-preserving isometries
\begin{equation*}
	T(t) \colon T_{\gamma(0)} X \to T_{\gamma(t)}X
\end{equation*}
such that 
\begin{equation}
\label{Eqn:ParallelTransportEqn}
	\nabla_{\dot \gamma(t)} T(t) = 0, \quad \forall t \in \RR.
\end{equation}

In the easier geometries ($\EE^3, S^3, \HH^3, S^2 \times \EE,$ and $\HH^2 \times \EE$), for each geodesic $\gamma$ through a point $p$, there is a one-parameter subgroup $\{g(t)\} \subset G$ such that $\gamma(t) = g(t)p$. In these cases, the parallel transport operator is $T(t) = d_p g(t)$. 

\subsubsection{Using a transitive normal subgroup} 
\label{Sec:Parallel transport - Grayson}
In Nil, Sol, and $\SLR$, we do not have the above one-parameter subgroup. Instead,
in order to compute the path of isometries $t \to T(t)$ we again use Grayson's method.
Assume as above that
 $G_0$ is a connected normal subgroup of $G$ acting freely and transitively on $X$.
Define $u \colon \RR \to T_oX$ by the relation
\begin{equation*}
	\dot \gamma(t) = d_o L_{\gamma(t)}u(t)
\end{equation*}
where $L_p$ is the unique isometry of $G_0$ sending $o$ to $p$.
Similarly, we define a path $Q \colon \RR \to \rm SO(3)$ by letting
\begin{equation}
\label{Eqn:ParallelTransportPullBack}
	T(t) \circ d_oL_{\gamma(0)} = d_o L_{\gamma(t)} \circ Q(t) 
\end{equation}
It turns out that for each of our harder geometries, $Q$ satisfies a linear differential equation of the form 
\begin{equation}
\label{Eqn:ParallelTransportLinEq}
	\dot Q + B(u)Q = 0
\end{equation}
where $B$ is skew-symmetric matrix which only depends on $u$ (and not on $\gamma$) and with initial condition $Q(0) = {\rm Id}$.
To solve \refeqn{ParallelTransportLinEq} we use the following observation.
By definition of parallel transport, for every $t \in \RR$, we have $T(t) \dot\gamma(0) = \dot \gamma(t)$, hence 
\begin{equation*}
	Q(t)u(0) = u(t).
\end{equation*}
Fix now an arbitrary vector $e_0 \in \RR^3$ and a path $R \colon \RR \to {\rm SO}(3)$ such that $R(t)u(t) = e_0$, for every $t \in \RR$.
Then 
\begin{equation*}
	S(t) = R(t)Q(t)R(0)^{-1}
\end{equation*}
is a rotation of angle $\theta(t)$ around $\RR e_0$.
Hence, in order to compute $Q$, and thus $T$, it suffices to know the value of the angle $\theta$.
To that end, we substitute $Q(t) = R(t)^{-1}S(t)R(0)$ into \refeqn{ParallelTransportLinEq} and obtain a first-order differential equation on $\theta$ that we solve.
This strategy gives an effective way to compute the parallel-transport operator.

Assume that $k \in K$ is an isometry of $X$ fixing $o$ and let $u' = d_ok \circ u$.
We observed previously that $u'$ is also a solution of \refeqn{GraysonMethodFlowSphere}.
With the same kind of computation we get that $Q'(t) = d_ok \circ Q(t) \circ d_ok^{-1}$ is a solution of 
\begin{equation*}
	\dot Q' + B(u')Q' = 0
\end{equation*}
Again we can use the symmetries of $X$ to reduce the amount of computation needed to solve \refeqn{ParallelTransportLinEq}.

During a motion it is convenient to use the pulled-back parallel-transport operator $Q$ to update the position and facing.
Recall that we store the position and facing of the observer as a pair $(g,m) \in G \times {\rm O}(3)$.
At time $t = 0$ the observer is at the point $\gamma(0) = go$ where $g = L_{\gamma(0)}$.
Its facing is given by the frame 
\begin{equation*}
	f = d_o g \circ m(e) = d_o L_{\gamma(0)} \circ m (e)
\end{equation*}
After moving along the geodesic $\gamma$ for time $t$ the observer reaches the point $\gamma(t)$.
The observer's new facing corresponds to the frame 
\begin{equation*}
	f' = T(t) f = T(t) \circ d_o L_{\gamma(0)} \circ m (e).
\end{equation*}
By the definition of $Q$, we get
\begin{equation*}
	f' =  d_o L_{\gamma(t)} \circ Q(t) \circ m (e).
\end{equation*}
Hence the position and facing of the observer after time $t$ is given by the pair $(L_{\gamma(t)}, Q(t)m)$.

\subsection{Rendering an image from a fixed location}
\label{Sec:render fixed image}
Assume that the position and the facing of the observer is given as pair $(g,m) \in G \times {\rm O}(3)$.
In order to render what the observer would see, we proceed as follows. Let $p$ be the point obtained by applying $g$ to the origin $o$. Recall that the observer is looking in the direction $-f_3$, where $f = (f_1, f_2, f_3)$ is the frame $f = me$. The set of vectors $u \in T_pX$ such that $\left<u, f_3\right> = -1$
defines an affine plane $P$ in $T_pX$.
We identify the screen of the computer with a rectangle in $P$ centered at $-f_3$. See \reffig{FrameScreen}.
The exact size of the rectangle is computed in terms of the field of view of the observer.
For each vector $u \in T_pX$ in this rectangle, we follow (using the ray-marching algorithm) the geodesic starting at $p$ in the direction of $u$ (or more precisely the unit vector with the same direction) until it hits an object.
We color the corresponding pixel on the screen with the color of this object, or more realistically, using a physical model of lighting as described in \refsec{Lighting}.

The formulas for geodesic flow starting from an arbitrary point $p$ can be efficiently factored using the homogeneity of $X$. That is, a conjugation by $g$ identifies the flow from $o$ with the flow from $p$.
In practice, for the easier geometries one might as well work at the position of the observer, $p$, rather than at $o$. However, for the harder geometries,  this significantly simplifies the code.

\subsection{Stereoscopic vision}
\label{Sec:Stereoscopic}
A virtual reality headset has a separate screen for each eye. This allows it to show the two eyes slightly different images -- parallax differences between these images can then be interpreted by the user's brain to give depth cues.

Given positions and facings for the left eye, $(p^\triangleleft, f^\triangleleft)$, and the right eye, $(p^\triangleright, f^\triangleright)$, we can render an image for each eye exactly as in \refsec{render fixed image}.
The question is how to determine the positions and facings for the two eyes. 
Let $\ell$ be the \emph{interpupillary distance}; that is, the distance between the eyes. We track the position and facing $(p,f)$ of the user's nose, using the sensors of the virtual reality headset as in \refsec{moving in the space}. In $\EE^3$, the canonical thing to do is to set  $f^\triangleleft$ and $f^\triangleright$ equal to $f$, and to set 
\[
p^\triangleleft = p - (\ell/2) f_1 \qquad p^\triangleright = p + (\ell/2) f_1
\]
recalling that $f_1$ is the frame vector in $f$ pointing to the right. 

This works because in euclidean space, one may naturally identify the tangent spaces at all points.
For non-euclidean geometries, a natural analogue is as follows. We set $(p^\triangleleft, f^\triangleleft)$ to be the result of flowing from $(p,f)$ for distance $\ell/2$ in the direction of $-f_1$, and we set $(p^\triangleright, f^\triangleright)$ to be the result of flowing from $(p,f)$ for distance $\ell/2$ in the direction of $f_1$.

This works reasonably well for $S^3$, $\HH^3$, and $\HH^2 \times \EE$, although there are some problems. As mentioned in~\cite[Section 6]{NEVR2}, in geometries in which geodesics diverge, parallax cues tell our euclidean brains that all objects are relatively nearby. In $\HH^3$ for example, two eyes pointing directly at an object that is infinitely far away are angled towards each other. 
One alternate strategy we briefly experimented with was to rotate the frames $f^\triangleleft$ and $f^\triangleright$ slightly inwards, so that geodesics emanating from $p^\triangleleft$ and $p^\triangleright$ in the directions of their forward vectors $-f_3^\triangleleft$ and $-f_3^\triangleright$ converge at infinity. This might then match the behavior our euclidean brains expect: that objects at infinity can be seen by looking straight ahead with both eyes. We did not notice much difference in our ability to perceive the space in making this change, although this line of thinking leads us to conclude that predators in hyperbolic space would evolve to look somewhat cross-eyed to us native euclideans.

In $S^3$, points at distance $\pi/2$ away from the user appear to be ``infinitely far away'', while objects further than $\pi/2$ away have depth cues reversed. One possible future direction to try to improve this experience is as follows. Modern virtual reality headsets have the ability to track where the user's eyes are looking. Based on this information, we could determine what object the user is looking at. Using the distance from the viewer to the object, we could rotate the frames $f^\triangleleft$ and $f^\triangleright$ to imitate the effects of parallax for objects at that distance in $\EE^3$. It remains to be seen whether or not these frequent rotations would induce nausea. 

The situation is worse in $S^2 \times \EE$, Nil, Sol, and $\SLR$, where geodesics ``spiral''. 
\reffig{parallax nil} illustrates how a small parallax in Nil can produce very different pictures:
On each row, the scene consists of a single ball textured as the earth. The different images are views of this ball from slightly different positions.
Using the convention that one unit represents one meter, the offset between two consecutive images is approximately half the interpupillary distance.
Our euclidean brains are not able to interpret the combination of these pictures.
One might think that the sphere is too small (a few centimeters) and too far away form the observer (a few meters) for our eyes to see that level of detail.
However geodesic rays in Nil spiral in such a way that the angular size of the object in the observer's view is very large. This makes the object appear as if it is very close to the observer.
Thus this parallax distortion cannot be ignored.
New ideas are thus needed to produce stereoscopic images in all eight geometries that can be pertinently analyzed by the brain. 

Weeks~\cite[Section 5]{Weeks20} uses the following approach in $S^3$ and $\HH^3$. The observer is represented by a point $p$ in the space $X$. Using the inverse of the exponential map, we send objects in that space to the tangent space $T_pX$ based at $p$. We then implement stereoscopic vision using cameras based at two points near the origin of $T_pX$. This works well when the inverse exponential map is single-valued, but seems challenging in the harder geometries.

\begin{figure}[htbp]
\centering
\subfloat[$r = 0.2$, $L = 7.3$, $\alpha = 68\degree$ and $L' = 0.36$.]{
	\stackunder{\includegraphics[width=0.3\textwidth]{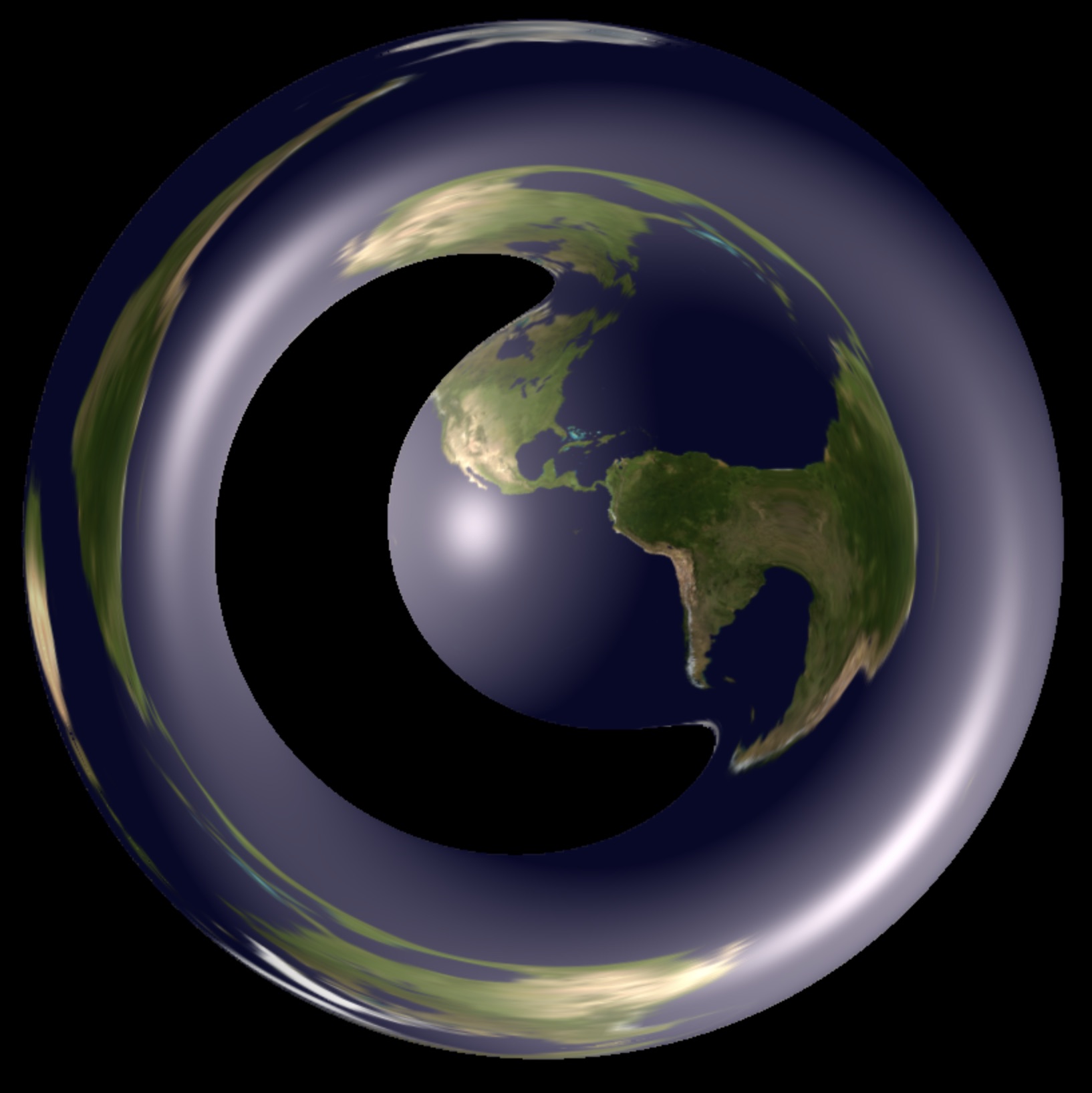}}{\footnotesize $\Delta = -0.03 $}
	\stackunder{\includegraphics[width=0.3\textwidth]{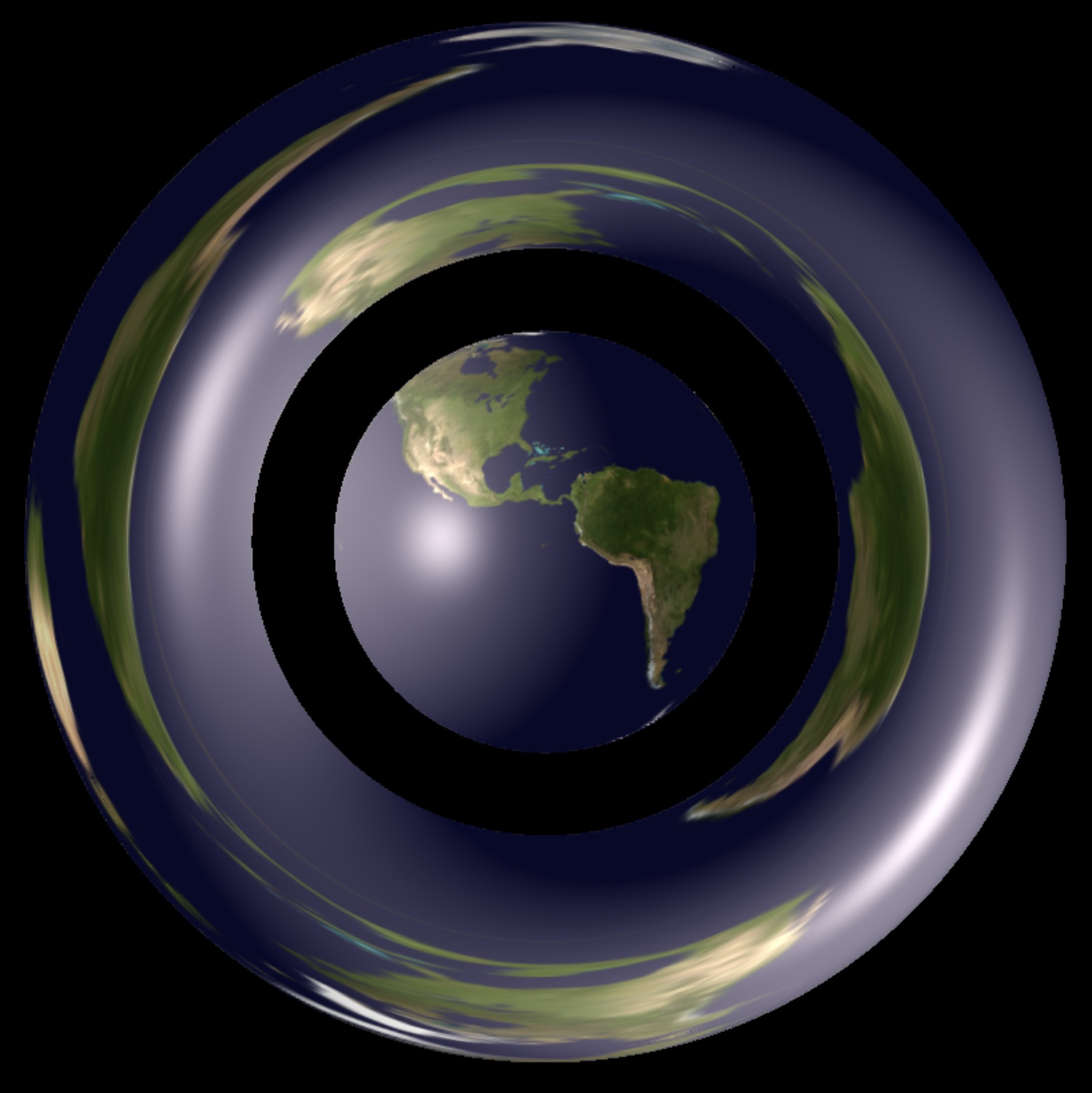}}{\footnotesize $\Delta = 0 $}
	\stackunder{\includegraphics[width=0.3\textwidth]{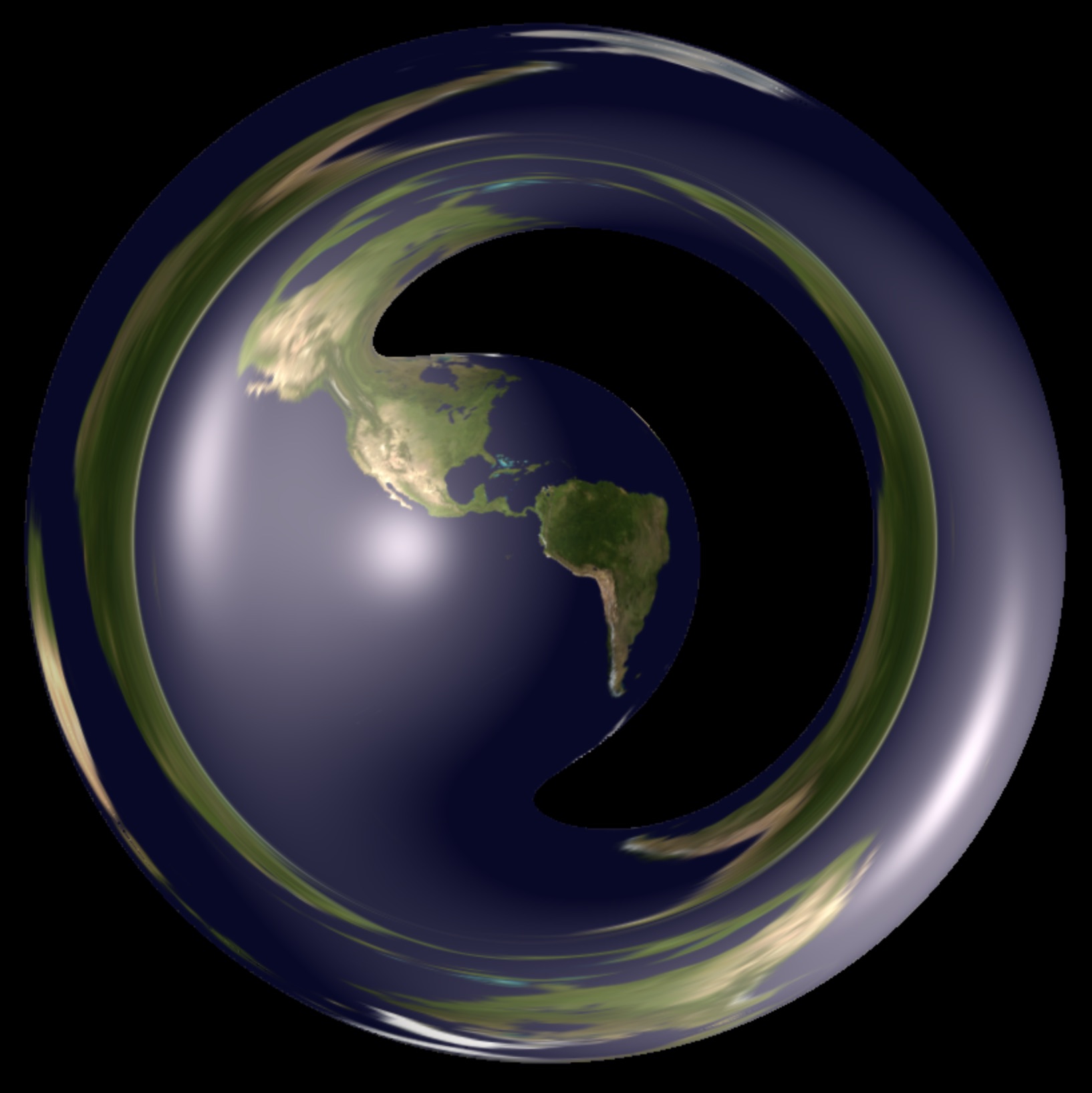}}{\footnotesize $\Delta = +0.03 $}
}\\
\subfloat[$r = 0.06$, $L = 6.55$, $\alpha = 39\degree$ and  $L' = 0.18$.]{
	\stackunder{\includegraphics[width=0.3\textwidth]{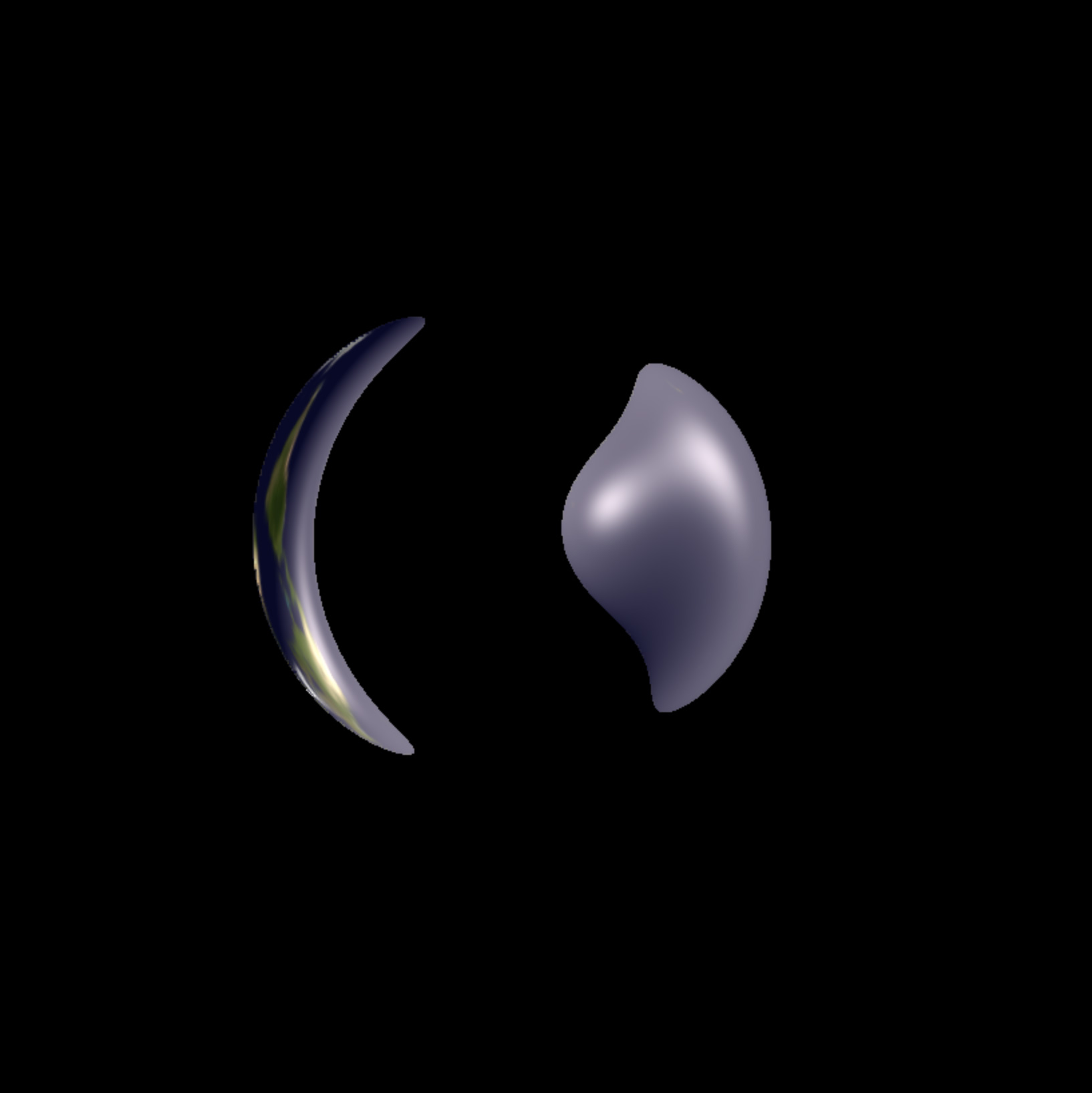}}{\footnotesize $\Delta = -0.03 $}
	\stackunder{\includegraphics[width=0.3\textwidth]{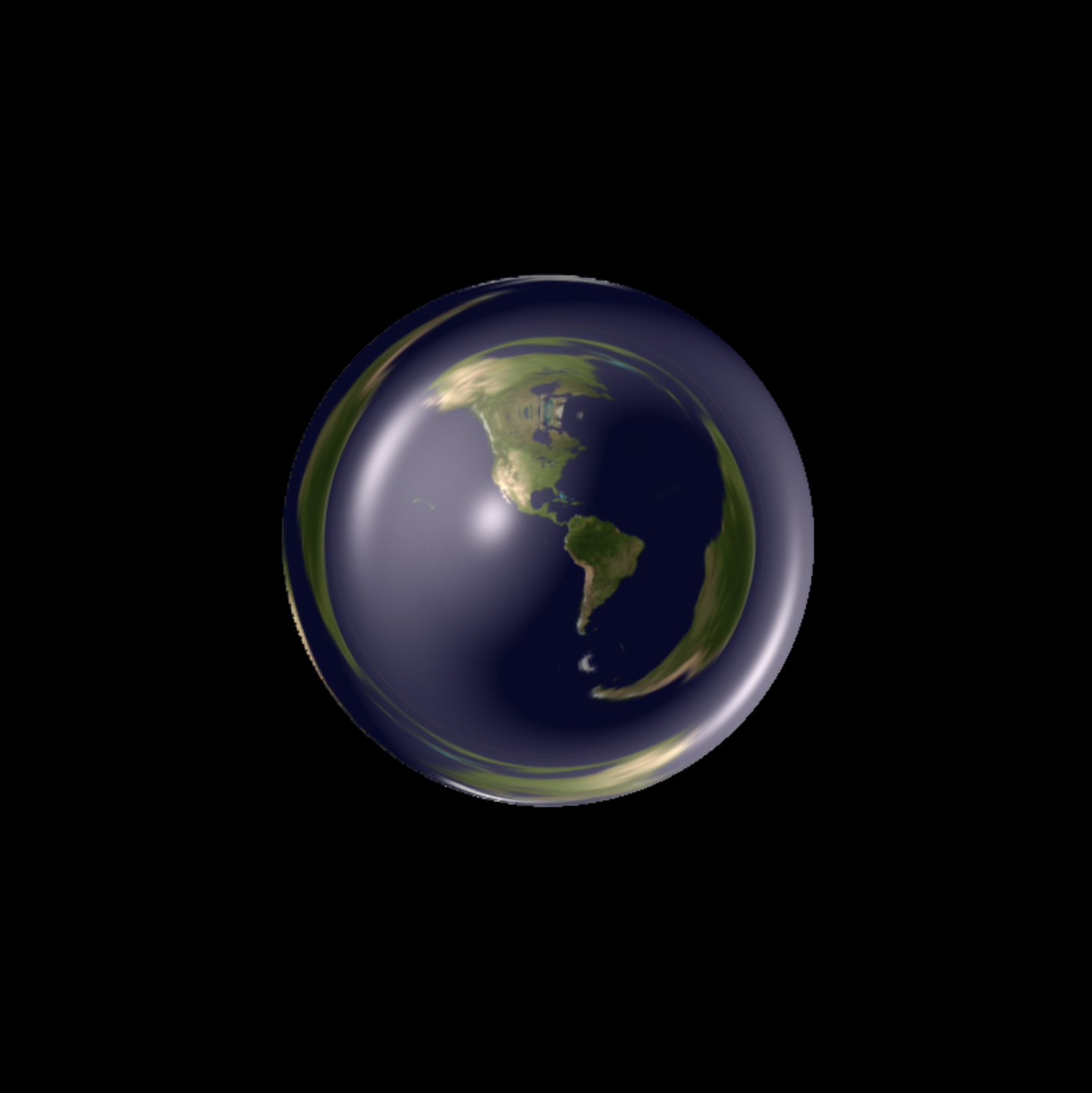}}{\footnotesize $\Delta = 0 $}
	\stackunder{\includegraphics[width=0.3\textwidth]{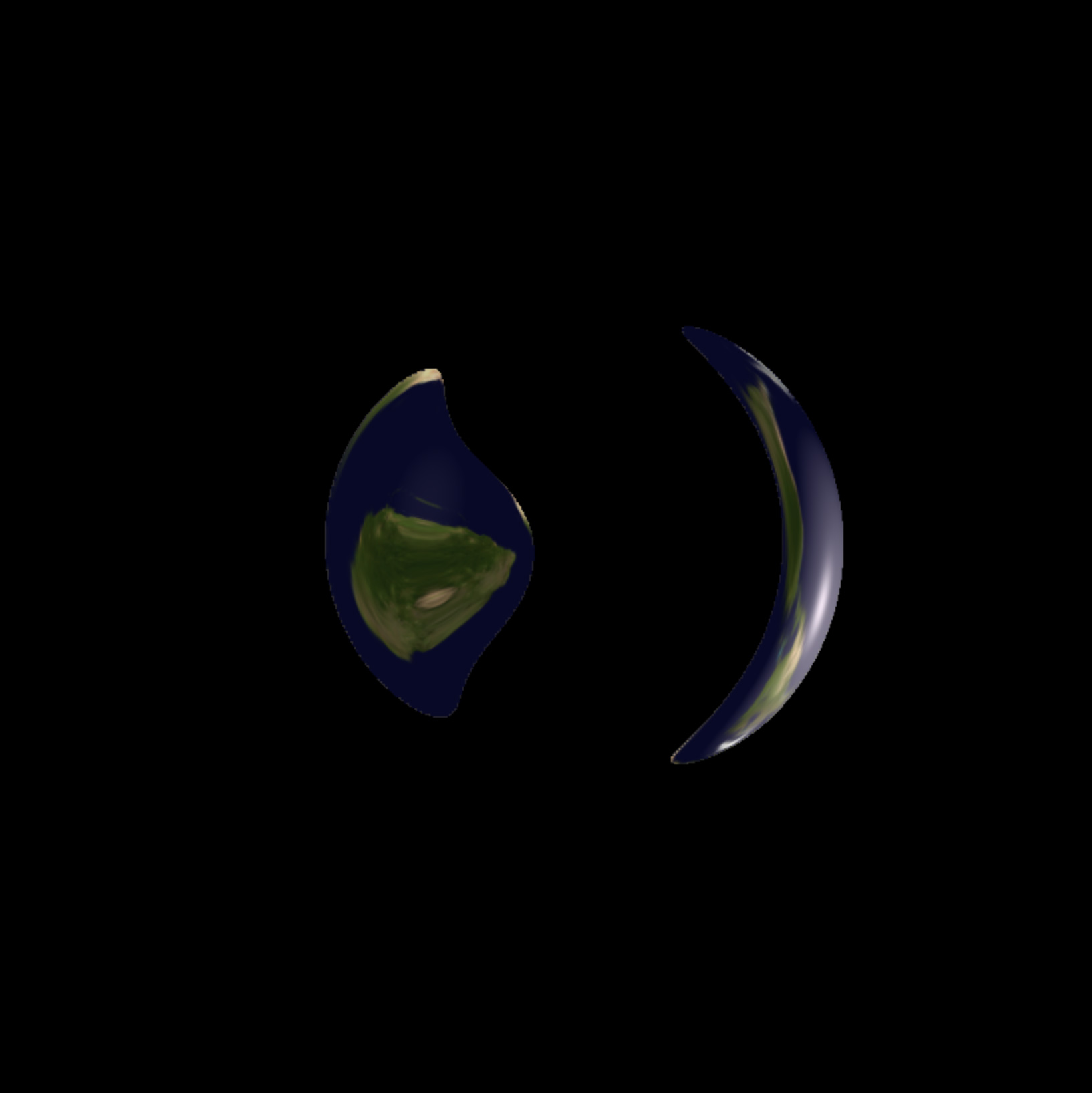}}{\footnotesize $\Delta = +0.03 $}
}
\caption{Parallax in Nil makes stereoscopic vision difficult.  
The earth has radius $r$ and is centered at the origin.
In the middle picture the observer is located on the $z$-axis at a distance $L$ from the origin. 
On the left and right pictures, the observer is offset by a distance $\Delta$ in the $x$-direction.
The angular size of the ball in the observer's view is $\alpha$. 
Note that due to the spiraling of geodesics in Nil, this angular size is much larger than it would be for an equivalent ball in euclidean space. Indeed, an observer assuming that they are in euclidean space would think that the ball is at distance $L'$ from them.
}
\label{Fig:parallax nil}
\end{figure}

\subsection{Signed distance functions in $X$}
\label{Sec:SDF}

The algorithms described so far render the in-space view of a scene in the geometry $X$, given a signed distance function $\sigma\colon  X\to\RR$ for it.
In the interest of both simplicity and geometric accuracy, we focus on scenes built from intrinsically defined objects, including
\begin{itemize}
\item balls (bounded by equidistant surfaces from a point),
\item solid cylinders (bounded by equidistant surfaces from a geodesic), and
\item half-spaces (bounded by totally geodesic codimension one submanifolds).
\end{itemize}

Note that a single object may fall into more than one of the above categories. For example, a hemi-hypersphere of $S^3$ is both a ball and a half-space.

\subsubsection{Simple Scenes}
In some cases, viewing and moving relative to a single simple object is all that is needed to illustrate surprising features of a geometry. In previous work for example, we qualitatively described counterintuitive features of Nil geometry~\cite{NEVR3}  with a scene consisting of a single ball, and we studied a single isometrically embedded copy of the euclidean plane in Sol geometry~\cite{NEVR4}.
From a collection of basic objects, many other simple scenes can be created through finitely many applications of union, intersection and difference.
These operations of constructive solid geometry are particularly suited to producing scenes in a ray-marching application, as $\{\cup, \cap, \smallsetminus\}$ are faithfully represented on the space of signed distance functions by $\{\min,\max, -\}$ respectively~\cite{Quilez}.

In many cases however, the interesting features of the geometry are best exhibited by more complex, unbounded scenes, which cannot be built from the basic objects in finitely many operations.

\subsubsection{Complex Scenes and Symmetry}

Scenes which display interesting features across unbounded regions are useful to highlight various geometric features, including
\begin{itemize}
\item exponential growth of volume in negative curvature,
\item anisotropy in the product geometries,
\item non-integrability of the contact distribution in Nil, and
\item the lack of any continuous rotation symmetry in Sol.	
\end{itemize}

The particular details of the scene's contents do not matter so much as the requirement that the user may travel unbounded distances in any direction and still be surrounded with an approximately homogeneous collection of objects.

One way to do this is to use the homogeneity of $X$ to build an extremely symmetric scene, by choosing a signed distance function $\sigma\colon X\to \RR$ invariant under the action of a discrete subgroup $\Gamma<G$. 

As geometric topologists however, we cannot help but note that covering space theory provides an alternative perspective.
Consider a scene invariant under the action of $\Gamma$.  
This is described by a signed distance function $\sigma\colon X\to \RR$ with $\sigma\circ\gamma=\sigma$ for all $\gamma\in \Gamma$.
Such maps are in natural correspondence with maps from the quotient $\overline{\sigma}\colon X/\Gamma\to \RR$.

Indeed, the view from a point $q\in X/\Gamma$ of a signed distance function $\overline{\sigma}$ is identical to the view from a lift $\cover{q} \in X$ of a signed distance function $\sigma$ invariant under $\Gamma$. This follows from the above topological correspondence together with the fact that the covering map is a local isometry.

This suggests exploring the unbounded geometry of $X$ indirectly, through the geometry of its quotients $X/\Gamma$.

\section{Non-simply connected manifolds}
\label{Sec:NonSimplyConnected}

Let $(G,X)$ be a homogeneous geometry. A $(G,X)$-manifold is a smooth manifold $M$ together with an atlas of charts 
$$\{(U_\alpha\subset M, \quad f_\alpha\colon U_\alpha\to X)\}$$ 
with transition maps in $G=\Isom(X)$.
The elementary theory of such $(G,X)$-manifolds shows that one may globalize this atlas into a \emph{developing map} from $\widetilde{M}$ to  $X$, equivariant with respect to a \emph{holonomy homomorphism} from $\pi_1M$ to $G$~\cite{GoldmanBook}.
Furthermore, if $M$ is geodesically complete,  
then the developing map is a diffeomorphism and $M\cong X/\Gamma$ is a quotient, where $\Gamma\cong\pi_1(M)$ is the image of the holonomy homomorphism.
The simplest $(G,X)$-manifold is $X$ itself, and we have seen above how to ray-march simple scenes in $X$.
Covering space theory implies that $X$ is the unique complete simply connected $(G,X)$-manifold, but non-simply connected $(G,X)$-manifolds abound.
Indeed the classification of compact hyperbolic manifolds up to diffeomorphism is still incomplete.
Additionally, while there are only ten  euclidean manifolds up to diffeomorphism, there are uncountably many distinct euclidean structures in each diffeomorphism class.
Simulating not just the Thurston geometry $X$ but also various $(G,X)$-manifolds is a natural extension of our original goals. These manifolds may or may not have finite volume, corresponding to the discrete subgroups $\Gamma < G$ being lattices or not. Generalizing further, our algorithms can also simulate $(G,X)$-orbifolds and incomplete $(G,X)$-manifolds.
Thus we may experience both the three-dimensional homogeneous spaces, and also the atomic building blocks of geometrization.

In the next section, we describe a method to ray-march (or ray-trace) within a quotient manifold, using a fundamental domain. Similar ideas are outlined in~\cite{BergerLaierVelho} and~\cite{Kopczysk}.

\subsection{Teleporting}
\label{Sec:Teleporting}
Let $\Gamma$ be a discrete subgroup of $G$, and $M=X/\Gamma$.
To produce an intrinsic simulation of $M$, we wish to reuse as much as possible the work that goes into producing a simulation of $X$. To that end, we describe $M$ using a connected fundamental domain $D\subset X$ with $2n$ faces $\{F_i^\pm\}_{i=1\ldots n}$. (Alternatively, one could embed $M$ in a higher-dimensional ambient space, and try to implement the techniques of \refsec{General} in that context.)
The quotient manifold $M$ is obtained by identifying each $F_i^-$ with $F_i^+$ via an isometry $\gamma_i\in\Gamma$.
These face pairings form a generating set $\{\gamma_1,\ldots,\gamma_n\}$ for $\Gamma$.
This allows us to ray-march using the geodesic flow on $D\subset X$, and calculate parallel transport and position/facing using the parametrization of $\mathcal{O}X$ restricted to $D$.
Indeed, given a signed distance function $\sigma\colon X/\Gamma\to \RR$ pulled back to $D$, the only substantial change is that we must modify the ray-march algorithm to keep the geodesic flow in $D$. We can do this by using the face pairings. Similarly, when the user moves outside of $D$, we move them by an isometry to keep them inside of $D$. In either case, we call this process \emph{teleporting}. See \reffig{LightRayInQuotientManifold}.

\begin{figure}[htbp]
\centering
\includegraphics[width=0.45\textwidth]{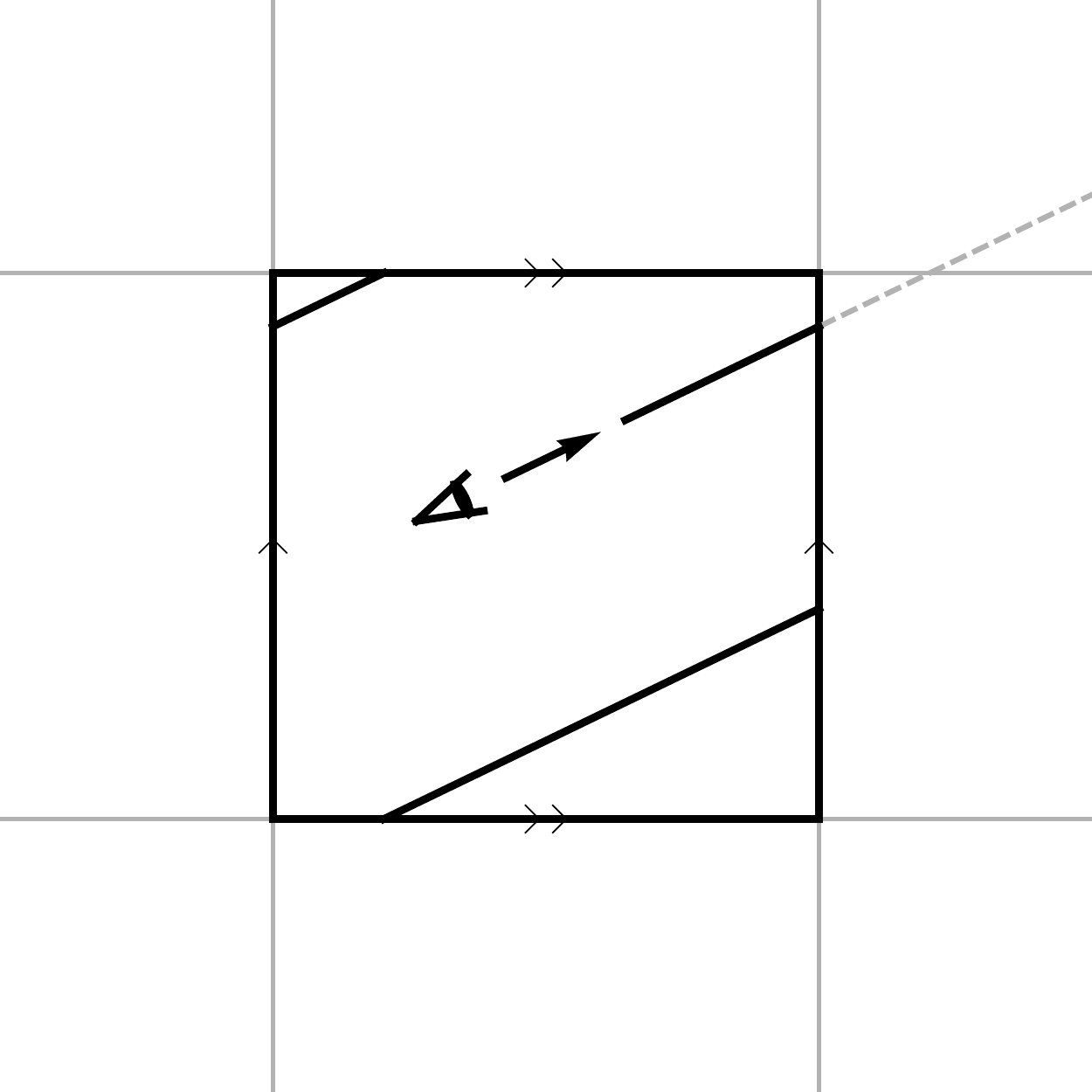}
\caption{A light ray traveling in a domain $D$ must teleport at the boundary to simulate the view within a torus.}
\label{Fig:LightRayInQuotientManifold}
\end{figure}

\begin{remark}
\label{Rem:SideBenefit}
As a side benefit, the quotient manifold approach helps with floating point errors. At each step of our ray-marching algorithm, the basepoint of our ray is within $D$. In the case that $M$ is compact for example, the coordinates of our basepoint are bounded by a function of the diameter of $D$. This then avoids problem \refitm{ExpCoords} of \refsec{FloatingPoint}. In our experience, we see less noise in images such as \reffig{CompareH3} with this strategy, despite the potential accumulation of errors (see \refsec{Accumulation}) introduced by repeatedly teleporting a ray's position and tangent vector back inside of $D$.
\end{remark} 

\begin{remark}
\label{Rem:SideBenefit2}
It may be useful to employ teleporting even when we are simulating a scene inside of the simply connected geometry $X$ rather than inside a quotient manifold. That is, we have a discrete subgroup of isometries and a fundamental domain $D$, and we use teleportation to keep the viewer always within $D$. Whenever we teleport the user, we also teleport all other objects in the scene, and update the signed distance function as appropriate. The advantage here is that rays begin inside of $D$, where their coordinates are small. Therefore floating point errors only accumulate to a noticeable degree for objects which are far from the viewer. For some geometries, such distant objects will be very small on the visual sphere. Alternatively, they may be hidden by fog.
\end{remark}

\subsubsection{Teleporting with a Dirichlet domain}
\label{Sec:TeleportingDirichlet}
A simple, geometry independent implementation involves choosing the Dirichlet domain $D$ for the action of $\Gamma$, centered at the origin $o\in X$.
To determine whether or not a point $p$ is outside of $D$, we compare the distance $d(p,o)$ with $d(p,\gamma_i^\pm o)$ for each face pairing isometry $\gamma_i$. 
When $d(p,o) > d(p,\gamma_i^\pm o)$, the point $p$ can be brought back closer to $o$ via an application of $\gamma_i^\mp$. Iterating this (relabelling our point as $p$ after each step) until $d(p,o) \leq d(p,\gamma_i^\pm o)$, we ensure that $p$ is inside of $D$.

An advantage of this approach is that one does not need an analytic description of the boundary $\partial D$ to accurately adjust the ray-march.
When the intrinsic distance $d$ is expensive to calculate however, this adds a significant extra computational burden.

\subsubsection{Teleporting with a projective model and linear algebra}
\label{Sec:TeleportingProjective}
A second implementation that removes the need to calculate distances is possible for the Thurston geometries. Up to covers (in the cases of $\SLR$ and $S^3$), these have \emph{projective models}: a representation of the geometry as an open subset $r\colon X\hookrightarrow\RR\mathbb{P}^3$, together with a linear representation $\Isom(X)\to PGL(4;\RR)$~\cite{Molnar97}.  

To lighten the notation in this section, we identify $X$ with its image under $r$. We choose our fundamental domain $D$ for the action of $\Gamma$ such that $D=\bigcap_i H_i^\pm$, where $\{H_i^\pm\}$ is a collection of $2n$ half-spaces of $X$. 
The point $p$ is outside of $D$ if and only if there is a half-space $H_i^\pm$ such that $p \not\in H_i^\pm$. Each half-space $H$ of $\RR^3$ is in natural correspondence with a linear functional $\phi\colon\RR^3\to\RR$, where $v\in H$ if and only if $\phi(v)\geq 1$, so we can check if $p \in H_i^\pm$ by computing the value $\phi_i^\pm(p)$.
The embeddings $r\colon X\to \RR\mathbb{P}^3$ are inexpensive to compute in our models (see \reftab{GeometryDetails}): for $S^3,\HH^3,S^2\times\RR,\HH^2\times\RR$ we divide by the fourth coordinate, and $\EE^3$, Nil, Sol are already affine patches. The situation for $\SLR$ is slightly more complicated, but similar ideas work for the fundamental domains we have implemented.
Thus, we reduce the problem to a quick calculation in linear algebra.

Knowing which of the half-planes $p$ is not contained in, we now must find 
the element of $\Gamma$ which moves $p$ back into $D$. We iteratively construct this element from the $\gamma_i^\pm$ for which (at each step) $\phi_i^\pm(p) > 1$. 
In many cases (for example when $\Gamma$ is a finite index subgroup of a reflection group), it does not matter which such $\gamma_i^\pm$ we choose at each step. In other cases, for reasons of efficiency, one must be more careful with the ordering, see for example \refsec{NilDiscreteSubgroupsFundDoms}.

Since we have projective models for the eight Thurston geometries, we use this strategy rather than the Dirichlet domain strategy. 

\begin{remark}
\label{Rem:ProjectiveNoFacePairings}
In practice, when using the projective model we can take $S = \{\gamma_i\}$ to be an arbitrary generating set for $\Gamma$. We then generate the half-spaces $H^\pm_i$ from $S$. Their intersection forms a fundamental domain $D$. Note that multiple faces of $D$ may lie in the boundary of a single half-space, and the face pairings of $D$ may involve elements of $\Gamma$ other than those in $S$.  However, we need only use elements of $S$ to implement teleportation. See \refsec{NilDiscreteSubgroupsFundDoms} for a detailed example.
\end{remark}

\subsection{Signed distance functions in $X/\Gamma$}
\label{Sec:SDF in X/Gamma}
With the addition of teleportation, we may draw scenes in any complete $(G,X)$-manifold using the same algorithms as we use in $X$ itself, given the input data of a signed distance function mapping $X/\Gamma$ to $\RR$ describing the scene.
Unfortunately, efficiently calculating a signed distance function (or even a distance underestimator) for a scene in a quotient manifold is often non-trivial. In practice, we will often use an approximation. 

We can construct a very simple approximation for a scene $S$ as follows. Let $D \subset X$ be a fundamental domain for the quotient manifold $X/\Gamma$. We then view $S$ as a subset of $D$. For a point $p \in D$, we may then return the signed distance from $p$ to $S$, where we measure distance in $X$, ignoring the quotient manifold structure entirely. 
Let us call this simplest approximation $\sigma \from X \to \RR$. (Here we implicitly extend the signed distance function from $D$ to $X$.)

\begin{figure}[htbp]
\centering
\subfloat[The signed distance function for a disk in a torus, drawn in the universal cover. ]{
\includegraphics[width=0.45\textwidth]{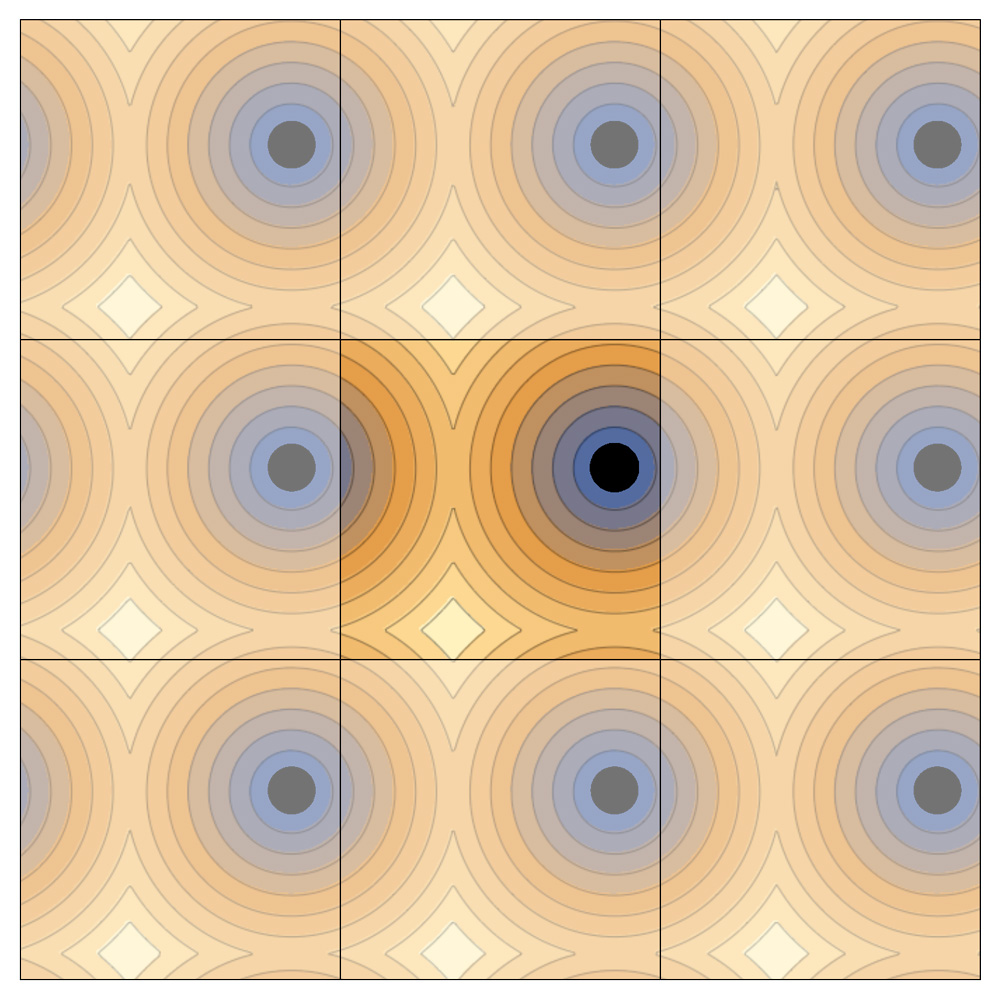}
\label{Fig:SDF_in_M}
}
\quad
\subfloat[The simplest approximation to the signed distance function, $\sigma$.]{
\includegraphics[width=0.45\textwidth]{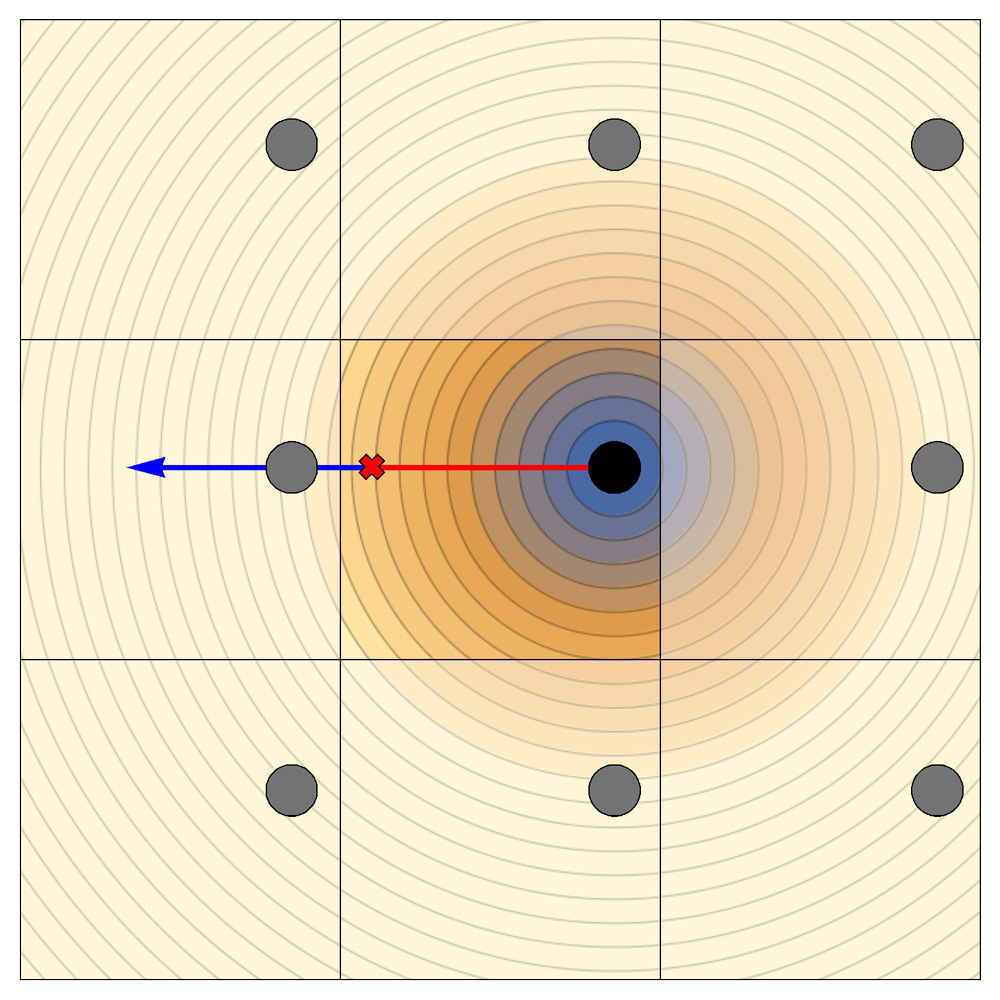}
\label{Fig:SDF_in_X}
}
\caption{Functions on a torus. We indicate the level sets by bands of color. }
\label{Fig:SDFsTorus}
\end{figure}
 
As an example, \reffig{SDF_in_M} shows the correct signed distance function for a disk in a square torus, while \reffig{SDF_in_X} shows $\sigma$. 
For such a square torus, $\sigma |_D$ will be the correct signed distance function for the quotient torus only if the disk is centered in the square. 
Using $\sigma |_D$ in place of the correct signed distance function can lead to some serious visual artifacts. For example, consider a ray starting at the position $p$ marked with a small red ``$\times$'' in \reffig{SDF_in_X} and heading to the left. This ray should leave through the left side of $D$, teleport to the right side of $D$, then hit the disk. However, the function $\sigma |_D$ reports that the distance from $p$ to the disk (indicated with the red interval) is more than half the width of the square. A march along the ray by this distance is shown with the blue arrow: we jump straight through the disk. The result is that this lift of the disk is invisible when viewed from $p$. 

A similar but less extreme form of visual artifact is shown in \reffig{Jaggies}. Here we see jagged errors on the boundaries between cells. In some places near the boundary of $D$ we erroneously jump through points of the scene. Whether or not we make such a jump depends on how close to the boundary of $D$ we land before jumping across the boundary. The variability in this leads to the jaggedness. \reffig{BasicSDF} shows related artifacts.

\subsubsection{Creeping over the boundary of $D$}
\label{Sec:Creep}
One strategy to avoid these kinds of errors uses the observation that flowing by the distance given by $\sigma$ is only dangerous if our ray leaves $D$. Thus, we should detect when a ray passes outside of $D$, and stop just outside. As usual, we are teleported back inside of $D$, and continue ray-marching.

\begin{figure}[htbp]
\centering
\subfloat[No Creeping.]{
\includegraphics[width=0.45\textwidth]{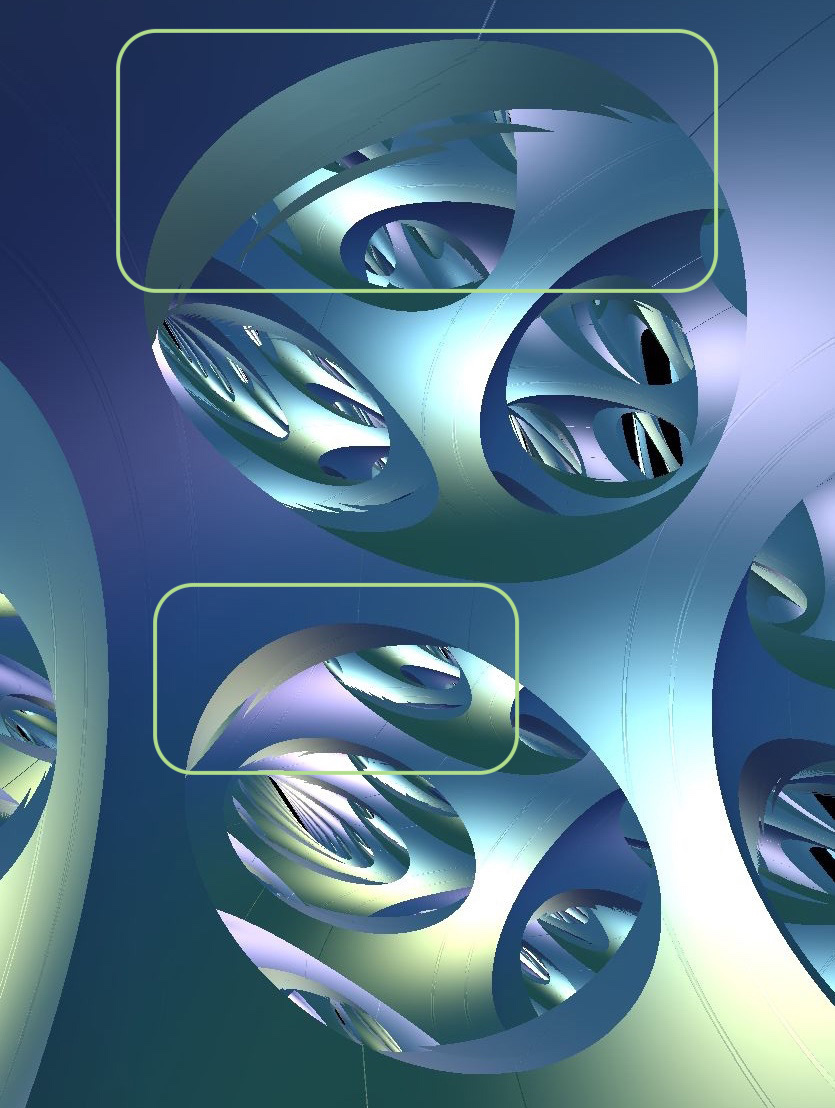}
\label{Fig:Jaggies}
}
\quad
\subfloat[Creeping.]{
\includegraphics[width=0.45\textwidth]{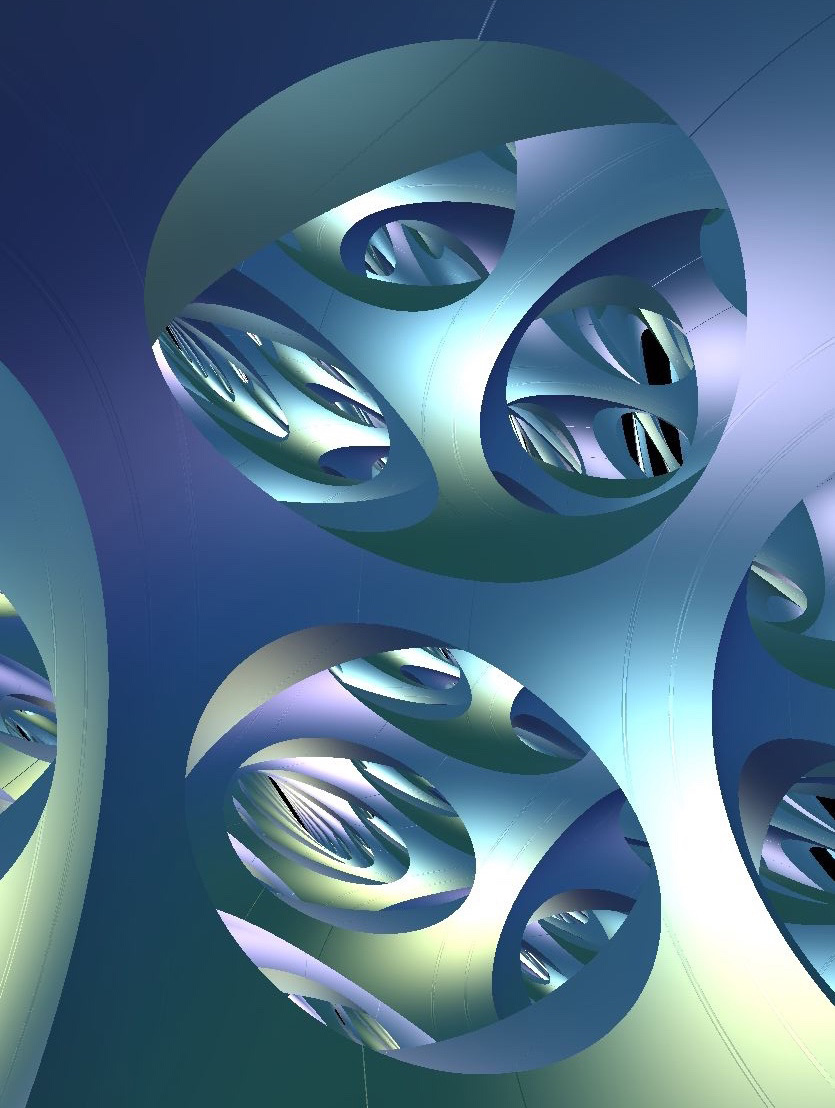}
\label{Fig:JaggiesGone}
}
\caption{Allowing the ray-march to leave the fundamental domain can cause visual artifacts on objects near its faces.  Creeping up to the boundary fixes this.}
\label{Fig:JaggiesProblem}
\end{figure}

Detecting when a ray hits $\bdy D$ is a similar problem to that of detecting when the ray hits an object in the scene. We employ a variety of different methods, as follows.
\begin{enumerate}
\item One way to do this is to use ray-tracing: we solve for the intersection between the ray and the boundary, and measure the distance between this intersection point and the start of the ray. 
\label{Itm:CreepRayTrace}
\item If it is difficult to solve for this point of intersection, but the faces of $D$ have computable signed distance functions, then we can instead use ray-marching. We flow by the minimum of $\sigma$ and the distance to $\bdy D$.\footnote{In practice, we allow a march of the distance to the nearest wall plus some small $\varepsilon$: this prevents wasting many steps approaching the boundary to no appreciable theoretical disadvantage: the teleportation scheme returns us to $D$ immediately upon overstep.}
\label{Itm:CreepRay-march}
\item When the faces do not have computable signed distance functions but we can still detect whether or not we are inside of $D$, we proceed as follows: We flow by the distance given to us by $\sigma$, and ask if the result puts us outside of $D$. If it does, then we perform binary search on the distance we flow to find a point just outside of $D$.
\label{Itm:CreepBinarySearch}
\end{enumerate}

Creeping just over the boundary solves the problem shown in \reffig{Jaggies}, giving the correct image, \reffig{JaggiesGone}. In general, creeping produces the correct pictures as long as all objects in the scene are contained within the domain $D$. However, this breaks down if we wish to, for example, move a ball from one domain to another. When a ball intersects $\bdy D$, calculating the approximation $\sigma$ requires measuring the distance to the center of the ball in $D$, and at least one translate of its center under some element of $\Gamma$. See \reffig{CrossingBorder}. Without this extra calculation, one sees objects cut in half by the boundary of $D$. See \reffig{Creep}.
Solving this problem led us to the following alternate (or additional) strategy to creeping.

\begin{figure}[htbp]
\centering
\subfloat[An incorrect calculation of $\sigma$, using only the disk whose center is in $D$.]{
\includegraphics[width=0.45\textwidth]{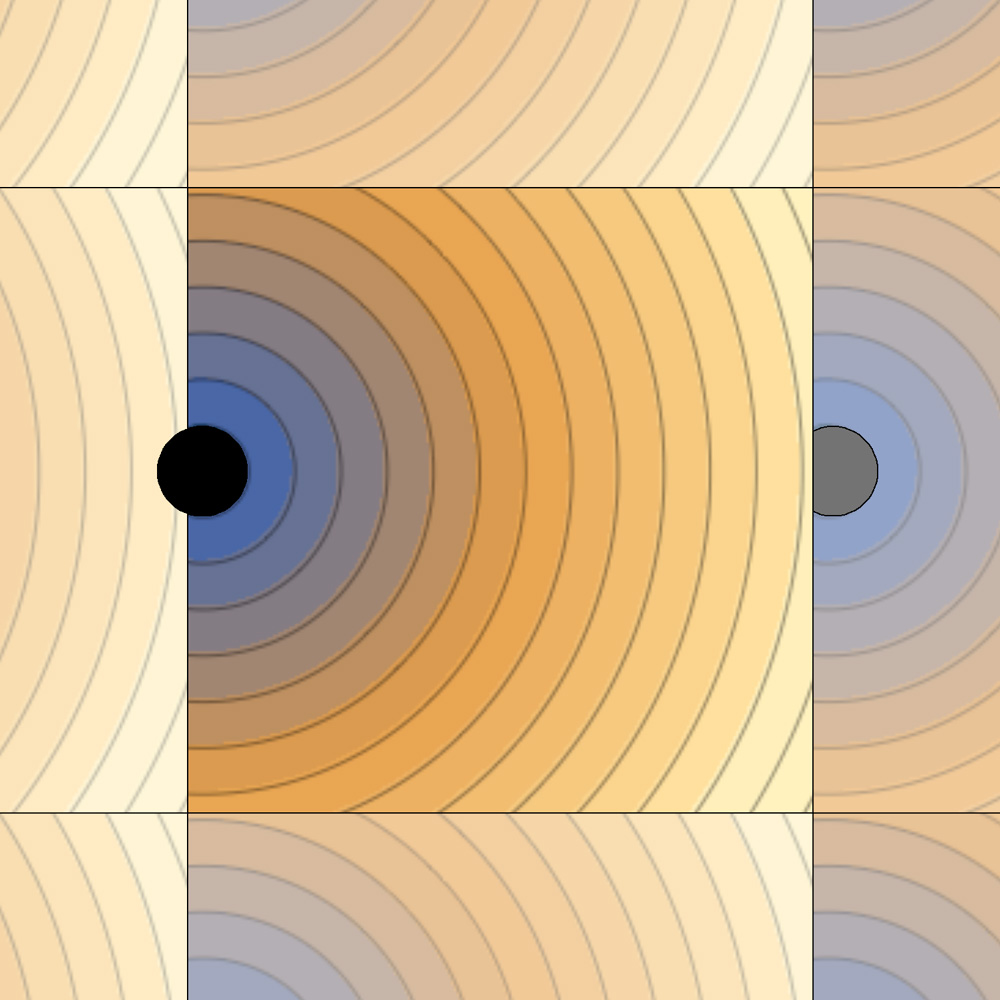}
\label{Fig:WrongSigma}
}
\subfloat[The correct calculation of $\sigma$ requires calculation of the distance to at least two points.]{
\includegraphics[width=0.45\textwidth]{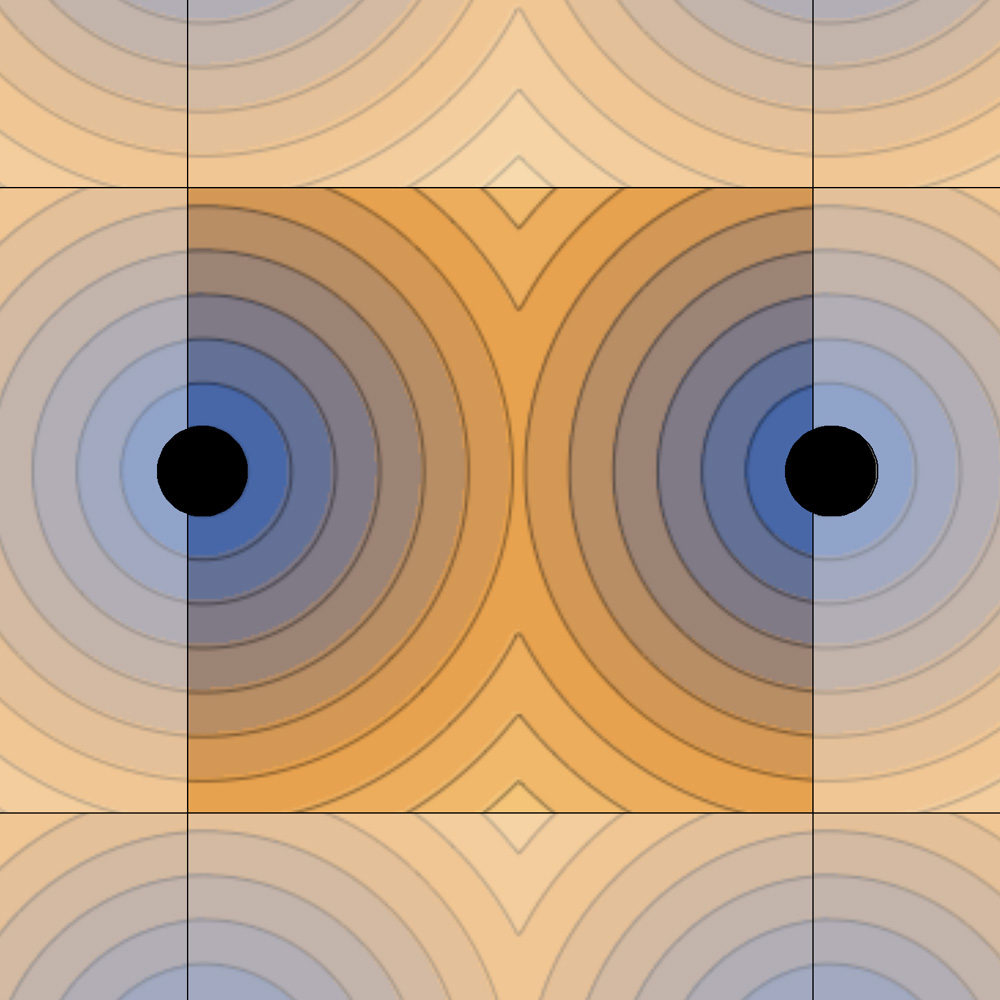}
\label{Fig:RightSigma}
}

\caption{Calculating $\sigma$ for a disk overlapping the boundary of $D$. }
\label{Fig:CrossingBorder}
\end{figure}

\subsubsection{Nearest neighbors signed distance functions}
\label{Sec:NearestNeighbors}

Here we use a signed distance function on $D$ that takes into account the effects of the nearby translates of $D$. 

Let $A \subset \Gamma$ be a set of isometries. Define
$$\sigma_A = \min_{a\in A}\{\sigma \circ a\}$$

\noindent For example, $\sigma_{\{\text{id}\}}$ is just $\sigma$, and $\sigma_\Gamma$ is the correct $\Gamma$-invariant signed distance function. 
If $\Gamma$ is infinite, then we cannot calculate $\sigma_\Gamma$ directly. However, if the tiling of $X$ by copies of the fundamental domain is locally finite, then there is a finite subset $A\subset \Gamma$ such that $\sigma_A$ and $\sigma_\Gamma$ are equal on $D$. Indeed, we may choose for $A$ the set of all $\gamma \in \Gamma$ such that the distance from $D$ to $\gamma(D)$ is at most the diameter of $D$.
Depending on the shape of the fundamental domain and how it is glued to itself however, the size of $A$ may be large. If so, calculating this signed distance function may be prohibitively expensive.

We find that most visual artifacts can be resolved without the use of creeping by using $\sigma_A$, where $A = \{ \text{id} \} \cup \{ \gamma_i^\pm \}$. That is, we use $\sigma$ in $D$ and its nearest neighbors, directly connected by face pairings. See \reffig{NearestNeighbor}.
In some circumstances this may not be enough; see for example \reffig{NearestNeighborNotEnough}. Here a ray passing close to a vertex of the tiling may not see an object diagonally adjacent to the starting domain. In three dimensions the equivalent problem can appear for rays crossing close to an edge of the tiling. 

\begin{figure}[htbp]
\centering
\subfloat[Signed distance function restricted to $D$. Note the striped artifacts in various copies of the red ball.]{
\includegraphics[width=0.90\textwidth]{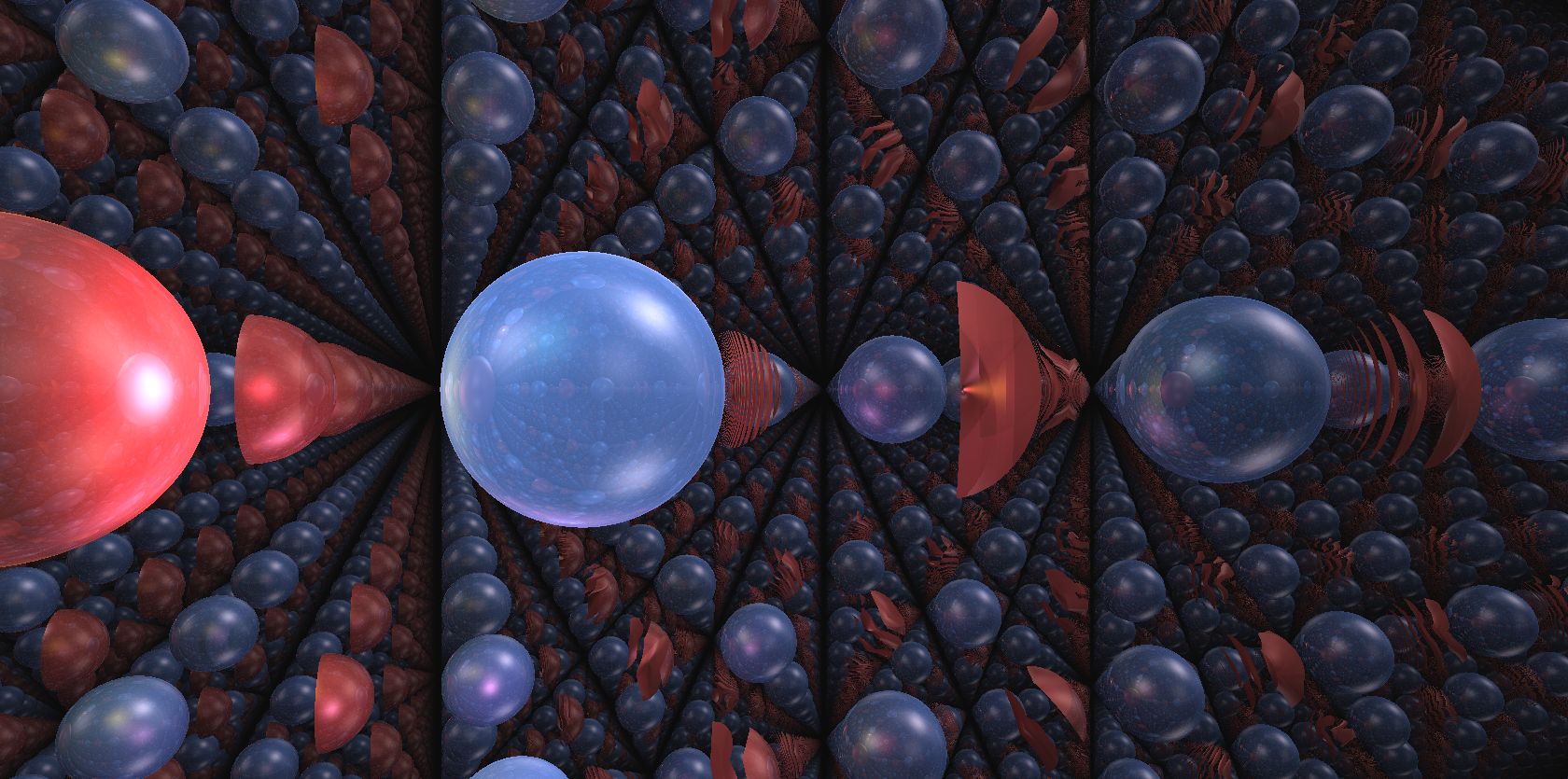}
\label{Fig:BasicSDF}
}\
\subfloat[Creeping to the boundary of $D$. The striped artifacts are gone, but we can see only half of the red ball.]{
\includegraphics[width=0.90\textwidth]{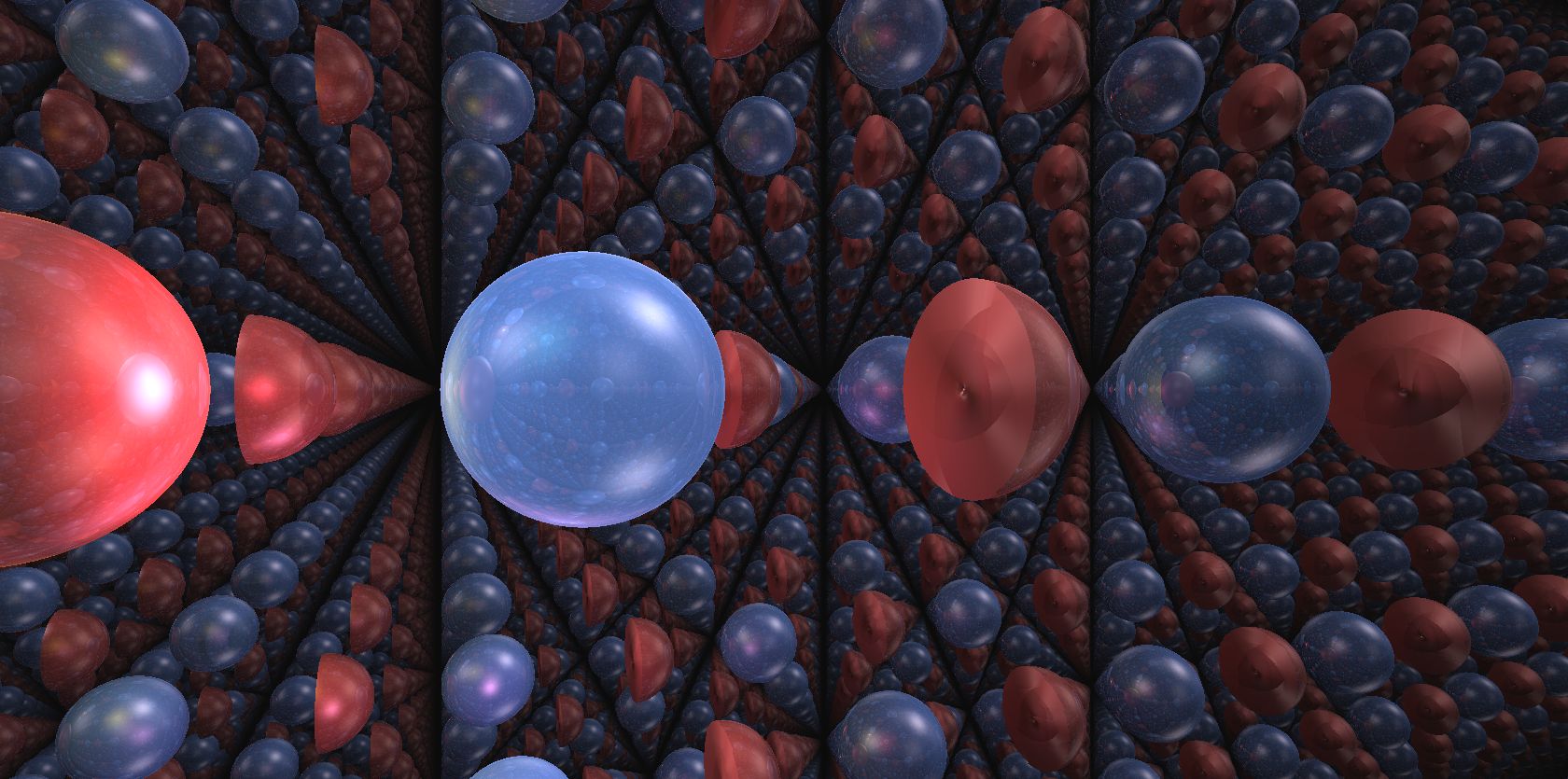}
\label{Fig:Creep}
}\
\subfloat[Using a nearest neighbors signed distance function, without creeping.]{
\includegraphics[width=0.90\textwidth]{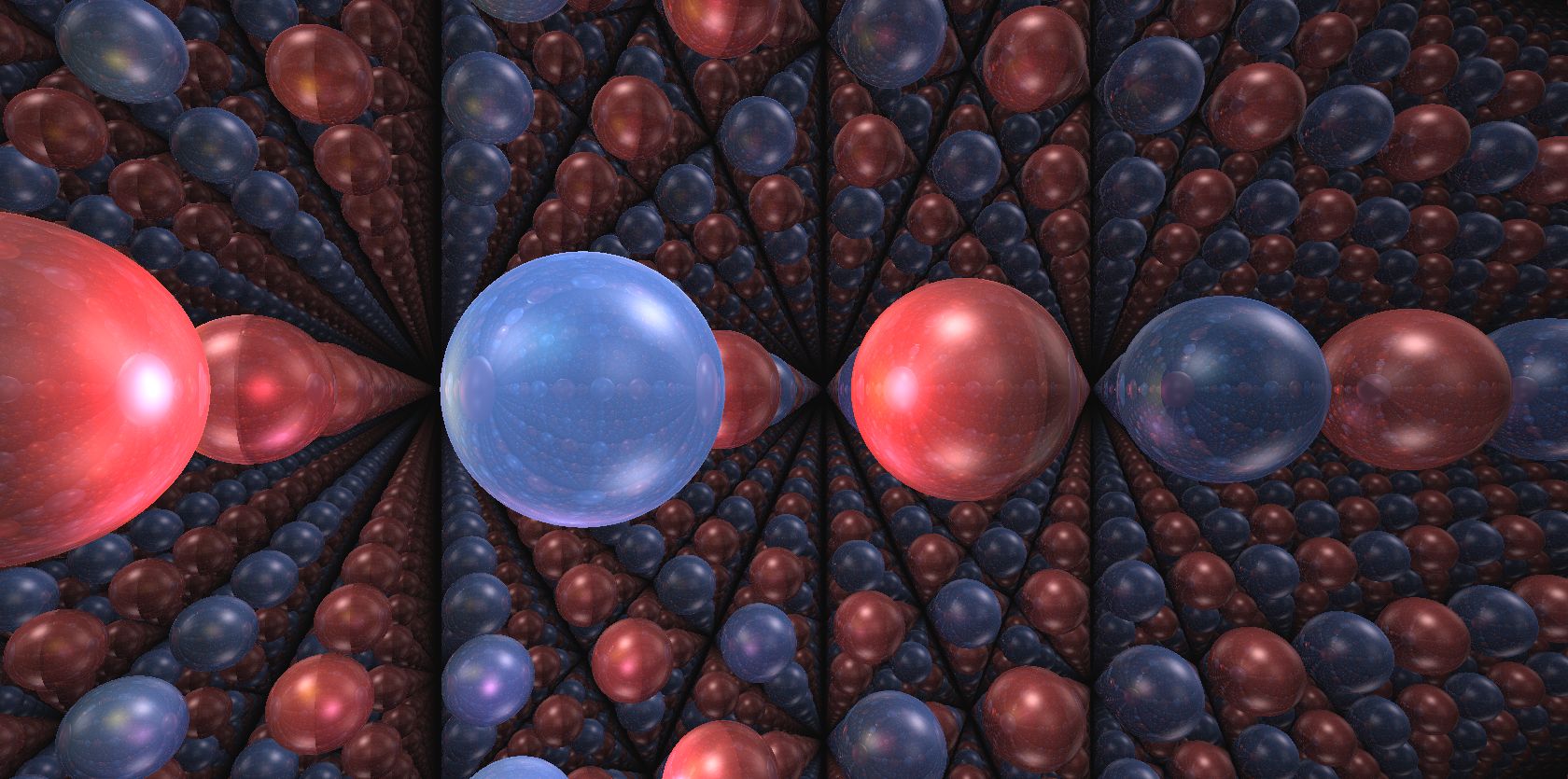}  
\label{Fig:NearestNeighbor}
}\
\caption{Difficulties when ray-marching in a fundamental domain $D$.  The blue sphere is contained fully in $D$.  The red sphere is only half contained in $D$. }
\label{Fig:CreepAndNeighbor}
\end{figure}

\begin{figure}[htbp]
\centering
\includegraphics[width=0.45\textwidth]{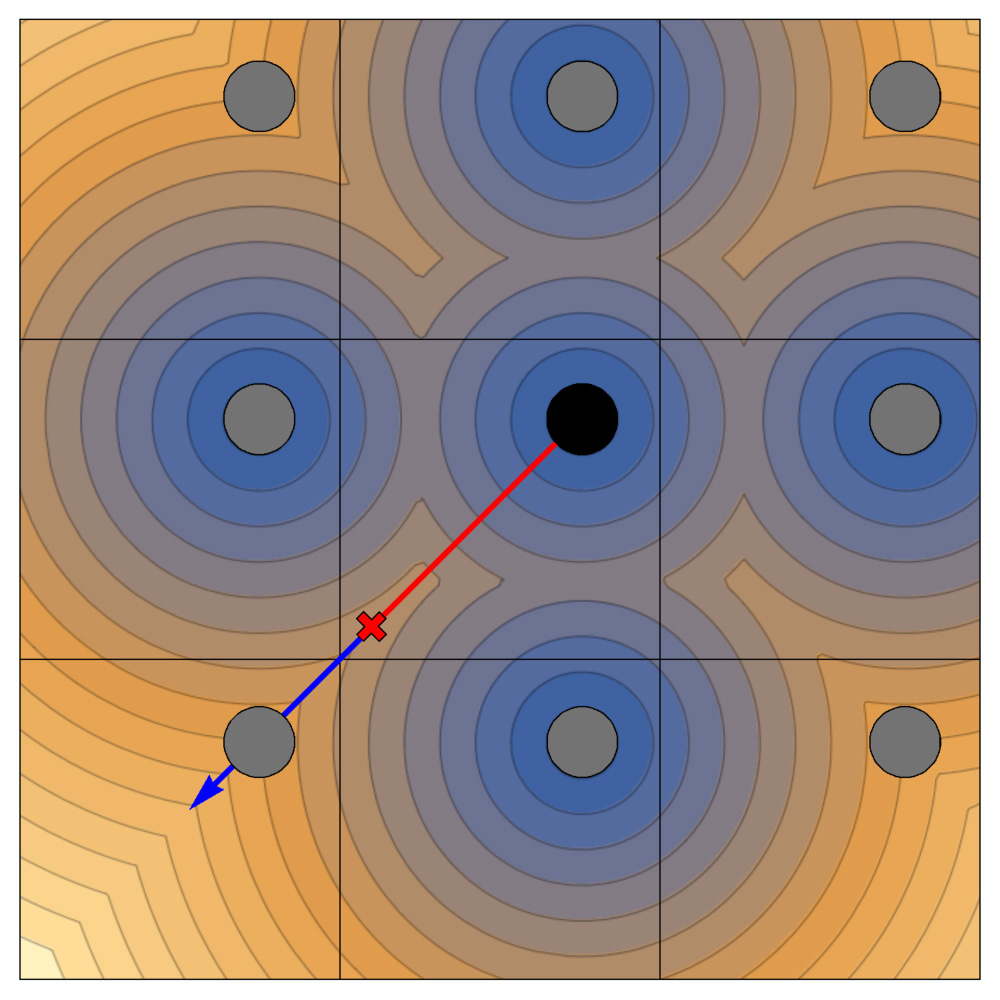}
\caption{For rays traveling near to a vertex, only using the nearest neighbors of a tile may not be enough to remove all visual artifacts without creeping.}
\label{Fig:NearestNeighborNotEnough}
\end{figure}

In general, depending on the circumstance, either creeping or using a nearest neighbors signed distance function, or some combination of the strategies may be the most efficient strategy to obtain correct images. Even the combination of both strategies can produce errors in some circumstances.  In \reffig{SphereAtEdge}, the only solution would be to use more translates of $\sigma$ than just the nearest neighbors. 

\begin{figure}[htbp]
\centering
\includegraphics[width=0.90\textwidth]{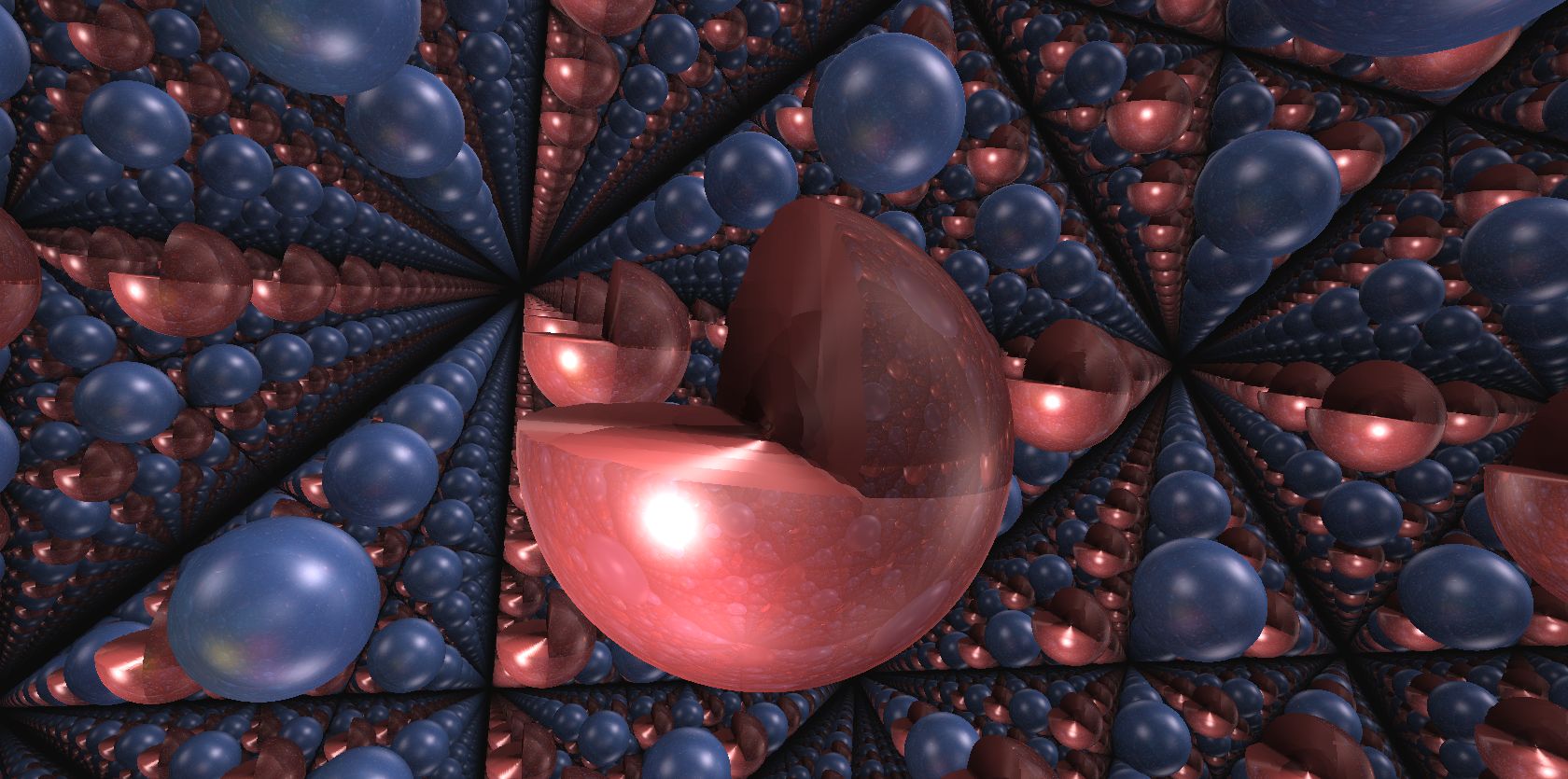}
\caption{Even combining creeping to the boundary with nearest neighbors may not fix all problems. Here the scene consists of a ball that overlaps an edge of a cubical domain $D$.}
\label{Fig:SphereAtEdge}
\end{figure}

\begin{remark}
\label{Rem:SphereFundDom}

We would like to choose a scene for $X/\Gamma$ which illustrates the geometry and topology while having a signed distance function that is very efficient to calculate. We often use the following strategy.
	We delete from a fundamental domain $D$ a large ball (or solid ellipsoid).
	The signed distance function for the complement of a ball in $D$ is
	\begin{equation*}
		\sigma(p) = r - \dist(o,p).
	\end{equation*}
	Here $r$ is a sufficiently large radius so that the deleted ball opens windows into neighboring fundamental domains.
	The corresponding tile for the cubic lattice in $\EE^3$ is shown in \reffig{primitive cell E3 - primitive cell}.
	Depending on the geometry, we may also remove a sphere centered at each vertex of the fundamental domain, as in \reffig{primitive cell E3 - advanced cell}.
\end{remark}

\subsection{Orbifolds and incomplete structures}

In our discussion so far we have assumed that $X/\Gamma$ is a manifold, but in fact nothing is lost by generalizing to orbifolds.
Briefly, an orbifold is a topological space locally modeled on patches of $\RR^n/ G$ for $G$ some finite group of diffeomorphisms. When $G$ is the trivial group, this reduces to the definition of a manifold.
This additional flexibility in the definition allows for certain controllable singularities, such as cone axes (with cone angle $\pi/k$ for some integer $k>0$), while still behaving very similarly to the manifold case.
Indeed, many topological notions such as fundamental groups, covering spaces, and geometric structures carry over directly to orbifolds.
Geometric structures on orbifolds are defined similarly to those on manifolds (see the beginning of \refsec{NonSimplyConnected}), with the main difference being that the action of the fundamental group under the holonomy homomorphism need not be free.
However, as the image $\Gamma$ of the holonomy homomorphism is still discrete, we may find a fundamental domain $D$ for its action and draw pictures of the quotient orbifold $X/\Gamma$ as before.
There is however little change in visual effect:
by~\cite[Corollary~2.27]{CooperHK}, every orbifold with a $(G,X)$ structure is finitely covered by some $(G,X)$ manifold. Thus, up to a finite amount of local information in the scene, the large scale picture will look the same as its manifold cover.

We can generalize still further.  Manifolds and orbifolds have complete geometric structures, meaning that 
 the developing map is a diffeomorphism. This allows the identification $M\cong X/\Gamma$.
The more general notion of \emph{incomplete} $(G,X)$-manifolds are also fundamental objects in geometric topology. Allowing general immersions as developing maps $\widetilde{M}\to X$ naturally captures various kinds of singularities, such as cone axes (where the cone angle can now be any real number) or punctures.
This sort of flexibility is crucial in some core results of geometric topology. For example, the natural extension of the Geometrization Theorem to orbifolds requires the analysis of incomplete hyperbolic structures.
However, incomplete structures are typically difficult to deal with, as the image of the holonomy homomorphism is indiscrete. Previous work here includes hand-drawn examples by Thurston (including two-dimensional structures in chapter three of \cite{thurston_book}, and a three-dimensional drawing reproduced here in \reffig{FromThurstonPaper} from \cite{Thurston98}) and tilings of $\HH^2$ by Bonahon~\cite{bonahonlow}.

Our ray-marching procedure for quotient manifolds extends without change to incomplete structures, allowing the accurate rendering of these as well.
Note that throughout the algorithm, only local data is required: the existence of a fundamental domain $D$ and face pairings $\{\gamma_i^\pm\}$. Both of these exist equally well for incomplete structures.
Here the inside view is quite different than the complete case.
The ability to render incomplete structures may aid in visualization projects, such as animating hyperbolic Dehn surgery or geometric transitions. Indeed, version~2.8 of SnapPy~\cite{SnapPy} implements the inside view of hyperbolic manifolds undergoing hyperbolic Dehn surgery.
However, interpreting these requires more mathematical sophistication than for more familiar manifolds and orbifolds, so we will not focus on them in this paper.

\begin{figure}[htbp]
\centering
\subfloat[A cone axis of angle $2\pi-\varepsilon$ causes double images. These images are Figures 1 and 3 in Thurston's paper \emph{How to See Three Manifolds}~\cite{Thurston98}.]{
\includegraphics[width=0.6\textwidth]{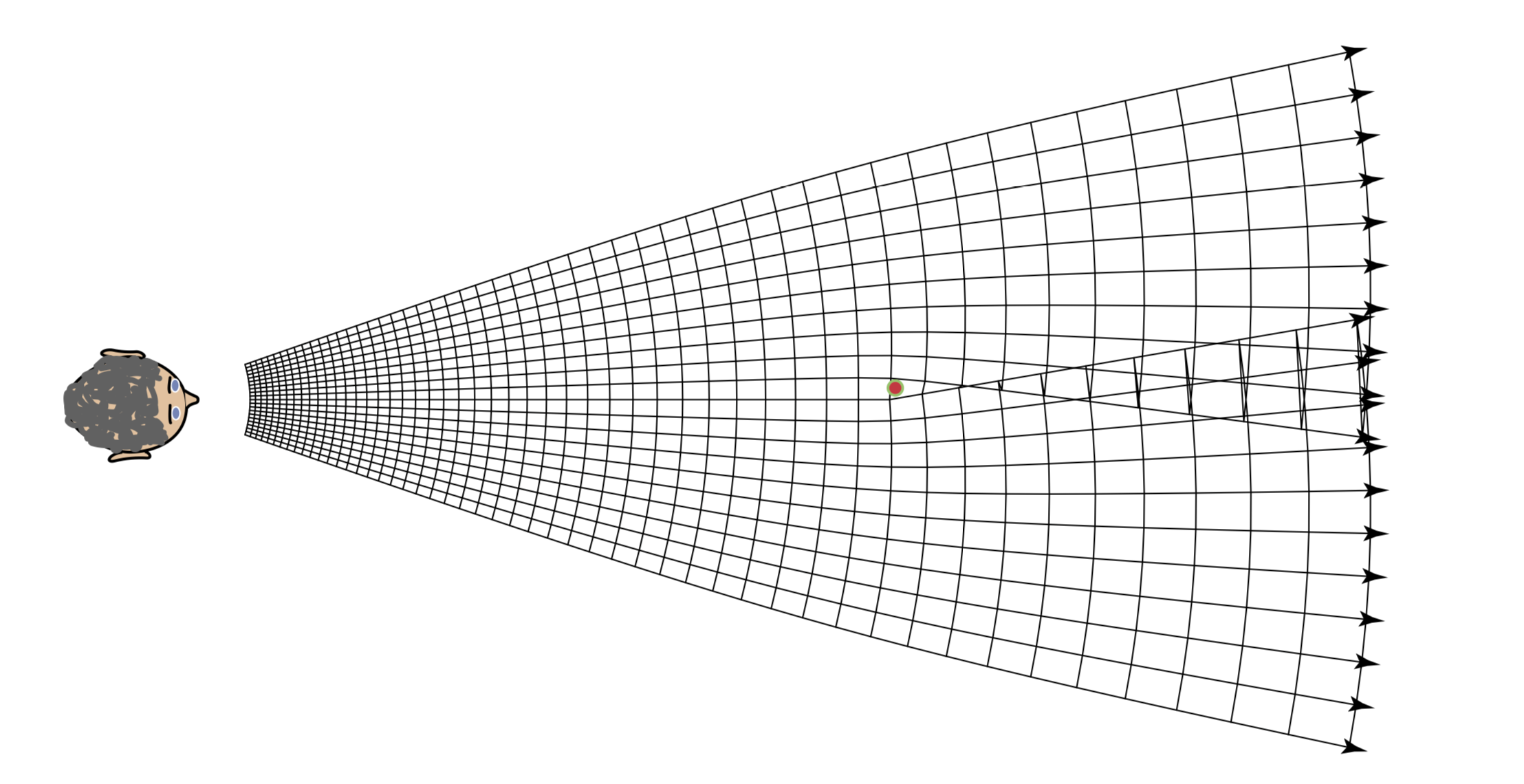}
\includegraphics[width=0.3\textwidth]{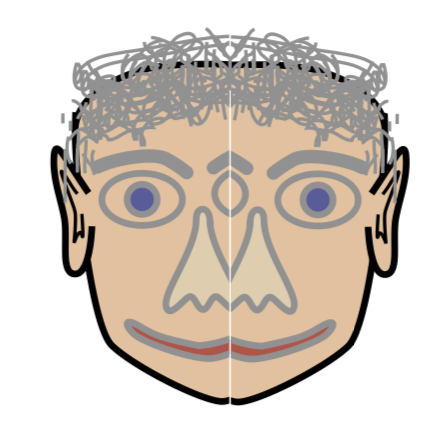}
\label{Fig:FromThurstonPaper}
}
\\
\subfloat[Hyperbolic cone manifold with cone axis of angle $2\pi-\varepsilon$.]{
\includegraphics[width=0.45\textwidth]{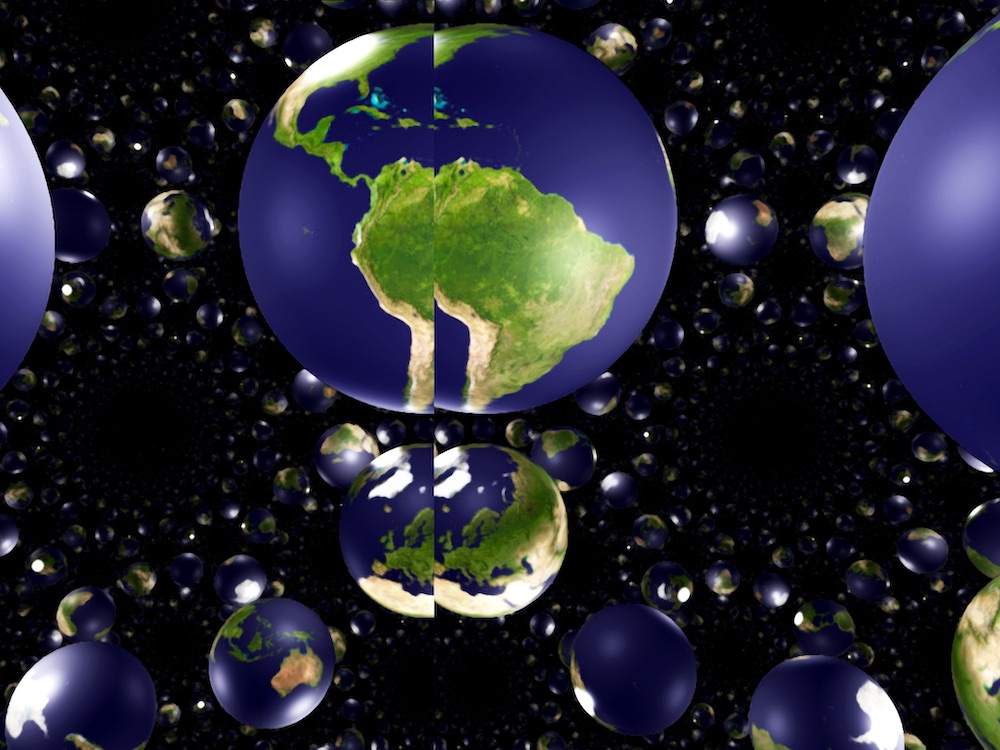}}
\subfloat[Hyperbolic cone manifold with cone axis of angle $2\pi+\varepsilon$.]{
\includegraphics[width=0.45\textwidth]{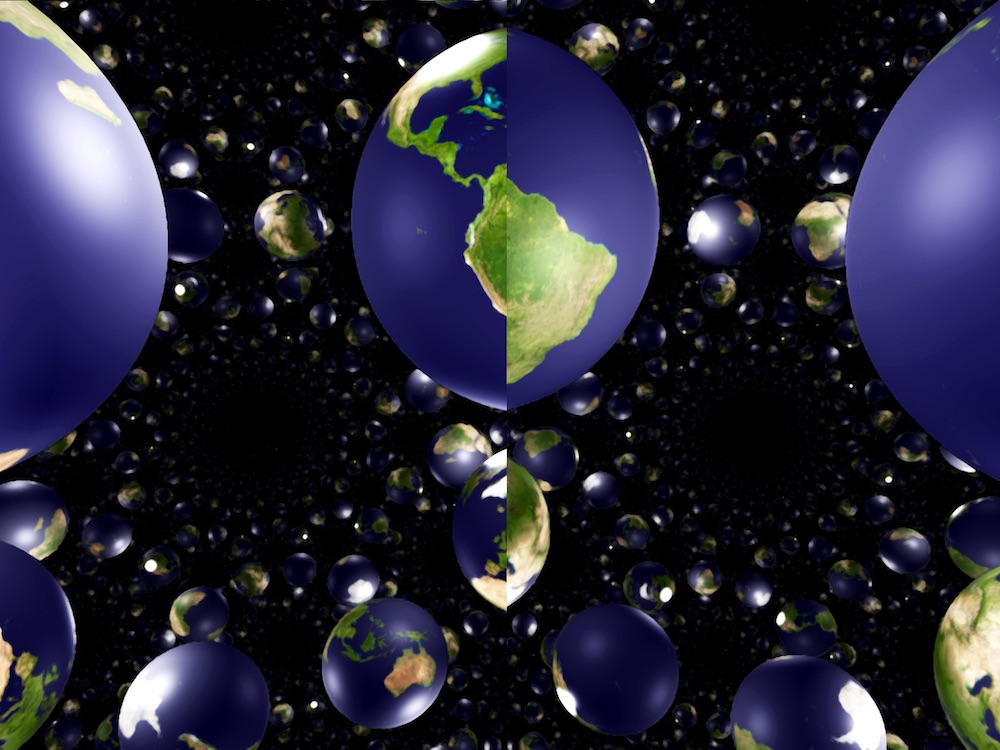}}
\caption{The inside view of a manifold with a cone axis has double imaging of some points when the cone angle is slightly less than $2\pi$, and hidden regions when the cone angle is slightly greater than $2\pi$.
}
\label{Fig:ConeAxis}
\end{figure}

\begin{remark}
We create some of our spaces by directly constructing a fundamental domain $D$, then later figure out which manifold, orbifold, or incomplete manifold it is. In other cases, we start with a desired manifold, or lattice $\Gamma < G$, and have to work out a fundamental domain $D$. For the easier geometries, this generally involves (spherical, hyperbolic, or euclidean) trigonometry. We discuss the construction of fundamental domains for the harder geometries in Sections~\ref{Sec:NilDiscreteSubgroupsFundDoms},~\ref{Sec:SLRDiscreteSubgroupsFundDoms}, and~\ref{Sec:SolDiscreteSubgroupsFundDoms}.
\end{remark}

\section{Lighting} 
\label{Sec:Lighting}
Common physics-based shading techniques in computer graphics (diffuse and specular lighting, reflections, shadows, ambient occlusion, and atmospheric effects) are all computed from geometric data, and so generalize naturally to riemannian geometry. 
Below we briefly review some of these techniques, and the modifications required.  Also see~\cite{NOVELLO202061} for a path-tracing lighting model in the constant curvature spaces.

The effect from each light source in the scene can be computed separately, and the final color determined through a weighted (by intensity) average of each light's contribution.
Thus it suffices to describe the contribution of a single light source.
However, in the geometries with positive sectional curvatures ($S^3,S^2\times\EE$, Nil, Sol, $\SLR$), non-uniqueness of geodesics may cause even a single light source to illuminate an object from multiple directions.
As these individual contributions also combine linearly to the total, we may further reduce the problem to understanding single-source lighting from a single direction at a time.

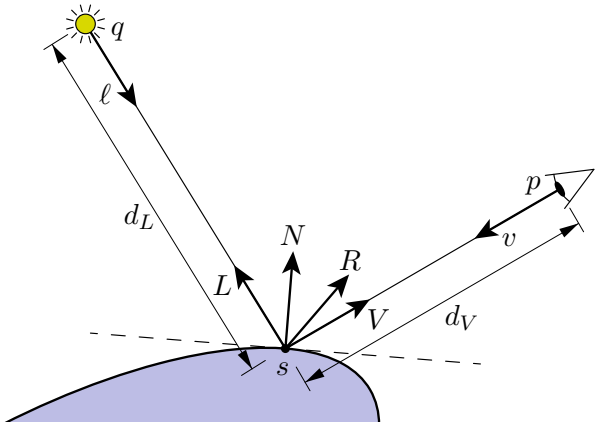
\begin{figure}[htbp]
\centering
\labellist
\small\hair 2pt
\pinlabel $q$ [l] at 102 335
\pinlabel $\ell$ [tr] at 110 293
\pinlabel $d_L$ [tr] at 148 190
\pinlabel $L$ [tr] at 214 129
\pinlabel $s$ [t] at 254 55
\pinlabel $N$ [b] at 263 147
\pinlabel $R$ [bl] at 299 127
\pinlabel $V$ [tl] at 323 100
\pinlabel $d_V$ [tl] at 390 103
\pinlabel $v$ [tl] at 439 165
\pinlabel $p$ [r] at 480 203
\endlabellist
\includegraphics[width=0.65\textwidth]{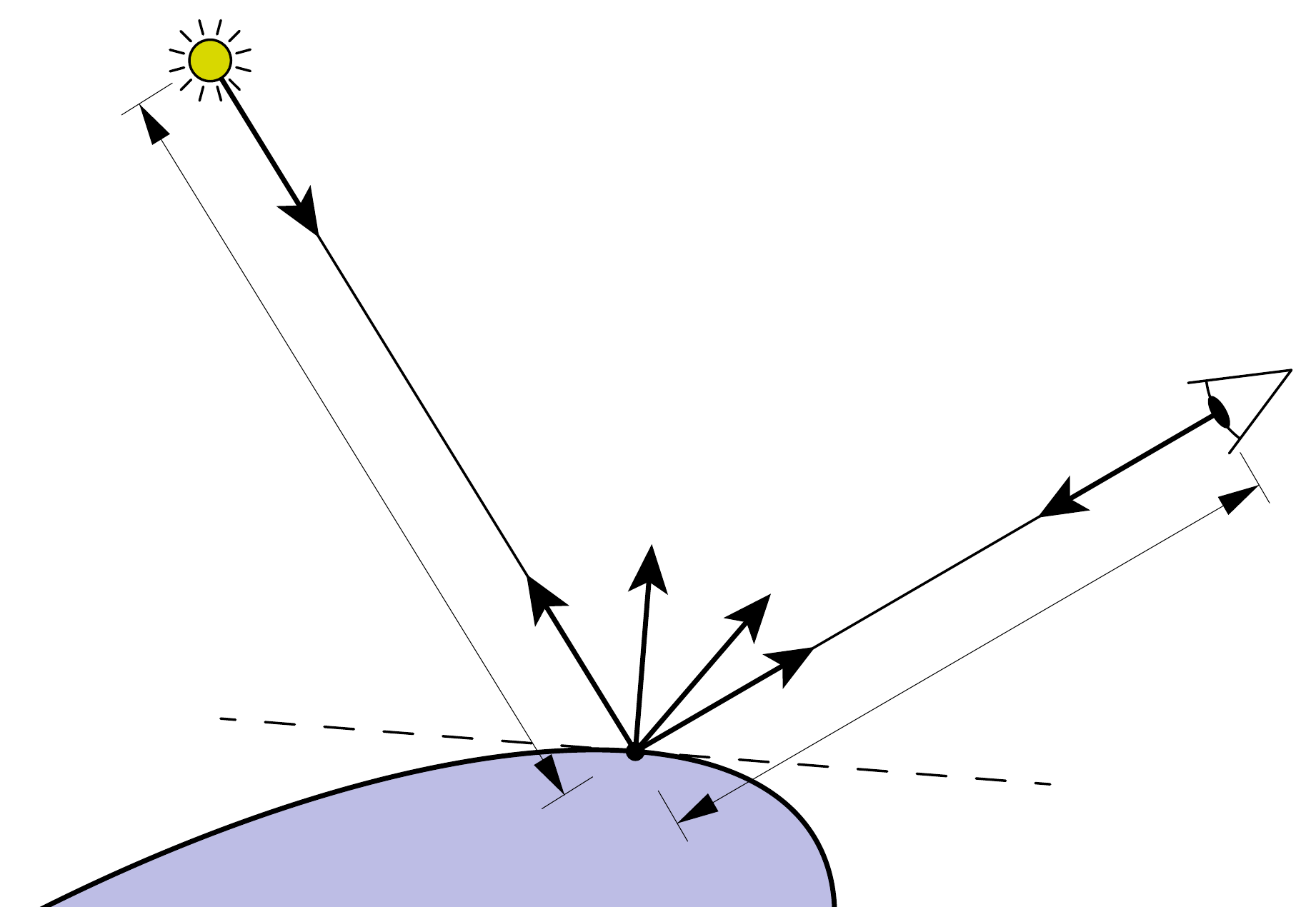}
\caption{The geometric data required to calculate the color observed when looking from the point $p$ in the direction $v\in T_p X$ at a point $s$, lit by a light at a point $q$ from the direction $L \in T_s X$. }
\label{Fig:LightingSetup}
\end{figure}

To fix notation, let $S$ be a scene in $X$ given by a signed distance function $\sigma$, lit by a light source at $q\in X$.  
See \reffig{LightingSetup}.
Let $C_s$ be the base color of the point $s$ of the scene, (represented as a three-vector storing its red-green-blue components), let $C_{\rm light}$ be the color of light source, and $I_{\rm light}$ be its intensity.
Now suppose that we are at a point $p\in X$, looking in the direction $v\in T_p X$. Assume that this line of sight ends by impacting the point $s\in S$ of the scene. To compute the aforementioned lighting effects, we need the following data:

\begin{itemize}
\item $N\in T_sX$: unit outwards normal to $\bdy S$ at $s$,
\item $L\in T_s X$: unit vector at $s$ pointing to $q$,
\item $R\in T_s X$: reflection of $-L$ with respect to $N$,
\item $V\in T_s X$: unit vector at $s$ pointing to $p$,
\item $v\in T_p X$: unit vector at $p$ pointing to $s$,
\item $\ell\in T_q X$: unit vector at $q$ pointing to $s$,	
\item $d_L$: distance from $s$ to $q$ along the geodesic with tangent $L$, 
\item $d_V$: distance from $s$ to $p$ along the geodesic with tangent $V$, and
\item $I_{L}$: the light intensity experienced at $s$ from the direction $L$.
\end{itemize}
\noindent
Here we employ the convention that vectors in the tangent space at $s$ are written in upper case, while vectors in tangent spaces at other points are written in lower case.

\begin{remark}
The base colour $C_s$ for a point $s$ of the scene can be a single colour for each object, or we can texture objects in a more complicated way.
For example, we sometimes texture balls 
as the Earth. This provides a globally 
recognized coordinate system and allows one to infer the final endpoints of geodesics leaving your eye. See \reffig{EarthTexture}.
\end{remark}

\begin{figure}[htbp]
\centering
\subfloat[A ball in spherical geometry: more than half of its surface is visible.]{
\includegraphics[width=0.4\textwidth]{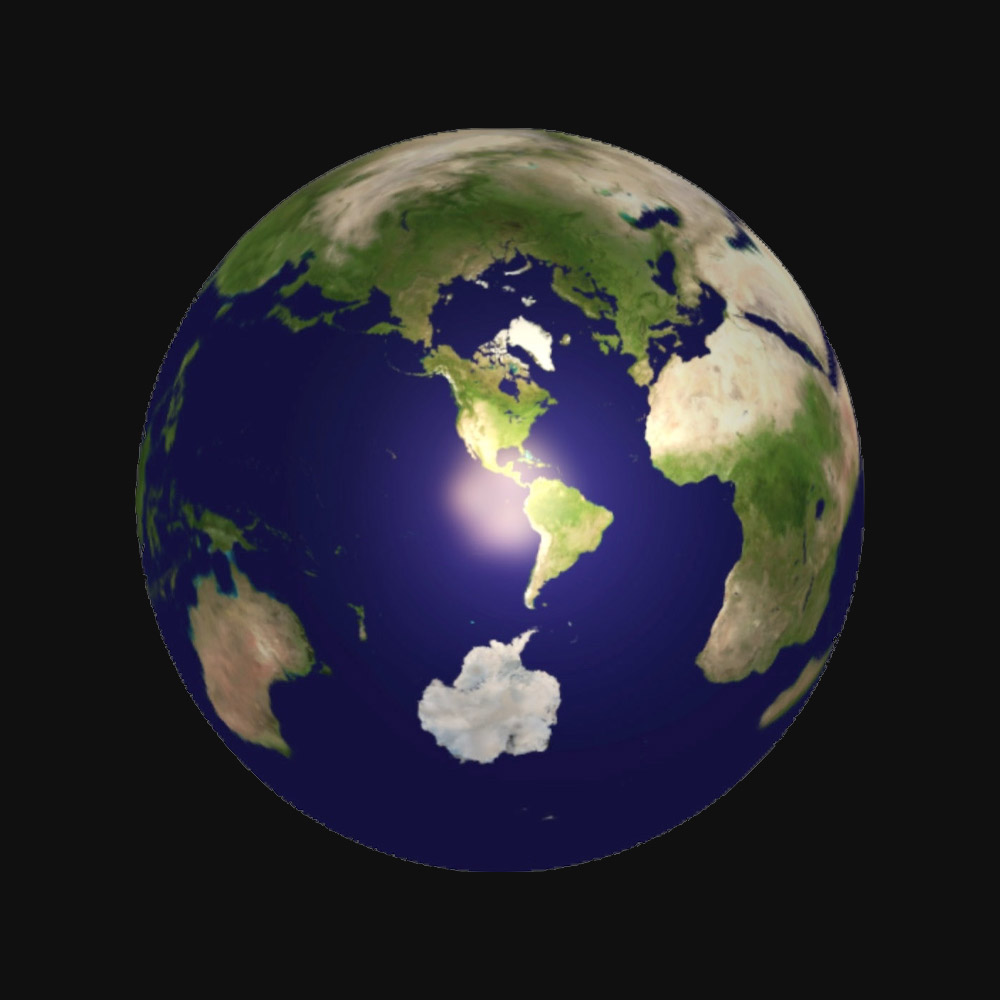}
\label{Fig:SphEarth}
}
\quad
\subfloat[A ball in Nil geometry: the non-uniqueness of geodesics causes a triple image of South America.]{
\includegraphics[width=0.4\textwidth]{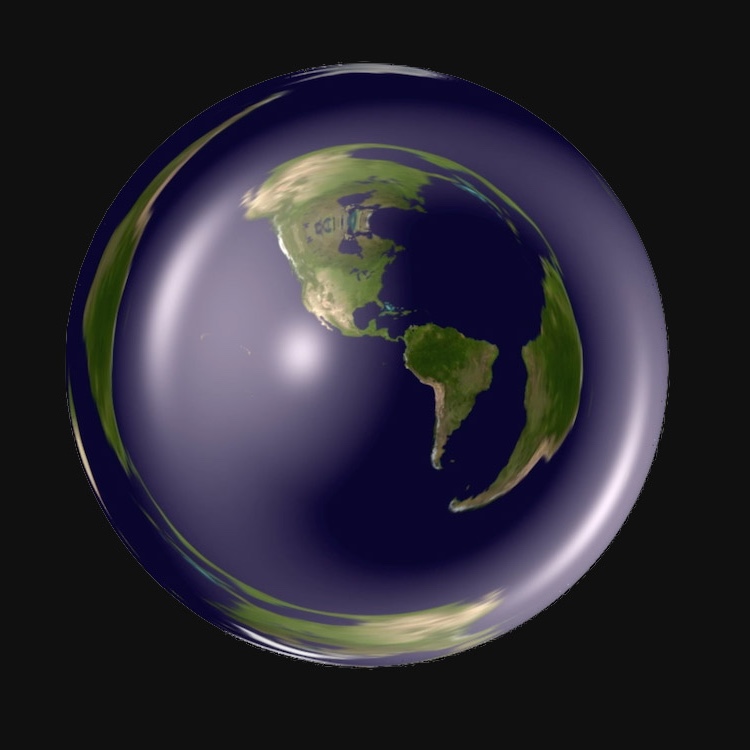}
\label{Fig:NilEarth}
}
\caption{Balls textured as the Earth.}
\label{Fig:EarthTexture}
\end{figure}

\subsection{Phong lighting model}

An empirical formula for accurate diffuse and specular reflection in computer graphics was published by Phong in his 1975 dissertation~\cite{Phong} and now bears his name.
The \emph{Phong lighting model} (also called the Phong reflection model) decomposes the total color of the surface as a sum of three components: \emph{ambient, diffuse} and \emph{specular}.
The ambient contribution is simply the base color $C_s$ of the object at $s$.
The remaining two terms are proportional to the light color $C_{\rm light}$ and the intensity $I_L$ of the light source, as well as a third geometric quantity, as follows.
Diffuse lighting is also proportional to the cosine of the angle between the light direction and the surface normal.
Specular reflection is proportional to some power of the cosine of the angle between the viewer and reflected ray directions. This power is a parameter controlling the ``shininess'' of the material of the object.
When either of these angles is obtuse, the corresponding lighting contribution is taken to be zero.
This allows us to express the total lighting contribution of Phong lighting using the riemannian metric at $s$:

\begin{equation}\mathrm{Phong}(N,L,R,V,I_L)=k_{\rm amb}C_s+\big(k_{\rm diff} \langle N,L\rangle +k_{\rm spec}\langle R,V\rangle^\alpha\big)I_L C_{\rm light},\end{equation}

\noindent
where the constants are chosen to satisfy $k_{\rm amb}+k_{\rm diff}+k_{\rm spec}=1$. These control the relative contribution of each of these factors.

\begin{figure}[htbp]
\centering
\subfloat[Diffuse lighting.]{
\includegraphics[width=0.45\textwidth]{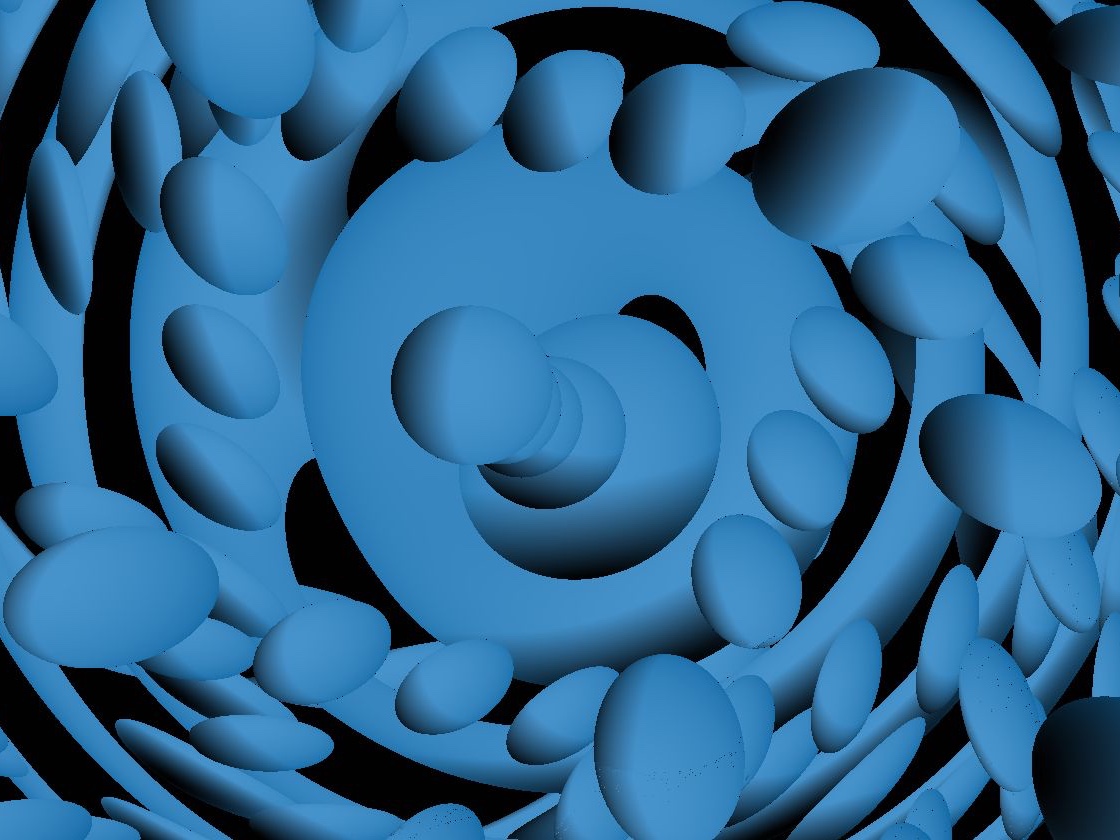}
}
\quad
\subfloat[Specular highlights.]{
\includegraphics[width=0.45\textwidth]{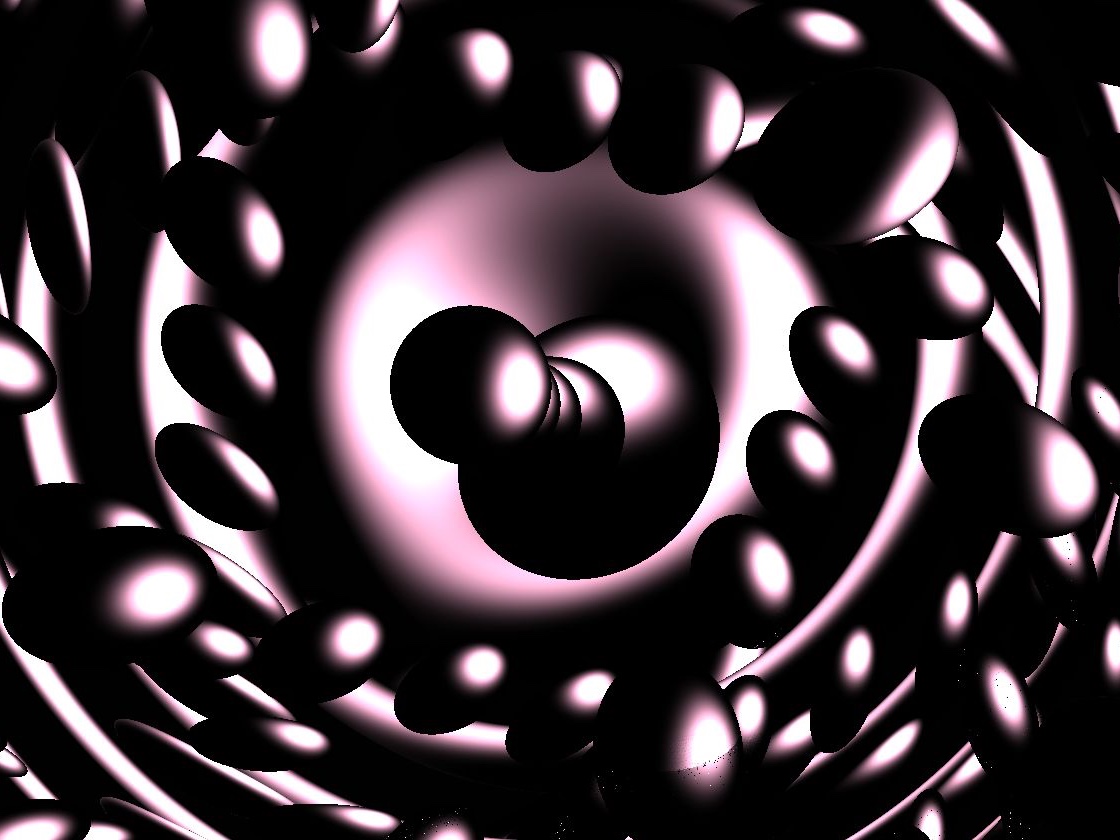}
}

\subfloat[Phong model: ambient, diffuse, and specular.]{
\includegraphics[width=0.45\textwidth]{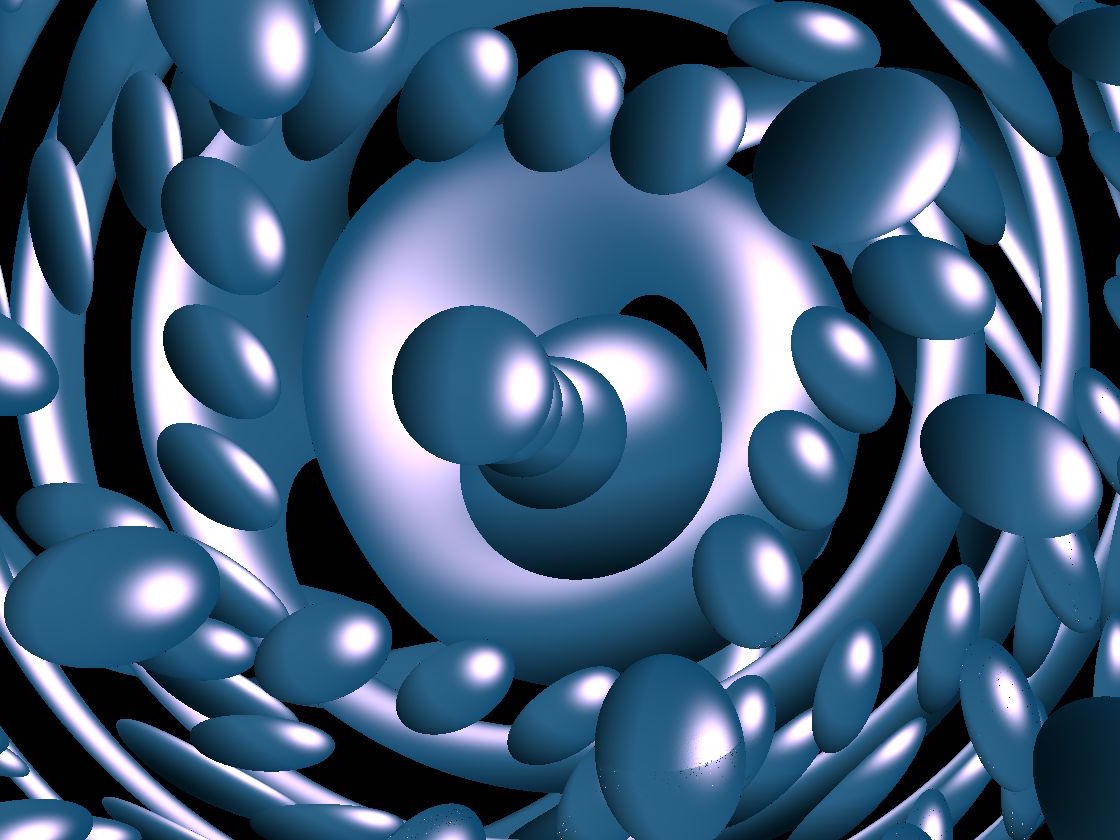}
}
\quad
\subfloat[Phong lighting with multiple light sources provides realistic depth cues.]{
\includegraphics[width=0.45\textwidth]{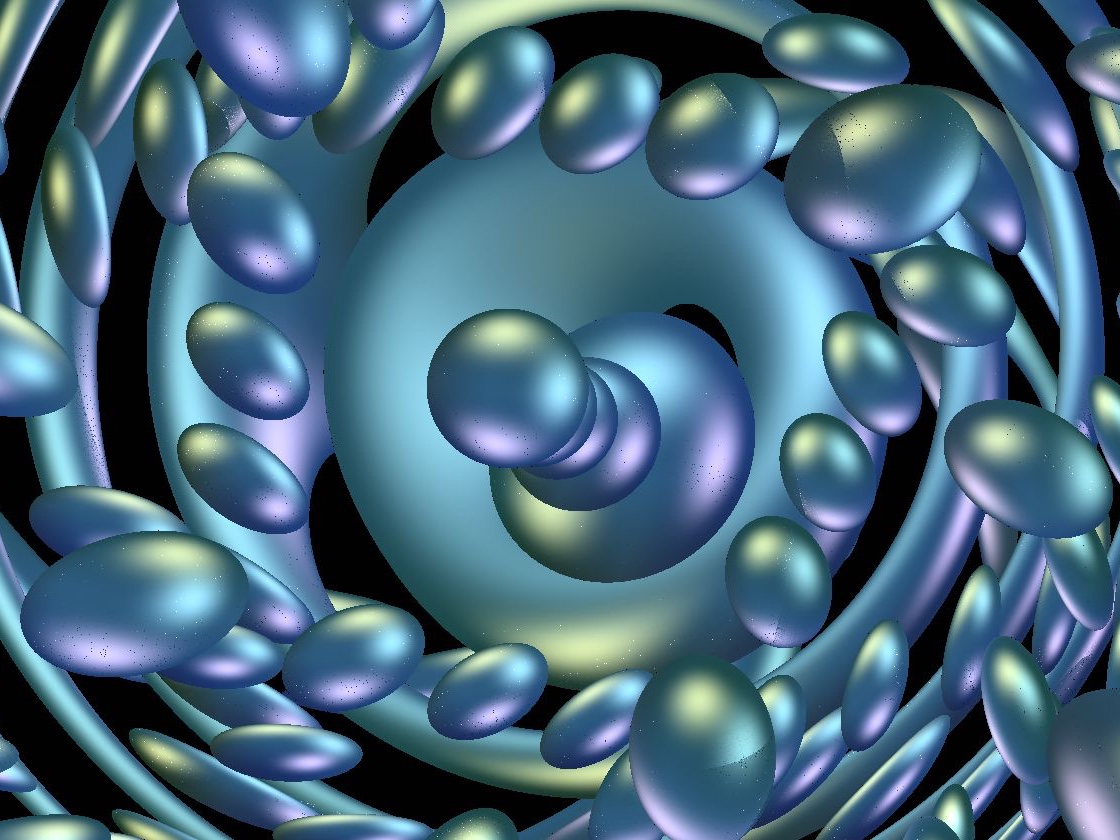}
}
\caption{A collection of balls in Nil geometry.}
\label{Fig:QuotientLighting}
\end{figure}

\begin{remark}
Phong justifies his model empirically, by comparing a render with a real-life photograph of a (euclidean) scene. We use his model far outside of the setting in which it was designed for, so one could question whether or not it produces accurate results in our non-euclidean spaces. A reasonable test would be to compare our results with a more physically correct ray-tracer. 
\end{remark}

\subsection{Shadows}

Phong lighting calculates the contribution of the observed color at $s$ due to a light source in the direction $L$ using only local computations in $T_sX$.
While efficient, this ignores the existence of other objects in the scene, effectively rendering them transparent to the lighting calculation.

Happily there is a simple solution to detecting objects which block the path from $s$ to the light: simply ray-march starting at $s$ in the direction towards the light and see if you hit anything.
If you do then there is no need to calculate the Phong lighting contribution for that light/direction, as $s$ is in shadow.
When modeling lights as point sources, this produces \emph{hard} shadows. 
Realistic light sources which emit light over an area instead produce \emph{soft} shadows, as there are points in space where the light source is only partially obscured.
While modeling an extended source is computationally demanding, a multitude of empirical formulas for approximating soft shadows with point source lights have been developed in computer graphics.
We briefly discuss a solution particularly well suited for ray-marching below. See~\cite{QuilezSoftShadows} for more details.

Instead of a simple binary value, the shadow is modeled as a scaling factor to be multiplied by the Phong lighting contribution, smoothly interpolating between zero and one.
To compute this value, we track the distance of the light ray from other objects in the scene as we follow it backwards from $s$ in the direction $L$.
Let $\gamma \from [0,T] \to X$ be the arc length parametrized geodesic from $s$ to the light at $q$ with initial tangent $L$.
The degree of shadow imparted by the surrounding scene at a point $\gamma(t)$ is modeled by the distance of $\gamma(t)$ from an object in the scene, normalized by the distance traveled from $s$.
The total degree of shadow is proportional to the minimal value of this ratio over the path, or

\begin{equation}
\mathrm{Shadow}(s,L)=\min\left\{1,K \frac{ \sigma(\gamma(t))}{t}: t\in [0, T] \right\}.
\end{equation}

Here $K\geq 1$ is a parameter determining softness. As $K\to\infty$ this reproduces the hard shadows above.
In practice, we approximate this by computing this ratio at each step of the ray-march from $s$ to $q$, and then take the minimum.

\begin{figure}[htbp]
\centering
\subfloat[Hard Shadow.]{
\includegraphics[width=0.43\textwidth]{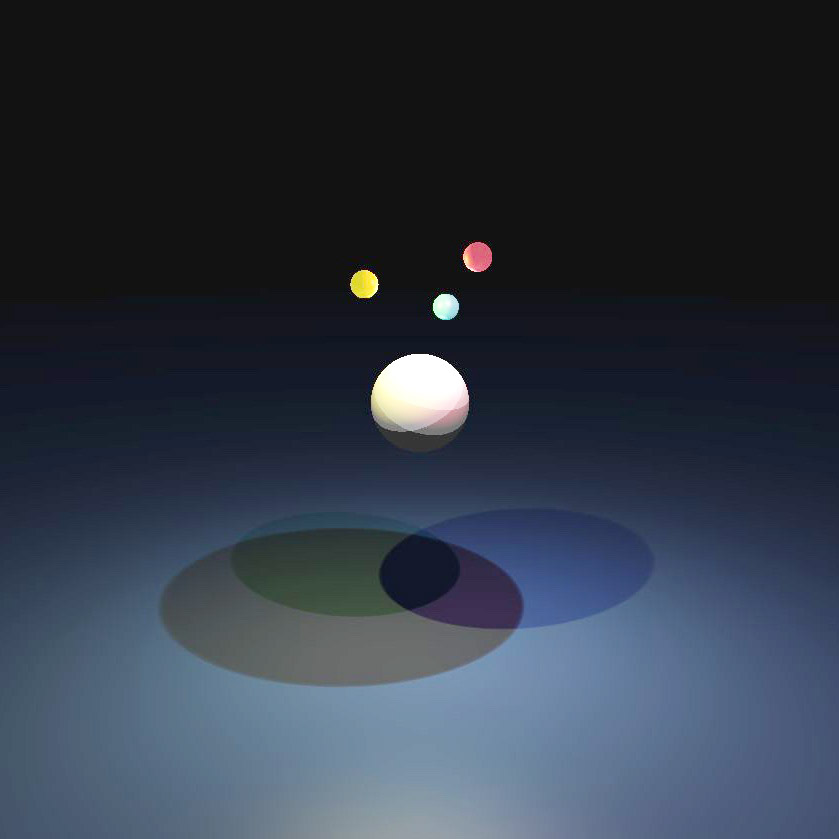}
\label{Fig:HardShadow}
}
\quad
\subfloat[Soft Shadow ($K=5$).]{
\includegraphics[width=0.43\textwidth]{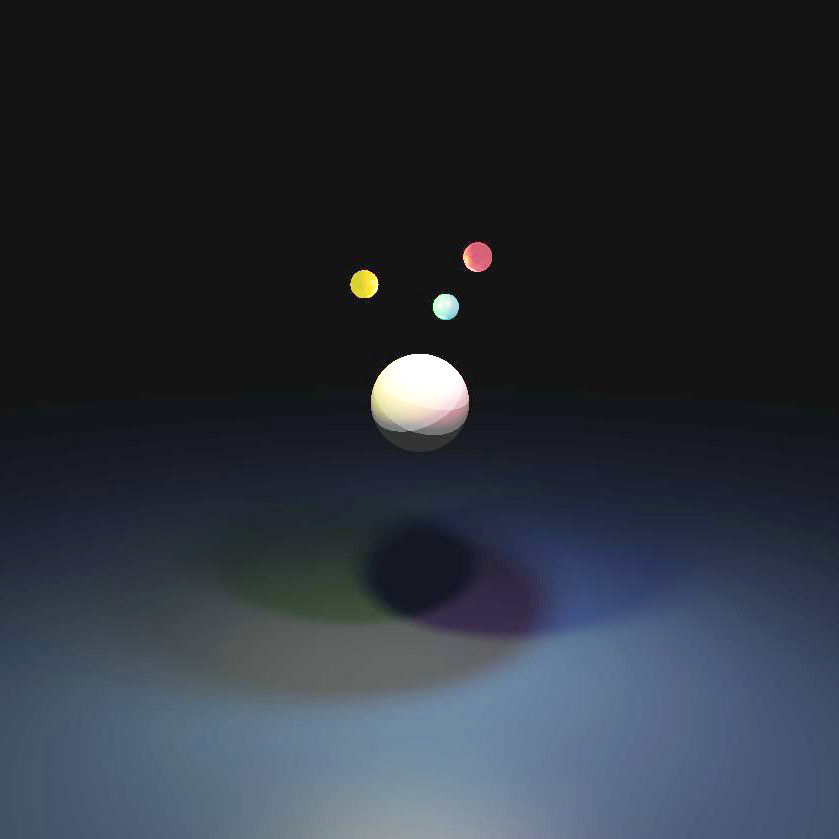}
\label{Fig:SoftShadow}
}
\caption{A comparison of different shadow rendering techniques with a sphere lit by three light sources above a plane in euclidean space.}
\label{Fig:RenderShadows}
\end{figure}

\subsection{Atmospheric Effects}
The fact that computing the total distance traveled along a path is trivial in a ray-marching application makes the above soft shadow approximation efficient.
This almost free availability of path lengths also lends itself well to volumetric rendering: accounting for contributions to the lighting from atmospheric media encountered along the path.
The simplest such effect, \emph{distance fog}, is computationally inexpensive to implement and provides helpful distance cues in complex scenes. This replaces a fraction of the color of a pixel with a ``fog'' color, $C_{\rm fog}$, depending on the distance the ray travels before hitting an object.

In many computer graphics applications, this fraction is linear in path length. This has the advantage that there is a distance at which all of the pixel is given the fog color, and no further calculation is necessary.
However, a physically realistic model based on scattering along a path (the Beer-Lambert law in physics) implies that the fraction is actually exponential in the path length.
We give these two models below.
\begin{equation}
\mathrm{Fog}(d_V)=1-\min\left\{\frac{d_V}{K},1\right\}, 
\hspace{1cm}
\mathrm{Fog}(d_V)=e^{-K d_V}	
\end{equation}

Here $K>0$ is a constant determining the rate of scattering.
Each of these are extremely easy to implement, as they are standard functions of the already-available path length.

\begin{figure}[htbp]
\centering

\subfloat[Without fog.]{
\includegraphics[width=0.29\textwidth]{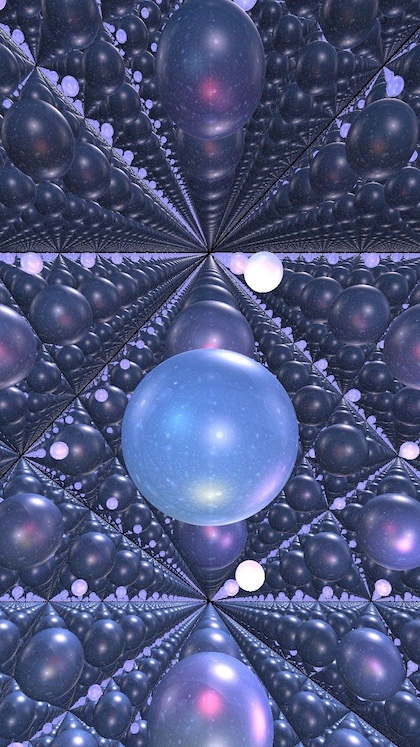}
\label{Fig:NoFog}
}
\;\;
\subfloat[With linear fog.]{
\includegraphics[width=0.29\textwidth]{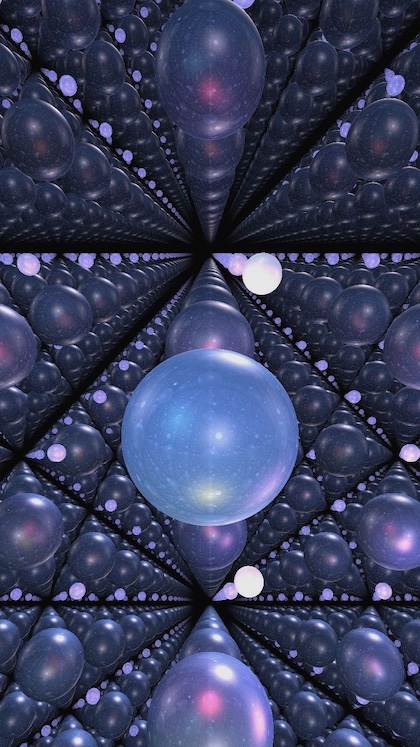}
\label{Fig:Fog}
}
\;\;
\subfloat[With exponential fog.]{
\includegraphics[width=0.29\textwidth]{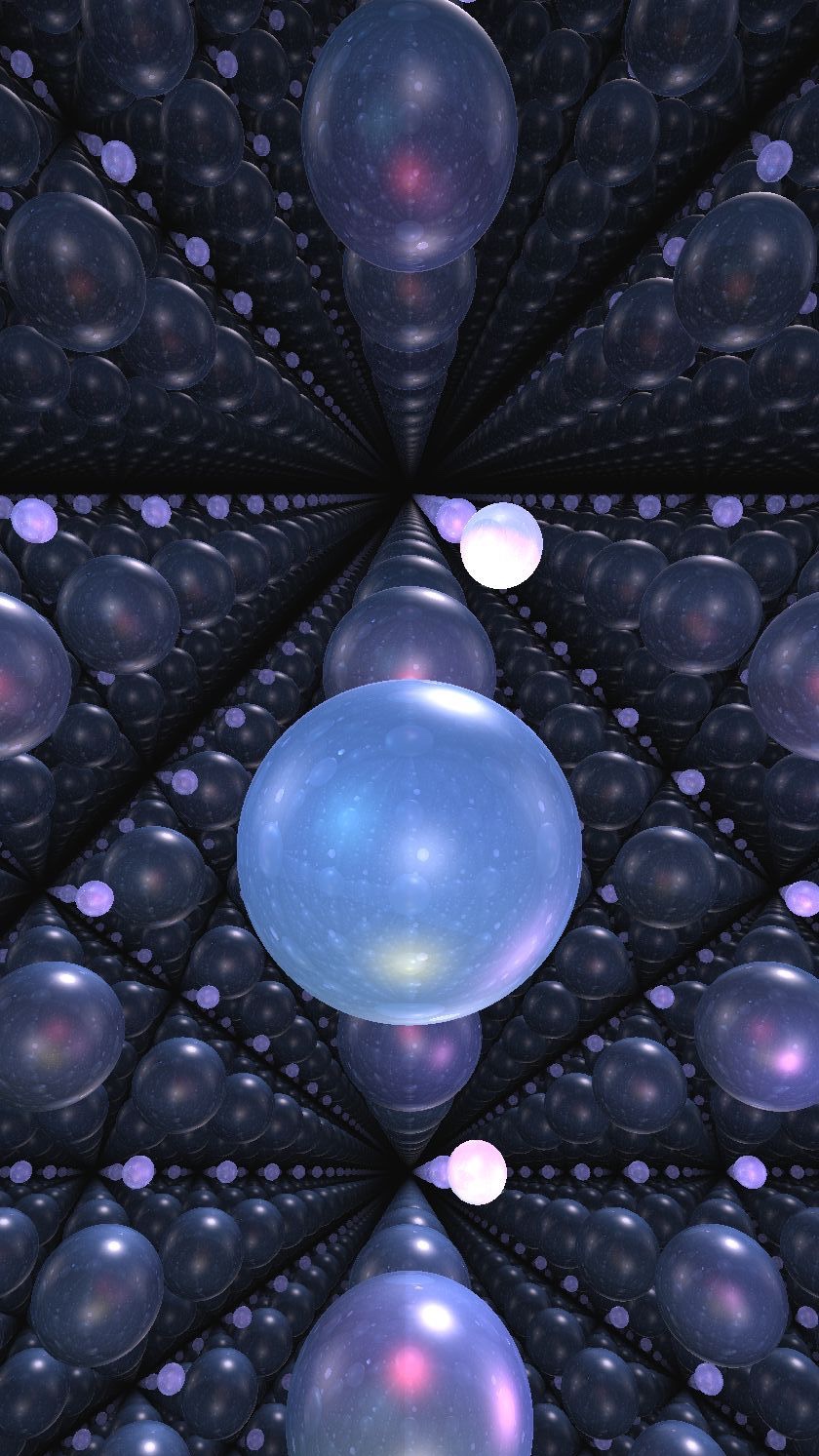}
\label{Fig:Fog}
}
\caption{A lattice of balls in euclidean space. }
\label{Fig:AllDirections}
\end{figure}

Combining the contributions from both shadows and fog, we obtain the following.

\begin{equation}
\begin{split}
{\rm Col} = \,& {\rm Fog}(d_V)\cdot{\rm Shadow}(L)\cdot {\rm Phong}(N,L,R,V,I_L)  + \\ &  (1 - {\rm Fog}(d_V))\cdot C_{\rm fog}
\end{split}
\end{equation}
	
Outside of this section, our in-space images use exponential fog unless otherwise noted. We always set $C_{\rm fog}$ to be black.

\subsection{Reflections}

It is also relatively simple to allow for reflective materials in ray-marching,
Upon impacting a reflective surface at $s_1$, one simply initiates a new ray-march from $s_1$ in the direction of the reflected ray. This ray-march may impact another object, at $s_2$ say. If so, we may reflect again.
Computing the observed colors ${\rm Col}_i$ at the points $s_i$ as above, the final color is an average, weighted by the reflectivity $r_i\in[0,1]$ of the material at $s_i$.
This can be carried out iteratively with no additional difficulty (other than increase in computation time). 
The weighted averages for one and two reflections are given below.

$$(1-r_1){\rm Col}_1+r_1{\rm Col}_2\hspace{1cm}
(1-r_1){\rm Col}_1+r_1\left((1-r_2){\rm Col}_2+r_2{\rm Col}_3\right)$$

\begin{figure}[htbp]
\centering
\subfloat[No reflections.]{
\includegraphics[width=0.90\textwidth]{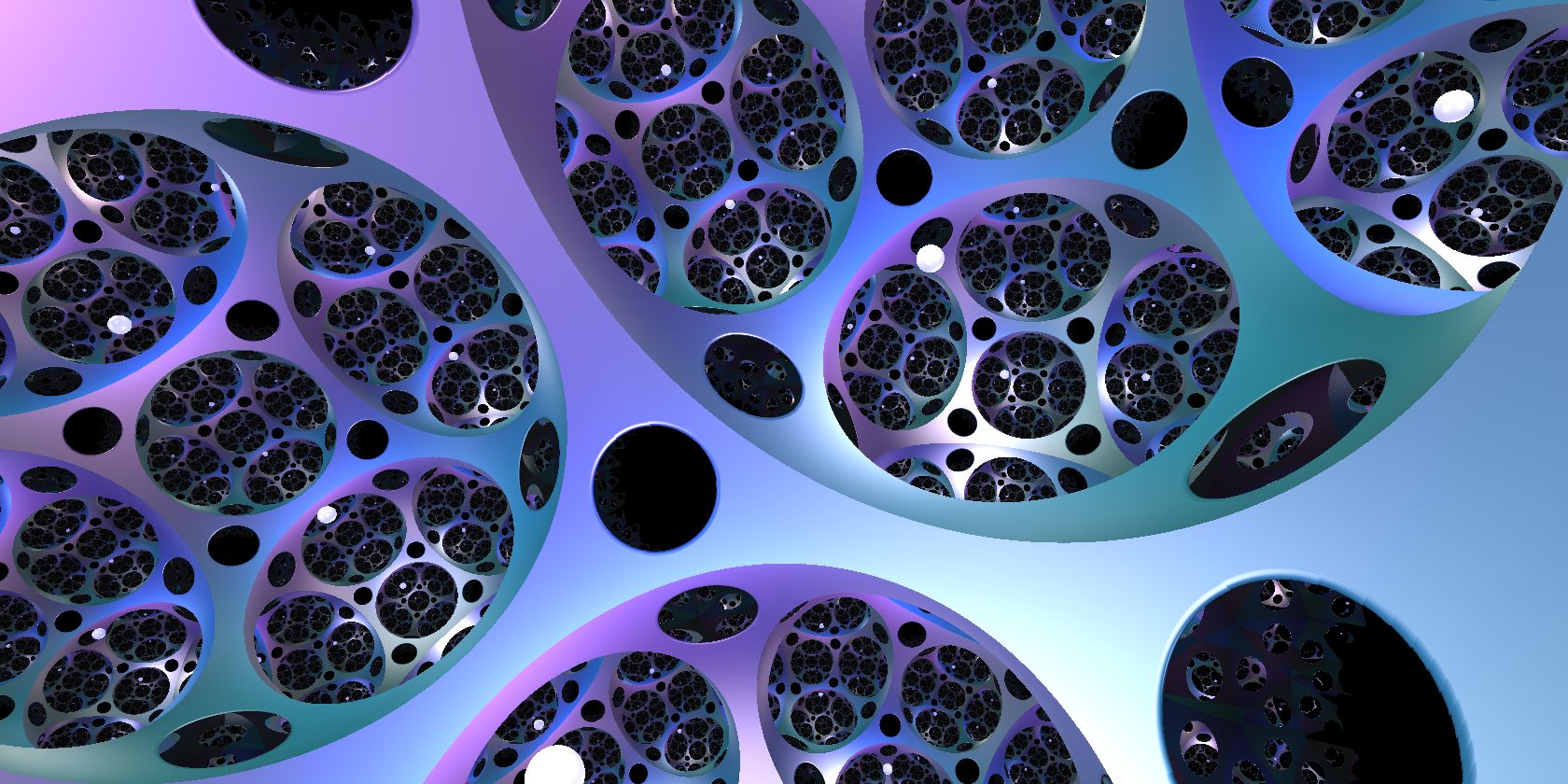}
\label{Fig:NoReflect}
}\
\subfloat[A single reflection pass.]{
\includegraphics[width=0.90\textwidth]{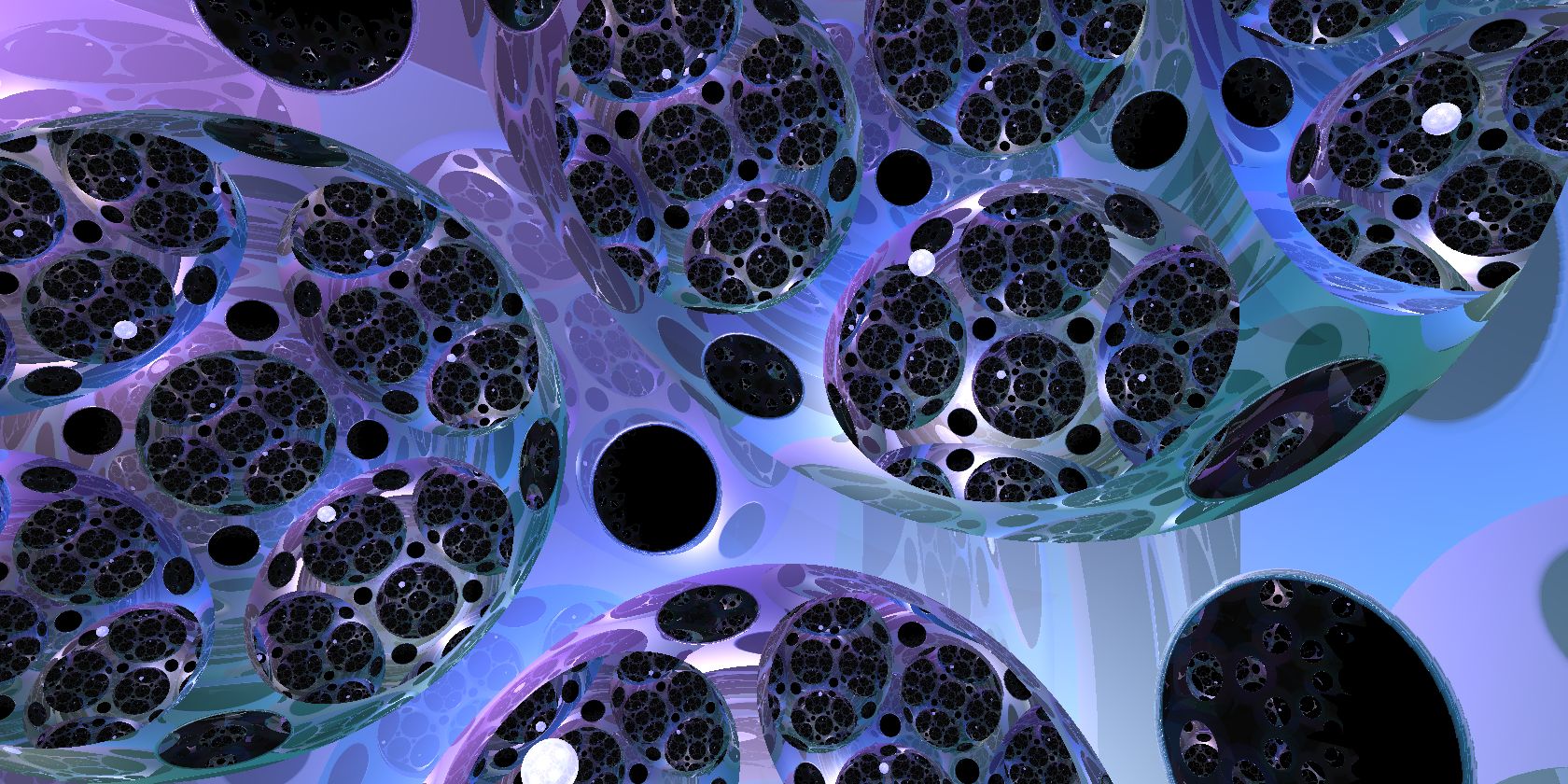}
\label{Fig:OneReflect}
}\
\subfloat[Two reflection passes.]{
\includegraphics[width=0.90\textwidth]{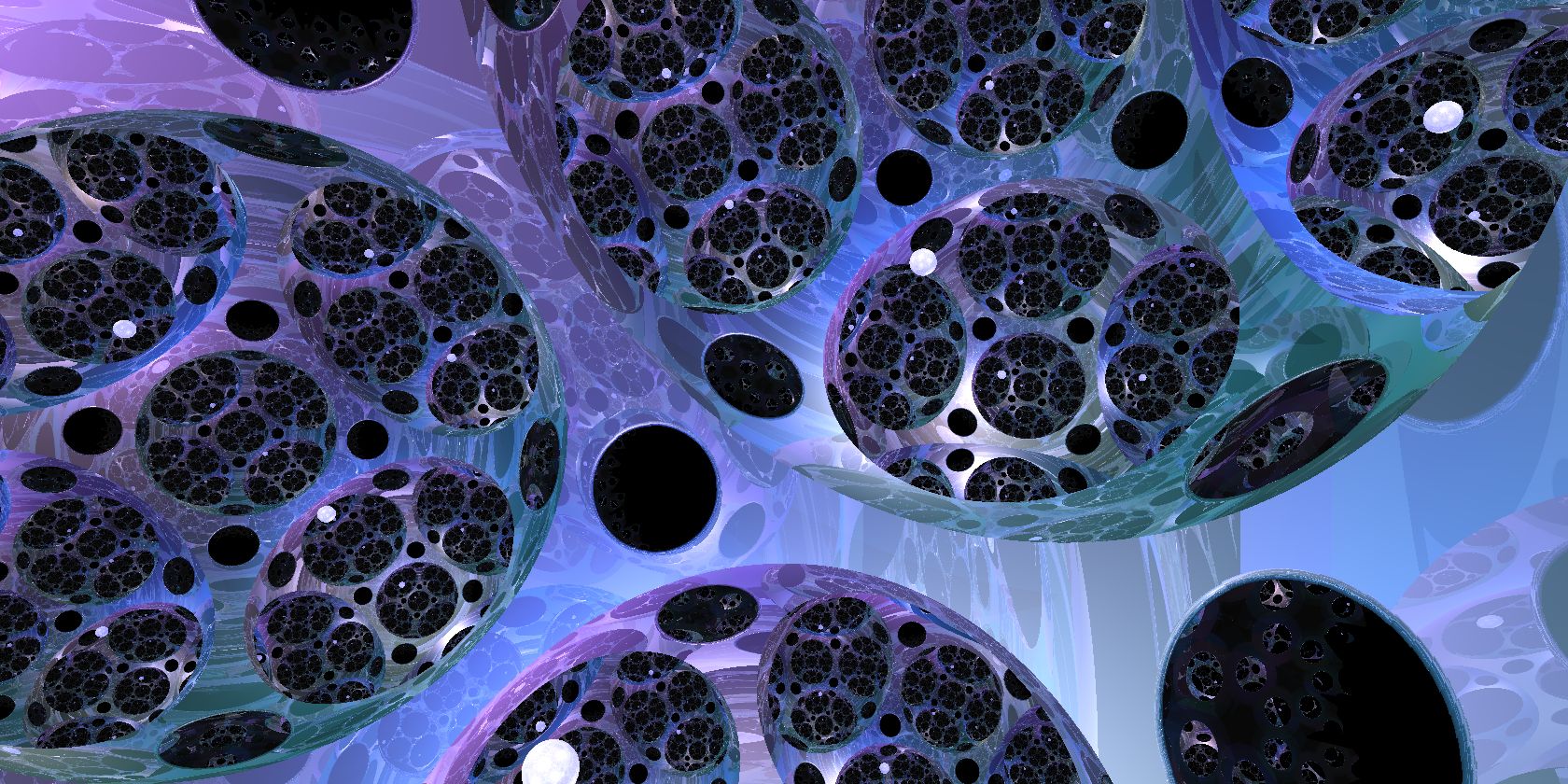}
\label{Fig:TwoReflect}
}\
\caption{Reflections in a complicated scene in hyperbolic space.}
\label{Fig:AllDirections}
\end{figure}

\subsection{Computing the necessary geometric quantities}

As the above sections illustrate, it is relatively straightforward to calculate accurate lighting, given the geometric quantities listed at the beginning of this section.  Here we turn to the issues involved in computing these.
Some of these quantities are available directly from the ray-march itself.

\subsubsection{Computing $v$}
The vector $v\in T_pX$ pointing from the viewer to the observed point $s\in S$ is the initial tangent vector for the ray-march.

\subsubsection{Computing $V$}
The vector $V\in T_s X$ pointing back at the viewer is the negation of the final tangent vector for the ray-march.

\subsubsection{Computing $d_V$}
The distance $d_V$ from the viewer to the observed point is the path length returned by the ray-march.

Other quantities require further computation.

\subsubsection{Computing $N$}
The unit surface normal $N\in T_sX$ is computable directly from the signed distance function $\sigma$. It is the gradient vector $\grad\sigma(s)$ dual to $d_s\sigma$ via the riemannian metric.
As in multivariable calculus, fixing a basis $\{f_1,f_2,f_3\}$ for $T_s X$, this is approximated for some small $\varepsilon>0$ by

$$\grad\sigma(s)\simeq \sum_{i=1}^3 \frac{\sigma(s+\varepsilon f_i)-\sigma(s-\varepsilon f_i)}{2\varepsilon}f_i$$

While in principle any choice of basis of $T_sX$ suffices, even slight discontinuities in the normal field over a surface are plainly visible in the output of the Phong lighting model.
To prevent this source of error, we make a globally continuous choice of basis by selecting a section of the frame bundle.
A simple construction of such a section follows from the transitivity of the $G$-action.
Let $B\subset G$ be a subset (not necessarily a subgroup) of the isometry group such that the orbit map $B\to X$ defined by $g\mapsto g.o$ is a diffeomorphism.
(For example, when $G$ has a subgroup acting simply transitively, we may take this as $B$.)
The inverse of this orbit map provides a section $X\to G$ with image $B$, sending $s\in X$ to $g(s)$.
We promote this to a section of $\mathcal{O}X$ by choosing an orthonormal frame $f  = \{f_1,f_2,f_3\}$ for $T_o X$ and translating by the $G$-action.
This assigns to $s\in X$ the frame $d_o g(s) f$.

\subsubsection{Computing $R$}
The unit normal provides a means of reflecting rays in the surface.
Given any vector $U\in T_sX$ we may compute its reflection in the surface by
$$\mathrm{Refl}(U)=U-2\langle U,N\rangle N$$
Thus, given the direction to the light source $L\in T_s X$, we may find the final direction needed for Phong lighting, $R=-\mathrm{Refl}(L)$.
This leaves only four quantities to be computed, all dealing with the location of the light source; two directions $L,\ell$ and two scalars $d_L,I_L$.
These require global information about the geometry of $X$. We discuss this next.

\subsection{Computing lighting directions, $L$, $\ell$, and distance $d_L$}
\label{Sec:Lighting directions}
Calculating the direction $L$ in which a light is visible from a point on the surface (and the other related quantities)
cannot be reduced to linear algebra in some tangent space: it involves the global geometry of $X$.
This requires a procedure that takes two points $s,q\in X$ and returns the set of \emph{lighting pairs} $\calL_s(q)\subset T_s X \times \RR_+$. Here each element $(L,d_L) \in \calL_s(q)$ represents the direction, $L$, of a geodesic $\gamma$ connecting $s$ to $q$, and the length, $d_L$, of the geodesic segment $\gamma$ connecting $s$ to $q$.
Since we use explicit formulas for the geodesic flow, one can directly compute from $(L, d_L)$ the direction $\ell \in T_qX$ and the reverse geodesic $\gamma'$ joining $q$ to $s$. 
In all cases, we may use the homogeneity of $X$ to reduce the problem to understanding geodesics from the origin, and focus on calculating the lighting pairs $\calL_o(q)$ for $q\in X$.
However, for the convenience of the reader, in the isotropic and product geometries we provide formulas for a general point $s \in X$.

In geometries with nonpositive sectional curvature, geodesics are unique by Cartan-Hadamard. Thus for each $s, q\in X$ the
set $\calL_s(q)$ is a singleton. 
In other geometries $\calL_s(q)$ may be a singleton, finite, countably infinite, or uncountably infinite, depending on $q$.
See Figure \ref{Fig:AllDirections} for examples of lighting along multiple geodesics in $S^3$ and $S^2\times\EE$.
There is no uniform approach to calculate $\calL_s(q)$, so we deal with this computation in later, geometry-dependent sections of this paper.

\begin{figure}[htbp]
\centering
\subfloat[Only the shortest geodesic.]{
\includegraphics[width=0.45\textwidth]{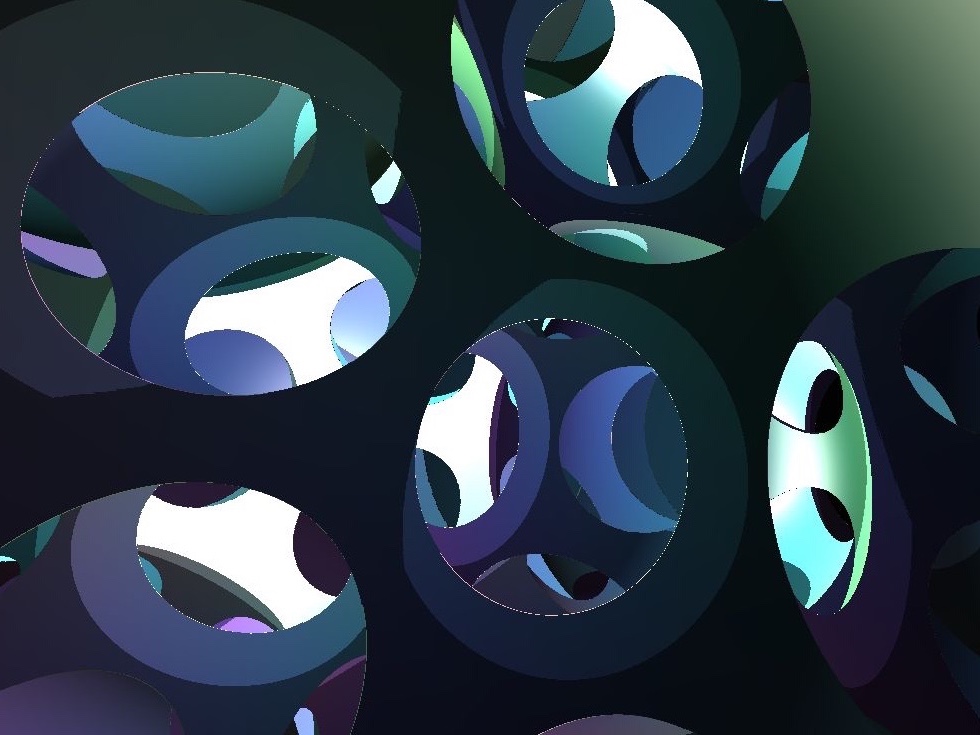}
\label{Fig:OneDirection}
}
\subfloat[Correct lighting (two geodesics)]{
\includegraphics[width=0.45\textwidth]{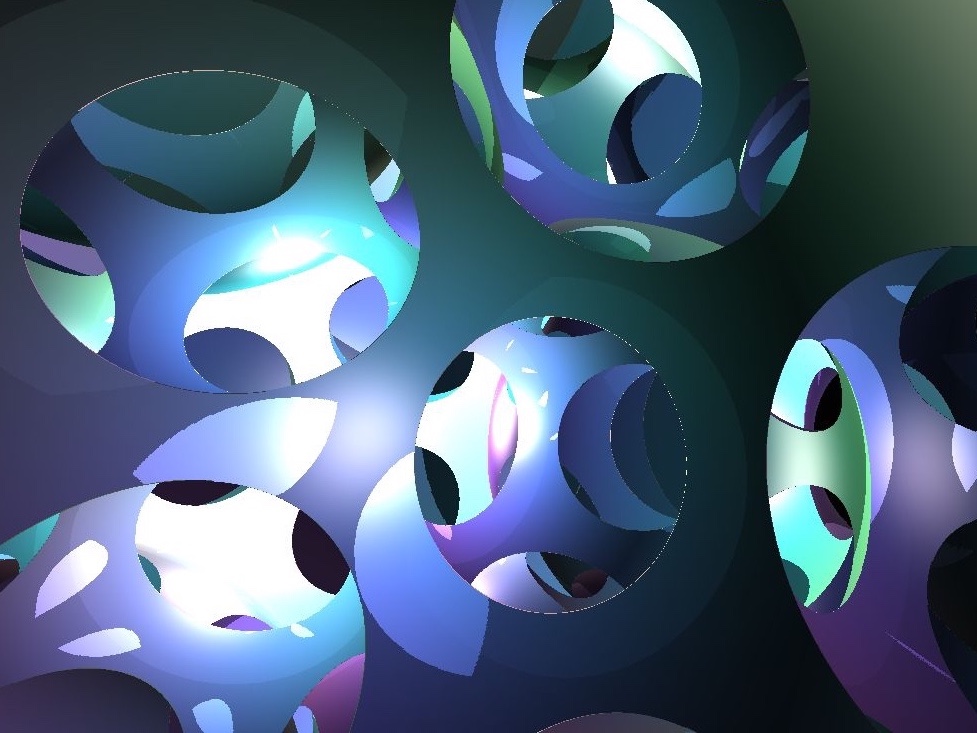}
\label{Fig:AllDirections}
}\\
\subfloat[Only the shortest geodesic.]{
\includegraphics[width=0.45\textwidth]{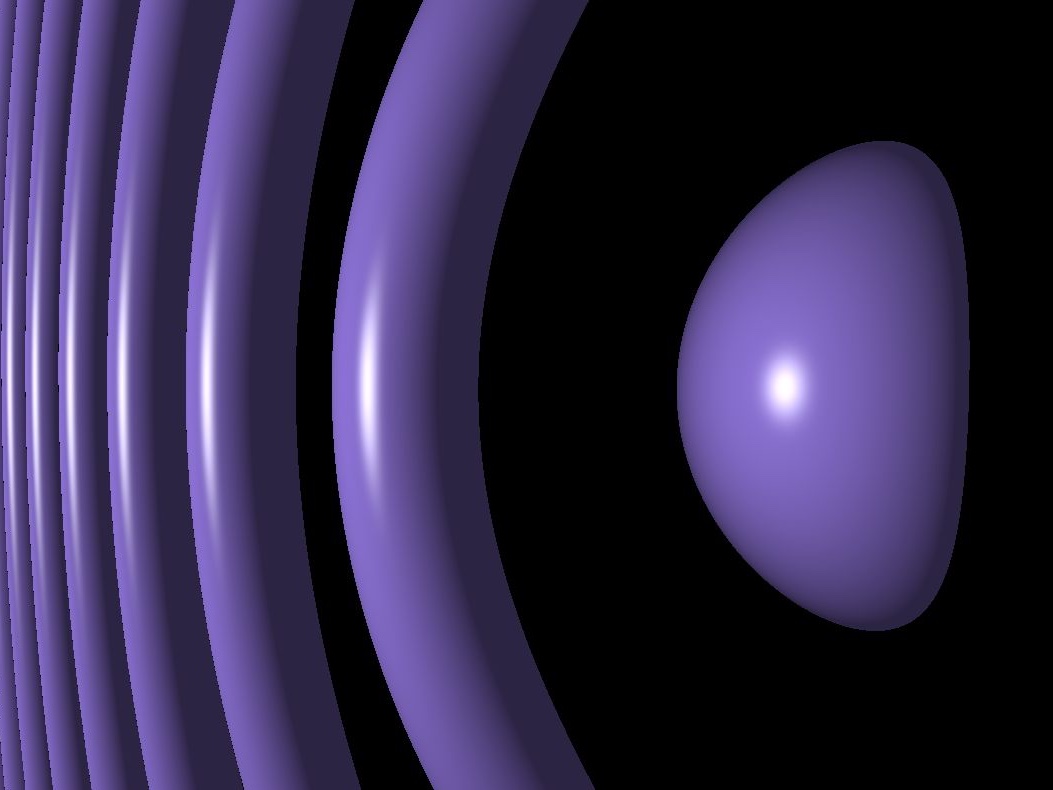}
\label{Fig:OneDirectionS2xE}
}
\subfloat[200 Geodesics.]{
\includegraphics[width=0.45\textwidth]{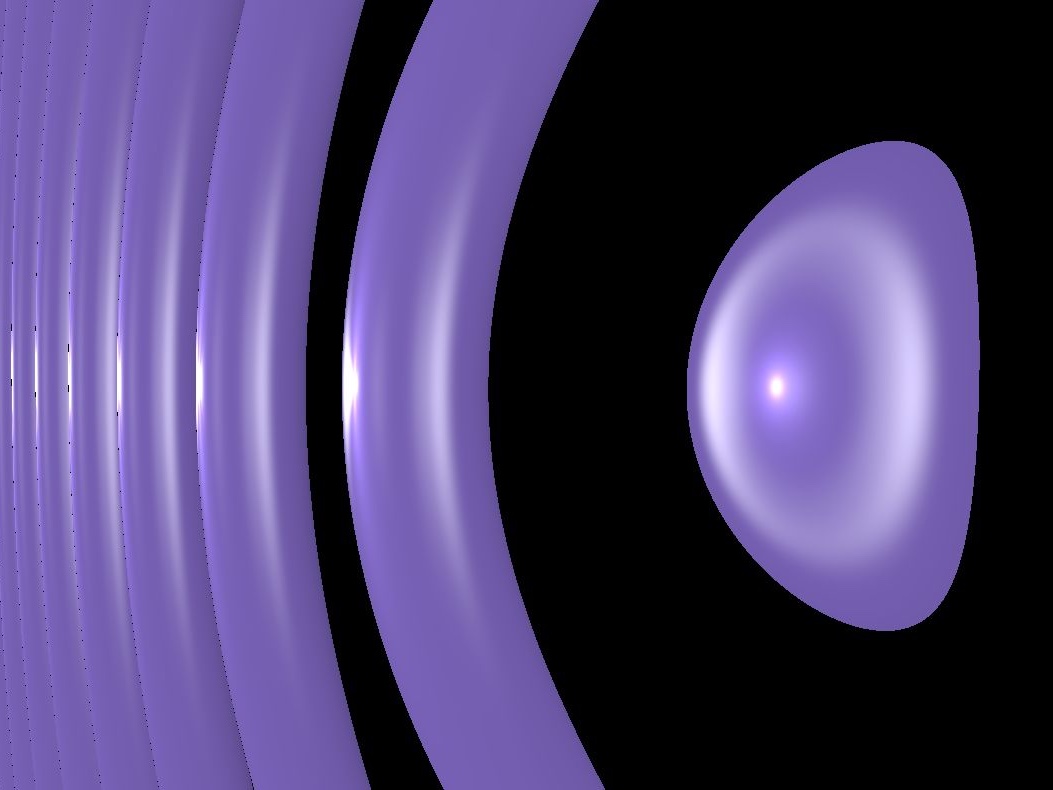}
\label{Fig:AllDirectionsS2xE}
}

\caption{A single light in $S^3$ (top) and $S^2\times\EE$ (bottom).  This demonstrates the necessity of dealing with multiple directions in $\calL_p(q)$.}
\label{Fig:AllDirections}
\end{figure}

\subsection{Computing the light intensity $I_L$}

We have one remaining quantity to compute: $I_L$, the intensity of the light source at $q$, as observed at $s$ from direction $L$.
We model our light source as isotropic with constant intensity $I_{\rm light}$.
To fix some notation, for any distance $t>0$ and unit direction vector $u\in T_q X$, let $I(t,u)$ be the intensity arriving from the light source after traveling along the geodesic ray in the direction $u$ for distance $t$.  
For any solid angle $\Omega$ (that is a subset of the unit tangent sphere at $q$ denoted by $UT_qX$), let $\Omega_t\subset X$ be the surface formed by flowing outwards from $q$ along geodesics in the directions in $\Omega$ by distance $t$. See \reffig{SolidAnglesSOL}.

\begin{figure}[htbp]
\subfloat[Solid angle around the $x$-axis.]{
\includegraphics[height=4.3cm]{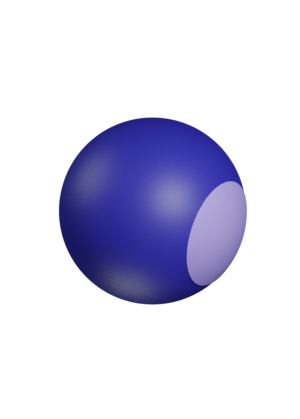}\qquad
\includegraphics[height=4.3cm]{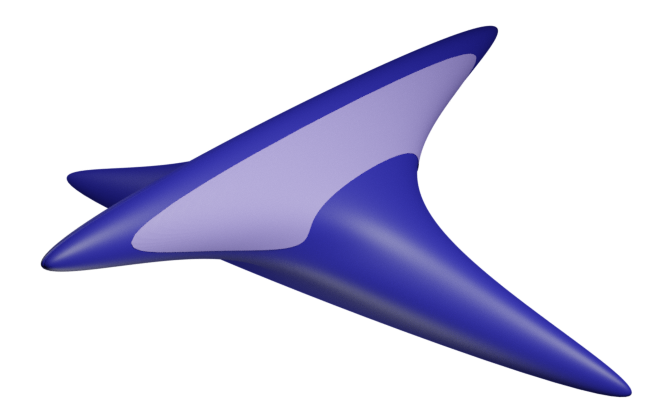}
}
\\
\subfloat[Solid angle around the $z$-axis. The image of the lighter area is a tiny strip on the top of the Sol sphere.]{
\includegraphics[height=4.3cm]{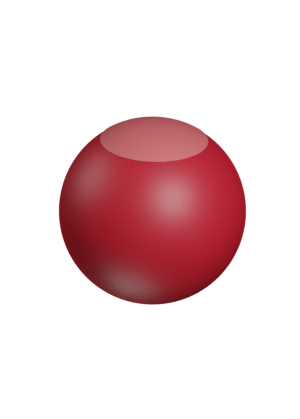}\qquad
\includegraphics[height=4.3cm]{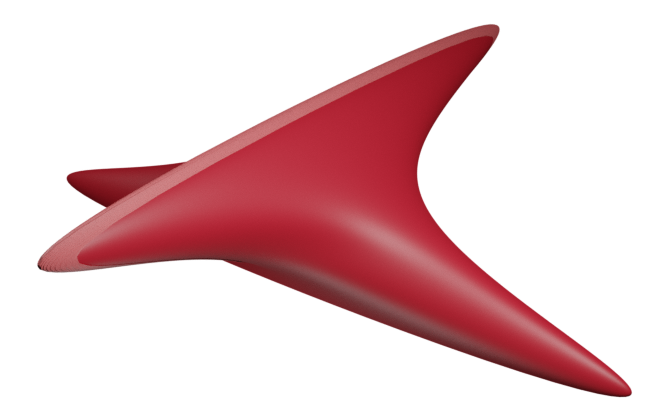}
}
\\
\subfloat[Solid angle around a diagonal line]{
\includegraphics[height=4.3cm]{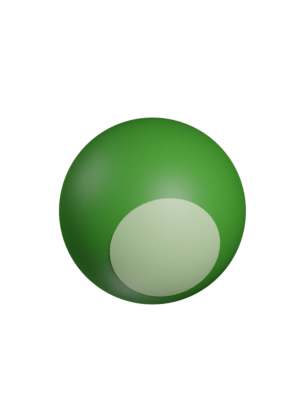}\qquad
\includegraphics[height=4.3cm]{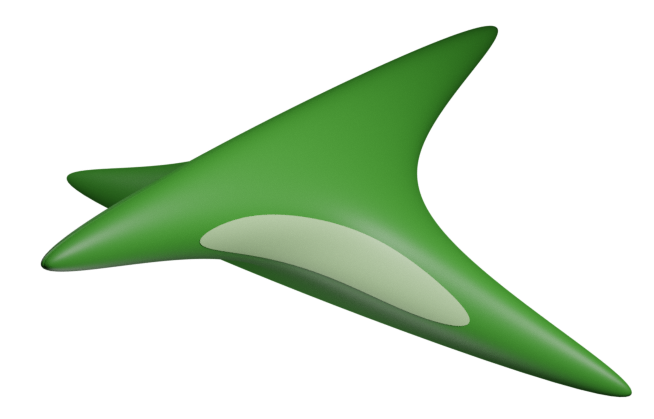}
}
\caption{Extrinsic views of spheres in Sol.
In each figure, the left hand picture represents the unit sphere in the tangent space at the origin of Sol. 
The lighter areas correspond to solid angles $\Omega$ with the same measure, but pointing in different directions.
The right hand picture shows an extrinsic view of the image of the unit tangent sphere after following the geodesic flow for time $r=3$.
The lighter area is the image $\Omega_t$ of $\Omega$.}
\label{Fig:SolidAnglesSOL}
\end{figure}

We assume that the total energy flux through the surface $\Omega_t$ is constant, independent of the distance traveled. (Energy is transported by the light rays along geodesics, but not created or destroyed along the way.)
This relates $I(t,u)$ directly to the area density of geodesic spheres. 
That is, for any $\Omega,t$ we have
$$\int_\Omega I_{\rm light} dA=\int_{\Omega}I(t,u)dA^\prime$$
where $dA$ is the standard area form on the unit sphere in the tangent space, and $dA^\prime$ is the pullback of the area form on $\Omega_t\subset X$ to $\Omega \subset UT_qX$.
We may express $dA^\prime$ in terms of $dA$; the resulting scale factor is the area density $dA^\prime=\mathcal{A}(t,u)dA$.
Thus, the quantity $\int_\Omega I(t,u)\mathcal{A}(t,u)dA$ is constant in $r$ for every solid angle $\Omega\subset UT_qX$.
Assuming continuity and taking the limit over shrinking solid angles promotes this to a pointwise invariant: $I(t,u)\mathcal{A}(t,u)$ is independent of $t$.
Thus, $I$ is inversely proportional to $\mathcal{A}$, and 
\begin{equation}
\label{Eqn:IntensityDensity}
I(t,u)=\frac{I_{\rm light}}{\mathcal{A}(t,u)}.
\end{equation}

\begin{remark}
The intensity $I_L$ experienced at $s$ from the direction $L$ is then just
$I_L=I(d_L,\ell)=I_{\mathrm{light}}/\mathcal{A}(d_L,\ell)$.
A further correction to $I_L$ can occur when we add fog.
Here the intensity drops off due both to (1) divergence/convergence of geodesics, and (2) distance traveled through the medium.
A physically correct model for scattering from an isotropic source is already complex in euclidean space.
However, as the primary goal of modeling fog is to provide useful depth cues (and hide sins), we treat these sources of loss as if they were independent, and use
$$I_L^{\rm fog}=\mathrm{Fog}(d_L)\cdot I_L(d_L,\ell)=e^{-K d_L}\frac{I_{\rm light}}{\mathcal{A}(d_L,\ell)}$$
when distance-dependent attenuation (fog) is desired.
\end{remark}

\refeqn{IntensityDensity} reduces the calculation of lighting intensity directly to the area density $\mathcal{A}$.
In the next section, we calculate this area density by following infinitesimal patches of area along the geodesic flow.

\subsubsection{Area density under the geodesic flow}

Fix $q\in X$ to be the location of a light source, and
let $F\colon T_q X\to X$ be the exponential map.
 For fixed $t>0$, define $f_t(u)=F(tu)$, so $f_t\colon UT_qX\to X$ is a map of the unit tangent sphere at $q$, $UT_qX$ into $X$, formed by flowing along geodesics from $q$ for distance $t$. Note that the image is not the sphere of radius $t$ about $q$ when $t$ is greater than the injectivity radius of $X$.
Recalling the notation above $\Omega_t$ is defined as $f_t(\Omega)$, for a solid angle $\Omega\subset UT_qX$. 
We denote the entire image as $S^2_t=f_t(UT_qX)$.
Let $dA$ be the standard area form on $UT_qX$, and let $dA_t$ be the area form on $S_t^2\subset X$.
Recall that the area density $\mathcal{A}(t,u)$ is the proportionality factor of the pullback $f^\ast_t dA_t$ to $dA$.
We may compute this given any choice two non-collinear vectors $\{v,w\}$ in $u^\perp$ as
$$\mathcal{A}(t,u)=\frac{(f^\ast_t dA_t)(v,w)}{dA(v,w)}=\frac{dA_t\left((df_t)_uv,(df_t)_uw\right)}{dA(v,w)}.$$

The area forms $dA$ and $dA_t$ measure the areas in $X$ of infinitesimal parallelograms in $T_qX$ and $T_{f_t(u)}X$ respectively, and so may be evaluated in the algebra of bivectors on $TX$, where the area spanned by $v,w\in T_pX$ is given by 
$$\|v\wedge w\|=\sqrt{\langle v,v\rangle\langle w,w\rangle-\langle v,w\rangle^2}$$

Thus, we have 
\begin{equation}
\label{Eqn:AreaDensityDifferential}\mathcal{A}(t,u)=\frac{\|(df_t)_u v\wedge (df_t)_u w\|}{\|u\wedge w\|}.\end{equation}
Note that the numerator is the Jacobian derterminant of $f_t$.
As computing area elements requires nothing more than some evaluations of the metric, this reduces the calculation of area density to the computation of the differential $df_t$.

Recall that $f_t(u)=F(tu)$, where $F$ is the exponential map. We see that $(df_t)_uv=dF_{tu}v$ for all $u\in T_q X$ and $v\in T_u(T_qX)$. 
To lighten notation, for the rest of this paragraph we identify $T_{u'}(T_qX)$ with $T_qX$ for every $u'\in T_q X$.
Given $u$ in $UT_qX$ and $v$ in $u^\perp$ of unit length, this allows an explicit computation of $(df_t)_uv$ in terms of the exponential map, as follows.  
Let $\eta(v,s)=\cos(s)u+\sin(s)v$ be the unit vector in $T_q X$ making angle $s$ with $u$ in the plane spanned by $\{u,v\}$.
Note that $\eta'(0)=v$ so we may calculate $(df_t)_u v$ as
$$(df_t)_uv=dF_{tu}v=\left.\frac{d}{ds}\right|_{s=0}F(t\eta(v,s)).$$
For each fixed $s$, 
the map $t\mapsto F(t\eta(v,s))$ is a unit speed geodesic in $X$,
and the derivative $dF_{tu}v\in T_{F(tu)}X$ is a vector field along this geodesic.
Computed as above, we see this is a particularly nice vector field:
it is the derivative of the geodesic flow along a one-parameter family of geodesics.
Such vector fields are called \emph{Jacobi fields}.

Given a smooth one-parameter family of geodesics $\{\gamma_s(t)\}$ through $\gamma_0 = \gamma$, the \emph{Jacobi field} associated to $\gamma_s$ is given by $J(t) =\partial \gamma_s(t) / \partial s |_{s = 0}$.
In general, one may bypass explicit computations involving $\gamma_s$, and compute such Jacobi fields by solving a differential equation. 
The Jacobi field $J_v$ along $\gamma$ with initial conditions $J(0)=0,\dot{J}(0)=v$ satisfies the so called \emph{Jacobi equation}, 
\begin{equation}
\label{Eqn:Jacobi}
\ddot{J}_v=\Riem\left(J_v,\dot{\gamma}\right)\dot{\gamma}
\end{equation}
where $\Riem$ is the Riemann curvature tensor.
For us then, $(df_t)_uv$ and $(df_t)_uw$ are the Jacobi fields along $f_t(u)$ corresponding to the variations $F(t\eta(v,s))$ and $F(t\eta(w,s))$ respectively, so 
\begin{equation}\label{Eqn:Jacob_df}
(df_t)_uv=J_v(t)\quad\textrm{and}\quad(df_t)_uw=J_w(t).
\end{equation}

In the isotropic geometries and product geometries, \refeqn{Jacobi} reduces to a second-order differential equation with constant coefficients. In any geometry where one may solve \refeqn{Jacobi}, the area density is given as follows.
For fixed $u\in UT_qX$, choose two vectors $v,w\in u^\perp$  with $\|v\wedge w\|=1$ and solve the Jacobi equation for the two Jacobi fields $J_v,J_w$.  
Then using Equations \ref{Eqn:AreaDensityDifferential} and \ref{Eqn:Jacob_df}, we have

\begin{equation}
\label{Eqn:AreaDensity_Jacobi}
\mathcal{A}(t,u)=\|J_v(t)\wedge J_w(t)\|\end{equation}

In the harder geometries, solving \refeqn{Jacobi} is more challenging.  
Following \refsec{GeodesicFlow - Grayson}, one could use Grayson's method to replace \refeqn{Jacobi} with a system of differential equations on $T_oX$.
This is not what we do though.
Since we already computed the exponential map $F$ (using Grayson's method) we directly compute its differential $dF_{tu}$.

Let $r,\theta, \phi$ be the standard spherical coordinates on $T_qX$, with $\phi$ the angle measured from the north pole. Let $u\in UT_qX$ have coordinates $[\theta,\phi]$.
Note that as the coordinate vector fields $\partial_\theta,\partial_\phi$ are orthogonal to $\partial_r$,
we may use them to make a uniform choice $v=\partial_\phi,w=\partial_\theta$, and compute

\begin{equation*}
\mathcal{A}(r,u)=\frac{\|dF_{ru}(\partial_\phi)\wedge dF_{ru}(\partial_\theta)\|}{\|\partial_\phi\wedge \partial_\theta\|}
=\frac{\|\frac{\partial F}{\partial\phi}(r,\theta,\phi)\wedge\frac{\partial F}{\partial\theta}(r,\theta,\phi)\|}{\sin\phi}.
\end{equation*}

In practice, due to the rotational symmetry in Nil and $\SLR$ about a single axis, it is more convenient to perform this computation in cylindrical coordinates, with $\rho=r\cos\phi$ and $z=r\sin\phi$.  For ease of notation, we retain $r=\sqrt{\rho^2+z^2}$ from spherical coordinates to denote the distance traveled along the geodesic.

\begin{equation}
\label{Eqn:AreaDensity_Coordinates}
\mathcal{A}(r,u)=\frac{2}{r}\left\|\left(\frac{\partial F}{\partial \rho}-\frac{\rho}{z}\frac{\partial F}{\partial z}\right)\wedge \frac{\partial F}{\partial \theta}\right\|
\end{equation}

Using either \refeqn{AreaDensity_Jacobi} or \refeqn{AreaDensity_Coordinates}, the computation of area density is necessarily geometry-dependent, so we give details for each geometry in the corresponding section later. See Sections~\ref{Sec:IsotropicLighting}, \ref{Sec:ProductLighting}, \ref{Sec:NilLighting}, and \ref{Sec:SLRLighting}.

\subsection{Lighting in quotient manifolds}

The basic algorithms for lighting remain virtually unchanged in a quotient manifold.
Phong lighting is still computed in the tangent space, and the only modification to the computation of shadows and reflections is to modify the ray-march as in \refsec{NonSimplyConnected}.
There is only one major change worthy of discussion: the calculation of direction vectors pointing from the surface to a given light.
This is even more necessarily multi-valued here, as light may travel in loops around the manifold before impacting the surface.
Indeed, a light in $X/\Gamma$ is the same as a $\Gamma$-equivariant collection of lights in $X$. 
When required for disambiguation, we will denote the set of lighting pairs in a space $Y$ as $\calL^Y$.
For the location of a light $q$ in $D$, thought of as the fundamental domain for $X/\Gamma$, 
the lighting pairs $\calL_p^{X/\Gamma}(q)$ can be written in terms of the lighting pairs $\calL^X_p$ of \refsec{Lighting directions}:
$$\calL_p^{X/\Gamma}(q):=\bigcup_{\gamma\in\Gamma}\calL_p^X(\gamma.q)$$ 
Note that there is no sense in which $\calL_p^X(\gamma.q)$ is some sort of ``$\gamma$-translate'' of $\calL_p^X(q)$: the individual sets in this union may not even have the same cardinality.
This occurs for instance in Nil, where even if the distance from $p$ to $q$ is less than the injectivity radius, there may be a $\gamma\in\Gamma$ with arbitrarily many geodesics from $p$ to $\gamma.q$.
As lighting is calculated individually for each direction and summed  weighted by intensity, it is in general impossible to compute this exactly for any manifold with infinite fundamental group.
Instead, for all but spherical manifolds and orbifolds, we must approximate the lighting by computing only for those paths with significant intensity.

Light intensity is inversely correlated with geodesic length of a segment from $p$ to $q$ in geometries with non-positive sectional curvature, and in all geometries if we use fog.
Thus we get a reasonable approximation to the correct image by restricting to directions corresponding to `sufficiently short' geodesics.
Considering only the directions from lights within $D$ (that is, when $\gamma = \textrm{id}$)
is not enough, as some nearby translates $\gamma.q$ still contribute significantly. Compare \reffig{NoNeighbors} with \reffig{LightingAllCells}. The latter shows the correct lighting in the quotient of the three-sphere by the binary tetrahedral group. The former shows lighting using one of the 24 light sources. An improved approximation is to use the `nearest neighbors' idea from \refsec{NearestNeighbors}, and consider only tangent directions at $p$ which reach the light at $q\in D$, or its translates \emph{through the faces of $D$}. See \reffig{LightingNeighbors}.

This is even an issue in euclidean manifolds. Note that there is a discontinuity in the lighting of the red balls in \reffig{NearestNeighbor}. The left and right hemispheres are lit by different collections of lights, since they sit in different fundamental domains.

\begin{figure}[htbp]
\centering
\subfloat[Lighting from within $D$ only.]{
\includegraphics[width=0.85\textwidth]{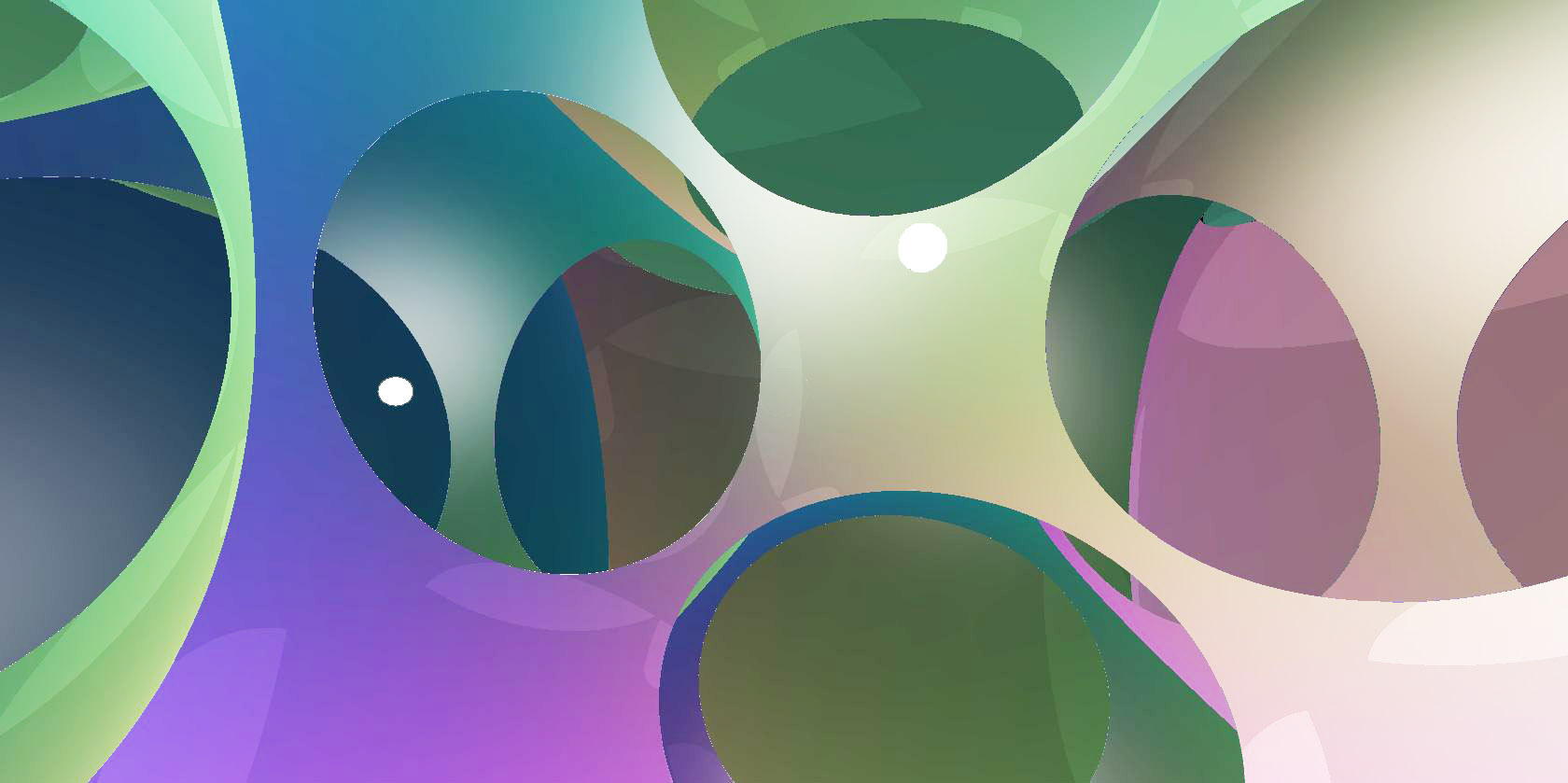}
\label{Fig:NoNeighbors}
}

\subfloat[Lighting from within $D$ and its eight neighbors.]{
\includegraphics[width=0.85\textwidth]{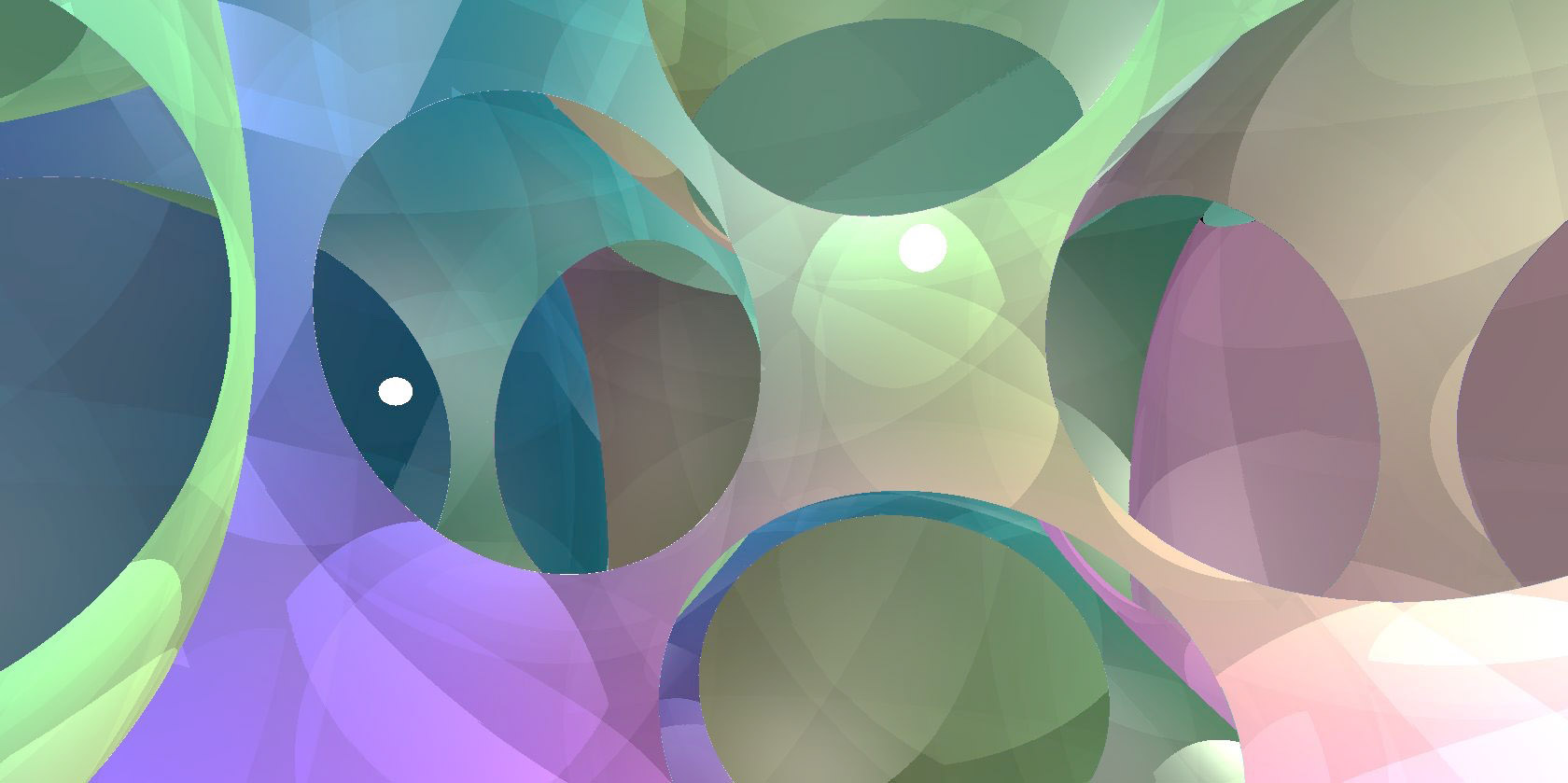}
\label{Fig:LightingNeighbors}
}

\subfloat[Lighting from all 24 cells.]{
\includegraphics[width=0.85\textwidth]{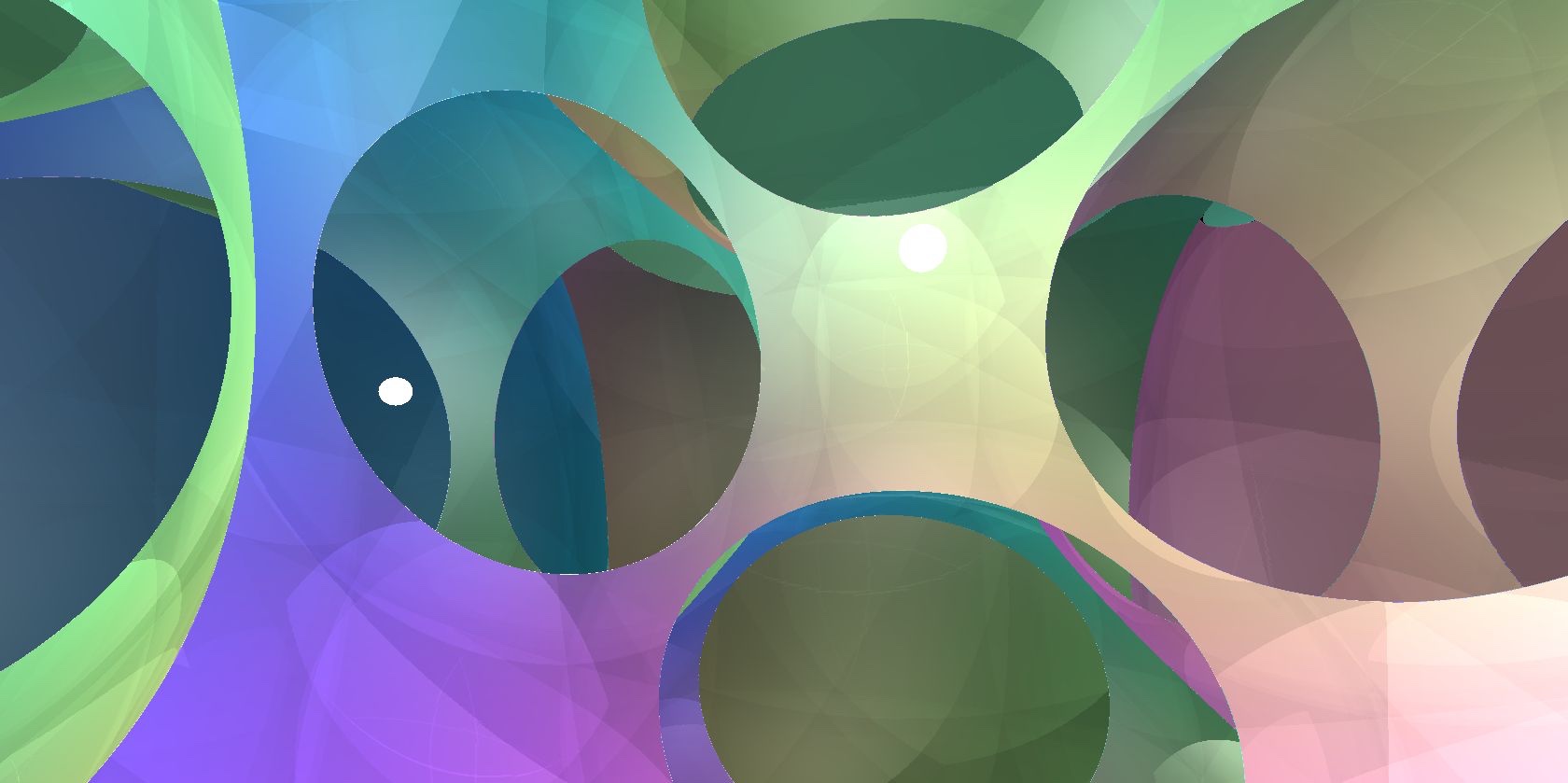}
\label{Fig:LightingAllCells}
}
\caption{Lighting of the quotient of $S^3$ by the binary tetrahedral group, with a single point source light. There are no reflections: the patterns are the result of (hard) shadows cast by the scene.}
\label{Fig:QuotientLighting}
\end{figure}

In geometries with positive sectional curvatures, light can converge again over long distances, meaning that there are certain directions where even long geodesics make significant contributions to the overall sum unless we use fog.
Which translates of the lights to include in a calculation then depends on both the geometry and the scene. So far, we have only a heuristic understanding of how to choose translates appropriately, based on the light intensity function for each geometry.

\subsection{Cheating}
\label{Sec:Cheating}

Accurate lighting and shading is a complex problem,  requiring many calculations, and many ray-marches per pixel to perform correctly.
As we strive to produce as accurate a simulation as possible, we have worked to implement lighting, shadows, reflections, and fog as described above.
However, insistence on complete ``physical'' accuracy is not ideal for all applications.
Sometimes lighting is best thought of as a means for euclidean humans to better perceive the geometry, rather than as a feature of the geometry in itself. This is analogous to astrophysical simulations, where it is more important to correctly render the size and position of celestial bodies, rather than to faithfully reproduce the brightness of the sun. 
In these situations it is often desirable to purposely employ nonphysical lighting to improve speed and/or visibility.

We find that the most often useful change to make is in the relationship of light intensity $I_L$ with distance.
There are two main problems that we can solve here. 
\begin{itemize}
\item First, correct lighting may give intensities of vastly different magnitudes for different parts of the same scene. This means that parts of the scene will be too dark for our eyes to see any structure. Alternatively, we can increase the brightness of the lights, but then other parts of the scene will be oversaturated. 
\item Second, and more subtly, we use variation in brightness as a depth cue, telling us how far away an object is from a light source. 
\end{itemize}

\reffig{HypCorrect} shows a scene in $\HH^3$ lit by a single light. Here, exponential falloff in intensity with distance leaves everything other than the central cell shrouded in darkness. We see similar behavior in \reffig{HypDirection}, when looking in a hyperbolic direction in $\HH^2 \times \EE$. When we look in a euclidean direction in \reffig{RealDirection}, we do see neighboring cells, giving the impression that cells are closer in that direction than in the hyperbolic directions. In \reffig{S2ECorrect}, the correct lighting calculations in $S^2 \times \EE$ give an approximately even brightness over the whole image, even though only the ball at the center is particularly close to the viewer. 
The space $S^2 \times \EE$ works like a fiber-optic cable -- on average, the intensity of the light does not decrease with distance as we move along the cable.

Instead of the correct lighting intensity $I_L$, we may cheat, and use 
an artificial slowly decreasing intensity (say, inversely proportional to geodesic length). This provides more helpful depth cues and may also be less expensive to compute. See Figures~\ref{Fig:HypLinear},~\ref{Fig:S2ELinear},~\ref{Fig:IsoRealDirection}, and~\ref{Fig:IsoHypDirection}. 
As a side benefit, this also allows one to see distant reaches of a negatively curved space with only a few light sources. This also reduces computational cost.

\begin{figure}[htbp]
\centering
\subfloat[Correct intensity calculation.]{
\includegraphics[width=0.45\textwidth]{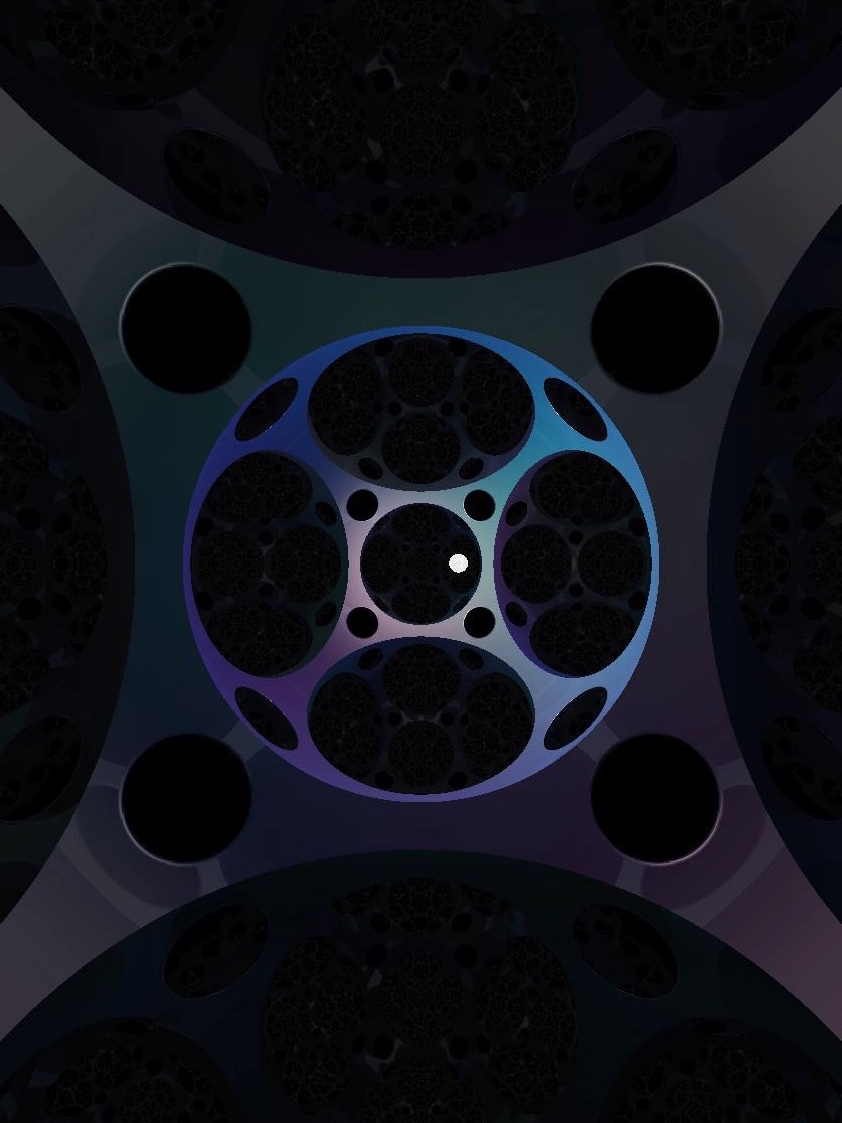}
\label{Fig:HypCorrect}
}
\quad
\subfloat[Intensity inversely proportional to geodesic length.]{
\includegraphics[width=0.45\textwidth]{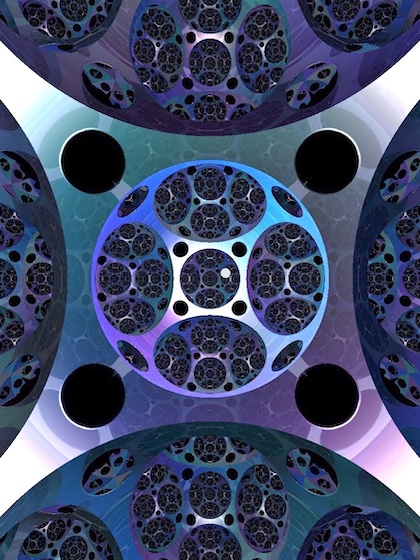}
\label{Fig:HypLinear}
}
\caption{A single light in hyperbolic space.}
\label{Fig:AllDirectionsH3}
\end{figure}

\begin{figure}[htbp]
\centering
\subfloat[Correct intensity calculation.]{
\includegraphics[width=0.45\textwidth]{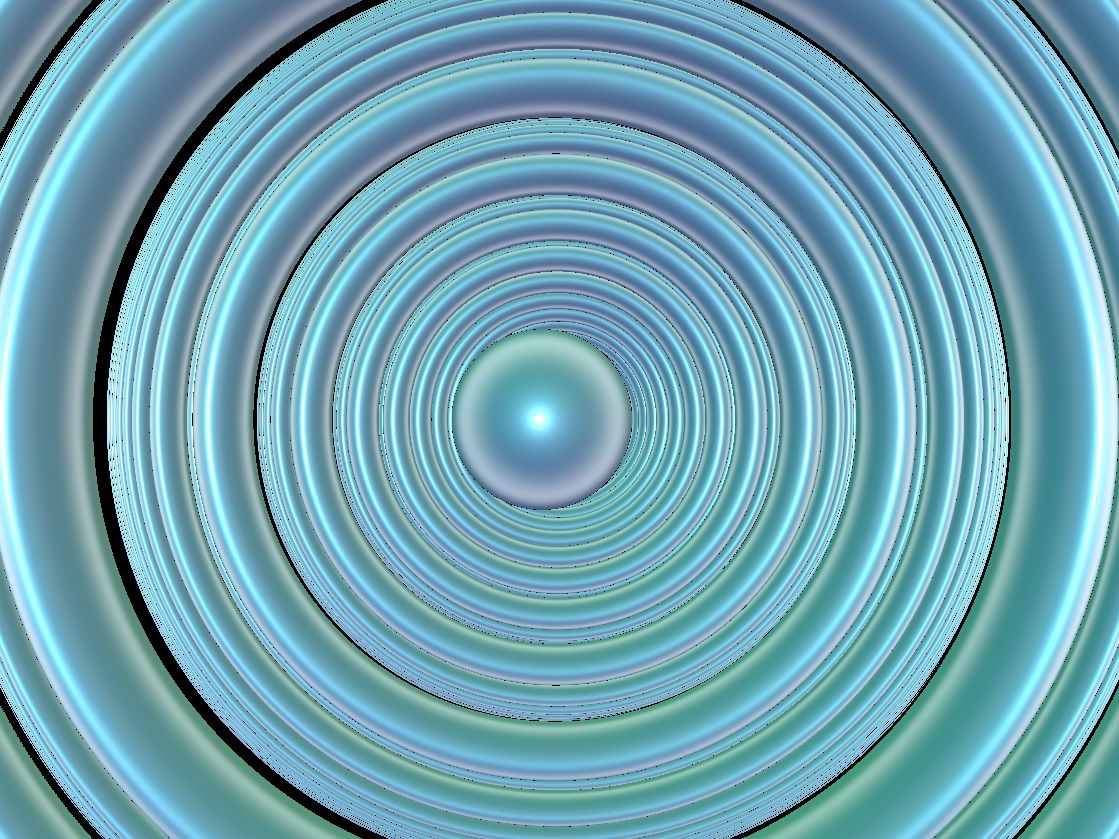}
\label{Fig:S2ECorrect}
}
\quad
\subfloat[Intensity inversely proportional to geodesic length.]{
\includegraphics[width=0.45\textwidth]{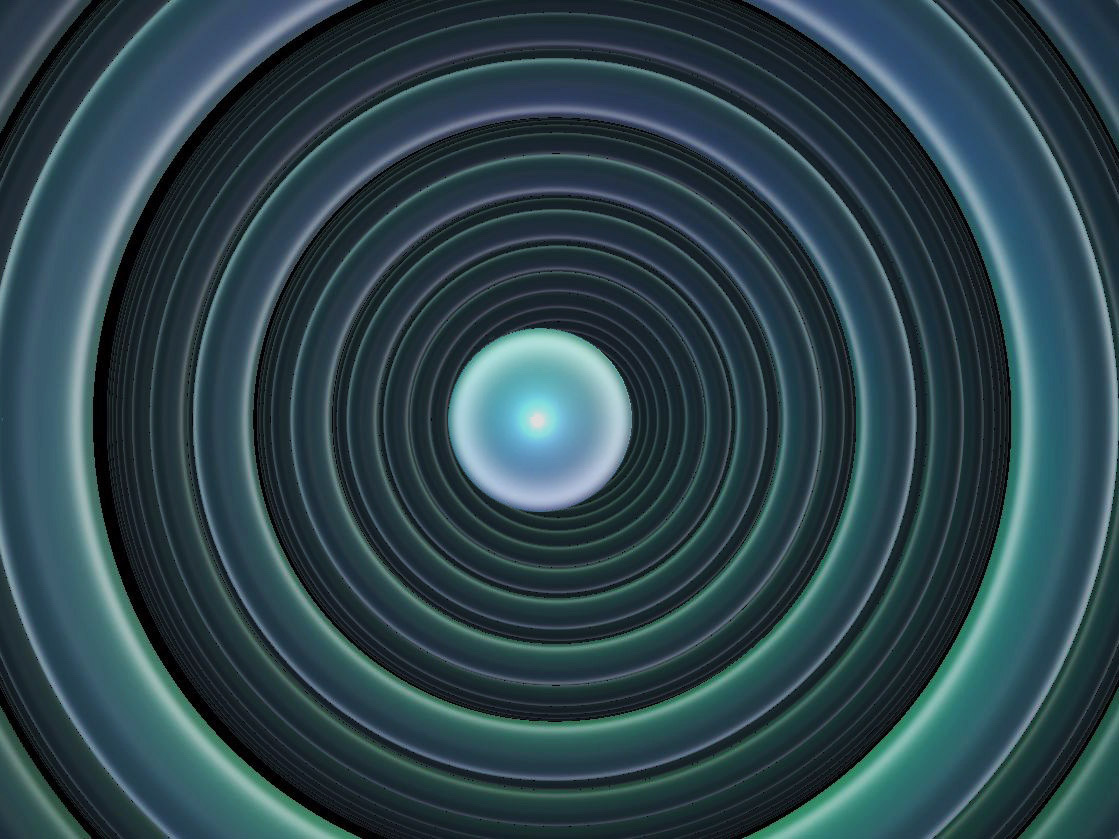}
\label{Fig:S2ELinear}
}
\caption{A line of balls in $S^2 \times \EE$ lit by a single light. Each ball is also visible as a collection of rings, seen along rays that wrap around the $S^2$ direction at least once.}
\label{Fig:AllDirectionsS2xS1}
\end{figure}

\begin{figure}[htbp]
\centering
\subfloat[Correct lighting, view in the $\EE$ direction.]{
\includegraphics[width=0.45\textwidth]{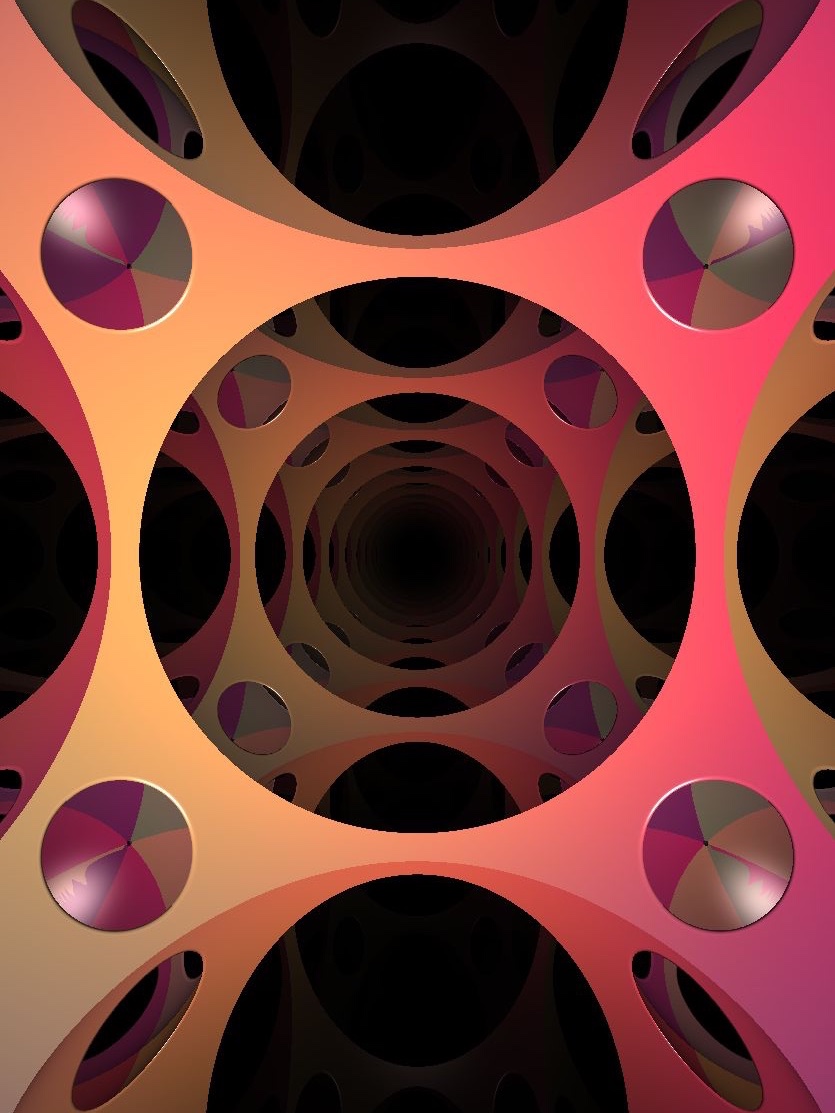}
\label{Fig:RealDirection}
}
\quad
\subfloat[Intensity inversely proportional to geodesic length, view in the $\EE$ direction.]{
\includegraphics[width=0.45\textwidth]{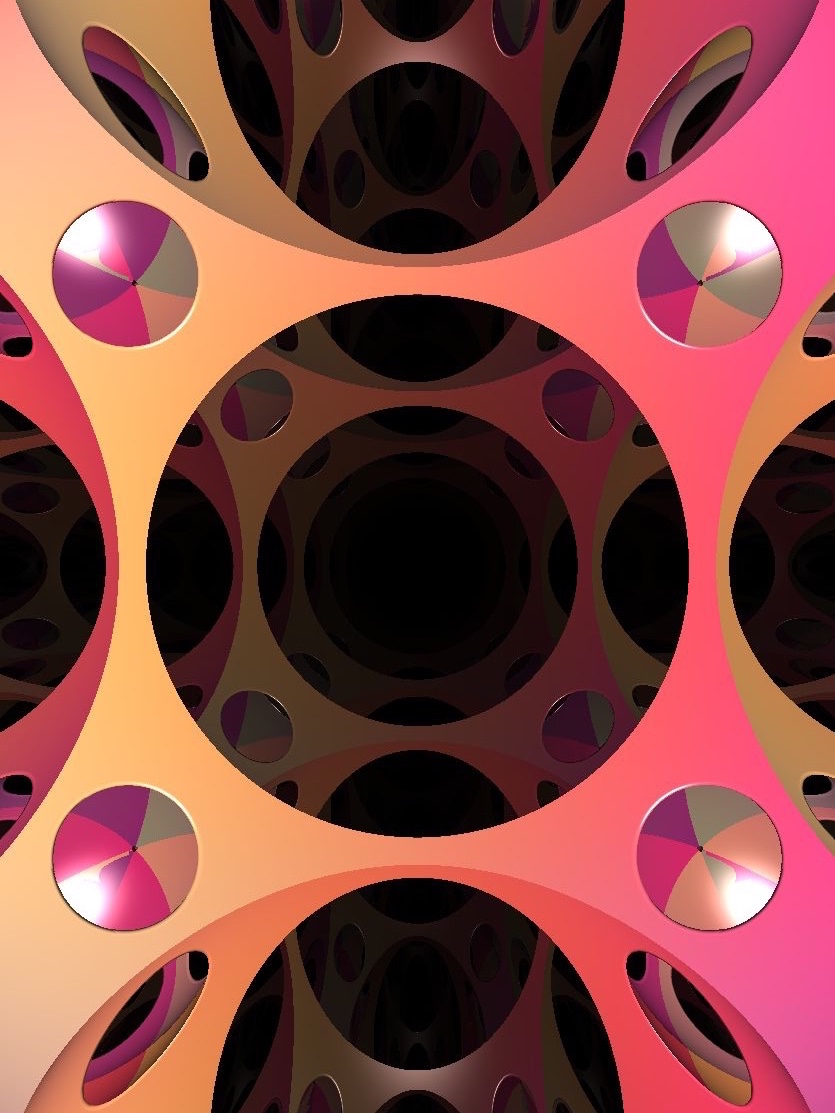}
\label{Fig:IsoRealDirection}
}\\
\subfloat[Correct lighting, view in an $\HH^2$ direction.]{
\includegraphics[width=0.45\textwidth]{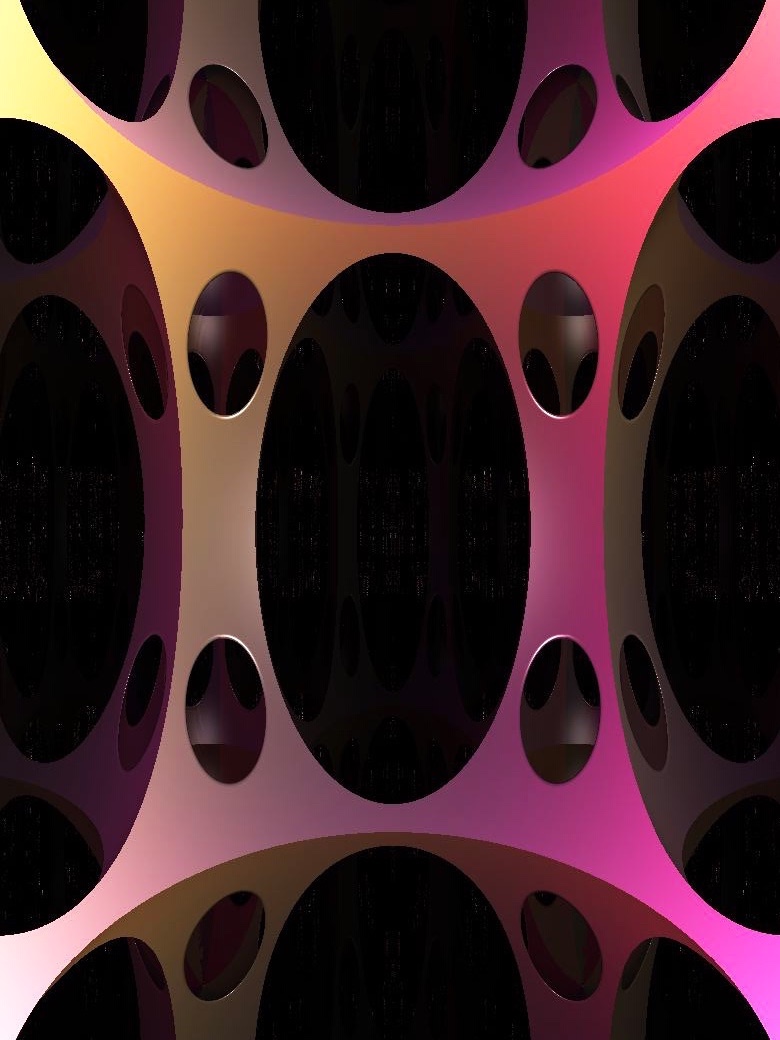}
\label{Fig:HypDirection}
}
\quad
\subfloat[Intensity inversely proportional to geodesic length, view in an $\HH^2$ direction.]{
\includegraphics[width=0.45\textwidth]{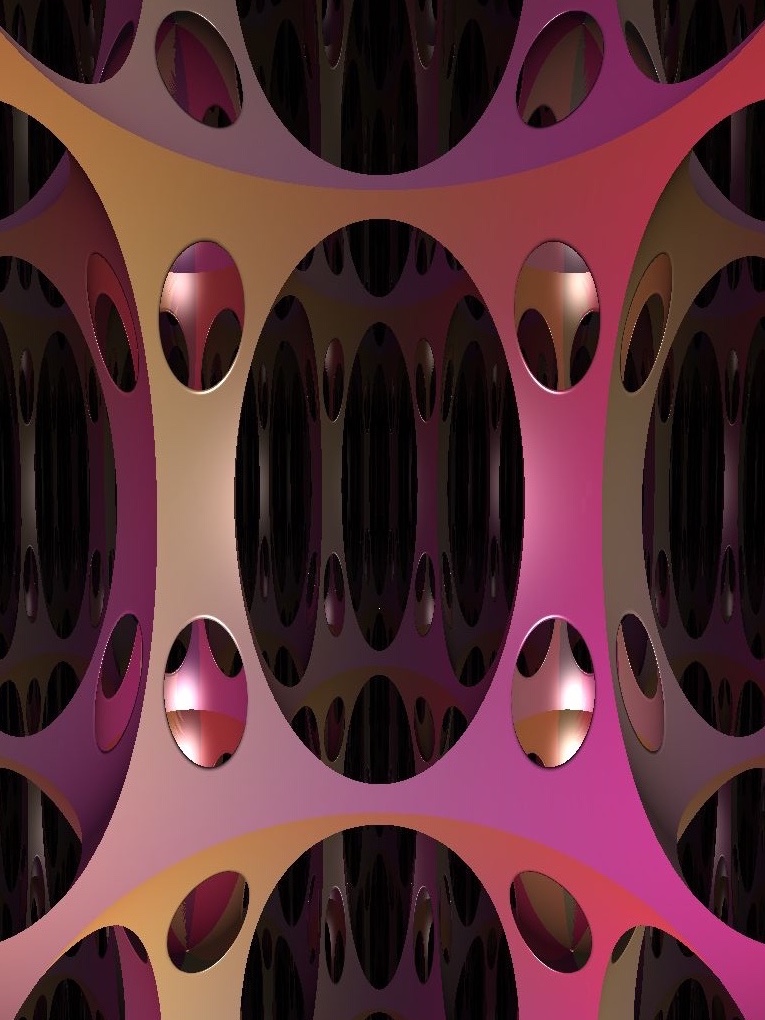}
\label{Fig:IsoHypDirection}
}\\
\caption{A lattice lit by a single light in $\HH^2\times\EE$. 
The distance between the centers of neighboring cells of the lattice is the same in all directions.
With correct lighting, we see many cells in the $\EE$ direction, while we can barely see our neighbor in an $\HH^2$ direction. With fake lighting, cells dim with distance equally in all directions. (Note that there is no fog in these images.) }
\label{Fig:DirectionalLight}
\end{figure}

When it comes to improving speed, we may pare down the lighting pipeline to focus on giving accurate depth cues.
This means preserving Phong lighting and fog, while perhaps ignoring shadows, or not using reflective materials.
Another efficiency gain which does not affect the intelligibility of the scene is to consider only the direction to the light along the \emph{shortest geodesic}, instead of the set of all directions.
Even when attempting accurate rendering, it is often acceptable to ignore lighting along all but the shortest few geodesics. This is the case when using fog, or when the intensity fall-off makes the contribution to the weighted average along longer geodesics negligible. 

However, using fewer geodesics can introduce very visible errors. In a quotient manifold, as we saw in Figures \ref{Fig:QuotientLighting} and \ref{Fig:NearestNeighbor} we may lose shadows, or introduce discontinuities in the perceived light intensity. In some geometries, using fewer geodesics can in fact remove discontinuities in lighting intensity that should be there.  

We usually indicate the position of a light with a ball in the scene centered on the light source, making sure that the shadow calculation for that light ignores the ball. To remove visual complication, we sometimes choose to not render these balls. 
Along these lines, in some situations we may not actually care, or may not be able to efficiently calculate, the lighting pairs $\calL_s(q)$.  Instead, we may simply choose for each light source a continuously varying direction field $X \to TX$. We give up on correctness, but still provide a seamless view and give visual cues.
\reffig{NilAllDirections} compares different choices of illumination in Nil.

\begin{figure}[htbp]
\centering
\subfloat[Artificial direction field (straight line in $\RR^4$ from $s \in X \subset \RR^4$ to the light position).]{
\includegraphics[width=0.47\textwidth]{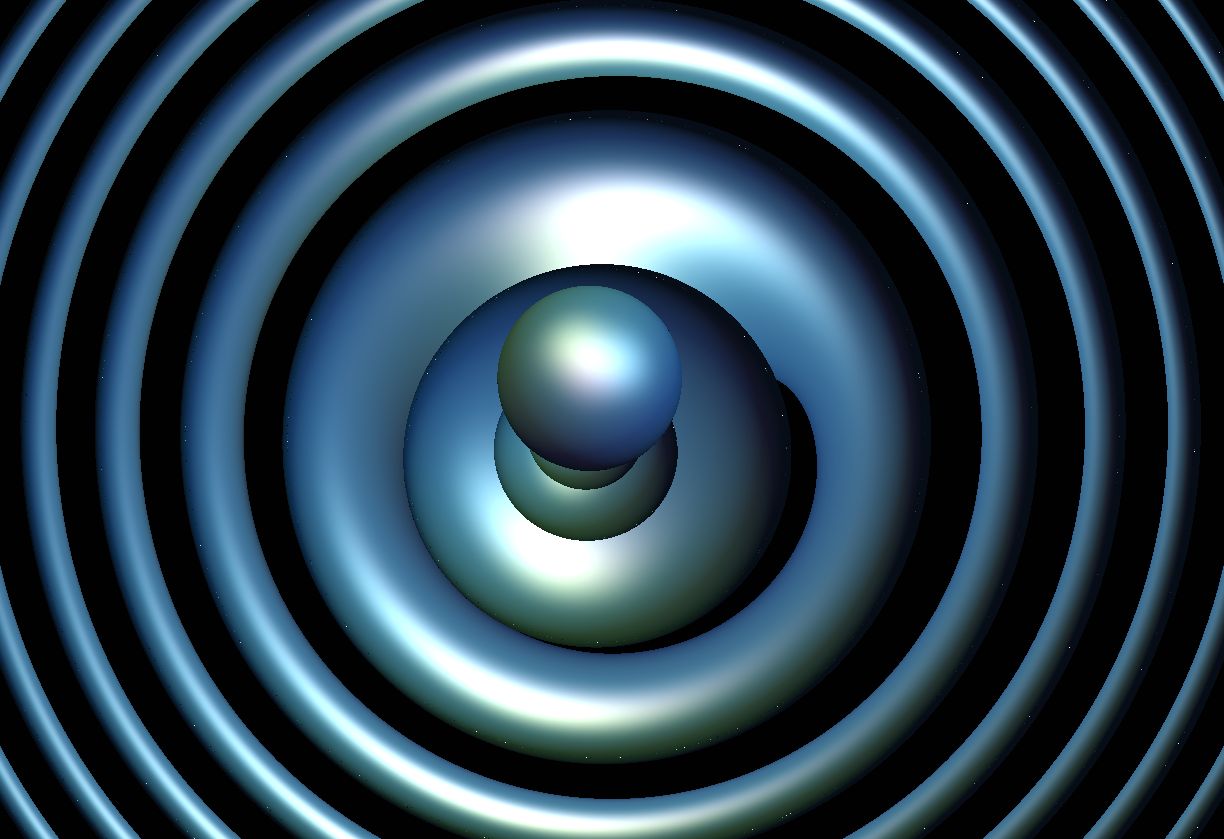}
\label{Fig:nilfake}
}
\subfloat[Direction of the shortest geodesic only.]{
\includegraphics[width=0.47\textwidth]{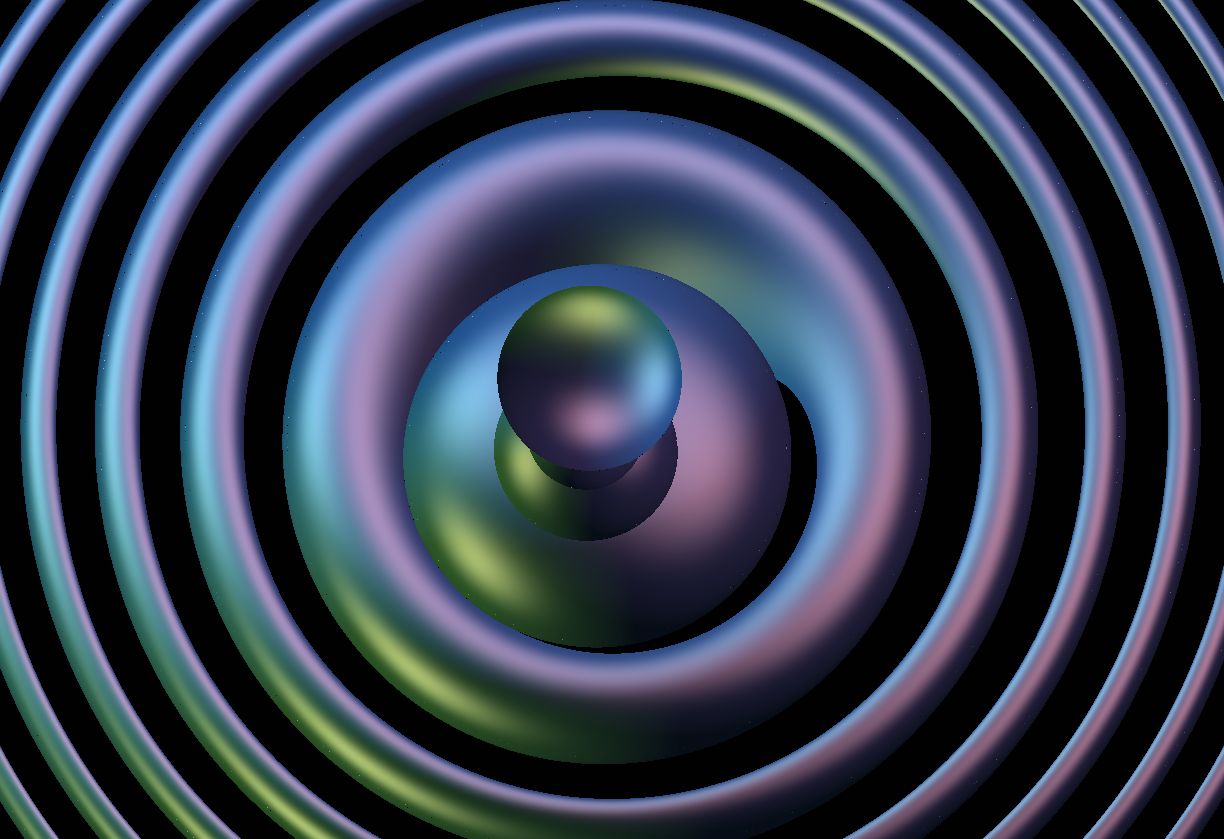}
\label{Fig:niltrue1}
}\\
\subfloat[At most two geodesics.]{
\includegraphics[width=0.47\textwidth]{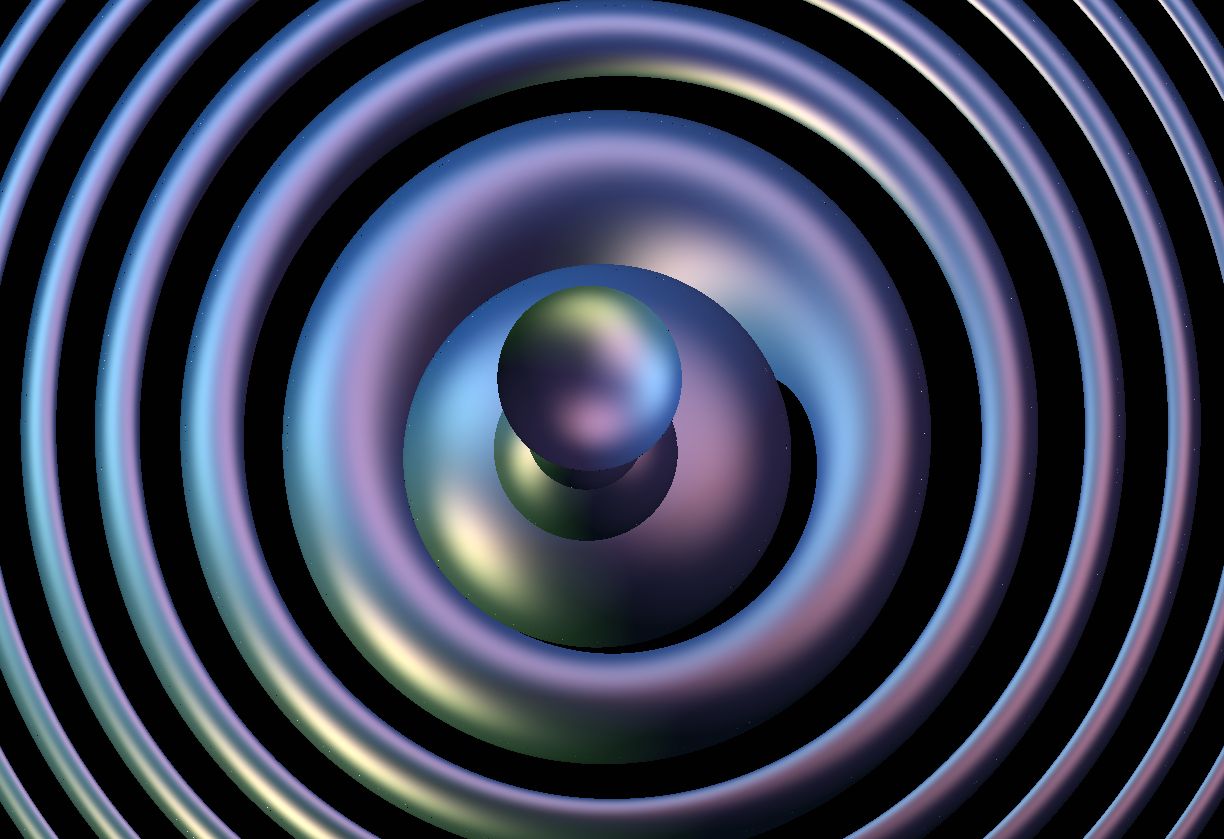}
\label{Fig:niltrue2}
}
\subfloat[At most three geodesics.]{
\includegraphics[width=0.47\textwidth]{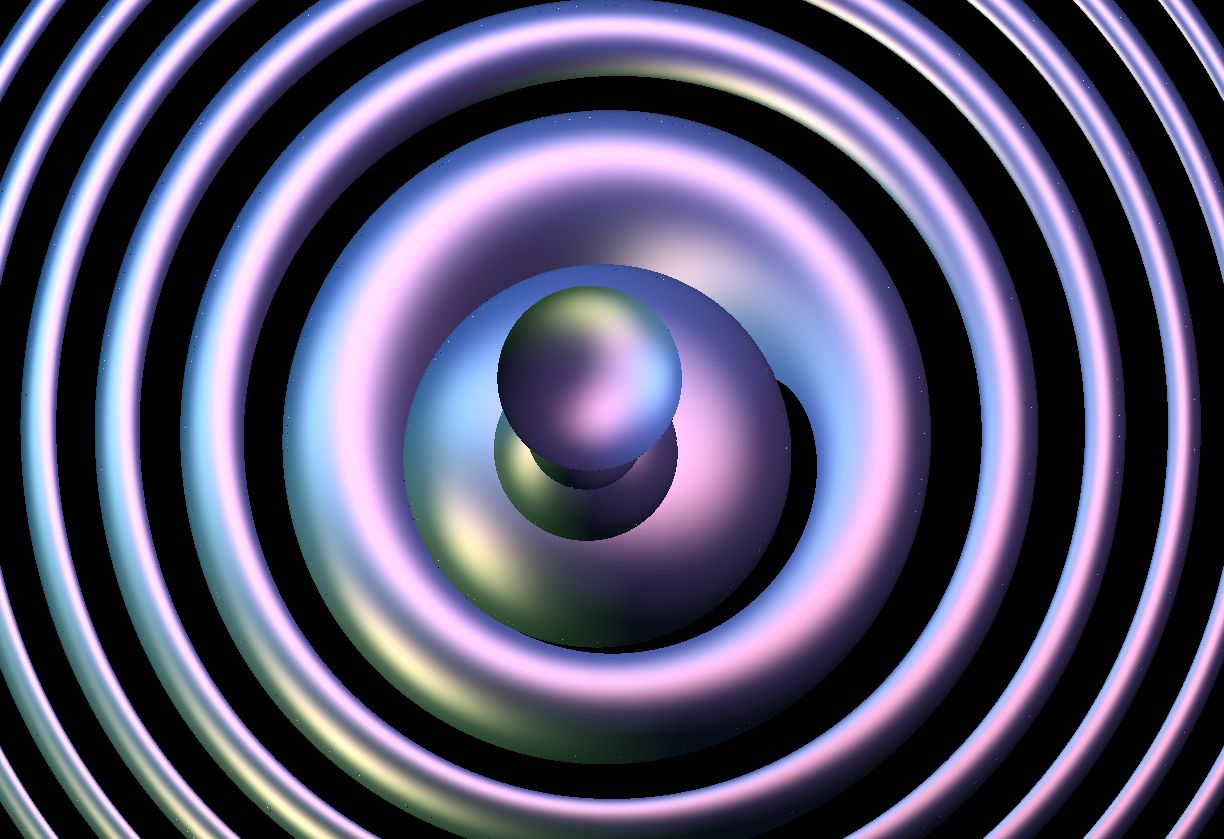}
\label{Fig:niltrue2}
}
\caption{A line of balls in Nil along the $z$-axis, lit by three light sources (cyan, yellow, and magenta). 
The magenta light is sufficiently far away from the first ball that they are connected by several geodesics.
The intensity attenuation has been turned off to emphasize the contribution of each source of light.}
\label{Fig:NilAllDirections}
\end{figure}

\section{Implementing specific geometries}

In previous sections we have described our strategies in a more-or-less geometry independent manner.
Here we begin to give specific details for each of the eight Thurston geometries.
To summarize the previous sections, for each geometry, we require the following:

\begin{enumerate}
\item A model for $X$ with action of the group of isometries $G$. That is, we must now be explicit about how points and isometries are described by vectors or matrices of floating point numbers.
\item Arc length parametrized geodesics in the model. That is, a way to flow a position and tangent vector at that position along the ray by a given distance, as described in \refsec{GeodesicFlow}.
\item Signed distance functions in the model.
\end{enumerate}

\noindent In order to render a quotient manifold with this geometry, we also need:
\begin{itemize}
\item[(4)] A fundamental domain $D$ with face pairings  $\{\gamma_i\} \subset G$.  
\end{itemize}
\noindent
For the Phong reflection model of lighting, we need:
\begin{itemize}
\item[(5)] For a point $s$ (where a ray hits a surface) and the location of a light source $q$, the set of lighting pairs $\calL_s(q)$ of geodesics joining $s$ to $q$ and vice versa. See \refsec{Lighting directions}.
\end{itemize}

\noindent To allow the user to move, we also require
\begin{itemize}
\item[(6)] Parallel transport along geodesic arcs. (Used to translate movement of the user's frame in $\mathbb{R}^3$ into isometries of $X$.)
\end{itemize}

For each of the eight Thurston geometries, we list some of these ingredients in \reftab{GeometryDetails}. All of our models are subsets of $\RR^4$. 

We give further details in the following sections. We consider the isotropic geometries in \refsec{Isotropic}, the product geometries in \refsec{Product}, and Nil, $\SLR$, and Sol in Sections \ref{Sec:Nil}, \ref{Sec:SLR}, and \ref{Sec:Sol} respectively.
A general reference for Thurston's geometries is \cite{Scott}.

    \thispagestyle{empty}
    \begin{landscape}
		\begin{table}[htp]
		\renewcommand{\arraystretch}{3.2}
		\scriptsize
		\begin{tabular}{| c || l |  c | c |  c |}
		\hline
		\bfseries Geometry & {\bfseries Model} (Set, Metric, Origin $o$) & \bfseries  \parbox{1.5in}{\centering Geodesic from $o$\\ in direction $\bv$} & \bfseries Isometries & \bfseries Example Lattices \\ 
		\hline
		\hline
		$\EE^3$ 			& \parbox{4.2cm}{$\RR^4$, $w = 1$,\\ $ds^2 = dx^2 + dy^2 + dz^2$, \\ $o = \be_w$}  									
		& $t \bv$ & $\RR^3 \rtimes {\rm O}(3)$ & $\ZZ^3$ \\
		\hline
		
		$S^3$ 				& \parbox{4.2cm}{$\RR^4$, $x^2 + y^2 + z^2 + w^2 = 1$\\ $ds^2 = dx^2 + dy^2 + dz^2 + dw^2$, \\ $o = \be_w$} 			
		& $\cos(t)\be_w + \sin(t)\bv$ & ${\rm O}(4)$ & \parbox{1.3in}{\centering The eight element quaternion group.}\\ 
		\hline
		
		$\HH^3$				& \parbox{4.2cm}{$\RR^{3,1}$, $x^2 + y^2 + z^2 - w^2 = -1$\\$ds^2 = dx^2 + dy^2 + dz^2 - dw^2$, \\ $o = \be_w$}		
		& $\cosh(t)\be_w + \sinh(t)\bv$ & ${\rm O}(3,1)$ & \parbox{1.3in}{\centering The isometry group of Seifert-Weber space.}\\ 
		\hline
		
		$S^2 \times \EE$ 	& \parbox{4.2cm}{$\RR^3 \times \RR$, $x^2 + y^2 + z^2 = 1$\\ $ds^2 = dx^2 + dy^2 + dz^2 + dw^2$, \\ $o = \be_z$} 		
		& \parbox{1.8in}{ \centering $\displaystyle \left(\cos(\lambda t)\be_z + \sin(\lambda t)\frac{\bv_{S^2}}{\lambda},  t \bv_{\EE}\right)$ where $\bv = (\bv_{S^2}, \bv_\EE)$\\ and $\lambda = \|\bv_{S^2}\|$} 
		& ${\rm O}(3) \times {\rm Isom}(\RR)$ & \parbox{1.3in}{\centering $\Lambda \times \ZZ$ \\ where $\Lambda$ is a discrete \\subgroup of $\textrm{Isom}(S^2)$}\\ 
		\hline
		
		$\HH^2 \times \EE$	& \parbox{4.2cm}{$\RR^{2,1} \times \RR$, $x^2 + y^2 - z^2 = -1$\\ $ds^2 = dx^2 + dy^2 - dz^2 + dw^2$, \\ $o = \be_z$} 
		& \parbox{1.9in}{ \centering $\displaystyle \left(\cosh(\lambda t)\be_z + \sinh(\lambda t)\frac{\bv_{\HH^2}}{\lambda},  t \bv_{\EE}\right)$ where $\bv = (\bv_{\HH^2}, \bv_\EE)$\\ and $\lambda = \|\bv_{\HH^2}\|$} 
		& ${\rm O}(2,1) \times {\rm Isom}(\RR)$ & 
		\parbox{1.32in}{\centering $\Lambda \times \ZZ$ \\ where $\Lambda$ is a discrete \\subgroup of $\textrm{Isom}(\HH^2)$} \\ 
		\hline
		
		Nil 				& \parbox{4.2cm}{$\RR^4, w=1,$ \\ See \refsec{Nil rot model}, \\ $o = \be_w$}						
		& See \refsec{nil geo flow} & \parbox{1.5in}{\centering ${\rm Nil} \rtimes {\rm O}(2)$} & \parbox{1.3in}{\centering$\ZZ^2 \rtimes_M \ZZ$ \\ with $M\in {\rm SL}(2,\ZZ)$, \\ parabolic}\\
		\hline
			
		$\SLR$ 				& \parbox{4.2cm}{$\RR^{2,1} \times \RR$, $x^2 + y^2 - z^2 = -1$\\ See \refsec{model sl2}, \\ $o = \be_z$}									
		& See Sections~\ref{Sec:flow sl2} and \ref{Sec:flow sl2 tilde} & \parbox{1.5in}{\centering$\SLR \rtimes {\rm O}(2)$} &  \parbox{1.3in}{\centering ``Lift'' of $\pi_1(\Sigma_g)$ \\ with $\Sigma_g$ compact genus $g$ surface}\\ 
		\hline
		
		Sol 				& \parbox{4.2cm}{$\RR^4, w=1,$\\ $ds^2 = e^{-2z}dx^2 + e^{2z}dy^2 + dz^2$, \\ $o = \be_w$} 							
		& See \refsec{FlowSol} & ${\rm Sol} \rtimes D_8$&  \parbox{1.3in}{\centering$\ZZ^2 \rtimes_M \ZZ$ \\ with $M\in {\rm SL}(2,\ZZ)$, \\ hyperbolic}\\ 
		\hline

		\end{tabular}
		
		\caption{The eight Thurston geometries. We denote the canonical basis $\{ \be_x, \be_y, \be_z, \be_w \}$.  We write $(x,y,z,w)$ for the coordinates of a vector vector $\bv$ in this basis.  Note that $ {\rm Isom}(\RR) \cong \RR \rtimes \ZZ/2$.}
		\label{Tab:GeometryDetails}
		\renewcommand{\arraystretch}{1}
		\end{table}
	\end{landscape}

\section{Isotropic geometries}
\label{Sec:Isotropic}

In this section we give implementation details for $\EE^3$, $S^3$ and $\HH^3$.
For further background, we refer the reader to \cite[Chapter~I.2]{Bridson:1999ky}. 
See also~\cite{Weeks:2002}. Many details for these three geometries are very similar; for the convenience of the reader, we list these explicitly. In particular, we give distance functions for some simple shapes in standard positions. They can be conjugated by isometries to give signed distance functions for these shapes in general position. We also reference the possible discrete groups (or equivalently, manifolds) for each geometry.

\subsection{Euclidean space}\label{sec:Euclidean}
We represent $\EE^3$ as the affine subspace $X = \{ w = 1\}$ of $\RR^4$.
The origin is the point $o = [0,0,0,1]$. 
The distance between two points $p_1 = [x_1, y_1, z_1, 1]$ and $p_2 = [x_2, y_2, z_2, 1]$ is given by
\begin{equation*}
	\dist(p_1, p_2) = \sqrt{(x_1-x_2)^2 + (y_1-y_2)^2 + (z_1-z_2)^2}.
\end{equation*}
Using a hyperplane to represent $\EE^3$ is standard in computer graphics because the isometry group of $\EE^3$ acts on $X$ by linear transformations of $\RR^4$ preserving $X$.
We identify the tangent space $T_pX$ at a point $p \in X$ with the linear subspace $\{w =0\}$ of $\RR^4$.
The arc length parametrized geodesic $\gamma(t)$ starting at $p$ and directed by the unit vector $v \in T_pX$ is simply $\gamma(t) = p + tv$.
In \reftab{SDF E3}, we list signed distance functions for some simple objects in $\EE^3$.
\begin{table}[htp]
\renewcommand{\arraystretch}{1.5}
\begin{tabular}{|c|c|}\hline
	Object & Signed distance function \\ \hline \hline
	\parbox{5cm}{\centering Ball of radius $r$ centered at the origin $o$} & $\sigma(p) = \sqrt{x^2 + y^2 + z^2} - r$ \\ \hline
	\parbox{5cm}{\centering Solid cylinder of radius $r$ with axis the geodesic $\gamma(t) = o + t \be_z$} & $\sigma(p) = \sqrt{x^2 + y^2} - r$ \\ \hline
	\parbox{5cm}{\centering Half-space $\{ z \leq 0\}$} & $\sigma(p) = z$ \\ \hline
\end{tabular}
\renewcommand{\arraystretch}{1.}
\caption{Examples of signed distance functions in $\EE^3$.}
\label{Tab:SDF E3}
\end{table}

From a group theoretic point of view, the co-compact discrete subgroups of $\EE^3$ have been classified. These are the crystallographic groups \cite{bradley1972mathematical}. Note that every finite volume euclidean three-manifold is finitely covered by the three-torus. 
In \reffig{EucExamples}, we show the in-space view for various scenes within the regular three-torus, rendered with a multicolor collection of five lights. In these images, light intensity falls off proportional to the inverse square of distance. An object receives lighting from the cell it is contained in and that cell's nearest neighbors. 

\begin{figure}[htbp]
\centering
\subfloat[A single large ball.]{
\includegraphics[width=0.90\textwidth]{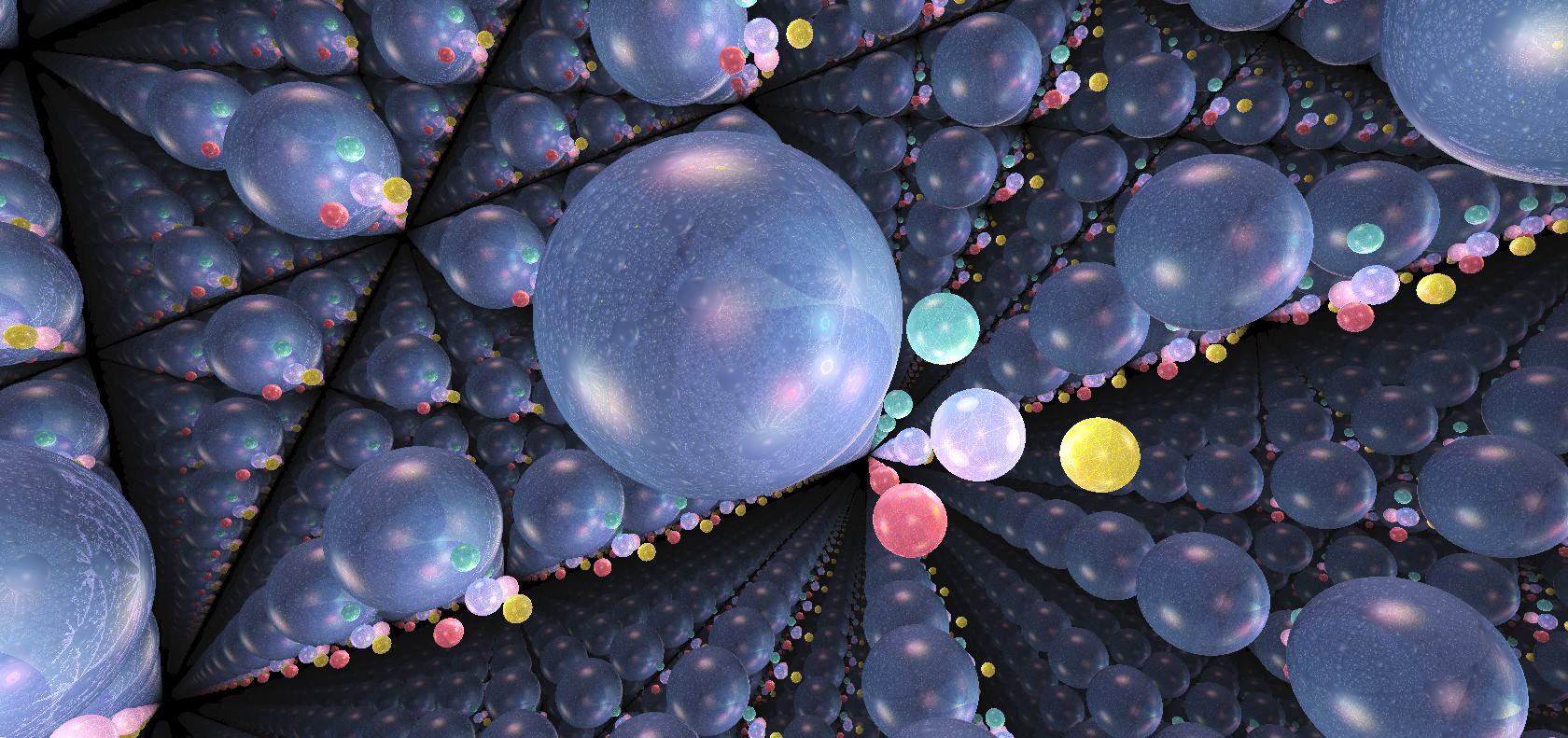}
\label{Fig:EucBalls}
}\
\subfloat[Solid cylinders around the edges of a fundamental domain.]{
\includegraphics[width=0.90\textwidth]{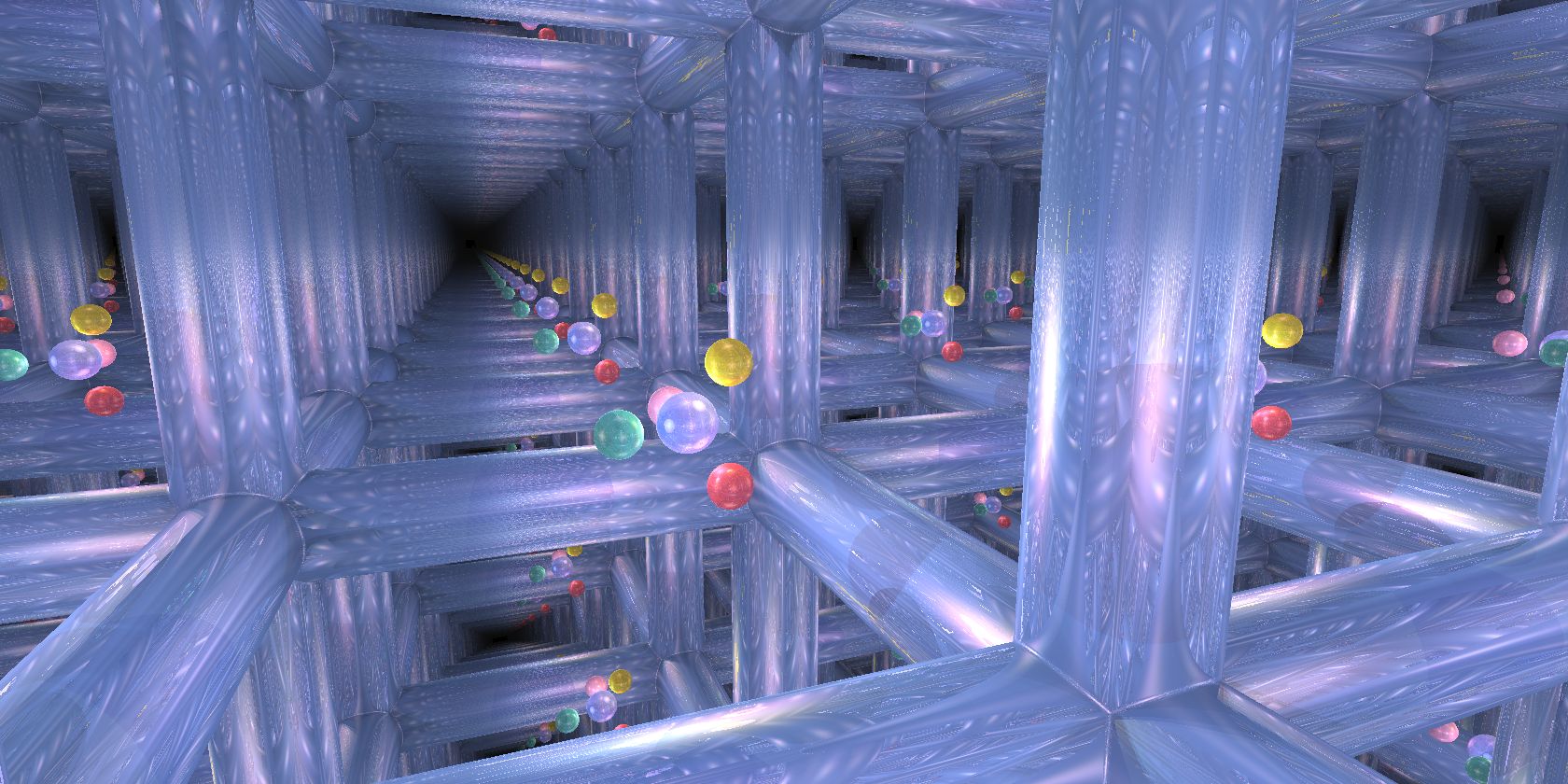}
\label{Fig:EucTubes}
}\
\subfloat[Edges of the fundamental domain rendered by deleting a large ball from the center and smaller balls from the vertices, as in \reffig{primitive cell E3 - advanced cell}.]{
\includegraphics[width=0.90\textwidth]{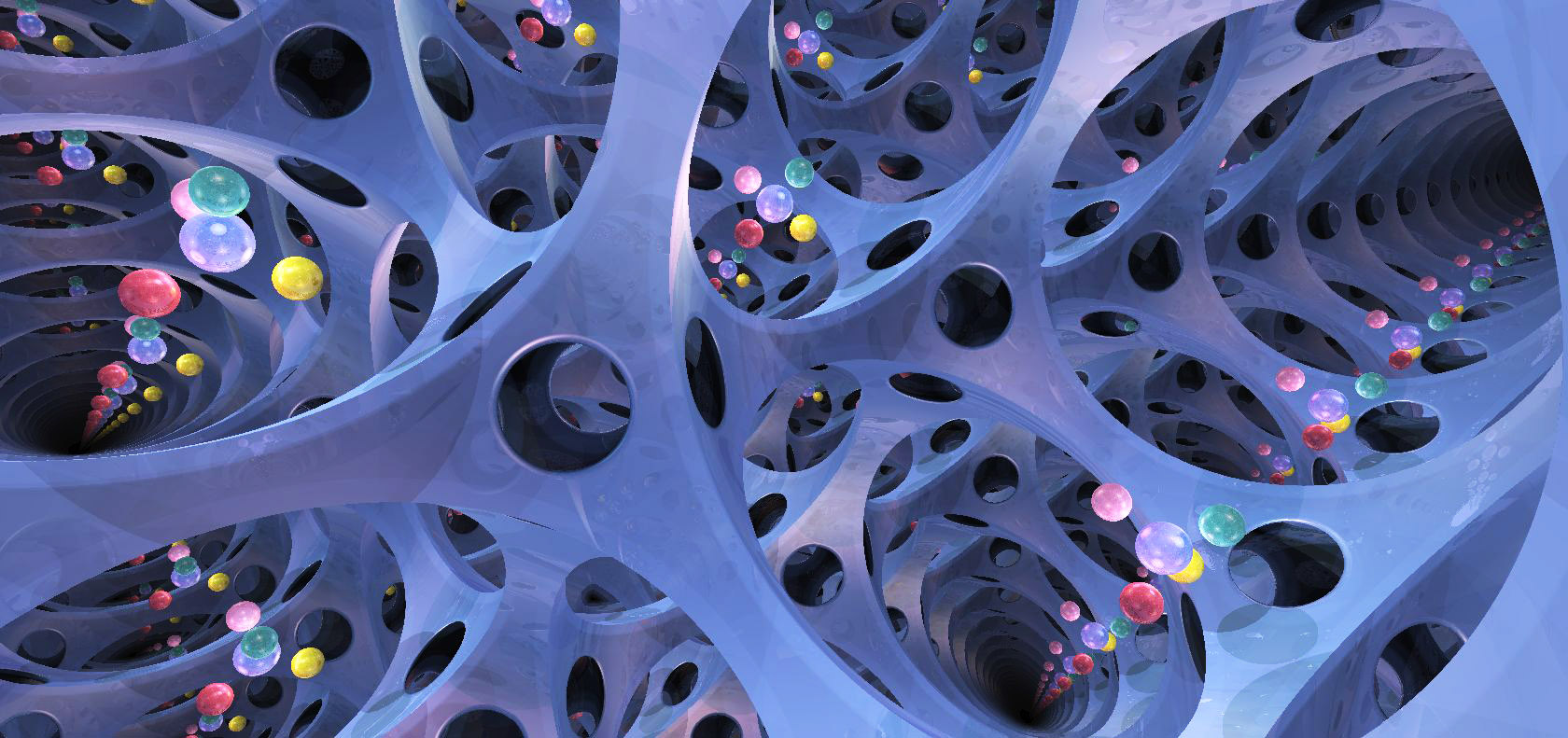}
\label{Fig:EucTiles}
}\
\caption{Scenes in the regular three-torus, lit by a collection of lights represented by balls.}
\label{Fig:EucExamples}
\end{figure}

\subsection{The three-sphere}
We endow $\RR^4$ with the standard scalar product. That is, given $p_1 = [x_1, y_1, z_1, w_1]$ and $p_2 = [x_2, y_2, z_2, w_2]$ we let
\begin{equation*}
	\left< p_1, p_2 \right> = x_1x_2 + y_1y_2 + z_1z_2 + w_1w_2.
\end{equation*}
We view $S^3$ as the set $X$ of points $p\in\RR^4$ satisfying the identity $\left< p, p \right> = 1$.
We choose for the origin the point $o = [0,0,0,1]$.
The distance between two points $p_1 $ and $p_2$ is characterized by 
\begin{equation*}
	\cos\left( \dist(p_1,p_2)\right) = \left< p_1, p_2 \right>.
\end{equation*}
The isometry group of $S^3$ acts on $X$ by linear transformations of $\RR^4$ preserving the scalar product and so $X$.
We identify the tangent space $T_pX$ at a point $p$ in $X$ with the linear subspace 
\begin{equation*}
	\left\{ v \in \RR^4 \mid \left<p,v\right> = 0\right\}
\end{equation*}
of $\RR^4$.
The arc length parametrized geodesic $\gamma(t)$ starting at $p$ and directed by the unit vector $v \in T_pX$ is given by $\gamma(t) = \cos(t)p + \sin(t)v$.
In \reftab{SDF S3}, we list a few examples of signed distance functions in $S^3$.

\begin{table}[htp]
\renewcommand{\arraystretch}{1.5}
\begin{tabular}{|c|c|}\hline
	Object & Signed distance function \\ \hline \hline
	\parbox{5cm}{\centering Ball of radius $r$ centered at the origin $o$} & $\sigma(p) = \arccos(w) - r$ \\ \hline
	\parbox{5cm}{\centering Solid cylinder of radius $r$ whose axis is the geodesic $\gamma(t) = \cos(t)o + \sin(t) \be_z$} & $\sigma(p) = \arccos(\sqrt{w^2+z^2}) - r$ \\ \hline
	\parbox{5cm}{\centering Half-space $\{ z \leq 0\}$} & $\sigma(p) = \arcsin(z)$ \\ \hline
\end{tabular}
\renewcommand{\arraystretch}{1.}
\caption{Examples of signed distance functions in $S^3$.}
\label{Tab:SDF S3}
\end{table}

\begin{figure}[htbp]
\centering
\subfloat[The quotient of $S^3$ by the quaternion group of order eight, $Q_8$.]{
\includegraphics[width=0.90\textwidth]{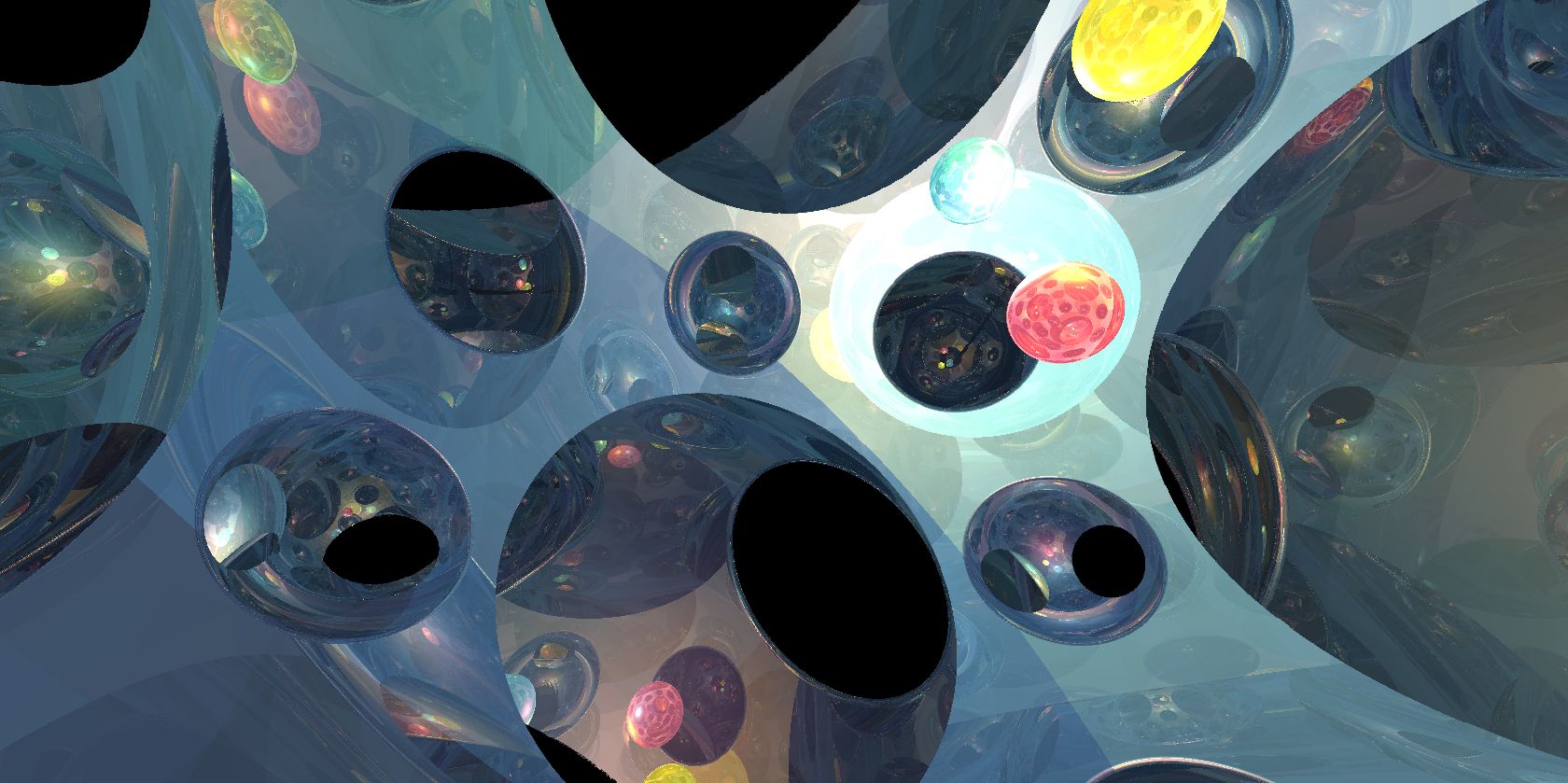}
\label{Fig:ThreeSph_Quaternion}
}

\subfloat[Poincar\'{e} dodecahedral space.]{
\includegraphics[width=0.90\textwidth]{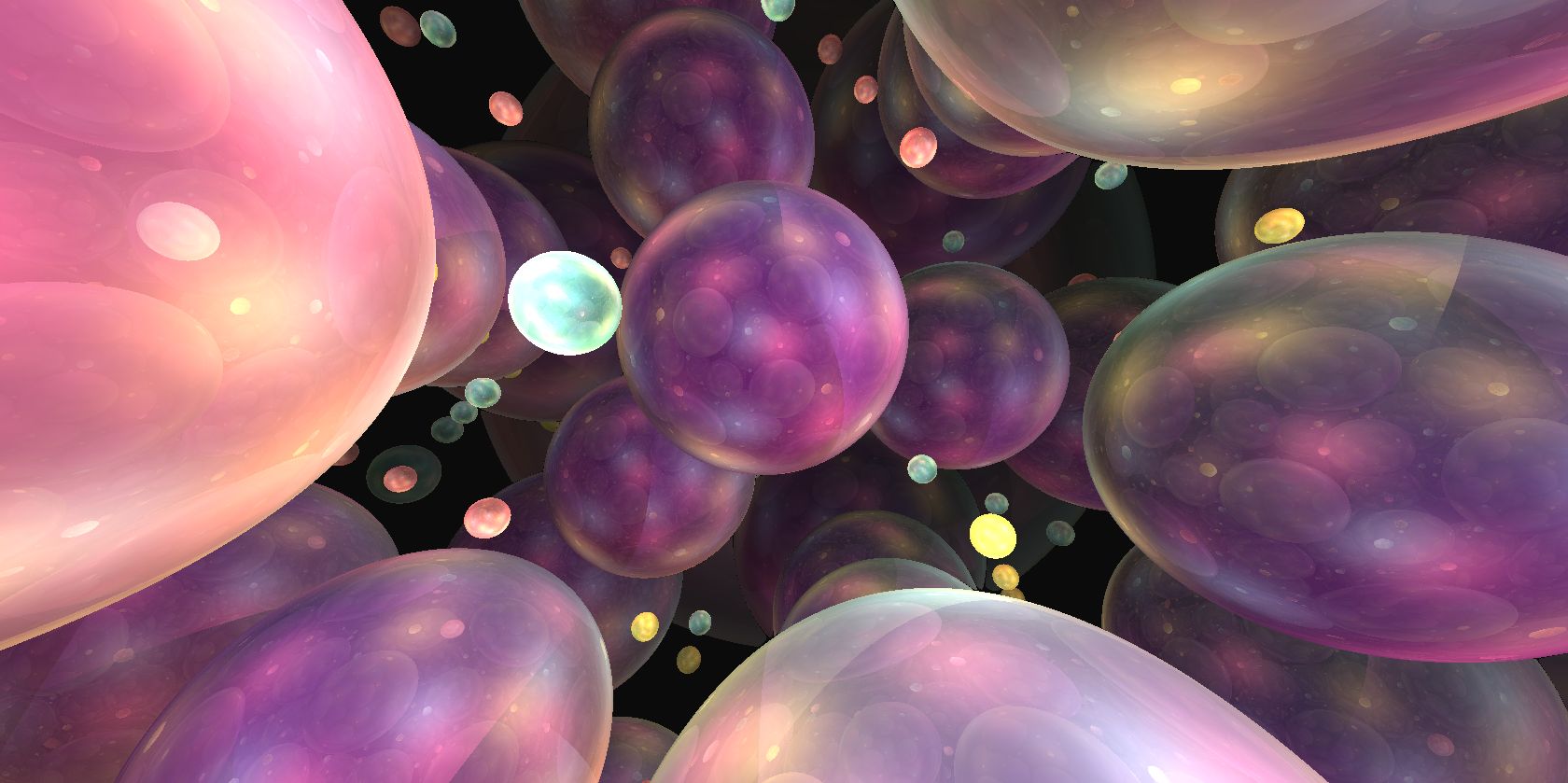}
\label{Fig:PoincareDodecahedral}
}

\subfloat[Hopf fibration.]{
\includegraphics[width=0.90\textwidth]{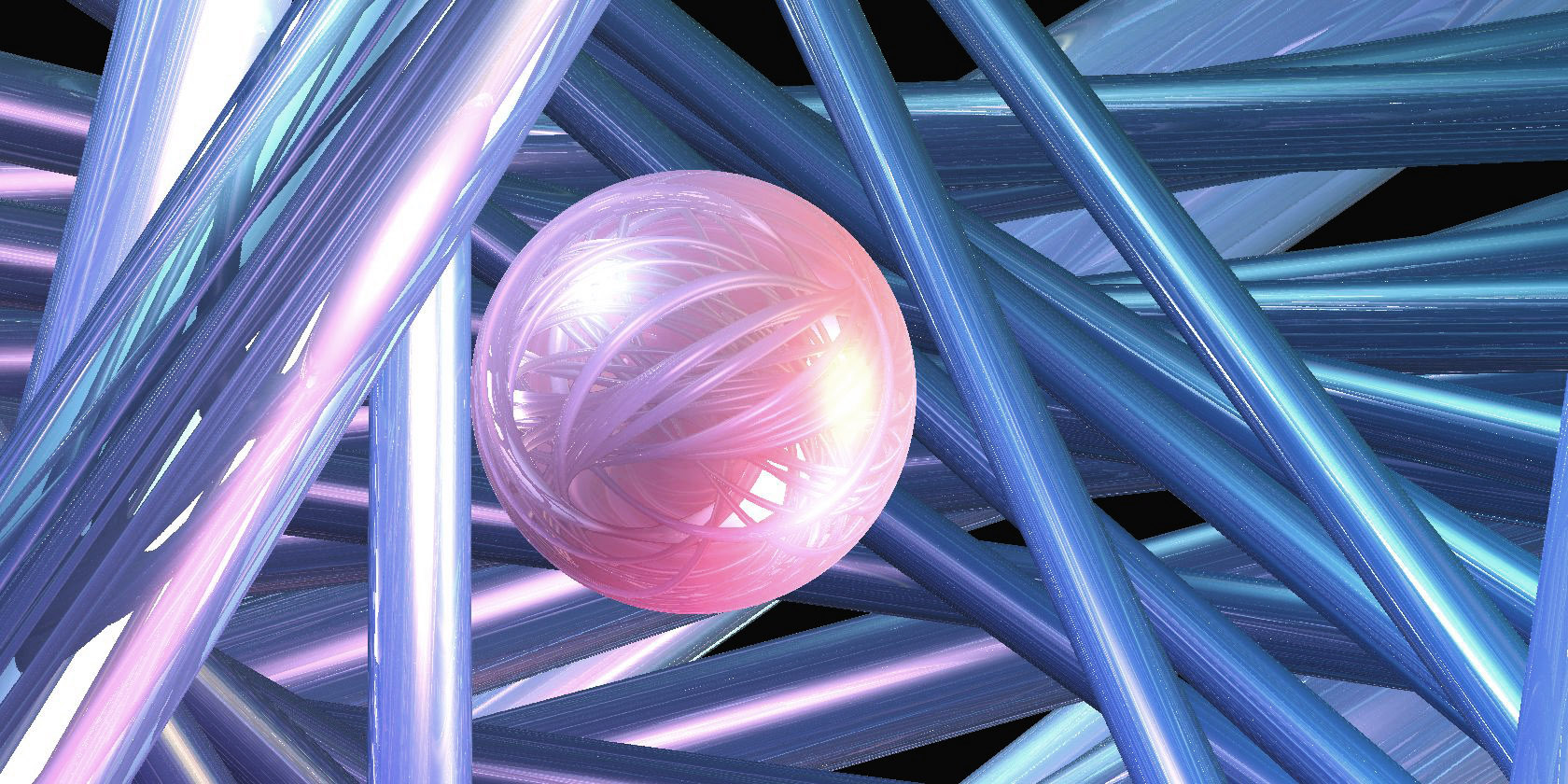}
\label{Fig:HopfFibration}
}
\caption{Spherical Geometry.}
\label{Fig:SphExamples}
\end{figure}

The finite subgroups of ${\rm O}(4)$ are classified in~\cite[page 449]{Scott}. In \reffig{SphExamples} we show the in-space view for various scenes in spherical geometry. \reffig{ThreeSph_Quaternion} shows the quotient of $S^3$ by the quaternion group of order eight, $Q_8$.  Edges of the fundamental domain are shown as in \reffig{primitive cell E3 - advanced cell}, but with balls also deleted from the edge midpoints. \reffig{PoincareDodecahedral} shows a single mirrored ball and three light sources in Poincar\'{e} dodecahedral space. \reffig{HopfFibration} shows the lifts of some randomly chosen fibers of the unit tangent bundle over $S^2$ (the Hopf fibration), and their reflected images in a ball.  These are the fibers of the Seifert fiber space structure on spherical three-manifolds.

\subsection{Hyperbolic space}
\label{Sec:H3}
We endow $\RR^4$ with a lorentzian inner product: for every $p_1 = [x_1, y_1, z_1, w_1]$ and $p_2 = [x_2, y_2, z_2, w_2]$ we let
\begin{equation*}
	\left< p_1, p_2 \right> = x_1x_2 + y_1y_2 + z_1z_2 - w_1w_2.
\end{equation*}
We use the hyperboloid model of $\HH^3$.
This consists of the set $X$ of points $p = [x,y,z,w]$ in $\RR^4$ such that $\left< p, p \right> =-1$ and $w > 0$.
We choose for the origin the point $o = [0,0,0,1]$.
The distance between two points $p_1 $ and $p_2$ is given by 
\begin{equation*}
	\cosh\left(\dist(p_1,p_2)\right) = - \left< p_1, p_2 \right>.
\end{equation*}
The isometry group of $\HH^3$ acts on $X$ by linear transformations of $\RR^4$ preserving the lorentzian product and so $X$.
We identify the tangent space $T_pX$ at a point $p = [x,y,z,w]$ in $X$ with the linear subspace 
\begin{equation*}
	\left\{ v \in \RR^4 \mid \left<p,v\right> = 0\right\}
\end{equation*}
of $\RR^4$.
The arc length parametrized geodesic $\gamma(t)$ starting at $p$ and directed by the unit vector $v \in T_pX$ is given by $\gamma(t) = \cosh(t)p + \sinh(t)v$.
In \reftab{SDF H3}, we list a few examples of signed distance functions in $\HH^3$.
\begin{table}[htp]
\renewcommand{\arraystretch}{1.5}
\begin{tabular}{|c|c|}\hline
	Object & Signed distance function \\ \hline \hline
	\parbox{5cm}{\centering Ball of radius $r$ centered at the origin $o$} & $\sigma(p) = \arccosh(w) - r$ \\ \hline
	\parbox{5cm}{\centering Solid cylinder of radius $r$ whose axis is the geodesic $\gamma(t) = \cosh(t)o + \sinh(t) \be_z$} & $\sigma(p) = \arccosh(\sqrt{w^2-z^2}) - r$ \\ \hline
	\parbox{5cm}{\centering Half-space $\{ z \leq 0\}$} & $\sigma(p) = \arcsinh(z)$ \\ \hline
\end{tabular}
\renewcommand{\arraystretch}{1.}
\caption{Examples of signed distance functions in $\HH^3$.}
\label{Tab:SDF H3}
\end{table}

\begin{figure}[htbp]
\centering
\subfloat[Seifert-Weber dodecahedral space.]{
\includegraphics[width=0.90\textwidth]{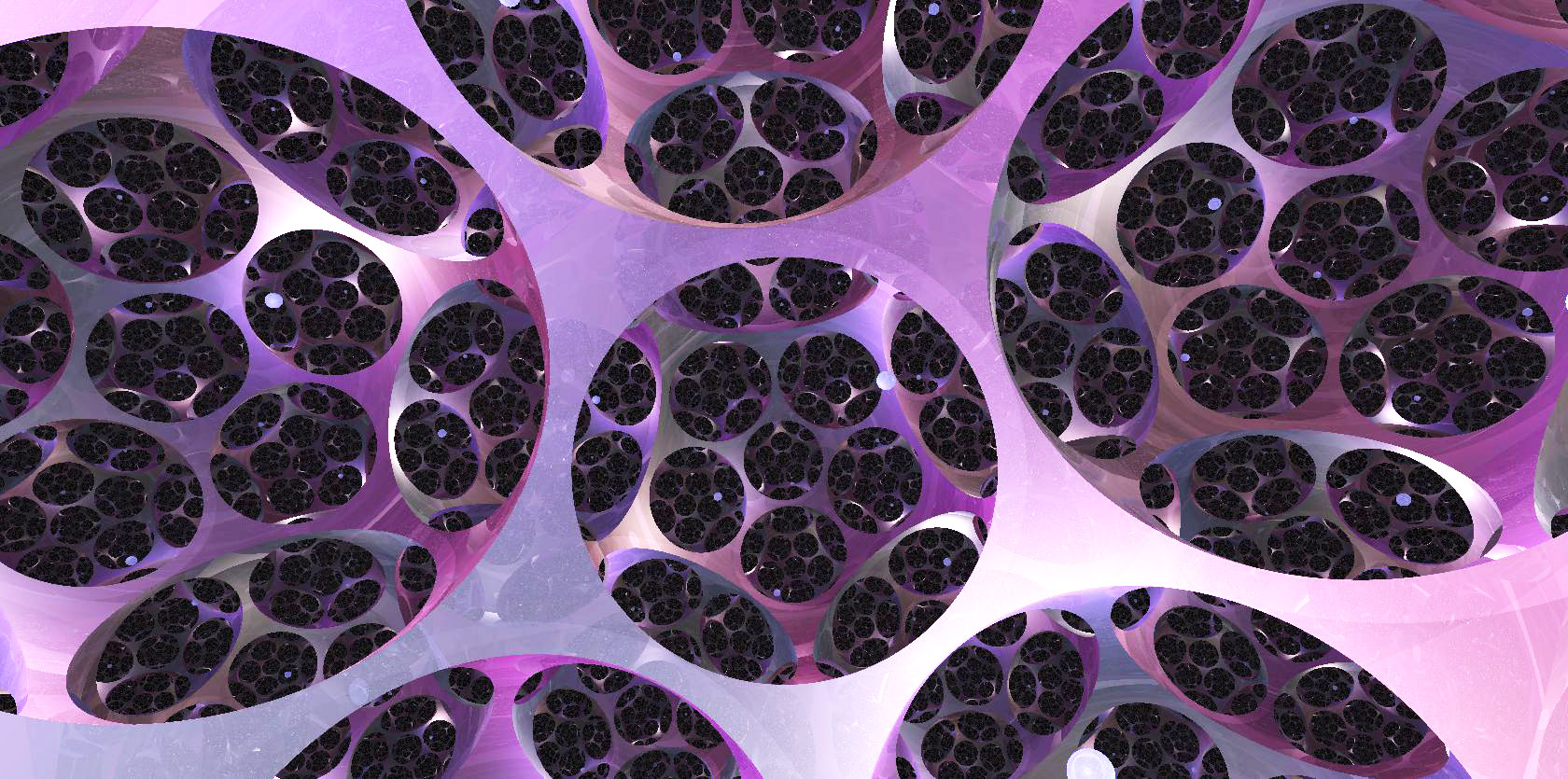}
\label{Fig:SeifertWeber}
}\
\subfloat[A finite volume hyperbolic orbifold.]{
\includegraphics[width=0.90\textwidth]{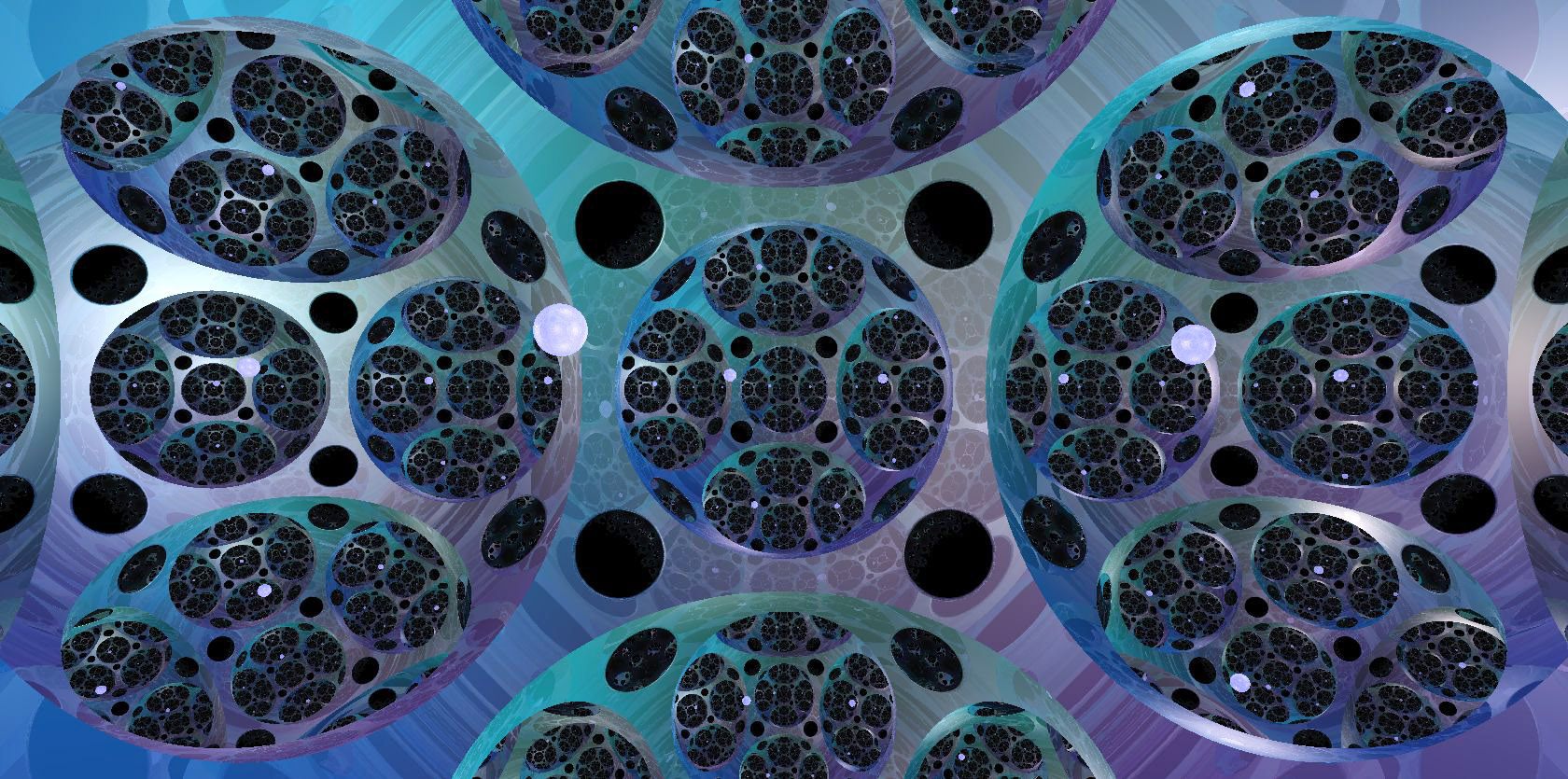}
\label{Fig:HypCube}
}\
\subfloat[An infinite volume hyperbolic orbifold.]{
\includegraphics[width=0.90\textwidth]{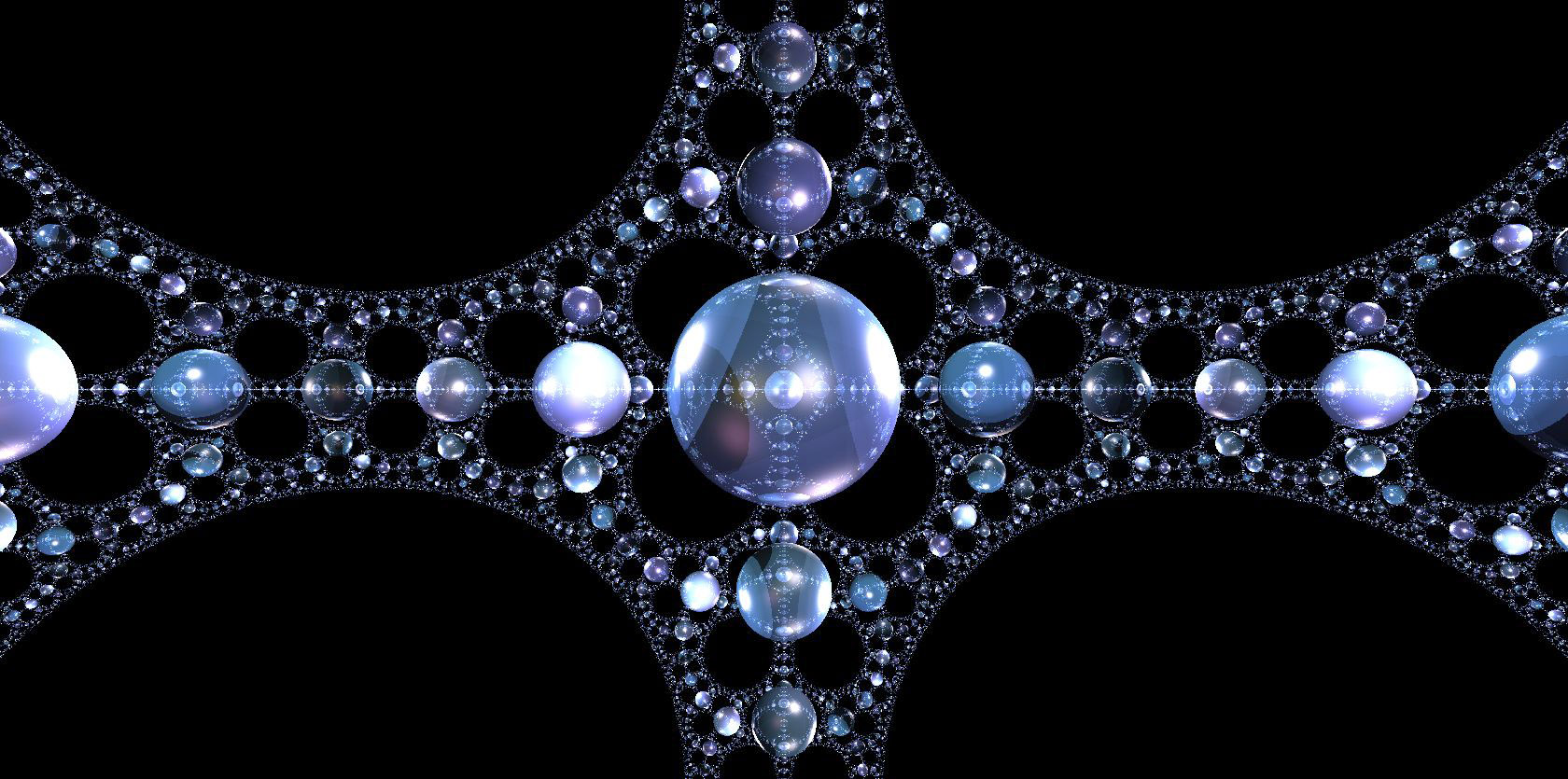}
\label{Fig:HypInfVol}
}\
\caption{Hyperbolic geometry.}
\label{Fig:HypExamples}
\end{figure}

Of the eight Thurston geometries, the classification of hyperbolic manifolds (and orbifolds) is the least well understood.
The software SnapPy~\cite{SnapPy} lists numerous censuses of finite volume hyperbolic manifolds. In \reffig{HypExamples} we show the in-space view for various scenes in hyperbolic geometry. \reffig{SeifertWeber} shows  Seifert-Weber dodecahedral space, with a fundamental domain drawn in a style similar to \reffig{primitive cell E3 - primitive cell}. \reffig{HypCube} shows the finite volume cusped orbifold formed from an ideal cube (with dihedral angles of $\pi/3$), by identifying opposite faces with a $\pi/2$ turn. The underlying manifold is $S^3 / Q_8$ (see \reffig{ThreeSph_Quaternion}) minus the vertices of the cube, with cone angles of $\pi$ at each edge of the cube.
\reffig{HypInfVol} shows a sphere in an infinite volume hyperbolic orbifold formed from a hyperideal cube~\cite[Section 6.1]{visualizing_hyperbolic_honeycombs} (with dihedral angles of $\pi/4$), by identifying opposite faces by translation.  The limit set is the visible as the limiting pattern of spheres. The underlying manifold is the three-torus, minus a ball around the vertex, with cone angles of $\pi$ at each edge of the cube.

\subsection{Facing and parallel transport}
By definition, for each isotropic geometry $X$, the isometry group $G = {\rm Isom}(X)$ acts transitively on the unit tangent bundle of $X$.
It follows that the position and facing of an observer can be captured by a single isometry, as explained in \refsec{PositionFacing}.
Nevertheless, to keep the code as geometry-independent as possible, we encode our position and facing by a pair $(g,{\rm id})$ where $g$ is an isometry of $X$ and ${\rm id} \in {\rm O}(3)$ is the identity.

As we noted in \refsec{moving in the space}, given any geodesic $\gamma \colon \RR \to X$ starting at $p \in X$, there is a one-parameter orientation preserving subgroup $h \colon \RR \to G$ such that $\gamma(t) = h(t)p$.
Thus the corresponding parallel transport operator $T(t) \colon T_{\gamma(0)}X \to T_{\gamma(t)}X$ is simply $T(t)  = d_ph(t)$.
This considerably simplifies the computations: if an observer starts at $(g,{\rm id})$ and follows $\gamma$ for time $t$, then the observer's new position and facing are $(h(t)g,{\rm id})$.

\subsection{Lighting}
\label{Sec:IsotropicLighting}
The calculation of lighting intensity for the isotropic geometries is straightforward in comparison to the other geometries.
Recall from \refeqn{IntensityDensity} that the intensity $I(r,u)$ is inversely proportional to the area density of geodesic spheres.\refeqn{AreaDensity_Jacobi} relates area density directly to Jacobi fields along the geodesic in the direction $u$. 
Here, all sectional curvatures are equal, so all Jacobi fields are parallel along geodesics, and have magnitude controlled by the curvature.
Precisely, if $v\in u^\perp$ and $v_t$ is the parallel transport of $v$ along the geodesic with initial tangent $u$, the corresponding Jacobi fields $J$ are below.
$$J_{\EE^3}(t)=tv_t\hspace{1cm}J_{S^3}(t)=\sin(t)v_t\hspace{1cm}J_{\HH^3}(t)=\sinh(t)v_t$$

Choosing a pair of orthonormal initial conditions and using \refeqn{AreaDensity_Jacobi} gives the area densities:
$$
\mathcal{A}_{\EE^3}(r,u)=r^2\hspace{1cm}	\mathcal{A}_{S^3}(r,u)=\sin(r)^2\hspace{1cm}\mathcal{A}_{\HH^3}(r,u)=\sinh(r)^2.
$$

\noindent
Thus light intensity falls off quadratically with distance in euclidean space, and exponentially in hyperbolic space.
In the three-sphere, the intensity initially decreases with distance, but beyond a distance of $\pi/2$, it increases as all light rays begin to converge towards the antipode.
\reffig{IsotropicIntensities} shows graphs of the intensity function $I(r,u)$ on the tangent space $T_q X$.
A point at distance $r$ from the origin in the direction $u$ is colored by the value of $I(r,u)$. Dark blues represent low intensity, and yellows represent high intensity.
Each plot depicts a ball of radius ten.
Note that $I(r,u)$ for the three-sphere diverges to infinity along spheres with $r=\pi n$ as, under the exponential map, all light refocuses at the light source or its antipode.

\begin{figure}[htbp]
\centering
\subfloat[Euclidean intensity.]{
\includegraphics[width=0.31\textwidth]{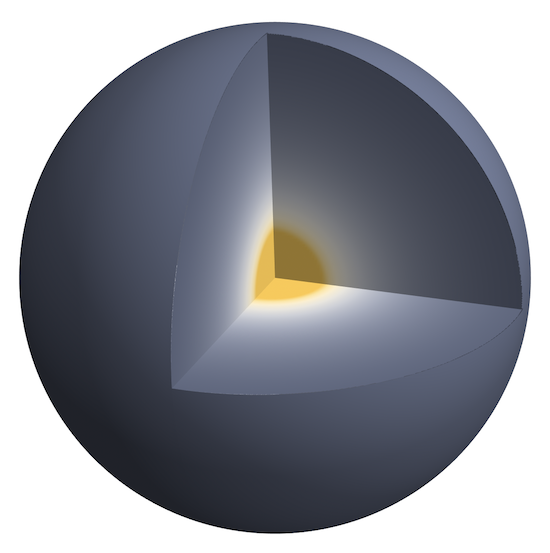}
\label{Fig:EucIntensity}
}
\subfloat[Spherical intensity.]{
\includegraphics[width=0.31\textwidth]{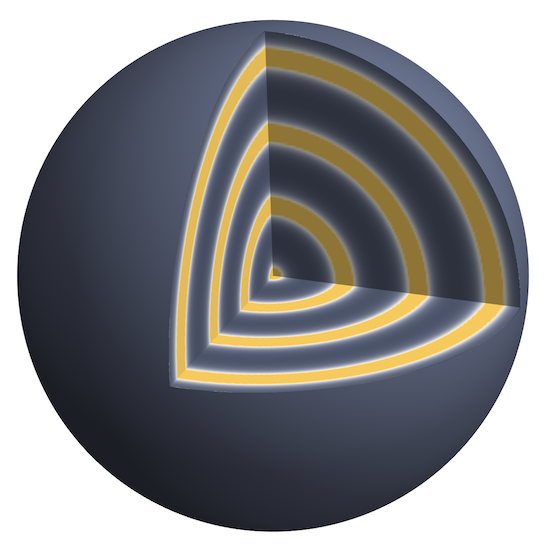}
\label{Fig:SphIntensity}
}
\subfloat[Hyperbolic intensity.]{
\includegraphics[width=0.31\textwidth]{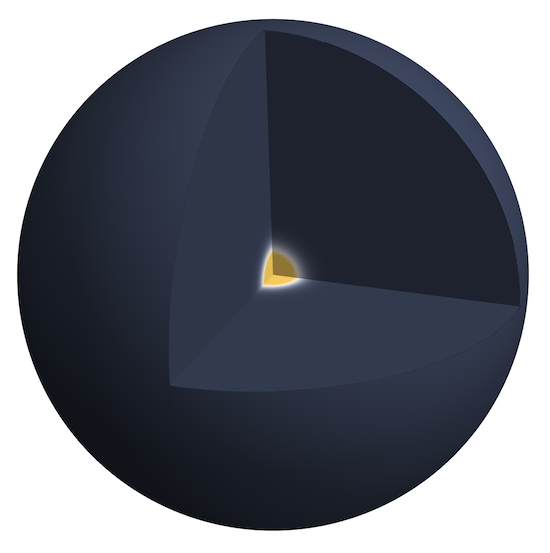}
\label{Fig:HypIntensity}
}
\caption{Graphs of the lighting intensity functions $I(r,u)$ for the isotropic geometries, drawn in the tangent space at the light source.}
\label{Fig:IsotropicIntensities}
\end{figure}

We now turn to the calculation of the lighting pairs $\calL_s(q)$: the set of pairs $(L,d_L)$ of initial tangent vectors $L$ to geodesics joining $s$ to $q$, and their corresponding lengths $d_L$.
In all three isotropic geometries, this can be calculated using linear algebra in the ambient space $\RR^4$ where the models reside.

In euclidean space, geodesics are unique. Given $s,q \in \EE^3$, the required direction vector is simply $q-s$.
$$\calL_s^{\EE^3}(q)=\left\{\left(\frac{q-s}{\|q-s\|},\|q-s\|\right)\right\}$$
In spherical geometry, given $s,q\in S^3$ non-antipodal, let $\theta=\arccos\langle q,s\rangle$ be the acute angle between them.
The shortest geodesic from $s$ to $q$ has length $\theta$ and direction $v=q-\langle s,q\rangle q$, appropriately rescaled.
The second geodesic points in the opposite direction, with length $2\pi-\theta$.
$$\calL_s^{S^3}(q)=\left\{
\left(\frac{v-\cos(\theta)s}{\sin\theta},\theta \right),\left(\frac{\cos(\theta)s-v}{\sin\theta},2\pi-\theta \right)\right\}$$

\begin{remark}
\label{Rem:GoingInCircles}
Strictly speaking, we should also include copies of the above pairs with distances modified by adding $2\pi n$ for all integers $n > 0$. However, if either the light source or the scene is opaque then these copies are never relevant. 
\end{remark}

In practice, we don't worry about $s$ and $q$ being antipodal: in a generic render, no pixels will involve such a situation.
Moreover, if we are exceedingly unlucky and do have such a pixel, GPU code does not crash when asked to, for example, divide by zero. It just gives up and moves on to the next pixel. 
However, one could special-case this situation:
for a pair of antipodal points $s,q$, all directions from $s$ reach $q$ after traveling a distance $\pi$, and so we find that the set $\calL^{S^3}_s(q)$ is uncountable. As the lighting intensity diverges to infinity as one approaches such a configuration, the pixels should be colored as bright as possible.

In hyperbolic geometry we proceed analogously to the three sphere, except that we use the Minkowski inner product.
Given $s,q\in\HH^3$, let $\delta=\arccosh|\langle q,s\rangle|$ be the hyperbolic distance between them.
Geodesics between pairs of points in $\HH^3$ are unique, so $\calL^{\HH^3}_s(q)$ is again a singleton:
$$\calL_s^{\HH^3}(q)=\left\{\left(\frac{v-\cosh(\delta)s}{\sinh\delta},\delta \right)\right\}.$$

\section{Product geometries}
\label{Sec:Product}
Before describing the product geometries, we quickly introduce model spaces for $S^2$ and $\HH^2$.

\subsection{Models of $S^2$ and $\HH^2$}
\label{Sec:S2H2Models}
Our models for $S^2$ and $\HH^2$ are the same as those for $S^3$ and $\HH^3$, with one fewer dimension:
\begin{itemize}
	\item We view $S^2$ as the set $\mathcal S$ of points $q =[x,y,z]$ in $\RR^3$ such that $\left< q,q \right> = 1$, where $\left< \cdot\, , \cdot \right>$ is the canonical scalar product in $\RR^3$.
	\item We represent $\HH^2$ as the set $\mathcal H$ of points $q = [x,y,z]$ in  $\RR^3$ such that $\left< q,q \right> = -1$, where 
$	\left<q_1, q_2 \right> = x_1x_2 + y_1y_2 - z_1z_2 $
is the lorentzian product in $\RR^3$.
\end{itemize}

\subsection{Product geometries}
Our model for $S^2 \times \EE$  (respectively $\HH^2 \times \EE$) is the subset $X = Y \times \RR$ of $\RR^4$, where $Y = \mathcal S$ (respectively  $Y = \mathcal H$).
We choose for the origin the point $o = [0,0,1,0]$.
The space $X$ is equipped with the product distance. That is, given two points $p_1 = (q_1, w_1)$ and $p_2 = (q_2, w_2)$ in $Y \times \RR$ we have
\begin{equation*}
	\dist_X(p_1,p_2)^2 = \dist_Y(q_1,q_2)^2 + |w_1 - w_2|^2.
\end{equation*}
The tangent space $T_pX$ at a point $p = (q,w)$ naturally splits as $T_q Y \times \RR$.
Given a vector $v \in T_pX$ we denote by $v_Y$ and $v_{\EE}$ its components in $T_q Y$ and $\RR$ respectively.
The arc length parametrized geodesic $\gamma(t)$ starting at $p=(q,w)$ in the direction of the unit vector $v \in T_pX$ is given by 
\begin{equation*}
	\gamma(t) = \big( \gamma_Y(\| v_Y\| t), w + t v_{\EE}\big),
\end{equation*}
where $\gamma_Y \colon \RR \to Y$ is the geodesic ray in $Y$ starting at $q$ with initial tangent vector $v_Y / \| v_Y\|$.

Next, we consider signed distance functions.
As usual, the distance formula gives us the signed distance function for a ball.
We call an object $\mathcal V$ \emph{vertical} if it is the pre-image of a non empty subset $\mathcal U \subset Y$ by the projection $\pi \colon X \to Y$.
The signed distance function for such an object $\mathcal V$ is given by 
\begin{equation*}
	\sigma(p) = \dist_X(p, \mathcal V) =  \dist_Y(\pi (p), \mathcal U).
\end{equation*}
We define \emph{horizontal} objects, and obtain their signed distance functions in an analogous way.
Tables~\ref{Tab:SDF S2xE} and \ref{Tab:SDF H2xE} list a few examples of such signed distance functions.

\begin{table}[htp]
\renewcommand{\arraystretch}{1.5}
\begin{tabular}{|c|c|}\hline
	Object & Signed distance function \\ \hline \hline
	\parbox{5cm}{\centering Solid cylinder of radius $r$ with axis the geodesic $\gamma(t) = o + t \be_w$} & $\sigma(p) = \arccos(z) - r$ \\ \hline
	\parbox{5cm}{\centering Half-space $\{ y \leq 0\}$} & $\sigma(p) = \arcsin(y)$ \\ \hline
	\parbox{5cm}{\centering Half-space $\{ w \leq 0\}$} & $\sigma(p) = w$ \\ \hline
\end{tabular}
\renewcommand{\arraystretch}{1.}
\caption{Examples of signed distance functions in $S^2 \times \RR$.}
\label{Tab:SDF S2xE}
\end{table}

\begin{table}[htp]
\renewcommand{\arraystretch}{1.5}
\begin{tabular}{|c|c|}\hline
	Object & Signed distance function \\ \hline \hline
	\parbox{5cm}{\centering Cylinder of radius $r$ whose axis is the geodesic $\gamma(t) = o + t \be_w$} & $\sigma(p) = \arccosh(z) - r$ \\ \hline
	\parbox{5cm}{\centering Half-space $\{ y \leq 0\}$} & $\sigma(p) = \arcsinh(y)$ \\ \hline
	\parbox{5cm}{\centering Half-space $\{ w \leq 0\}$} & $\sigma(p) = w$ \\ \hline
\end{tabular}
\renewcommand{\arraystretch}{1.}
\caption{Examples of signed distance functions in $\HH^2 \times \RR$.}
\label{Tab:SDF H2xE}
\end{table}

\begin{figure}[htbp]
\centering
\subfloat[$S^2\times\EE$.]{
	\includegraphics[width=0.45\textwidth]{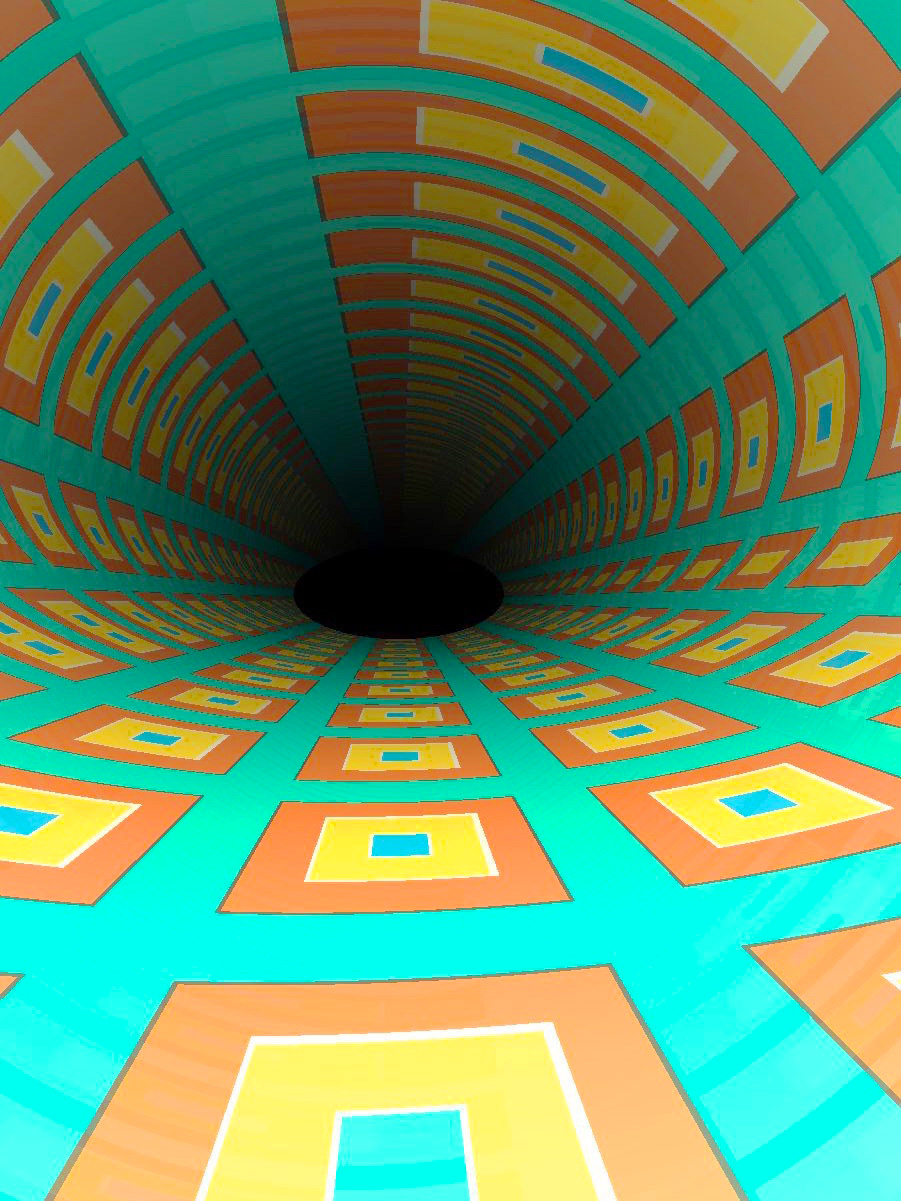}
}
\subfloat[$\HH^2\times\EE$.]{
	\includegraphics[width=0.45\textwidth]{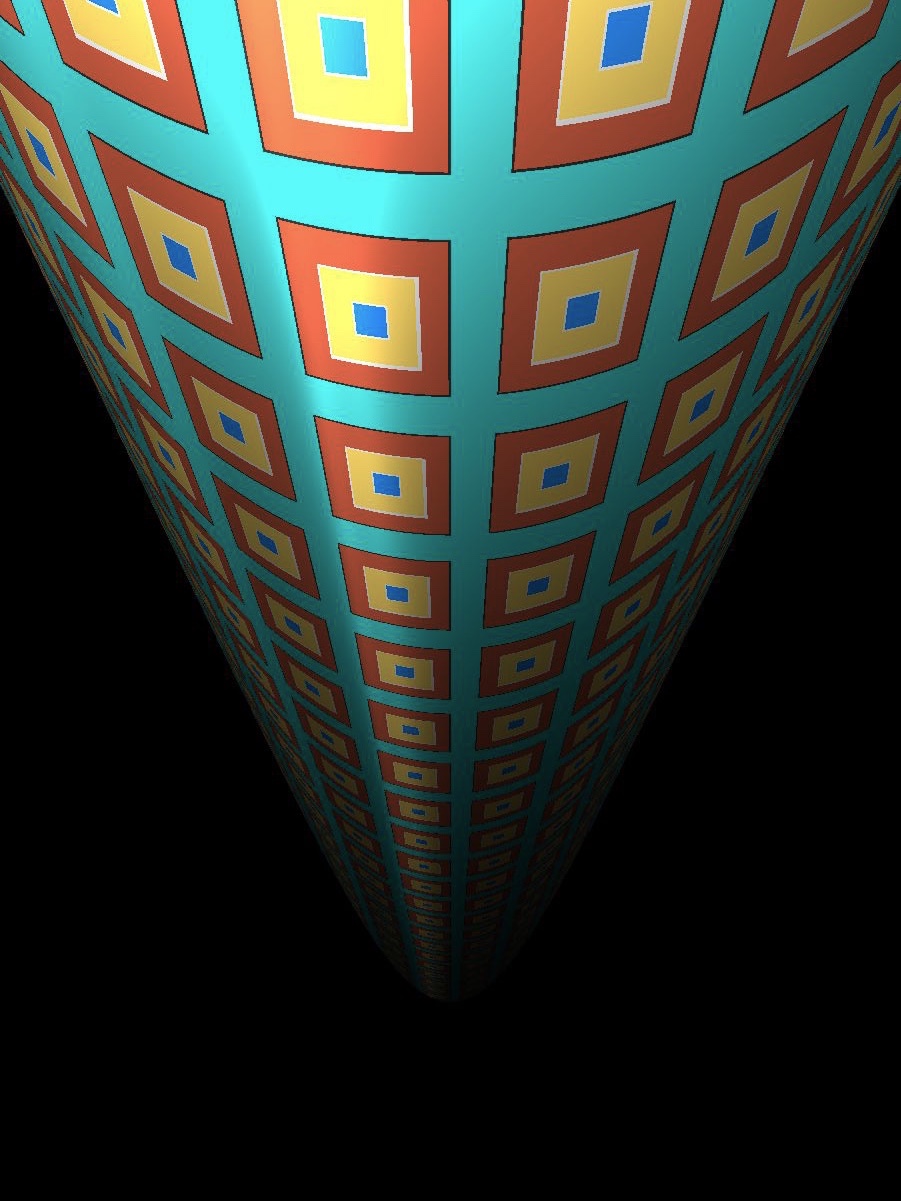}
}
\caption{Vertical half-spaces in the $S^2\times\EE$ and $\HH^2\times\EE$ geometries.}
\label{Fig:ProductVertPlanes}
\end{figure}

\begin{figure}[htbp]
\centering
\subfloat[Fibers in the $S^2\times\EE$. ]{
\includegraphics[width=0.45\textwidth]{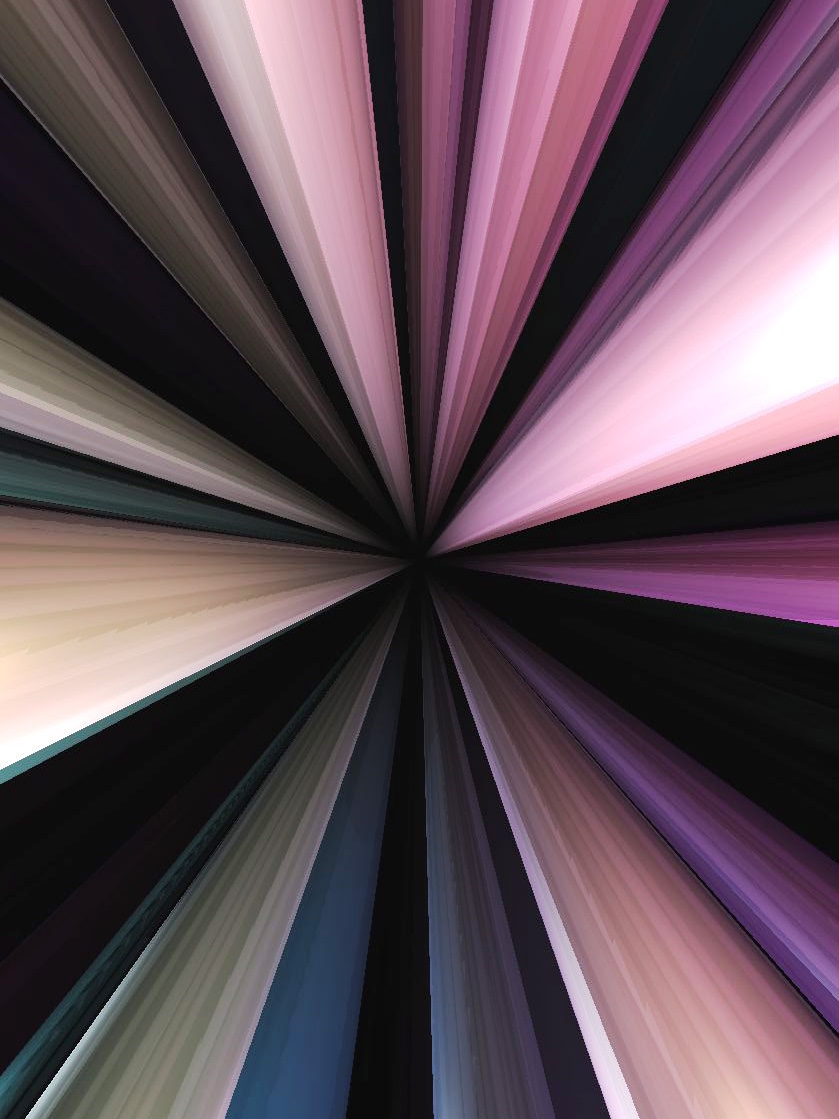}
\label{Fig:S2xEFibers}
}\
\subfloat[Fibers in $\HH^2\times\EE$. ]{
\includegraphics[width=0.45\textwidth]{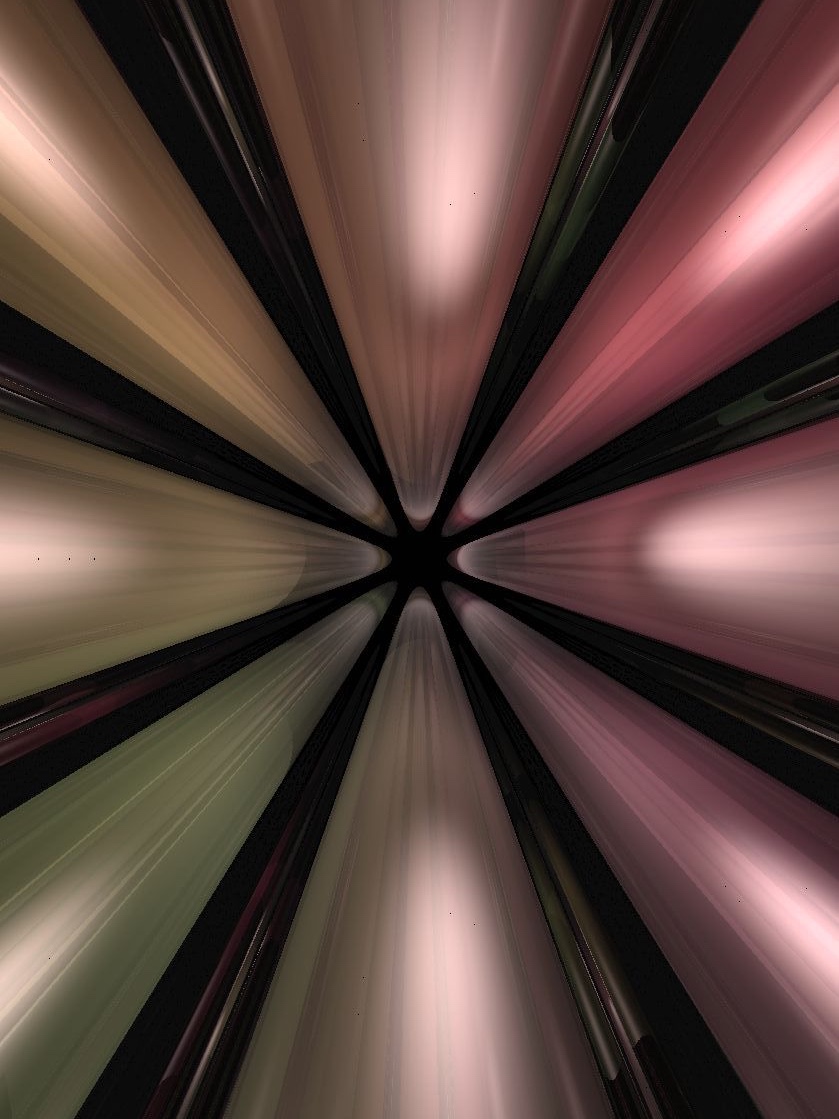}
\label{Fig:H2xEFibers}
}\
\caption{Fibers of the Seifert fiber space structures in manifolds with product geometry.}
\label{Fig:SeifFiberProduct}
\end{figure}

\reffig{ProductVertPlanes} shows vertical half-spaces in the product geometries.
\reffig{SeifFiberProduct} shows solid cylinders around some fibers in the $\EE$ direction for the product geometries.
In \reffig{S2xEFibers} we place solid cylinders around fibers above the vertices of an icosahedron in the $S^2$ factor. In   
 \reffig{H2xEFibers} the solid cylinders are around fibers in the $\EE$ direction.

\subsection{Facing and parallel transport}
Unlike for the isotropic geometries, the position and facing of the observer cannot be encoded with a single element of $G = {\rm Isom}(X)$.
Hence we represent it by a pair $(g,m) \in G \times {\rm O}(3)$ as explained in \refsec{PositionFacing}.
Nevertheless, if $\gamma \colon \RR \to X$ is a geodesic starting at the observer's position $p$, there is still a one-parameter orientation preserving subgroup $h \colon \RR \to G$ such that $\gamma(t) = h(t)p$.
Thus after moving along $\gamma$ for a time $t$, the observer's new position and facing is $(h(t)g,m)$.

\subsection{Lighting}
\label{Sec:ProductLighting}
We again use \refeqn{AreaDensity_Jacobi} to reduce the calculation of  area density (and hence light intensity) to the computation of Jacobi fields.
Let $q\in X$, choose a unit vector $u\in T_qX$ and let $\gamma$ be the geodesic starting at $q$ with initial tangent $u$.
General Jacobi fields need not be parallel along $\gamma$, and may rotate in the presence of a gradient in sectional curvature.
When $v\in u^\perp$ is such that the curvature $\kappa$ of the plane spanned by $\{u,v\}$ is a local extremum however, then the Jacobi field with initial condition $\dot{J}(0)=v$ is parallel along $\gamma$. In this case, its magnitude is determined by $\kappa$, as in \refsec{IsotropicLighting}.

If $u$ is vertical (that is, $u_Y = 0$), then $X$ is symmetric under rotation about $u$, and all planes containing $u$ have zero sectional curvature.
If $v\in u^\perp$ has parallel translate $v_t$ along $\gamma$, then the corresponding Jacobi field is $J(t)=tv_t$.
Choosing two such orthonormal conditions, \refeqn{AreaDensity_Jacobi} implies that $\mathcal{A}_X(r,u)=r^2$.

In general, suppose that $u$ makes an angle of $\beta$ with the vertical.
Then $u$ is contained in a unique vertical plane $V$, which again has zero sectional curvature. This realizes one of the extremal curvatures at $u$ (it is a maximum for $\HH^2\times\EE$ and a minimum for $S^2\times\EE$).
Choosing $v\in T_q X$ extending $u$ to an orthonormal basis for $V$, the Jacobi field with initial condition $v$ is $J(t)=tv_t$ as above.

The other extremal curvature is realized by the plane $P$, orthogonal to $V$ and containing $u$.
Using the bilinearity of the Riemann curvature tensor, one can calculate this extremal curvature from the angle $\beta$ that $u$ makes with the vertical, and the curvature $K(H)=\pm 1$ of the horizontal $H$ plane $H$:
$$K(P)=\cos^2(\beta) K(V)+\sin^2(\beta)K(H)=\pm\sin^2(\beta)$$

Let $w\in T_qX$ extend $u$ to an orthonormal basis for $P$, and $w_t$ be its parallel translate along $\gamma$.  The Jacobi field with initial condition $w$ is $J(t)=\frac{f(t\sin\beta)}{\sin\beta}w_t$, where $f$ is either sine or hyperbolic sine as $K(P)$ is greater or less than zero respectively.
Combining these with \refeqn{AreaDensity_Jacobi} 
gives the area density for each of the product geometries below.

\begin{equation}
\mathcal{A}_{S^2\times\EE}(r,u)=r\frac{\sin(r\sin\beta)}{\sin\beta}
\hspace{1cm}
\mathcal{A}_{\HH^2\times\EE}(r,u)=r\frac{\sinh(r\sin\beta)}{\sin\beta}
\end{equation}

\reffig{ProductIntensity} shows the behavior of $I(r,u)=1/\mathcal{A}(r,u)$ on a ball of radius ten in the tangent space at $q$ for the two product geometries. \reffig{Product_LightInSpace} shows some effects of this behavior on the in-space view.

\begin{figure}[htbp]
\centering
\subfloat[The intensity in $S^2\times\EE$ periodically blows up.]{
\includegraphics[width=0.4\textwidth]{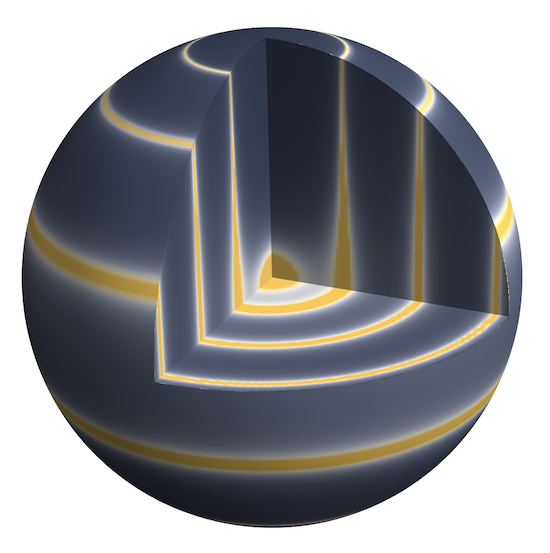}
\label{Fig:H2E_Intensity}
}
\quad
\subfloat[The intensity in $\HH^2\times\EE$ drops off exponentially away from the $\EE$ direction.]{
\includegraphics[width=0.4\textwidth]{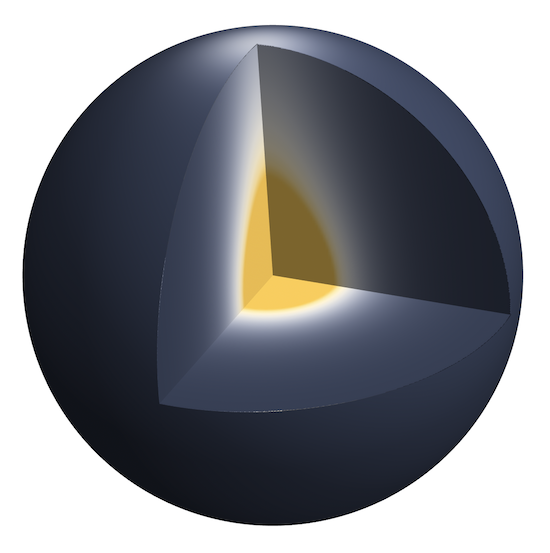}
\label{Fig:H2E_Intensity}
}
\caption{The lighting intensity functions $I(r,u)$ in the product geometries.}
\label{Fig:ProductIntensity}
\end{figure}

\begin{figure}[htbp]
\centering
\subfloat[The sphere $S^2 \times \{0\}$ in $S^2\times\EE$, lit by a single light above the north pole. The viewer is in the same position as the light, looking along the $\EE$ direction. The light intensity blows up at both the north and south poles of $S^2 \times \{0\}$. The viewer sees each pole as a collection of concentric rings, together with a point for the north pole directly below.]{
\includegraphics[width=0.90\textwidth]{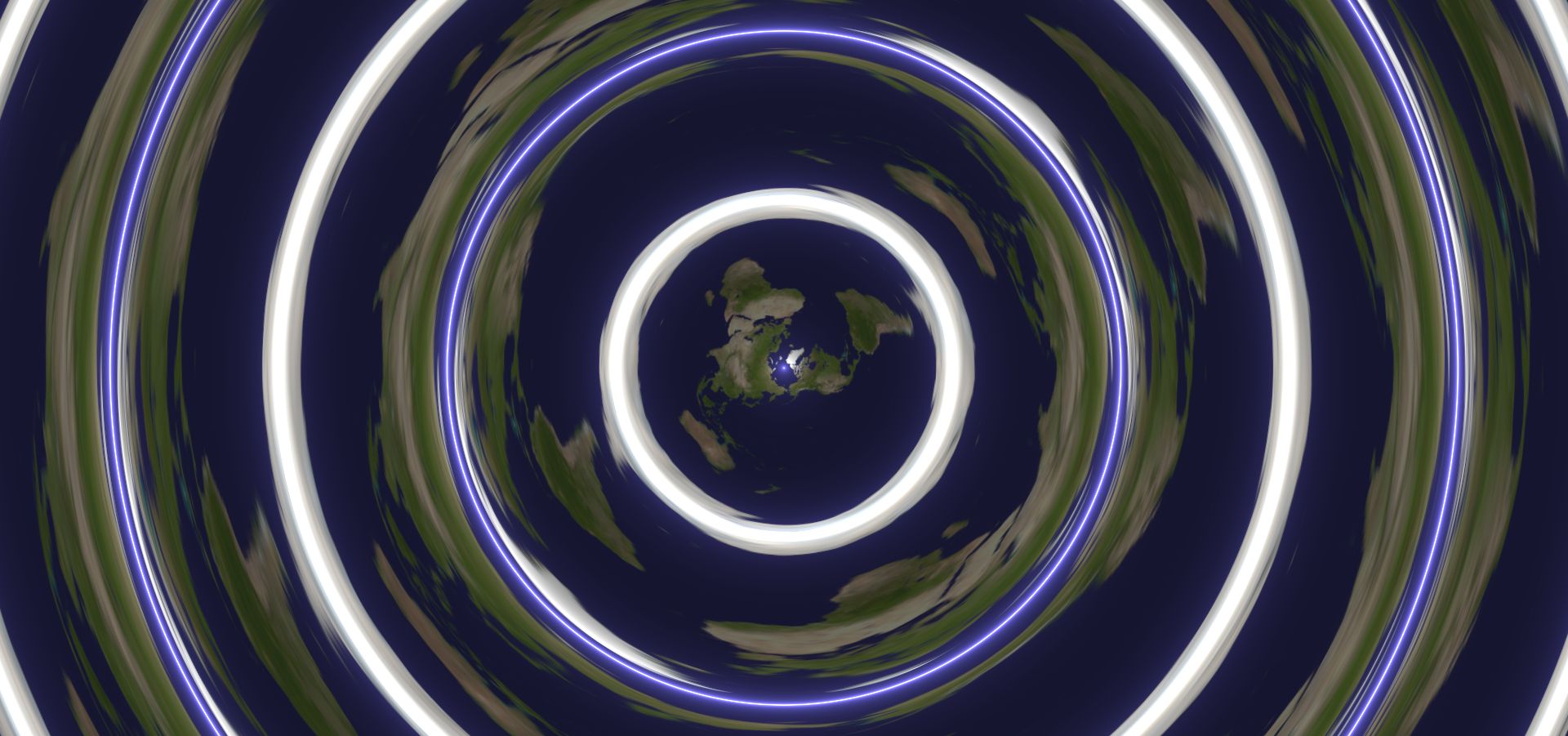}
\label{Fig:S2E_LightInSpace}
}\\
\subfloat[A tiling with a single light source in $\HH^2\times\EE$.  The light source is in the center of one of the bright tiles in front of the viewer. From a distance, the exponential fall-off in the hyperbolic directions makes the light look like a spotlight shining along the $\EE$ direction.]{
\includegraphics[width=0.90\textwidth]{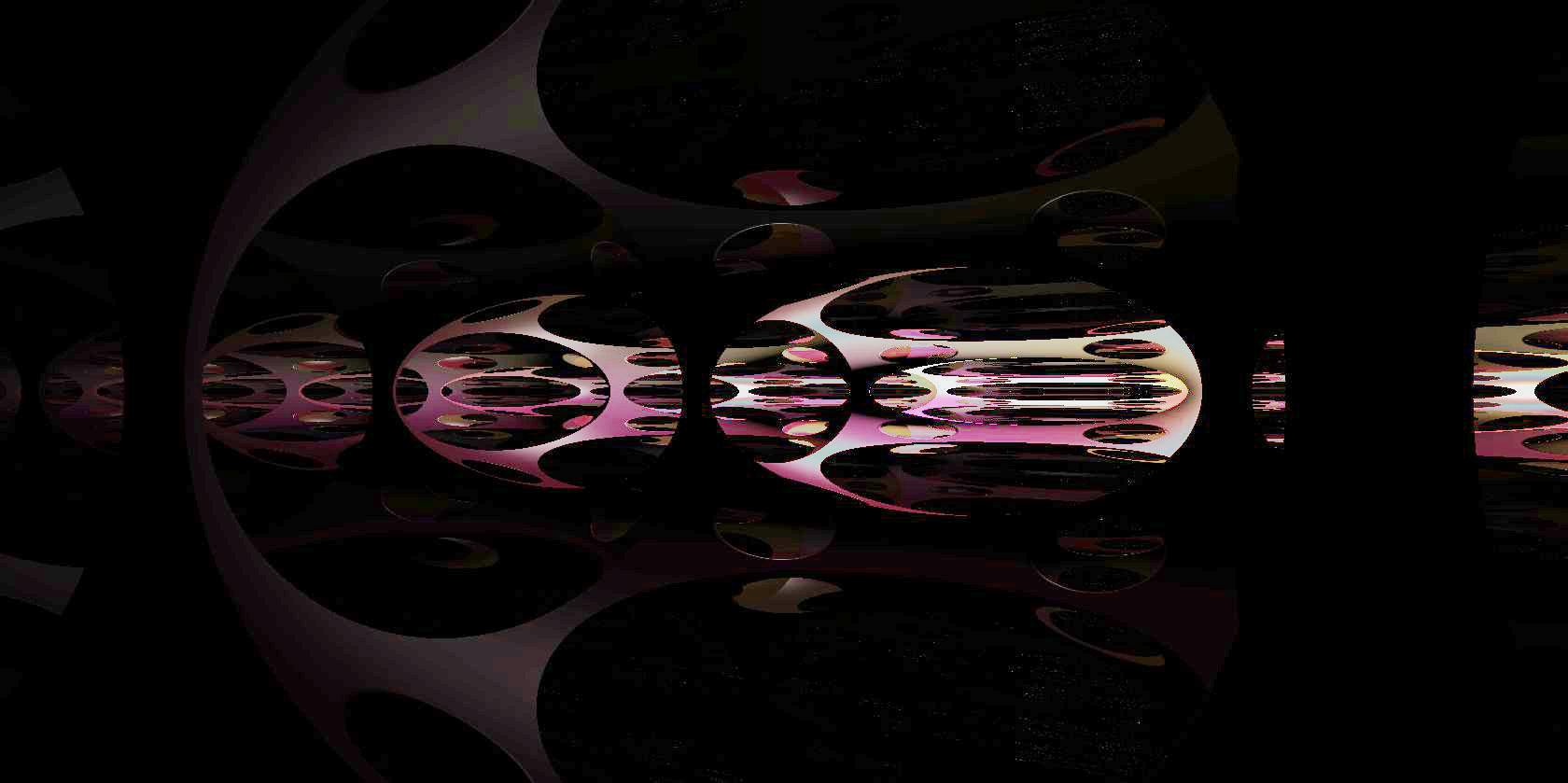}
\label{Fig:H2E_LightInSpace}
}\
\caption{In-space views highlighting consequences of the lighting intensities for the product geometries.}
\label{Fig:Product_LightInSpace}
\end{figure}

Finally, we must also compute the directions from a point $s\in X$ to the light source at $q\in X$.
To simplify the notation here, we will write each lighting pair of $\calL_s(q)$ 
not as a pair $(L,d_L)$, but as a vector $d_L L$ of length $d_L$ in the direction $L$.
Let $s,q\in X$ and let $d_Y=\dist_Y(s_Y,q_Y)$, $d_\EE=|q_\EE-s_\EE|$ be the distances between their projections into the respective factors of $X=Y\times\EE$. Recall that the standard basis vector $\be_w$ points along the $\EE$ direction.
We compute the unit vector $v_Y\in T_{s_Y} Y$ pointing along the shortest geodesic from $s_Y$ to $q_Y$ as in \refsec{IsotropicLighting}.
The element of $\calL_s(q)$ corresponding to the shortest geodesic is then $d_Y v_Y+d_\EE \be_w$.
In $\HH^2\times\EE$ geodesics are unique, so with this we are done:
$$\calL^{\HH^2\times\EE}_s(q)=\left\{d_Y v_Y+d_\EE \be_w\right\}=\left\{d_Y\frac{s_Y-\cosh(d_Y)q_Y}{\sinh(d_Y)}+d_\EE \be_w\right\}$$

In $S^2\times\EE$, there are three cases to deal with: first the generic case, second when $s,q$ lie on the same horizontal $S^2$, and third when $s_Y, q_Y$ are antipodal. As for $S^3$, in the implementation we don't worry about the non-generic cases; the lighting intensity at such points is the limit of the lighting intensity for the generic case.

In the generic case, there are countably many geodesics between $s$ and $q$. All of these geodesics lie on the cylinder formed by taking the product of the $\EE$ direction with the great circle containing $s_Y$ and $q_Y$.
For each natural number $n\geq 0$, there are two geodesics -- one starting by traveling the `short way' around the $S^2$ factor, followed by $n$ additional full turns, and the other the `long way' followed by $n$ additional turns. All together, this gives the set of directions
$$\calL^{S^2\times\EE}_s(q)=\bigcup_{n\geq 0}\Big\{
(2\pi n+d_Y)v_Y+d_\EE \be_w,(2\pi(n+1)-d_Y)v_Y+d_\EE \be_w
\Big\}.$$

If $s_\EE=q_\EE$ and $s_{S^2},q_{S^2}$ are not antipodal, then we just set $d_\EE = 0$ above. As in \refrem{GoingInCircles}, all but the shortest two are irrelevant if either the light source or the scene is opaque.

In the third case, where $s_Y,q_Y$ are antipodal in $S^2$, there are uncountably many geodesics joining $s$ to $q$. Their directions are a countable sequence of rings in the unit sphere in $T_s X$ accumulating on the horizontal equatorial circle. 

\begin{figure}[htbp]
\centering
\subfloat[
Looking along the $\EE$ direction.]{
\includegraphics[width=0.90\textwidth]{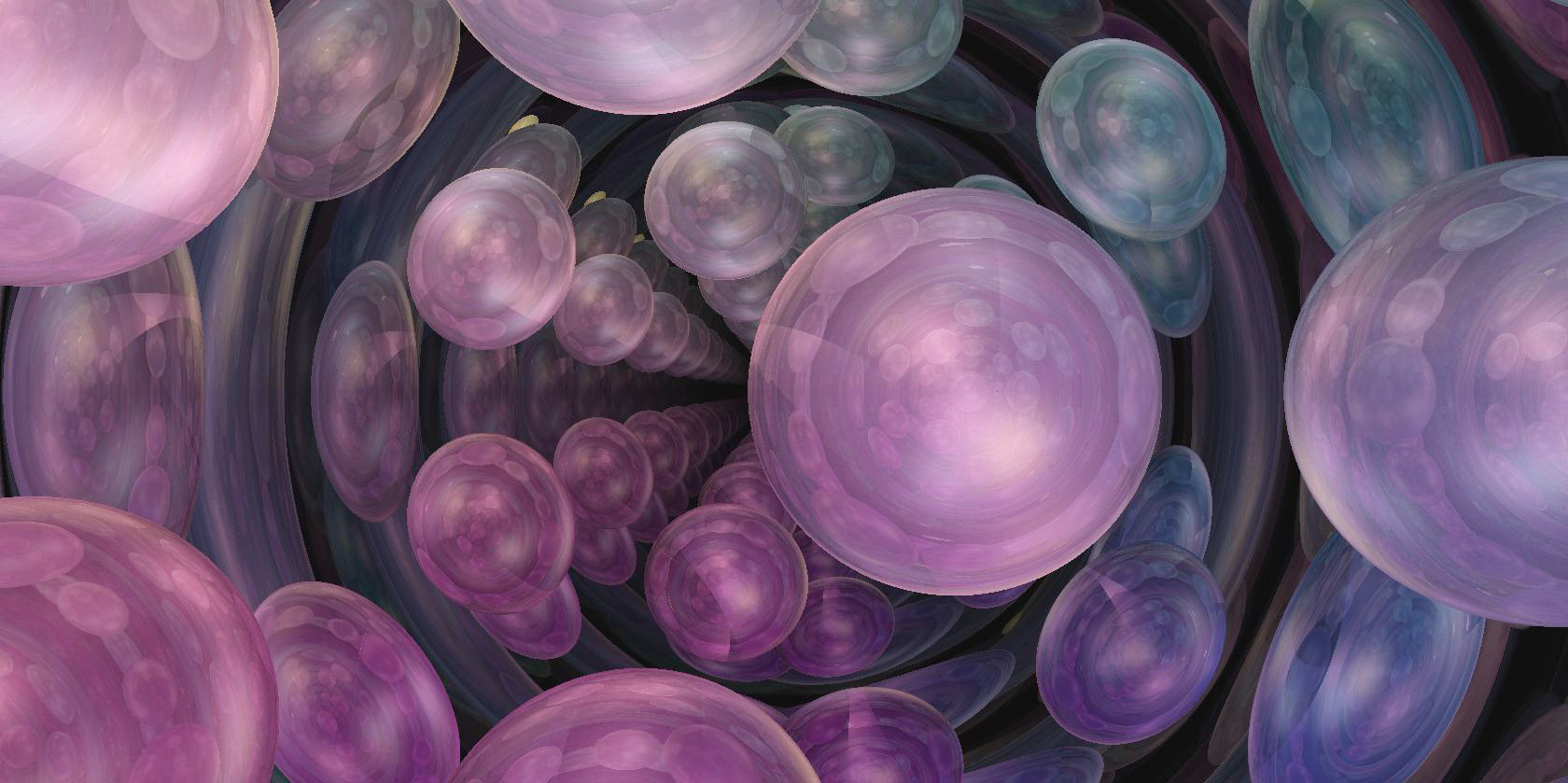}
\label{Fig:HopfManifold}
}\
\subfloat[Looking along an $S^2$ direction.]{
\includegraphics[width=0.90\textwidth]{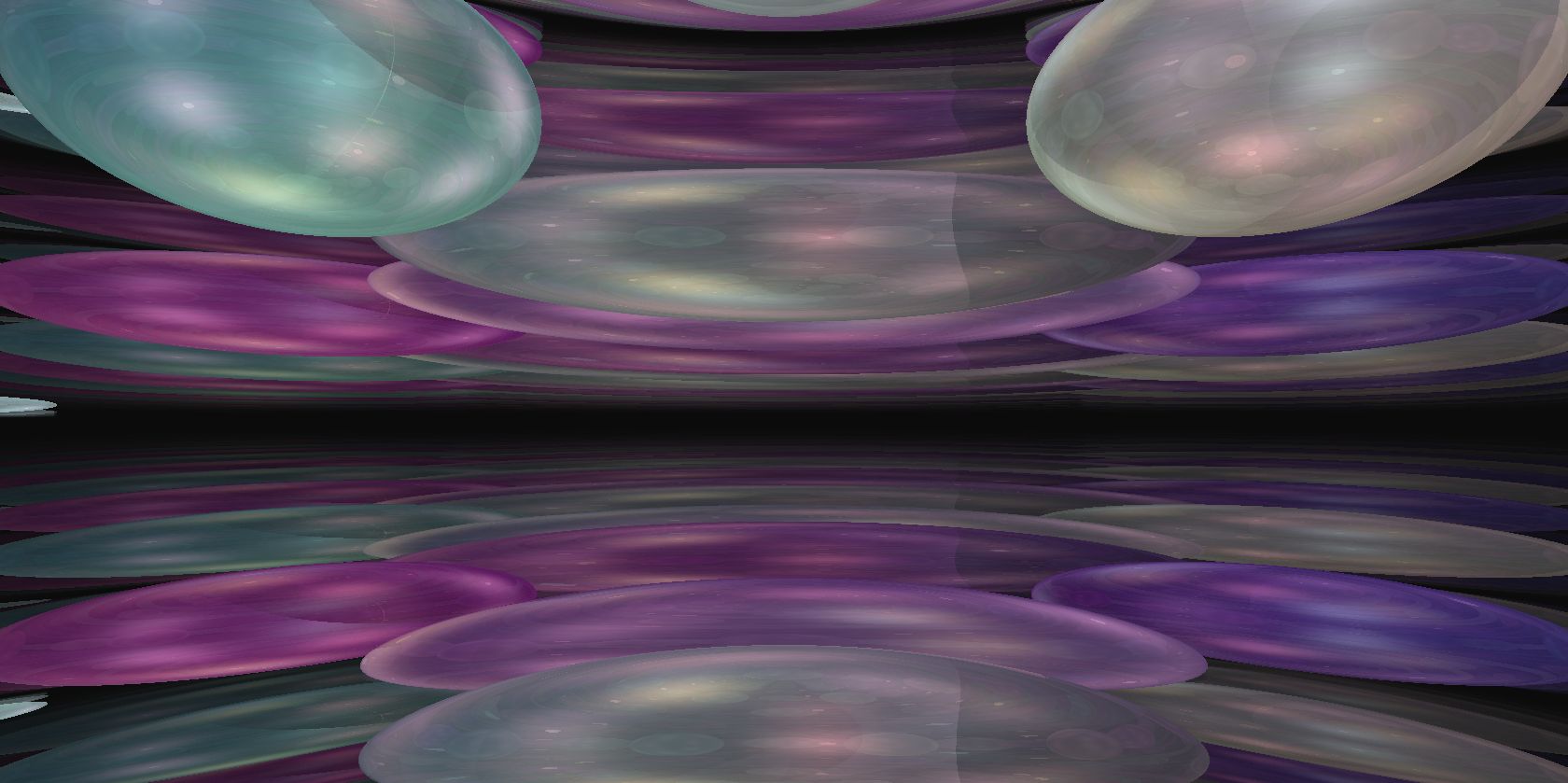}
\label{Fig:HopfManifold2}
}\
\subfloat[Looking along the $\EE$ direction.]{
\includegraphics[width=0.90\textwidth]{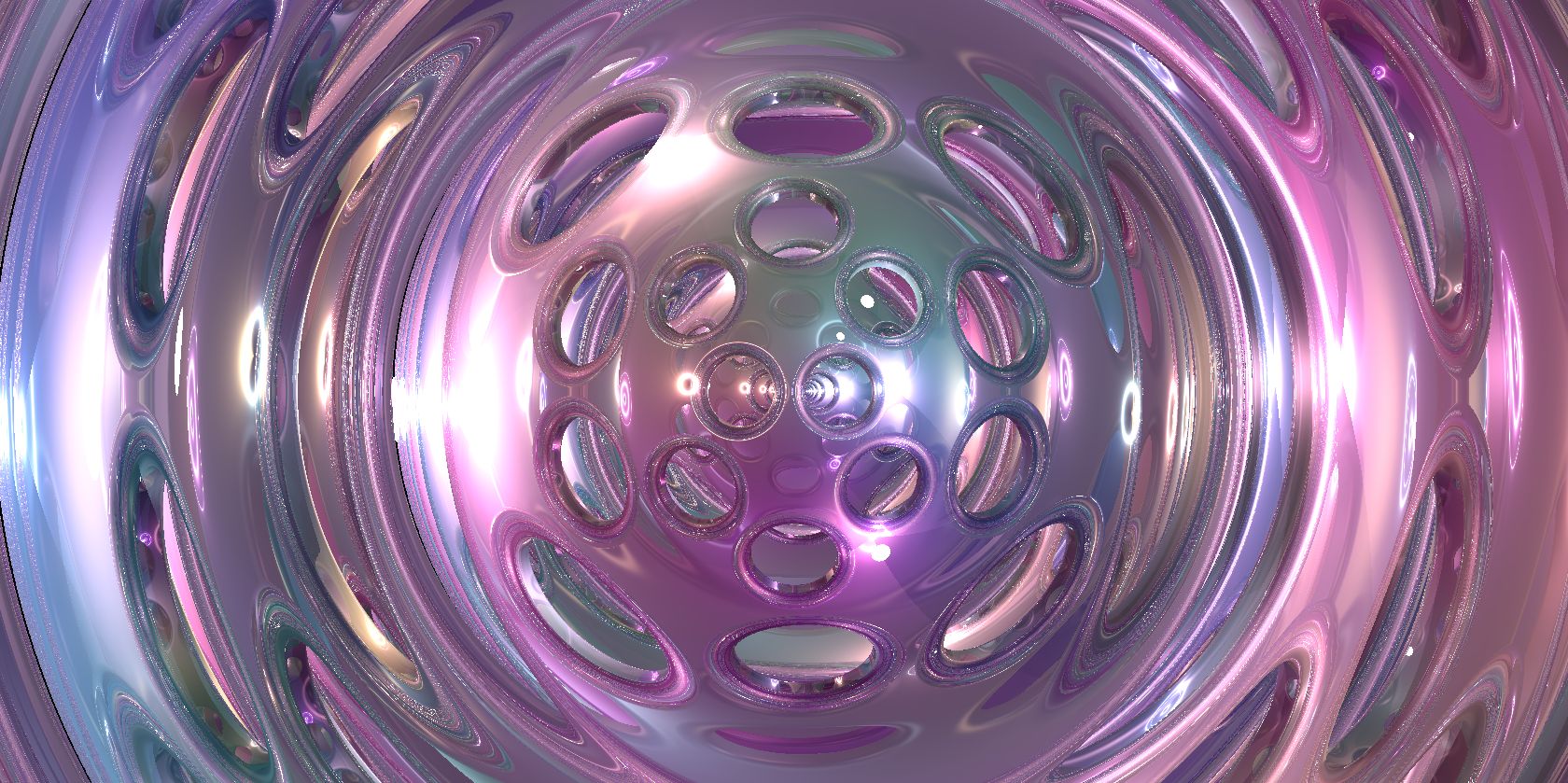}
\label{Fig:HopfDodecahedron}
}\
\caption{$S^2\times\EE$ Geometry. The Hopf manifold $S^2\times S^1$.}
\label{Fig:S2xEExamples}
\end{figure}

There are only seven manifolds with $S^2\times\EE$ geometry. These are listed in~\cite[page 457]{Scott}. In \reffig{S2xEExamples}, we show the in-space view for various scenes in the Hopf manifold $S^2 \times S^1$.
Figures \ref{Fig:HopfManifold} and \ref{Fig:HopfManifold2} show a collection of spheres spaced at the vertices of a regular dodecahedron.  \reffig{HopfDodecahedron} shows a slab $S^2 \times [-\epsilon, \epsilon]$, with holes cut out at the vertices of the dodecahedron.

\begin{figure}[htbp]
\centering
\subfloat[Looking along the $\EE$ direction.]{
\includegraphics[width=0.90\textwidth]{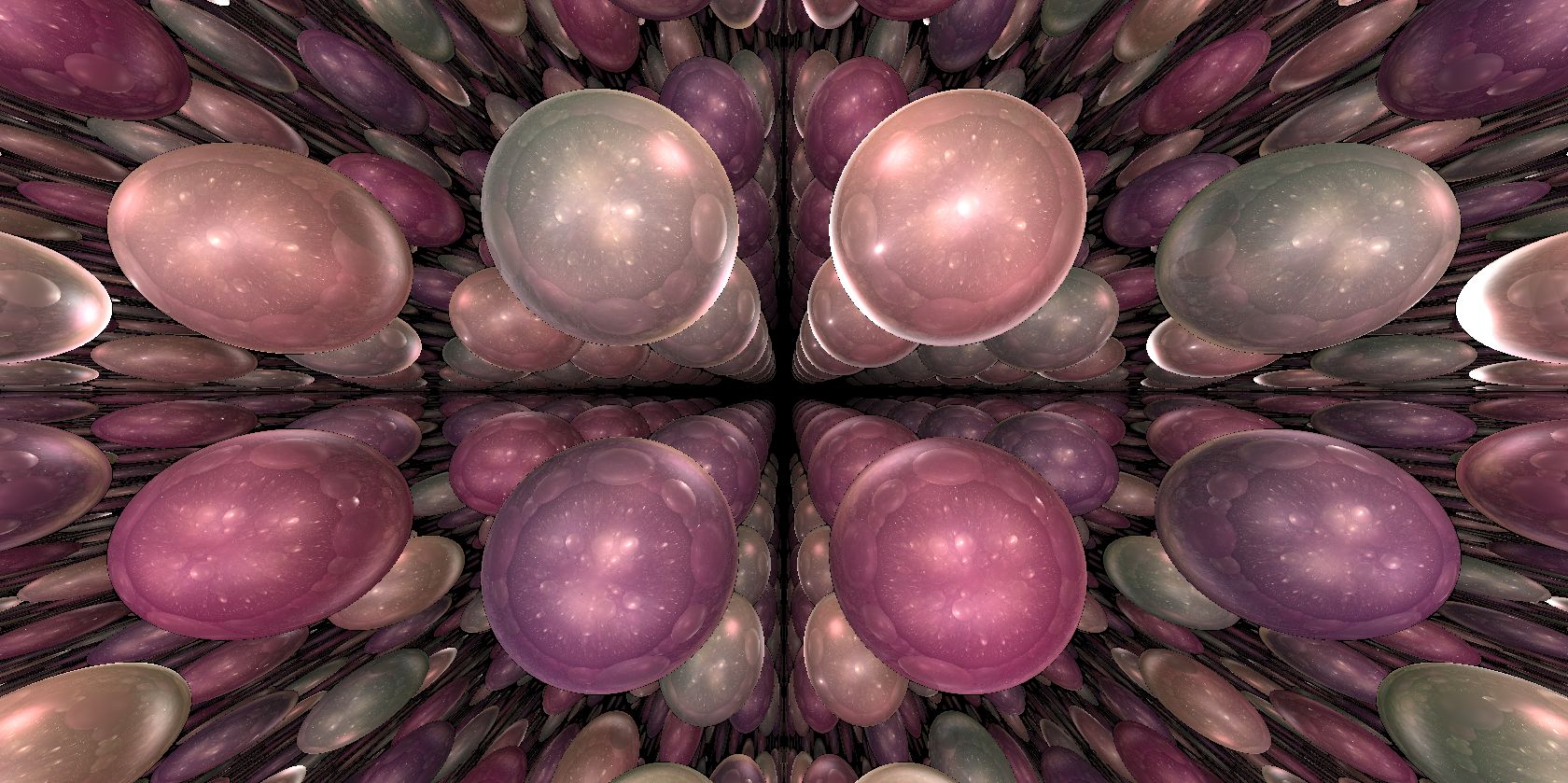}
\label{Fig:H2xEOrbifoldSpheres}
}\
\subfloat[Looking along an $\HH^2$ direction.]{
\includegraphics[width=0.90\textwidth]{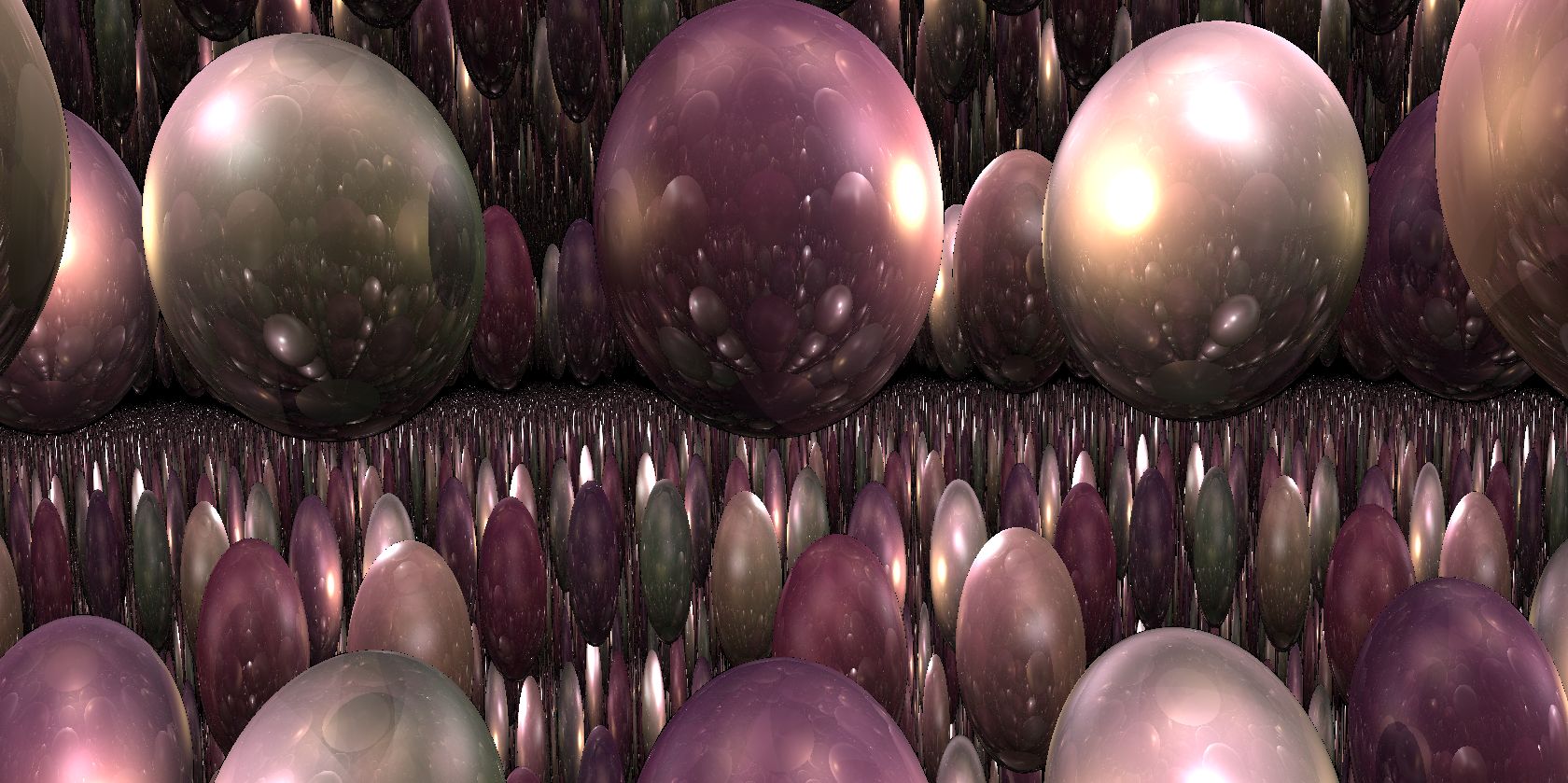}
\label{Fig:H2xEOrbifoldSpheres2}
}\
\subfloat[Looking along the $\EE$ direction.]{
\includegraphics[width=0.90\textwidth]{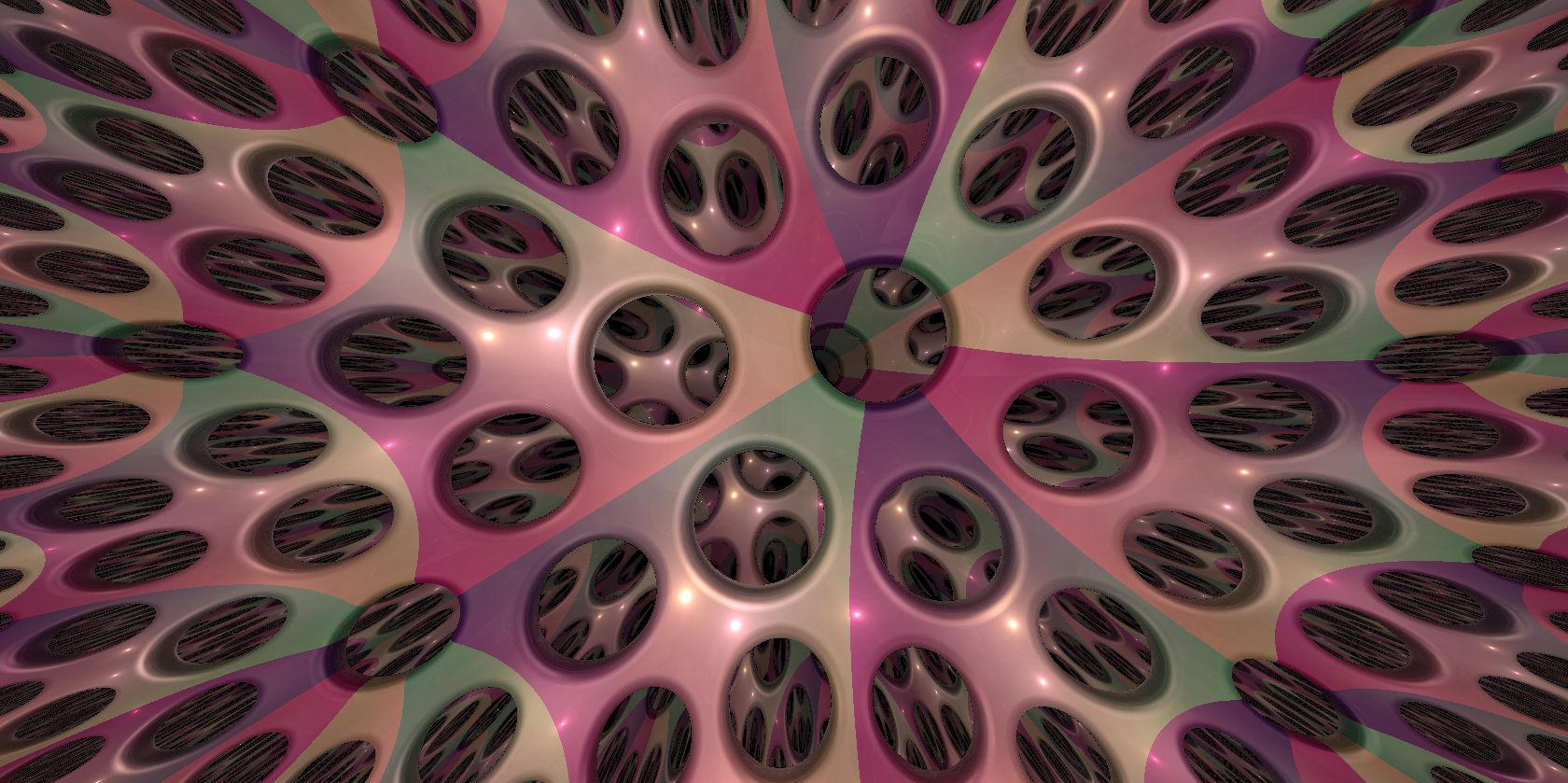}
\label{Fig:H2xEHoles}
}\
\caption{$\HH^2\times\EE$ Geometry. The product of a torus with cone point of angle $\pi$, with a circle.}
\label{Fig:H2xEExamples}
\end{figure}

The manifolds with $\HH^2\times\EE$ geometry are classified in~\cite[Theorem~4.13]{Scott}. In \reffig{H2xEExamples}, we show the in-space view for various scenes in $\HH^2\times\EE$ geometry. All of these images show the orbifold which is the product of a circle with a torus $T$ containing a cone point of angle $\pi$. Figures \ref{Fig:H2xEOrbifoldSpheres} and \ref{Fig:H2xEOrbifoldSpheres2} show a collection of spheres, four in each fundamental domain. 
\reffig{H2xEHoles} shows a slab $T \times [-\epsilon, \epsilon]$, with four holes cut from the fundamental domain of $T$, and a further hole cut around the cone point.

\section{Nil}
\label{Sec:Nil}

\subsection{Heisenberg model of Nil}
\label{Sec:Nil Heis model}
There are several models for Nil.
Probably the most commonly used is the Heisenberg model (also known as the polarized model of the first Heisenberg group).
The Heisenberg group $H$ is the group of $3\times 3$ upper triangular matrices of the form 
\begin{equation*}
	\left[\begin{array}{ccc}
		1 & x & z \\
		0 & 1 & y \\
		0 & 0 & 1
	\end{array}\right].
\end{equation*}
We identify this with $\RR^3$ through the $x$-, $y$-, and $z$-coordinates.
The metric 
\begin{equation*}
	ds^2 = dx^2 + dy^2 + (dz - xdy)^2
\end{equation*}
is invariant under the left action of $H$ on itself.

The space $(H,ds^2)$ has a major drawback for our purposes. 
To see this, let $o$ be point $[0,0,0]$ (corresponding to the identity matrix) which we see as the origin of the space.
The group of isometries of $(H,ds^2)$ fixing $o$ is isomorphic to ${\rm O}(2)$.
In particular, it contains a one-parameter subgroup of rotations.
These rotations are difficult to visualize in the Heisenberg model of Nil. See Figure~\ref{fig: nil balls - blender}.

\begin{figure}[h!tbp]
\begin{center}
	\includegraphics[width=\textwidth]{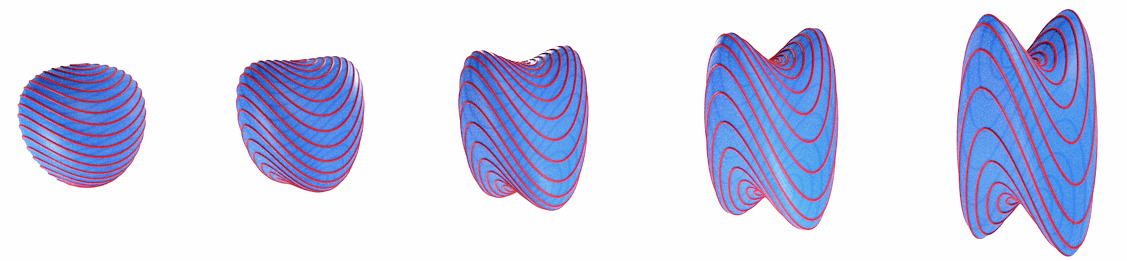}
\caption{Balls of radius one to five in the Heisenberg model of Nil. The images have been rescaled to take up approximately the same space on the page. The red curves are invariant under the rotations fixing the origin.}
\label{fig: nil balls - blender}
\end{center}
\end{figure}

\subsection{Rotation invariant model of Nil}
\label{Sec:Nil rot model}
For our computations we use a ``rotation invariant'' model of Nil.
The underlying space of the model is the affine subspace $X$ of $\RR^4$ defined by $w= 1$.
The group law is as follows: the point $[x,y,z,1]$ acts on $X$ on the left as the matrix
\begin{equation*}
	\left[\begin{array}{cccc}
		1 & 0 & 0 & x\\
		0 & 1 & 0 & y\\
		-y/2 & x/2 & 1 & z\\
		0 & 0 & 0 & 1
	\end{array}\right].
\end{equation*}
The origin $o$ is the point $[0,0,0,1]$.
Its tangent space $T_oX$ is identified with the linear subspace of $\RR^4$ given by the equation $w = 0$.
Our reference frame is $e = (\be_x,\be_y,\be_z)$ where $(\be_x, \be_y, \be_z, \be_w)$ is the standard basis of $\RR^4$.
The metric tensor at the point $p = [x,y,z,1]$ is now given by 
\begin{equation*}
	ds^2 = dx^2 + dy^2 + \left(dz - \frac12\left(xdy - ydx\right) \right)^2.
\end{equation*}
The map 
\begin{equation*}
	\begin{array}{ccc}
		H & \to & X \\
	\left[x,y,z\right] & \mapsto & \left[x,y,z - \frac 12 xy, 1 \right]
	\end{array}
\end{equation*}
is an isometry between the Heisenberg model and the rotation invariant model of Nil.
For every $\alpha \in \RR$, we write $R_\alpha$ 
for the transformation with matrix
\begin{equation*}
	\left[\begin{array}{cccc}
		\cos \alpha & -\sin \alpha & 0 & 0\\
		\sin \alpha & \cos \alpha & 0 & 0\\
		0 & 0 & 1 & 0\\
		0 & 0 & 0 & 1
	\end{array}\right].
\end{equation*}
One can check that $R_\alpha$ is an isometry of $X$, rotating by angle $\alpha$ around the $z$-axis.
Let $F$ be the transformation with matrix
\begin{equation*}
	\left[\begin{array}{cccc}
		0& 1 & 0 & 0\\
		1 & 0 & 0 & 0\\
		0 & 0 & -1 & 0\\
		0 & 0 & 0 & 1
	\end{array}\right].
\end{equation*}
This is another isometry of $X$, in this case it \emph{flips} the $z$-axis, and satisfies $F \circ R_\alpha \circ F^{-1} = R_{-\alpha}$.
These two kinds of isometries generate the stabilizer $K = {\rm O(2)}$ of $o$ in $G = {\rm Isom}(X)$.

\subsection{Geodesic flow and parallel transport}
\label{Sec:nil geo flow}
The solution of the geodesic flow in the Heisenberg model of Nil has been computed, for example in \cite{Molnar:2003aa}.
We could convert the solution into our rotation invariant model $X$. Instead, we take this opportunity to illustrate Grayson's method (Sections~\ref{Sec:GeodesicFlow - Grayson} and \ref{Sec:Parallel transport - Grayson}) and calculate the geodesic flow and parallel-transport operator directly in $X$, as follows.

Let $\gamma \colon \RR \to X$ be a geodesic in Nil and $T(t) \colon T_{\gamma(0)}X \to T_{\gamma(t)}X$ be the corresponding parallel-transport operator.
We define two paths $u \colon \RR \to T_oX$ and $Q \colon \RR \to {\rm SO}(3)$ by the following relations
\begin{equation*}
	\begin{split}
		\dot \gamma(t) & = d_oL_{\gamma(t)} u(t), \\
		T(t) \circ d_oL_{\gamma(0)} & = d_oL_{\gamma(t)} Q(t).
	\end{split}
\end{equation*}
Recall that the identification of parallel transport with the path $Q$ in ${\rm SO}(3)$ is done via our reference frame $e = (\be_x,\be_y,\be_z)$ at the origin $o$.
After some computation, Equations (\ref{Eqn:GraysonMethodFlowSphere}) and (\ref{Eqn:ParallelTransportLinEq}) respectively become
\begin{equation*}
	\left\{ \begin{split}
		\dot u_x & = -u_zu_y \\
		\dot u_y & = u_zu_x \\
		\dot u_z & = 0
	\end{split}\right.
\end{equation*}
and 
\begin{equation*}
	\dot Q + BQ = 0
	\quad \text{where} \quad
	B = \frac 12 \left[\begin{array}{ccc}
		0 & u_z & u_y \\
		-u_z & 0 & -u_x \\
		-u_y & u_x & 0
	\end{array}	\right].
\end{equation*}

For the initial condition $u(0) = [a \cos\alpha, a \sin\alpha, c, 0]$, where $a \in \RR_+$ and $c \in \RR$ satisfy $a^2 + c^2 = 1$, one gets 
\begin{equation*}
	u(t) = \left[ a \cos(ct + \alpha) , a \sin (ct + \alpha), c, 0 \right].
\end{equation*}
In order to get the expression for $Q$, we follow the strategy detailed in \refsec{Parallel transport - Grayson} and obtain
\begin{equation*}
	Q(t) = dR_\alpha e^{ct U_1} P e^{-\frac 12tU_2} P^{-1}dR_\alpha^{-1}, \quad \forall t \in \RR,
\end{equation*}
where
\begin{equation*}
	U_1 = 
	\left[\begin{array}{cccc}
	    0 & -1 & 0  \\
	    1 & 0 & 0  \\
	    0 & 0 & 0 
	\end{array}\right],	
	\quad
	U_2 =
	\left[\begin{array}{ccc}
	    0 & 0  & 0 \\
	    0 & 0 & -1 \\
	    0 & 1 & 0 
	\end{array}\right],
\end{equation*}
and 

\begin{equation*}
	dR_\alpha = \left[\begin{array}{ccc}
	    \cos\alpha & -\sin\alpha & 0 \\
	    \sin\alpha & \cos\alpha & 0 \\
	    0 & 0 & 1  \\
	\end{array}\right],
	\quad
	P = 
	\left[\begin{array}{ccc}
	    a & 0 & -c \\
	    0 & 1 & 0 \\
	    c & 0 & a  \\
	\end{array}\right].
\end{equation*}
Note that $dR_\alpha \colon T_oX \to T_oX$ is the differential of the rotation $R_\alpha$ written in the reference frame $e = (\be_x, \be_y, \be_z)$.

Let us now move back to the original geodesic $\gamma \colon \RR \to X$, which we write as
\begin{equation*}
	\gamma(t) = \left[x(t), y(t), z(t), 1\right].
\end{equation*}
Without loss of generality we can assume that $\gamma(0) = o$.
\refeqn{GraysonMethodPullBack} becomes
\begin{equation*}
	\left\{
		\begin{split}
			\dot x & = u_x \\
			\dot y & = u_y \\
			\dot z & = u_z + \frac 12(x u_y - y u_x)
		\end{split}
	\right..
\end{equation*}
Plugging in  our solution for $u$, we finally get
\begin{equation}
\label{Eqn:nil geodesic flow 1}
	\left\{
		\begin{split}
			x(t) & = \frac {2a}c \sin \left(\frac{ct}2\right)\cos\left(\frac{ct}2 + \alpha\right) \\
			y(t) & = \frac {2a}c \sin \left(\frac{ct}2\right)\sin\left(\frac{ct}2 + \alpha\right)\\
			z(t) & = ct + \frac 12 \frac {a^2}{c^2} \big(ct - \sin (ct) \big)
		\end{split}
	\right.\quad 
	\text{whenever}\ c \neq 0,
\end{equation}
and otherwise
\begin{equation}
\label{Eqn:nil geodesic flow 2}
	\left\{
		\begin{split}
			x(t) & = a \cos(\alpha)t \\
			y(t) & = a \sin(\alpha)t \\
			z(t) & = 0.
		\end{split}
	\right.
\end{equation}

\begin{remark}
	If $c$ is very small but not zero, the above formulas are the source of significant numerical errors.
	This is due to the term 
	\begin{equation*}
		\frac {ct - \sin (ct) }{c^2},
	\end{equation*}
	see \refsec{FloatingPoint}\refitm{RemovableSingularities}.
	In practice, this causes noise around the $xy$-plane, see \reffig{nil noise}.
	To fix this issue, when $ct$ is small, we replace the formula given in \refeqn{nil geodesic flow 1} by its asymptotic expansion of order seven around zero.
\end{remark}

\begin{figure}[htbp]
\subfloat[][Using \refeqn{nil geodesic flow 1}. One observes noise around the $xy$-plane.]{
	\includegraphics[width=0.7\textwidth]{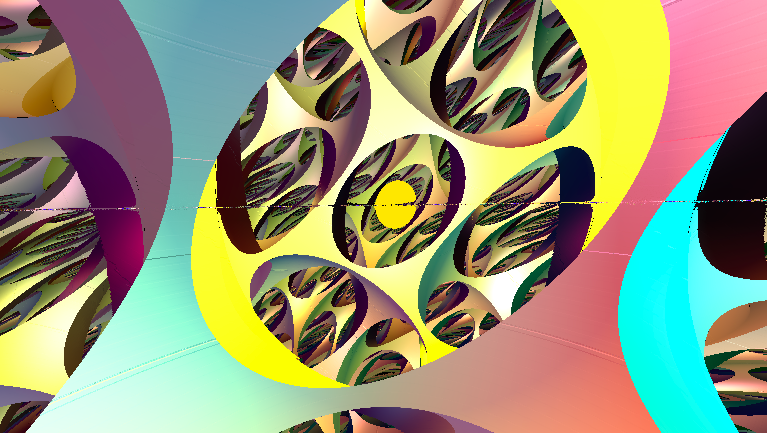}
}

\subfloat[][Replacing \refeqn{nil geodesic flow 1}  by an asymptotic expansion in a neighborhood of the $xy$-plane.]{
	\includegraphics[width=0.7\textwidth]{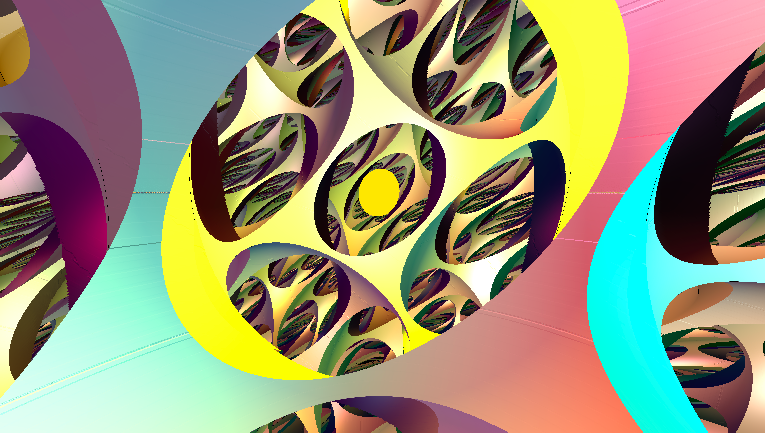}
}
\caption{Fixing the instability of the formula around $c = 0$. Both pictures represent the lattice of Nil given by the integer Heisenberg group. 
The yellow ball in the center represents a light.
Here we choose a simple color scheme to highlight the noise.}
\label{Fig:nil noise}
\end{figure}

\subsection{Distance to a vertical object}
Observe that Nil comes with a natural $1$-Lipschitz projection $\pi \colon X \to \EE^2$, sending $[x,y,z,1]$ to $[x,y]$.
In analogy with objects in the product geometries, we call the pre-image under $\pi$ of any non-empty subset of $\EE^2$ a \emph{vertical} object.
For example, any affine plane with equation $\alpha x + \beta y = \gamma$ is a vertical object. 

\begin{lemma}
	Let $S$ be a subset of $\EE^2$ and $Z = \pi^{-1}(S)$ the associated vertical object.
	The distance from any point $p \in X$ to $Z$ coincides with the distance between $\pi(p)$ and $S$ in $\EE^2$.
\end{lemma}

\begin{proof}
	Any isometry of $X$ preserves the fibers of the projection $\pi$ and induces an isometry of $\EE^2$.
	Hence, applying a translation, it suffices to prove the claim in the case that $p$ is the origin $o$.
	Similarly, applying a rotation, we can assume that the projection $\pi(o)$ on $S$ is a point $q$ of the form $q = [x,0]$, with $x \geq 0$.
	Since $\pi$ is $1$-Lipschitz, we have 
	\begin{equation*}
		\dist\left( \pi(o), S \right) \leq \dist\left(o, Z\right).
	\end{equation*}
	Let us explain the reverse inequality.
	We have seen previously that the map $\gamma \colon \RR \to X$, mapping $t$ to $[t,0,0,1]$, is a geodesic of Nil.
	Hence the distance in Nil between $o$ and the pre-image $\cover q = [x,0,0,1]$ of $q$ is at most $x$.
	Consequently 
	\begin{equation*}
		\dist\left( o, Z \right) \leq \dist\left(o, \cover q\right) \leq x \leq \dist\left(\pi(o), S\right).\qedhere
	\end{equation*}
\end{proof}

\reffig{Nil_VertPlane} shows a vertical half-space in Nil. We texture the boundary with squares of side length one in its euclidean metric. Note that here (and in Sections \ref{Sec:SLR} and \ref{Sec:Sol}) we extend the notion of a half-space from that given at the start of \refsec{SDF}: the boundary may not be totally geodesic.
\reffig{NilFibers} shows vertical solid cylinders.

\begin{figure}[htbp]
\centering
\subfloat[$p=(0.5,0,0,1)$ looking along DIR][]{
	\includegraphics[width=0.3\textwidth]{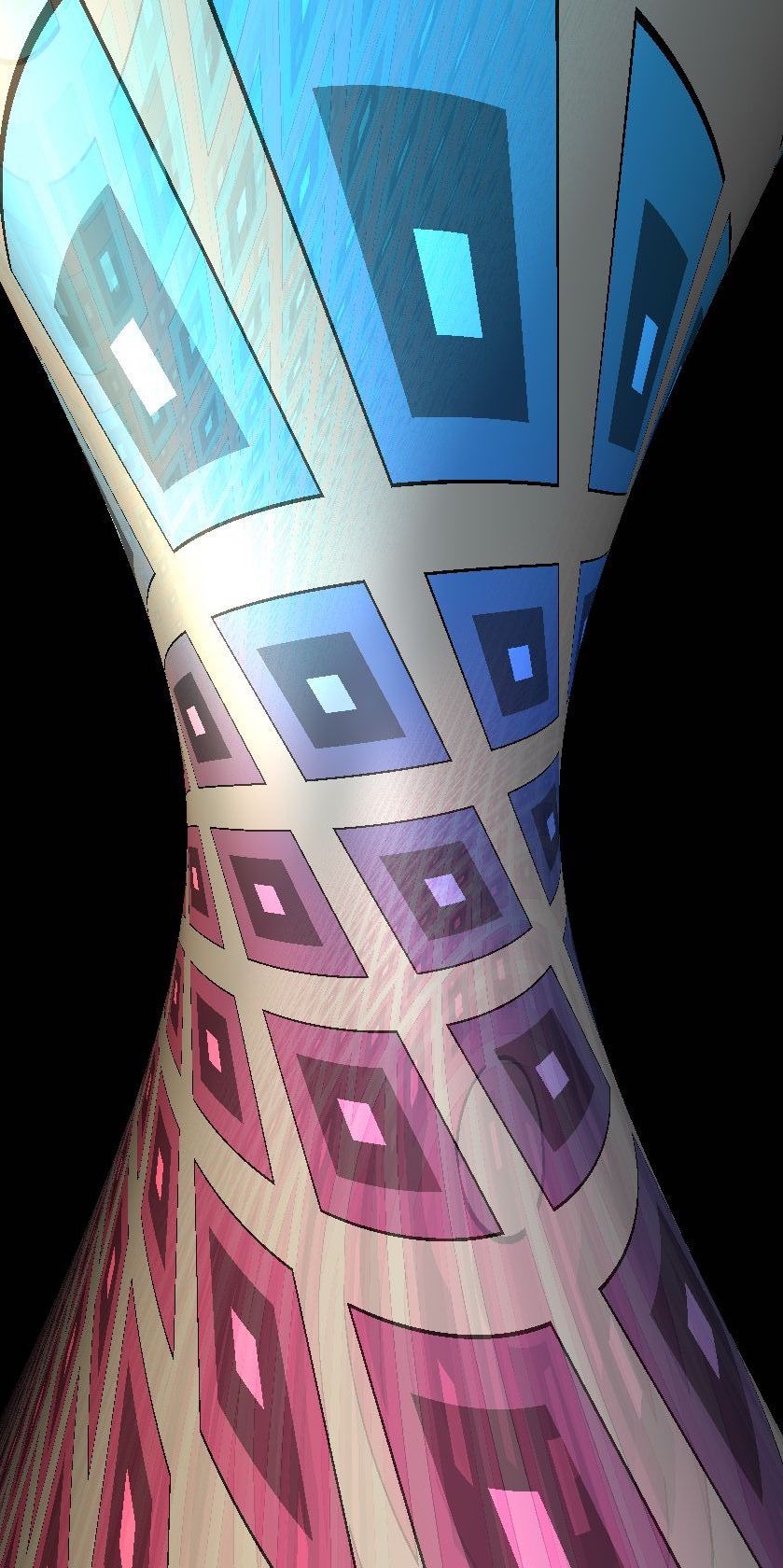}
}
\subfloat[$p=(1.5,0.5,0,1)$ looking along DIR][]{
	\includegraphics[width=0.3\textwidth]{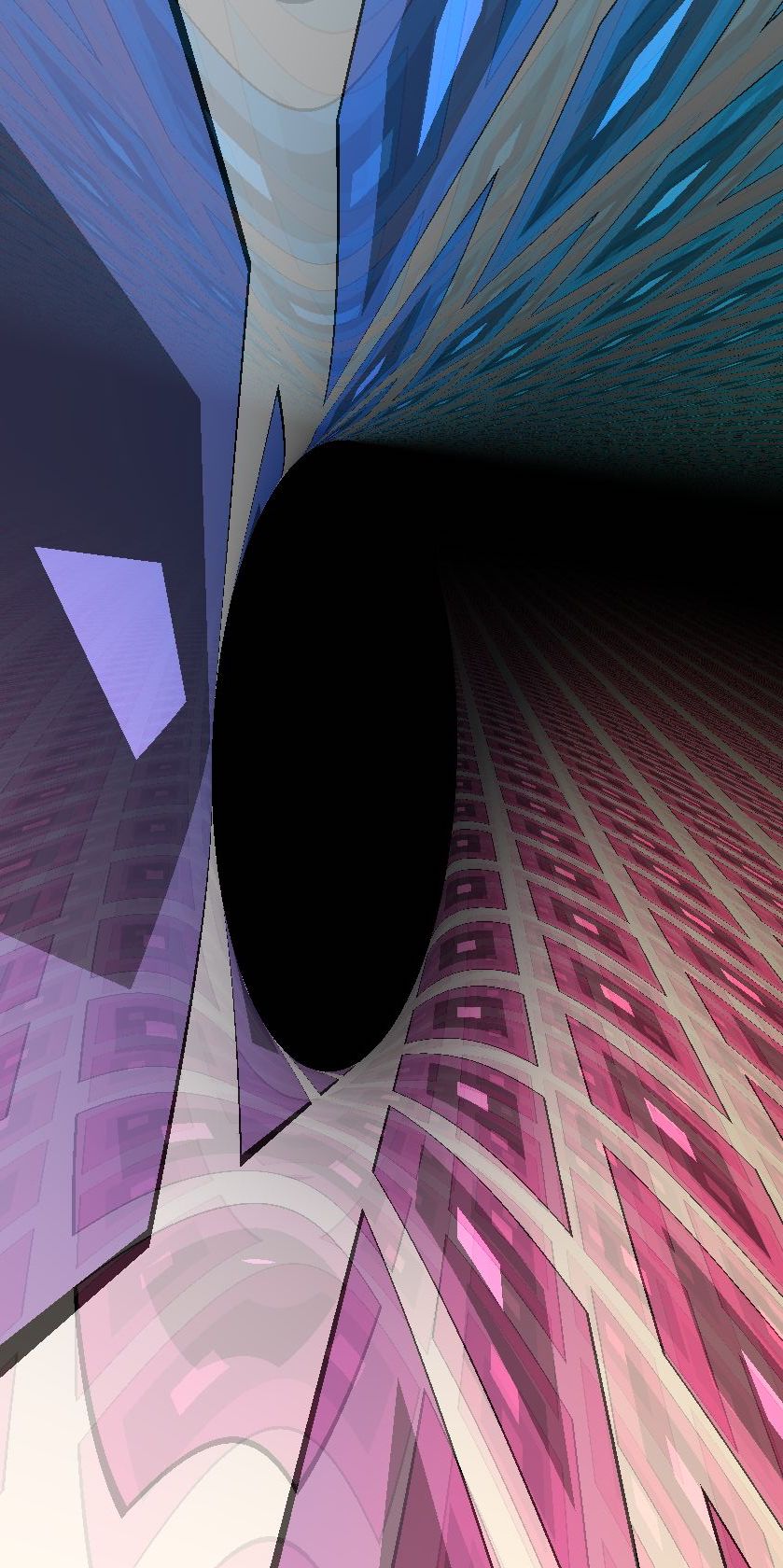}
}
\subfloat[$p=$ looking along DIR][]{
	\includegraphics[width=0.3\textwidth]{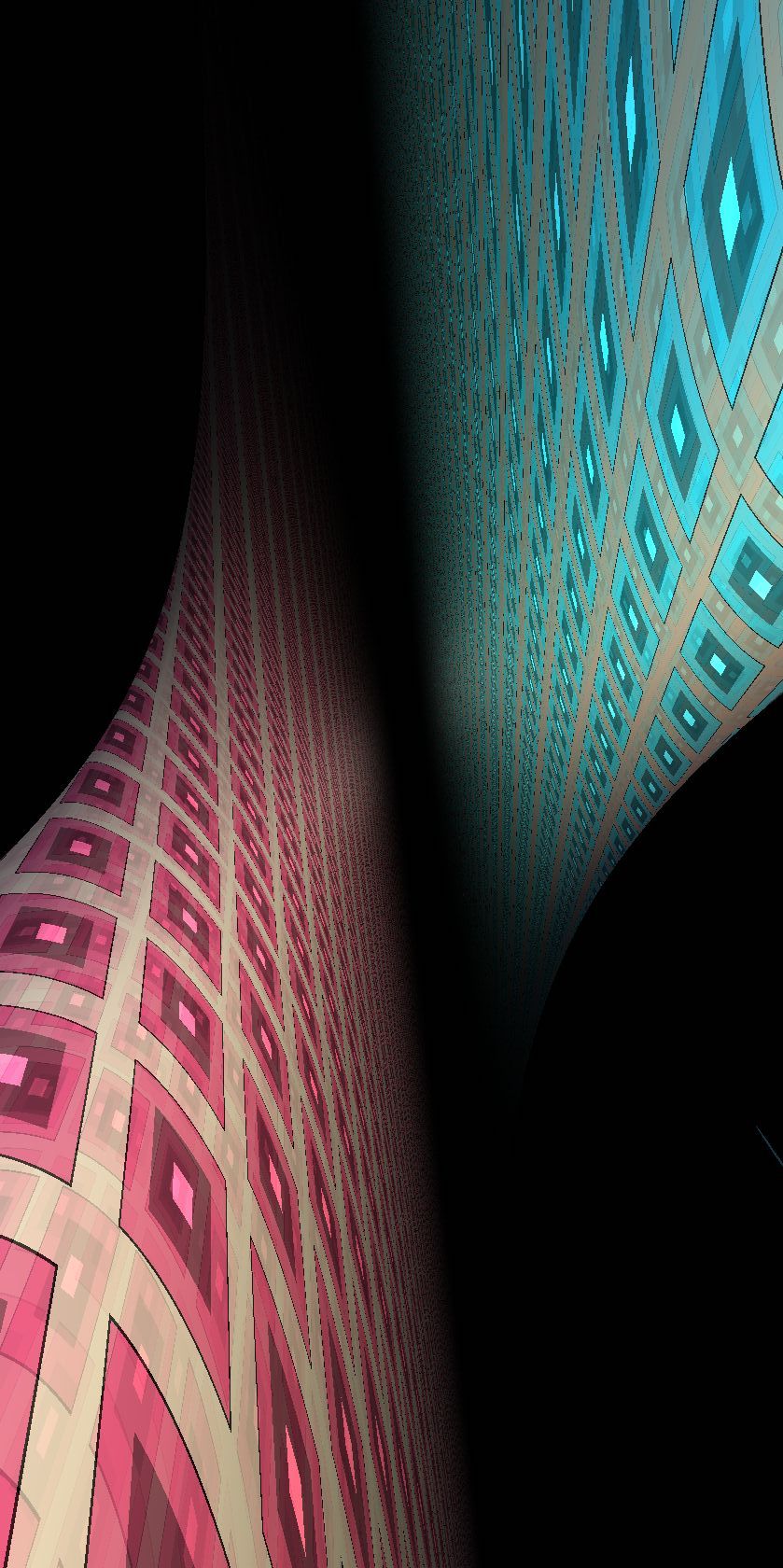}
}
\caption{Three views of a vertical half-space in Nil geometry. We see multiple reflections of the plane in itself, due to the spiraling of geodesics. Rendered with artificial (constant) light intensity, and fog.}
\label{Fig:Nil_VertPlane}
\end{figure}

\subsection{Exact distance and direction to a point}
\label{Sec:ExactGeodesicsNil}

Since $X$ is homogeneous, we only need to compute the distance between any point $p \in X$ and the origin.
In order to calculate lighting pairs, we need to compute the direction at the origin $v\in T_o X$ of the geodesics $\gamma$ from $o$ to $p$. Unfortunately, in Nil there is no closed-form expression for either of these two quantities. 
We compute both using the same numerical approach.
Using the flip symmetry, we may assume that the coordinates $[x,y,z,1]$ of $p$ satisfy $z \geq 0$.

Assume first that the point $p = [x,y,z,1]$ lies neither on the $xy$-plane nor on the $z$-axis.
Let $\gamma$ be a geodesic from $o$ to $p$. That is, $\gamma(0) = o$ and $\gamma(t) = p$, for some $t \geq 0$.
As in \refsec{nil geo flow}, we write $v= [a \cos\alpha,a \sin\alpha,c,0]$ for its (unit) tangent vector at $o$.
We deduce from \refeqn{nil geodesic flow 1} that 
\begin{equation*}
	z = \phi + \frac{\rho^2}{8\sin^2(\phi/2)}\left(\phi - \sin \phi \right),
\end{equation*}
where $\rho^2 = x^2 + y^2$ and $\phi = ct$.
These quantities have the following useful geometric interpretation:
\begin{itemize}
	\item $\rho$ is the distance in $\EE^2$ between $\pi(o)$ and $\pi(p)$, and
	\item $\phi$ is the angle described by the projection of $\gamma$ in $\EE^2$.
\end{itemize}
Computing the directions from $o$ to $p$ consists of solving a system with five unknowns ($a$, $c$, $\alpha$, $t$, and $\phi$) and five equations (the three given by \refeqn{nil geodesic flow 1} along with the relations $a^2 + c^2 = 1$ and $\phi = c t$).
Once $\phi$ has been found, it is an exercise to uniquely recover $a,c, \alpha$ and $t$ by directly solving the equations.
Hence there is a one-to-one correspondence between the geodesics joining $o$ to $p$ and the zeros of the function
\begin{equation*}
	\chi_{\rho,z} (\phi) = -z + \phi + \frac{\rho^2}{8\sin^2(\phi/2)}\left(\phi - \sin \phi \right),
\end{equation*}
see \reffig{graph chi nil}.

\begin{figure}[htbp]
\centering
\includegraphics[width=\textwidth]{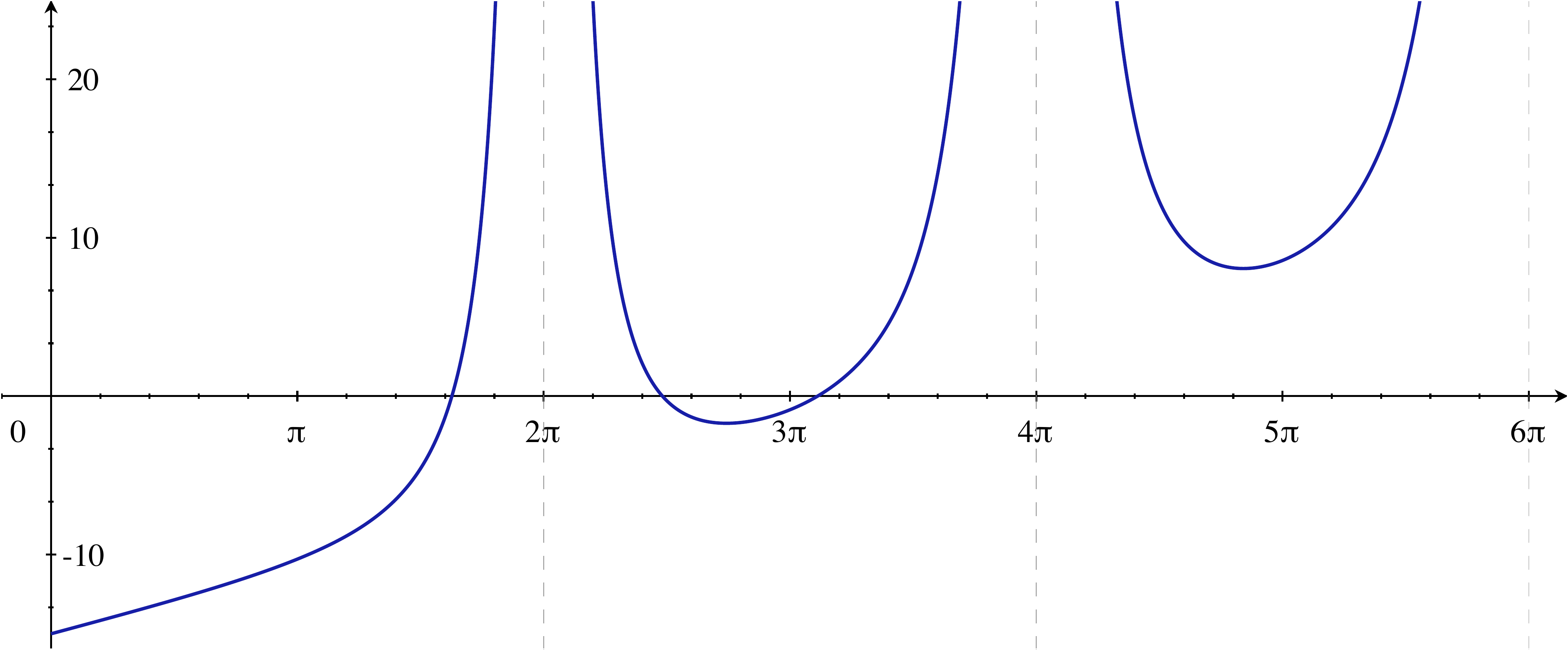}
\caption{The graph of the function $\chi_{\rho, z}$ for $\rho = 2$ and $z = 15$.
The function is not defined at $\phi=2k\pi$ for $k\in\ZZ_{>0}$.
In this case, there are exactly three geodesics joining the origin to any point $p$ with coordinates $[\rho \cos \theta, \rho \sin\theta, z,1]$.}
\label{Fig:graph chi nil}
\end{figure}

A geodesic $\gamma$ is minimizing if and only if the corresponding angle $\phi$ belongs to $(0, 2\pi)$.
It turns out that $\chi_{\rho, z}$ is strictly convex on the interval $(2k\pi, 2k\pi + 2\pi)$ for every integer $k\geq 0$.
Moreover, it is increasing on $(0,2\pi)$.
In order to find the minimizing geodesic from $o$ to $p$ we numerically compute the unique zero of $\chi_{\rho, z}$ on $(0, 2\pi)$ using Newton's method.

For physically accurate lighting, we also need the lighting pairs, $\calL_o(p)$, as defined in \refsec{Lighting directions}. Using binary search, we find a value of $\phi_0 \in (2\pi, 4\pi)$ where $\chi_{z, \rho}$ is positive and $d\chi_{z, \rho}/d\phi$ is negative.
We then run Newton's method, starting from $\phi_0$, producing a sequence $\{\phi_n\}$.
Recall that $\chi_{z, \rho}$ is strictly convex on $(2\pi, 4\pi)$.
Hence if the equation $\chi_{z, \rho}(\phi) = 0$ admits a solution in this interval, then $\{\phi_n\}$ will converge toward the first such solution.
Otherwise, either $\{\phi_n\}$ escapes the interval $(2\pi, 4\pi)$, or the sign of the derivative $d\chi_{z, \rho}/d\phi$ becomes positive. Either case is a halting condition for our algorithm. Repeating this procedure starting with a point for which $d\chi_{z, \rho}/d\phi$ is positive, we find the other solution in the interval.
Depending on the level of precision that we want for lighting, we can repeat the procedure on the next intervals $(4\pi, 6\pi)$, $(6\pi, 8\pi)$, \dots

Assume now that $p = [x,y,0,1]$ lies in the $xy$-plane.
Then there is a unique geodesic $\gamma$ joining $o$ to $p$. 
It coincides with the euclidean geodesic of $\RR^2$ between the same points. 
Hence its direction and length can be computed explicitly.
Alternatively, using continuity, we can extend the definition of the previous function $\chi_{\rho, z}$ at $\phi = 0$ by letting $\chi_{\rho, z} (0) = -z$.
In this way, this particular case is included in the previous discussion.
Indeed the only zero of $\chi_{\rho, 0}$ is $\phi = 0$.

If $p = [0,0,z,1]$ lies on the $z$-axis, then the path $\gamma(t) = [0,0,t,1]$ is a geodesic from $o$ to $p$ with initial direction $v = [0,0,1,0]$ and length $t = z$.
If $2n \pi \leq z < 2n\pi + 2\pi$, for some integer $n \geq 1$, then $o$ and $p$ are joined by $n$ other rotation-invariant families of geodesics $\{\gamma_{1,\alpha}\}, \dots, \{\gamma_{n,\alpha}\}$ where $\alpha$ runs over $[0, 2\pi)$.
The $k$th of these has length 
\begin{equation*}
	t_{k,\alpha} = 2k \pi \sqrt{\frac z{k\pi} - 1}
\end{equation*}
and direction at the origin
\begin{equation*}
	v_{k,\alpha} = \left[ \sqrt{\frac{z - 2k\pi }{z-k\pi}}\cos(\alpha), \sqrt{\frac{z - 2k\pi }{z-k\pi}}\sin(\alpha), \sqrt{\frac{k\pi }{z-k\pi}},0  \right].
\end{equation*}

\subsection{Distance underestimator for a ball} 
\label{Sec:distance underestimator Nil}
As mentioned in \refsec{DistanceUnderestimators}, we don't necessarily need the exact distance to an object to perform ray-marching. 
A distance underestimator also works.
A rough estimate using the solution of the geodesic flow shows the following. 

\begin{lemma}
Let $f \colon \RR_+ \to \RR_+$ be the continuous increasing map defined by
\begin{equation*}
	f(d) = \left\{\begin{split}
		d, &\ \text{if}\ d < \sqrt 6, \\
		\frac 43 \left( 1 + \frac 1{12} d^2\right)^{3/2}, &\ \text{if}\  \sqrt 6 \leq d< 2 \sqrt 6, \\
		\frac 1{2\sqrt 3} d^2,&\ \text{if}\ 2 \sqrt 6 \leq d.
	\end{split}\right.
\end{equation*}
If $p = [x,y,z,1]$ is a point at distance $d$ from the origin $o$, then  $\sqrt{x^2 + y^2} \leq d$ and $|z| \leq f(d)$. \qed
\end{lemma}

As a consequence, for every $\psi \in (0,1)$, for every $m \geq 1$, we have
\begin{equation*} 
  0 < \left[ (1-\psi) \left(x^2 + y^2\right)^{\frac m2} + \psi \left( f^{-1}(|z|) \right)^m\right]^{\frac 1 m} \leq \dist(o, p).
\end{equation*}
This allows us to build a distance underestimator $\sigma' \from X \to \RR$ to render a ball of radius $r$ centered at $o$, as follows. Let

\begin{equation*}
	\sigma'(p) = \left\{
	\begin{split}
		\sigma(p) - r, & \ \text{if}\ \sigma(p) > r + \eta \\
		\dist(o,p) - r, & \ \text{otherwise}
	\end{split}
	\right.
\end{equation*}
where 
\begin{equation*}
	\sigma(p) = \left[ (1-\psi) \left(x^2 + y^2\right)^{\frac m2} + \psi \left( f^{-1}(|z|) \right)^m\right]^{\frac 1 m} 
\end{equation*}
and $\eta > 0$ is a constant that is much larger than the threshold $\epsilon$ used to stop the ray-marching algorithm.
This is more efficient than the exact signed distance function: here, the rough (and inexpensive to calculate) estimate $\sigma$ is used to handle points at a large distance from the ball.
When the point $p$ is close  to the ball we replace this estimate by the exact distance computed numerically as explained in \refsec{ExactGeodesicsNil}.  
We use this distance underestimator to render the balls in \reffig{NilAllDirections} (a line of balls along the fiber direction), and \reffig{NilBalls} (a lattice of balls in Nil).

\subsection{Creeping to horizontal half-spaces}
\label{Sec:CreepingNilz=0Plane}
In the case of vertical objects, we can use the geometry of Nil to help us build signed distance functions. For ``horizontal'' objects, for example the $z \leq 0$ half-space, we do not have anything equivalent. Thus, it is difficult to come up with a signed distance function (or even a distance underestimator). However, we can still detect whether a point is in a half-space or not, and so we can use the same binary search algorithm as used to detect the boundary of a fundamental domain in \refsec{Creep}. \reffig{Nil_XYPlane} shows the $z \leq 0$ half-space in Nil geometry, with boundary textured by squares in the coordinate grid of side length one.

\begin{figure}[htbp]
\centering
\subfloat[ {$p=[0,0,8.5,1]$} ]{
	\includegraphics[width=0.3\textwidth]{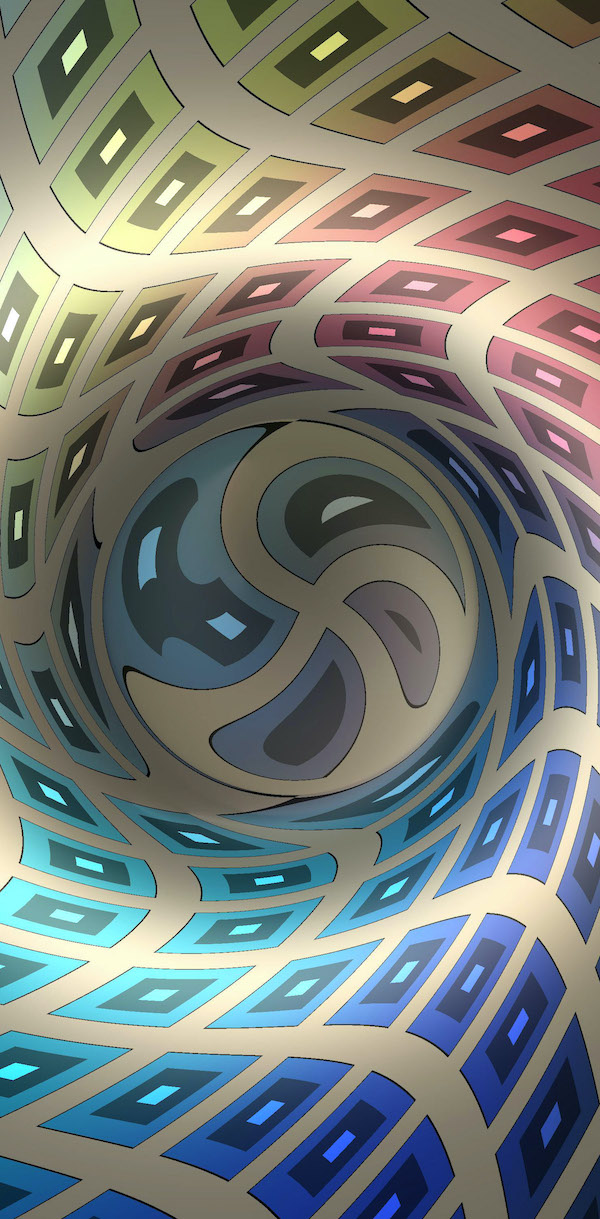}
}
\subfloat[{$p=[5,0,3,1]$}]{
	\includegraphics[width=0.3\textwidth]{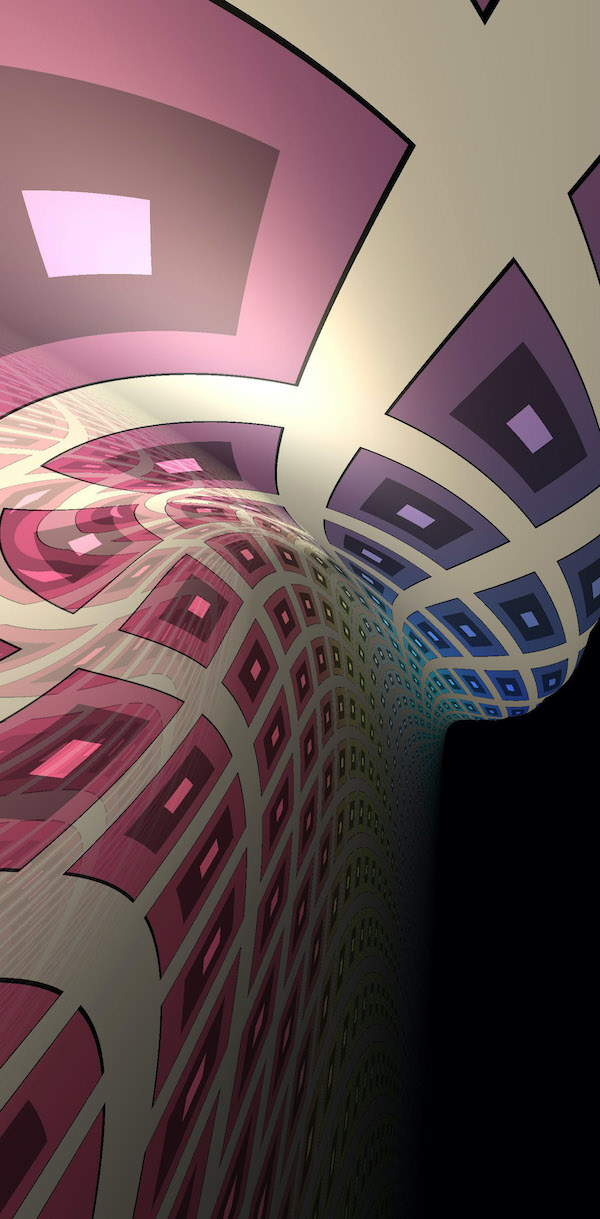}
}
\subfloat[{$p=[0,5,15,1]$}]{
	\includegraphics[width=0.3\textwidth]{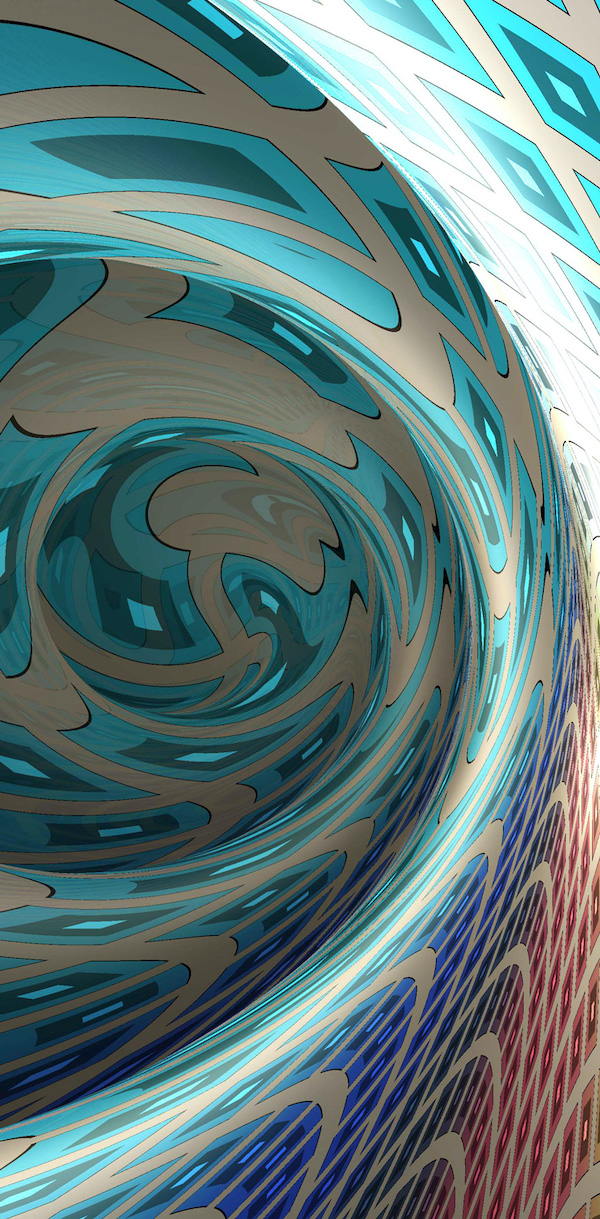}
}
\caption{The $z \leq 0$ half-space in Nil geometry, viewed from the point $p$.  Rendered with artificial (constant) light intensity, and fog.}
\label{Fig:Nil_XYPlane}
\end{figure}

\subsection{Lighting}
\label{Sec:NilLighting}

We addressed the calculation of lighting pairs in \refsec{ExactGeodesicsNil}. Here, we calculate the intensity $I(r,u)$ experienced from an isotropic light source traveling a distance $r$ with initial tangent $u$. Recall that this is inversely proportional to the area density $\mathcal{A}(r,u)$.
Here we calculate this area density directly by taking the derivative of the geodesic flow as in \refeqn{AreaDensityDifferential}.

A parameterization of the speed geodesic starting at the origin $o=\be_4$ with arc length parameter $r$ in the direction $u=[a\cos\alpha,a\sin\alpha, c,0]\in T_o \textrm{Nil}$ is given  by \refeqn{nil geodesic flow 1} for the generic case (when $c\neq 0$) and by \refeqn{nil geodesic flow 2} for geodesics in the $xy$-plane.
Below we concern ourselves with the generic case.
Let $(L,z,\alpha)$ be the cylindrical coordinates on $T_o\textrm{Nil}$ with $(L,z)$ the norm of the projections onto the $xy$ plane and $z$ axis respectively, and $\alpha\in[0,2\pi)$ measured from the positive $x$ axis.
In these coordinates the point $ru\in T_o\textrm{Nil}$ is expressed $(L,z,\alpha)=(ra,rc,\alpha)$.
Thus, using \refeqn{AreaDensity_Coordinates} we may calculate the area density in terms of the $L,z$ and $\alpha$ derivatives of \refeqn{nil geodesic flow 1}.

\begin{equation}
\label{Eqn:nil_AreaDensity}
\mathcal{A}=\frac{2r^2}{z^4}\left|\sin\frac{z}{2}\right|\left|L^2z\cos\frac{z}{2}-2r^2\sin\frac{z}{2}\right|.\end{equation}
See \reffig{NilIntensity}.
As with the computation of the geodesic flow in \refsec{nil geo flow}, to obtain correct lighting along the $xy$ plane direction, one should use the asymptotic expansion of \refeqn{nil_AreaDensity} around $z=0$.

\begin{figure}[htbp]
\centering
\subfloat[Within a ball of radius 10.]{
\includegraphics[width=0.4\textwidth]{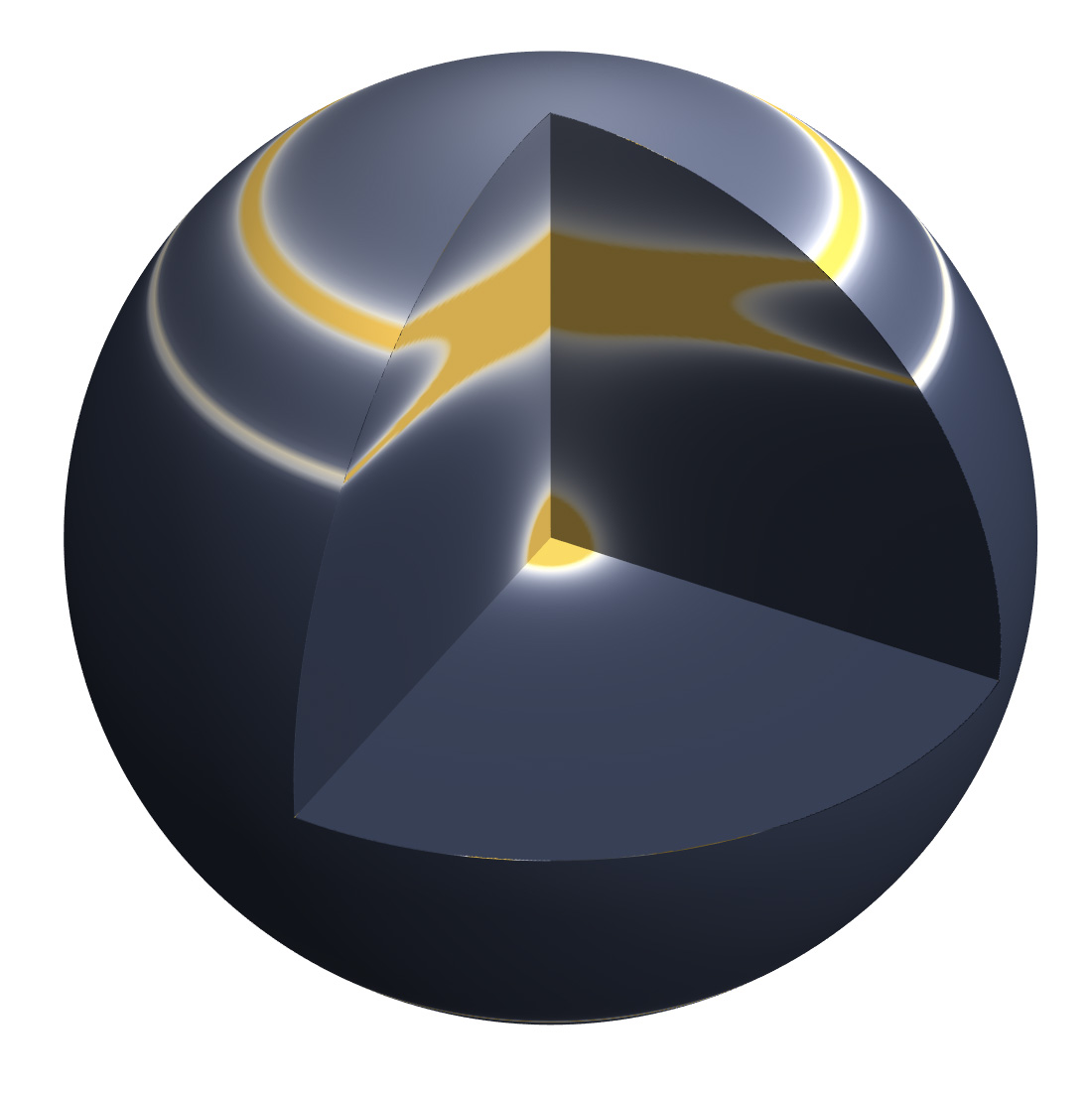}
}
\quad
\subfloat[Within a ball of radius 30.]{
\includegraphics[width=0.4\textwidth]{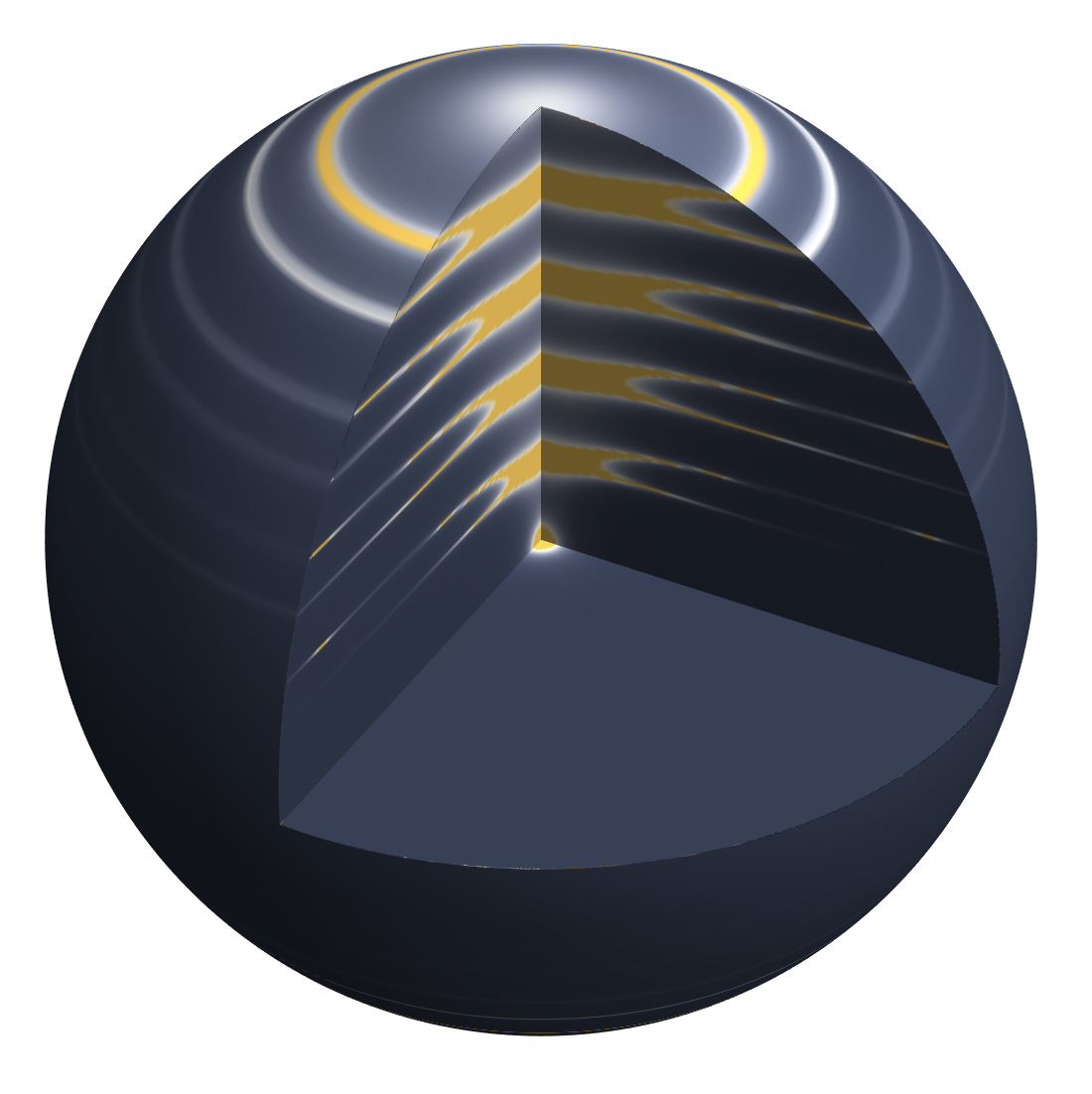}
}
\caption{The lighting intensity function $I(r,u)$ in Nil geometry.}
\label{Fig:NilIntensity}
\end{figure}

In horizontal directions, the light intensity quickly drops away. Near the vertical axis, the intensity of a light source periodically blows up as geodesics reconverge.  See Figures \ref{Fig:NilIntensityBallLine}, \ref{Fig:NilIntensityExample}, \ref{Fig:NilIntensityExampleXZPlane}, and \ref{Fig:NilIntensityExampleXYPlane}.

\begin{figure}[htbp]
\centering
\subfloat[At most 2 geodesics.]{
\includegraphics[width=0.47\textwidth]{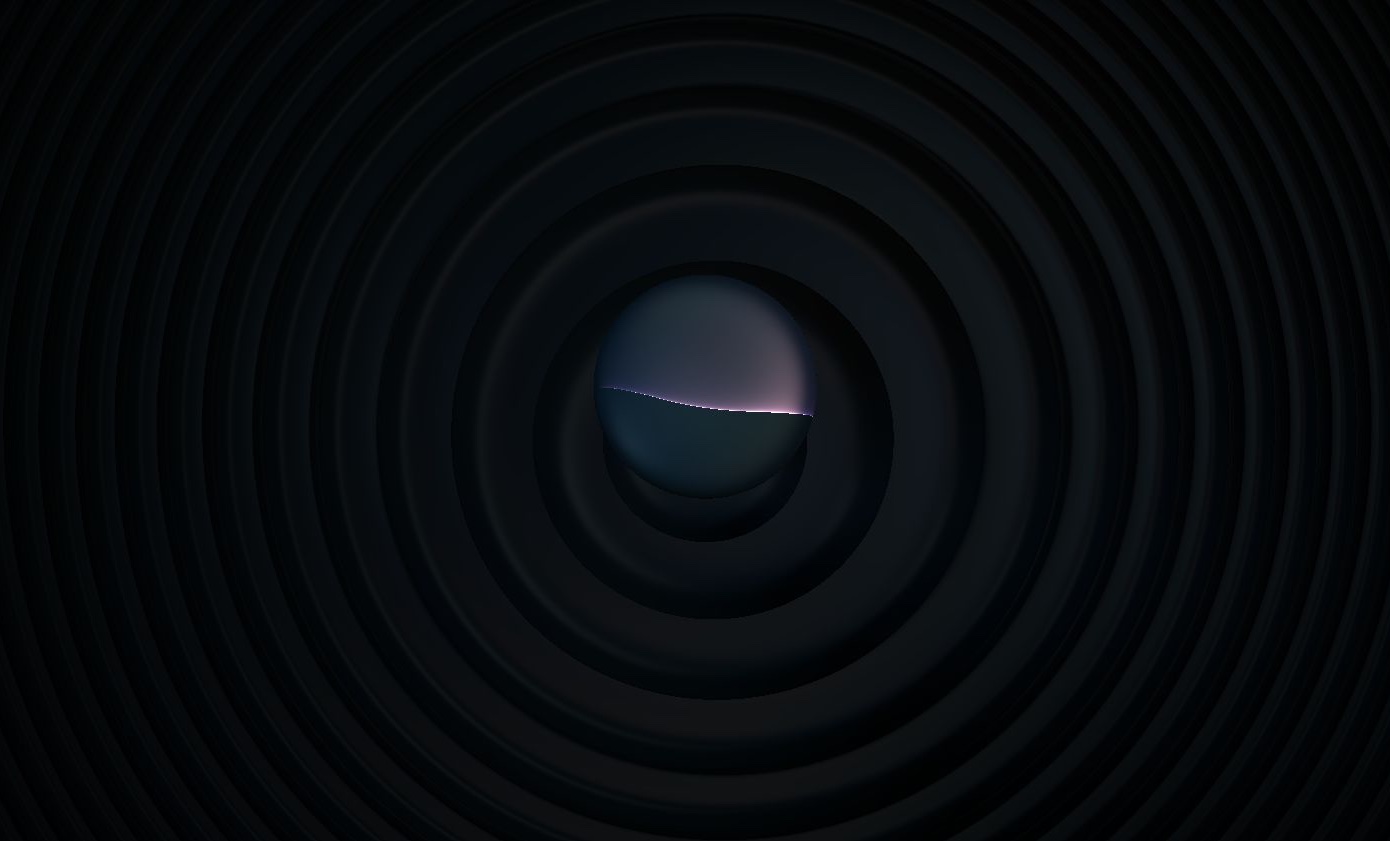}
\label{}
}
\subfloat[At most 4 geodesics.]{
\includegraphics[width=0.47\textwidth]{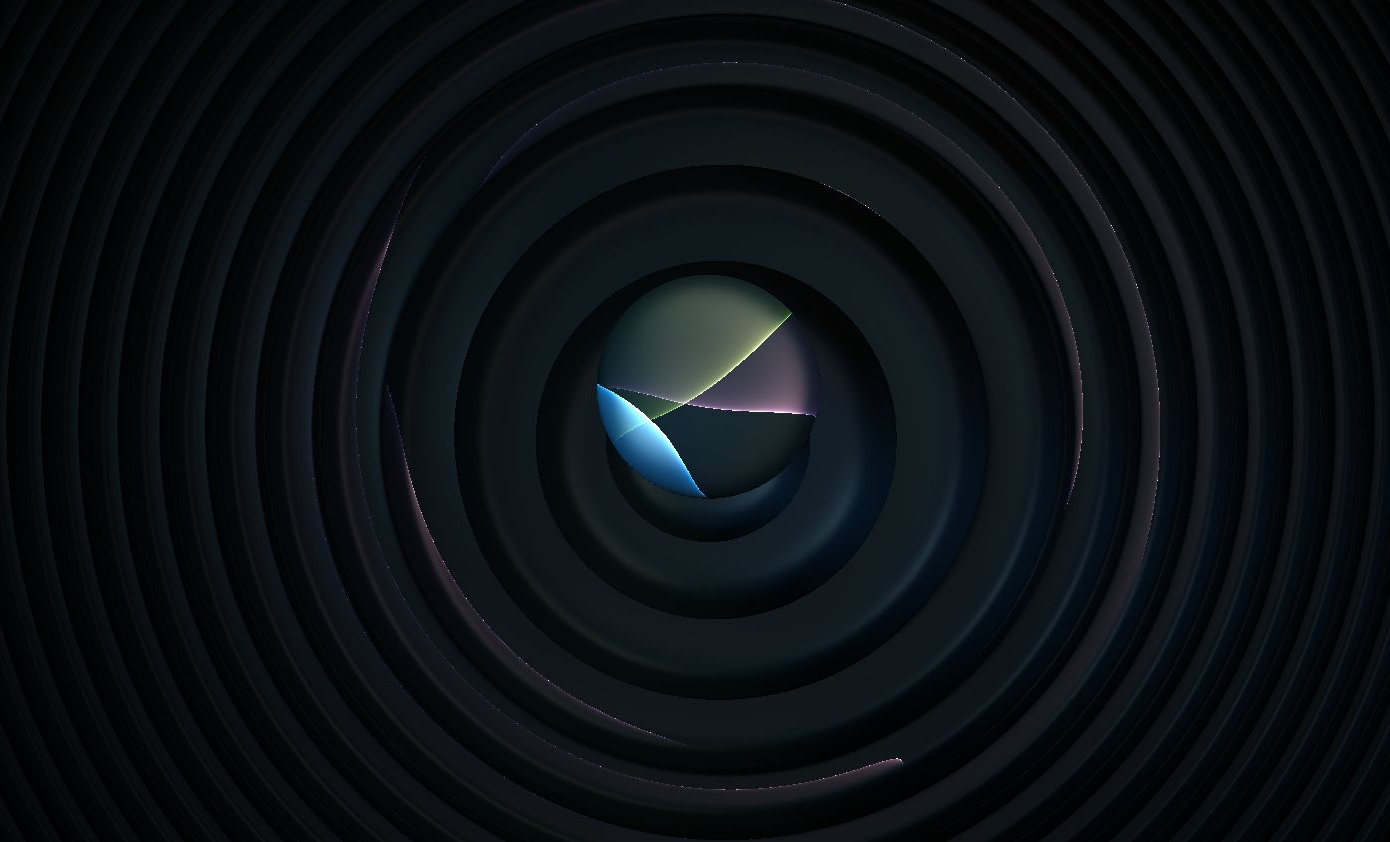}
\label{}
}\\
\subfloat[At most 8 geodesics.]{
\includegraphics[width=0.47\textwidth]{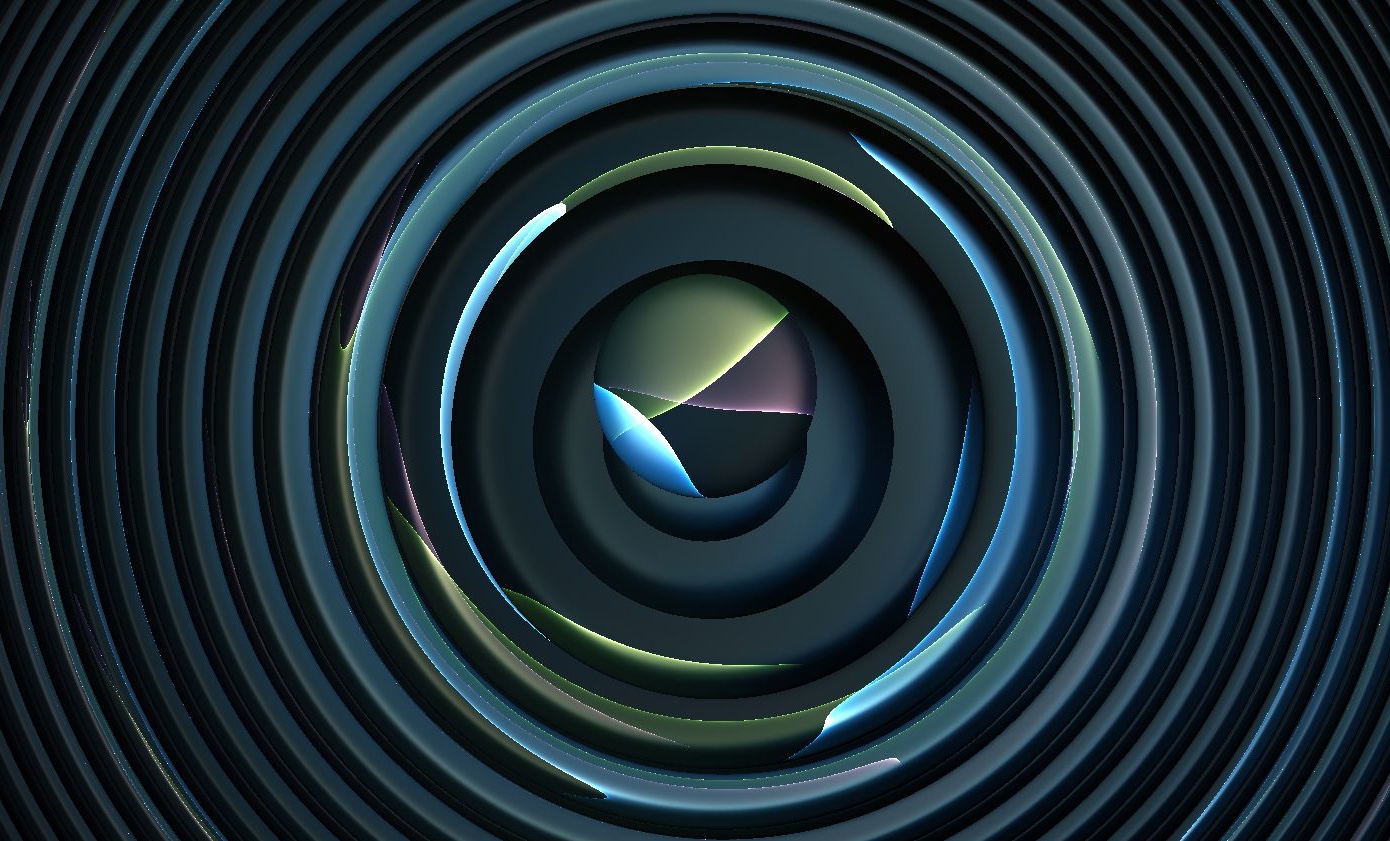}
\label{}
}
\subfloat[At most 16 geodesics.]{
\includegraphics[width=0.47\textwidth]{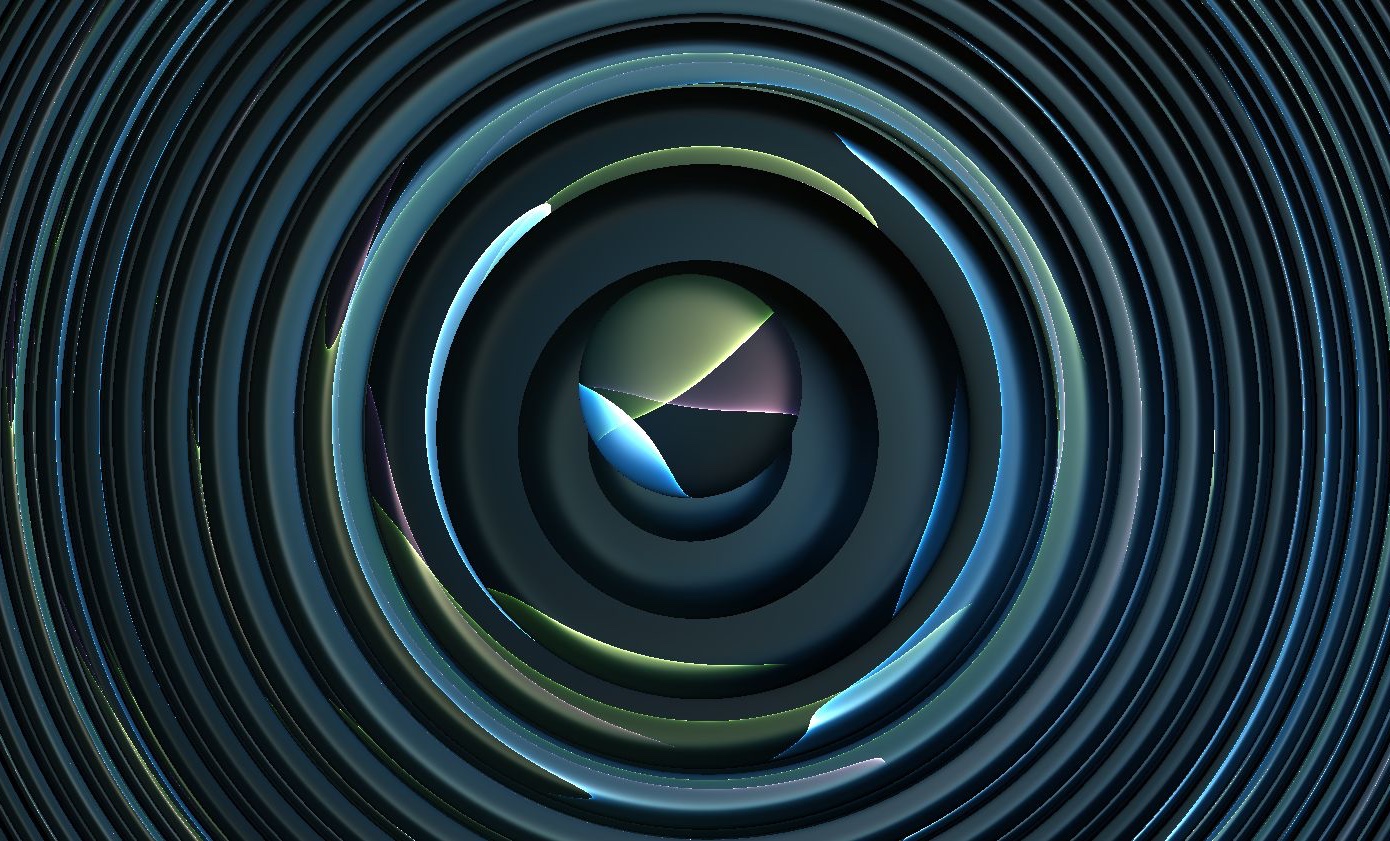}
\label{}
}
\caption{A line of balls in Nil along the $z$-axis, lit by three light sources (cyan, yellow, and magenta) far behind the viewer, using correct light intensity.  Compare with \reffig{NilAllDirections}, which is rendered with constant light intensity.
As almost every point is reached by finitely many geodesics, one may render accurate lighting for any compact region of Nil by computing a sufficiently large number of possible directions.}
\label{Fig:NilIntensityBallLine}
\end{figure}

\begin{figure}[htbp]
\centering
\subfloat[Near the lights (one unit in front of viewer, in the $z$-direction).]{
\includegraphics[width=0.45\textwidth]{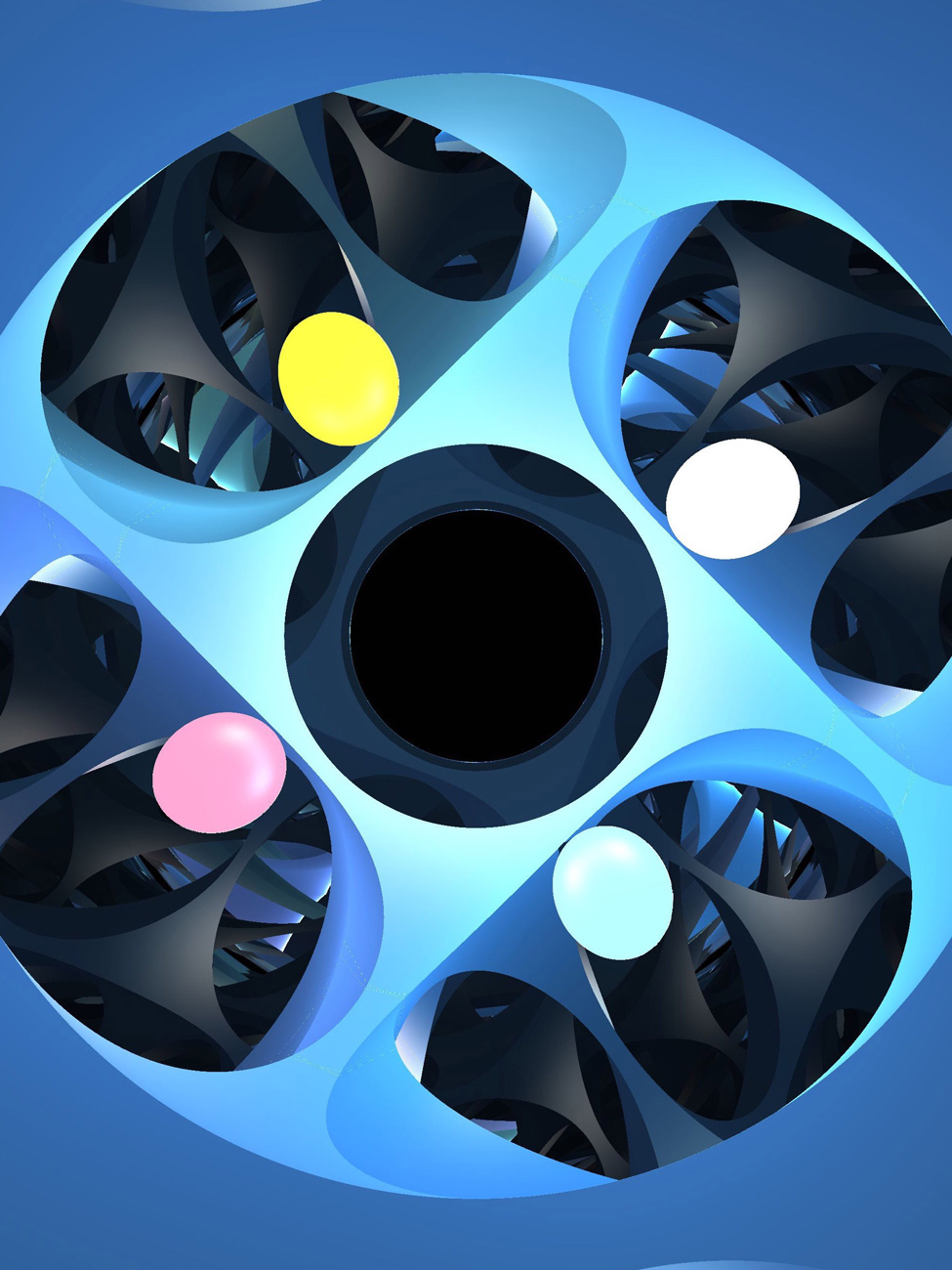}
\label{Fig:Nil_Close}
}
\quad
\subfloat[Far from the lights (seven units behind the viewer, in the $z$-direction).]{
\includegraphics[width=0.45\textwidth]{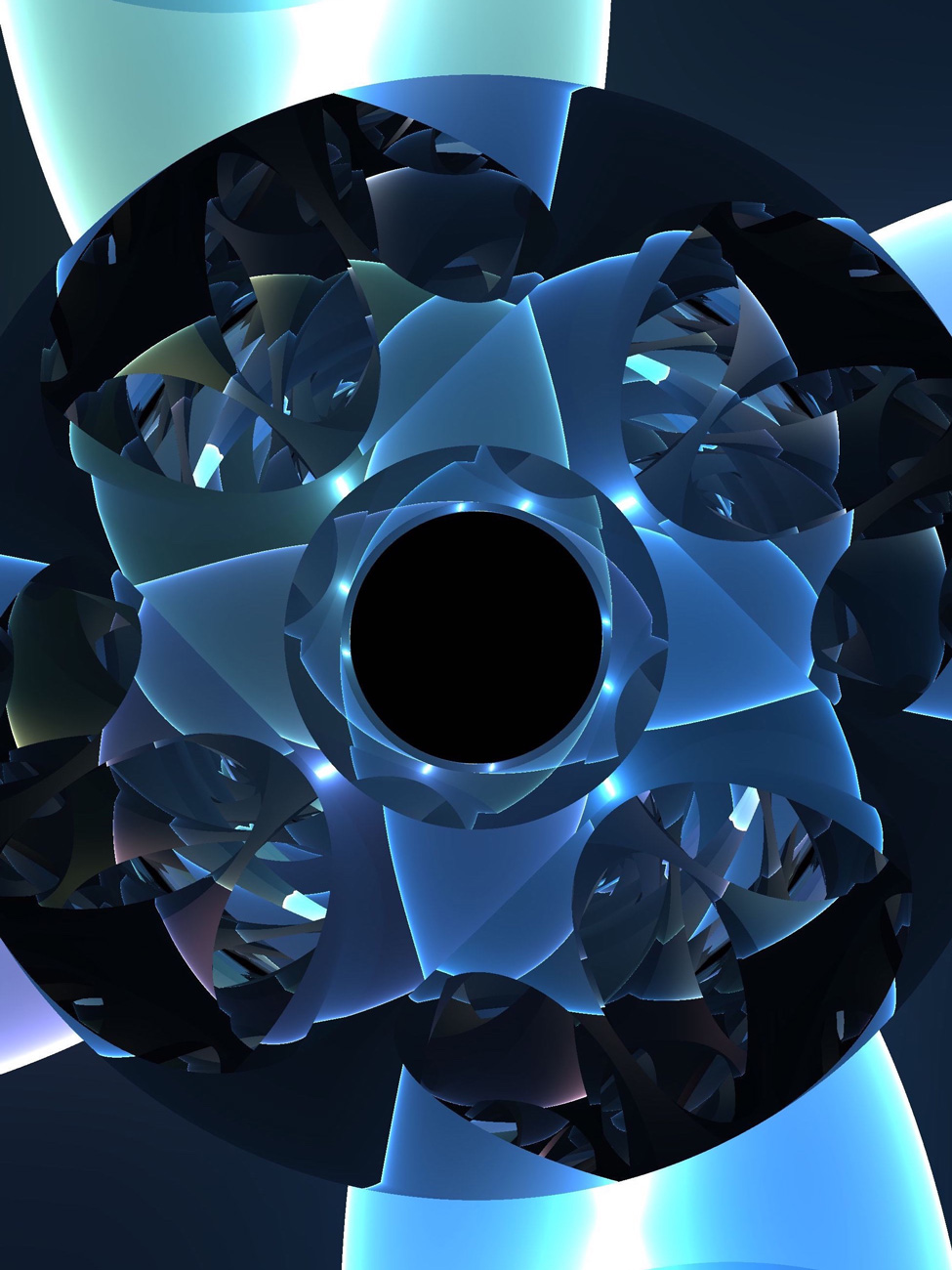}
\label{Fig:Nil_Far}
}
\caption{Four lights (white, yellow, cyan, and magenta) illuminate a tiling of Nil in the style of \reffig{primitive cell E3 - primitive cell}.  Far away, there are curves of high intensity light caused by the convergence of one-parameter families of geodesics. The scene does not cast shadows in these images.}
\label{Fig:NilIntensityExample}
\end{figure}

\begin{figure}[htbp]
\centering
	\includegraphics[width=0.97\textwidth]{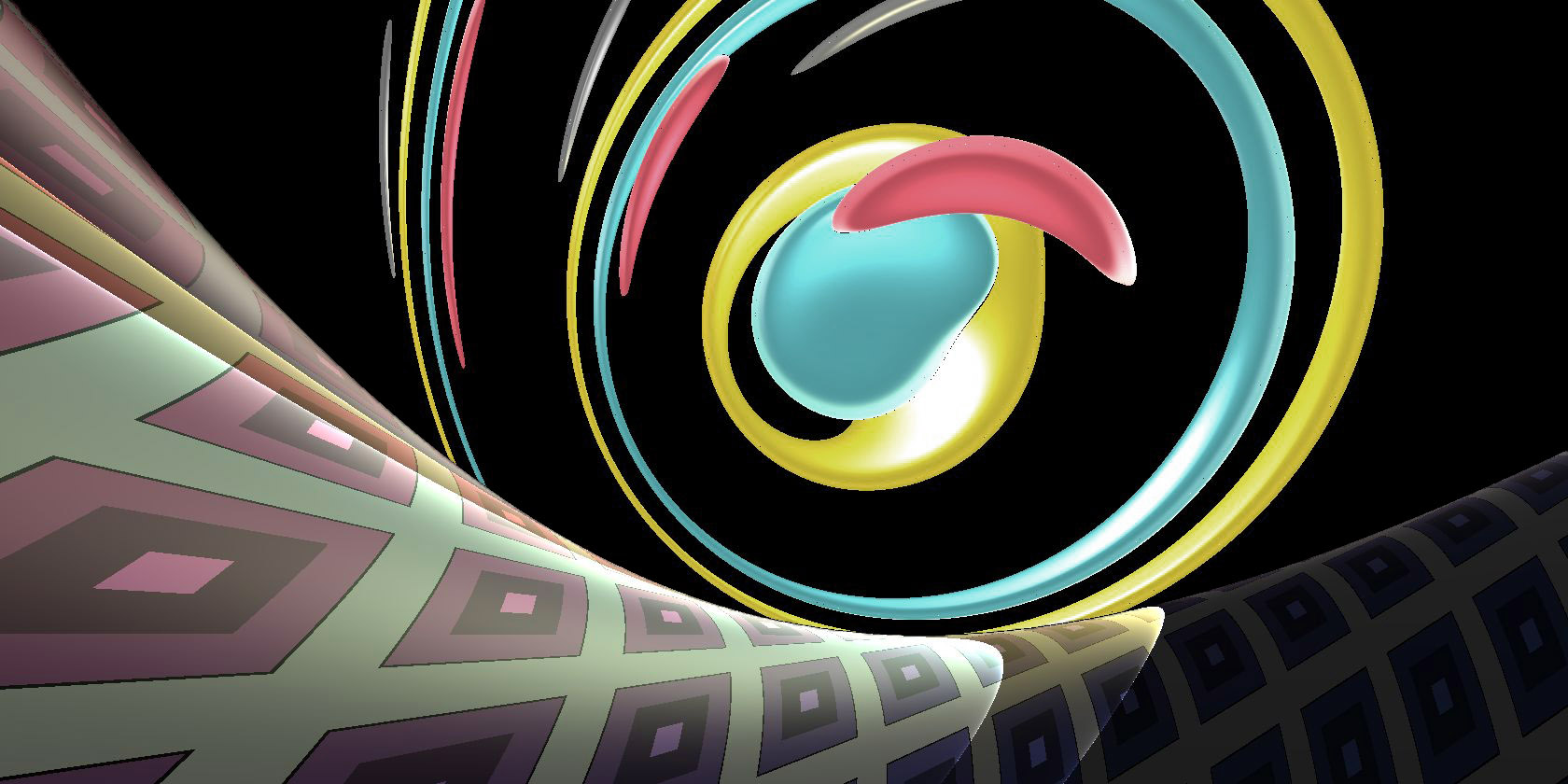}
\caption{Sunset in Nil. When the light intensity blows up far away from the light source, it may illuminate distant parts of an otherwise dark object.  Standing at such a location, the distant light sources appear large in the sky.}
\label{Fig:NilIntensityExampleXZPlane}
\end{figure}

\begin{figure}[htbp]
\centering
\subfloat[$(d,h)=(1,1)$]{
\includegraphics[width=0.3\textwidth]{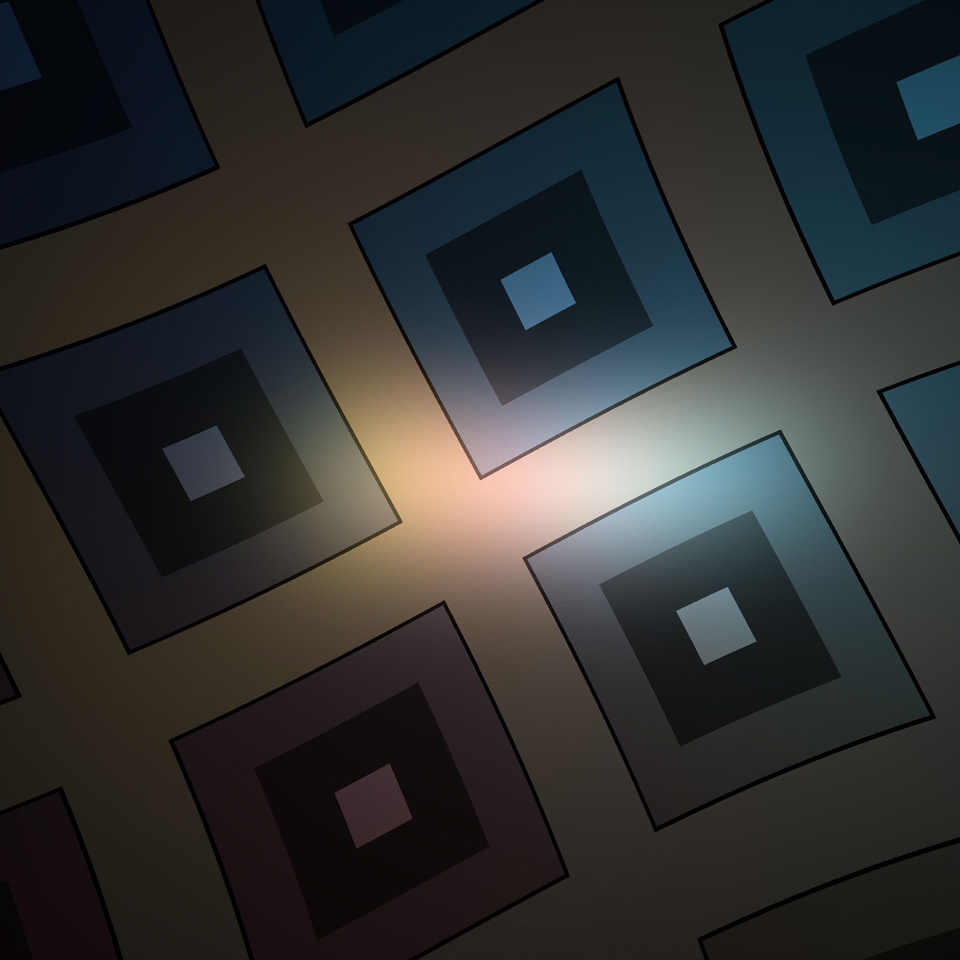}
\label{Fig:NilPlane1}
}
\subfloat[$(d,h)=(8.5,1)$]{
\includegraphics[width=0.3\textwidth]{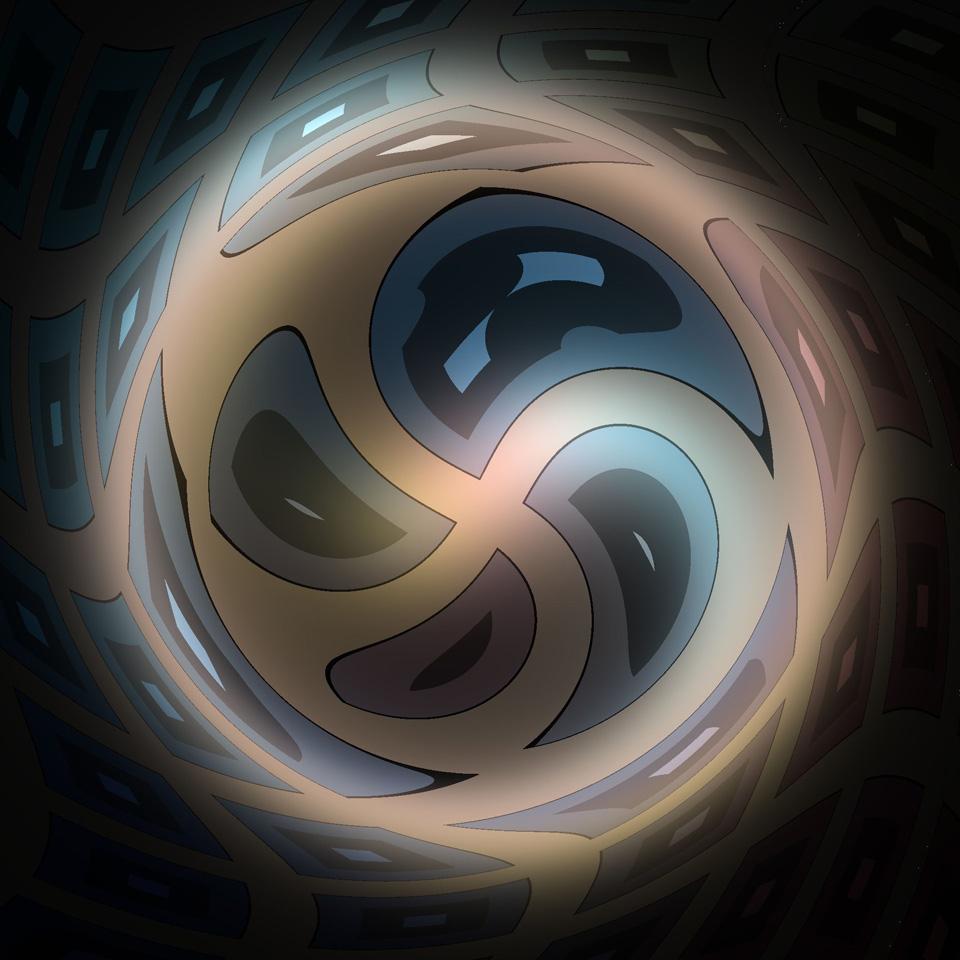}
\label{Fig:NilPlane2}
}
\subfloat[$(d,h)=(45,1)$]{
\includegraphics[width=0.3\textwidth]{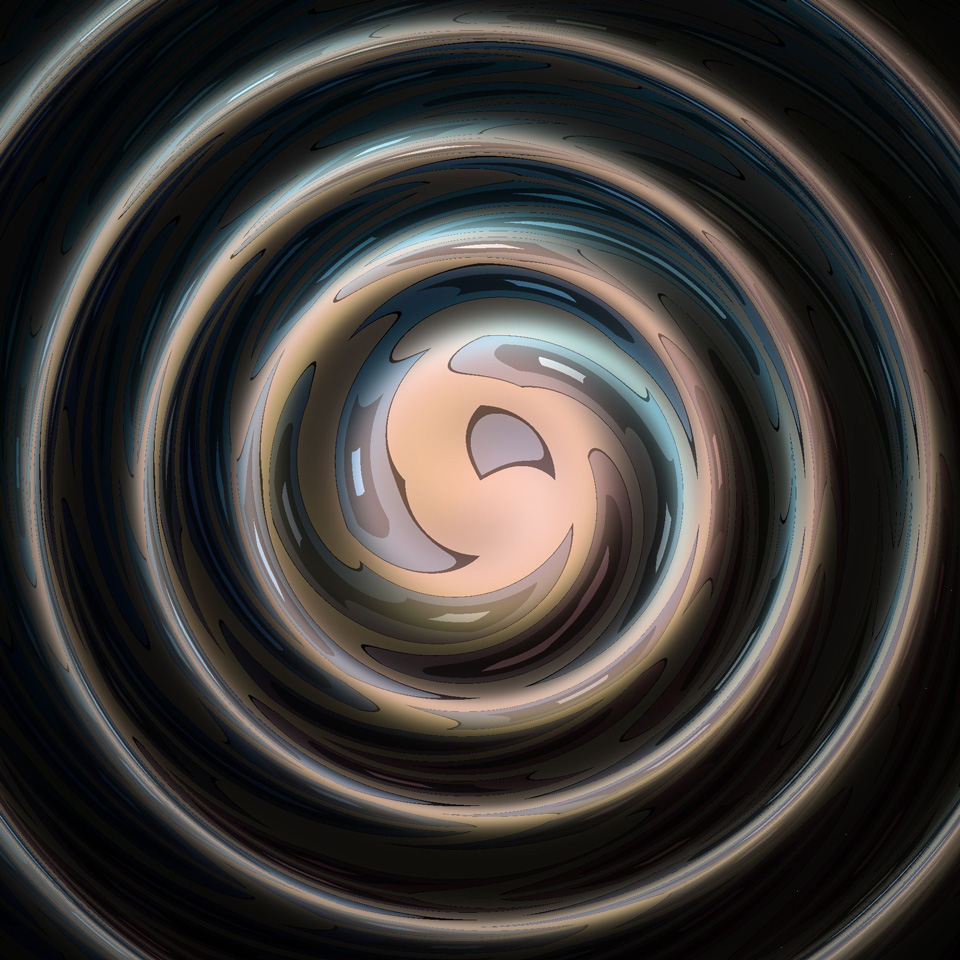}
\label{Fig:NilPlane3}
}\\
\subfloat[$(d,h)=(1,7.7)$]{
\includegraphics[width=0.3\textwidth]{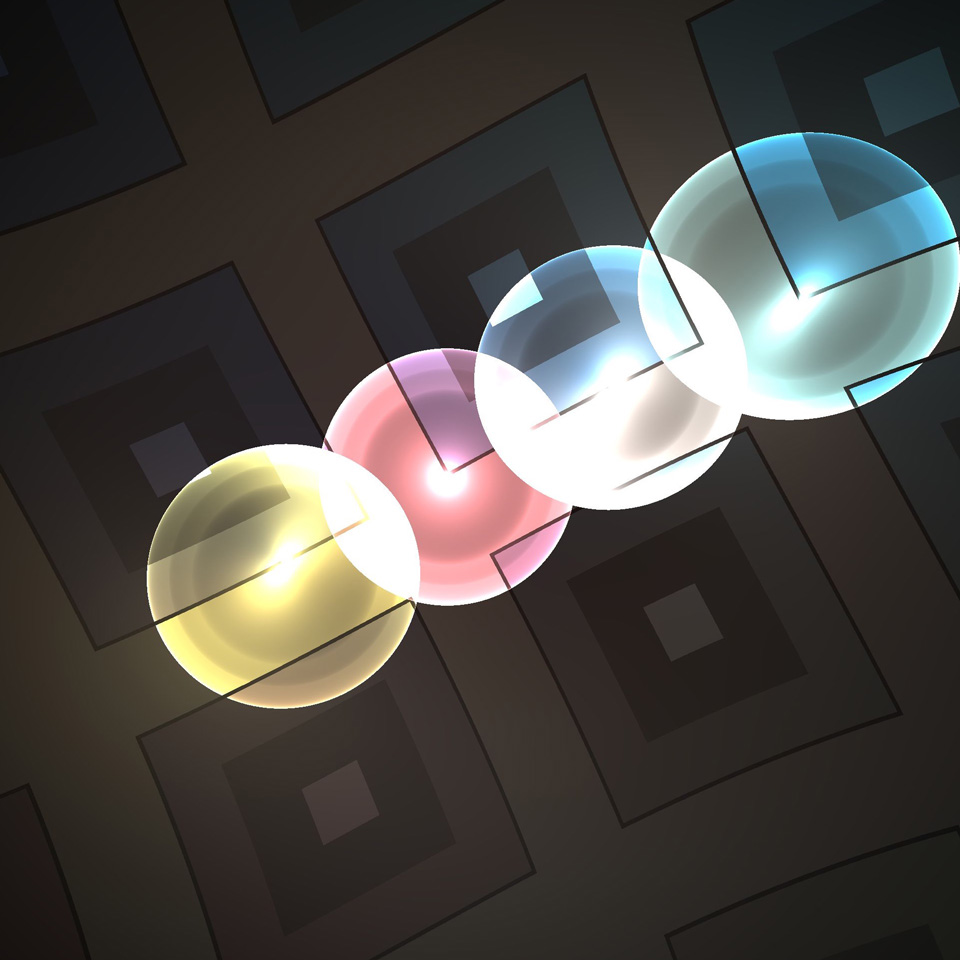}
\label{Fig:NilPlane4}
}
\subfloat[$(d,h)=(8.5,7.7)$]{
\includegraphics[width=0.3\textwidth]{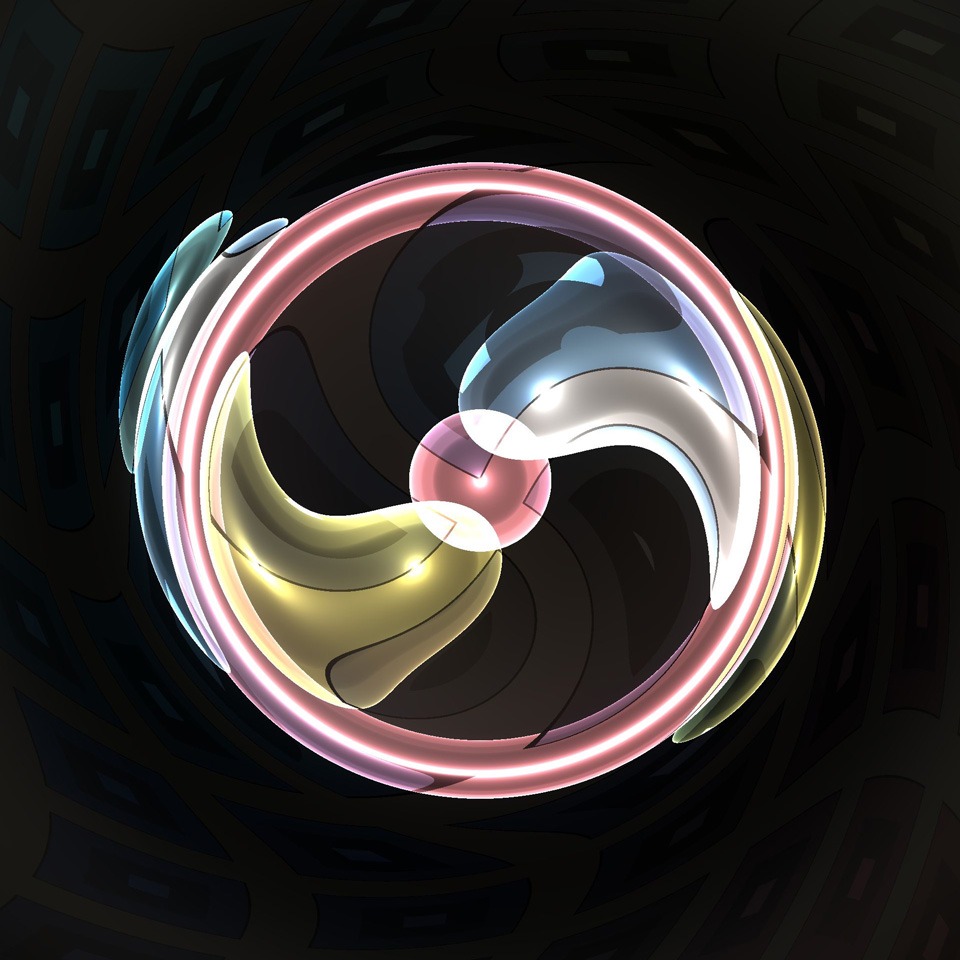}
\label{Fig:NilPlane5}
}
\subfloat[$(d,h)=(45,7.7)$]{
\includegraphics[width=0.3\textwidth]{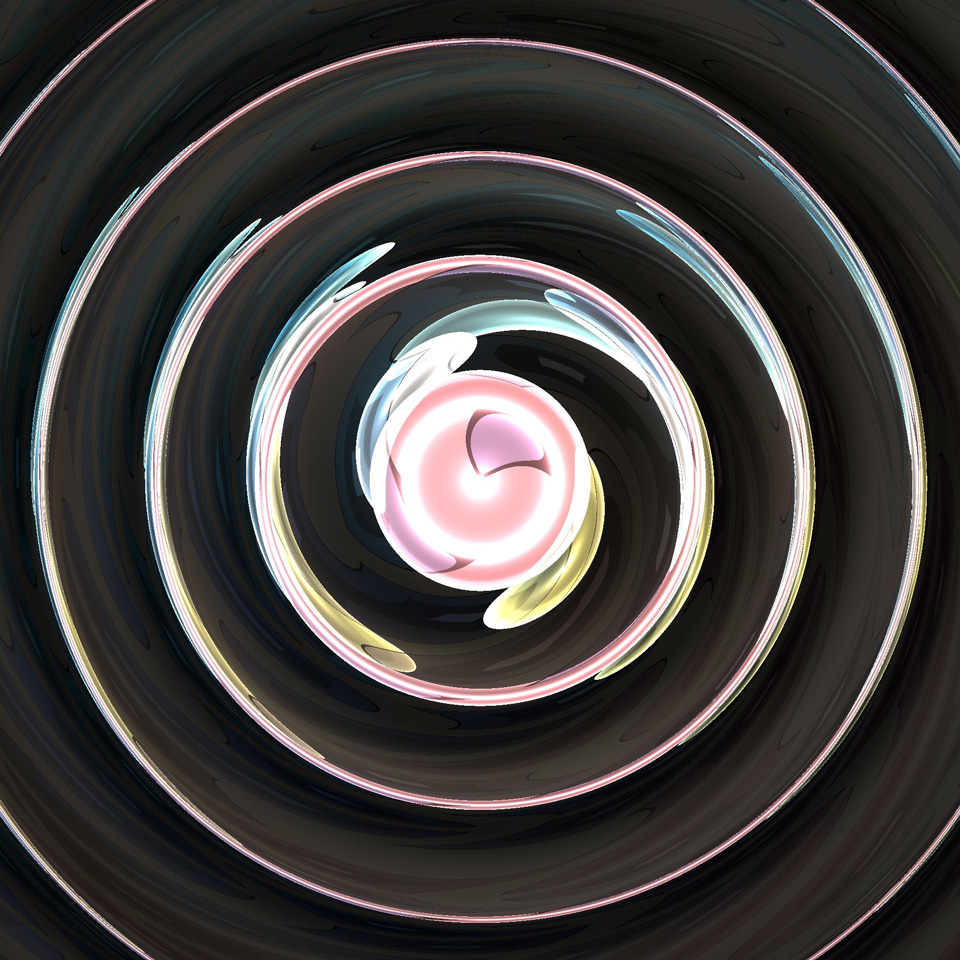}
\label{Fig:NilPlane6}
}\\
\subfloat[$(d,h)=(1,16)$]{
\includegraphics[width=0.3\textwidth]{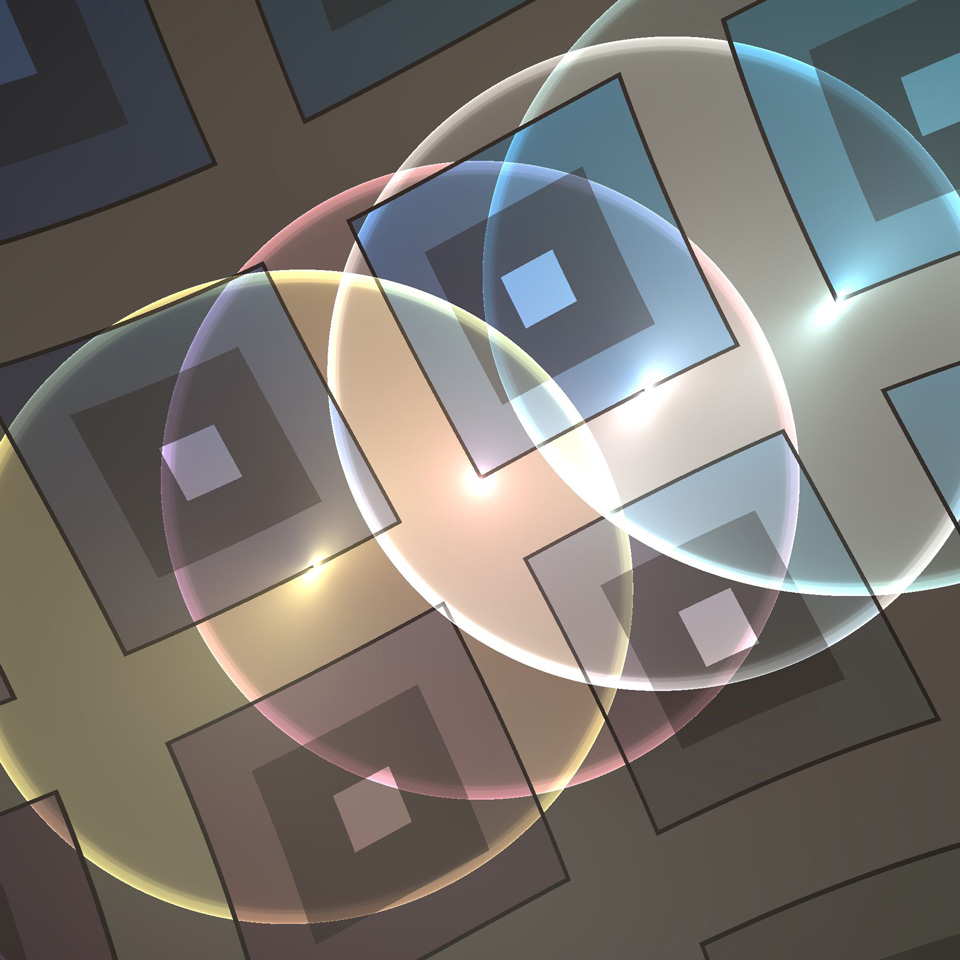}
\label{Fig:NilPlane7}
}
\subfloat[$(d,h)=(8.5,16)$]{
\includegraphics[width=0.3\textwidth]{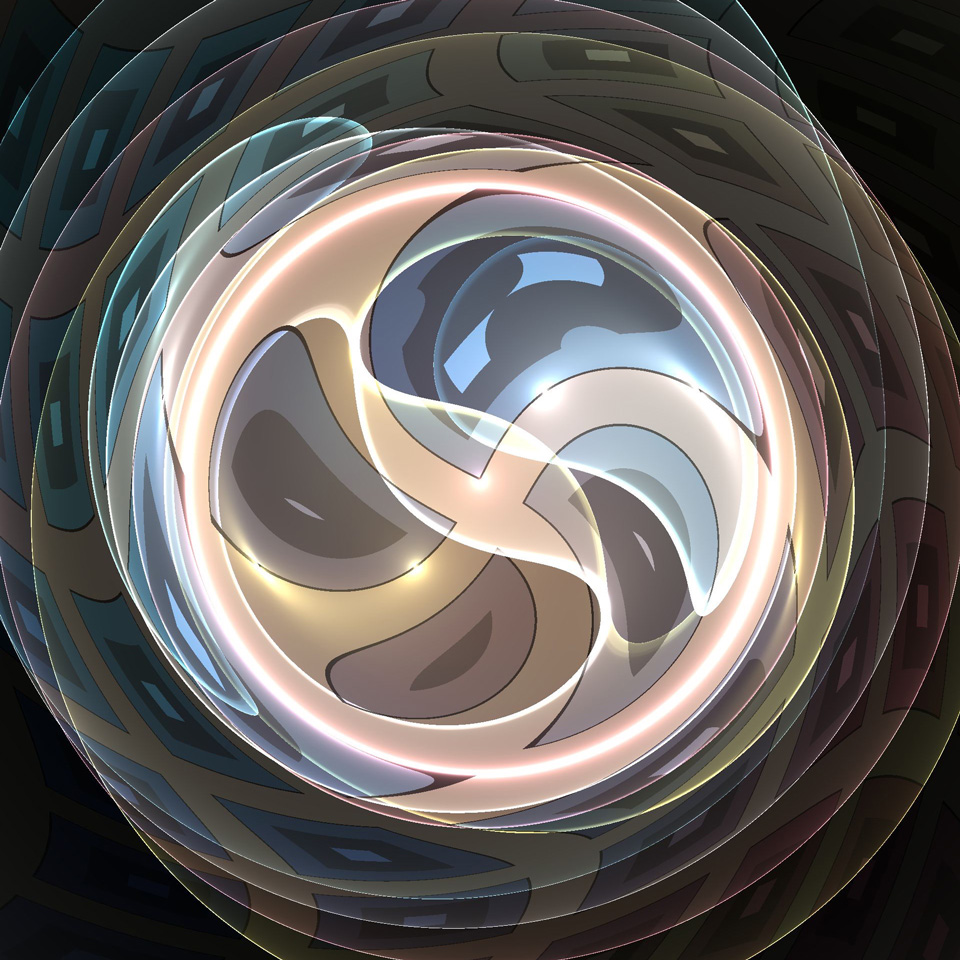}
\label{Fig:Plane8}
}
\subfloat[$(d,h)=(45,16)$]{
\includegraphics[width=0.3\textwidth]{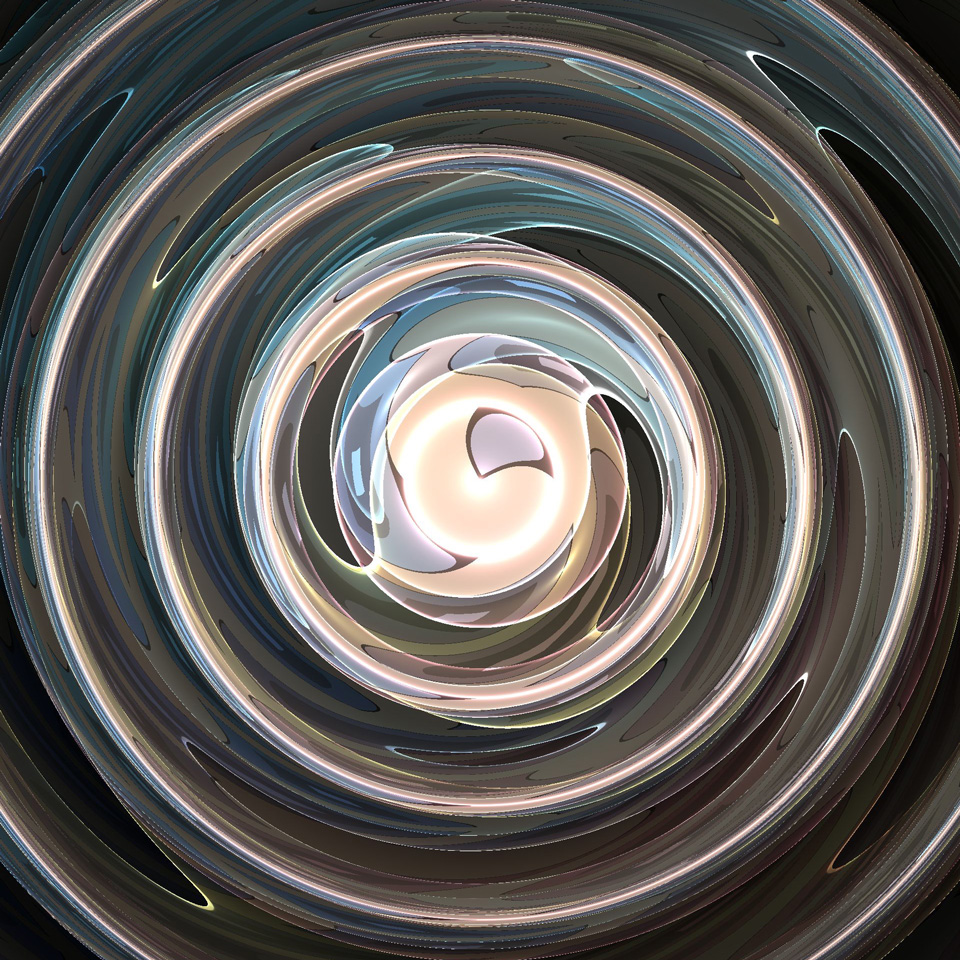}
\label{Fig:NilPlane9}
}\\
\subfloat[$(d,h)=(1,30)$]{
\includegraphics[width=0.3\textwidth]{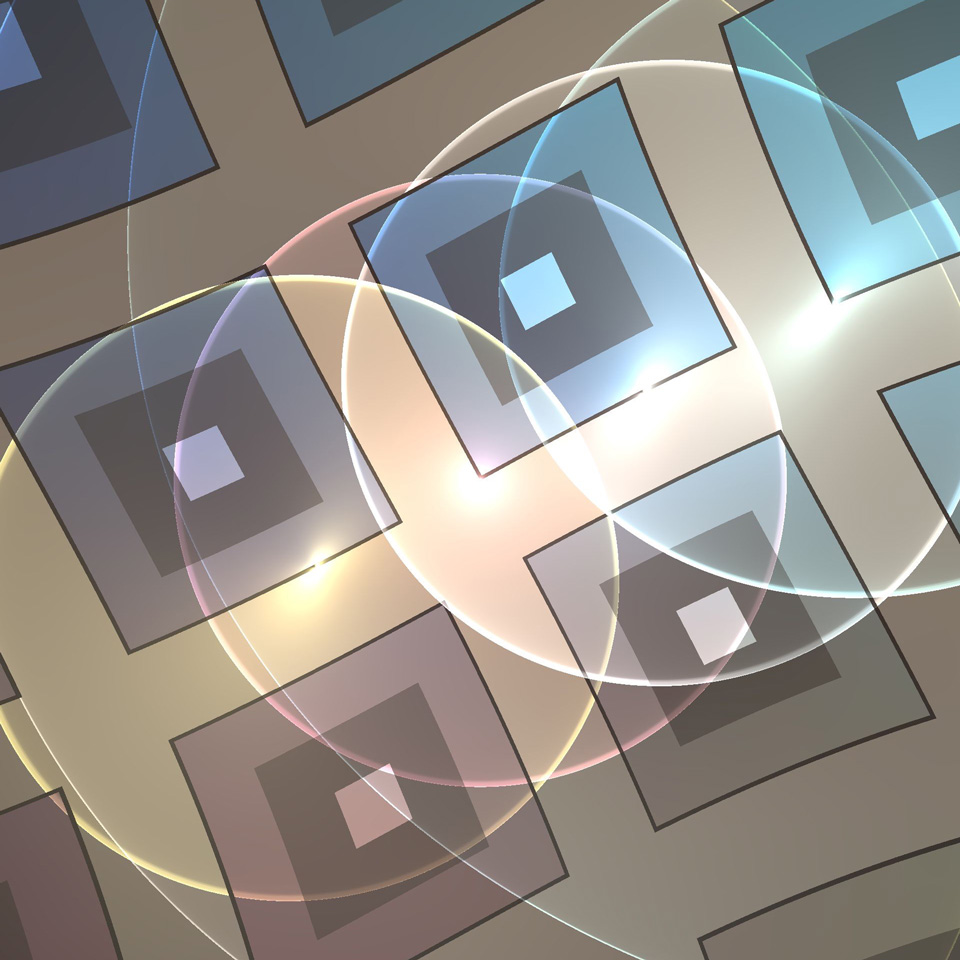}
\label{Fig:NilPlane7}
}
\subfloat[$(d,h)=(8.5,30)$]{
\includegraphics[width=0.3\textwidth]{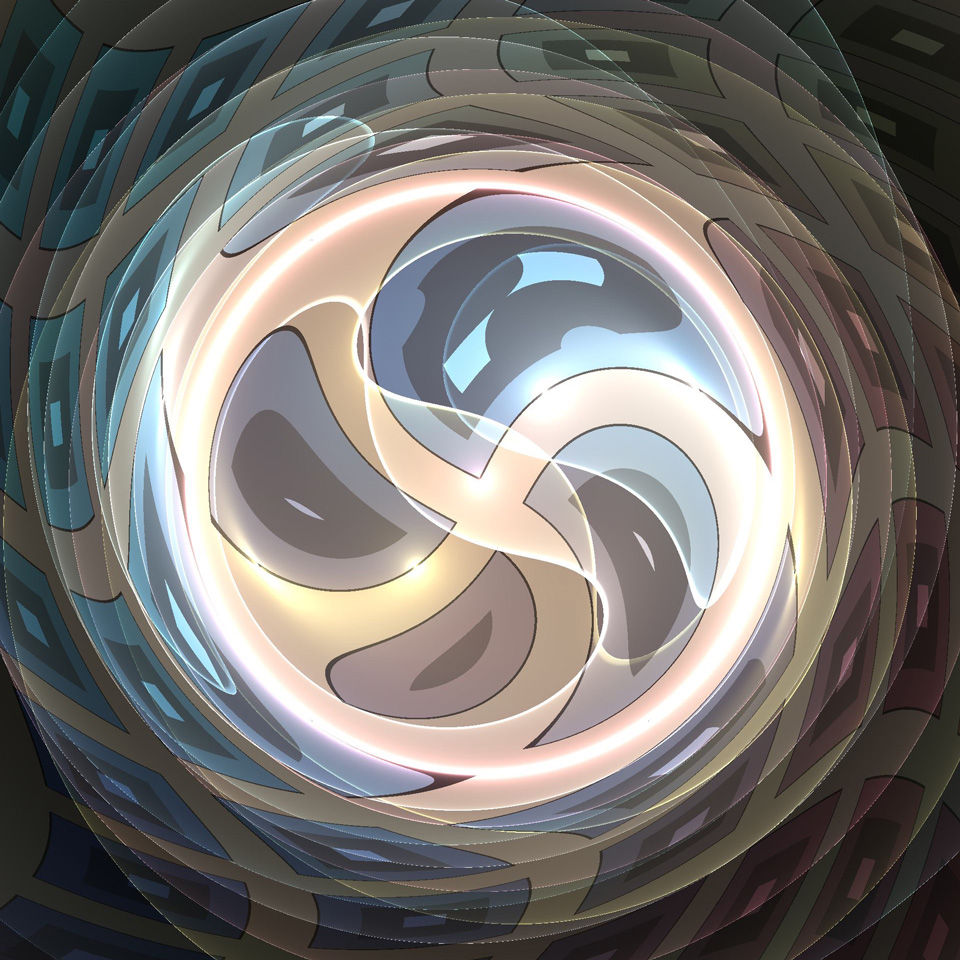}
\label{Fig:Plane8}
}
\subfloat[$(d,h)=(45,30)$]{
\includegraphics[width=0.3\textwidth]{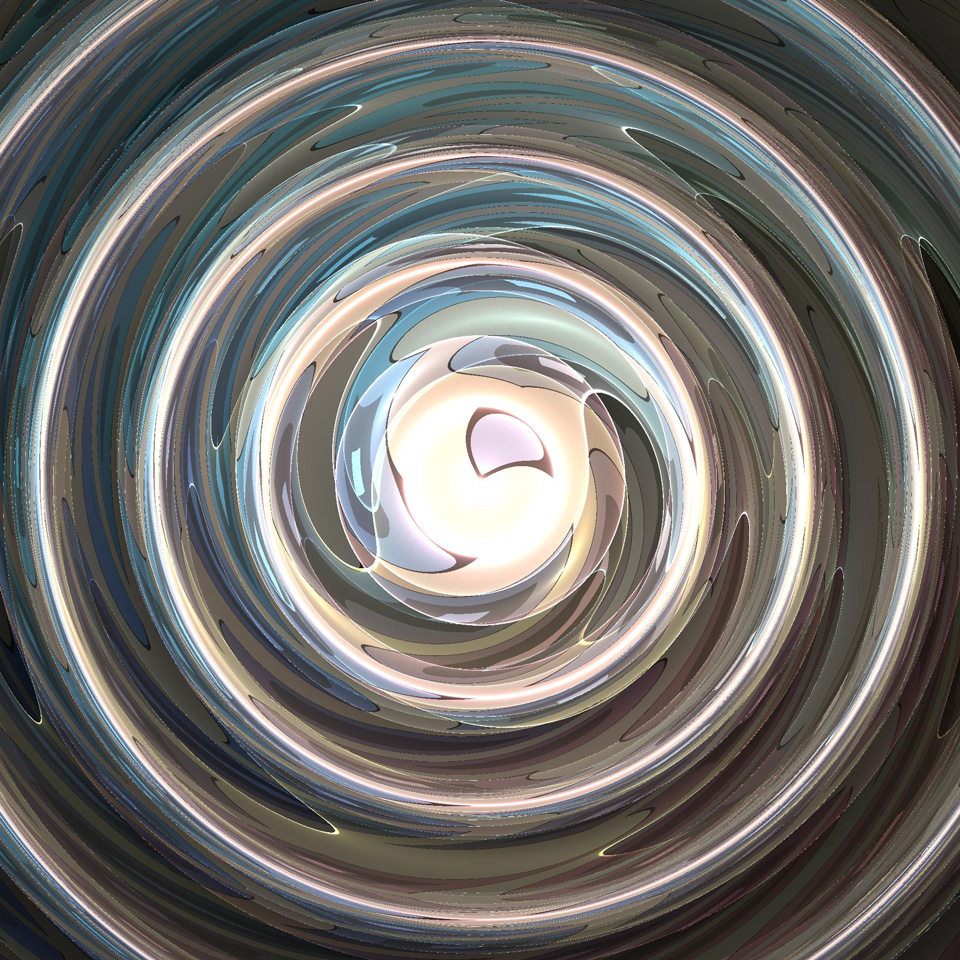}
\label{Fig:NilPlane9}
}\\
\caption{Four lights illuminate the $z\leq 0$ half-space in Nil.
The viewer is above the plane at position $[0,0,d,1]$, and the light sources are positioned at $[k/2,0,h,1]$ for $k\in\{-1,0,1,2\}$.
We use correct lighting with up to three geodesics, and no fog.}
\label{Fig:NilIntensityExampleXYPlane}
\end{figure}

\subsection{Discrete subgroups and fundamental domains.}
\label{Sec:NilDiscreteSubgroupsFundDoms}

The compact Nil manifolds are circle bundles over euclidean two-orbifolds with non-zero Euler class~\cite[Theorem~4.17]{Scott}.
The simplest example of a Nil manifold can also be seen as the suspension $M$ of a regular two-torus $T$ by a Dehn twist.
The fundamental group $\Gamma$ of $M$ is a lattice in $G$. 
We explain here with a concrete example how to construct a fundamental domain $D$ for the action on $\Gamma$ on $X$.  

Let $f$ be the Dehn-twist of the standard two-torus $T = \RR^2 / \ZZ^2$ with action given by the matrix 
\begin{equation*}
	\left[\begin{array}{cc}
		1 & 1 \\
		0 & 1
	\end{array}
 \right].
\end{equation*}
Consider the Dehn-twist torus bundle which is the mapping torus of $T$ with monodromy $f$.
Its fundamental group $\Gamma$ has presentation
\begin{equation*}
	\Gamma = \left< A,B,C \mid [A,B] = C, [A,C] = 1, [B, C] = 1\right>.
\end{equation*}
Here $A$ and $C$ can be interpreted as the standard generators of $\pi_1(T) \cong \ZZ^2$. The conjugation by $B$ is the automorphism of $\ZZ^2$ induced by $f$.
Note that $C$ is central, hence corresponds to the loop along which we are performing our Dehn twist.
The group $\Gamma$ is actually generated by $A$ and $B$ only.
Nevertheless it is more convenient to keep three generators as they represent translations in three independent directions.
The group $\Gamma$ can be identified with the discrete Heisenberg group, that is the set of points with integer coordinates in the Heisenberg model of Nil.
(Recall that $X$ is the rotation-invariant model of Nil.
Hence the group $\Gamma$ is not the set of integer points in $X$.
This set is actually even not a subgroup.)
Concretely, $A$, $B$, and $C$ are the elements of Nil whose coordinates in $X$ are 
\begin{equation*}
	A = [1,0,0,1], \quad B = [0,1,0,1], \quad \text{and} \quad C = [0,0,1,1].
\end{equation*}
Observe that via the projection $\pi \colon X \to \EE^2$, every element of $\Gamma$ induces an isometry of $\EE^2$: $A$ and $B$ correspond to translations along the $x$- and $y$-axis respectively, while $C$ acts trivially on $\EE^2$.
It follows that the ``cube''
\begin{equation*}
	D = \left[ -1/2, 1/2 \right]^3 \times \{1 \}
\end{equation*}
is a fundamental domain for the action of $\Gamma$ on $X$. Note that $A, B$, and $C$ do not directly pair the square sides of the ``cube''. See \refrem{Cellular}.
Our rotation-invariant model $X$ for Nil is also a projective model.
The fundamental domain $D$ can be seen as the intersection of a collection of half-spaces $H_x^\pm$, $H_y^\pm$, $H_z^\pm$ as described in \refsec{TeleportingProjective}.
Here
\begin{equation*}
	H_x^- = \left\{ x \geq -1/2\right\} \quad \text{and} \quad H_x^+ = \left\{ x \leq 1/2\right\}.
\end{equation*}
while $H_y^\pm$ and $H_z^\pm$ are defined in the same way.
The teleporting algorithm has two main steps.
Let $p = [x,y,z,1]$ be a point in $X$.
\begin{enumerate}
	\item If $p$ does not belong to $H_x^-$ (respectively $H_x^+$, $H_y^-$ $H_y^+$), then we move it by $A$ (respectively $A^{-1}$, $B$, $B^{-1}$).
	After finitely many steps, the new point $p$ will lie in 
	\begin{equation*}
		H_x^- \cap H_x^+ \cap H_y^- \cap H_y^+.
	\end{equation*}
	The isometries of $\EE^2$ induced by $A$ and $B$ commute, so we don't pay attention to the order in which we perform these operations.
	\item Once this is done, if $p$ does not belong to $H_z^-$ (respectively $H_z^+$), then we move it by $C$ (respectively $C^{-1}$).
	Note that $C$ does not affect the $xy$-coordinates of $p$. Therefore, after this process, $p$ lies in $D$.
\end{enumerate}

\begin{remark}
\label{Rem:Cellular}
	Note that the collection of isometries $\{ A^{\pm 1}, B^{\pm 1}, C^{\pm 1}\}$ does not provide a face pairing of our fundamental domain $D$ in the sense of \refsec{Teleporting}.
	Consider for example the square sides $F_x^-$ and $F_x^+$ which are the intersections of $D$ with the affine planes $\partial H_x^-$ and $\partial H_x^+$ respectively.
	The generator $A$ is a shear, not an affine translation in $\RR^4$ along the $x$-axis. See \reffig{NilFacePairing_Shear}.
	Thus it does not map $F_x^-$ to $F_x^+$.
	In order to get a proper face pairing, one must subdivide the sides of $D$ and increase the number of generators.
	This is illustrated on \reffig{NilFacePairing_Pairing}.
	We draw the one-skeleton of $D$ and color the sides $F_x^-$ (on the left) and $F_x^+$ (on the right).
	The yellow (respectively blue, red) face in $F_x^-$ is mapped bijectively to the face with the same color in $F_x^+$ via $A$ (respectively $AC$, $AC^{-1}$).
	\begin{figure}[htbp]
	\centering
		\subfloat[The fundamental domain $D$ (yellow) and its image (red) under the shear $A$. The red, green, and blue lines correspond to the $x$-, $y$-, and $z$-axes in our model of Nil.]{
			\includegraphics[width=0.48\textwidth]{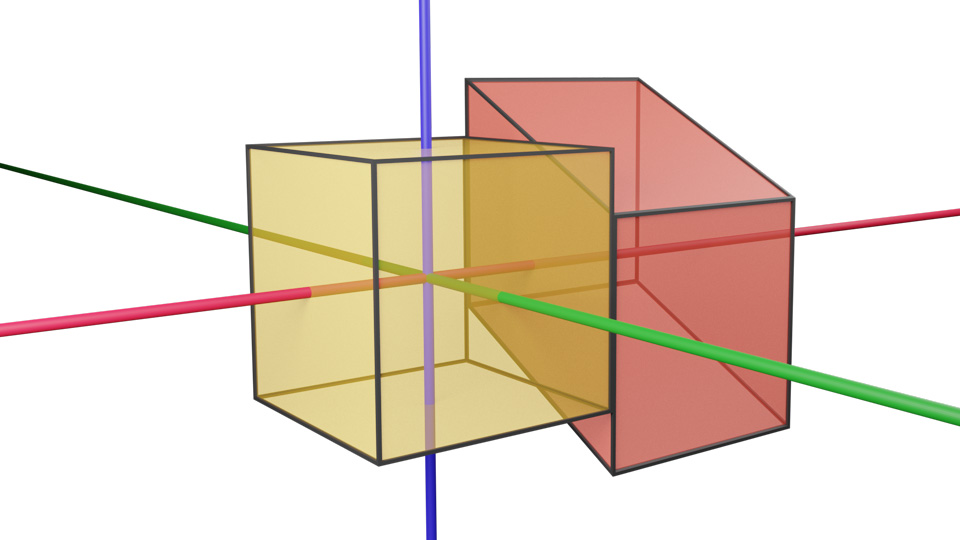}
			\label{Fig:NilFacePairing_Shear}
		}
		\subfloat[The yellow, blue, and red faces are in one-to-one correspondence via $A$, $AC$, and $AC^{-1}$. The decorations indicate the edge identifications induced by $A$ and $C$ only.]{
			\labellist
			\footnotesize\hair 2pt
			\pinlabel $a$ at 250 375
			\pinlabel $b$ at 200 543
			\pinlabel $b$ at 223 175
			\pinlabel $c$ at 104 541 
			\pinlabel $d$ at 89 422
			\pinlabel $e$ at 129 310
			\pinlabel $e$ at 100 620
			\pinlabel $f$ at 226 274
			\pinlabel $g$ at 23 570
			\pinlabel $h$ at 312 214
			
			\pinlabel $a$ at 550 411
			\pinlabel $b$ at 465 569
			\pinlabel $c$ at 354 626 
			\pinlabel $c$ at 338 360 
			\pinlabel $d$ at 323 531
			\pinlabel $e$ at 360 439
			\pinlabel $f$ at 458 363
			\pinlabel $f$ at 493 660
			\pinlabel $g$ at 320 421
			\pinlabel $h$ at 572 584		
			\endlabellist
			\includegraphics[width=0.47\textwidth]{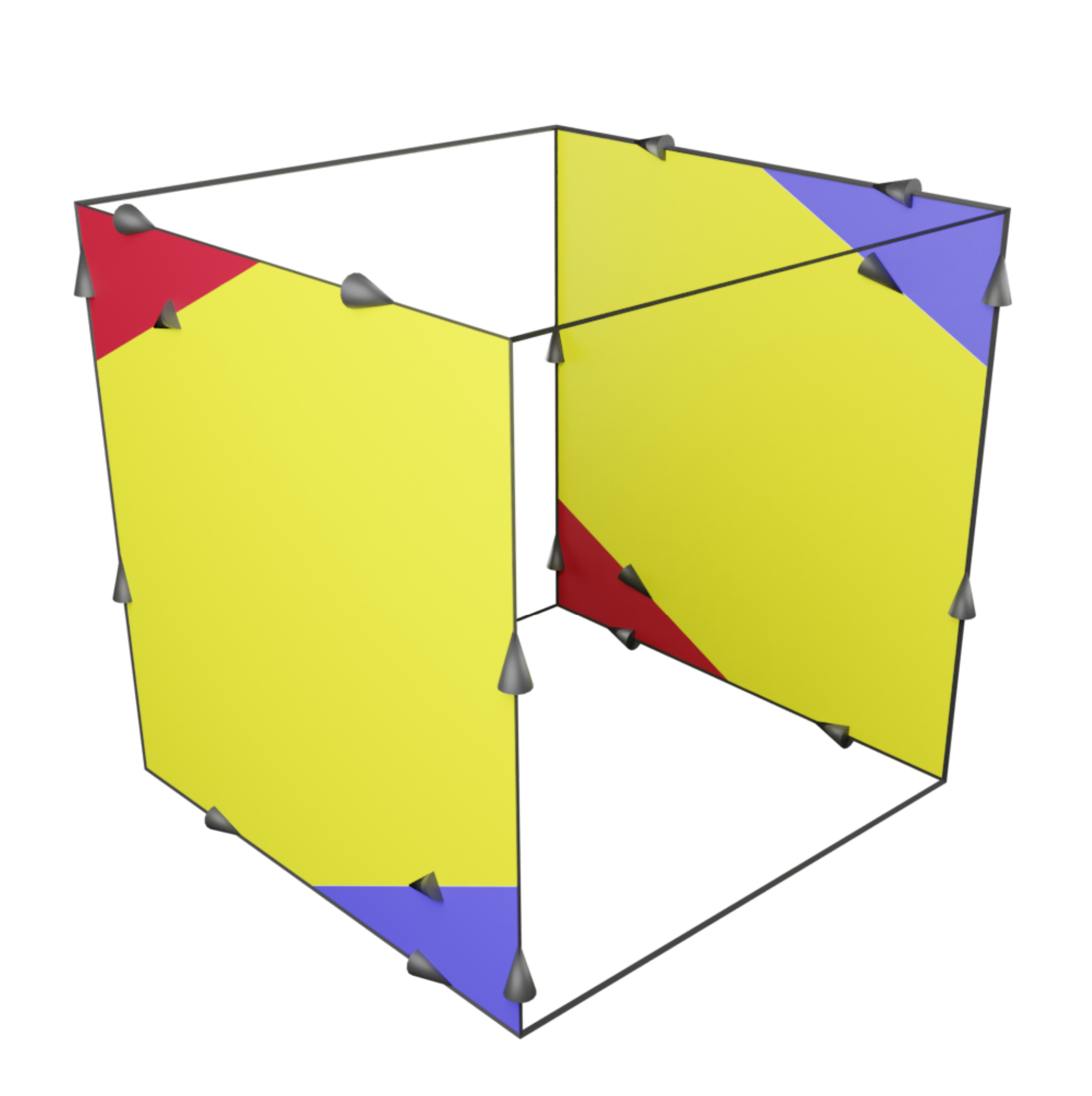}
			\label{Fig:NilFacePairing_Pairing}
		}

	\caption{Face pairing in Nil.}
	\label{Fig:NilFacePairing}
	\end{figure}
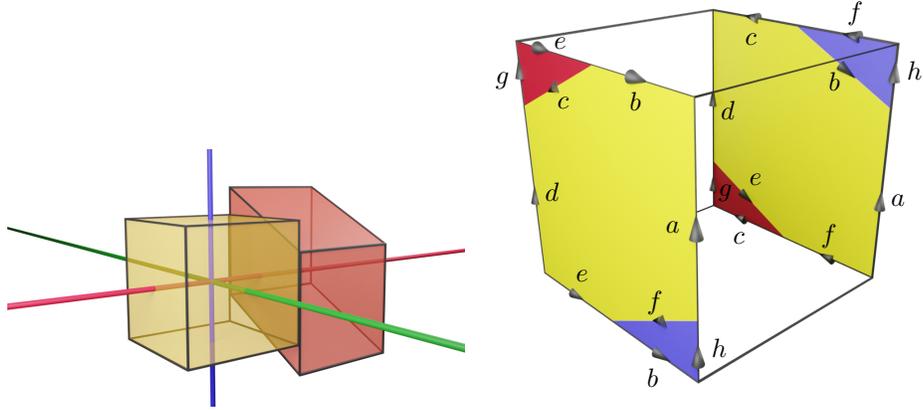
	A similar subdivision can be found for the sides $F_y^- = D \cap \partial H_y^-$ and $F_y^+ = D \cap \partial H_y^+$. (No subdivision of the horizontal faces of $D$ is needed as $C$ is an affine translation along the $z$-axis.)
	As explained in \refrem{ProjectiveNoFacePairings} and illustrated by the above algorithm, when using a fundamental domain defined as the intersection of projective half-spaces, we do not need a proper face pairing to implement teleportation.
\end{remark}

In \reffig{NilExamples}, we show the in-space view for various scenes in Nil geometry. \reffig{NilTorusBundle} shows the 
Dehn-twist torus bundle with monodromy $f$ as in \refsec{NilDiscreteSubgroupsFundDoms}, with a fundamental domain drawn in the style of \reffig{primitive cell E3 - primitive cell}. 
\reffig{NilBalls} shows a lattice of spheres, textured as the Earth, lit by a corresponding lattice of light sources.
\reffig{NilFibers} shows solid cylinders (which we implement as vertical objects) around fibers of Nil. Compare with \reffig{SeifFiberProduct}.

\begin{figure}[htbp]
\centering
\subfloat[The Dehn-twist torus bundle with monodromy $f$.]{
\includegraphics[width=0.90\textwidth]{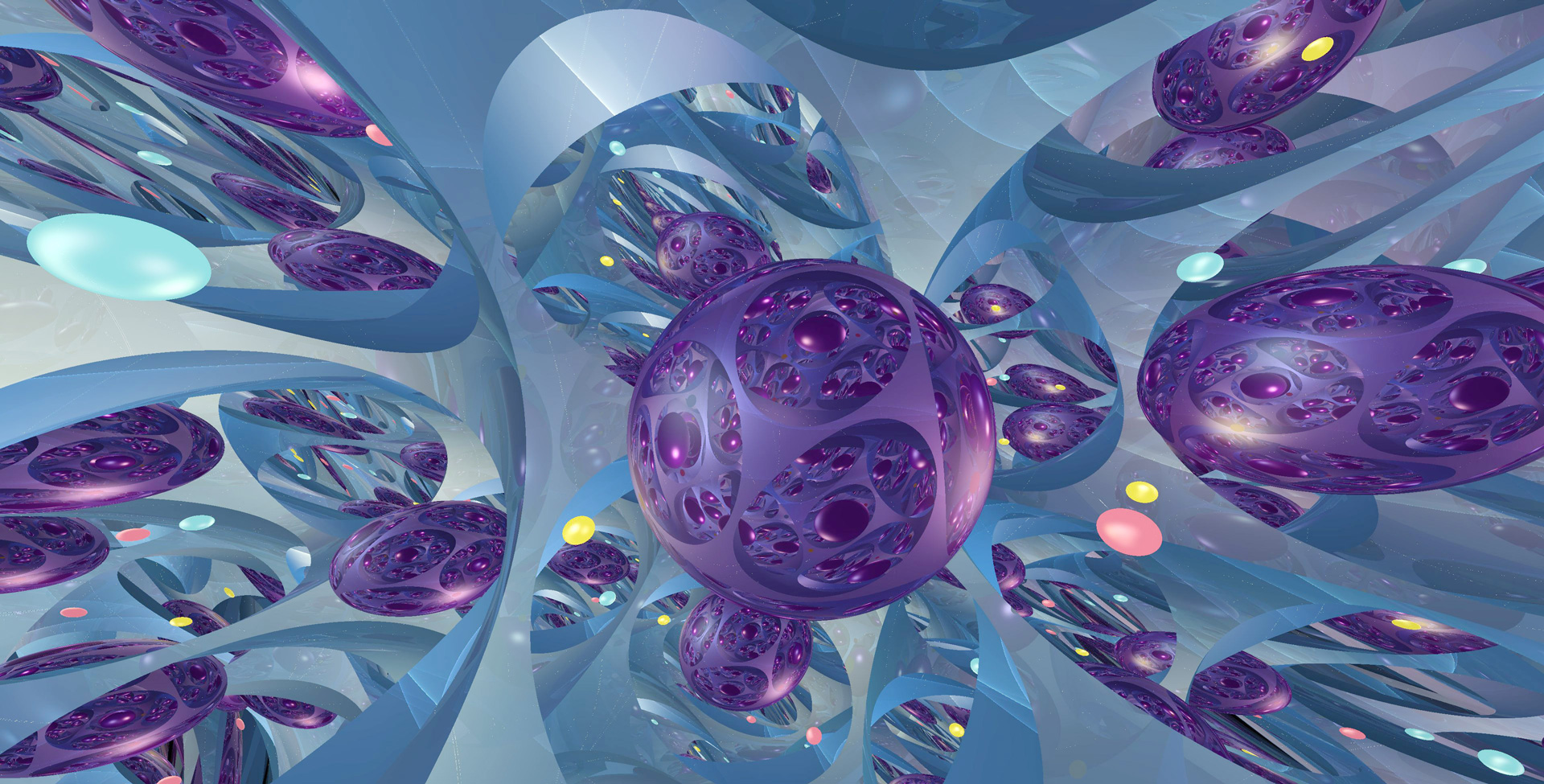}
\label{Fig:NilTorusBundle}
}\
\subfloat[Lattice of balls.]{
\includegraphics[width=0.90\textwidth]{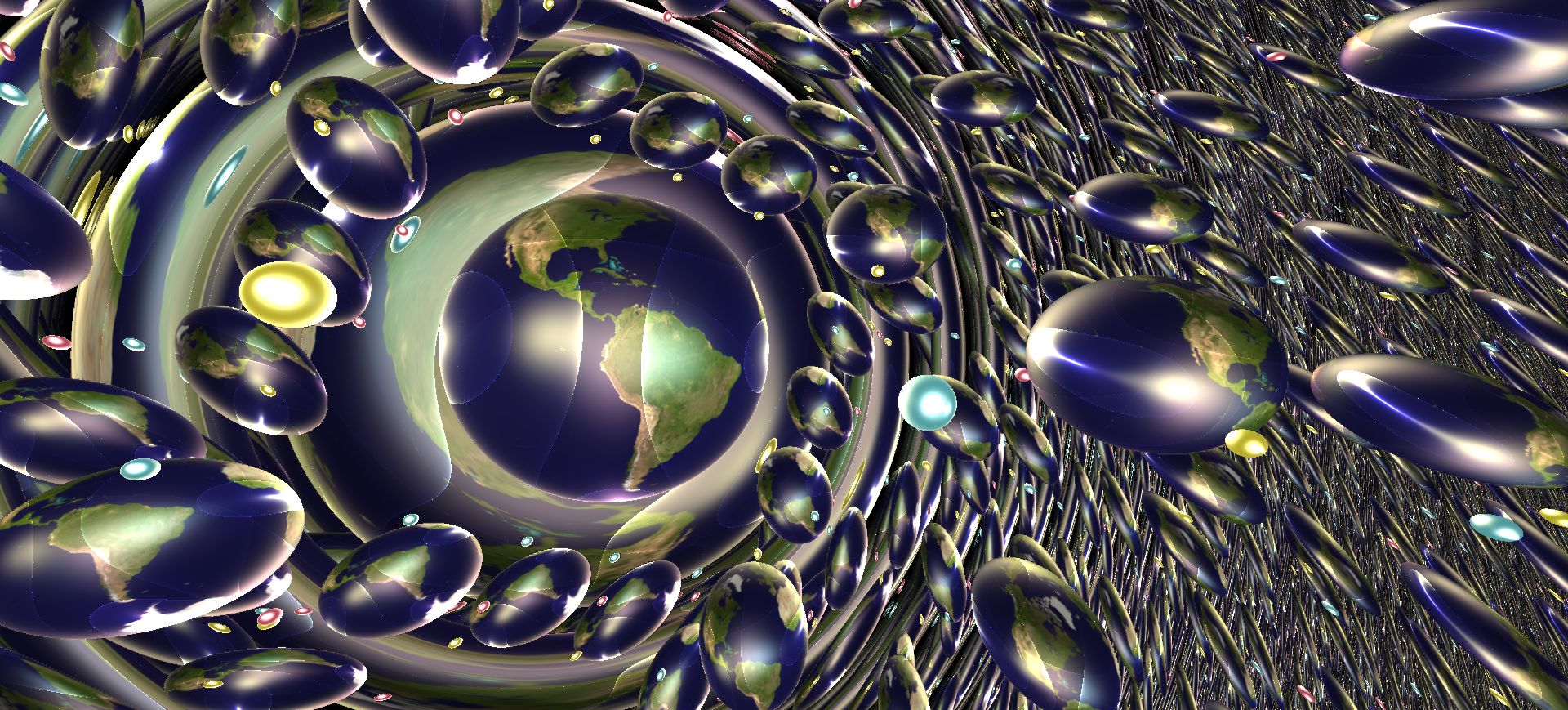}
\label{Fig:NilBalls}
}\
\subfloat[Fibers.]{
\includegraphics[width=0.90\textwidth]{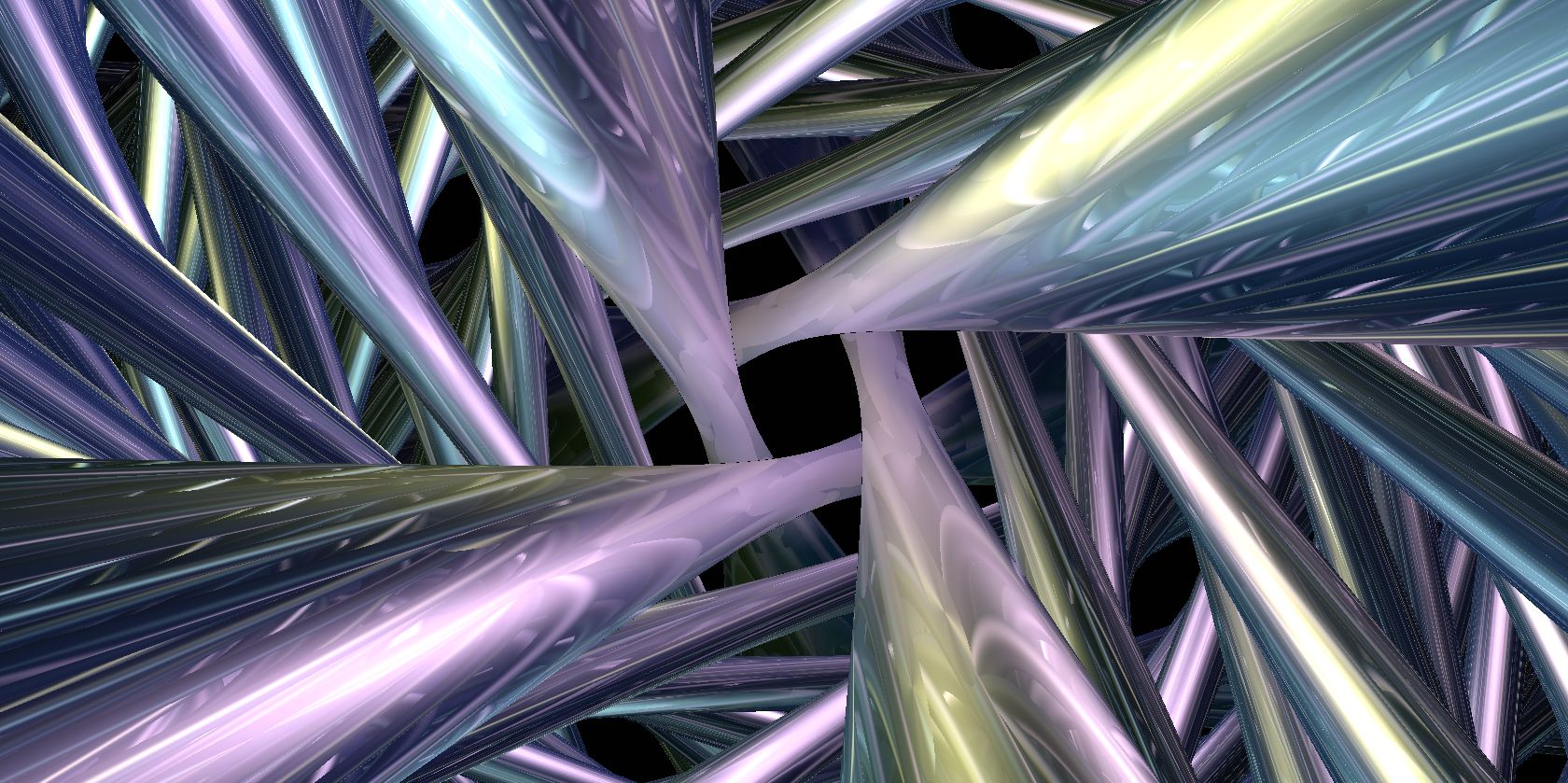}
\label{Fig:NilFibers}
}\
\caption{Nil Geometry.}
\label{Fig:NilExamples}
\end{figure}

\section{$\SLR$}
\label{Sec:SLR}
\subsection{Model}
\label{Sec:model sl2}

In order to build our model space, we view ${\rm SL}(2,\RR)$ as a circle bundle over $\HH^2$.
This construction is analogous to the Hopf fibration ${\rm SU}(2) \to S^2$.
In the spherical case, conjugation by  ${\rm SU}(2)$ (thought of as the unit quaternions) defines an action by rotations on $\RR^3$, seen as the Lie algebra of ${\rm SU}(2)$.
More precisely, this action preserves the Killing form, which in this case has signature $(0, 3)$.
In particular, after fixing a base point in the unit sphere of $\RR^3$, the orbit map defines a projection from ${\rm SU}(2)$ onto $S^2$, whose fibers are circles.

We follow the same strategy for  ${\rm SL}(2,\RR)$.
The action by conjugation of ${\rm SL}(2,\RR)$ on its Lie algebra $\mathfrak{sl}_2$ preserves the Killing form which here has signature (2,1).
The level set of this form is a model $\HH^2$.
As above, the orbit map defines a projection from ${\rm SL}(2,\RR)$ onto $\HH^2$, whose fibers are also circles.
Topologically this realizes ${\rm SL}(2,\RR)$ as a trivial bundle homeomorphic to $\HH^2 \times S^1$.
Its universal cover $\SLR$ is homeomorphic to $\HH^2 \times \RR$.
This is the description we use for our model.

We now give detailed computations.
We identify the space $\mathcal M_{2,2}(\RR)$ of $2\!\times\!2$-matrices with $\RR^4$ via the basis $E = (E_0, E_1,E_2,E_3)$ given by 
\begin{equation*}
	\begin{split}
		E_0 = \left[\begin{array}{cc}
			1 & 0 \\
			0 & 1
		\end{array}\right],\quad
		E_1 = \left[\begin{array}{cc}
			0 & 1 \\
			-1 & 0
		\end{array}\right],\\
		E_2 = \left[\begin{array}{cc}
			0 & 1 \\
			1 & 0
		\end{array}\right],\quad
		E_3 = \left[\begin{array}{cc}
			1 & 0 \\
			0 & -1
		\end{array}\right].
	\end{split}
\end{equation*}
The quadratic form $k = -\det$ is diagonal in this basis: given any point $p = [p_0, p_1, p_2, p_3]$ in $\RR^4$, we have 
\begin{equation*}
	k(p) = -p_0^2 -p_1^2 + p_2^2 + p_3^2.
\end{equation*}
In particular ${\rm GL}(2,\RR)$ and ${\rm SL}(2,\RR)$ correspond to the subsets
\begin{equation*}
	\mathcal Q_0 = \{ p \in \RR^4 \mid k(p) \neq 0\}
	\quad\text{and}\quad
	\mathcal Q = \{ p \in \RR^4 \mid k(p) = -1\}
\end{equation*}
of $\RR^4$.
We choose for the origin the point $o = [1,0,0,0]$. This corresponds to the identity.
The group law can be rewritten as follows: given a point $p = [p_0,p_1,p_2,p_3]$ in $\mathcal Q_0$, the corresponding element of ${\rm GL}(2,\RR)$ acts on $\mathcal Q_0$ as the matrix
\begin{equation*}
	\left[\begin{array}{rrrr}
		p_0 & -p_1 & p_2 & p_3 \\
		p_1 & p_0 & p_3 & -p_2 \\
		p_2 & p_3 & p_0 & -p_1 \\
		p_3 & -p_2 & p_1 & p_0
	\end{array}\right].
\end{equation*}
We endow $\mathcal Q_0$ with an ${\rm GL}(2,\RR)$-invariant riemannian metric:
\begin{align*}
	ds^2  =\ & \frac{4\beta_0(p)}{k(p)^2} \left(dp_0^2 + dp_1^2 + dp_2^2 + dp_3^2\right) \\
	& - \frac{4\beta_1(p)}{k(p)^2} \left(dp_0dp_2 - dp_1dp_3\right) - \frac{4\beta_2(p)}{k(p)^2} \left(dp_0dp_3 + dp_1dp_2\right),
\end{align*}
where
\begin{equation*}
	\left\{ \begin{split}
		\beta_0(p) & = p_0^2 + p_1^2 + p_2^2 + p_3^3 \\
		\beta_1(p) & = p_0p_2 - p_1p_3\\
		\beta_2(p) & = p_0p_3 + p_1p_2.
	\end{split} \right.
\end{equation*}
It turns out that the level sets of $k$ are totally geodesic subspaces of $\mathcal Q_0$.
The stabilizer $K < G$ of the origin $o \in \mathcal Q$ is generated by: 
\begin{itemize}
	\item \emph{rotations} $R_\alpha$ of angle $\alpha$, with matrix
	\begin{equation*}
		\left[\begin{array}{cccc}
			1 & 0 & 0 & 0 \\
			0 & 1 & 0 & 0 \\
			0 & 0 & \cos \alpha & -\sin \alpha \\
			0 & 0 & \sin \alpha & \cos \alpha
		\end{array}\right],\textrm{ and}
	\end{equation*}
	\item the \emph{flip} $F$, with matrix
	\begin{equation*}
		\left[\begin{array}{cccc}
			1 & 0 & 0 & 0 \\
			0 & -1 & 0 & 0 \\
			0 & 0 & 0 & 1 \\
			0 & 0 & 1 & 0
		\end{array}\right].
	\end{equation*}
\end{itemize}
Observe that $F \circ R_\alpha \circ F^{-1} = R_{-\alpha}$, so $K$ is isomorphic to ${\rm O}(2)$.

\medskip
The space we are really interested in is not ${\rm SL}(2, \RR)$, but its \emph{universal cover}.
Topologically, the latter is a line bundle over $\HH^2$.
The identification goes as follows.
Consider the adjoint representation of ${\rm SL}(2,\RR)$ on its Lie algebra
\begin{equation*}
	\mathfrak{sl}_2 = \{M \in \mathcal M_{2,2}(\RR) \mid {\rm Tr}(M) = 0\}.
\end{equation*}
This action preserves the Killing quadratic form 
\begin{equation*}
	\mathfrak K(M) = \frac 12 {\rm Tr}(M^2)
\end{equation*}
which has signature $(2,1)$.
Hence it induces an action by isometries on the hyperboloid model of $\HH^2$.
In our context, the Lie algebra $\mathfrak{sl}_2$ is isomorphic to the linear space $T_o \mathcal Q \subset \RR^4$ spanned by
\begin{equation*}
	e_x = -E_3,\quad e_y = E_2,\quad \text{and}\quad e_z = E_1. 
\end{equation*}
The Killing form is diagonal in this basis: if $M = xe_x + ye_y + ze_z$, then $\mathfrak K(M) = x^2 + y^2 - z^2$.
So we choose for the hyperboloid model of $\HH^2$ the set $\mathcal H$ as defined in \refsec{S2H2Models}: 
\begin{equation*}
	\mathcal H = \left\{ [x, y, z] \in T_o\mathcal Q \mid x^2 + y^2 - z^2 = -1 \ \text{and} \ z > 0 \right\}.
\end{equation*}
We define a $1$-Lipschitz, ${\rm SL}(2, \RR)$-equivariant projection $\pi \colon {\rm SL}(2, \RR) \to \HH^2$ by sending the origin $o$ to the point $[0,0,1] \in \mathcal H$. (The scaling factor four in the metric on $\mathcal Q_0$ was precisely chosen so that the best Lipschitz constant for $\pi$ is one.)
The fiber of the point $q = [x, y, z]$ is a circle parametrized as follows.
\begin{equation*}
	\pi^{-1}(q) = \left\{ S_w \zeta(q) \mid w \in [0,4\pi) \right\},
\end{equation*}
where $\zeta \colon \mathcal H \to \mathcal Q$ is the section given by 
\begin{equation*}
	\zeta(q) = \left[ \sqrt{\frac{z+1}2}, 0, \frac{x}{\sqrt{2(z+1)}}, \frac{y}{\sqrt{2(z+1)}}\right]
\end{equation*}
and $S_w$ is the transformation of $\mathcal Q$ with matrix
\begin{equation*}
	\renewcommand{\arraystretch}{1.2}
	\left[\begin{array}{rrrr}
		\cos\left(\frac w2\right) & -\sin\left(\frac w2\right) & 0 & 0 \\
		\sin\left(\frac w2\right) & \cos\left(\frac w2\right) & 0 & 0 \\
		0 & 0 & \cos\left(\frac w2\right) & \sin\left(\frac w2\right) \\
		0 & 0 & -\sin\left(\frac w2\right) & \cos\left(\frac w2\right) 
	\end{array}\right].
	\renewcommand{\arraystretch}{1}
\end{equation*}
Note that $S_w$ translates points along the fiber by an angle $w/2$, not $w$. 
This accounts for the fact that the map ${\rm SL}(2, \RR) \to {\rm SO}(2,1)$ is a two-sheeted cover.
Finally one observes that the projection 
\begin{equation*}	
	\begin{array}{lccc}
		\lambda \colon & \mathcal H \times \RR & \to & \mathcal Q \\
		& (q, w) & \mapsto & S_w \zeta(q)
	\end{array}
\end{equation*}
is a covering map, providing an identification between $\SLR$ and our model space $X = \mathcal H \times \RR$.
We call the factor $\mathcal H$ (respectively $\RR$) the \emph{horizontal} (respectively \emph{vertical} or \emph{fiber}) component of $X$.

\begin{remark}
	In practice, we adopt a slightly different point of view. We store a point $p \in X$ as a pair $(g, w) \in {\rm SL}(2,\RR) \times \RR$ where $g$ is the image of $p$ by the covering map $\lambda$ and $w$ is the fiber component of $p$.
	This representation is redundant, but allows us to go quickly back and forth between ${\rm SL}(2,\RR)$ and its universal cover.
\end{remark}

We choose as a base point of $X$ the point $\cover o = [0,0,1,0]$ which is a pre-image of $o$.
The covering map $\lambda$ induces an isomorphism between the stabilizer of $\cover o$ and the stabilizer of $o$, that is ${\rm O}(2)$.
\begin{itemize}
	\item We choose a lift $\cover R_\alpha$ of $R_\alpha$ with the following properties. It fixes the fiber component, and acts on the horizontal component as the usual rotation of $\HH^2$ by angle $\alpha$ centered at $\pi(o)$.
Beware that $R_\alpha$ is a rotation of our model space $\mathcal Q$ of ${\rm SL}(2,\RR)$ which is distinct from the element of ${\rm SL}(2,\RR)$ representing a rotation of $\HH^2$.
	\item The map $\cover F$ sending $p = [x,y,z,w]$ of $X$ to $p' = [y,x,z,-w]$ is a lift of $F$.
\end{itemize}

\subsection{Geodesic flow and parallel transport in ${\rm SL}(2, \RR)$}
\label{Sec:flow sl2}
The solution of the geodesic flow has been computed in \cite{Divjak:2009aa}.
We follow a slightly more geometric approach.

Since the covering map is a local isometry, the geodesics of $X$ are lifts of geodesics in $\mathcal Q$.
Hence we first integrate the geodesic flow in $\mathcal Q$ using Grayson's method.
We endow the tangent space $T_{\cover o}X $ with the reference frame $\cover e = (\cover e_x, \cover e_y, \cover e_w)$, where
\begin{equation*}
	\cover e_x = \frac \partial{\partial x}, \quad
	\cover e_y = \frac \partial{\partial y}, \quad \text{and}\quad
	\cover e_w = \frac \partial{\partial w}.
\end{equation*}
We write $e = (e_x, e_y, e_w)$ for its image under $d_{\cover o}\lambda \colon T_{\cover o} X \to T_o \mathcal Q$. (Note that $e_x$ and $e_y$ coincide with the previous definition.)
It follows from our choice of metric that $e$ is an orthonormal basis of $T_o\mathcal Q$.

Let $\gamma \colon \RR \to \mathcal Q$ be a geodesic in ${\rm SL}(2,\RR)$ and let $T(t) \colon T_{\gamma(0)}\mathcal Q \to T_{\gamma(t)}\mathcal Q$ be the corresponding parallel-transport operator.
As in Sections~\ref{Sec:GeodesicFlow - Grayson} and \ref{Sec:Parallel transport - Grayson}, we define paths $u \colon \RR \to T_o\mathcal Q$ and ${Q \colon \RR \to {\rm SO}(3)}$ by the relations
\begin{equation*}
	\begin{split}
		\dot \gamma(t) & = d_oL_{\gamma(t)} u(t), \textrm{ and} \\
		T(t) \circ d_oL_{\gamma(0)} & = d_oL_{\gamma(t)} Q(t).
	\end{split}
\end{equation*}
After some computation, Equations (\ref{Eqn:GraysonMethodFlowSphere}) and (\ref{Eqn:ParallelTransportLinEq}) can be written relative to the basis $e$ as follows
\begin{equation*}
	\left\{ \begin{split}
		\dot u_x & =  2 u_y u_w \\
		\dot u_y & = -2 u_x u_w \\
		\dot u_w & = 0
	\end{split}\right.
\end{equation*}
and 
\begin{equation*}
	\dot Q + BQ = 0,
	\quad \text{where} \quad
	B = \frac 12 \left[\begin{array}{ccc}
		0 & -3u_w & -u_y \\
		3u_w & 0 & u_x \\
		u_y & -u_x & 0 
	\end{array}	\right].
\end{equation*}
For the initial condition $u(0) = a \cos(\alpha) e_x +  a \sin(\alpha) e_y + c e_w$, where $a \in \RR_+$ and $c \in \RR$ satisfy $a^2 + c^2 = 1$, one gets 
\begin{equation*}
	u(t) = a \cos(\alpha - 2ct)e_x + a \sin (\alpha - 2ct)e_y + ce_w.
\end{equation*}
In order to calculate the expression for $Q$, we follow the strategy detailed above and obtain
\begin{equation*}
	Q(t) = dR_\alpha e^{-2ct U_1} P e^{\frac 12tU_2} P^{-1}dR_\alpha^{-1}, \quad \forall t \in \RR,
\end{equation*}
where
\begin{equation*}
	U_1 = 
	\left[\begin{array}{ccc}
	    0 & -1 & 0  \\
	    1 & 0 & 0  \\
	    0 & 0 & 0 
	\end{array}\right],	
	\quad
	U_2 =
	\left[\begin{array}{ccc}
	    0 & 0  & 0 \\
	    0 & 0 & -1 \\
	    0 & 1 & 0 
	\end{array}\right], \quad
\end{equation*}
and 

\begin{equation*}
	dR_\alpha = \left[\begin{array}{ccc}
	    \cos\alpha & -\sin\alpha & 0 \\
	    \sin\alpha & \cos\alpha & 0 \\
	    0 & 0 & 1  \\
	\end{array}\right],
	\quad
	P = 
	\left[\begin{array}{ccc}
	    a & 0 & -c \\
	    0 & 1 & 0 \\
	    c & 0 & a  \\
	\end{array}\right].
\end{equation*}
Note that $dR_\alpha \colon T_o \mathcal Q \to T_o \mathcal Q$ is the differential at $o$ of the rotation $R_\alpha$, written in the frame $e = (e_x, e_y, e_w)$.

Let us now move back to the original geodesic $\gamma \colon \RR \to X$.
\refeqn{GraysonMethodPullBack} becomes $\dot \gamma(t) = A(t)\gamma(t)$, where 
\begin{equation*}
	A(t) = \frac 12
	\left[\begin{array}{cccc}
		0 & -u_w & u_x & u_y \\
		u_w & 0 & -u_y & u_x \\
		u_x & -u_y & 0 & u_w \\
		u_y & u_x & - u_w & 0
	\end{array}\right].
\end{equation*}
Using a change of variables, one can reformulate the previous equation into a first-order differential system with \emph{constant} coefficients that we integrate with standard methods.
We obtain that the geodesic $\gamma$, such that $\gamma(0) = o$ and $\dot \gamma(0) = a \cos(\alpha) e_x +  a \sin(\alpha) e_y + c e_w$, decomposes  (up to a rotation) as a product of two one-parameter subgroups:
\begin{equation}
\label{Eqn:sl2_Geodesic_1PSubg}
	\gamma(t) = R_\alpha \left(\eta(t) \ast \xi(t)\right).
\end{equation}
As before, $R_\alpha$ is the rotation of $\mathcal Q$ by angle $\alpha$ and $\ast$ is group multiplication in ${\rm SL}(2, \RR)$.
The \emph{spin} factor $\xi \colon \RR \to {\rm SL}(2, \RR)$ represents a rotation of $\HH^2$ fixing the origin $\pi(o) \in \mathcal H$. It can be written in $\mathcal Q$ as
\begin{equation*}
	\xi(t) = [\cos(ct), \sin(ct), 0, 0].
\end{equation*}
The \emph{translation} factor $\eta \colon \RR \to {\rm SL}(2, \RR)$ can have three forms, corresponding to the three types of isometries of $\HH^2$.
For simplicity we let $\kappa = \sqrt{|c^2 - a^2|}$.
\begin{itemize}
	\item If $c > a$, then $\eta$ is an \emph{elliptic} transformation, given in $\mathcal Q$ by 
	\begin{equation*}
		\eta(t) = \left[\cos\left(\frac{\kappa t}2\right), - \frac c\kappa \sin \left(\frac{\kappa t}2\right), \frac a \kappa \sin\left(\frac{\kappa t}2\right), 0\right].
	\end{equation*}
	\item If $c = a$, then $\eta$ is a \emph{parabolic} transformation, given in $\mathcal Q$ by 
	\begin{equation*}
		\eta(t) = \left[1, - \frac t {2\sqrt 2} , \frac t {2\sqrt 2}, 0\right].
	\end{equation*}
	\item If $c < a$, then $\eta$ is a \emph{hyperbolic} transformation, given in $\mathcal Q$ by 
	\begin{equation*}
		\eta(t) = \left[\cosh\left(\frac{\kappa t}2\right), - \frac c\kappa \sinh \left(\frac{\kappa t}2\right), \frac a \kappa \sinh\left(\frac{\kappa t}2\right), 0\right].
	\end{equation*}
\end{itemize}

\subsection{Passing to the universal cover}
\label{Sec:flow sl2 tilde}
Let us now consider the geodesic $\cover \gamma$ in the universal cover $X = \mathcal H \times \RR$ starting at $\cover o$ with initial velocity $a \cos(\alpha)\cover e_x + a \sin(\alpha)\cover e_y+c \cover e_w$.
This is a lift of the geodesic $\gamma$ computed above.
The horizontal component of $\cover \gamma$ is obtained as the image of $\gamma$ under the projection $\pi \colon {\rm SL}(2,\RR) \to \HH^2$.
Note that the spin factor $\xi(t)$ fixes the base point $\pi(o) \in \mathcal H$.
Moreover, the rotation $R_\alpha$ of $\mathcal Q$ induces (via the projection $\pi$) the rotation $r_\alpha$ of $\HH^2$ by angle $\alpha$ centered at $\pi (o)$.
Consequently, $\pi \circ \gamma(t)$ is the image under $r_\alpha$ of one of the following points, depending on whether $c > a$, $c = a$, or $c < a$ respectively:
\begin{equation*}
	\renewcommand{\arraystretch}{1.3}
	\left[\begin{array}{r}
		\frac {2a}\kappa \sin\left(\frac{\kappa t}2\right) \cos\left(\frac{\kappa t}2\right) \\
		 -\frac {2ac}{\kappa^2} \sin^2\left(\frac{\kappa t}2\right)\\
		 1 + \frac{2a^2}{\kappa^2} \sin^2\left(\frac{\kappa t}2\right)
	\end{array}\right], \quad
	\left[\begin{array}{r}
		\frac{\sqrt 2}2 t\\
		 -\frac 14t^2 \\
		 1 + \frac 14t^2
	\end{array}\right], \quad
	\left[\begin{array}{r}
		\frac {2a}\kappa \sinh\left(\frac{\kappa t}2\right) \cosh\left(\frac{\kappa t}2\right) \\
		 -\frac {2ac}{\kappa^2} \sinh^2\left(\frac{\kappa t}2\right) \\
		 1 + \frac{2a^2}{\kappa^2} \sinh^2\left(\frac{\kappa t}2\right)
	\end{array}\right].
	\renewcommand{\arraystretch}{1}
\end{equation*}
These are parametrizations of orbits under the one-parameter subgroups above. Their images are a circle, a horocycle, and an equidistant curve to a geodesic, respectively. 

In order to compute the fiber component of $\cover \gamma$ it is convenient to introduce \emph{cylindrical coordinates} on $\mathcal Q$.
Given $\rho \in \RR_+$, $\theta \in [0, 2\pi)$ and $w \in [0, 4\pi)$, the point of $\mathcal Q$ with cylindrical coordinates $[\rho, \theta, w]$ is 
\begin{equation*}
	\renewcommand{\arraystretch}{1.3}
	\left[\begin{array}{l}
		\cosh\left(\frac\rho2\right) \cos\left(\frac w2\right) \\
		\cosh\left(\frac\rho2\right) \sin\left(\frac w2\right) \\
		\sinh\left(\frac\rho2\right) \cos\left(\theta - \frac w2\right) \\
		\sinh\left(\frac\rho2\right) \sin\left(\theta - \frac w2\right)
	\end{array}\right].
	\renewcommand{\arraystretch}{1}
\end{equation*}
This choice has been made so that the projection $\pi$ from ${\rm SL}(2,\RR)$ (in cylindrical coordinates) to $\HH^2$ (in polar coordinates) is given by ${[\rho, \theta, w] \mapsto [\rho, \theta]}$.
In view of the expression for $\gamma$ and its projection onto $\mathcal H$, 
we may calculate an expression for the fiber component $w(t)$ of $\cover \gamma(t)$. This calculation is greatly simplified by the use of polar coordinates. We obtain
\begin{equation*}
	w(t) = 2 c t + 2 \phi(t)
\end{equation*}
where $\phi(t)$ is characterized by 
\begin{equation*}
	\tan \phi(t) = \left\{
	\begin{split}
		- \frac c\kappa \tan\left(\frac{\kappa t}2\right), & \ \text{if}\ c > a \\
		- \frac t{2\sqrt 2}, & \ \text{if}\ c = a \\
		- \frac c\kappa \tanh\left(\frac{\kappa t}2\right), & \ \text{if}\ c < a
	\end{split}
	\right.
\end{equation*}
Observe that if $c \leq a$, then $\phi(t) \in (-\pi/2, \pi/2)$.
Therefore its value can be computed from the above equation using the standard $\arctan$ function.
On the other hand, if $c > a$, then the geodesic $\cover \gamma$ spirals, and the value of $\phi(t)$ needs to be adjusted by the correct multiple of $2\pi$.

Note that the covering map $\lambda \colon X \to \mathcal Q$ is a local isometry.
Hence the parallel transport operator in $X$ can be obtained by lifting the parallel transport operator in $\mathcal Q$.
In view of Grayson's method, this operator is encoded by a local path $\cover Q \colon \RR \to {\rm SO}(3)$, see \refsec{Parallel transport - Grayson}.
The identification relies on a choice of a preferred frame $\cover e$ in the tangent space $T_{\cover o}X$ at the origin.
By construction $\lambda$ is equivariant with respect to the projection $\SLR \to {\rm SL}(2, \RR) $.
Moreover it maps $\cover e$ to our preferred frame $e$ in $T_o\mathcal Q$.
Thus $\cover Q$ and $Q$ actually coincide.

\begin{figure}[htbp]
\centering
\subfloat[One half-space.]{
	\includegraphics[width=0.3\textwidth]{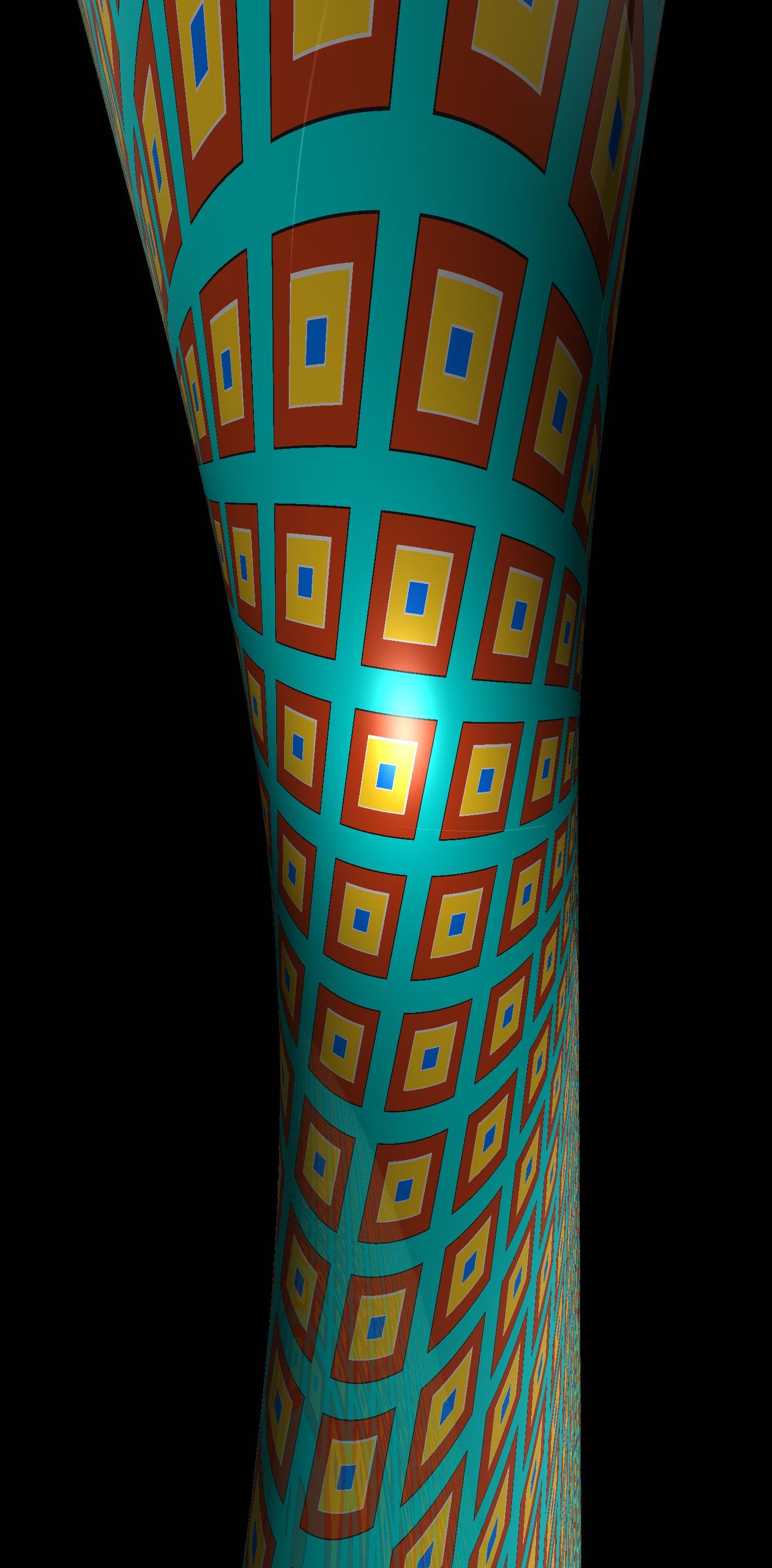}
}
\subfloat[One half-space.]{
	\includegraphics[width=0.3\textwidth]{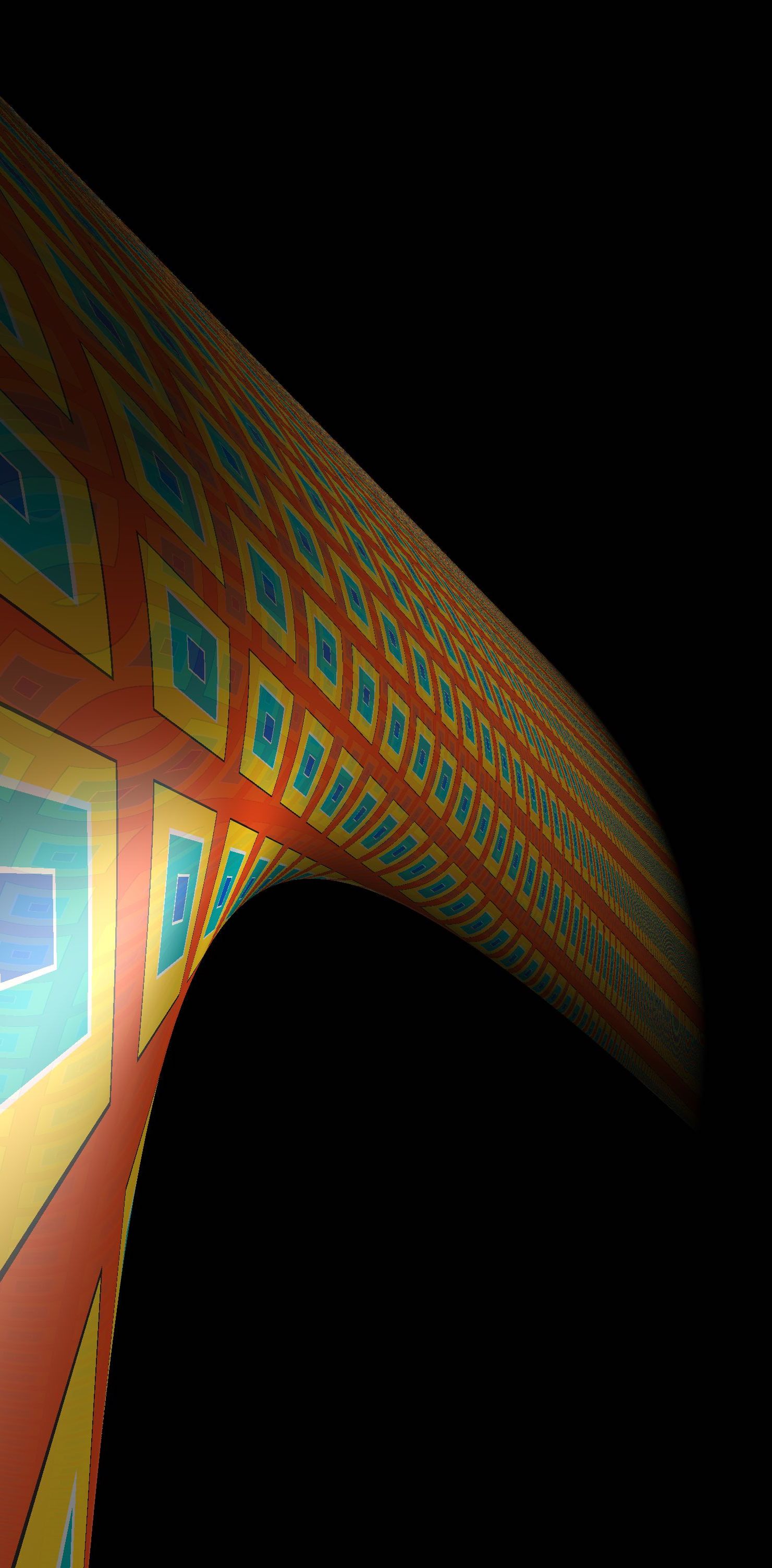}
}
\subfloat[Two half-spaces.]{
	\includegraphics[width=0.3\textwidth]{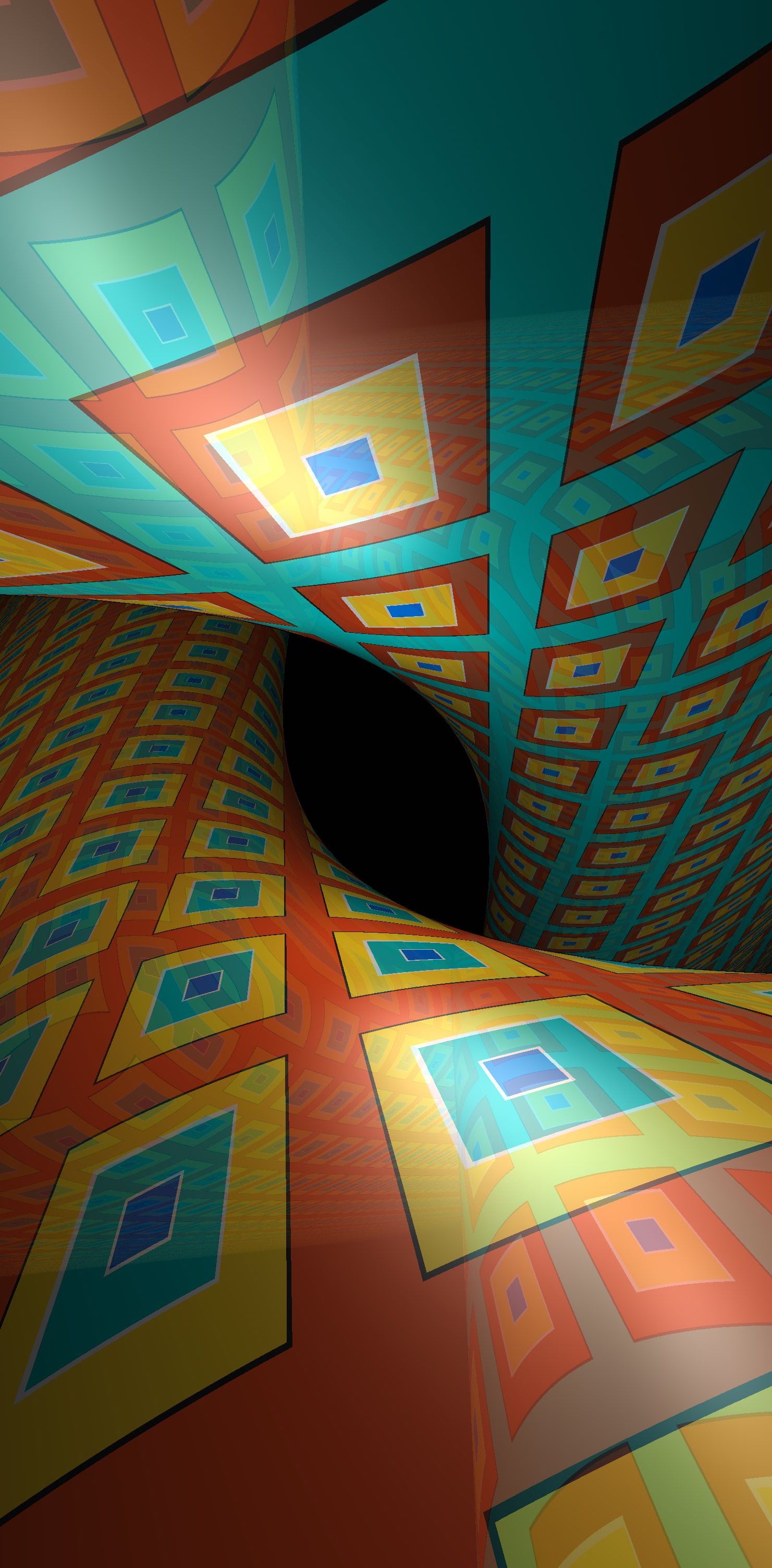}
}
\caption{Vertical half-spaces in $\SLR$ geometry.}
\label{Fig:SLR_VertPlanes}
\end{figure}

\subsection{Distance to a vertical object}
Exactly as in Nil, we say that an object $Z \subset X$ is \emph{vertical} if it is the pre-image of the projection ${\pi\circ\lambda}\colon X \to \HH^2$ of a non-empty subset $S$ of $\HH^2$.
In this situation, for any point $p \in X$ we have
\begin{equation*}
	\dist_X (p, Z) = \dist_{\HH^2}(\pi \circ \lambda(p), S).
\end{equation*}

\reffig{SLR_VertPlanes} shows pre-images of a half-space with geodesic boundary in $\HH^2$. The boundary of each is patterned with a square grid following the induced euclidean metric on the plane. The grid has side-length $1/2$.

\subsection{Exact distance and direction to a point}
\label{Sec:ExactGeodesicsSLR}
The strategy to compute the distance and direction from the origin to an arbitrary point $p$ with cylindrical coordinates $[\rho, \theta, w]$ is similar to the strategy used in Nil.
Because of the flip symmetry, we may assume that $w \geq 0$.
First assume that $\rho > 0$.
Using the solution of the geodesic flow, we observe that the geodesics $\cover \gamma$ joining $\cover o$ to $p$ are in one-to-one correspondence with the zeros of a function
\begin{equation}
	\phi \mapsto \chi_{\rho, w}(\phi).
\end{equation}
We define this function in \reffig{ChiSLR}; see \reffig{graph chi sl2} for its graph.

\thispagestyle{empty}
\begin{landscape}
        \begin{figure}[htbp]
        \small
        \renewcommand{\arraystretch}{1.5}
	\begin{equation*}
		\chi_{\rho,w}(\phi) = 
		\left\{ 
			\begin{array}{l}
				\displaystyle
				- \frac 12 w  + \phi - 2 \tan \phi  \frac{\cosh(\rho/2) }{\sqrt{\sinh^2(\rho/2) - \tan^2\phi}} \arctanh\left( \frac{\sqrt{\sinh^2(\rho/2) - \tan^2\phi}}{\cosh(\rho/2) }\right),\\
				\displaystyle
				\hspace{0.5\linewidth} \text{if}\ \phi > -\frac\pi2 \ \text{and}\ |\tan \phi | < \sinh(\rho/2) \\%
				\\%
				 \displaystyle
				 - \frac 12 w  + \phi - 2 \tan\phi, \\
				 \displaystyle
				\hspace{0.5\linewidth} \text{if}\ \phi > -\frac\pi2 \ \text{and}\ |\tan\phi| = \sinh(\rho/2) \\%
				\\%

				\displaystyle
				- \frac 12 w  + \phi - 2 \tan \phi  \frac{\cosh(\rho/2) }{\sqrt{\tan^2\phi - \sinh^2(\rho/2) }} \left(\arctan\left( \frac{\sqrt{\tan^2\phi - \sinh^2(\rho/2)}}{\cosh (\rho/2) }\right) - \sign(\tan\phi) \left\lfloor \frac 12 - \frac \phi \pi \right\rfloor \pi \right), \\
				\displaystyle
				\hspace{0.5\linewidth} \text{if}\  |\tan\phi| > \sinh(\rho/2) \ \text{and}\ \pi \neq -\frac\pi2\mod \pi \\%
				\displaystyle
				- \frac 12 w  + \phi - 2 \cosh(\rho/2), \\%
				\displaystyle
				\hspace{0.5\linewidth} \text{if}\  \phi =-\frac\pi2\mod \pi 
			\end{array}
		\right.
	\end{equation*}
	\normalsize
	\caption{The map $\chi_{\rho, w}$. The first regime corresponds to geodesics with a hyperbolic translation factor, the second to geodesics with a parabolic translation factor, and the third and fourth to geodesics with an elliptic translation factor. In \reffig{graph chi sl2}, these are drawn in red, green, and blue respectively.}
	\label{Fig:ChiSLR}
	\renewcommand{\arraystretch}{1}
	\end{figure}
\end{landscape}

Observe that along a given geodesic $\cover\gamma$, the angle $\phi$ is a decreasing function of the time parameter $t$.
Said differently, when $\cover\gamma$ is moving up in the fiber direction, then its projection in $\HH^2$ turns clockwise.
Hence the domain of $\chi_{\rho, w}$ is contained in $\RR_-$. Moreover $\chi_{\rho, w}$ is decreasing around $\phi = 0$.

\begin{figure}[htbp]
\centering
\includegraphics[width=\textwidth]{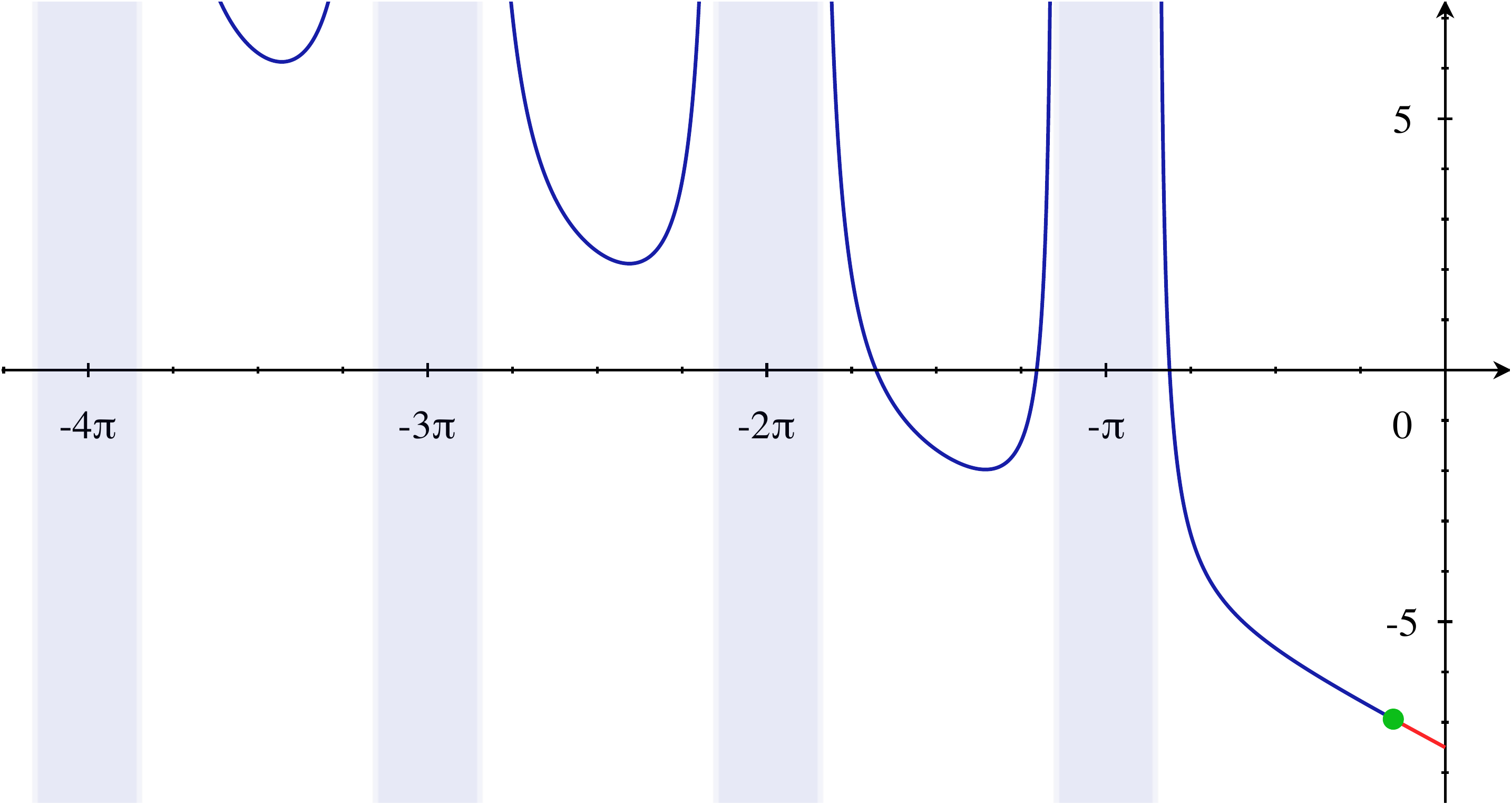}
\caption{The graph of the function $\chi_{\rho, w}$ for $\rho = 1$ and $w = 15$.
The green dot corresponds to a geodesic whose translation part is parabolic.
It separates the geodesics whose translation part are elliptic (dark blue) from those that are hyperbolic (red).
The light blue strips indicate the values of $\phi$ for which $\chi_{\rho,w}$ is not defined.
There are exactly three geodesics joining the origin to any point $p$ with cylindrical coordinates $[1,\theta, 15]$.}
\label{Fig:graph chi sl2}
\end{figure}

As in \refsec{ExactGeodesicsNil}, 
we compute the zeros of $\chi_{\rho,w}$ using Newton's method, and thus calculate the lighting pairs $\calL_\cover{o}(p)$.

Assume now that $\rho = 0$.
The path $\gamma(t) = [0,0,1,t]$ is a geodesic from $\cover o$ to $p$ with initial direction $v = \cover e_w$ and length $t = w$.
If $2n \pi \leq w < 2n\pi + 2\pi$, for some integer $n \geq 1$, then $\cover o$ and $p$ are joined by $n$ other rotation-invariant families of geodesics $\{\gamma_{1,\alpha}\}, \dots, \{\gamma_{n,\alpha}\}$, where $\alpha$ runs over $[0, 2\pi)$.
Each geodesic in the $k$th family has length 
\begin{equation*}
	t_{k,\alpha} = 2k \pi\sqrt{\frac 12 \left(\frac w{2k \pi}+1\right)^2 - 1}.
\end{equation*}
Moreover, the initial direction at the origin is characterized by 
\begin{equation*}
	\begin{split}
		d \cover R_\alpha^{-1} v_{k, \alpha} 
		& = v_{k,0} \\
		& = \sqrt{\frac{(w +2k\pi)^2 - (4k\pi)^2}{2(w +2k\pi)^2 - (4k\pi)^2}} \cover e_x + \frac {w + 2k \pi}{\sqrt{2(w+2k\pi)^2 - (4k\pi)^2}} \cover e_w.
	\end{split}
\end{equation*}

\subsection{Distance underestimator for a ball.}
\label{Sec:distance underestimator SL2}
As we explained in \refsec{model sl2}, $X \cong \SLR$ is a (metrically) twisted line bundle over $\HH^2$. 
As a subset of $\RR^4$, our model for $\SLR$ is identical to our model for $Y = \HH^2 \times \EE$ (see \refsec{Product}). This gives an identification (of course, not an isometry) between $X$ and $Y$, which we use to approximate distances in $X$ as follows. 

\begin{lemma}
\label{Lem:SLR-H2xE}
	For every point $p \in X$, we have 
	\begin{equation*}
		\dist_X(o,p) \leq \dist_Y (o,p) \leq 2 \dist_X(o,p).
	\end{equation*}
\end{lemma}

\begin{proof}
	Consider an arc length parametrized geodesic ${\gamma \colon [0, \ell] \to X}$ joining $o$ to $p$.
	We write $L_Y(\gamma)$ for its length in $Y$.
	A computation shows that $L_Y(\gamma) \leq 2 \ell$.
	Hence $ \dist_Y (o,p) \leq L_Y(\gamma) \leq 2 \dist_X(o,p)$.
	Second,
	a similar calculation shows that the arc length parametrized geodesic $\gamma'$ of $Y$ joining $o$ to $p$ is still parametrized by arc length when viewed as a path in $X$.
	Consequently $\dist_X(o,p) \leq L_X(\gamma') \leq \dist_Y(o,p)$.
\end{proof}
\begin{remark}
Note that the proof here relies on the fact that these geodesics begin at the origin, $o$. The result does not hold for general geodesics.
\end{remark}

As in Nil, we use this observation to construct a distance underestimator $\sigma' \from X \to \RR$ to render a ball of radius $r$ centered at $o$, as follows. Let 
\begin{equation*}
	\sigma'(p) = \left\{
	\begin{split}
		\sigma(p) - r, & \ \text{if}\ \sigma(p) > r + \eta \\
		2\sigma(p) - r, & \ \text{if}\ \sigma(p) < 2 (r - \eta) \\
		\dist(o,p) - r, & \ \text{otherwise,}
	\end{split}
	\right.
\end{equation*}
where 
\begin{equation*}
	\sigma(p) = \frac 12 \sqrt{ \arccosh^2\left( z^2-x^2 - y^2\right) + w^2}
\end{equation*}
is half the distance from the origin to $p$ in $Y$, and $\eta > 0$ is a constant that is much larger than the threshold $\epsilon$ used to stop the ray-marching algorithm.
In the last case of $\sigma'$, the exact distance is computed numerically as explained in \refsec{ExactGeodesicsSLR}.
We use this distance underestimator to render the balls in 
\reffig{LineOfBallsInSLR}. Compare with \reffig{NilAllDirections}, which shows a line of balls in Nil.

\begin{figure}[htbp]
\centering
\subfloat[Looking along the fiber.]{
	\includegraphics[width=0.3\textwidth]{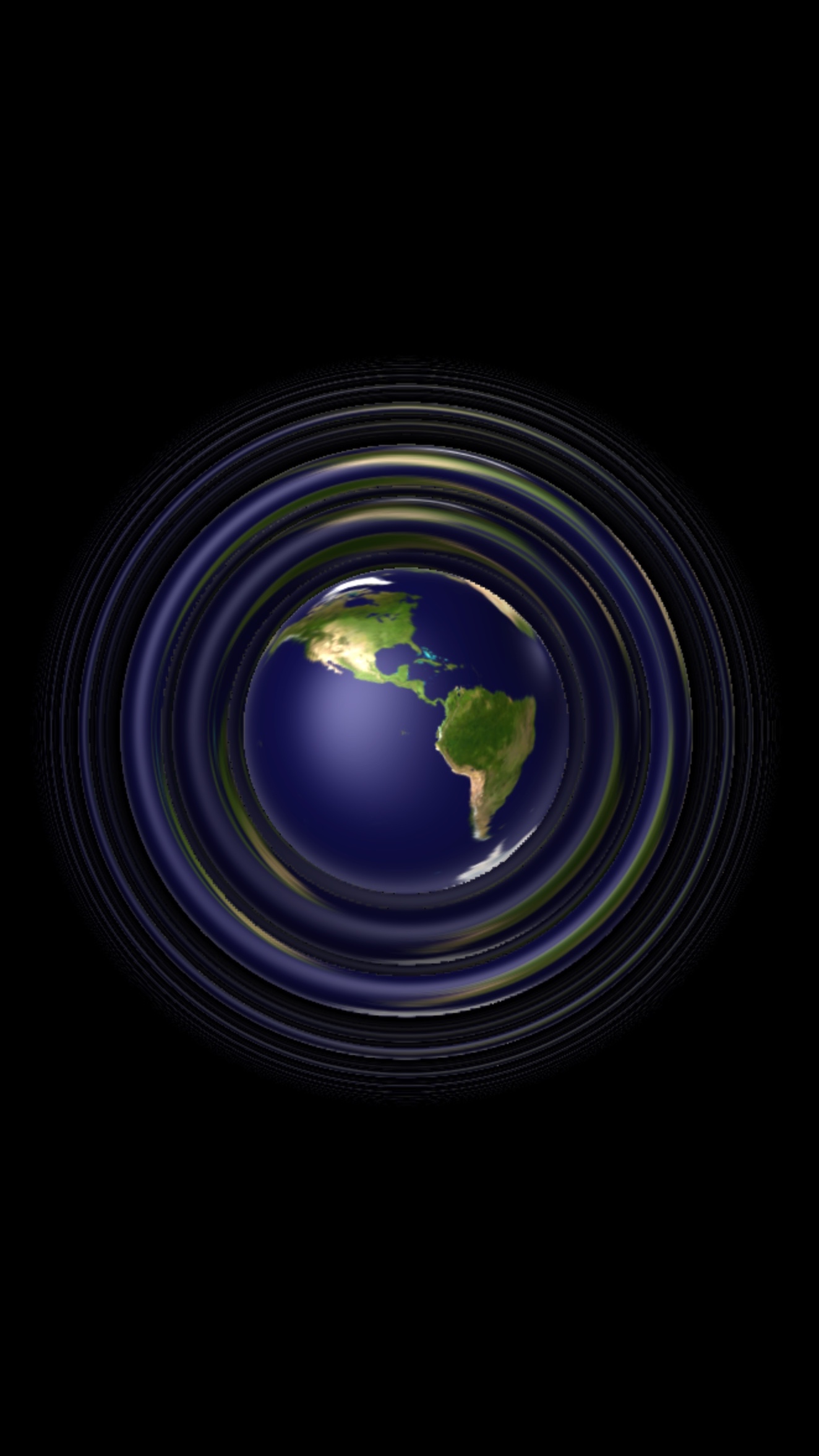}
}
\subfloat[Looking near the fiber direction.]{
	\includegraphics[width=0.3\textwidth]{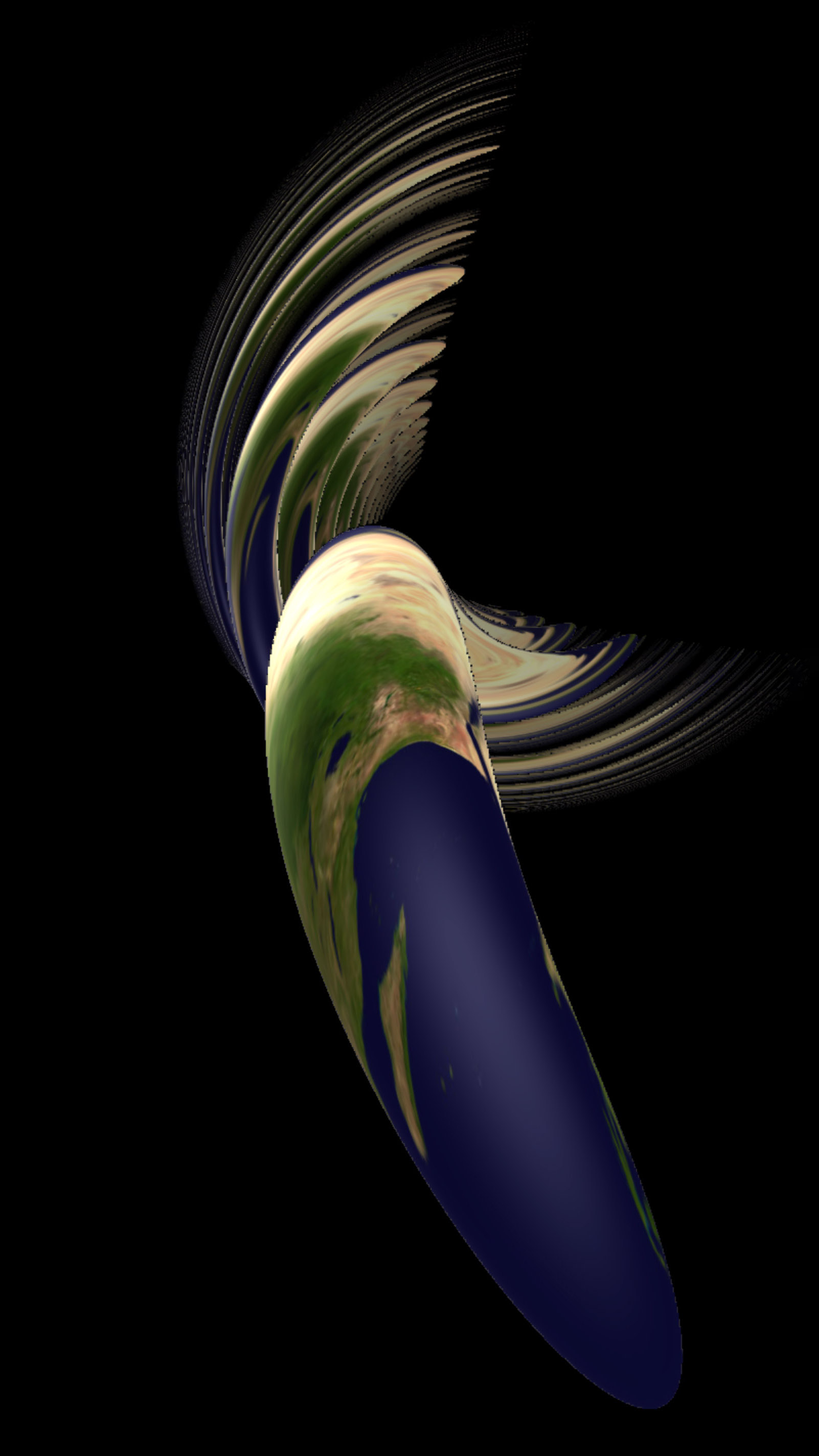}
}
\subfloat[A view further from the fiber direction.]{
	\includegraphics[width=0.3\textwidth]{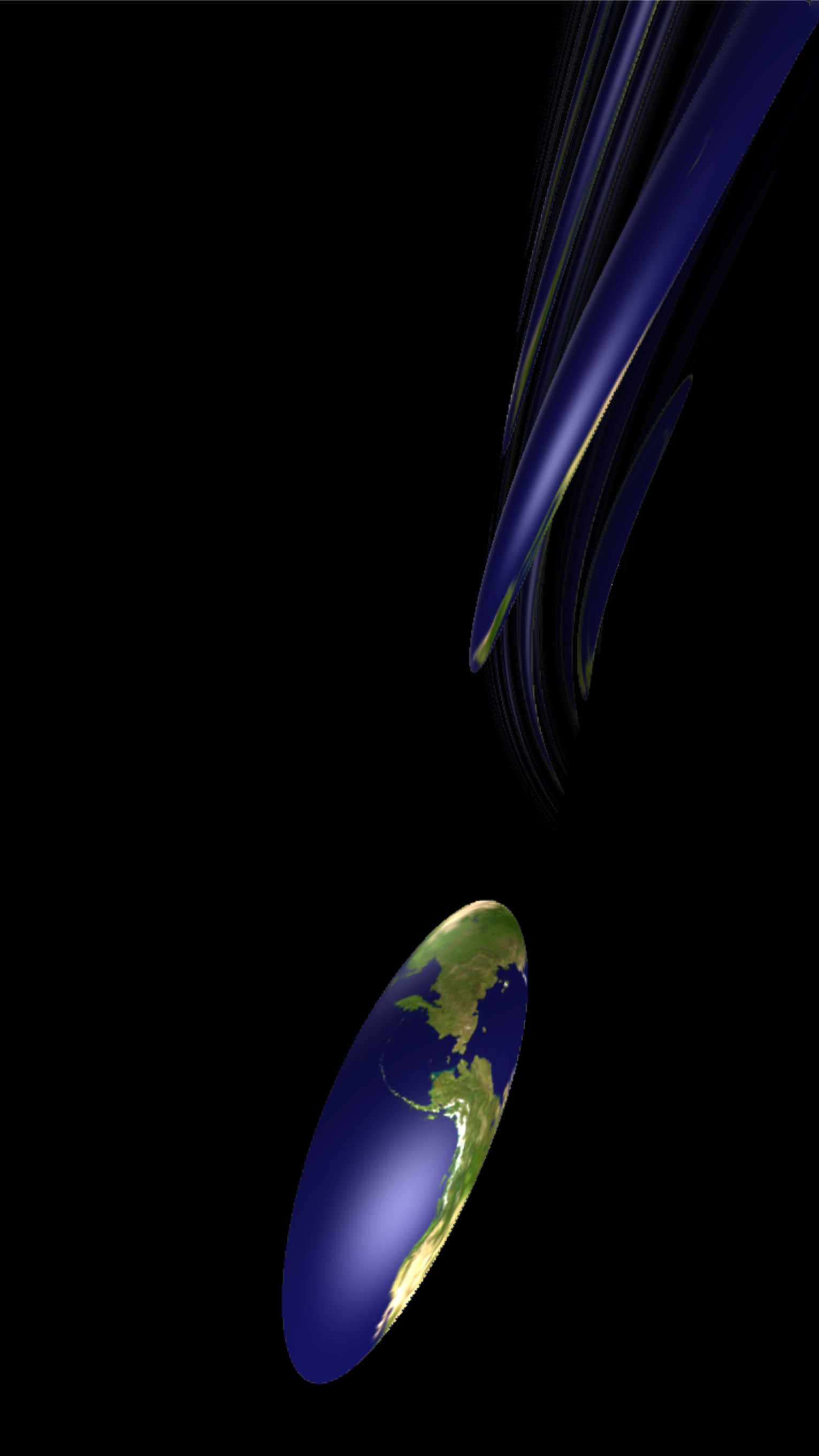}
}
\caption{A line of unit balls spaced every $2\pi$ along the fiber direction in $\SLR$. (Equivalently, 
a single ball of radius one in $\mathrm{SO}(2,1)$.)}
\label{Fig:LineOfBallsInSLR}
\end{figure}

\subsection{Creeping to horizontal half-spaces}
\label{Sec:Creepingw=0Plane}
As for Nil in \refsec{CreepingNilz=0Plane}, we can use a version of creeping to draw pictures of ``horizontal'' half-spaces. For example, in \reffig{HorizontalPlaneSLR} we draw the half-space $w\leq0$, the boundary $\HH^2$ patterned with equidistant curves to a geodesic (in white).

\begin{figure}[htbp]
\centering
\subfloat[$d=2$]{
	\includegraphics[width=0.3\textwidth]{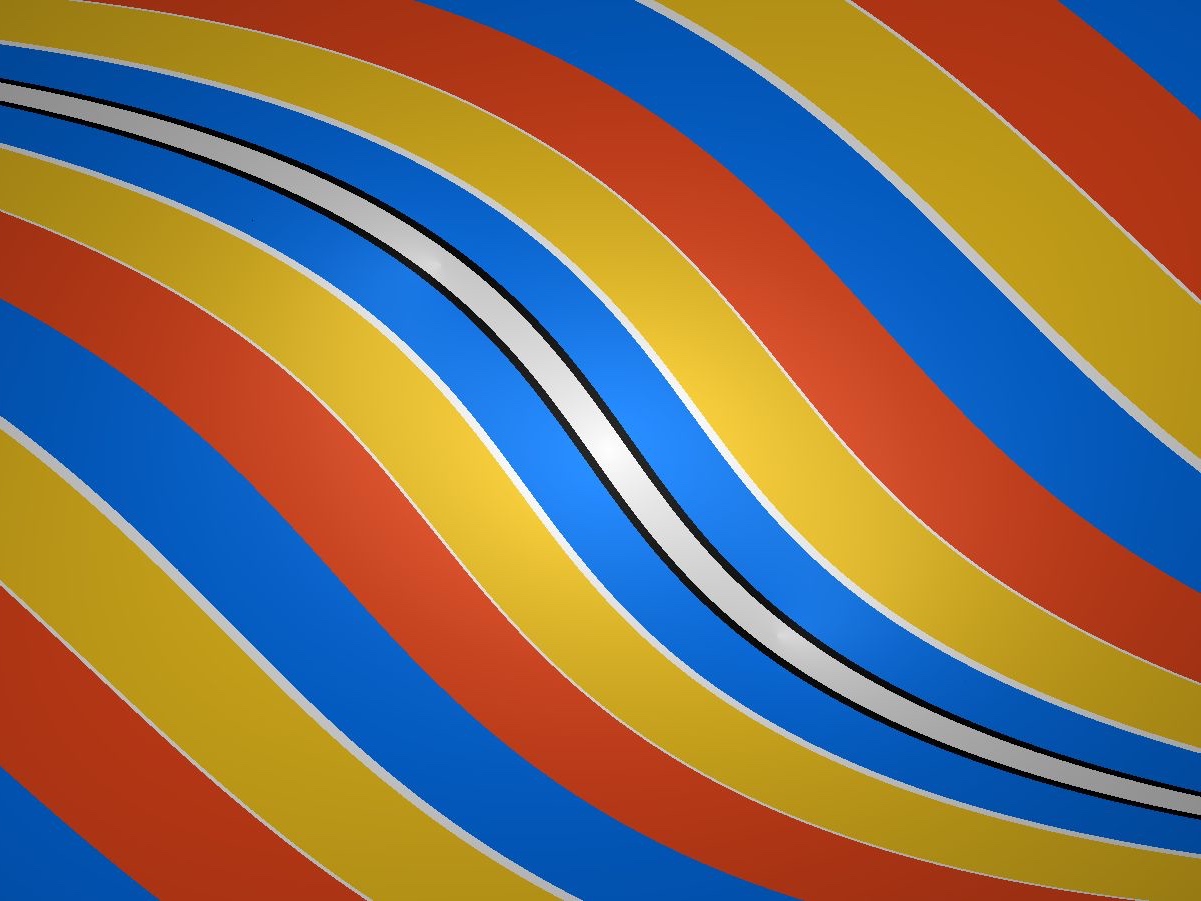}
}
\subfloat[$d=4$]{
	\includegraphics[width=0.3\textwidth]{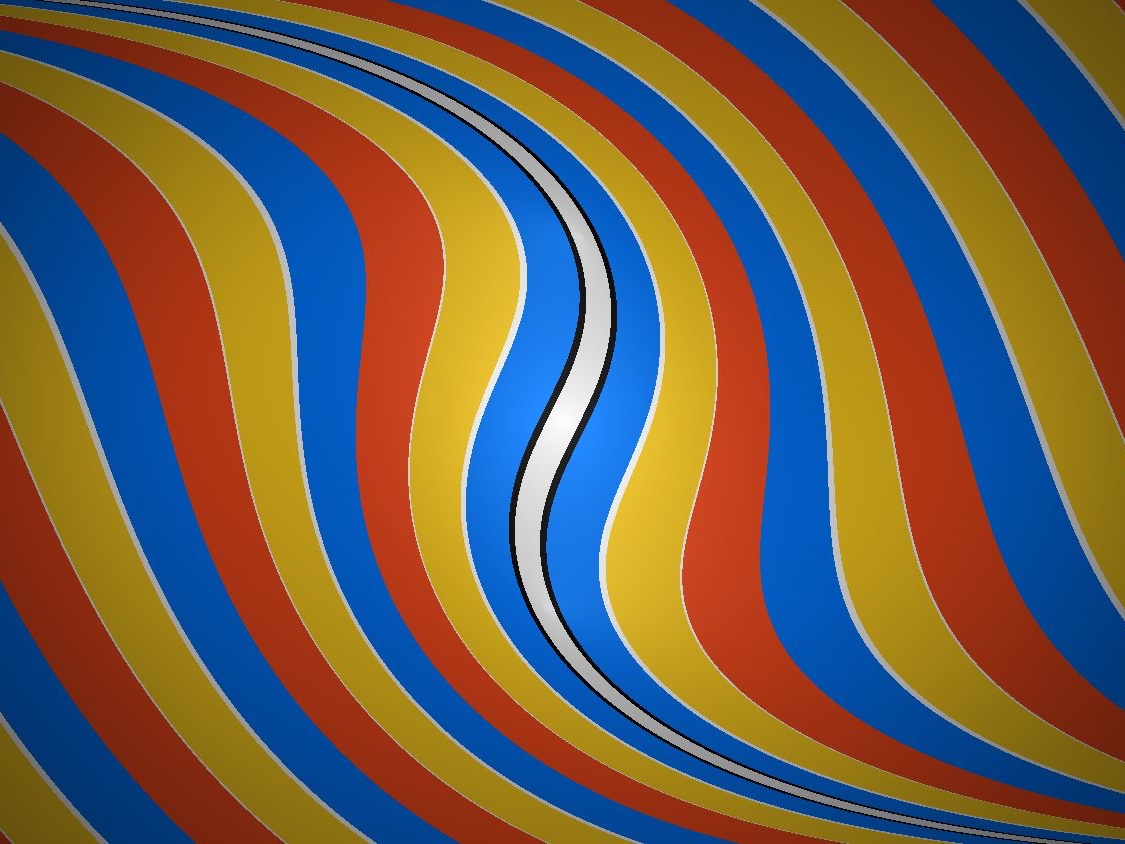}
}
\subfloat[$d=6$]{
	\includegraphics[width=0.3\textwidth]{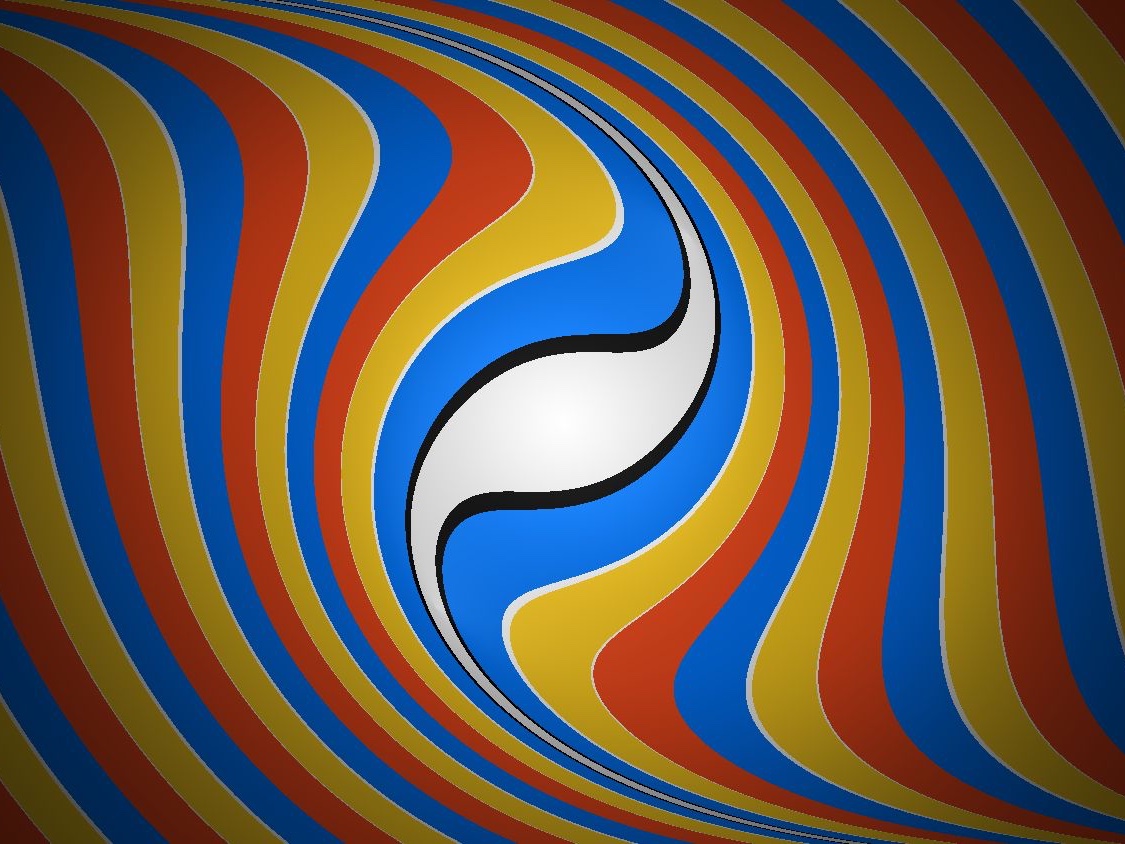}
}\\
\subfloat[$d=7$]{
	\includegraphics[width=0.3\textwidth]{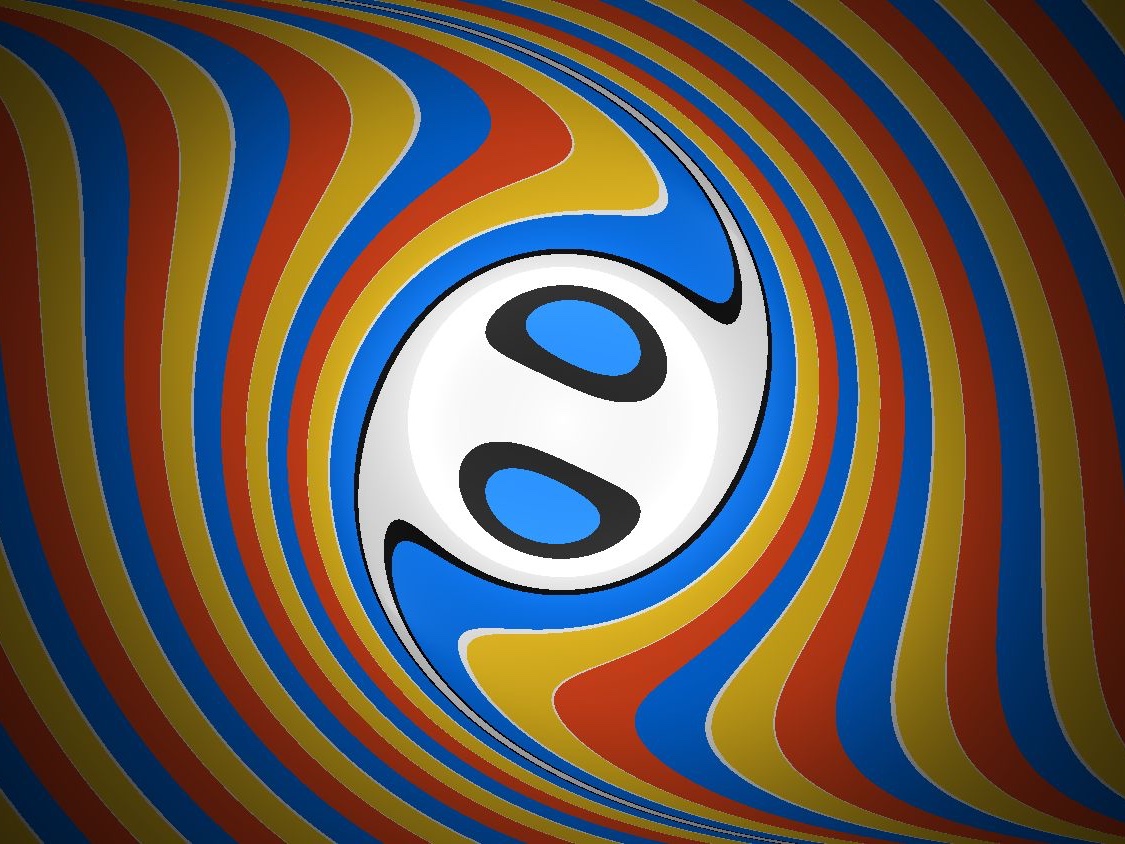}
}
\subfloat[$d=8$]{
	\includegraphics[width=0.3\textwidth]{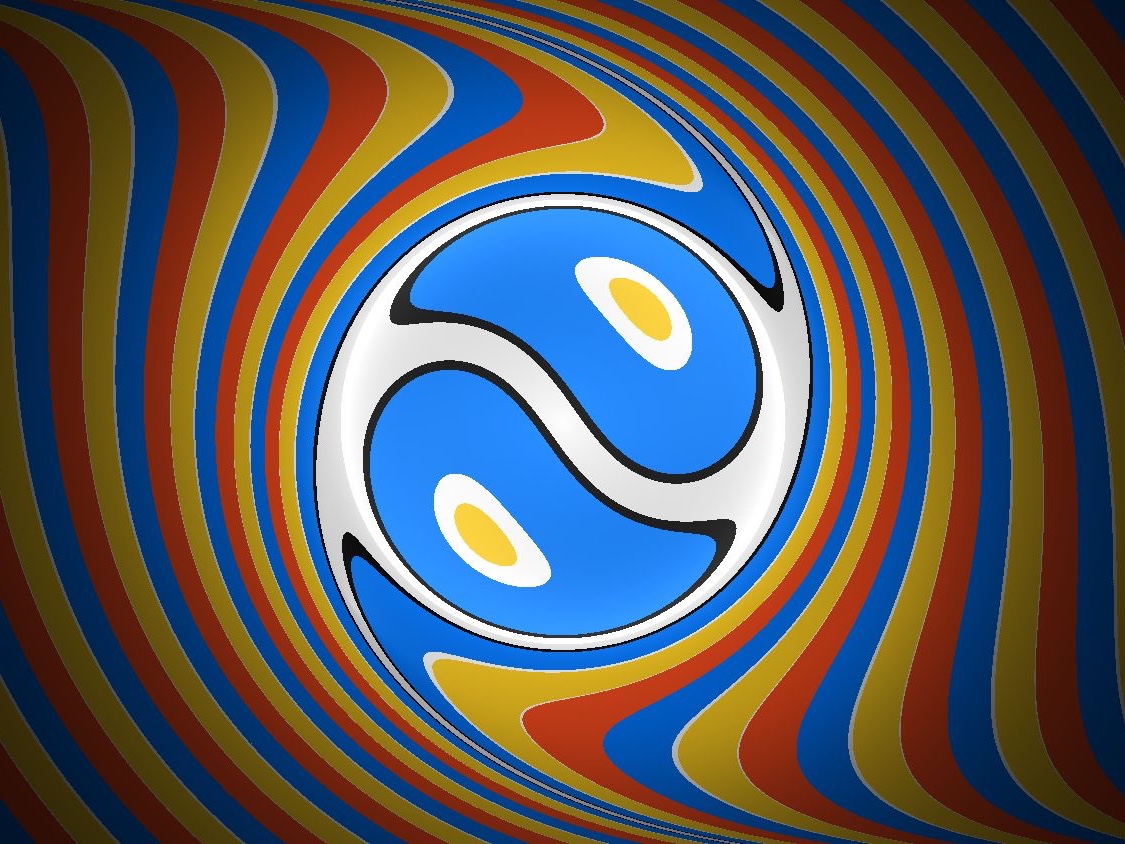}
}
\subfloat[$d=10$]{
	\includegraphics[width=0.3\textwidth]{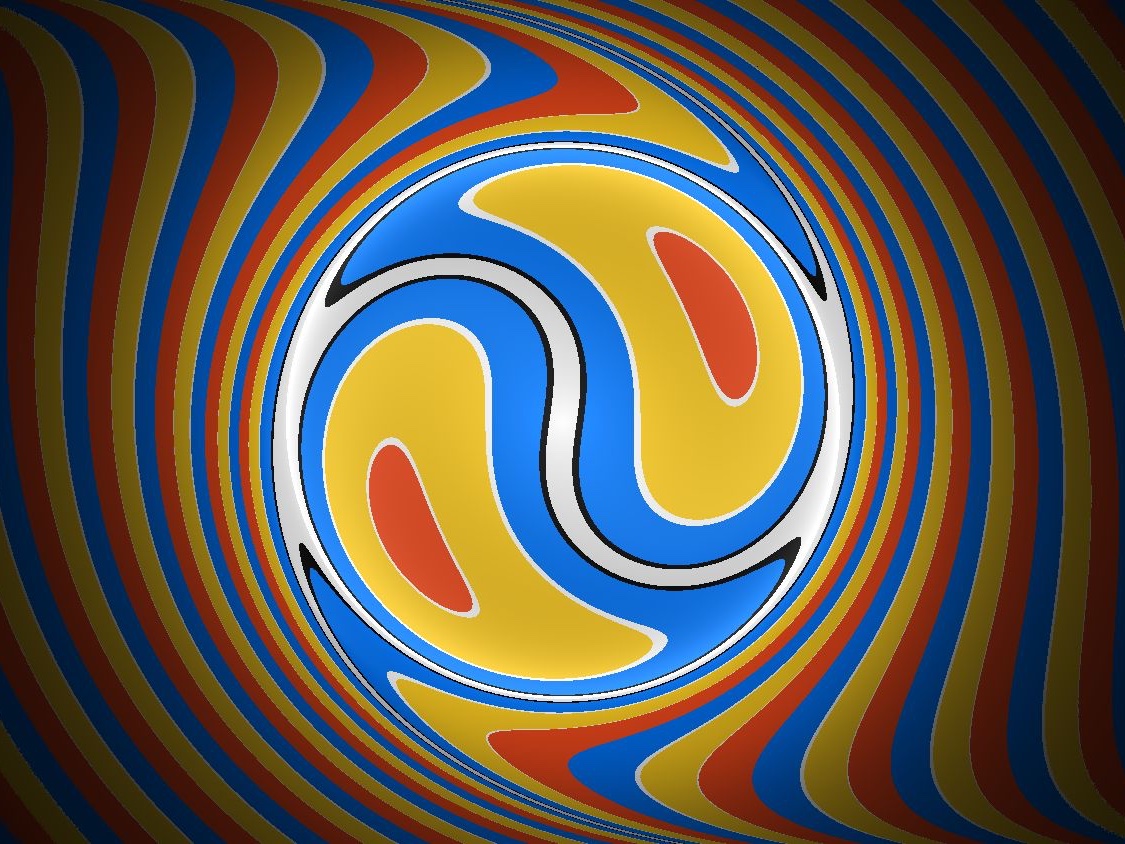}
}\\
\subfloat[$d=15$]{
	\includegraphics[width=0.3\textwidth]{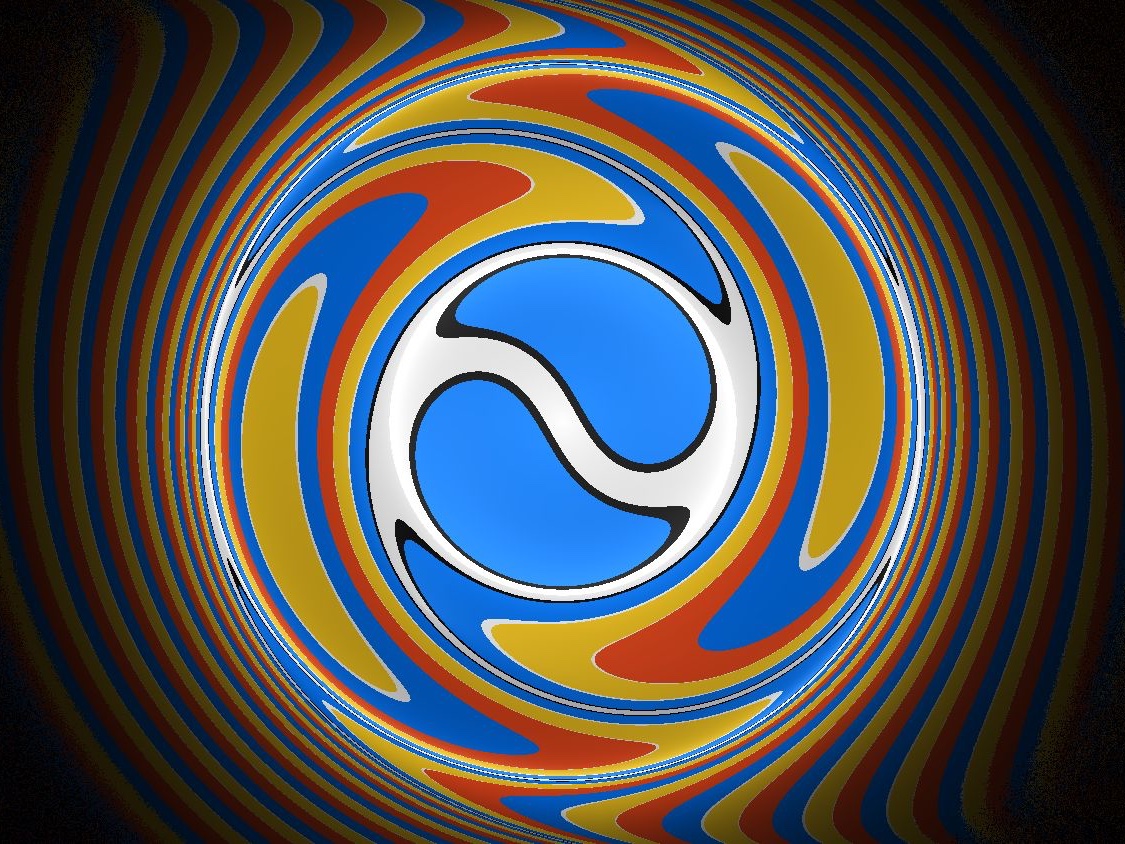}
}
\subfloat[$d=20$]{
	\includegraphics[width=0.3\textwidth]{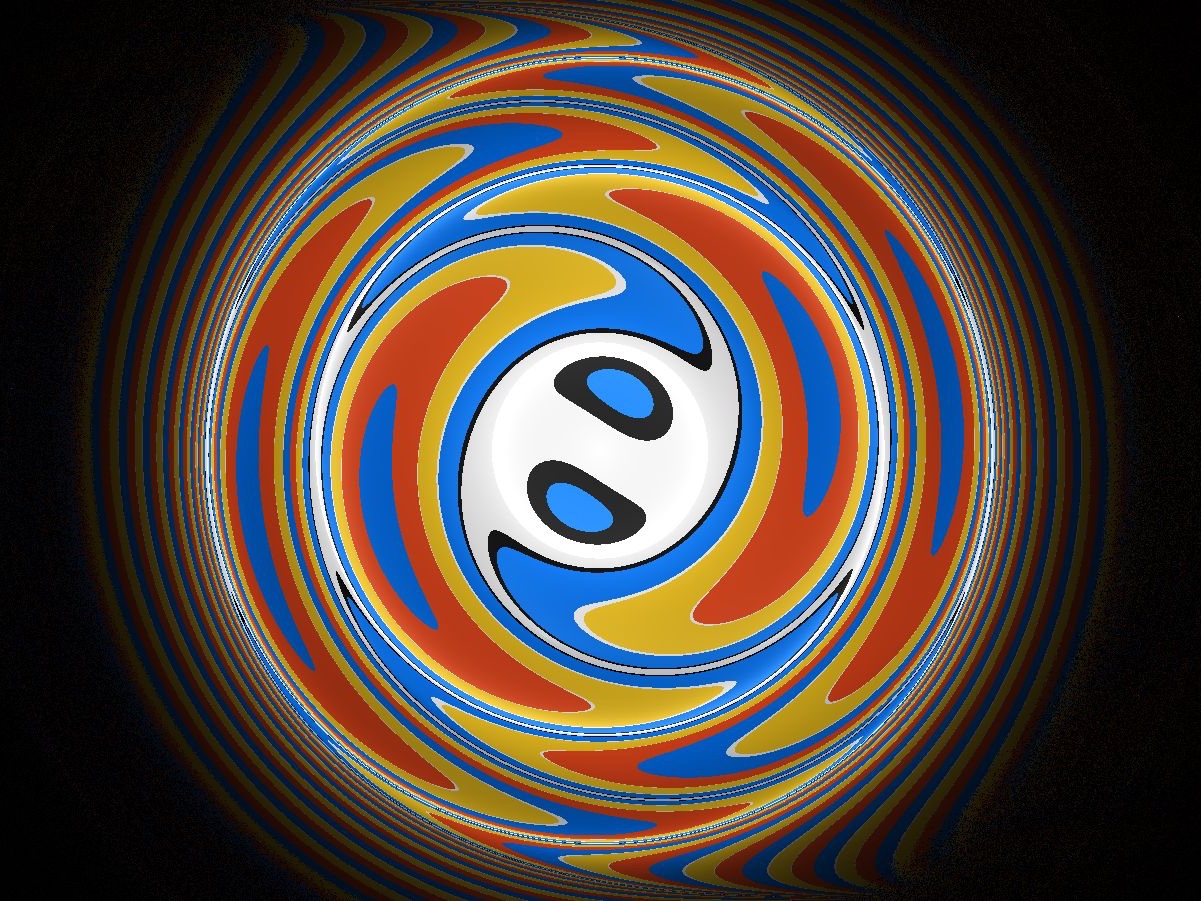}
}
\subfloat[$d=30$]{
	\includegraphics[width=0.3\textwidth]{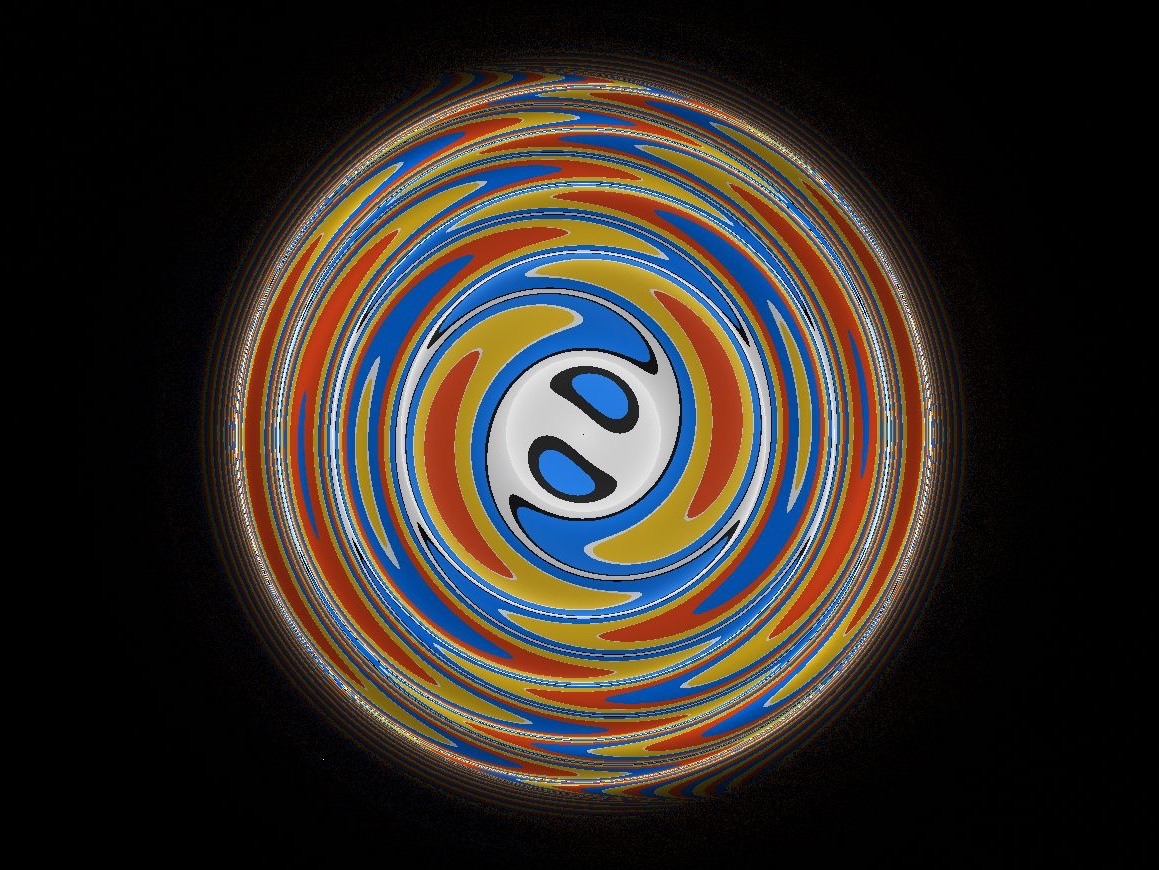}
}
\caption{Horizontal half-space with boundary the hyperbolic plane.  The plane is colored with a geodesic (white) and equidistant curves (primary colors).
The observer is at $p=[0,0,1,d].$
 The figures are rendered with a single light source at height 3 above $\HH^2$, and distance fog. }
\label{Fig:HorizontalPlaneSLR}
\end{figure}

\subsection{Lighting}
\label{Sec:SLRLighting}

We addressed the computation of lighting pairs in \refsec{ExactGeodesicsSLR}.
Here, we calculate the intensity $I(r,u)$ experienced from an isotropic light source at distance $r$ and in the direction $u$. By \refeqn{IntensityDensity}, this is inversely proportional to the area density $\mathcal{A}(r,u)$.
We calculate this directly by taking the derivative of the geodesic flow as in \refeqn{AreaDensityDifferential}.

As a first simplification, note that as the covering map $\lambda\colon X\to \calQ$ is a local isometry and $\mathcal{A}(r,u)$ is a local quantity, we may treat $d\lambda_{\cover o}\colon T_{\cover o}X\to T_o\calQ$ as an identification and work directly in $\calQ = \SL(2,\RR)$.
Let $u$ be the unit vector $u=[a\cos\alpha,a\sin\alpha, c]\in T_o \calQ$ expressed in the basis $(e_x,e_y,e_w)$.
Recall that \refeqn{sl2_Geodesic_1PSubg} gives a parameterization of the unit speed geodesic $\gamma(t)$ in direction $u$ as the product of two one-parameter subgroups of $\calQ = \SL(2,\RR)$ followed by a rotation of angle $\alpha$ about the fiber direction.
These one-parameter subgroups, and hence the geodesic flow, come in three regimes determined by whether $|c/a|$ is greater than, equal to, or less than one. Below we concern ourselves with the two generic cases.

Let $[\rho,\alpha,w]$ be the cylindrical coordinates on $T_o\calQ$ with $(\rho,w)$ the norm of the projections onto the $xy$-plane and $w$-axis respectively, and $\alpha\in[0,2\pi)$ measured from the positive $x$-axis.
In these coordinates, the point $ru\in T_o\calQ$ is expressed as $[\rho,\alpha,w]=[ra,\alpha, rc]$.
Using \refeqn{AreaDensity_Coordinates}, we may calculate the area density in terms of the $\rho, \alpha$, and $w$ derivatives of the geodesic flow.
Here one may deal with the two regimes (in these coordinates, $|\rho|>|w|$ and $|\rho|<|w|$) separately, or unify them into a single computation with complex trigonometric functions. This follows from the particularly nice form of the one-parameter subgroups in \refeqn{sl2_Geodesic_1PSubg}.
In either case, even after much simplification, the resulting formula for area density is rather complicated. We describe it below.

Let $K=\sqrt{|\rho^2-w^2|}$ and let  $f_1\ldots f_6$ denote the polynomials in $\rho,w,$ and $K$:
\begin{equation*}
  \begin{split}
      f_1 &= 17 \rho^6 + 7 \rho^4 w^2 + 16 \rho^2 w^4 + 32 w^6\\
    f_2 &= 48 \rho^2 w^2 (\rho^2 + w^2)\\
    f_3 &= 3 \rho^4 (5 \rho^2 + 3 w^2)\\
    f_4 &=\rho^6 - 2 \rho^2 w^2 - w^4 - \rho^4 (w^2+1)\\
    f_5 &= \rho^6 + 2 \rho^2 w^2 + w^4 - \rho^4 (w^2-1)\\
    f_6 &= 2 \rho^2  (\rho^2 + w^2) K.
      \end{split}
\end{equation*}
\noindent
To combine the two regimes, we let $(S(x),C(x))$ denote $(\sin(x),\cos(x))$ when $|w|>|\rho|$, and $(\sinh(x),\cosh(x))$ for $|w|<|\rho|$.
 Finally, let $g_1$ and $g_2$ be the functions 
\begin{equation*}
\begin{split}
 g_1(\rho,w)=&f_1(\rho,w) - f_2(\rho,w)C(K) + f_3(\rho,w)C(2K)\\
 g_2(\rho,w)=&f_4(\rho,w) + f_5(\rho,w)C(K) \pm f_6(\rho,w)S(K),
 \end{split}
\end{equation*}
 where the $\pm$ in $g_2$ is positive for $|w|>|\rho|$ and negative when $|w|<|\rho|$.  With this notation, the area density is given by 

\begin{equation}
\label{Eqn:sl2_AreaDensity}
\mathcal{A}(r,u)=\frac{\sqrt{\rho^2+w^2}}{2K^{6}}\left|S\left(\frac{K}{2}\right)\right|
\sqrt{
	\Big|
	g_1(\rho,w)g_2(\rho,w)
	\Big|}.
	\end{equation}

See \reffig{SLR_Intensity}.
As with the computation of the geodesic flow in \refsec{flow sl2}, one should use the asymptotic expansion of \refeqn{sl2_AreaDensity} to obtain correct lighting along the null cone $|w|=|\rho|$.

\noindent
\reffig{SLR_Intensity} shows the intensity variation, as seen in the tangent space to a point. 

\begin{figure}[htbp]
\centering
\subfloat[Within a ball of radius 10.]{
\includegraphics[width=0.4\textwidth]{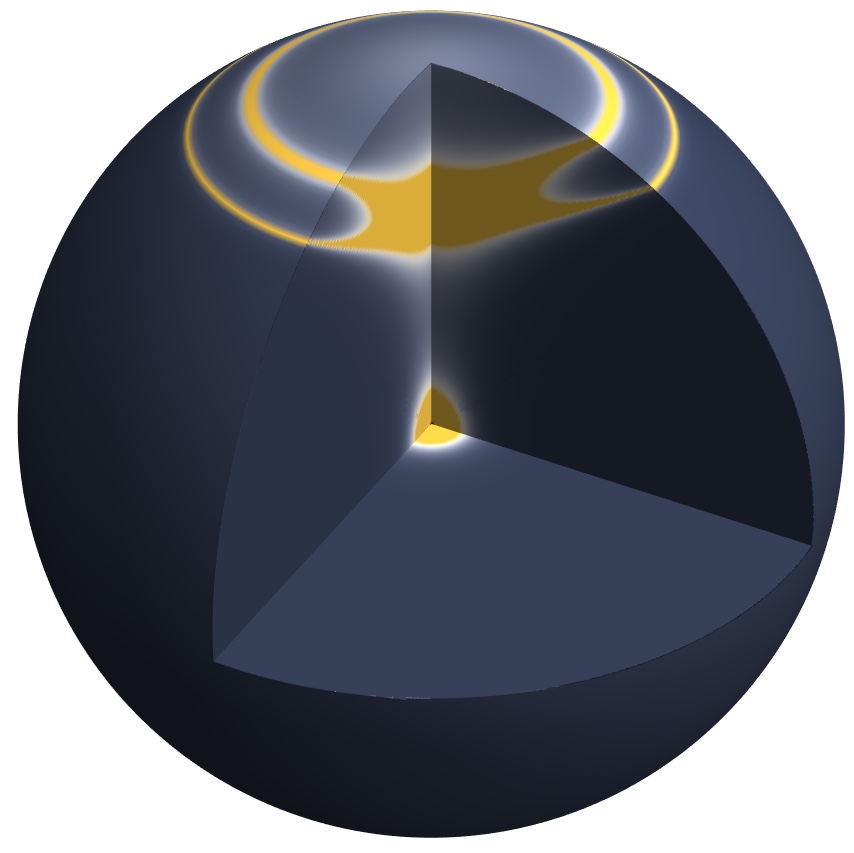}
\label{Fig:SLR_Intensity_10}
}
\quad
\subfloat[Within a ball of radius 30.]{
\includegraphics[width=0.4\textwidth]{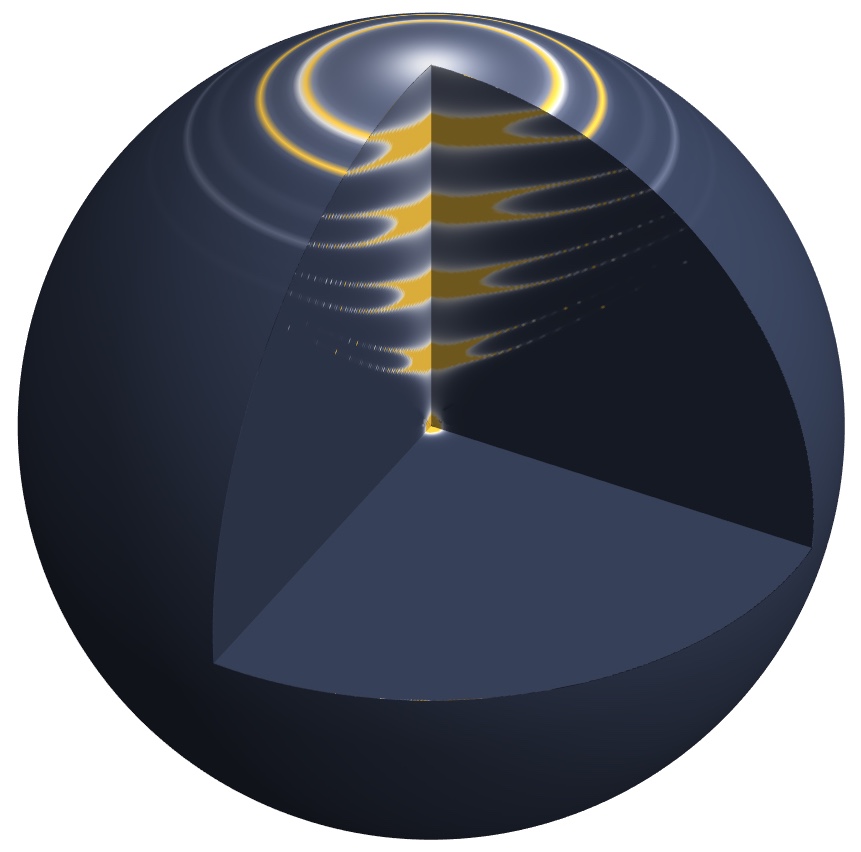}
\label{Fig:SLR_Intensity_30}
}
\caption{The lighting intensity function $I(r,u)$ in $\SLR$ geometry}
\label{Fig:SLR_Intensity}
\end{figure}

\subsection{Discrete subgroups and fundamental domains.}
\label{Sec:SLRDiscreteSubgroupsFundDoms}

The manifolds with $\SLR$ geometry are classified in \cite[Theorem~4.15]{Scott}. The main examples are unit tangent bundles of hyperbolic surfaces and two-dimensional orbifolds.

\medskip
Our model $\mathcal Q$ of ${\rm SL}(2, \RR)$ is a projective model, in the sense that it induces a faithful representation ${\rm Isom}(\mathcal Q) \to {\rm PGL}(4,\RR)$. This is not the case for $X$ however.
Nevertheless, we can adapt the strategy described in \refsec{TeleportingProjective} to produce an efficient fundamental domain.
We explain this strategy with an example.

Let $\Gamma$ be the fundamental group of a genus two surface $\Sigma$:
\begin{equation*}
	\Gamma = \left< A_1,A_2,B_1,B_2 \mid [A_1,B_1][A_2,B_2] = 1\right>.
\end{equation*}
A choice of hyperbolic metric on $\Sigma$ induces a representation $\Gamma \to {\rm SL}(2, \RR)$.
For our example, we choose this metric so that a fundamental domain $U$ for the action of $\Gamma$ on $\HH^2$ is a regular octagon centered at the origin, see \reffig{GenusTwoSurface}.
The generators of $\Gamma$ can now be written as points of $\mathcal Q$:
\begin{align*}
	A_1 & = \left[ \frac {\sqrt 2}2 +1, -\frac {\sqrt 2}2 -1, -\sqrt 2 \sqrt{\sqrt 2 +1}, 0\right], \\
	A_2 & = \left[ \frac {\sqrt 2}2 +1, -\frac {\sqrt 2}2 -1, \sqrt 2 \sqrt{\sqrt 2 +1}, 0\right], \\
	B_1 & = \left[ \frac {\sqrt 2}2 +1, \frac {\sqrt 2}2 +1, \sqrt{\sqrt 2+1}, - \sqrt{\sqrt 2+1}\right], \\
	B_2 & = \left[ \frac {\sqrt 2}2 +1, \frac {\sqrt 2}2 +1, -\sqrt{\sqrt 2+1}, \sqrt{\sqrt 2+1}\right].
\end{align*}

\begin{figure}
\centering
\labellist
\normalsize\hair 2pt
\pinlabel $A_1$ at 74 141
\pinlabel $A_2$ at 129 71
\pinlabel $B_1$ at 150 105
\pinlabel $B_2$ at 60 99

\pinlabel $n_1$ at 198 96
\pinlabel $n_2$ at 179 163
\pinlabel $n_3$ at 116 197
\pinlabel $n_4$ at 30 163

\endlabellist
\includegraphics[width=0.6\textwidth]{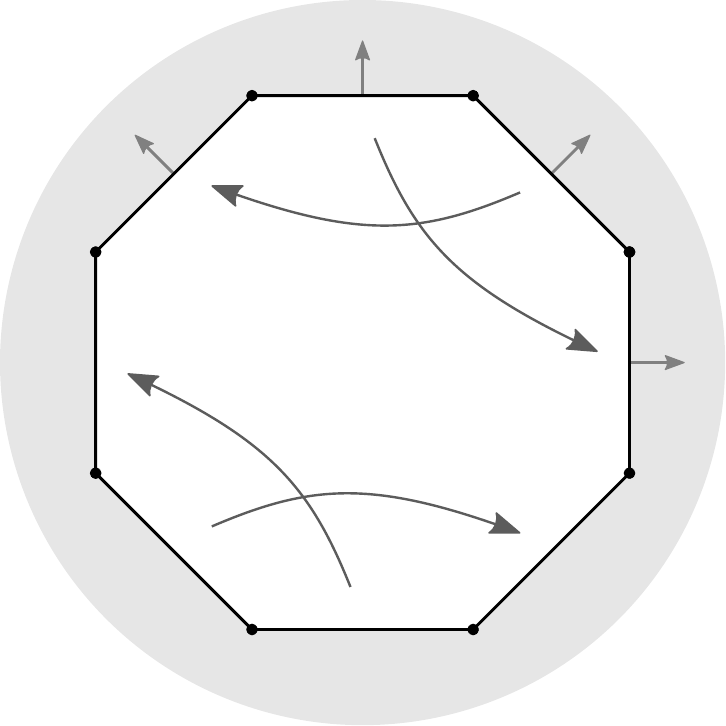}
\caption{A sketch of the fundamental domain in $\HH^2$.
The gray disc is the Klein model of the hyperbolic plane. 
The white octagon is the fundamental domain for the action on $\HH^2$ of the fundamental group $\Gamma$ of a genus-two surface.}
\label{Fig:GenusTwoSurface}
\end{figure}
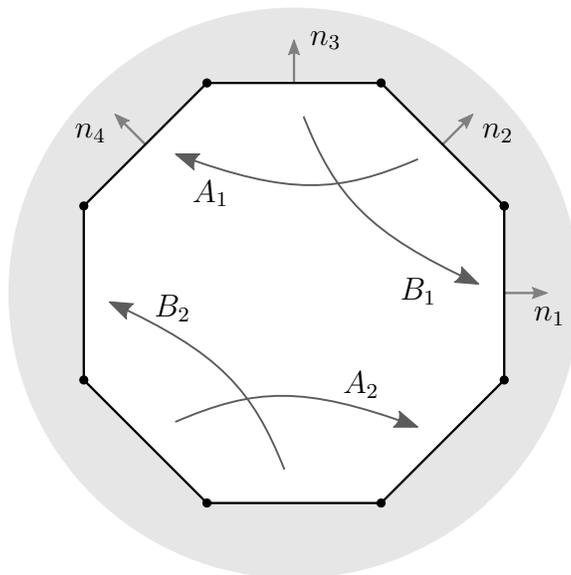

The pre-image $\cover \Gamma$ of $\Gamma$ by the covering map $\lambda \colon X \to \calQ$ is now a lattice in $\SLR$, viewed as a subset of $G$, the isometries of $X = \SLR$.
We choose lifts $\cover A_1$, $\cover A_2$, $\cover B_1$, and $\cover B_2$ of the previous generators so that their fiber components are respectively $-\pi/2$, $-\pi/2$, $\pi/2$, and $\pi/2$.
For convenience, we define a new element $\cover C$ that is the translation by $2\pi$ along the fiber direction.
One checks that $\cover C^{-2} = [\cover A_1,\cover B_1][\cover A_2,\cover B_2]$ in $\cover \Gamma$.
Note also that $\cover C$ commutes with $\cover A_1$, $\cover A_2$, $\cover B_1$, and $\cover B_2$.
A fundamental domain for the action of $\Gamma$ on $X$ is the subset $D = U \times [-\pi, \pi]$ of $X = \mathcal H \times \RR$.
However, our model $X$ is not well suited to checking easily whether a point belongs to $D$ or not.

To solve this problem, we consider the isometry $h \colon \mathcal H \to \calK$ between the hyperboloid model $\mathcal H \subset \RR^3$ and the Klein model $\calK\subset \RR^2$ of $\HH^2$.
The isometry $h$ extends to a bijection
\begin{equation*}
	\begin{array}{ccc}
		\mathcal H  \times \RR & \to & \calK \times \RR \\
		(q,w) & \mapsto & \left(h(q), w\right).
	\end{array}
\end{equation*}
This provides yet another model $X' = \calK \times \RR$ for $\SLR$.
The image of $D$ under this identification is $D' = U' \times [-\pi, \pi]$ where $U'$ is now an octagon in $\calK$ whose sides are straight lines.
We define the following normal vectors in $\RR^3$
\begin{align*}
	n_1 &= \left[1, 0, 0\right], & n_3 &= \left[0, 1, 0\right], \\	
	n_2 &= \left[\frac {\sqrt 2}2, \frac {\sqrt 2}2, 0\right], & n_4 &= \left[-\frac {\sqrt 2}2 , \frac {\sqrt 2}2, 0\right], \\
	n_5 &= \left[0, 0, 1\right], && 
\end{align*}
see \reffig{GenusTwoSurface}.
To each index $k \in \{1,2,3, 4\}$ we associate two half-spaces 
\begin{equation*}
	H_k^- = \{ v \in \RR^3 \colon \left<v, n_k\right> \geq - \delta \}, 
	\quad \text{and} \quad
	H_k^+ = \{ v \in \RR^3 \colon \left<v, n_k\right> \leq \delta \}, 
\end{equation*}
where $\left< \cdot,\cdot \right>$ is the standard dot product in $\RR^3$ and $\delta = \sqrt 2 \sqrt{\sqrt 2 - 1}$. We choose $\delta$ so that $D'$ is the intersection of these half-spaces.
Similarly, we let 
\begin{equation*}
	H_5^- = \{ v \in \RR^3 \colon \left<v, n_5\right> \geq - \pi \}, 
	\quad \text{and} \quad
	H_5^+ = \{ v \in \RR^3 \colon \left<v, n_5\right> \leq \pi \}. 
\end{equation*}
The teleporting algorithm has two main steps.
Let $p = (q,w)$ be a point in our new model $\calK \times \RR$ of $\SLR$.
\begin{enumerate}
	\item If $q$ does not belong to $H_1^+$ (respectively $H_2^+$, $H_3^+$, $H_4^+$, $H_1^-$, $H_2^-$, $H_3^-$, $H_4^-$), then we move $p$ by $\cover B_1^{-1}$ (respectively $\cover A_1$, $\cover B_1$, $\cover A_1^{-1}$, $\cover B_2^{-1}$, $\cover A_2$, $\cover B_2$, $\cover A_2^{-1}$).
	Observe that $U'$ is also a Dirichlet domain for the action of $\Gamma$ on $\HH^2$.
	More precisely, $H_1^+ \cap H_3^+$ is the set of points in $\calK$ which are closer to the origin $o$ than their translates by $B_1^{\pm 1}$.
	Hence the translation by $\cover B_1^{-1}$ moves the projection $q$ of $p$ to $\calK$ closer to $o$.
	It follows that after finitely many steps, we can ensure that $q$ belongs to $U'$.
	Since we always reduce the distance from $o$ to $p$, the order in which we perform the algorithm does not matter.
	\item Once this is done, if $p$ does not belong to $H_5^-$ (respectively $H_5^+$), then we move it by $\cover C$ (respectively $\cover C^{-1}$).
	Note that $\cover C$ does not affect the horizontal component $q$ of $p$. Therefore, after this process $p$ lies in the fundamental domain $D'$.
\end{enumerate}

\reffig{GenusTwo} shows some views within the unit tangent bundle to a genus two surface, as described in this section. The fundamental domain is a very tall octagonal prism. To better illustrate the geometry, our scene is the complement of three spheres stacked vertically within this domain.

\begin{figure}[htbp]
\centering
\subfloat[Looking along the fiber.]{
	\includegraphics[width=0.3\textwidth]{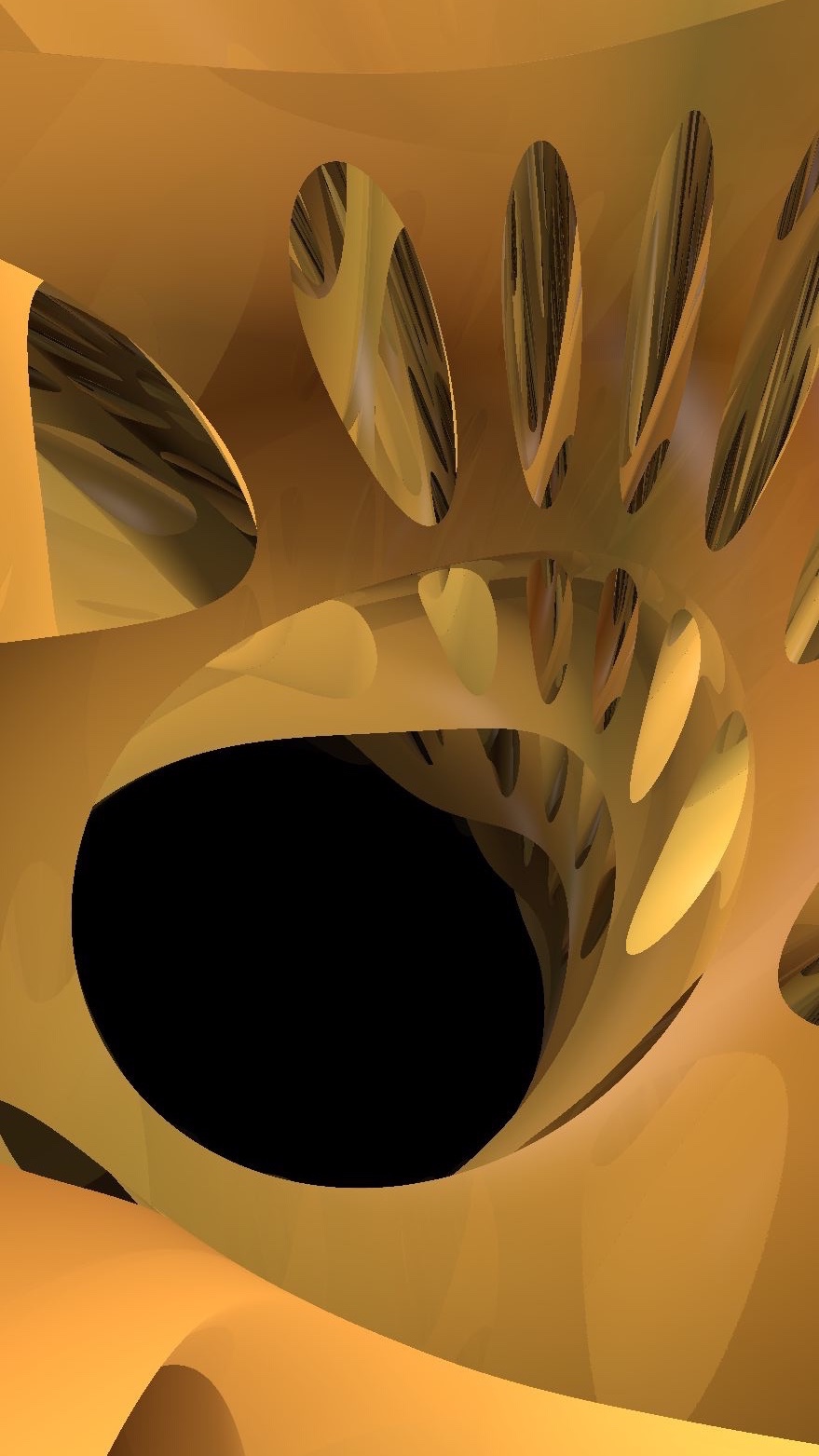}
}
\subfloat[Looking near the fiber direction.]{
	\includegraphics[width=0.3\textwidth]{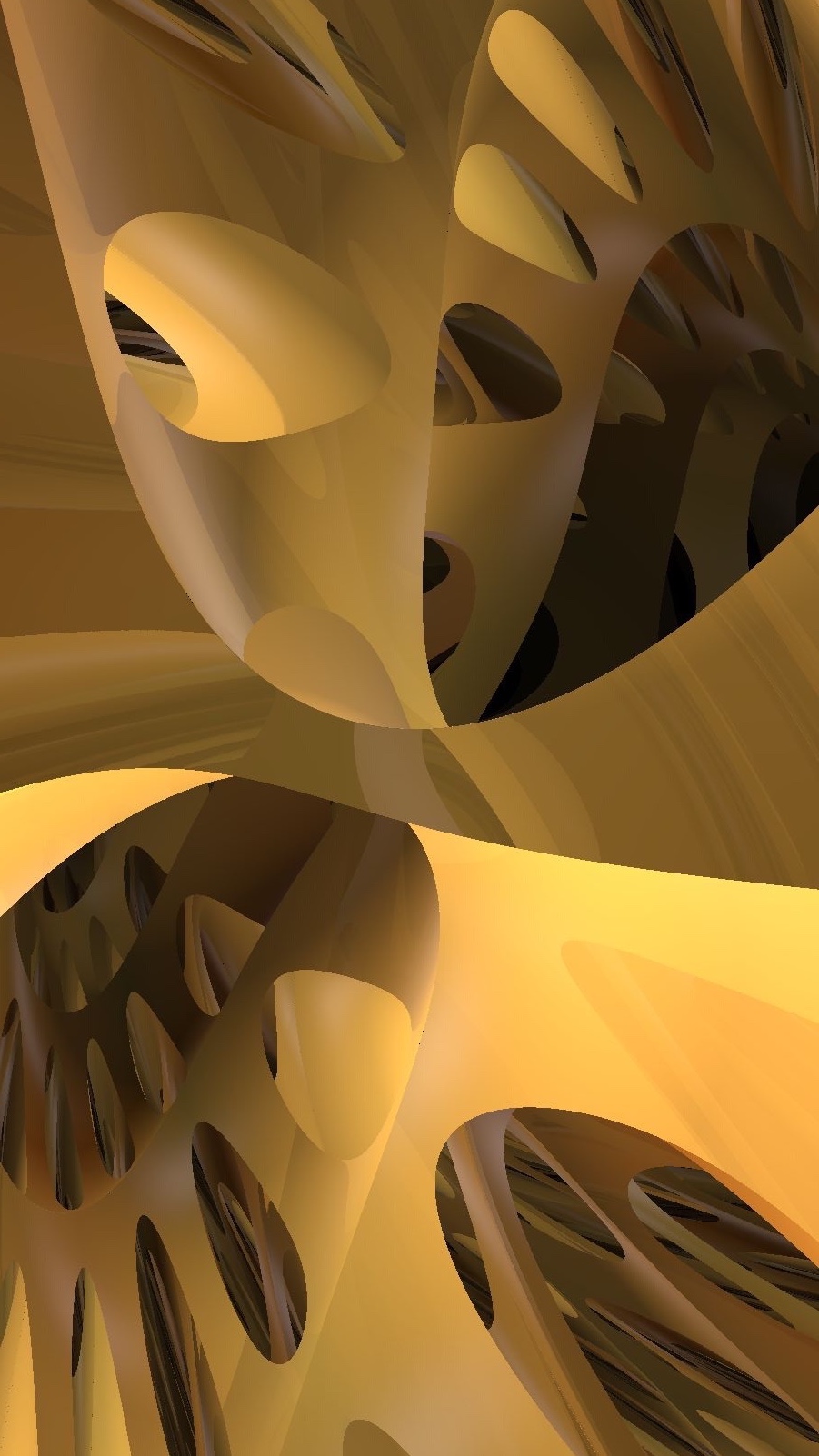}
}
\subfloat[A view further from the fiber direction.]{
	\includegraphics[width=0.3\textwidth]{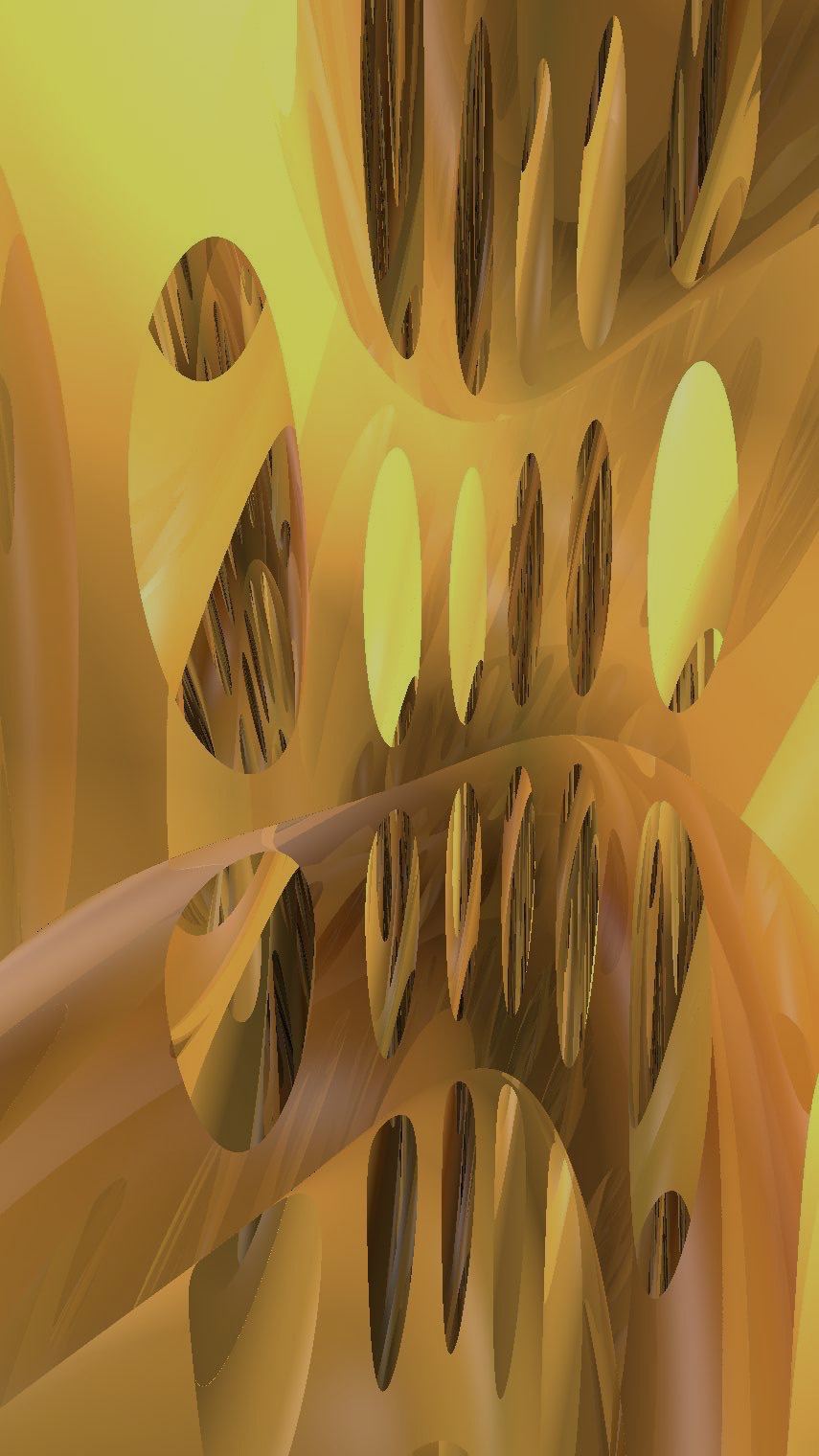}
}
\caption{The unit tangent bundle of a genus two surface.}
\label{Fig:GenusTwo}
\end{figure}

In \reffig{SLRExamples}, we show the in-space view for various scenes in $\SLR$ geometry.  \reffig{SLRGenusTwo} shows the same scene as \reffig{GenusTwo}, with a globe added at the center of each of the three spheres.
\reffig{SLRSpheres} shows a lattice of globes in the unit tangent bundle for a sphere with cone points $\pi/3, \pi/3,$ and $2\pi/3$.
\reffig{SLRFibers} shows solid cylinders (which we implement as vertical objects) around fibers of $\SLR$. The lighting in these images is based on a continuously varying direction field rather than point light sources.

\begin{figure}[htbp]
\centering
\subfloat[A globe within the unit tangent bundle of a genus two surface.]{
\includegraphics[width=0.90\textwidth]{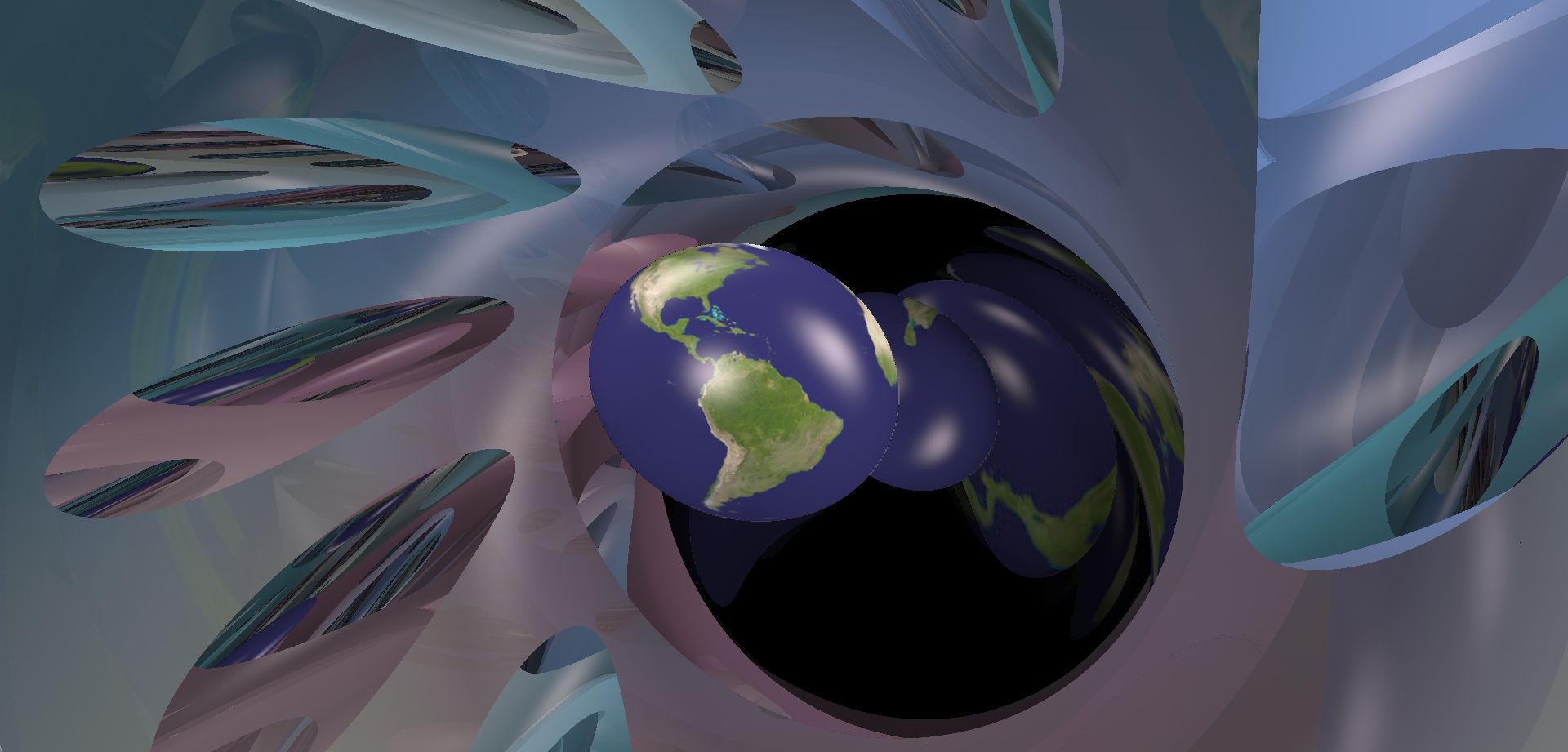}
\label{Fig:SLRGenusTwo}
}\
\subfloat[A lattice of globes.]{
\includegraphics[width=0.90\textwidth]{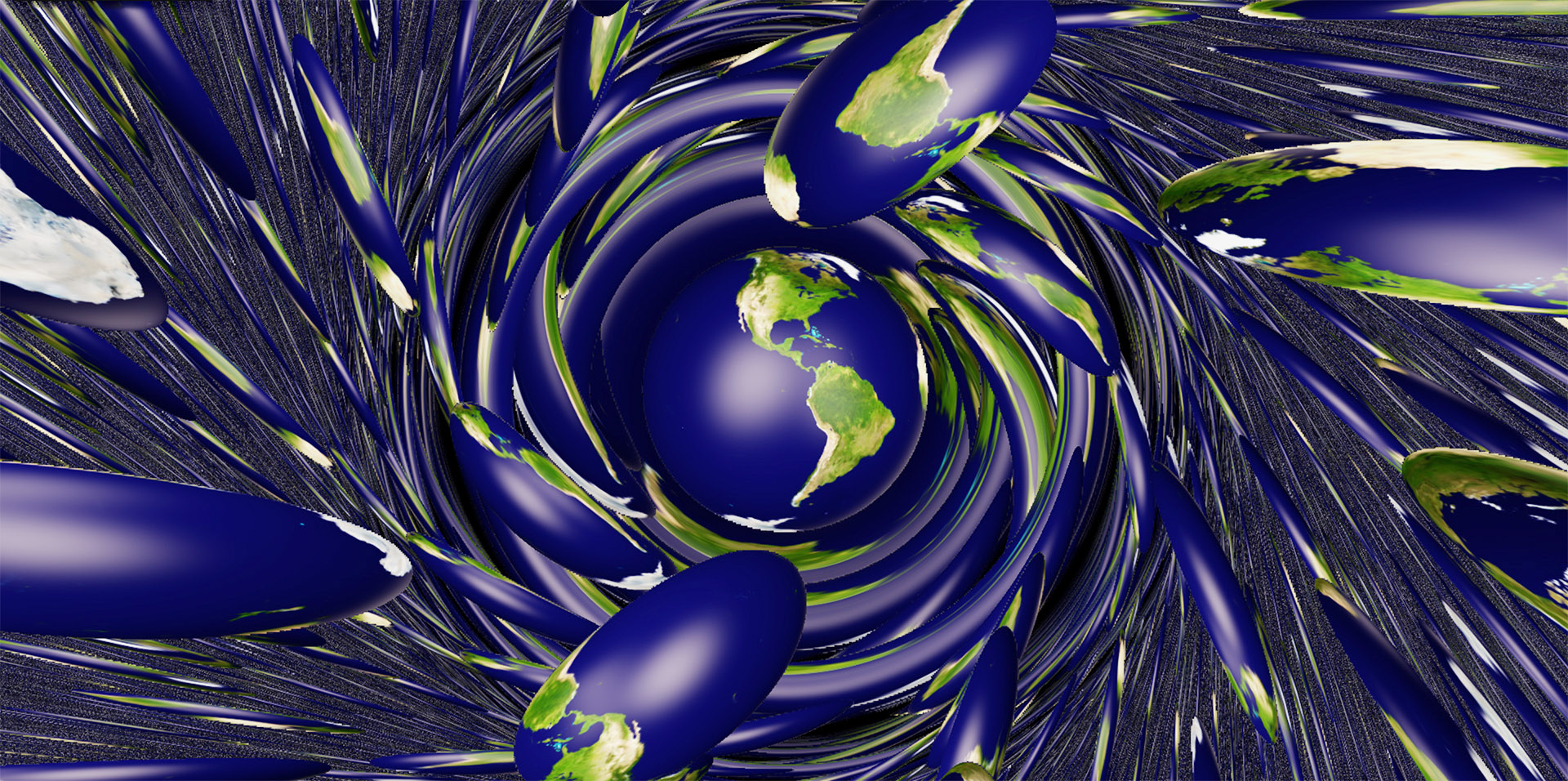}
\label{Fig:SLRSpheres}
}\
\subfloat[Lifts of the fibers in $UT\HH^2$.]{
\includegraphics[width=0.90\textwidth]{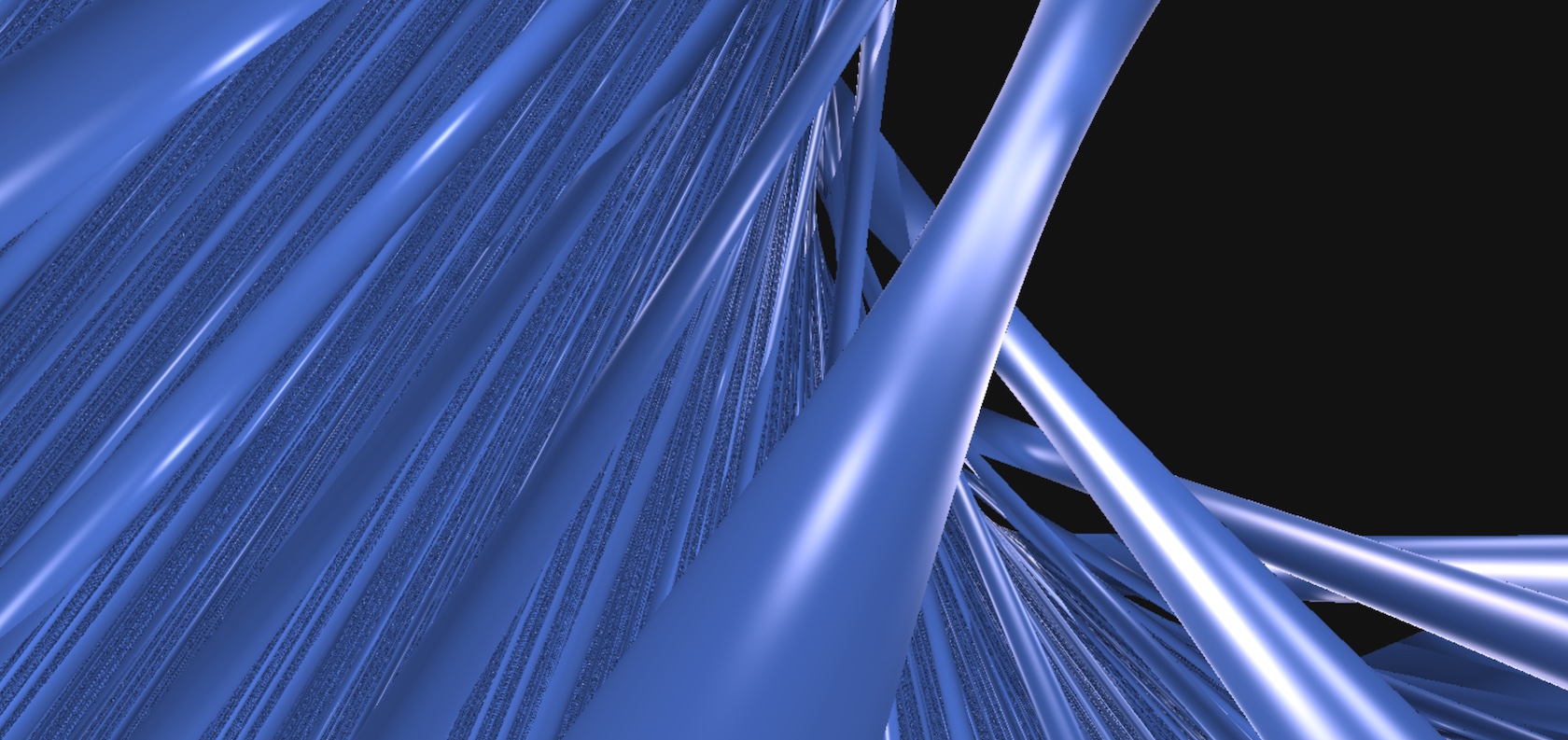}
\label{Fig:SLRFibers}
}\
\caption{$\SLR$ Geometry. }
\label{Fig:SLRExamples}
\end{figure}

\section{Sol}
\label{Sec:Sol}

\subsection{Model}
As with Nil and $\SLR$, Sol is a Lie group.
The underlying space of our model is the affine subspace $X$ of $\RR^4$ defined by $w= 1$.
The group law is as follows: the point $[x,y,z,1]$ acts on $X$ on the left as the matrix
\begin{equation*}
	\left[\begin{array}{cccc}
		e^z & 0 & 0 & x\\
		0 & e^{-z} & 0 & y\\
		0 & 0 & 1 & z\\
		0 & 0 & 0 & 1
	\end{array}\right].
\end{equation*}
The origin $o$ is the point $[0,0,0,1]$.
Its tangent space $T_oX$ is identified with the linear subspace of $\RR^4$ given by the equation $w = 0$.
The metric tensor at an arbitrary point $p = [x,y,z,1]$ is
\begin{equation}
\label{Eqn:metric sol}
	ds^2 = e^{-2z}dx^2 + e^{2z}dy^2 + dz^2.
\end{equation}
With this metric, the action of Sol on itself is an action by isometries.
The stabilizer $K$ of the origin $o$ is isomorphic to the dihedral group of order eight, $D_8$, which is generated by two symmetries acting on $X$ as the matrices
\begin{equation*}
	S_1 = 
	\left[\begin{array}{cccc}
		-1 & 0 & 0 & 0\\
		0 & 1 & 0 & 0\\
		0 & 0 & 1 & 0\\
		0 & 0 & 0 & 1
	\end{array}\right]
	\quad \text{and} \quad
	S_2 = 
	\left[\begin{array}{cccc}
		0 & 1 & 0 & 0\\
		1 & 0 & 0 & 0\\
		0 & 0 & -1 & 0\\
		0 & 0 & 0 & 1
	\end{array}\right]
\end{equation*}
respectively.
These symmetries can be observed in the balls of Sol, see \reffig{sol balls - 3d}.

\begin{figure}[h!tbp]
\begin{center}
	\includegraphics[width=\textwidth]{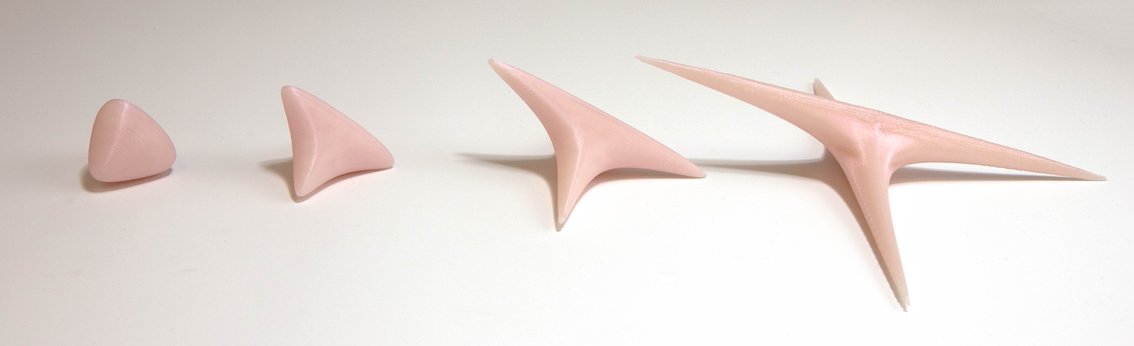}
\caption{3D-printed models of the balls of radius one to four in Sol. We scaled the ball of radius $r$ down by a factor of $r$ in order to keep the physical sizes reasonable. Photograph by Edmund Harriss.}
\label{Fig:sol balls - 3d}
\end{center}
\end{figure}

\subsection{Geodesic flow and parallel transport}
\label{Sec:FlowSol}
As for Nil and $\SLR$ in Sections \ref{Sec:nil geo flow} and \ref{Sec:flow sl2}, we can use Grayson's method to study the geodesic flow and parallel transport.
Let $\gamma \colon \RR \to X$ be a geodesic, and let be $T(t) \colon T_{\gamma(0)}X \to T_{\gamma(t)}X$ be the corresponding parallel-transport operator.
Following Sections~\ref{Sec:GeodesicFlow - Grayson} and \ref{Sec:Parallel transport - Grayson}, we define two paths $u \colon \RR \to T_oX$ and $Q \colon \RR \to {\rm SO}(3)$ by the relations
\begin{equation*}
	\begin{split}
		\dot \gamma(t) & = d_oL_{\gamma(t)} u(t), \textrm{ and} \\
		T(t) \circ d_oL_{\gamma(0)} & = d_oL_{\gamma(t)} Q(t).
	\end{split}
\end{equation*}
After some computation, Equations (\ref{Eqn:GraysonMethodFlowSphere}) and (\ref{Eqn:ParallelTransportLinEq}) respectively become
\begin{equation*}
	\left\{ \begin{split}
		\dot u_x & = u_xu_z \\
		\dot u_y & = -u_yu_z \\
		\dot u_z & = u_y^2 - u_x^2
	\end{split}\right.
\end{equation*}
and 
\begin{equation*}
	\dot Q + BQ = 0,
	\quad \text{where} \quad
	B =\left[\begin{array}{ccc}
		0 & 0 & -u_x \\
		0 & 0 & u_y \\
		u_x & -u_y & 0
	\end{array}	\right].
\end{equation*}

The path $u$, as well as the geodesic $\gamma$, can be computed explicitly \cite{Troyanov:1998aa}.\footnote{
	A commonly cited reference for solving the geodesic flow in Sol is \cite{Bolcskei:2007aa}. 
	However, the authors do not conduct the computation to the final stage -- see their Theorem~4.1(1). Moreover, the formulas given in Theorem~4.1(2) have some errors.
}
Assume that $\gamma$ starts at the origin $o$, so that the initial condition is $u(0) = \dot \gamma(0) = [a,b,c,0]$.
Because of the symmetries of Sol, we can assume without loss of generality that $a \geq 0$ and $b \geq 0$.
We distinguish three cases.

\paragraph{\itshape Case $a = 0$.}
Here the solution for $u$ is
\begin{equation*}
	u(t) = \left[ 0, \frac b{\cosh t + c \sinh t}, \frac{c + \tanh t}{1 + c\tanh t},0 \right].
\end{equation*}
It follows that 
\begin{equation*}
	\gamma(t) = \left[0, \frac{b \tanh t}{1 + c \tanh t}, \ln(\cosh t + c \sinh t), 1\right].
\end{equation*}
In particular, $\gamma$ stays in the plane $\{x = 0\}$.
This plane is totally geodesic and isometric to $\HH^2$.
\paragraph{\itshape Case $b = 0$.}
Here $u$ and $\gamma$ can be deduced from the previous case, via a conjugation by the symmetry $S_2$ fixing the origin.
That is, 
\begin{equation*}
	u(t) = \left[ \frac a{\cosh t - c \sinh t}, 0, \frac{c - \tanh t}{1 - c\tanh t},0 \right]
\end{equation*}
and
\begin{equation*}
	\gamma(t) = \left[\frac{a \tanh t}{1 - c \tanh t}, 0, -\ln(\cosh t - c \sinh t), 1\right].
\end{equation*}
Note that $\gamma$ stays in the plane $\{ y=0\}$, which is also a totally geodesic, isometrically embedded copy of $\HH^2$.

\paragraph{\itshape Case $ab \neq 0$.}
We first define some auxiliary parameters.
Let
\begin{equation*}
	k = \sqrt{\frac{1 - 2ab}{1 + 2ab}} \quad \text{and} \quad k' = 2\sqrt{\frac{ab}{1 + 2ab}}.
\end{equation*}
The associated \emph{complete elliptic integrals} of the first and second kind are respectively
\begin{equation*}
	K(k) = \int_0^{\frac\pi2} \frac {d\theta}{\sqrt{1 - k^2 \sin^2\theta}} \quad \text{and} \quad 
	E(k) = \int_0^{\frac\pi2} \sqrt{1 - k^2 \sin^2\theta}d\theta.
\end{equation*}
We denote by $\sinj$ and $\cosj$ the \emph{Jacobi elliptic sine} and \emph{cosine} functions with elliptic modulus $k$. 
We write $\ampj$ for the \emph{delta amplitude} and $\zeta$ for the \emph{Jacobi zeta function}, also with elliptic modulus $k$.
For an in-depth study of elliptic functions, we refer the reader to \cite{Jacobi:1829aa, Oberhettinger:1949aa, Lawden:1989aa}.
Recall that $\sinj$ and $\cosj$ are $4K(k)$-periodic.	Let 
\begin{equation*}
	\mu = \sqrt{1 + 2 ab}. 
	\end{equation*}
We also fix $\alpha \in [0, 4K(k))$ such that 
\begin{equation*}
	\sinj\alpha = - \frac c{\sqrt{1 - 2ab}} \quad \text{and} \quad \cosj \alpha = \frac {a - b}{\sqrt{1 - 2ab}}.
\end{equation*}
Setting $s = \mu t + \alpha$, we now have 
\begin{equation*}
	u(t) = \left[
		\sqrt{ab}\ \frac {k \cosj s  +  \ampj s }{k'},
		\sqrt{ab}\ \frac {k'}{k \cosj s  +  \ampj s},
		- k \mu \sinj s,
		0
	 \right].
\end{equation*}
In order to write the solution for $\gamma$, we let 
\begin{equation*}
	L = \frac {E(k)}{k'K(k)} - \frac{k'}2.
\end{equation*}
We finally get
\begin{equation*}
	\gamma(t) = \left[\begin{split}
		\sqrt{\frac ba} \left( \frac 1 {k'} \big(\zeta(s) - \zeta(\alpha)\big) + \frac k{k'} \big(\sinj s - \sinj \alpha\big) + \big(s - \alpha\big)L \right)\\
\sqrt{\frac ab} \left( \frac 1 {k'} \big(\zeta(s) - \zeta(\alpha)\big) - \frac k{k'} \big(\sinj s - \sinj \alpha\big) + \big(s - \alpha\big)L \right)\\
		\frac 12 \ln \left(\frac ba\right) + \arcsinh\left( \frac k{k'} \cosj s \right) \\
		1
	\end{split}
	 \right].
\end{equation*}

In practice, we use a mixed approach, as follows.
\begin{itemize}
	\item When we need to flow for a long time (for example when all objects in the scene are very far away from the camera), then we use the explicit formula above.
	However, if the initial direction $\dot \gamma(0)$ is close to one of the hyperbolic planes, this formula suffers from many numerical errors. This is an example of the kind of error described in \refsec{FloatingPoint}\refitm{RemovableSingularities}. 
	In this case, we replace the exact solution by its asymptotic expansion of order two.
	\item  When we need to flow for a short time, the above method again seems to suffer from significant numerical errors.
	This happens when during the ray-marching algorithm some object is very close, or when updating the position and facing of the observer between two frames.
	In this situation, we numerically integrate the geodesic flow and the parallel transport equations using the Runge--Kutta method of order two.
\end{itemize}

\begin{remark*}
	Since Jacobi elliptic and zeta functions are not available in the OpenGL library, we implemented them directly, using the AGM algorithm~\cite{Bulirsch:1965ab, Olver:2010aa, Abramowitz:1964aa}.
\end{remark*}

\subsection{Distance to coordinate half-spaces}
\label{Sec:SolHalfSpaces}
Given $\alpha \in \RR$, we write $H^+_z(\alpha) = \{ z \geq \alpha\}$ and $H^-_z(\alpha) = \{ z \leq \alpha\}$.
Note that the boundary $\{ z = \alpha \}$ of these half-spaces is isometric to a euclidean plane, but is not convex as a subspace of $X$.
Recall that we write $\sdf(\cdot,S)$ for the signed distance function for the scene $S$.

\begin{lemma}
	Fix a real number $\alpha$.
	For every point $p = [x,y,z,1]$ in $X$, we have 
	\begin{equation*}
		\sdf\left(p, H^-_z(\alpha)\right) = z - \alpha \quad \text{and} \quad \sdf\left(p, H^+_z(\alpha)\right) = \alpha - z .
	\end{equation*}
\end{lemma}

\begin{proof}
	Observe that the collections $\{H^+_z(\alpha) \mid \alpha \in \RR\}$ and $\{H^-_z(\alpha) \mid \alpha \in \RR\}$ are both invariant under the action of Sol on itself.
	Thus without loss of generality, we can assume that $p$ is the origin $o$.
	Similarly, the symmetry $S_2$ fixes the origin and permutes $H^+_z(\alpha)$ and $H^-_z(\alpha)$.
	Hence it suffices to prove the statement for $H^+_z(\alpha)$.
	Suppose that $\alpha \geq 0$ (the other case works in the same way).
	The path $\gamma(t) = [0,0, t,1]$ is a geodesic starting at the origin and hitting $H^+_z(\alpha)$ at time $t = \alpha$.
	Hence we have $\dist(o, H^+_z(\alpha)) \leq \alpha$.
	
	Let us prove the other inequality.
	Consider a point $q \in H^+_z(\alpha)$ and a minimizing arc length parametrized geodesic $\gamma \colon [0, \ell] \to X$ from $o$ to a point $q$.
	If we write the path $\gamma$ as $\gamma(t) = [x(t), y(t), z(t),1]$, then from the metric given in \refeqn{metric sol} we get that $|\dot z(t)| \leq 1$, because $\gamma$ is arc length parametrized.
	Consequently, we have
	\begin{equation*}
		\dist(o,q) \geq \ell \geq z(\ell) \geq \alpha.
	\end{equation*}
	This inequality holds for every point $q \in H^+_z(\alpha)$, hence the result.
\end{proof}

In \reffig{HalfSpaceHoriz}, we use these signed distance functions to draw horizontal half-spaces, patterned with square tilings.

\begin{figure}[htbp]
\begin{center}
\subfloat[$H_z^+(1)$]{%
\includegraphics[width=0.30\textwidth]{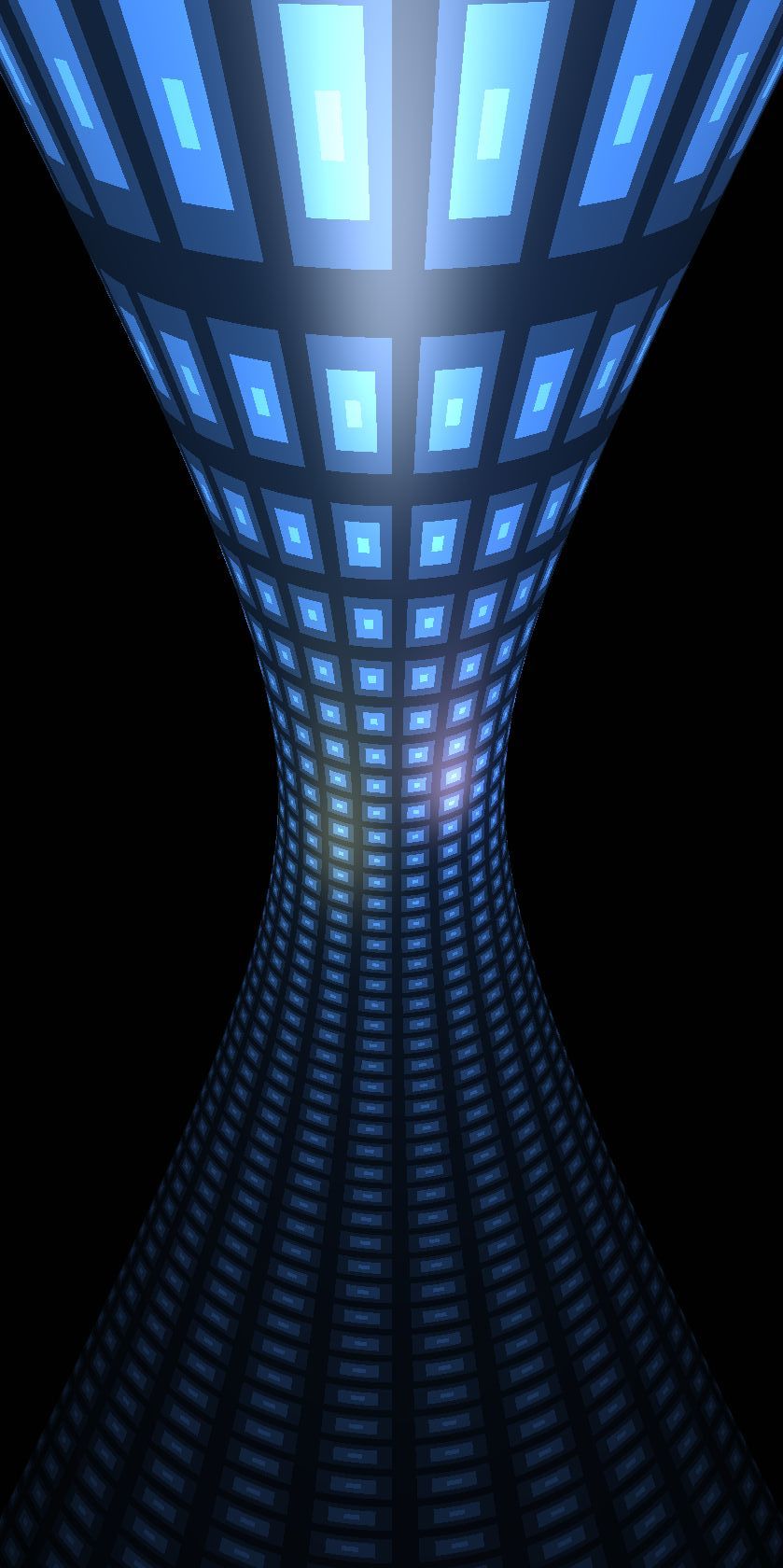}%
\label{Fig:SolPlanes1}%
}%
\quad
\subfloat[$H_z^-(-1)$]{%
\includegraphics[width=0.30\textwidth]{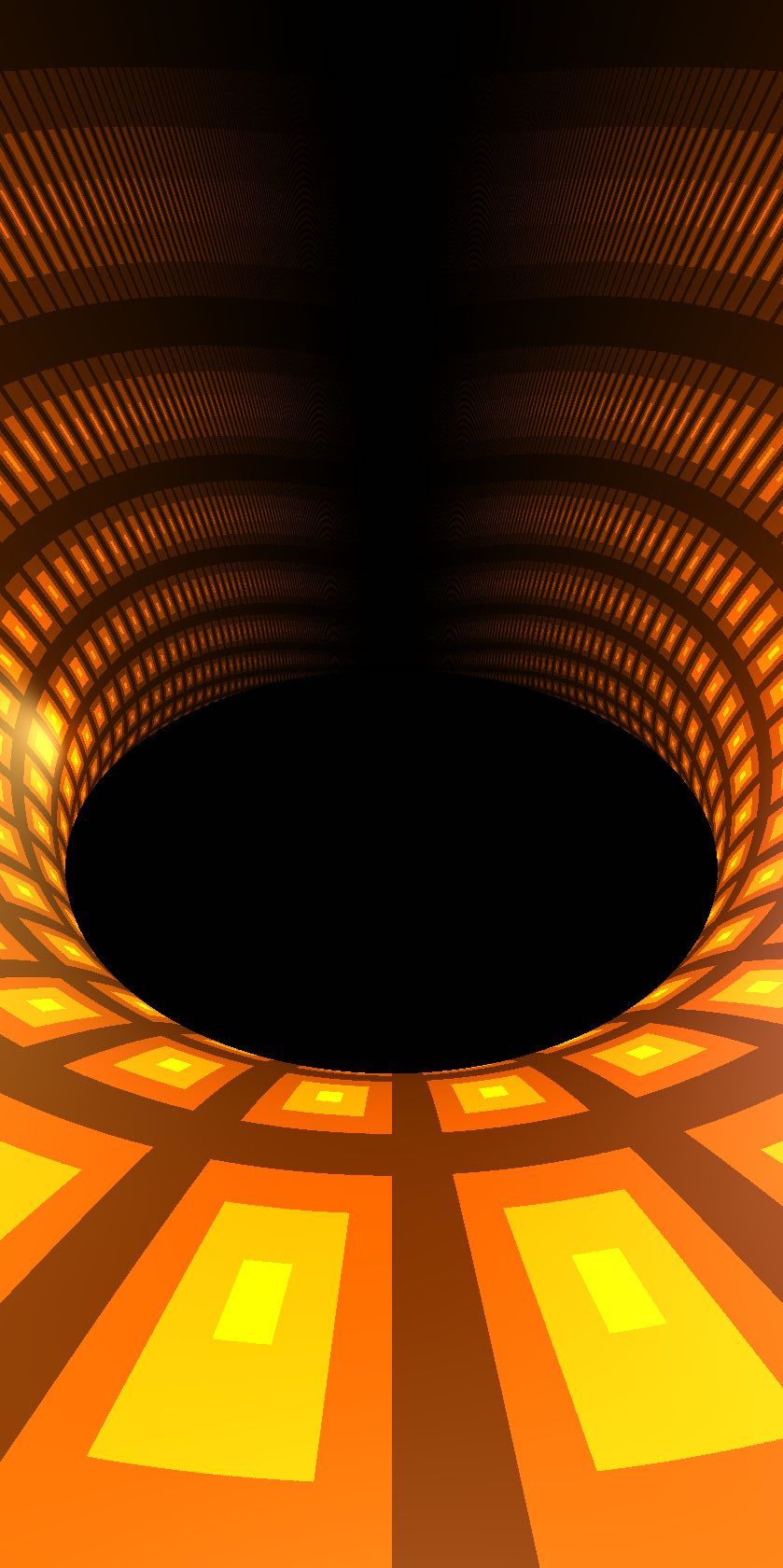}%
\label{Fig:SolPlanes2}%
}%
\quad
\subfloat[$H_z^+(1)\cup H_z^-(-1)$]{%
\includegraphics[width=0.30\textwidth]{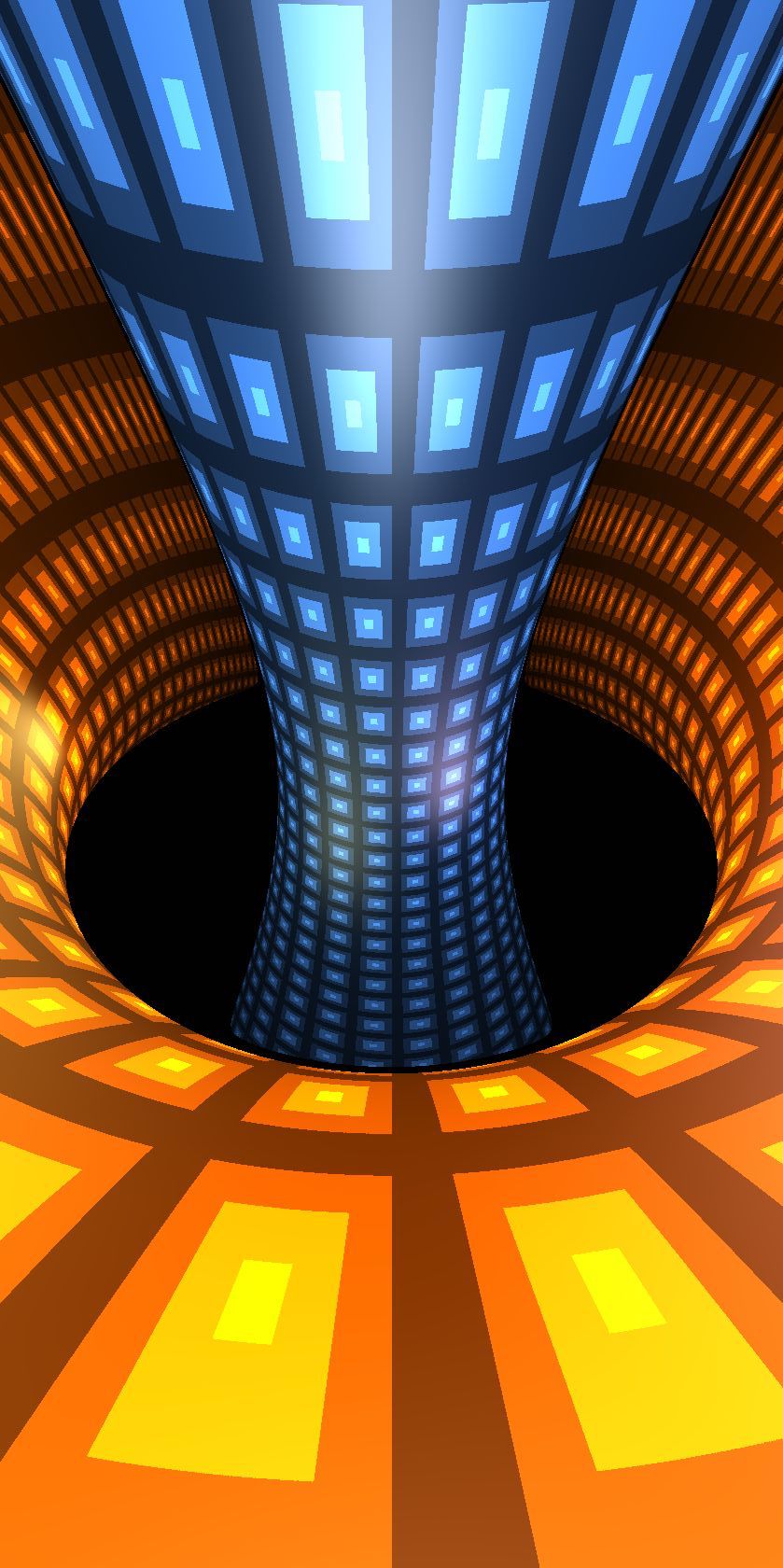}%
\label{Fig:SolPlanes}%
}%
\caption{Wide-angle views of horizontal half-spaces in Sol geometry.  The boundaries of these half-spaces are tiled by squares of side length $1/5$.}
\label{Fig:HalfSpaceHoriz}
\end{center}
\end{figure}

We can similarly compute the exact distance function to a half-space bounded by a hyperbolic plane in Sol.
Let $H_x^+(\alpha) = \{ x \geq \alpha\}$ and $H_x^-(\alpha) = \{ x \leq \alpha\}$.

\begin{lemma}
\label{Res:SolHypSheetDist}
	Fix a real number $\alpha$.
	For every point $p = [x,y,z,1]$ in $X$, we have 
	\begin{equation*}
		\sdf\left(p, H_x^+(\alpha)\right) = - \sdf\left(p, H_x^-(\alpha)\right) = \arcsinh\big((\alpha - x)e^{-z}\big).
	\end{equation*}
\end{lemma}

\begin{proof}
	Assume first that $p = o$ is the origin.
	We write the proof for $H^+_x(\alpha)$ with $\alpha > 0$. The other cases work in the same way.
	We claim that the distance from $o$ to $H_x^+(\alpha)$ is also the distance in the hyperbolic plane $U = \{y = 0\}$ from $o$ to the half plane $U^+(\alpha) = \{ x \geq \alpha \ \text{and}\ y = 0\}$.
	We have
	\begin{equation*}
		\dist_X \left(o, H_x^+(\alpha) \right) \leq \dist_U\left(o, U^+(\alpha)\right).
	\end{equation*}
	In order to prove the converse inequality, it suffices to show that the projection $X \to U$ sending $[x,y,z,1]$ to $[x,0,z,1]$ is $1$-Lipschitz.
	To see this, take two points $q$ and $q'$, and a geodesic $\gamma \from [0,T] \to X$ joining them. We write $\gamma(t) = [x(t), y(t), z(t),1]$. From the metric given in \refeqn{metric sol} we get 
	\begin{equation*}
		\begin{split}
			\dist(q,q') = L(\gamma) & = \int_0^T \sqrt{e^{-2z}\dot x^2 + e^{2z}\dot y^2 + \dot z^2}dt \\
			& \geq \int_0^T \sqrt{e^{-2z}\dot x^2 + \dot z^2}dt \\
			& = L(\pi \circ \gamma)
		\end{split}
	\end{equation*}
	where $L(\gamma)$ and $L(\pi \circ \gamma)$ stands for the length in $X$ of $\gamma$ and $\pi \circ \gamma$ respectively. Thus, $\dist(q,q') \geq \dist(\pi(q), \pi(q'))$.

	We now compute $\dist_U(o, U^+(\alpha))$.
	Recall that $U$ is isometric to the hyperbolic plane $\HH^2$.
	More precisely $[x,z]$ is a horocycle-based coordinate system of $\HH^2$: the distance between $p_1 = [x_1,0,z_1,1]$ and $p_2 = [x_2,0, z_2,1]$ is characterized by 
	\begin{equation*}
		\cosh \dist(p_1,p_2) = \cosh(z_1 - z_2) + \frac 12 e^{-(z_1+z_2)} (x_1 - x_2)^2.
	\end{equation*}
	One checks that the projection of $o$ onto $U^+(\alpha)$ is the point 
	\begin{equation*}
		\left[\alpha, 0, \frac 12\ln(1 + \alpha^2),1\right]
	\end{equation*}
	and 
	\begin{equation*}
		\dist_U\left(o, U^+(\alpha)\right) = \arcsinh(\alpha).
	\end{equation*}
	
	Assume now that $p = [x,y,z,1]$ is an arbitrary point.
	There is a unique element $L$ of Sol sending $o$ to $p$.
	Observe that $L^{-1}$ maps $H_x^+(\alpha)$ to $H_x^+(\alpha')$ where $\alpha' = (\alpha-x)e^{-z}$.
	The result then follows from the previous discussion.
\end{proof}

We can define the half-spaces $H_y^\pm(\alpha)$ as we did for $H^\pm_x(\alpha)$.
Using the fact that the isometry $S_2$ fixing the origin sends $[x,y,z,1]$ to $[y,x,-z,1]$ we get the following statement.
\begin{lemma}
	Fix a real number $\alpha$.
	For every point $p = [x,y,z,1]$ in $X$, we have 
	\begin{equation*}
		\sdf\left(p, H_y^+(\alpha)\right) = - \sdf\left(p, H_y^-(\alpha)\right) = \arcsinh\big((\alpha - y)e^{z}\big)
	\end{equation*}
\end{lemma}

In \reffig{HalfSpaceHyp}, we use these signed distance functions to draw half-spaces with hyperbolic boundary, patterned with square tilings. Combining these signed distance functions with boolean operations, we can make tubes around vertical geodesics with square cross-sections. See \reffig{SolVerticalTubes}.

\begin{figure}[htbp]
\begin{center}
\subfloat[$H_x^+(1)$]{%
\includegraphics[width=0.30\textwidth]{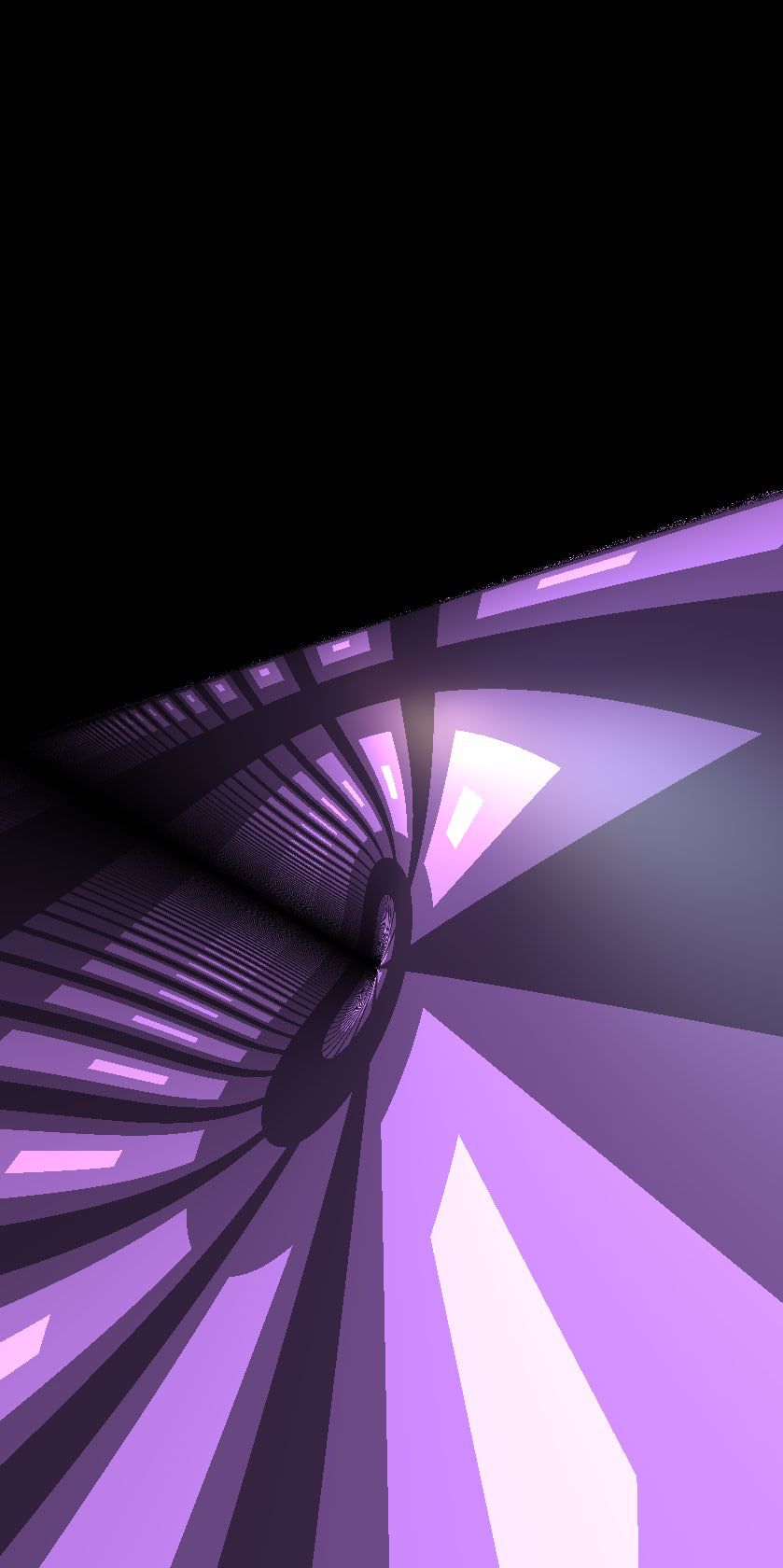}%
\label{Fig:SolPlanesHyp1}%
}%
\quad
\subfloat[$H_y^-(-1)$]{%
\includegraphics[width=0.30\textwidth]{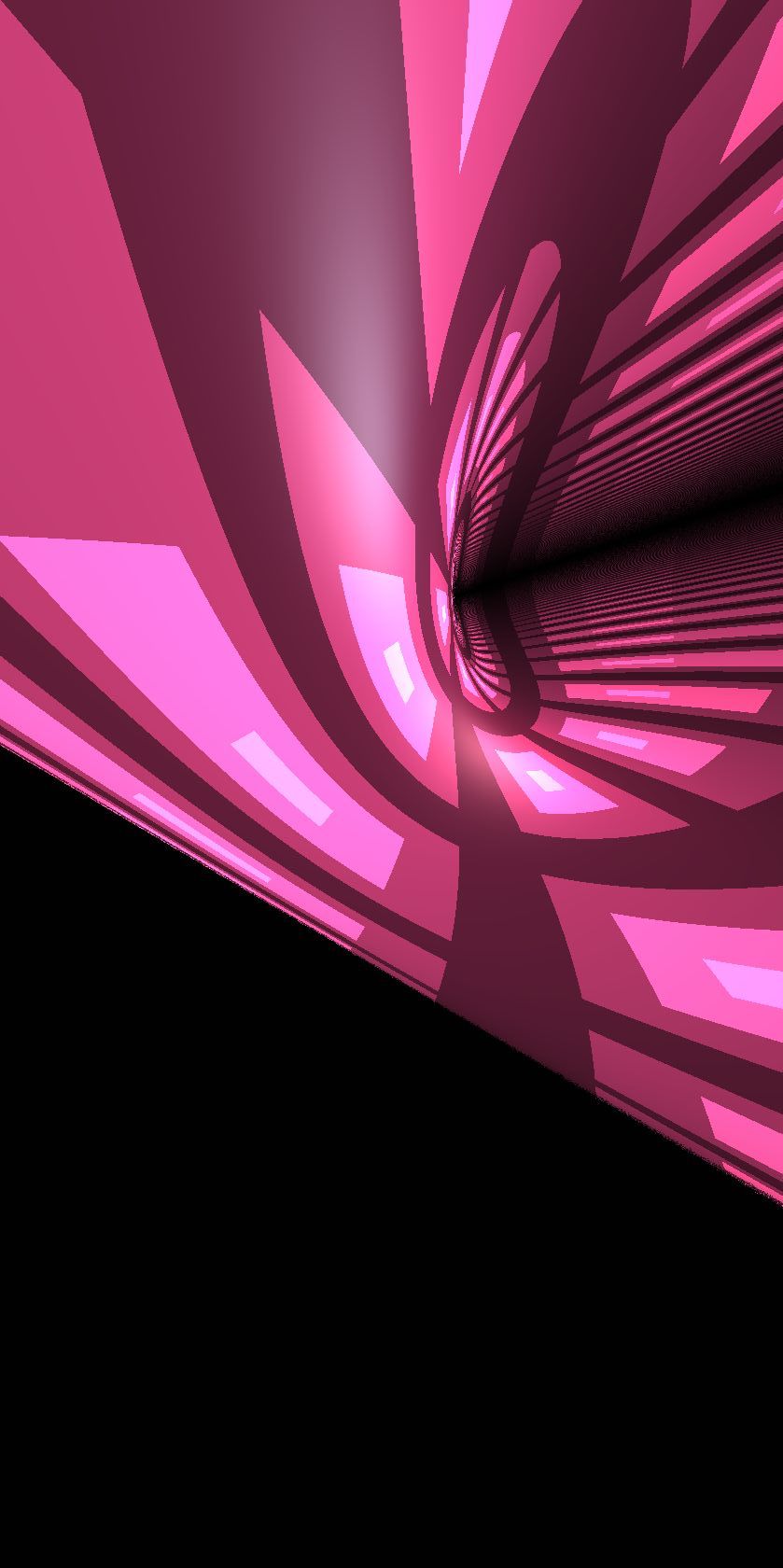}%
\label{Fig:SolPlanesHyp2}%
}%
\quad
\subfloat[$H_x^+(1)\cup H_y^-(-1)$]{%
\includegraphics[width=0.30\textwidth]{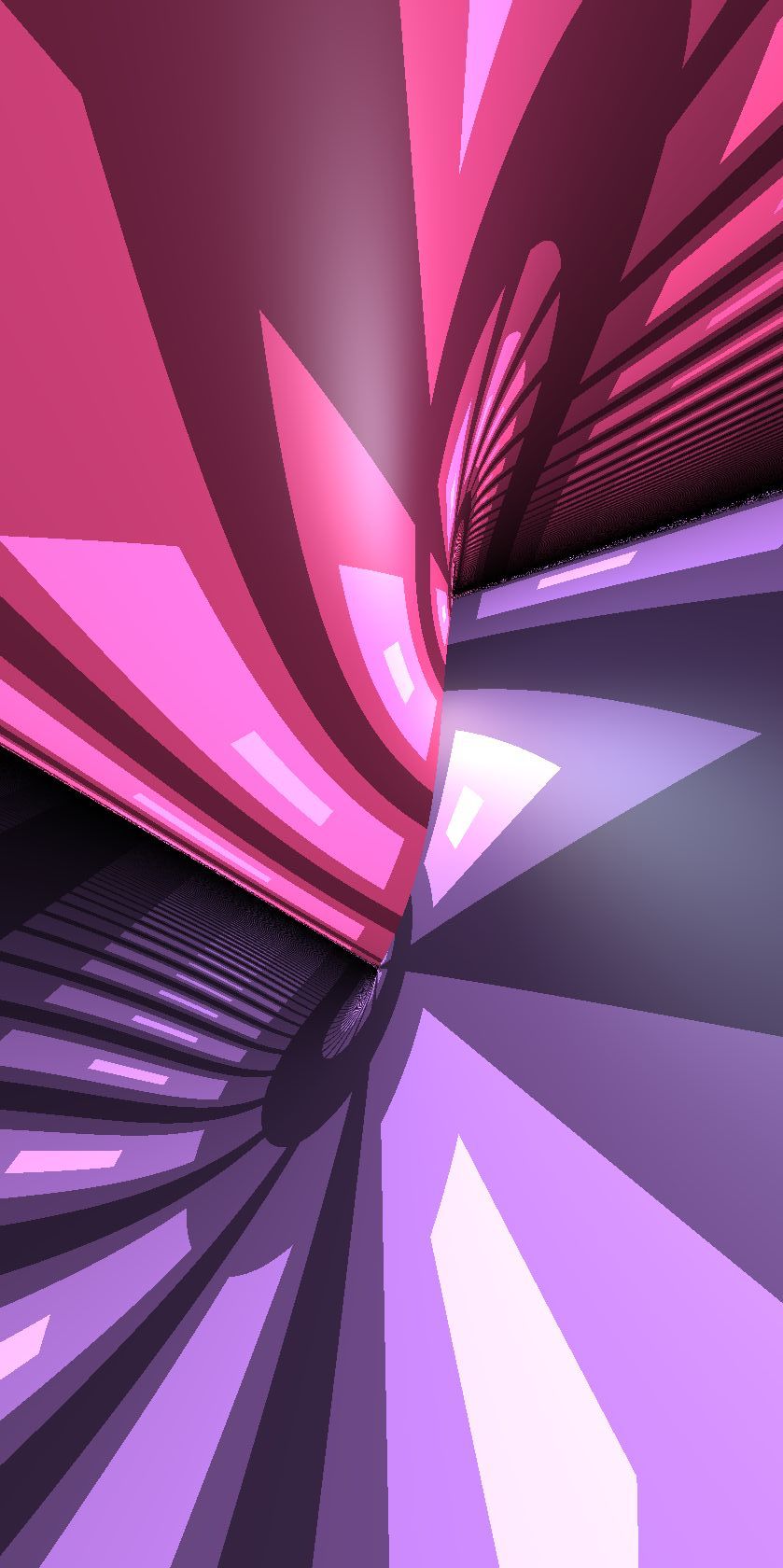}%
\label{Fig:SolPlanesHyp3}%
}%
\caption{Wide-angle view of half-spaces with hyperbolic plane boundary in Sol geometry.  The boundaries are tiled by quadrilaterals formed from a family geodesics parallel to the $z$-axis and the families of orthogonal horocycles.  The horocycles are evenly spaced, with distance one between neighbors.}
\label{Fig:HalfSpaceHyp}
\end{center}
\end{figure}

\begin{figure}[htbp]
\begin{center}
\subfloat[Tubes around vertical geodesics with square cross-sections.]{%
\includegraphics[width=0.30\textwidth]{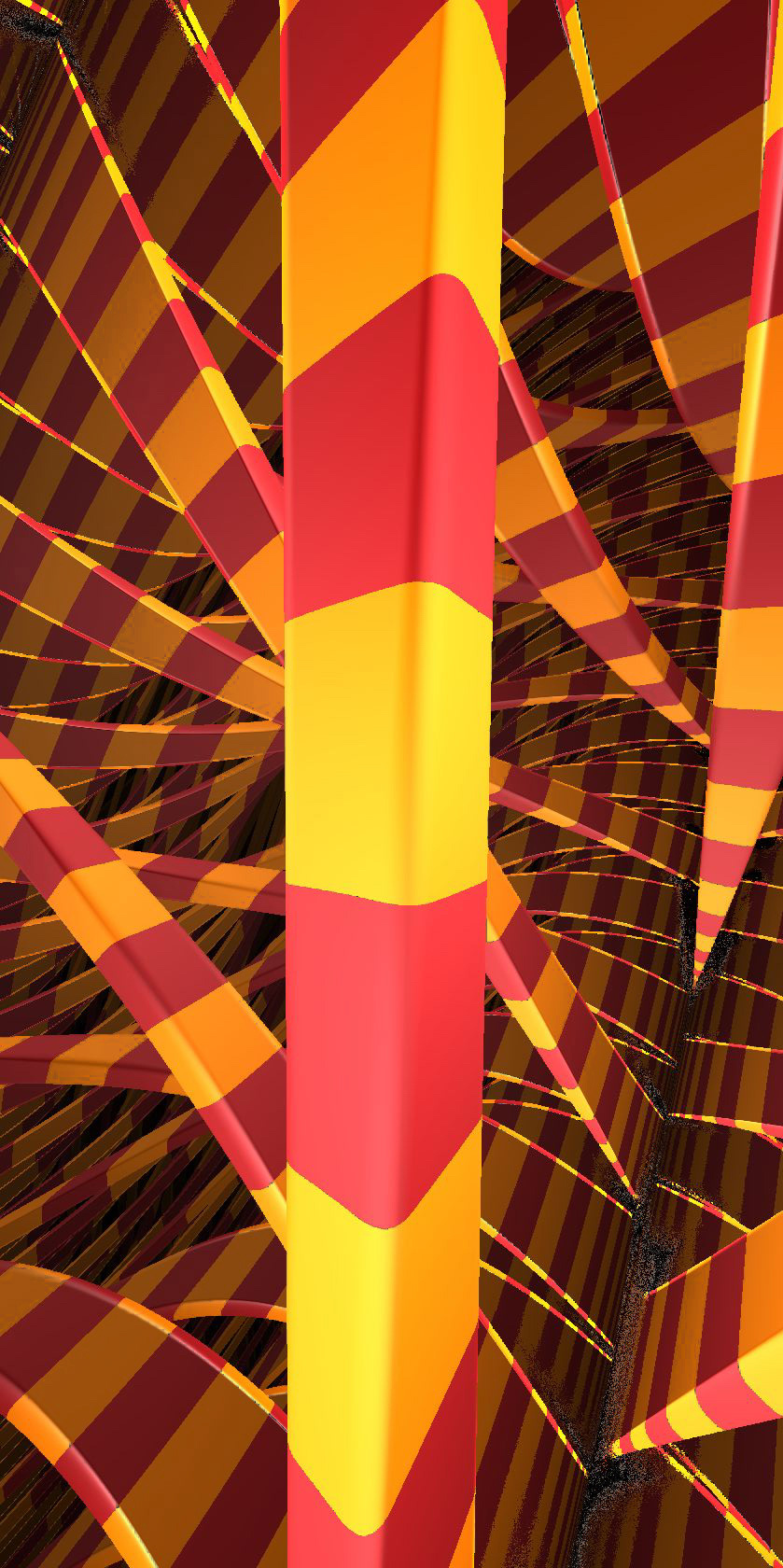}%
\label{Fig:SolVerticalTubes}%
}%
\quad
\subfloat[A cube with sidelength 3.]{%
\includegraphics[width=0.30\textwidth]{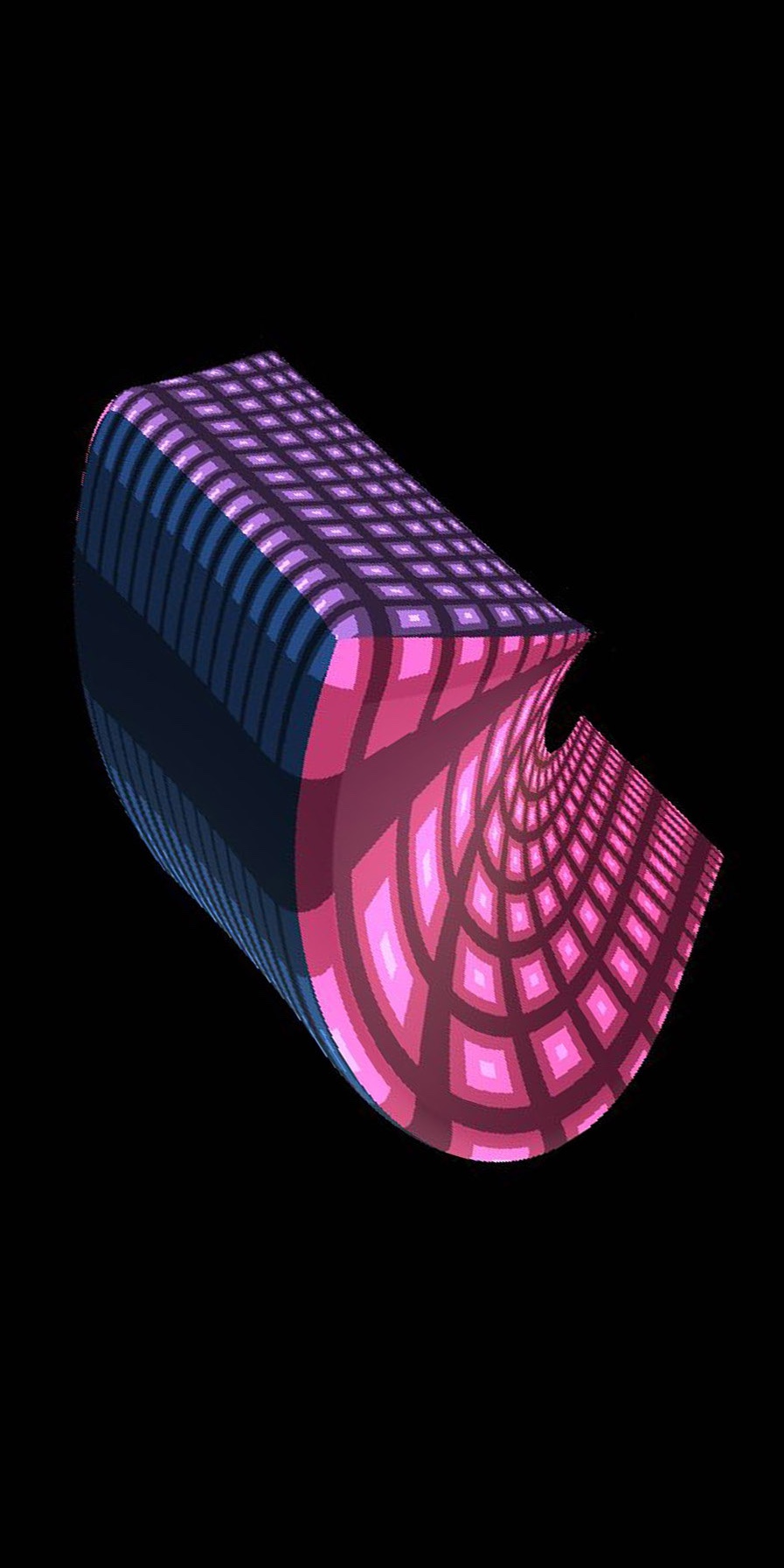}%
\label{Fig:SolCube}%
}%
\quad
\subfloat[The same (single) cube as in \reffig{SolCube}, viewed from a distance.
 ]{%
\includegraphics[width=0.30\textwidth]{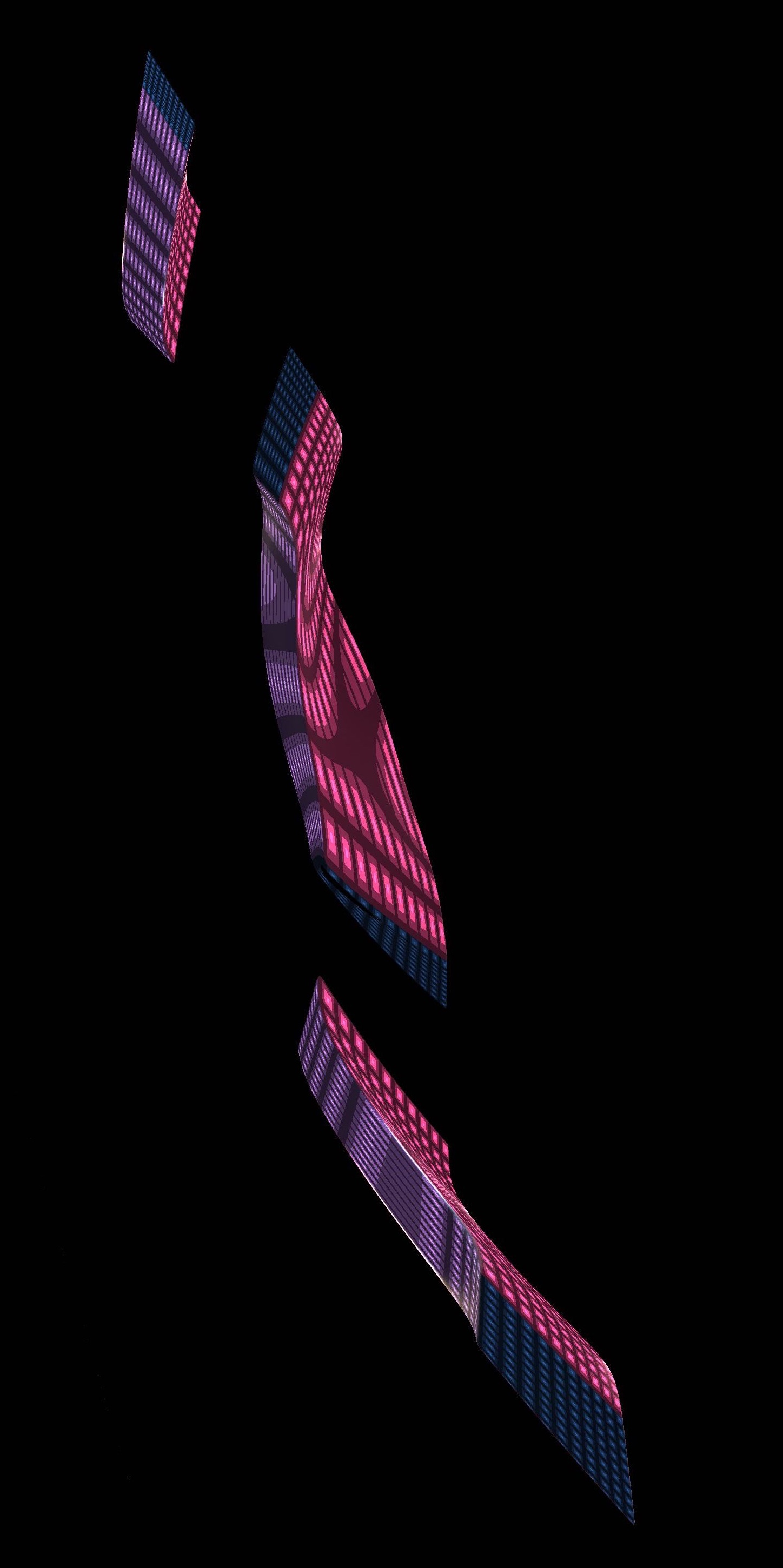}%
\label{Fig:SolCube2}%
}%
\\
\subfloat[A lattice of cubes, dense enough that anomalies seen in \reffig{SolCube2} are mostly hidden from view.]{%
\includegraphics[width=0.97\textwidth]{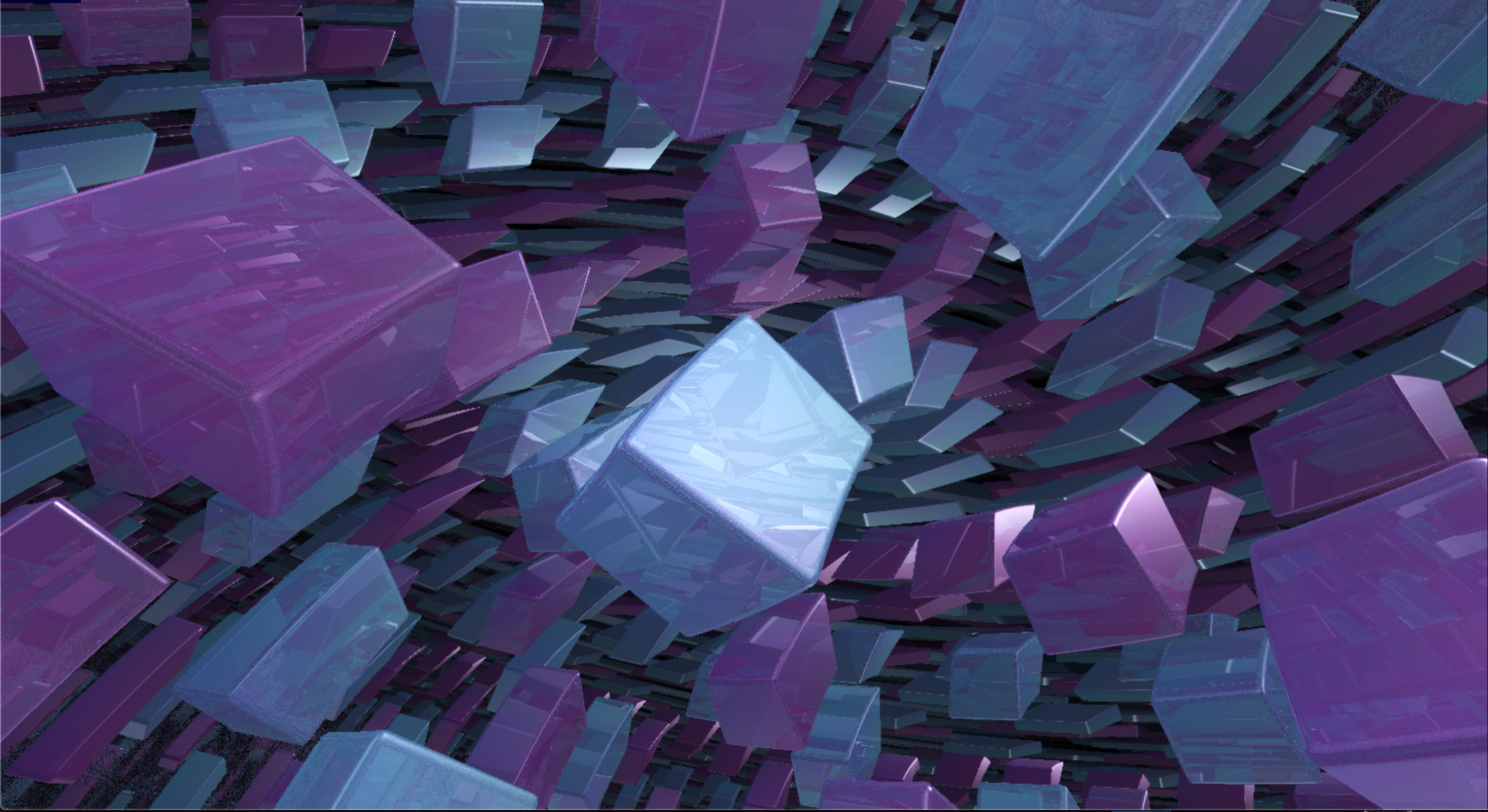}%
\label{Fig:SolMoreCubes2}%
}%
\caption{Scenes made from half-spaces with boolean operations.}
\label{Fig:SolExactTubes}
\end{center}
\end{figure}

\subsection{Distance to horizontal axis-aligned solid cylinders}
Following the same strategy as in \refsec{SolHalfSpaces}, we compute the signed distance function for certain solid cylinders.
Let $c_x \colon \RR \to X$ be the curve given by $c_x(t) = [t,0,0,1]$.
Note that $c_x$ is \emph{not} a geodesic of $X$, but it is a one-parameter subgroup of Sol.

\begin{lemma}
\label{Lem:SolHypDist}
	For every point $p = [x,y,z,1]$ in $X$, we have 
	\begin{equation*}
		\cosh \dist(p, c_x) = \cosh z + \frac 12 e^zy^2.
	\end{equation*}
\end{lemma}

\begin{proof}
	Since $c_x$ is invariant under translations along the $x$-axis (which are isometries of $X$), we can assume that $p$ has the form $p = [0,y,z,1]$.
	Following the argument given in the proof of Lemma~\ref{Res:SolHypSheetDist}, we observe that $\dist (p, c_x) = \dist(p, o)$.
	Using the distance formula in the hyperbolic plane $\{x = 0\}$, we get the result.
\end{proof}

Let $C_x(r)$ be the solid cylinder of radius $r$ around $c_x$. That is, $C_x(r)$ is the set of point $q \in X$ such that $\dist(q,c_x) \leq r$.
It follows from \reflem{SolHypDist} that the signed distance function $\sigma \colon X \to \RR$ for $C_x(r)$ is 
\begin{equation*}
	\sigma(p) = \arccosh \left(\cosh z + \frac 12 e^zy^2\right) - r.
\end{equation*}
Similarly, we define the solid cylinder of radius $r$ around the curve $c_y$ given by $c_y(t) = [0,t,0,1]$.
The signed distance function $\sigma \colon X \to \RR$ for $C_y(r)$ is 
\begin{equation*}
	\sigma(p) = \arccosh \left(\cosh z + \frac 12 e^{-z}x^2\right) - r.
\end{equation*}

Using the elements of Sol, we can translate the solid cylinders $C_x(r)$ and $C_y(r)$ to get signed distance functions for solid cylinders around any translate of the $x$- and $y$-axes. See \reffig{SolExactCylinders}.

\begin{figure}[htbp]
\begin{center}
\subfloat[Around translates of the $x$- and $y$-axes (exact sdfs).]{%
\includegraphics[width=0.30\textwidth]{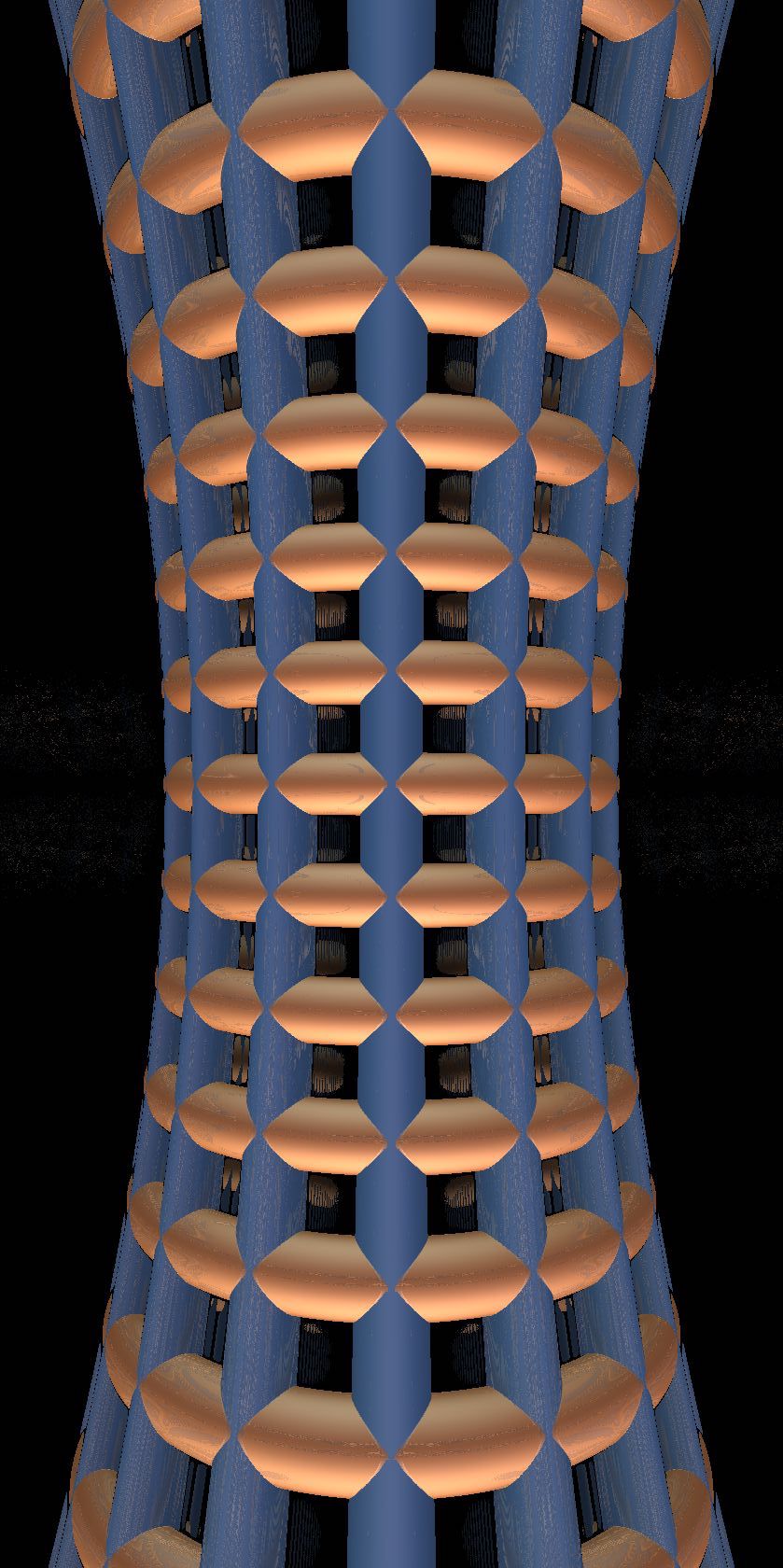}%
\label{Fig:SolExactCylinders}%
}%
\quad
\subfloat[Around translates of the $x$- and $y$-axes (approximations).]{%
\includegraphics[width=0.30\textwidth]{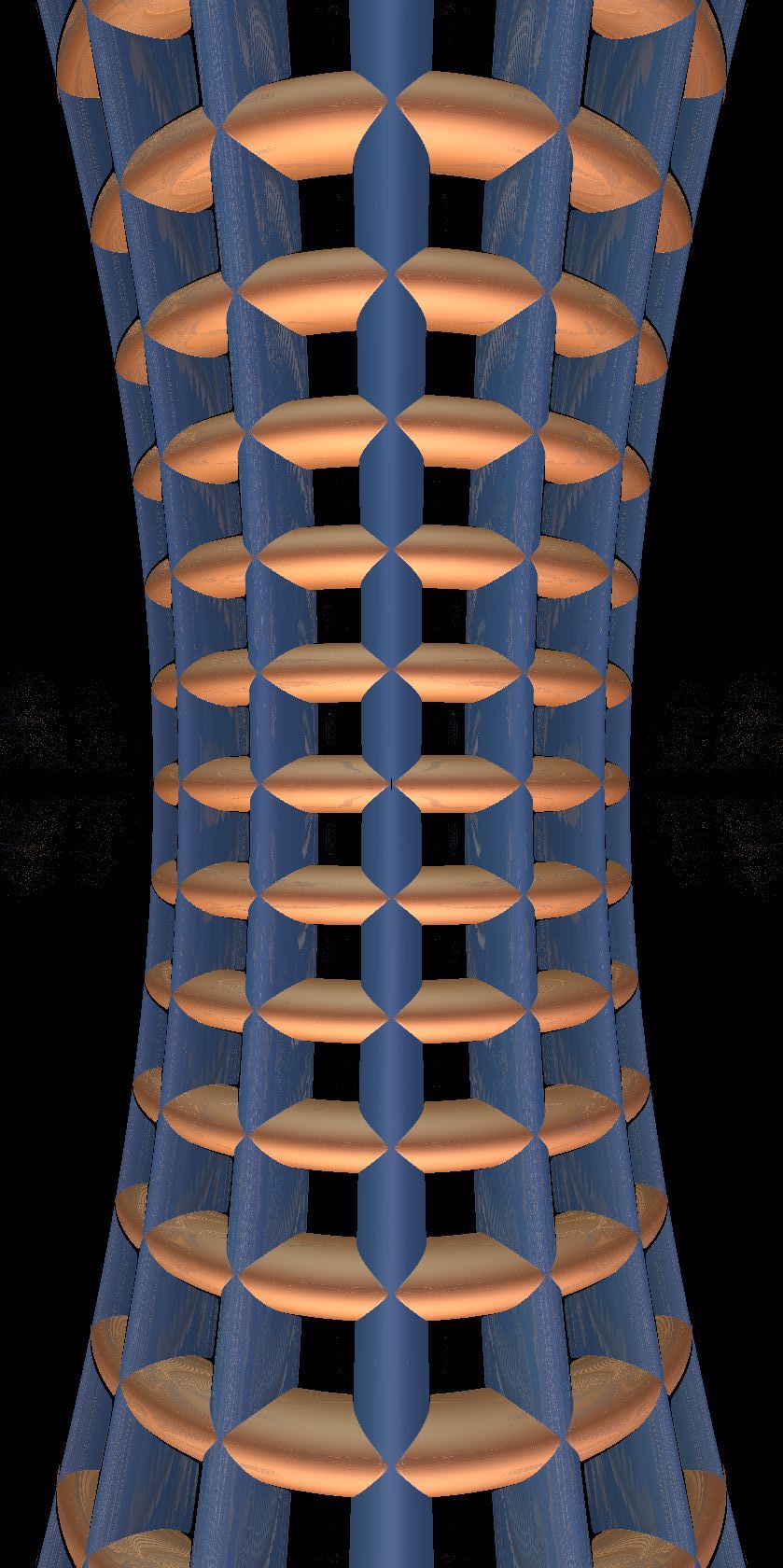}%
\label{Fig:SolFakeCylindersComparison}%
}%
\quad
\subfloat[Around geodesics in horizontal planes.]{%
\includegraphics[width=0.30\textwidth]{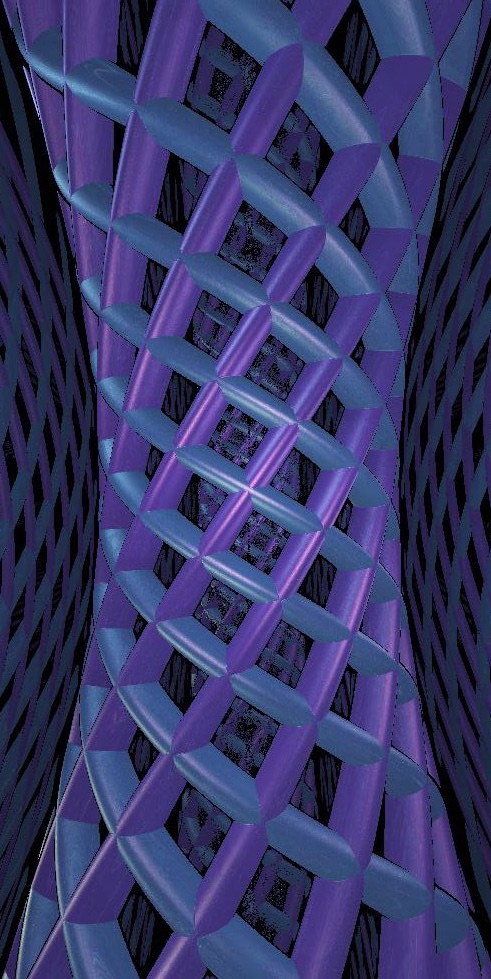}%
\label{Fig:SolFakeCylinders}%
}%
\\
\subfloat[Around horocyclic coordinate lines.]{%
\includegraphics[width=0.97\textwidth]{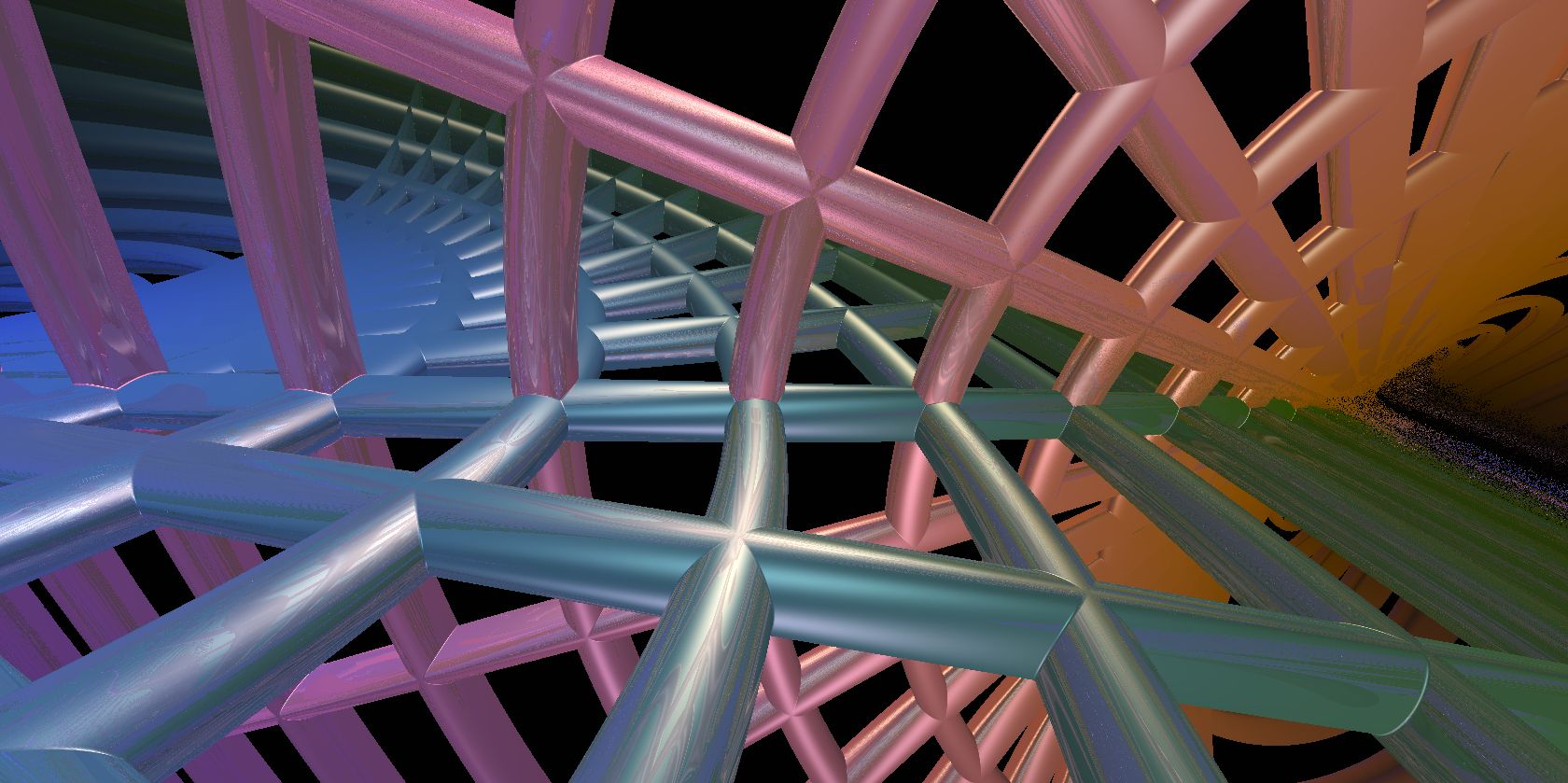}%
\label{Fig:SolHypPlanesGrid}%
}
\caption{Solid cylinders.}
\label{Fig:SolCylinders}
\end{center}
\end{figure}

\subsection{Approximating balls and more general solid cylinders}
\label{Sec:SolPseudo}
Given a point $p = [x,y,z,1]$, we approximate its distance to the origin with the function 
\begin{equation*}
	\sigma(p) = \sqrt{ e^{-2z}x^2 + e^{2z}y^2 + z^2},
\end{equation*}
rescaled by homotheties of the domain and co-domain.
This can be used to render decent ``pseudo-balls'', see Figures \ref{Fig:pseudosol} and \ref{Fig:SolEarthLattice}. It is not currently clear to us whether this function can be used to build a distance underestimator for correct balls.

\begin{figure}[htbp]
\begin{center}
\subfloat[Exact ball of radius~$1$.]{%
\includegraphics[width=0.33\textwidth]{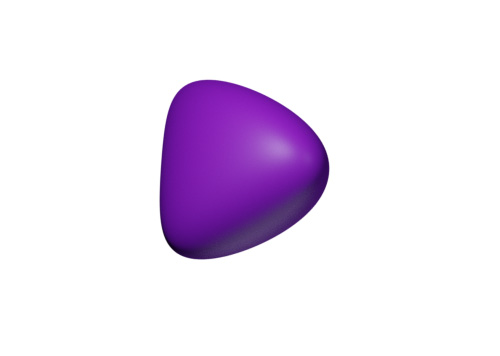}%
\label{Fig:SolExact1}%
}%
\subfloat[Exact ball of radius~$2$.]{%
\includegraphics[width=0.33\textwidth]{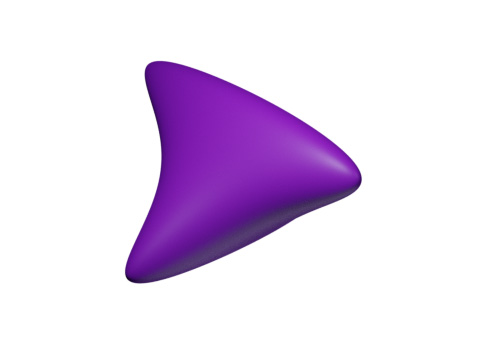}%
\label{Fig:SolExact2}%
}%
\subfloat[Exact ball of radius~$3$.]{%
\includegraphics[width=0.33\textwidth]{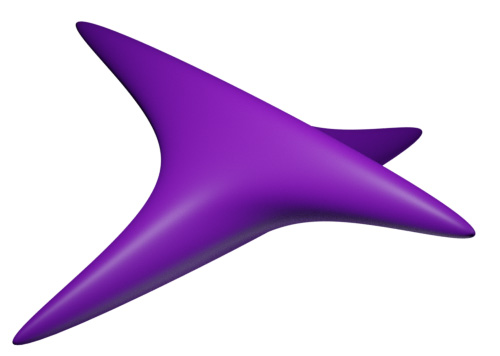}%
\label{Fig:SolExact3}%
}\\
\subfloat[Level set $\sigma = 0.6$.]{%
\includegraphics[width=0.33\textwidth]{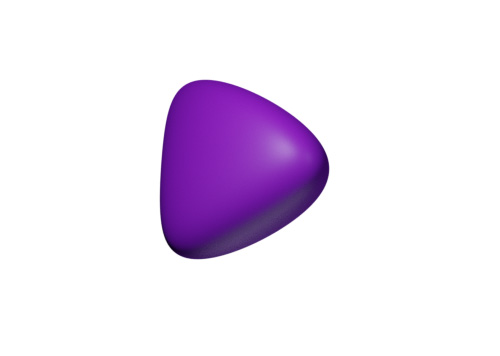}%
\label{Fig:SolFake1}%
}%
\subfloat[Level set $\sigma = 1.34$.]{%
\includegraphics[width=0.33\textwidth]{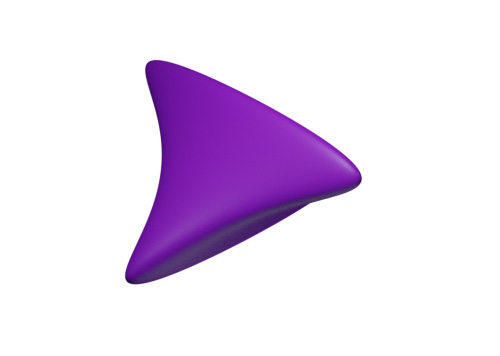}%
\label{Fig:SolFake2}%
}
\subfloat[Level set $\sigma = 2.16.$]{%
\includegraphics[width=0.33\textwidth]{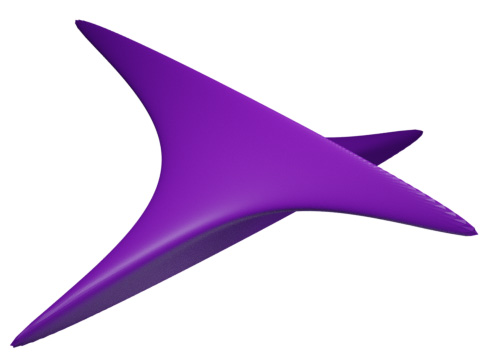}%
\label{Fig:SolFake2}%
}
\caption{
	Extrinsic comparison of exact and pseudo-balls.
	The objects have been rescaled so that they all have approximately the same size. 
}
\label{Fig:pseudosol}
\end{center}
\end{figure}

We can similarly produce ``solid pseudo-cylinders,'' approximating the distance from a point to an orbit of a one-parameter subgroup transverse to a plane (the horizontal plane or either hyperbolic plane). The idea is to move a point under the one-parameter subgroup to put it in the plane, and then calculate a signed distance function there. If distances are difficult to calculate (either theoretically or practically) even when restricted to the plane, then we can cheat further by measuring, say, euclidean distance in the model space.

\reffig{SolFakeCylindersComparison} shows solid pseudo-cylinders around the translates of the $x$- and $y$-axes. This compares well with the exact solid cylinders shown in \reffig{SolExactCylinders}. In \reffig{SolFakeCylinders} we draw solid pseudo-cylinders around the geodesics $x=\pm y$ and their translates. Note that these are the only geodesics contained in the $xy$-plane. In \reffig{SolHypPlanesGrid} we reproduce the two hyperbolic planes of \reffig{HalfSpaceHyp}, represented by grids of solid cylinders. The horocycles in each grid are drawn with exact signed distance functions; for the geodesics we use solid pseudo-cylinders.

\subsection{Direction to a point}
Although it is certainly possible to do so, we did not try to numerically compute the exact direction of geodesics joining two given points in Sol.
Recall that this data is only needed to compute lighting pairs for physically correct illumination as in \refsec{Lighting}.
As we explained in  \refsec{Cheating}, we choose instead a more-or-less arbitrary, continuously varying direction field:
if $s$ is a point of the scene $S$ and $q$ is the position of the light, then we run all the computations in the Phong model as if the direction from $s$ to $q$ were given by the straight line between $s$ and $q$ in the ambient space $\RR^4$ containing our model $X$.

\subsection{Discrete subgroups and fundamental domains}
\label{Sec:SolDiscreteSubgroupsFundDoms}

The classification of Sol manifolds is given in \cite[Theorem~4.17]{Scott}. Every Sol manifold is a surface bundle over a one-dimensional orbifold. In particular,
Sol can be seen as the universal cover of the suspension $M$ of a regular two-torus $T$ by an Anosov homeomorphism.

The fundamental group $\Gamma$ of $M$ provides a lattice in $X$.
We explain here with a concrete example how to construct a fundamental domain $D$ for the action of $\Gamma$ on $X$.

To avoid any confusion, we denote by $[u_1,u_2]$ the coordinates of a point in the universal cover $\RR^2$ of the two-torus $T$.
The fundamental group $\pi_1(T) \cong \ZZ^2$ acts on $\RR^2$ by integer translations.
Let $f$ be the Anosov homeomorphism of $T$ acting on $\RR^2$ as the matrix 
\begin{equation*}
	\left[\begin{array}{cc}
		2 & 1\\
		1 & 1
	\end{array}\right].
\end{equation*}
Let $M$ be the mapping torus of $T$ with monodromy $f$.
Its fundamental group $\Gamma$ is given by the presentation
\begin{equation*}
	\Gamma = \left< A_1,A_2,B \mid [A_1,A_2] = 1, BA_1B^{-1} = A_1^2A_2, BA_2B^{-1}=A_1A_2\right>.
\end{equation*}
Here $A_1$ and $A_2$ are the standard generators of $\ZZ^2$, while the conjugation by $B$ is the automorphism of $\ZZ^2$ induced by $f$.
As in Nil, $\Gamma$ is generated by $A_1$ and $B$.
Nevertheless, is it more convenient to keep three generators, as they correspond to translations in three independent directions.

We identify the universal cover $\cover M$ of $M$ with $\RR^3$, equipped with coordinates $[u_1,u_2,u_3]$. Here the set $\{ u_3 = 0\}$ corresponds to a copy of $\cover T$ inside $\cover M$. The generators $A_1$ and $A_2$ act by translation along $u_1$ and $u_2$, while $B$ translates along $u_3$ and applies $f$ to the orthogonal plane. 

The next step is to identify $X$ with $\cover M$.
Let $b$ be the point of Sol whose coordinates in $X$ are $b = [0,0,\tau,1]$. (The value of $\tau > 0$ will be determined later.)
We require that under our identification, the translation by $b$ in $X$ becomes the action of $B$ on $\cover M$.
Observe that $b$ dilates the $x$-axis while contracting the $y$-axis.
Thus we need to identify the $x$-direction (respectively $y$-direction) of $X$ with the expanding (respectively contracting) direction of $f$.

The matrix defining $f$ has two eigenvalues, namely $\phi^2$ and $\phi^{-2}$, where $\phi = (1 + \sqrt 5)/2$ is the golden ratio.
The corresponding eigenvectors are
\begin{equation*}
	v_+ = [\phi, 1], \quad \text{and}\quad v_- = [-1, \phi].
\end{equation*}
We now define a homeomorphism $h \colon X \to \cover M$ as the restriction to $X$ of the linear map $\RR^4 \to \RR^3$ given by the matrix
\begin{equation*}
	\left[\begin{array}{cccc}
		\phi & -1 & 0 & 0\\
		1 & \phi & 0 & 0 \\
		0 & 0 & \tau^{-1} & 0\\
	\end{array}\right],
\end{equation*}
where we now set $\tau = 2 \ln \phi$.
In addition, we write $a_1$ and $a_2$ for the elements of Sol whose coordinates in $X$ are
\begin{equation*}
	a_1 = \left[\frac \phi{\phi+2},-\frac 1 {\phi+2},0,1\right] \quad \text{and}\quad a_2 = \left[\frac 1 {\phi+2},\frac \phi{\phi+2},0,1\right].
\end{equation*}
It follows from our construction that the map $h$ conjugates the translation by $a_1$ (respectively $a_2$, $b$) in $X$ to the action of $A_1$ (respectively $A_2$, $B$) on $\cover M$.
A fundamental domain $D$ for the action of $\Gamma$ on $X$ is the image under $h^{-1}$ of the cube $[-1/2,1/2]^3 \subset \cover M$. That is, 
\begin{equation*}
	D = \left\{\left. \left[ \frac {u_1\phi + u_2}{\phi+2}, \frac {-u_1 + u_2\phi}{\phi+2}, u_3 \tau, 1 \right] \right| \, u_1,u_2,u_3 \in [-1/2,1/2] \right\}.
\end{equation*}

Our model $X$ for Sol is also a projective model.
The fundamental domain $D$ can be seen as the intersection of a collection of half-spaces $H_1^\pm$, $H_2^\pm$, $H_3^\pm$ as described in \refsec{TeleportingProjective}.
Here
\begin{align*}
	H_1^- & = \left\{ \left. \left[ \frac {u_1\phi + u_2}{\phi+2}, \frac {-u_1 + u_2\phi}{\phi+2}, u_3 \tau, 1 \right] \right|\, u_1 \geq -1/2,\ u_2,u_3 \in \RR\right\}, \\
	H_1^+ & = \left\{ \left. \left[ \frac {u_1\phi + u_2}{\phi+2}, \frac {-u_1 + u_2\phi}{\phi+2}, u_3 \tau, 1 \right] \right|\, u_1 \leq 1/2,\ u_2,u_3 \in \RR\right\}.
\end{align*}
The half-spaces $H_2^\pm$, $H_3^\pm$ are defined in a similar way.
The teleporting algorithm has two main steps.
Let $p = [x,y,z,1]$ be a point in $X$.
\begin{enumerate}
	\item If $p$ does not belong to $H_3^-$ (respectively $H_3^+$), then we move it by $b$ (respectively $b^{-1}$).
	After finitely many steps, the new point $p$ lies in $H_3^- \cap H_3^+$.
	\item Once this is done, if $p$ does not belong to $H_1^-$ (respectively $H_1^+$, $H_2^-$, $H_2^+$), then we translate it by $a_1$ (respectively $a_1^{-1}$, $a_2$, $a_2^{-1}$).
	Note that this does not change the $z$-coordinate of $p$.
	Since $a_1$ and $a_2$ commute, we don't pay attention to the order in which we perform these operations.
	After finitely many steps, the new point $p$ belongs to $D$.
\end{enumerate}

In \reffig{SolWireframeCubes}, we draw a lattice of cubes in a neighborhood of the $xy$-plane. The center of each cube is at a vertex of the tiling of the plane $\cover{T}$ corresponding to the action of the subgroup of $\Gamma$ generated by $A_1$ and $A_2$. \reffig{SolEarthLattice} shows the inside view of an Anosov torus bundle, with a ball textured as the Earth for the scene. \reffig{SolTiling} shows the same manifold, with the complement of three solid pseudo-cylinders around the curves $[t \phi, -t, 0, 1]$, $[t, t \phi, 0, 1]$, and $[0,0,t,1]$ as the scene.

\begin{figure}[htbp]
\centering
\subfloat[Cubes on the $xy$-plane.]{
\includegraphics[width=0.90\textwidth]{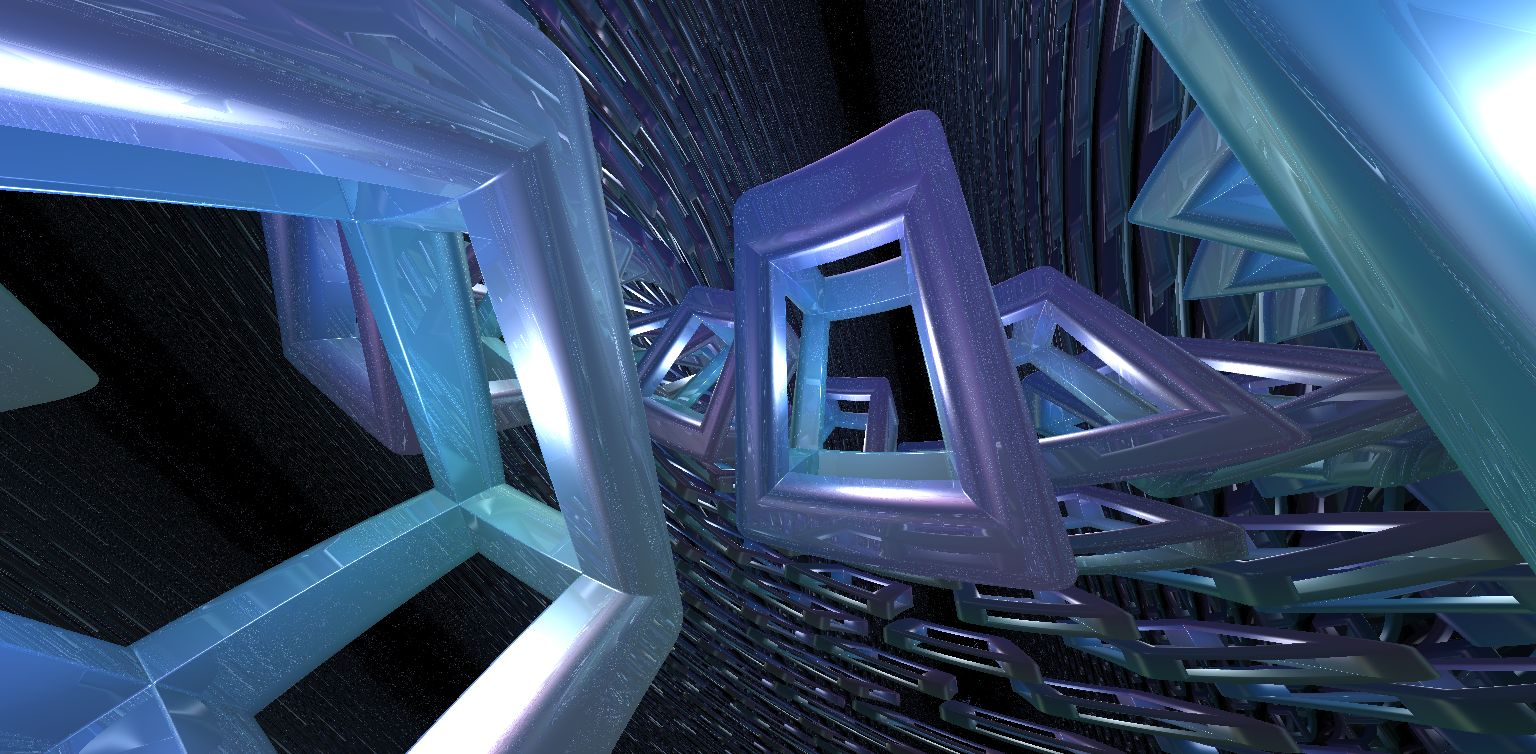}
\label{Fig:SolWireframeCubes}
}\
\subfloat[An Anosov torus bundle.]{
\includegraphics[width=0.90\textwidth]{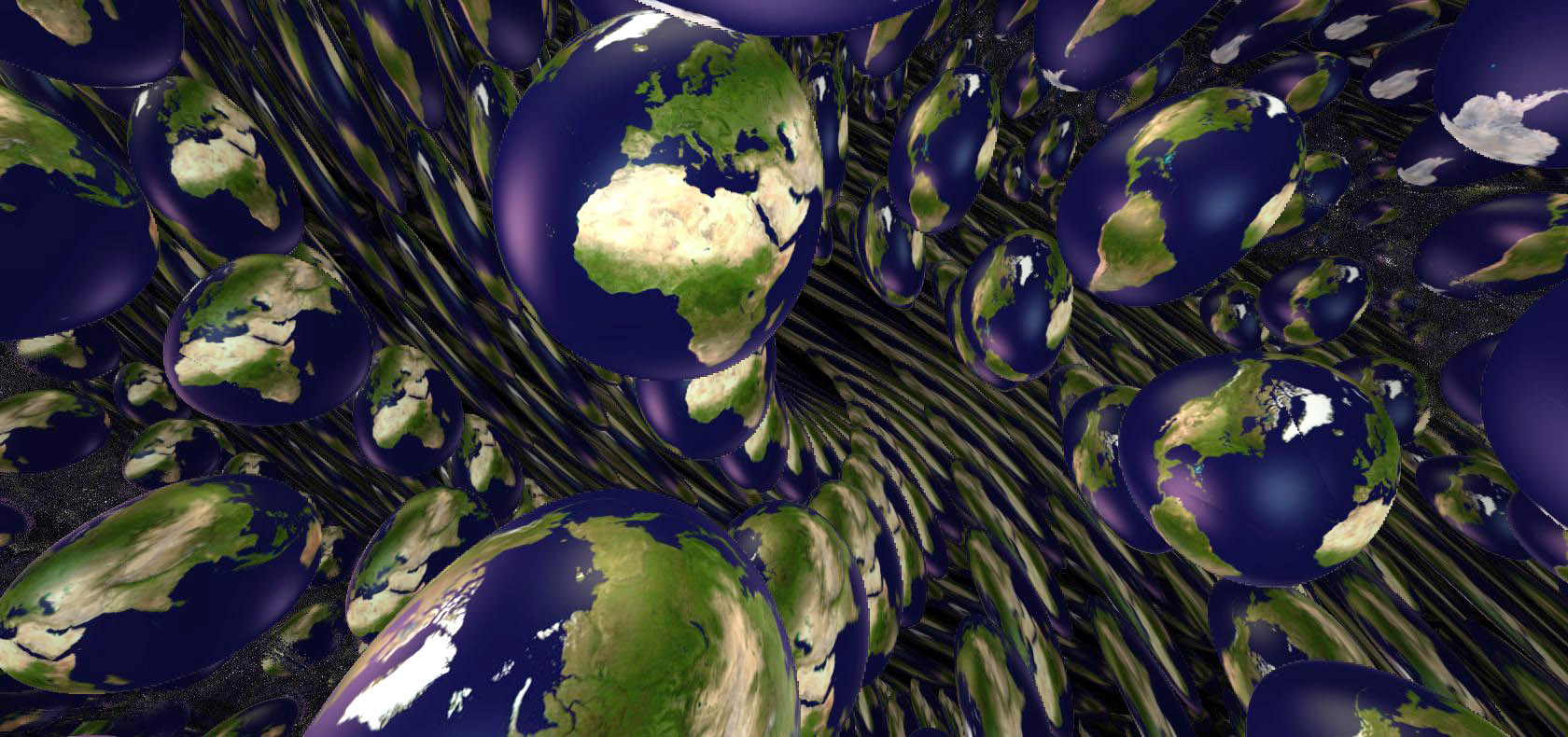}
\label{Fig:SolEarthLattice}
}\
\subfloat[An Anosov torus bundle.]{
\includegraphics[width=0.90\textwidth]{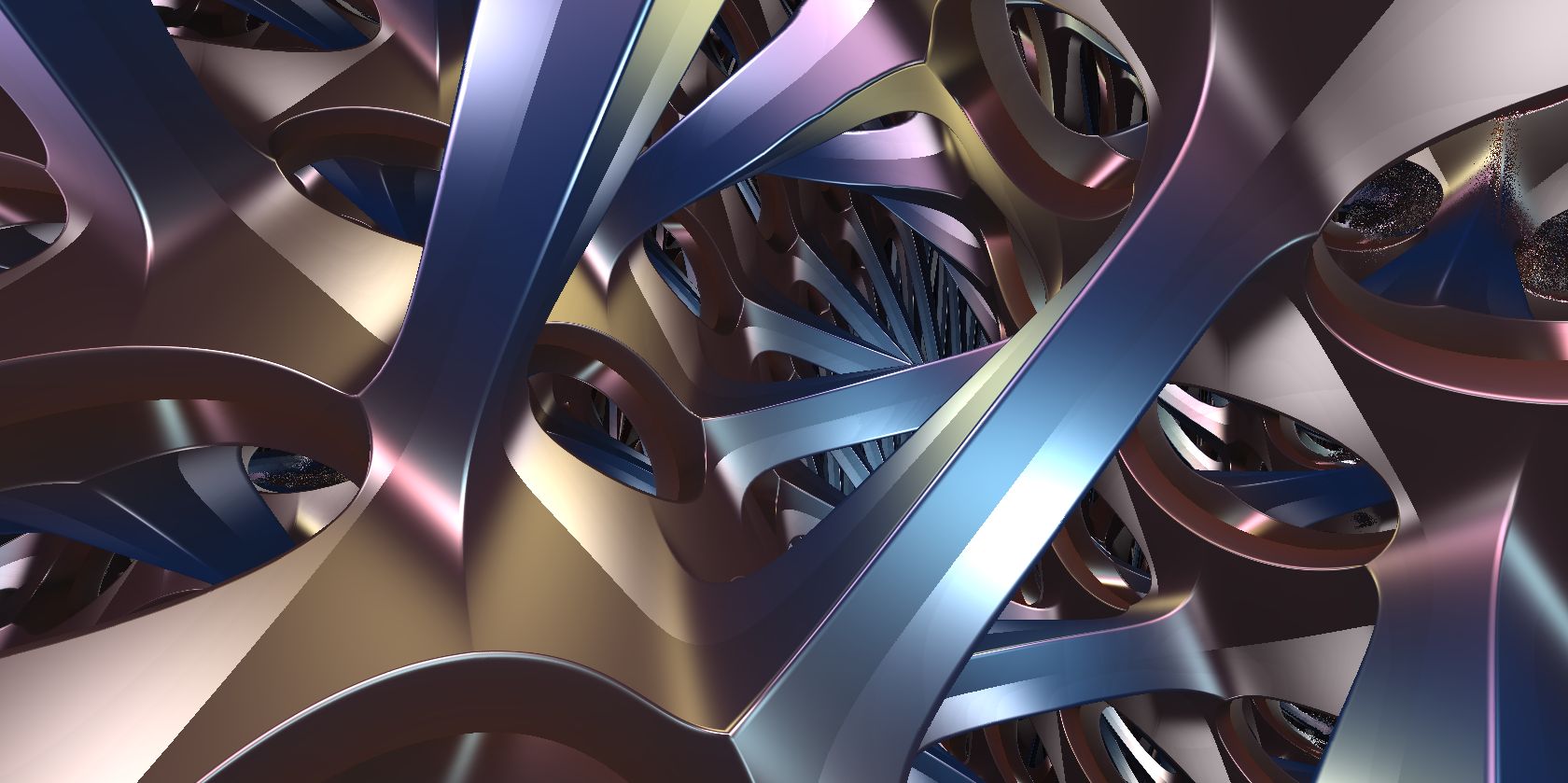}
\label{Fig:SolTiling}
}\
\caption{Sol Geometry.}
\label{Fig:SolExamples}
\end{figure}

\section{Future directions}

\subsection{Virtual reality}

As mentioned in \refsec{Stereoscopic}, there are serious problems that must be addressed before we can use stereoscopic vision to give the user depth cues in a virtual reality experience.

\subsection{Sol}

Some elements of our work are still incomplete for Sol geometry, namely correct lighting, and correct signed distance functions (or even distance underestimators) for balls.
One of the difficulties is that we do not yet have an efficient method to compute the lengths and directions of the geodesics from the origin $o$ to an arbitrary point $p$.

For Nil and $\SLR$, we used the rotation-invariance of our model to build a one-to-one correspondence between those geodesics and the zeros of a function $\phi \to \chi(\phi)$ (depending on $p$), see Sections~\ref{Sec:ExactGeodesicsNil} and \ref{Sec:ExactGeodesicsSLR}.
Since $\chi$ is convex on each interval $I$ where it is defined, Newton's method very efficiently computes its zeros. 
In particular, any value $\phi_0 \in I$ where $\chi(\phi_0) > 0$ can serve as a seed for the algorithm.

The lack of rotation invariance in Sol makes it much harder to implement similar ideas.
One could use a multi-variable Newton's method to find the geodesics from $o$ to $p$.
It is however not obvious where to start the procedure. A deeper analysis of the solutions of the geodesic flow is needed here. 

\subsection{Directed distance underestimators}
For certain scenes it can be difficult to produce the corresponding signed distance function, or even a distance underestimator.
An example is the $xy$-plane in the Nil geometry (\refsec{CreepingNilz=0Plane}).
However, when we are ray-marching along a geodesic $\gamma$, we do not in fact need to know the distance from any point $p \in X$ to the scene, but only the distance to the closest point of the scene lying on $\gamma$.
This leads us to the following definitions. 

\begin{definition}
Given a scene $S \subset X$, the associated \emph{directed signed distance function} $\sigma \colon TX \to \RR$ is a map characterized as follows.
Let $v \in T_pX$ be a tangent vector at $p$.
Let $\gamma$ be the geodesic starting at $p$ in the direction $v$.
\begin{itemize}
	\item If $p$ does not belong to $S$, then $\sigma(v)$ is the distance from $p$ to the closest point of $S$ on $\gamma$.
	\item If $p$ is in $S$, then $-\sigma(v)$ is the distance from $p$ to the closest point of $X \setminus S$ on $\gamma$. \qedhere
\end{itemize}
\end{definition}

Such a function is a priori also very hard to obtain.
Indeed it means that we can compute the intersection of any geodesic with our scene; this is precisely the data required for ray-tracing.
Nevertheless, as in \refsec{DistanceUnderestimators}, we can perform ray-marching using an underestimator that takes as its input a tangent vector to a ray. 

\begin{definition}
A \emph{directed distance underestimator} for the scene $S$ is a map $\sigma' \colon TX \to \RR$ such that 
\begin{enumerate}
\item The signs of $\sigma'(v)$ and $\sigma(v)$ are the same for all points $v \in TX$, 
\item $ |\sigma'(v)| \leq |\sigma(v)|$ for all $v\in TX$, and 
\item If $\{v_1, v_2, \ldots\}$ is a sequence of points in $TX$ such that $\sigma'(v_n)$ converges to zero, then so does $\sigma(v_n)$. \qedhere
\end{enumerate}
\end{definition}
Ray-marching with such a directed distance underestimator will produce the same pictures as ray-marching with an undirected signed distance function. These, in some sense, bridge the gap between ray-tracing and undirected ray-marching.
With directed distance underestimators, we expect to expand the collection of scenes that we can render.

Directed distance underestimators may also help improve efficiency.
When using a standard signed distance function (or distance underestimator), the length of the steps becomes very small as a geodesic ray passes very close to the scene without hitting it. 
If the maximal number of steps for the algorithm is not large enough, this creates background-colored halos around objects.
With a directed distance underestimator, we can hope that the length of the steps in this situation will be larger, thus making the algorithm converge faster.

 \subsection{Non-maximal homogeneous riemannian geometries}
 Recall that the transitive action of a Lie group $G$ on a manifold $X$ determines a homogeneous geometry. To be a Thurston geometry, a homogeneous geometry must satisfy four additional restrictions, see \refsec{ThurstonGeometries}.
 The first two of these conditions, having $X$ simply connected and $G$ act with compact point stabilizer, define a \emph{riemannian homogeneous space}. (For a complete classification of three-dimensional riemannian homogeneous spaces, see \cite{patrangenaru1996}.)
 These two conditions greatly simplify calculations of the geodesic flow, parallel transport, and more.  
 The second two conditions restrict to those needed for geometrization. However, we do not need these conditions anywhere in our ray-marching algorithms. There are many interesting geometries satisfying only the first two conditions that could be visualized in a similar fashion. Celi\'nska-Kopczy\'nska and Kopczy\'nski have begun to investigate visualizations of one-parameter spaces of metrics of this kind on the three-sphere.
 
\subsection{Homogeneous pseudo-riemannian \& lorentzian geometries}
 
Generalizing riemannian geometry, a \emph{pseudo-riemannian manifold} is a manifold $M$ together with a choice of (not necessarily positive definite) nondegenerate bilinear form on each tangent space.
When the bilinear form is not positive definite, the existence of null vectors (nonzero $v\in T_pM$ with $\langle v,v\rangle=0$) makes these spaces difficult to interpret visually (although see \cite{Egan17} for a literary interpretation).
 However, there is one class of pseudo-riemannian manifolds for which there is a clear interpretation of what the intrinsic view looks like:  lorentzian manifolds.
 These have bilinear forms of signature $(n-1,1)$ and are the basic models of space-time in relativistic physics.
 
In relativity, light travels along the null geodesics (geodesics with null tangents) in a lorentzian manifold, and so the intrinsic view may be computed by ray-marching starting with the lightcone of null vectors in the tangent space of the viewer.
In the real world, we see light that travels along null geodesics in a lorentzian four-manifold. Arguably, then, it is more natural to consider the inside view of a lorentzian four-manifold rather than of a riemannian three-manifold. 
 
 The natural starting place is flat space-time: the Minkowski space $\RR^{3,1}$.
 Ray-marching along lightcones in this geometry provides a method of simulating the inside view in special relativity.
Previous visualization work in special relativity includes \cite{Savage07, McGrath10, Muller10specialrelativistic, Sherin16}.
 Generalizing to homogeneous space-times of constant curvature, one could produce intrinsic simulations of de Sitter and anti-de Sitter space-time.
Many of the methods described in \refsec{General} can be adapted to this setting.
 All three of these have natural projective models in $\RR^5$, and explicit descriptions for their null geodesics and isometry groups are well known (see for example, \cite{Sokolowski6} and \cite{Kim:2002uz}).

Beyond these, the classification of general lorentzian homogenous four-manifolds has been completed \cite{Calvaruso14}, although it is more complex than the case of riemannian three-manifolds discussed above.
In all such manifolds we may use analogs of the techniques introduced in \refsec{General} to simplify computations.

\subsection{Non homogeneous geometries} 
Giving up on symmetry, there are many non-homogeneous riemannian and lorentzian manifolds for which intrinsic views may prove useful.
Examples include watching a three-manifold evolve under the Ricci flow, analyzing collapsing space-times, or space-times with singularities (black holes).
In most cases, the lack of symmetry forces us to use numeric solutions for the geodesic flow. However, there are also interesting non-homogeneous spaces with exactly solvable geodesic flow. These include the matrix group $\textrm{SL}(2,\RR)$ with the metric it inherits from the $2\!\times\!2$ matrices $\mathcal M_{2,2}(\RR) \cong \RR^4$ as a hypersurface.
However, these spaces all present considerable difficulties for the methods outlined in \refsec{General}, and will require more work.
 
\appendix
\section{Comparison between methods to integrate the geodesic flow}
\label{Sec:CompareFlowAlgorithms}

In this work, whenever possible we have avoided numerical methods for following geodesics and have instead exploited explicit solutions of the geodesic flow.
This allows us to quickly and accurately ray-march long distances, and thus render scenes with distant objects \cite{NEVR3, NEVR4}.
To support our choice, we ran some numerical experiments.
We explain our protocol below.

\begin{remark}
\label{Rem:ExperimentSetup}
We do not claim to give a comprehensive and rigorous comparison of the various methods to integrate the geodesic flow. 
The computations here are made in Python (using Numpy long double floats) on a standard desktop computer. We do not use the GPU, and no parallel computing is involved. Nevertheless, we can use these experiments to compare the relative efficiency of the algorithms.
\end{remark}

\subsection{Experimental protocol}
Let $(G,X)$ be one of the Thurston geometries.
We fix an integer $N \in \NN$ and a time $t \in \RR_+$.
We compare four methods: using exact formulas, Euler's method, and the Runge--Kutta methods of order two and four. For the numerical methods, we also compare different step sizes $\Delta t$.

We first generate a list $V$ of $N$ unit tangent vectors at the origin $o \in X$, chosen uniformly and independently at random. 
For each experiment $\calE$ in each of the Tables~\ref{Tab:compare Nil6},~\ref{Tab:compare Nil10}, \ref{Tab:compare SLR6}, and~\ref{Tab:compare SLR10}, we fix a method and (for the numerical methods) a time-step.  
We then do the following computations.
\begin{itemize}
	\item For each direction $v \in V$, we flow from $o$ for time $t$ and record the final position.
	This yields a list $Q_\calE$ of $N$ points.
	We also record the time needed to compute $Q_\calE$.
	\item Next, for each $q_\calE \in Q_\calE$, we measure the error of $q_\calE$ with respect to the exact flow.
		We discuss our choice of error measurements in \refsec{MeasuringErrors}.
	\item Finally, we compute the maximal and mean errors for the set $Q_\calE$.
\end{itemize}

\subsection{Measuring errors}
\label{Sec:MeasuringErrors}
We calculate two different measures of error.
Fixing notation, let $v \in V$ be one element in our collection of random tangent vectors and $q_\calE \in Q_\calE$ be the point obtained by following the geodesic flow starting at $o$ in the direction of $v$ in  for time $t$ in the experiment $\calE$.

\subsubsection{Distance error} We compute the coordinates of the point $q$ obtained by following the geodesic flow starting at $o$ in the direction of $v$ in  for time $t$ using the \emph{exact formulas}.
\begin{definition}
The \emph{distance error} is the distance in the metric of $X$ between $q_\calE$ and $q$.
\end{definition}

\begin{remark}
One should worry about how accurate our computer's implementation of the exact formulas is. As mentioned in \refrem{ExperimentSetup}, we use NumPy for all of our calculations here, and long doubles, giving us around 19 decimal digits of accuracy. While we have not looked into the actual implementations of the functions we use, we would certainly hope that these implementations lose at most one or two digits of accuracy on each operation. Of course the results of these functions then need to be combined, which compounds the errors. Without using interval arithmetic, it is hard to say how accurate our final results are. However, as we will see in our experiments, with small values of $t$ and small step size $\Delta t$, our exact calculation matches Runge--Kutta of order four ($\Delta t = 0.01$) up to a distance error of around $10^{-9}$ at worst. This provides evidence that our implementation of the exact formulas are at least this accurate in comparison with the true values, since the exact and Runge--Kutta methods take very different routes to their results.
\end{remark}

The distance error is natural, but the results are sometimes difficult to interpret because of our lack of intuition in those geometries. Moreover, our eyes place far more importance in which direction one sees an object in, over how far away it is. If we have an error in distance, then perhaps at worst the effect of fog is slightly incorrect. An error in direction could cause us to see objects in the wrong place, or distorted in some way. To better measure this, we introduce our second error measurement.

\subsubsection{Angle error}
We compute the tangent vector $v'$  so that the \emph{exact} geodesic flow starting from $o$ in the direction of $v'$ hits the point $q_\calE$. When there are multiple such tangent vectors $v'$, we choose the one which is closest (in angle) to $v$.
\begin{definition}
The \emph{angle error} is the angle between $v$ and $v'$.
\end{definition}

Following our goal of producing accurate images in \refsec{Goals}, it is reasonable to require that each pixel of our screen be colored according to an object that should be visible through that pixel. Therefore, given the resolution of our screen and a desired field of view, one can calculate a maximum acceptable angle error, as follows.

\begin{figure}[htbp]

\labellist
\footnotesize
\pinlabel $1$ [c] at 240 262
\pinlabel $1$ [c] at 414 146
\pinlabel $\beta/2$ [l] at 102 157
\pinlabel $\alpha$ at 188 128
\pinlabel $\Delta\alpha$ [c] at 170 105
\pinlabel $-f_3$ [l] at 136 154
\pinlabel $v$ [l] at 142 131
\pinlabel $v'$ [l] at 115 97
\pinlabel $m$ [l] at 275 96
\pinlabel $m'$ [l] at 240 62
\pinlabel $\Delta m$ [br] at 253 75
\endlabellist

\includegraphics[width=0.9\textwidth]{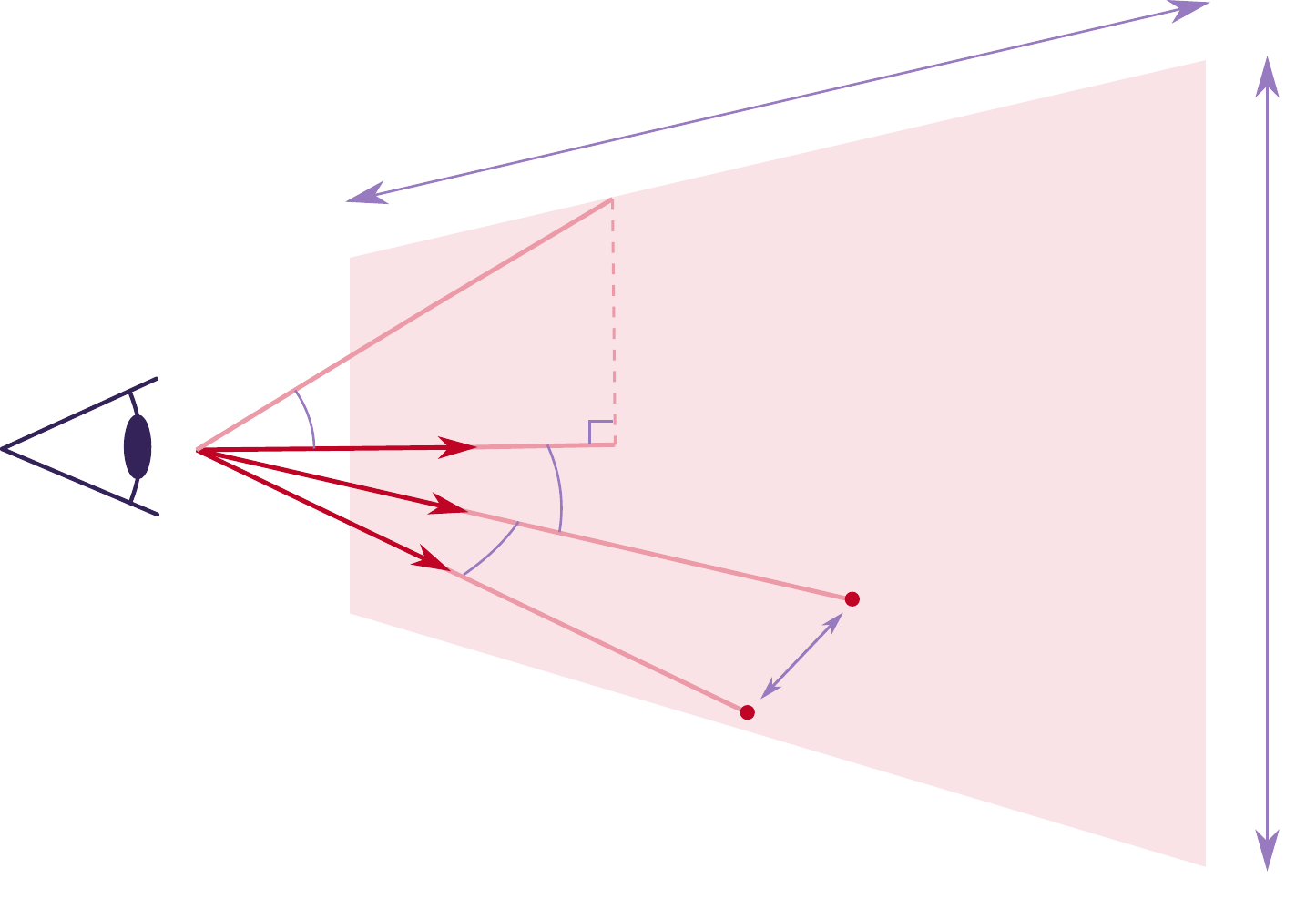}
\caption{Relation between angle error and resolution.}
\label{Fig:Screen}
\end{figure}
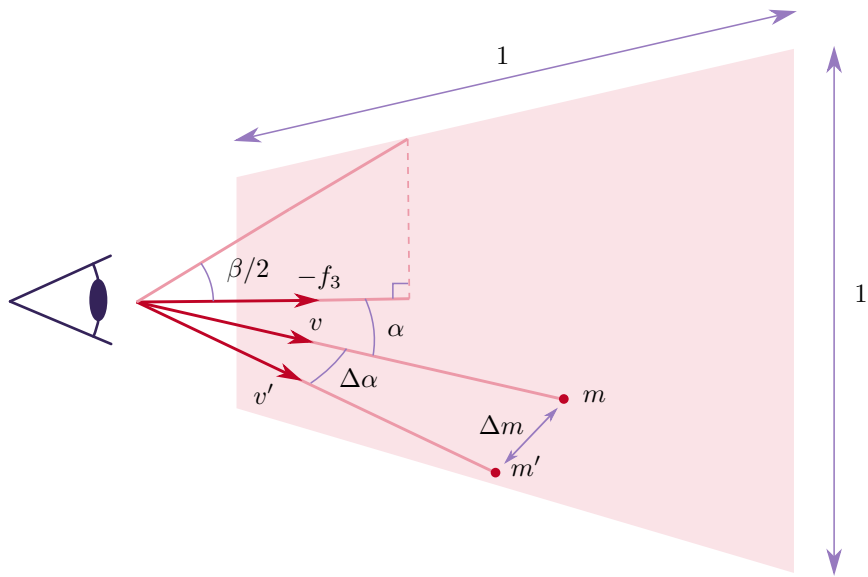

Let $m$ and $m'$ be the locations on our screen corresponding to the directions $v$ and $v'$. See \reffig{Screen}.
Suppose that the width of the screen is one unit.
The distance between $m$ and $m'$ can be estimated as follows.

Assume that the field of view is $\beta$.
Let $\alpha$ be the angle between $v$ and the vector $-f_3$ pointing forwards, and let $\Delta\alpha$ be the angle between $v$ and $v'$.
Then the distance $\dist (m, m')$ is at most
\begin{equation*}
	\frac { \left|\tan(\alpha + \Delta\alpha) - \tan (\alpha) \right|}{2 \tan (\beta/2)}.
\end{equation*}
For a fixed angle error $\Delta \alpha$, this quantity is the largest when $m$ is on the border of the screen, that is when $\alpha = \beta/2$.
Hence the worst error $\Delta m$ for $m$ is related to $\Delta \alpha$ by 
\begin{equation*}
	\tan(\Delta \alpha) = \frac{ \Delta m\sin \beta}{ 1+ 2 \Delta m\sin^2(\beta / 2) }
\end{equation*}
For the picture on the screen to be accurate, we need $\Delta m$ to be less than half the width of a pixel.
For example, fixing the field of view at $\beta = 100\degree$, the maximum acceptable angle error (in degrees) is 
\begin{itemize}
	\item $\Delta \alpha \approx$ \texttt{3e-02} to produce a $1000 \times 1000$ pixel image,
	\item $\Delta \alpha \approx$ \texttt{6e-03} to produce a $5000 \times 5000$ pixel image.
\end{itemize}

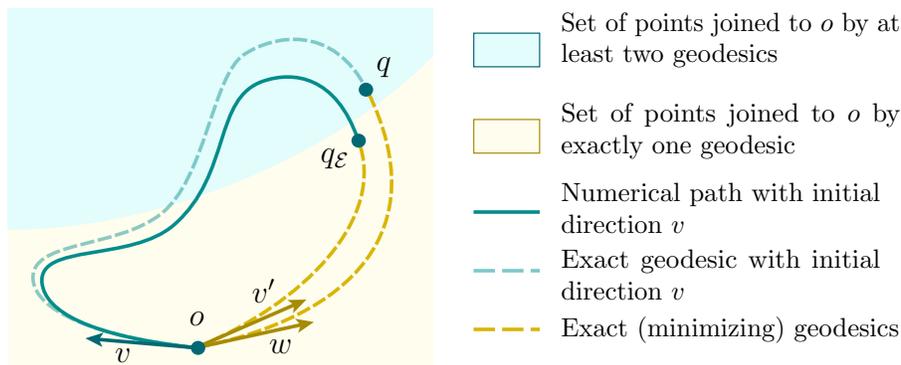
\begin{figure}[htbp]

\labellist
\normalsize
\pinlabel \parbox{4.5cm}{\footnotesize Set of points joined to $o$ by at least two geodesics} [l] at 650 396
\pinlabel \parbox{4.5cm}{\footnotesize Set of points joined to $o$ by exactly one geodesic} [l] at 650 287
\pinlabel \parbox{4.25cm}{\footnotesize Numerical path with initial direction $v$} [l] at 650 196
\pinlabel \parbox{4.25cm}{\footnotesize Exact geodesic with initial direction $v$} [l] at 650 116
\pinlabel \parbox{4.5cm}{\footnotesize Exact (minimizing) geodesics} [l] at 650 50

\pinlabel $o$ [c] at 225 63
\pinlabel $q$ [c] at 448 363
\pinlabel $q_\calE$ [c] at 390 254
\pinlabel $v$ [c] at 137 20
\pinlabel $v'$ [c] at 305 97
\pinlabel $w$ [c] at 325 30
\endlabellist

\includegraphics[width=\textwidth]{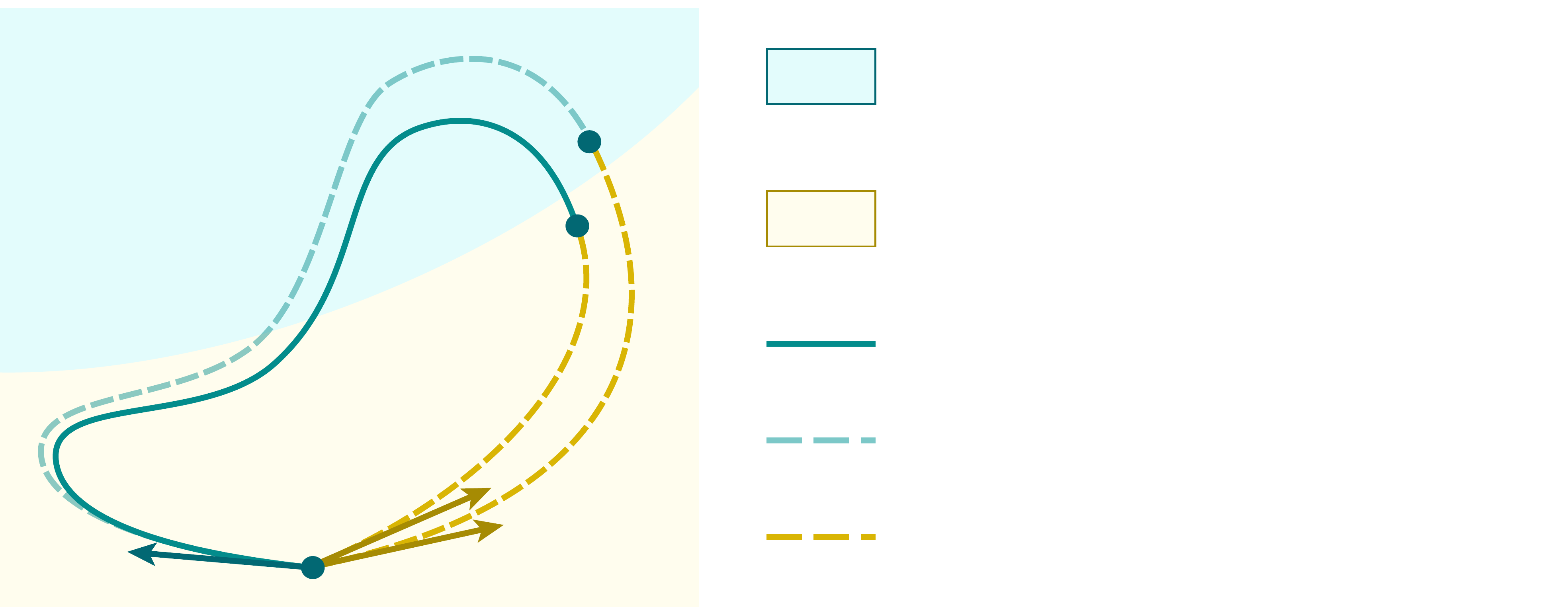}
\caption{A schematic picture of an exceptional sample.}
\label{Fig:Experiment}
\end{figure}

\begin{remark}
\label{Rem:ExceptionalPoints}
When following the geodesic flow for a time $t$ which is smaller than the injectivity radius of the geometry $X$, there is only one exact geodesic joining $o$ to $q_\calE$.
Here there is no choice in the definition of $v'$.
For longer flow times, a new phenomenon arises.
As usual, we numerically compute the path starting at $o$ in the direction $v$ and reach the point $q_\calE$. This path approximates the exact geodesic ray starting at $o$ in the direction $v$, which reaches the point $q$. See \reffig{Experiment}.
Suppose that multiple exact geodesics join $o$ to $q$, so that in addition to the direction $v$, we can also reach $q$ along a geodesic with starting direction $w$.
Suppose also that due to numerical errors, only one exact geodesic joins $o$ to $q_\calE$. The vector $v'$ is then the only possible initial direction pointing from $o$ to $q_\calE$. 
It will be close to one of $v$ and $w$, but it may be close to the wrong one: $w$.
Thus, the angle between $v$ and $v'$ can be very large.
However, the visual effect of this error will be indistinguishable from similar inaccuracies with the same distance error.

This situation is very rare: our numerical path must approximate a non-minimizing geodesic, with endpoint  $q_\calE$ close to the boundary of one of the sets 
\[
X_n = \{ x \in X \mid x \textrm{ is joined to $o$ by $n$ geodesics} \}
\]
\noindent
For the Thurston geometries, these boundaries form a zero-measure set.

In our results, we indicate the number of points for which the angle between $v$ and $v'$ is more than a large threshold, for example $40\degree$. These cases correspond to the situation described above.
We compute the maximal and mean errors excluding these exceptional samples.
\end{remark}

\subsection{Results.}
We carried out our protocol for Nil and $\SLR$. 
We computed the distance and angle errors using the numerical methods described in Sections~\ref{Sec:ExactGeodesicsNil} and \ref{Sec:ExactGeodesicsSLR} respectively. See 
Tables~\ref{Tab:compare Nil6},~\ref{Tab:compare Nil10}, \ref{Tab:compare SLR6}, and~\ref{Tab:compare SLR10}.
We made sure that the errors coming from use of Newton's method to calculate $v'$ are negligible compared to the results.
We ran the experiments for time $t = 6$ and $t = 10$. Note that $6$ is less than the injectivity radius ($2\pi$ for both Nil and $\SLR$).

\begin{table}[htp]
\scriptsize
\renewcommand{\arraystretch}{1.2}
\begin{tabular}{|l|l||r|l|l|r|l|l|}\hline
	\bfseries	Method
		& \bfseries  $\Delta t$
		& \bfseries \rotatebox[origin=c]{90}{\parbox{3cm}{\centering Time needed\\ (in~s.)}} 
		& \bfseries  \rotatebox[origin=c]{90}{\parbox{3.1cm}{\centering Maximal  distance error}}
		& \bfseries \rotatebox[origin=c]{90}{\parbox{3cm}{\centering Mean distance error}} 
		& \bfseries \rotatebox[origin=c]{90}{\parbox{3cm}{\centering Number of directions off ($\Delta\alpha > 60\degree$)}} 
		& \bfseries \rotatebox[origin=c]{90}{\parbox{3cm}{\centering Maximal  angle error (in~$\degree$)}} 
		& \bfseries \rotatebox[origin=c]{90}{\parbox{3cm}{\centering Mean  angle error (in~$\degree$)}} 
		\\\hline\hline
	Exact flow &  -- & 0.2 & -- & -- & --&--&--\\ \hline
	Euler & 0.1 & 12.0 & \texttt{8.4e-01} & \texttt{5.2e-01} & 0 & \texttt{8.7e+00} & \texttt{4.4e+00} \\ \hline
	Euler  & 0.01 & 115.3 & \texttt{8.6e-02} & \texttt{5.2e-02} & 0 & \texttt{9.4e-01} & \texttt{4.8e-01}\\ \hline
	Runge--Kutta 2  & 0.1 & 17.8 & \texttt{9.3e-03} & \texttt{5.0e-03} & 0 & \texttt{1.7e-01} & \texttt{3.1e-02}\\ \hline
	Runge--Kutta 2  & 0.01 & 176.0 & \texttt{9.6e-05} & \texttt{5.1e-05} & 0 & \texttt{1.7e-03} & \texttt{3.0e-04} \\ \hline
	Runge--Kutta 4 & 0.1 & 32.4 & \texttt{8.2e-06} & \texttt{4.2e-06} & 0 & \texttt{8.1e-05} & \texttt{1.8e-05}\\ \hline
	Runge--Kutta 4  & 0.01 & 323.3 & \texttt{8.2e-10} & \texttt{4.2e-10} & 0 & \texttt{2.7e-08} & \texttt{1.2e-09}\\ \hline
\end{tabular}
\renewcommand{\arraystretch}{1}

\medskip
\caption{Integrating the geodesic flow in Nil. Computation made with $N =  10,000$ and $t = 6$.}
\label{Tab:compare Nil6}
\end{table}

\begin{table}[htp]
\scriptsize
\renewcommand{\arraystretch}{1.2}
\begin{tabular}{|l|l||r|l|l|r|l|l|}\hline
	\bfseries	Method
		& \bfseries  $\Delta t$
		& \bfseries \rotatebox[origin=c]{90}{\parbox{3cm}{\centering Time needed \\ (in s.)}} 
		& \bfseries  \rotatebox[origin=c]{90}{\parbox{3.1cm}{\centering Maximal  distance error}}
		& \bfseries \rotatebox[origin=c]{90}{\parbox{3cm}{\centering Mean distance error}} 
		& \bfseries \rotatebox[origin=c]{90}{\parbox{3cm}{\centering Number of directions off ($\Delta\alpha > 60\degree$)}} 
		& \bfseries \rotatebox[origin=c]{90}{\parbox{3cm}{\centering Maximal  angle error (in~$\degree$)}} 
		& \bfseries \rotatebox[origin=c]{90}{\parbox{3cm}{\centering Mean  angle error (in~$\degree$)}} 
		\\\hline\hline
	Exact flow &  -- & 0.3 & -- & -- & --&--&--\\ \hline
	Euler & 0.1 & 20.4 & \texttt{3.8e+00} & \texttt{2.2e+00} & 690 & \texttt{6.0e+01} & \texttt{1.4e+01} \\ \hline
	Euler  & 0.01 & 191.8 & \texttt{4.3e-01} & \texttt{2.2e-01} & 71 & \texttt{5.3e+01} & \texttt{2.4e+00}\\ \hline
	Runge--Kutta 2  & 0.1 & 30.0 & \texttt{3.1e-02} & \texttt{1.4e-02} & 186 & \texttt{2.5e+00} & \texttt{7.5e-02}\\ \hline
	Runge--Kutta 2  & 0.01 & 297.0 & \texttt{3.3e-04} & \texttt{1.4e-04} & 19 & \texttt{1.5e-01} & \texttt{9.8e-04} \\ \hline
	Runge--Kutta 4 & 0.1  & 54.7 & \texttt{2.8e-05} & \texttt{1.2e-05} & 0 & \texttt{2.0e-02} & \texttt{5.7e-05}\\ \hline
	Runge--Kutta 4  & 0.01 & 540.1 & \texttt{2.8e-09} & \texttt{1.2e-09} & 0 & \texttt{3.5e-06} & \texttt{4.5e-09}\\ \hline
\end{tabular}
\renewcommand{\arraystretch}{1}

\medskip
\caption{Integrating the geodesic flow in Nil. Computation made with $N =  10,000$ and $t = 10$.}
\label{Tab:compare Nil10}
\end{table}

\begin{table}[htp]
\scriptsize
\renewcommand{\arraystretch}{1.2}
\begin{tabular}{|l|l||r|l|l|r|l|l|}\hline
	\bfseries	Method
		& \bfseries  $\Delta t$
		& \bfseries \rotatebox[origin=c]{90}{\parbox{3cm}{\centering Time needed (in s.)}} 
		& \bfseries  \rotatebox[origin=c]{90}{\parbox{3.1cm}{\centering Maximal  distance error}}
		& \bfseries \rotatebox[origin=c]{90}{\parbox{3cm}{\centering Mean distance error}} 
		& \bfseries \rotatebox[origin=c]{90}{\parbox{3cm}{\centering Number of directions off ($\Delta\alpha > 40\degree$)}} 
		& \bfseries \rotatebox[origin=c]{90}{\parbox{3cm}{\centering Maximal  angle error (in~$\degree$)}} 
		& \bfseries \rotatebox[origin=c]{90}{\parbox{3cm}{\centering Mean  angle error (in~$\degree$)}} 
		\\\hline\hline
	Exact flow &  -- & 0.8 & -- & -- & --&--&--\\ \hline
	Euler & 0.1 & 33.2 & \texttt{3.7e+00} & \texttt{2.7e+00} & 0 & \texttt{5.2e+01}	& \texttt{5.4e+00}\\ \hline
	Euler  & 0.01 & 325.1 & \texttt{6.6e-01} & \texttt{3.9e-01} & 0 & \texttt{2.7e+00} & \texttt{7.1e-01}\\ \hline
	Runge--Kutta 2  & 0.1 & 49.0 & \texttt{4.8e-02} & \texttt{2.8e-02} & 0 & \texttt{1.0e+00} & \texttt{1.3e-01}\\ \hline
	Runge--Kutta 2  & 0.01 & 485.5 & \texttt{6.6e-04} & \texttt{3.6e-04} & 0 & \texttt{9.4e-03} & \texttt{1.2e-03} \\ \hline
	Runge--Kutta 4 & 0.1 & 83.9 & \texttt{4.1e-05} & \texttt{2.4e-05} & 0 & \texttt{2.0e-03} & \texttt{1.8e-04}\\ \hline
	Runge--Kutta 4  & 0.01 & 817.5 & \texttt{4.7e-09} & \texttt{2.3e-09} & 0 & \texttt{1.9e-07} & \texttt{1.3e-08}\\ \hline
\end{tabular}
\renewcommand{\arraystretch}{1}

\medskip
\caption{Integrating the geodesic flow in $\SLR$. Computation made with $N =  10,000$ and $t = 6$.}
\label{Tab:compare SLR6}
\end{table}

\begin{table}[htp]
\scriptsize
\renewcommand{\arraystretch}{1.2}
\begin{tabular}{|l|l||r|l|l|r|l|l|}\hline
	\bfseries	Method
		& \bfseries  $\Delta t$
		& \bfseries \rotatebox[origin=c]{90}{\parbox{3cm}{\centering Time needed (in s.)}} 
		& \bfseries  \rotatebox[origin=c]{90}{\parbox{3.1cm}{\centering Maximal  distance error}}
		& \bfseries \rotatebox[origin=c]{90}{\parbox{3cm}{\centering Mean distance error}} 
		& \bfseries \rotatebox[origin=c]{90}{\parbox{3cm}{\centering Number of directions off ($\Delta\alpha > 40\degree$)}} 
		& \bfseries \rotatebox[origin=c]{90}{\parbox{3cm}{\centering Maximal  angle error (in~$\degree$)}} 
		& \bfseries \rotatebox[origin=c]{90}{\parbox{3cm}{\centering Mean  angle error (in~$\degree$)}} 
		\\\hline\hline
	Exact flow &  -- & 0.7 & -- & -- & --&--&--\\ \hline
	Euler & 0.1 & 53.7 & \texttt{1.1e+01} & \texttt{7.9e+00} & 594 & \texttt{4.0e+01} & \texttt{8.3e+00} \\ \hline
	Euler  & 0.01 & 523.3 & \texttt{6.2e+00} & \texttt{3.6e+00} & 74 & \texttt{3.7e+01} & \texttt{2.2e+00}\\ \hline
	Runge--Kutta 2  & 0.1 & 78.5 & \texttt{9.0e-01} & \texttt{3.3e-01} & 180 & \texttt{4.0e+01} & \texttt{5.9e-01}\\ \hline
	Runge--Kutta 2  & 0.01 & 775.1 & \texttt{1.3e-02} & \texttt{4.7e-03} & 25 & \texttt{3.4e-01} & \texttt{2.9e-03} \\ \hline
	Runge--Kutta 4 & 0.1 & 133.1 & \texttt{8.6e-04} & \texttt{3.6e-04} & 0 & \texttt{2.8e-01} & \texttt{9.1e-04} \\ \hline
	Runge--Kutta 4  & 0.01 & 1,316.9 & \texttt{7.4e-08} & \texttt{3.1e-08} & 0 & \texttt{1.1e-04} & \texttt{7.5e-08}\\ \hline
\end{tabular}
\renewcommand{\arraystretch}{1}

\medskip
\caption{Integrating the geodesic flow in $\SLR$. Computation made with $N =  10,000$ and $t = 10$.}
\label{Tab:compare SLR10}
\end{table}

\subsection{Discussion}

The maximal angle errors for Euler's method do not produce accurate $1000 \times 1000$ pixel images with field of view $100\degree$, even for small flow time ($t = 6$) and small time step ($\Delta t = 0.01$).
The Runge--Kutta method of order two is accurate enough for small distance only (and sometimes only for smaller step like $\Delta t = 0.01$).
For medium distances ($t = 10$), in $\SLR$ only the Runge--Kutta method of order four with step $\Delta t= 0.01$ meets our criterion.
The Runge--Kutta method of order four is also the only one that does not produce exceptional points in the sense of \refrem{ExceptionalPoints}.

In terms of the time needed to run the computations, the exact method is superior to the numerical ones. 
Lookup tables may be precomputed to avoid long calculation times, although one should then also worry about inaccuracies introduced by interpolation.

\vspace{-5pt}

\newcommand{\etalchar}[1]{$^{#1}$}
\providecommand{\bysame}{\leavevmode\hbox to3em{\hrulefill}\thinspace}
\providecommand{\MR}{\relax\ifhmode\unskip\space\fi MR }
\providecommand{\MRhref}[2]{%
  \href{http://www.ams.org/mathscinet-getitem?mr=#1}{#2}
}
\providecommand{\href}[2]{#2}

\vspace{10pt}
\noindent
\emph{R\'emi Coulon} \\
Univ Rennes, CNRS \\
IRMAR - UMR 6625 \\
F-35000 Rennes, France\\
\texttt{remi.coulon@univ-rennes1.fr} \\
\url{http://rcoulon.perso.math.cnrs.fr}\\

\noindent
\emph{Elisabetta A. Matsumoto} \\
School of Physics \\
Georgia Institute of Technology \\
837 State Street, 
Atlanta, GA, 30332, USA \\
\texttt{sabetta@gatech.edu} \\
\url{http://matsumoto.gatech.edu}\\

\noindent
\emph{Henry Segerman} \\
Department of Mathematics \\
Oklahoma State University \\
Stillwater, OK, 74078, USA \\
\texttt{segerman@math.okstate.edu} \\
\url{https://math.okstate.edu/people/segerman/}\\

\noindent
\emph{Steve J. Trettel} \\
Stanford University\\
450 Jane Stanford Way,\\
 Stanford, CA 94305\\
\texttt{trettel@stanford.edu} \\
\url{http://stevejtrettel.site}

\end{document}